\documentclass[bibliography=totoc,a4paper]{scrreprt}

\usepackage[utf8]{inputenc}
\usepackage[T2A,T1]{fontenc}
\usepackage{lmodern}
\usepackage[ngerman,british]{babel}
\usepackage{amsmath}
\usepackage{amsfonts}
\usepackage{amssymb}
\usepackage{amsthm}
\usepackage{dsfont}
\usepackage[pdftex]{graphicx}
\usepackage{multirow}
\usepackage{hhline}
\usepackage{thmtools}
\usepackage{framed}
\usepackage{mathtools}
\usepackage{tikz}
\usetikzlibrary{chains,arrows,shapes,decorations.pathmorphing}
\usepackage{tikz-cd}
\usepackage{rotating}
\usepackage{geometry}
\usepackage[page,toc,title,header,titletoc]{appendix}
\usepackage[hypcap]{caption}\usepackage{substitutefont}
\substitutefont{T2A}{lmr}{fcm}
\substitutefont{T2A}{lmss}{fcs}
\substitutefont{T2A}{lmtt}{fct}
\usepackage{enumitem}
\setenumerate[1]{label=(\arabic*)}
\setenumerate[2]{label=(\alph*)}
\usepackage[pdftex,
            pdfauthor={Sven Möller},
            pdftitle={A Cyclic Orbifold Theory for Holomorphic Vertex Operator Algebras and Applications},
            pdfsubject={},
            pdfkeywords={},
            pdfproducer={},
            pdfcreator={}]{hyperref}
\usepackage[numbered,open,openlevel=0]{bookmark}

\allowdisplaybreaks[1] \clubpenalty = 500 \widowpenalty = 500 \displaywidowpenalty = 500

\theoremstyle{plain}
\newtheorem{thm}{Theorem}[section]
\newtheorem*{thm*}{Theorem}
\newtheorem{prop}[thm]{Proposition}
\newtheorem*{prop*}{Proposition}
\newtheorem{cor}[thm]{Corollary}
\newtheorem{lem}[thm]{Lemma}

\newtheorem*{conj*}{Conjecture}

\theoremstyle{definition}
\newtheorem{defi}[thm]{Definition}
\newtheorem{rem}[thm]{Remark}

\newtheorem*{ass*}{Assumption}

\newtheorem{innercustomass}{Assumption}
\newenvironment{customass}[1]{\renewcommand\theinnercustomass{#1}\innercustomass}{\endinnercustomass}
\newtheorem{innercustomconj}{Conjecture}
\newenvironment{customconj}[1]{\renewcommand\theinnercustomconj{#1}\innercustomconj}{\endinnercustomconj}

\newcommand{\xmiddle}[1]{\;\middle#1\;}\newcommand{\Q}{\mathbb{Q}}
\newcommand{\Z}{\mathbb{Z}}
\newcommand{\Ns}{\mathbb{Z}_{>0}}
\newcommand{\N}{\mathbb{Z}_{\geq0}}
\renewcommand{\C}{\mathbb{C}}
\newcommand{\R}{\mathbb{R}}
\renewcommand{\H}{\mathbb{H}}
\newcommand{\h}{\mathfrak{h}}
\renewcommand{\k}{\mathfrak{k}}
\newcommand{\wt}{\operatorname{wt}}
\newcommand{\tr}{\operatorname{tr}}
\renewcommand{\i}{\mathrm{i}}\newcommand{\e}{\mathrm{e}}\newcommand{\End}{\operatorname{End}}
\newcommand{\Aut}{\operatorname{Aut}}
\newcommand{\Hom}{\operatorname{Hom}}
\newcommand{\Mor}{\operatorname{Mor}}
\newcommand{\Ind}{\operatorname{Ind}}
\newcommand{\rk}{\operatorname{rk}}
\newcommand{\im}{\operatorname{im}}
\newcommand{\ord}{\operatorname{ord}}
\newcommand{\sign}{\operatorname{sign}}
\newcommand{\voa}{vertex operator algebra}
\newcommand{\Voa}{Vertex operator algebra}
\newcommand{\VOA}{Vertex Operator Algebra}
\newcommand{\vosa}{vertex operator subalgebra}

\newcommand{\aia}{abelian intertwining algebra}
\newcommand{\Aia}{Abelian intertwining algebra}
\newcommand{\AIA}{Abelian Intertwining Algebra}
\newcommand{\mtc}{modular tensor category}

\newcommand{\MTC}{Modular Tensor Category}
\newcommand{\mtcs}{modular tensor categories}

\newcommand{\MTCs}{Modular Tensor Categories}
\newcommand{\fpvosa}{fixed-point vertex operator subalgebra}

\newcommand{\FPVOSA}{Fixed-Point Vertex Operator Subalgebra}
\newcommand{\BKMa}{Borcherds-Kac-Moody algebra}
\newcommand{\BKMA}{Borcherds-Kac-Moody Algebra}
\newcommand{\fqs}{finite quadratic space}
\newcommand{\Fqs}{Finite quadratic space}
\newcommand{\FQS}{Finite Quadratic Space}
\newcommand{\vac}{\textbf{1}}\newcommand{\ch}{\operatorname{ch}}\newcommand{\sch}{\operatorname{sch}}\newcommand{\id}{\operatorname{id}}
\newcommand{\amgis}{{\reflectbox{$\sigma$}}}
\newcommand{\eps}{\varepsilon}
\newcommand{\lcm}{\operatorname{lcm}}
\newcommand{\SLZ}{\operatorname{SL}_2(\mathbb{Z})}
\newcommand{\SLR}{\operatorname{SL}_2(\mathbb{R})}
\newcommand{\GL}{\operatorname{GL}}
\newcommand{\MpZ}{\operatorname{Mp}_2(\mathbb{Z})}
\newcommand{\ee}{\mathfrak{e}}
\newcommand{\g}{\mathfrak{g}}
\newcommand{\s}{\mathfrak{s}}
\newcommand{\ad}{\operatorname{ad}}
\newcommand{\Inn}{\operatorname{Inn}}
\newcommand{\Out}{\operatorname{Out}}
\renewcommand{\sl}{\mathfrak{sl}}
\newcommand{\so}{\mathfrak{so}}
\renewcommand{\sp}{\mathfrak{sp}}
\renewcommand{\S}{\mathcal{S}}
\newcommand{\T}{\mathcal{T}}
\newcommand{\Irr}{\operatorname{Irr}}
\newcommand{\V}{\mathcal{V}}
\newcommand{\qdim}{\operatorname{qdim}}\newcommand{\sdim}{\operatorname{sdim}}

\newcommand{\oddity}{\operatorname{oddity}}
\newcommand{\spn}{\operatorname{span}}
\newcommand{\CC}{\mathcal{C}}
\renewcommand{\Im}{\operatorname{Im}}

\newcommand{\unit}{{\mathbf{1}_\mathcal{C}}}\newcommand{\Res}{\operatorname{Res}}
\newcommand{\mult}{\operatorname{mult}}
\newcommand{\II}{I\!I}

\begin{hyphenrules}{ngerman}
\end{hyphenrules}

\tikzset{ch/.style={circle,draw,on chain,inner sep=2pt},chj/.style={ch,join}}

\newcommand{\mlabel}[1]{  \(#1\)
}

\let\dlabel=\alabel

\newcommand{\dnode}[2][chj]{\node[#1,label={below:\dlabel{#2}},fill=yellow] {};
}

\newcommand{\dnodenj}[1]{\dnode[ch]{#1}
}

\newcommand{\dnodebr}[1]{\node[chj,label={below right:\dlabel{#1}},fill=yellow] {};
}

\newcommand{\dydots}{\node[chj,draw=none,inner sep=1pt] {\dots};
}

\newcommand{\QLeftarrow}{\begingroup
\tikz
\draw[shorten >=0pt,shorten <=0pt] (0,3pt) -- ++(-1em,0) (0,1pt) -- ++(-1em-1pt,0) (0,-1pt) -- ++(-1em-1pt,0) (0,-3pt) -- ++(-1em,0) (-1em+1pt,5pt) to[out=-105,in=45] (-1em-2pt,0) to[out=-45,in=105] (-1em+1pt,-5pt);
\endgroup
}

\title{A Cyclic Orbifold Theory for Holomorphic \VOA{}s and Applications}

\author{}

\date{\normalsize
vom Fachbereich Mathematik\\
der Technischen Universität Darmstadt\\
zur Erlangung des Grades eines\\
Doktors der Naturwissenschaften\\
(Dr.\ rer.\ nat.)\\
genehmigte Dissertation\\
~\\
~\\
Tag der Einreichung: 07. Juli 2016\\
Tag der mündlichen Prüfung: 15. September 2016\\
~\\
Referent: Prof.\ Dr.\ Nils R.\ Scheithauer\\
1. Korreferent: Prof.\ Dr.\ Martin Möller\\
2. Korreferent: Prof.\ Dr.\ Gerald Höhn\\
~\\
von\\
Sven Möller, M.Sc.\\
aus\\
Wiesbaden\\
~\\
~\\
Darmstadt, D 17\\
2016}

\titlehead{
\centering
\includegraphics*[width=6cm]{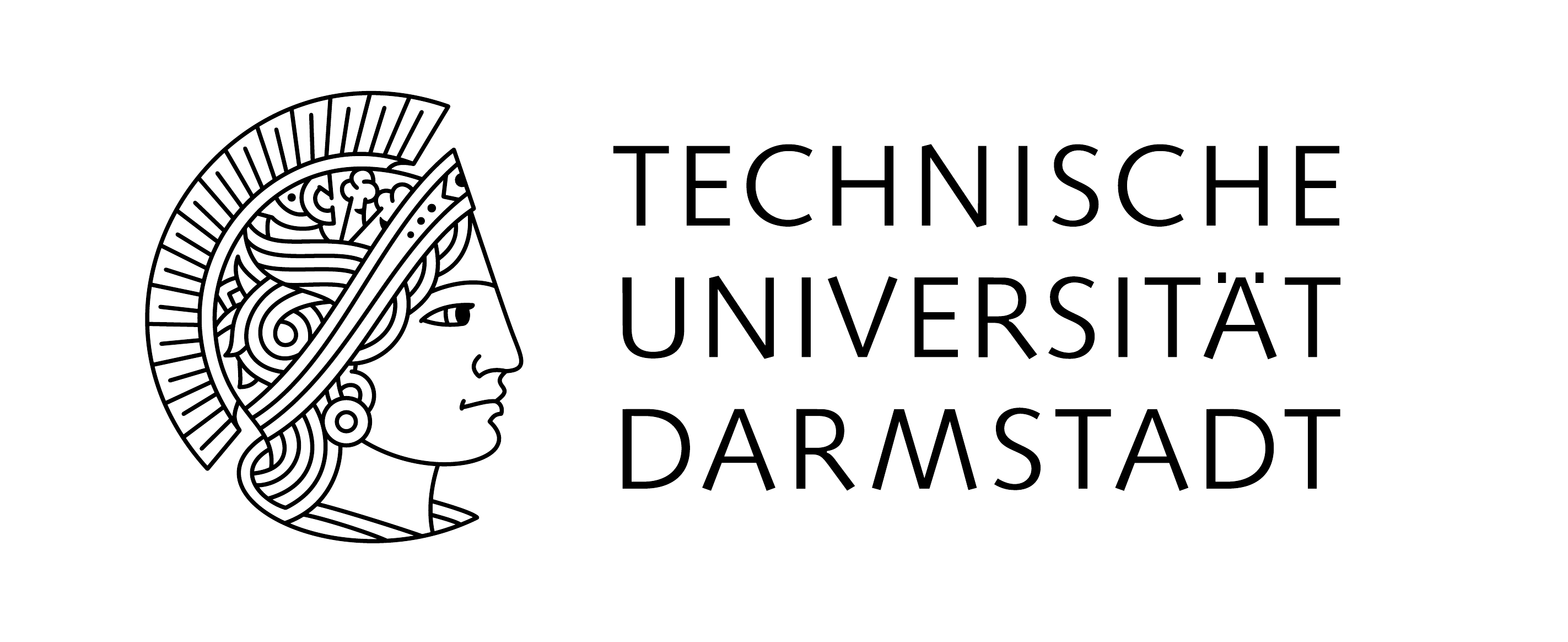}
}

\dedication{Für Ernst und Magda}

\begin{document}

\maketitle

\begin{otherlanguage}{ngerman}
\chapter*{Zusammenfassung}
Diese Dissertation beschäftigt sich mit der Konstruktion und Klassifikation von Vertexoperatoralgebren, die gewissen Regularitätsannahmen genügen. Wir zeigen, dass die Fusionsalgebra einer solchen Vertexoperatoralgebra, deren irreduzible Moduln zudem allesamt einfache Ströme sind, die Gruppenalgebra einer endlichen, abelschen Gruppe ist, die wir Fusionsgruppe nennen. Außerdem definiert die Modulo-1-Reduktion der konformen Gewichte der irreduziblen Moduln eine nicht-ausgeartete quadratische Form auf dieser Fusionsgruppe. Wir zeigen weiterhin, dass die direkte Summe aller irreduziblen Moduln eine abelsche Intertwining-Algebra ist, deren zugeordnete quadratische Form das Negative der gerade erwähnten ist. Schränkt man diese Summe auf eine isotrope Untergruppe ein, so erhält man eine Vertexoperatoralgebra, welche die ursprüngliche Vertexoperatoralgebra erweitert.

Weiterhin bestimmen wir die Fusionsalgebra der Fixpunktvertexoperatorunteralgebra einer holomorphen Vertexoperatoralgebra unter einer endlichen, zyklischen Gruppe von Automorphismen. Wir zeigen, dass die Fusionsgruppe eine zentrale Erweiterung dieser endlichen, zyklischen Gruppe mit sich selbst ist, und bestimmen deren Isomorphieklasse. Außerdem bestimmen wir die quadratische Form auf der Fusionsgruppe.

Durch Kombination der gerade genannten Resultate erhalten wir ausgehend von einer holomorphen Vertexoperatoralgebra und einer endlichen, zyklischen Automorphismengruppe eine weitere holomorphe Vertexoperatoralgebra, definiert auf gewissen irreduziblen Moduln der Fixpunktvertexoperatorunteralgebra. Diese holomorphe Vertexoperatoralgebra heißt Orbifold der ursprünglichen.

Als Anwendung der Orbifoldtheorie konstruieren wir fünf holomorphe Vertexoperatoralgebren von zentraler Ladung 24 und liefern somit einen Beitrag zur Klassifikation dieser Vertexoperatoralgebren, von denen es bis auf Isomorphie vermutlich genau 71 gibt, die sogenannte Liste von Schellekens.
\end{otherlanguage}

\chapter*{Acknowledgements}

First and foremost I would like to express my gratitude to Prof.\ Nils Scheithauer for his inspiration and excellent supervision. I thank Dr.\ Jethro van Ekeren for being a great colleague and many stimulating discussions. I enjoyed working together with both of them on the paper \cite{EMS15}, which is partly based on this dissertation and vice versa.

I thank the referees of this dissertation Prof.\ Nils Scheithauer, Prof.\ Martin Möller and, in particular, Prof.\ Gerald Höhn, who has made a number of useful suggestions concerning this text.

I am grateful to Prof.\ Gerald Höhn and Prof.\ Thomas Creutzig for many helpful discussions. I thank Dr.\ Stephan Ehlen for advice concerning some computer calculations.

I also thank Prof.\ Péter Bántay, Prof.\ Scott Carnahan, Prof.\ Thomas Creutzig, Prof.\ Gerald Höhn, Prof.\ Yi-Zhi Huang, Prof.\ Victor Kac, Prof.\ Ching Hung Lam, Prof.\ Masahiko Miyamoto, Prof.\ Bert Schellekens, Prof.\ Hiroki Shimakura and Prof.\ Hiroshi Yamauchi for comments on the paper \cite{EMS15}.

It is a pleasure to thank the organisers of the workshop ``Algebras, Groups and Geometries 2014'' at the University of Tokyo, where parts of this work were presented, and Prof.\ Hiroshi Yamauchi for his hospitality during our stay.

I thank Heiko Möller, Daniel Günzel, Dr.\ Moritz Egert and Dr.\ Friederike Steglich for proofreading parts of this text.

The author was partially supported by a doctoral scholarship from the German Academic Scholarship Foundation and by the German Research Foundation as part of the project ``Infinite-dimensional Lie algebras in string theory''.

\tableofcontents

\chapter*{Introduction}
\addcontentsline{toc}{chapter}{Introduction}

It has been three decades since the introduction of \emph{vertex algebras} into mathematics by Richard Borcherds \cite{Bor86}. Today---although still deemed quite exotic---these objects have become rather ubiquitous. They are essential for the representation theory of infinite-dimensional Lie algebras and bear influence on algebraic geometry, the theory of finite groups, topology, integrable systems and combinatorics, to name only a few subjects. Moreover, vertex algebras serve as a rigorous mathematical formulation of the chiral part of two-dimensional conformal field theories in physics. In a nutshell, vertex algebras are to quantum field theory what associative algebras of operators on Hilbert spaces are to quantum mechanics.

The main focus of this text is on special vertex algebras called \emph{\voa{}s}, carrying a representation of the Virasoro Lie algebra at a certain central charge $c\in\C$ \cite{FLM88}. The original motivation for introducing the notion of \voa{}s was to realise the largest sporadic group, the Monster group $M$, as the group of symmetries of a certain infinite-dimensional, $\Z$-graded vector space
\begin{equation*}
V=\bigoplus_{n\in\N}V_n
\end{equation*}
with a ``natural'' algebraic structure. This natural structure turned out to be that of a \voa{} and the infinite-dimensional representation of the Monster group is the famous \emph{Monster \voa{}} or \emph{Moonshine module} $V^\natural$ \cite{FLM88}. In his ground-breaking article \cite{Bor92} Borcherds proved the Conway-Norton conjecture for the Moonshine module \cite{CN79}, i.e.\ he showed that for any element $g\in M$ of the Monster group the graded trace
\begin{equation*}
T_g(q)=\sum_{n\in\Z}q^{n-1}\tr_{V^\natural_n}g,
\end{equation*}
called \emph{McKay-Thompson series}, is the $q$-expansion of a Hauptmodul for a genus 0 subgroup of $\SLR$. Borcherds was awarded the Fields Medal in 1998 in part for his proof of this conjecture.

The graded traces or \emph{characters} of \voa{}s play a decisive rôle in the representation theory of \voa{}s and also in this thesis as they are sometimes modular forms \cite{Zhu96,DLM00}. For example, the character of the Moonshine module $V^\natural$
\begin{equation*}
\ch_{V^\natural}(q)=T_{\id}(q)=\sum_{n\in\Z}q^{n-1}\dim_\C(V^\natural_n)=q^{-1}+196884q+21493760q^2+\ldots
\end{equation*}
is exactly the \emph{modular $j$-function} minus its constant term. This explains why the coefficients of the $j$-function are linear combinations of the dimensions of the irreducible representations of the Monster group $M$ with small non-negative coefficients, as was remarked by John McKay in 1978. His observation started the \emph{Monstrous Moonshine} research which climaxed in Borcherds' proof.

\minisec{Nice \VOA{}s}
\Voa{}s can exhibit numerous pathological behaviours, making it difficult to explore them beyond basic properties. Consequently, it is useful to restrict oneself to \voa{}s satisfying certain \emph{niceness} properties, especially when studying their representation theory. In the following we call a \voa{} \emph{nice} if it is simple, rational, $C_2$-cofinite, self-contragredient and of CFT-type (Assumption~\ref{ass:n}).

For instance, a \emph{rational} \voa{} has only finitely many isomorphism classes of irreducible modules and each module is isomorphic to a direct sum of these, i.e.\ the module category is semisimple. \emph{$C_2$-cofiniteness} is a technical condition ensuring that the graded traces of \voa{}s are modular forms or, more precisely, vector-valued modular forms for Zhu's representation \cite{Zhu96}. Many natural examples of \voa{}s are nice, including the aforementioned Moonshine module $V^\natural$ and \voa{}s associated with positive-definite, even lattices. Furthermore, a rational \voa{} $V$ is called \emph{holomorphic} if $V$ itself is the only irreducible $V$-module, i.e.\ the representation theory of $V$ is trivial.

\minisec{Important Problems}
While there are many well-understood examples of vertex algebras, the general theory is still rather poorly developed, despite some recent breakthroughs. Three of the main problems in the structure theory of \voa{}s are:
\begin{enumerate}
\item the \emph{extension problem}: constructing new \voa{}s from the irreducible modules of a given nice \voa{}, extending the original \voa{} structure,
\item the \emph{orbifold problem}: given a nice \voa{} $V$ and a finite group $G$ of automorphisms of $V$, understanding the properties and representation theory of the \fpvosa{} $V^G$,
\item the \emph{classification problem}: classifying the nice, holomorphic \voa{}s of a given central charge $c$.
\end{enumerate}
Solutions to those problems will be addressed in this thesis.

\minisec{Simple Currents}
The modules of a nice \voa{} admit a tensor product $\boxtimes$, called the \emph{fusion product} \cite{HL92,HL94,HL95,HL95b,Li98}. An irreducible module is a \emph{simple current} if its fusion product with any irreducible module is again irreducible. Nice \voa{}s whose irreducible modules are all simple currents have a particularly simple representation theory. We show that the fusion algebra $\V(V)$ of such a \voa{} $V$ is the group algebra $\C[F_V]$ of some finite abelian group $F_V$, called the \emph{fusion group} (Proposition~\ref{prop:sca}). In particular, the irreducible $V$-modules $W^\alpha$, $\alpha\in F_V$, are indexed by $F_V$ and
\begin{equation*}
W^\alpha\boxtimes W^\beta\cong W^{\alpha+\beta}
\end{equation*}
for all $\alpha,\beta\in F_V$ are the fusion rules for $V$. In this situation we say that $V$ has \emph{group-like fusion} (Assumption~\ref{ass:sn}).

Suppose in addition that the irreducible modules of $V$ other than $V$ itself have positive conformal weights, i.e.\ $V$ satisfies the \emph{positivity assumption} (Assumption~\ref{ass:p}). Then we prove that the modulo-1 reduction $Q_\rho$ of the \emph{conformal weights} is a non-degenerate quadratic form on $F_V$, i.e.\ it endows $F_V$ with the structure of a \fqs{} (Theorem~\ref{thm:fafqs}). Moreover, we show that Zhu's representation $\rho_V$ is up to a character identical to the well-known Weil representation $\rho_{F_V}$ on $F_V$ (Theorem~\ref{thm:zhuweil}). This character is closely related to the modular transformations of the Dedekind eta function (Corollary~\ref{cor:zhuweil}).

\minisec{Extension Problem}
Let $V$ be a nice \voa{} with group-like fusion. Then it is known that the direct sum of all irreducible $V$-modules
\begin{equation*}
A=\bigoplus_{\alpha\in F_V}W^\alpha
\end{equation*}
admits the structure of an \emph{\aia{}} associated with some abelian 3-cocycle $(F,\Omega)$ on $F_V$ \cite{Hua00,Hua05}. We prove, using the theory of \mtcs{}, that the quadratic form $Q_\Omega$ associated with $\Omega$ is exactly the negative of the quadratic form $Q_\rho$ if additionally the positivity assumption holds for $V$ (Theorem~\ref{thm:2.7}). This closes a gap in the theory pointed out by Scott Carnahan \cite{Car14}.

An important application of this result is a solution of the extension problem. We prove that in the above situation for an isotropic subgroup $I$ of $F_V$ the sum
\begin{equation*}
V_I=\bigoplus_{\alpha\in I}W^\alpha
\end{equation*}
admits the structure of a \voa{} naturally extending the \voa{} structure of $V$ and the module structures on the $W^\alpha$, $\alpha\in I$ (Theorem~\ref{thm:4.3}). $V_I$ is called a \emph{simple-current extension} of $V$ and we show that $V_I$ is holomorphic if and only if $I^\bot=I$.

\minisec{Orbifold Problem}
Then we turn to the orbifold problem. Let $V$ be a \voa{} and $G\leq\Aut(V)$ some finite group of automorphisms of $V$. Then the vectors in $V$ fixed by $G$ form a \vosa{} of $V$. \emph{Orbifold theory} is concerned with the properties and representation theory of the \fpvosa{} $V^G$. It is natural to ask whether $V^G$ inherits the niceness from $V$. It has recently been established that $V^G$ is again rational and $C_2$-cofinite for a finite, solvable group $G$ \cite{Miy15,CM16}. The other three niceness properties are easily seen to hold for any finite group $G$.

The fusion algebra of $V^G$ for nice $V$ is determined in this work in the simplest non-trivial case where $V$ is holomorphic and $G=\langle\sigma\rangle$ is a finite, cyclic group of some order $n\in\Ns$ (Assumption~\ref{ass:o}), generalising previous results for $n=2,3$ \cite{FLM88,DGM90,Miy13b}. We show that all $n^2$ irreducible $V^G$-modules are simple currents (Lemma~\ref{lem:step4}) and that the fusion group $F_{V^G}$ is a central extension of $\Z_n$ by $\Z_n$ whose isomorphism class is determined by the conformal weight of the unique $\sigma$-twisted $V$-module $V(\sigma)$ (Theorem~\ref{thm:main} and Corollary~\ref{cor:main}). We also determine the quadratic form $Q_\rho$ on $F_{V^G}$. We establish that the conformal weights of the irreducible $V^G$-modules lie in $(1/n^2)\Z$ (Theorem~\ref{thm:confwnn}) and we determine the level of the trace functions on the irreducible $V^G$-modules (Theorem~\ref{thm:modinv}), in both cases generalising results from \cite{DLM00}.

Combining the results from the theory of simple-current \voa{}s and orbifold theory, we prove the existence of an \emph{orbifold construction} of a nice, holomorphic \voa{} $\widetilde{V}$ starting from a nice, holomorphic \voa{} $V$ and some finite, cyclic group $G$ of automorphisms of $V$ such that $V^G$ satisfies the positivity assumption (Assumption~\ref{ass:op}). More precisely, we show that
\begin{equation*}
\widetilde{V}=\bigoplus_{\alpha\in I}W^\alpha
\end{equation*}
admits the structure of a nice, holomorphic \voa{} where $I$ is an isotropic subgroup of the fusion group $F_{V^G}$ of $V^G$ with $I^\bot=I$ (Theorem~\ref{thm:orbvoa}).

The orbifold construction is a powerful tool for constructing new holomorphic \voa{}s. For example, the aforementioned Moonshine module $V^\natural$ is an orbifold of the lattice \voa{} associated with the Leech lattice by a certain automorphism of order $n=2$ \cite{FLM88}.

Also note that Carnahan's recent proof of Norton's generalised Moonshine conjecture for the twisted modules of $V^\natural$ \cite{Car12b} relies on the orbifold results in this text.

\minisec{Classification Problem}
The main application of the orbifold construction developed in this dissertation is the classification of nice, holomorphic \voa{}s of central charge 24.

It follows from Zhu's modular invariance result \cite{Zhu96} that the central charge $c$ of a nice, holomorphic \voa{} is a positive multiple of 8. For $c=8,16$ all nice, holomorphic \voa{}s with this central charge are known up to isomorphism: they are exactly the lattice \voa{}s associated with the unimodular lattices $E_8$ of rank 8 and $E_8^2$ and $D_{16}^+$ of rank 16 \cite{DM04b}. Here, the classification of nice, holomorphic \voa{}s mirrors that of positive-definite, unimodular, even lattices.

Given a positive-definite, even lattice of rank $\rk(L)\in\N$, it is possible to associate with it a nice \voa{} $V_L$ of central charge $c=\rk(L)$. Its irreducible modules are indexed by the discriminant form $L'/L$ where $L'$ denotes the dual lattice. Hence, such a lattice \voa{} is holomorphic if and only if $L$ is unimodular, i.e.\ $L'=L$.

For $c=24$, the situation is more involved. Certainly, the 24 Niemeier lattices, i.e.\ the even, unimodular, positive-definite lattices of rank 24, yield nice, holomorphic \voa{}s of central charge 24. But also the Moonshine module $V^\natural$ belongs to this category.

It is well known that the weight-one space $V_1$ of any \voa{} $V$ carries the structure of a Lie algebra. In \cite{Sch93} the physicist Schellekens produced a list of 71 possible meromorphic conformal field theories of central charge $24$. It was recently proved in \cite{EMS15} that this is a rigorous theorem on \voa{}s: if $V$ is a nice, holomorphic \voa{} with central charge $c=24$, then $V_1=\{0\}$ (e.g.\ $V\cong V^\natural$), $V_1$ is 24-dimensional abelian or $V_1$ is one of the 69 semisimple Lie algebras given in Table~1 of \cite{Sch93}. Note that this result only lists the possible Lie algebra structures of $V_1$ but two \voa{}s with the same $V_1$ might be non-isomorphic as \voa{}s.

It is an ongoing effort to prove the existence of all 71 cases on Schellekens' list, i.e.\ that there is a \voa{} with the desired Lie algebra structure in its weight-one space. So far, all but one of the 71 cases have been constructed, namely 24 lattice \voa{}s associated with the Niemeier lattices \cite{Bor86,FLM88,Don93}, 15 $\Z_2$-orbifolds of those lattice \voa{}s \cite{DGM90}, including the Moonshine module $V^\natural$ \cite{FLM88}, 17 framed \voa{}s \cite{Lam11,LS12a,LS15a}, 3 lattice $\Z_3$-orbifolds \cite{Miy13b,SS16} and 6 orbifolds by inner automorphisms \cite{LS16a,LS16}. We contribute 5 new \voa{}s obtained as orbifolds of order 4, 5, 6 and 10 of lattice \voa{}s (Theorem~\ref{thm:newcases}). The results in \cite{LS16a,LS16} depend on the general orbifold theory developed in this text. In total, we obtain that for each Lie algebra on Schellekens' list, with the possible exception of $A_2F_4$, there exists a nice holomorphic \voa{} $V$ of central charge $24$ with this Lie algebra structure on $V_1$ (Theorem~\ref{thm:summary}). Recently, Ching Hung Lam and Xingjun Lin have announced they constructed the last remaining case $A_2F_4$ using mirror extensions \cite{LL16}.

Note that for $c\geq 32$ the classification of nice, holomorphic \voa{}s of this central charge is no longer feasible. For instance, there exist more than one billion even, unimodular, positive-definite lattices of rank 32 and consequently at least as many nice, holomorphic \voa{}s with $c=32$.

The main results of this dissertation described so far are also published in a condensed form as part of \cite{EMS15}.

\minisec{BRST Construction}

As a second application of the orbifold theory we present the natural construction of ten \BKMa{}s whose denominator identities are completely reflective automorphic products of singular weight, which were classified in \cite{Sch04b,Sch06}. \BKMa{}s are natural generalisations of Kac-Moody algebras introduced by Borcherds \cite{Bor88}. In \cite{Bor92}, Borcherds constructs a family of Borcherds-Kac-Moody (super)algebras by twisting the denominator identity of the Fake Monster Lie algebra by automorphisms of the Leech lattice $\Lambda$. As an open problem he asks for ``natural constructions'' of such Lie (super)algebras. A possible approach is the construction using the BRST quantisation. The BRST construction of a certain \BKMa{} also plays a central rôle in Borcherds' proof of the Moonshine conjecture.

In this thesis we present the BRST construction of ten \BKMa{}s corresponding to the elements of square-free order in the Mathieu group $M_{23}$, which acts on the Leech lattice (Theorem~\ref{thm:10brst}). The result depends on two technical conjectures, which we aim to prove in the future.

\section*{Outlook}

The main goal concerning the classification of \voa{}s is to prove the following conjecture:
\begin{conj*}
There are exactly 71 nice, holomorphic \voa{}s of central charge 24 up to isomorphism. They are uniquely determined by the Lie algebra structure of $V_1$ and the full list is given in Table~1 of \cite{Sch93}.
\end{conj*}
This conjecture includes as a special case the well-known conjecture in \cite{FLM88} that the Moonshine module $V^\natural$ is the unique holomorphic \voa{} of central charge 24 with $V_1=\{0\}$. 

Considering all cases on Schellekens' list as constructed, to complete this classification it remains to prove that the Lie algebra structure of $V_1$ determines the \voa{} structure of $V$ up to isomorphism.

The most promising approach to prove the uniqueness of the \voa{} structure for many of the Lie algebras on Schellekens' list is a case-by-case ansatz where the concept of the inverse orbifold developed in this text is used (Theorem~\ref{thm:invorb}). One attempts to reduce the uniqueness of the case studied to the uniqueness of the 24 cases with Lie algebras of rank 24, for which one can show that the corresponding \voa{} is isomorphic to a lattice \voa{} associated with one of the Niemeier lattices \cite{DM04b}. This approach was demonstrated recently for the Lie algebras $E_6G_2^3$, $A_2^6$ and $A_5D_4A_1^3$ in \cite{LS16b}.

The peculiar observation that summing up $\dim_\C(V_1)$ over all 71 cases on Schellekens' list yields $13824=24^3$ gives an aesthetic reason to believe that the above conjecture is correct.

Concerning the orbifold theory, the two main simplifying assumptions in this text are the holomorphicity of $V$ and the cyclicity of $G$. It is natural to pose the question what happens if these assumptions are relaxed: given a nice \voa{} $V$ and some finite group $G$ of automorphisms of $V$, determine the irreducible modules and the fusion algebra of $V^G$ for
\begin{enumerate}
\item $V$ not holomorphic,
\item $G$ not cyclic.
\end{enumerate}
One part of the above problem, namely the classification of the irreducible modules of $V^G$, can be answered using results in \cite{DLM00} and \cite{MT04} under the assumption that $V^G$ is rational and $C_2$-cofinite. In this case, using a simple argument due to Miyamoto, one can show that each irreducible $V^G$-module occurs as a $V^G$-submodule of the irreducible $g$-twisted $V$-module for some $g\in G$. The problem of classifying all $V^G$-modules is hence equivalent to that of classifying the $g$-twisted modules of $V$ for all $g\in G$. A classification of twisted modules for $V$ is known, for instance, if $V$ is holomorphic \cite{DLM00} or if $V$ is a lattice \voa{} \cite{BK04}.

The determination of the fusion rules between the irreducible $V^G$-modules, however, is in this generality still an open problem. For example, if $V$ is not holomorphic, the irreducible modules of $V^G$ need not be simple currents even for cyclic $G$ \cite{DRX15}.

The easiest generalisation of the results in this text is probably the determination of the fusion algebra of $V_L^G$ for a lattice \voa{} $V_L$ associated with some even, positive-definite, non-unimodular lattice $L$ and a finite, cyclic group of automorphisms of $V_L$ arising from lattice automorphisms of $L$ using the results from \cite{BK04}.

\part{Orbifold Theory}

\chapter{Preliminaries}

In this chapter, we collect important definitions and well-known properties of \voa{}s and related concepts that we will use in the subsequent chapters. We also set up most of the notation. The experienced reader may skip large parts of this chapter.

\section{Formal Calculus}\label{sec:formal}

An introduction into the calculus of formal series can be found in \cite{FHL93}, Section~2.1, and \cite{LL04}, Chapter~2. We consider formal series with coefficients in some $\C$-vector space $V$. The symbols $x,y,x_0,x_1,x_2,\ldots$ represent commuting \emph{formal} variables. Sometimes, the formal power series are viewed as functions in one or more complex variable by substituting complex numbers for the formal variables. We denote these complex variables by $z,w,z_0,z_1,z_2,\ldots\in\C$.

\begin{defi}[Formal Series]
Let $V$ be a $\C$-vector space. We define the following spaces of \emph{formal series} over $V$:
\begin{enumerate}
\item\label{enum:series1} the vector space of \emph{doubly infinite formal Laurent series} over V
\begin{equation*}
V[[x,x^{-1}]]:=\left\{\sum_{n\in\Z}v_nx^n\xmiddle|v_n\in V\right\},
\end{equation*}
\item\label{enum:series2} the ring of \emph{(truncated) formal Laurent series} over V
\begin{equation*}
V((x)):=\left\{\sum_{n\in\Z}v_nx^n\xmiddle|v_n\in V\text{, $v_n=0$ for sufficiently negative $n$}\right\},
\end{equation*}
\item\label{enum:series3} the ring of \emph{formal power series} over V
\begin{equation*}
V[[x]]:=\left\{\sum_{n\in\N}v_nx^n\xmiddle|v_n\in V\right\},
\end{equation*}
\item\label{enum:series4} the ring of \emph{Laurent polynomials} over V
\begin{equation*}
V[x,x^{-1}]:=\left\{\sum_{n\in\Z}v_nx^n\xmiddle|v_n\in V\text{, $v_n=0$ for almost all $n$}\right\},
\end{equation*}
\item\label{enum:series5} the ring of \emph{polynomials} over V
\begin{equation*}
V[x]:=\left\{\sum_{n\in\N}v_nx^n\xmiddle|v_n\in V\text{, $v_n=0$ for almost all $n$}\right\},
\end{equation*}
\item\label{enum:series6} spaces of formal series with fractional exponents (with a finite common denominator), obtained by inserting $x^{1/d}$ for $x$ for some $d\in\Ns$ in the definitions of items \ref{enum:series1} to \ref{enum:series5},
\item\label{enum:series7} the vector space of \emph{formal series with powers in $\C$} over $V$
\begin{equation*}
V\{x\}:=\left\{\sum_{n\in\C}v_nx^n\xmiddle|v_n\in V\right\}.
\end{equation*}
\end{enumerate}
\end{defi}
It is clear how to generalise these notions to the case of several commuting formal variables $x_1,\ldots,x_r$. Care is to be exercised in taking products of two formal series or substituting variables in multi-variable formal series. The resulting expressions are not always well-defined.

For a formal series $f(x)\in V\{x\}$ we denote by
\begin{equation*}
\left[f(x)\right](n)
\end{equation*}
the coefficient of $x^n$ in $f(x)$ for $n\in\C$. In particular, we define the residue
\begin{equation*}
\Res_x(f):=\left[f(x)\right](-1)
\end{equation*}
as the coefficient of $x^{-1}$.

\minisec{Functions in Complex Variables}
When passing from a formal series with non-integer exponents to a complex function we need a choice of the branch of the complex logarithm. More specifically, we interpret an arbitrary power $z^n$ for $n\in\C$ of the complex variable $z$ as
\begin{equation*}
z^n = \e^{n\log(z)}
\end{equation*}
for $z\in\C^\times$ where $\log(z) = \log(|z|) + \i\arg(z)$ and $0\leq\arg(z)<2\pi$.
Hence, for a formal series $f(x)=\sum_{n\in\C}v_nx^n$ we obtain
\begin{equation*}
f(z)=\sum_{n\in\C}v_n\e^{n\log(z)},
\end{equation*}
which might or might not be a convergent sum for a given $z\in\C^\times$. We also use the notation
\begin{equation*}
f(\e^\zeta):=\sum_{n\in\C}v_n\e^{n\zeta}
\end{equation*}
for $\zeta\in\C$, which is in conflict with the one above since it is possible that $f(z)\neq f(\e^\zeta)$ even though $z=\e^\zeta$. However, no confusion should arise from this. Finally, given the formal series $f(x)$ we obtain a new formal series $f(\e^\zeta x)$ defined by
\begin{equation*}
f(\e^\zeta x):=\sum_{n\in\C}v_n\e^{n\zeta}x^n.
\end{equation*}

\minisec{Laurent Series Expansion}

Throughout this work we will need to give meaning to expressions in two formal variables of the form $(x\pm y)^n$ for some $n\in\Z$. To this end, we consider the binomial expansion
\begin{equation*}
(x\pm y)^n\stackrel{\iota_{x,y}}{\longmapsto}\sum_{m\in\N}(\pm1)^m\binom{n}{m}y^mx^{n-m},
\end{equation*}
i.e.\ the Laurent series expansion as a formal series with non-negative powers of $y$ and arbitrarily negative powers of $x$. Replacing $x$ and $y$ with the complex variables $z$ and $w$, respectively, the above infinite series converges in the domain $|z|>|w|>0$ to the corresponding term on the left-hand side.

Consider the field of fractions $\operatorname{Quot}(\C[[x,y]])$ of the integral domain $\C[[x,y]]$. Therein lies the subalgebra $\C((x,y))[(x\pm y)^{-1}]=\C[[x,y]][x^{-1},y^{-1},(x\pm y)^{-1}]$ where $\C((x,y))$ denotes the formal Laurent series with powers of $x$ and $y$ both bounded from below. Then the binomial expansion defines a linear embedding
\begin{equation*}
\iota_{x,y}:\C((x,y))[(x\pm y)^{-1}]\hookrightarrow\C((x))((y))\subset\C[[x^{\pm1},y^{\pm1}]]
\end{equation*}
into the Laurent series in $y$ whose coefficients are Laurent series in $x$. These may have arbitrarily negative powers in $x$ but involve only finitely many negative powers of $y$. Similarly, there is a Laurent series expansion
\begin{equation*}
\iota_{y,x}:\C((x,y))[(x\pm y)^{-1}]\hookrightarrow\C((y))((x))\subset\C[[x^{\pm1},y^{\pm1}]]
\end{equation*}
corresponding to an expansion in the domain $|w|>|z|>0$ upon replacing $x$ and $y$ with the complex variables $z$ and $w$, respectively. Specifically
\begin{equation*}
(x\pm y)^n\stackrel{\iota_{y,x}}{\longmapsto}\sum_{m\in\N}(\pm1)^{n-m}\binom{n}{m}x^my^{n-m}.
\end{equation*}
The situation is summarised in the following \emph{non-commutative} diagram:
\begin{equation*}
\begin{tikzcd}
{}& \C((x,y))[(x\pm y)^{-1}] \arrow[hook]{dl}{\iota_{x,y}} \arrow[hook]{dr}{\iota_{y,x}} \\
\C((x))((y)) \arrow[hook]{dr} &&\C((y))((x)) \arrow[hook]{dl} \\
&\C[[x^{\pm1},y^{\pm1}]]
\end{tikzcd}
\end{equation*}
Note that for example $\iota_{x,y}(x\pm y)^n=\iota_{y,x}(x\pm y)^n$ if and only if $n\in\N$ and in this case both expressions are just polynomials in $x$ and $y$.

Many authors do not write down the embeddings $\iota_{x,y}$ and $\iota_{y,x}$ explicitly and always assume that the binomial expression $(x\pm y)^n$ is to be \emph{expanded in non-negative, integral powers of the second variable}. We will not make this assumption in this text.

\minisec{Binomials with Non-Integral Powers}
Note that the binomial expansion is not restricted to integral powers. For $n\in\C$ we can again consider the expansion
\begin{equation*}
\iota_{x,y}(x\pm y)^n=\sum_{m\in\N}(\pm1)^m\binom{n}{m}y^mx^{n-m},
\end{equation*}
which is an expansion in non-negative, integral powers of $y$ and complex powers of $x$ with arbitrarily negative real part.

A complication occurs for the expansion $\iota_{y,x}(x-y)^n$ since we have to deal with non-integral powers of $-1$. For consistency with \cite{DL93} we shall read $x-y$ as a shorthand notation for $x+\e^{-\pi\i}y$. Then the expansion $\iota_{y,x}$ gives
\begin{align}\begin{split}\label{eq:dl93}
\iota_{y,x}(x-y)^n&=\iota_{y,x}(x+\e^{-\pi\i}y)=\sum_{m\in\N}(\e^{-\pi\i})^{n-m}\binom{n}{m}x^my^{n-m}\\
&=\e^{-\pi\i n}\sum_{m\in\N}(-1)^{m}\binom{n}{m}x^my^{n-m}=\e^{-\pi\i n}\iota_{y,x}(y-x)^n.
\end{split}\end{align}

\minisec{Formal Delta Function}
An important formal series is the \emph{formal delta function} defined by
\begin{equation*}
\delta(x):=\sum_{n\in\Z}x^n\in\C[[x,x^{-1}]].
\end{equation*}
This formal series plays the rôle of the delta distribution at $x=1$, i.e.\
\begin{equation*}
f(x)\delta(x)=f(1)\delta(x)
\end{equation*}
for any formal Laurent polynomial $f(x)\in V[x^{\pm1}]$. Also, consider a formal series in two variables $f(x,y)\in V[[x^{\pm1},y^{\pm1}]]$ such that $f(x,x)$ exists, substituting $x$ for $y$ in $f(x,y)$. Then
\begin{equation*}
f(x,y)\delta(x/y)=f(x,x)\delta(x/y)=f(y,y)\delta(x/y)
\end{equation*}
in $V[[x^{\pm1},y^{\pm1}]]$ and in particular all three expressions exist.

In the various versions of the Jacobi identity (see e.g.\ Definition~\ref{defi:borid}) we will encounter three-variable expressions of the form $\delta\left((x_1-x_2)/x_0\right)$, which are usually expanded in non-negative powers of the second variable of the numerator, i.e.\
\begin{equation*}
\iota_{x_1,x_2}\delta\left(\frac{x_1-x_2}{x_0}\right)=\sum_{n\in\Z}\frac{\iota_{x_1,x_2}(x_1-x_2)^n}{x_0^n}=\sum_{m\in\N,n\in\Z}(-1)^m \binom{n}{m}x_0^{-n}x_1^{n-m}x_2^m,
\end{equation*}
using the binomial theorem. The following property holds (cf.\ the Jacobi identity below):
\begin{equation*}
\iota_{x_1,x_0}x_2^{-1}\delta\left(\frac{x_1-x_0}{x_2}\right)=\iota_{x_1,x_2}x_0^{-1}\delta\left(\frac{x_1-x_2}{x_0}\right)-\iota_{x_2,x_1}x_0^{-1}\delta\left(\frac{x_2-x_1}{-x_0}\right),
\end{equation*}
where the three terms are formal power series in non-negative, integral powers of $x_0$, $x_2$ and $x_1$, respectively.

\minisec{Formal $q$-Series}

When defining trace functions and characters we will encounter formal Laurent series $f(q)\in V((q))$ or $f(q)\in V((q^{1/n}))$ for some $n\in\Ns$ in the formal variable $q$. We view them as functions in $\tau$ on the complex upper half-plane $\H:=\{z\in\C\;|\;\Im(z)>0\}$ by replacing $q$ with $q_\tau=\e^{(2\pi\i)\tau}$ if $f(q_\tau)$ converges for $|q_\tau|<1$. We use the convention that $q_\tau^{1/n}=\e^{(2\pi\i)\tau/n}$. By an abuse of notation we write $f(\tau)$ for $f(q_\tau)$.

\section{\VOA{}s}

The concept of vertex algebras was first introduced by Borcherds in \cite{Bor86}. We assume the ground field to be the complex numbers, i.e.\ all vector spaces are over $\C$.

\minisec{Vertex Algebras}

The following definition of a vertex algebra is from \cite{Kac98}:
\begin{defi}[Vertex Algebra]
A \emph{vertex algebra} is given by the following data:
\begin{itemize}
\item (\emph{space of states}) a vector space $V$,
\item (\emph{vacuum vector}) a non-zero vector $\vac\in V$,
\item (\emph{translation operator}) a linear operator $T\colon V\to V$,
\item (\emph{vertex operators} or \emph{state-field correspondence}) a linear map
\begin{equation*}
Y(\cdot,x)\colon V\to\End_\C(V)[[x^{\pm1}]]
\end{equation*}
taking each $a\in V$ to a \emph{field}
\begin{equation*}
a\mapsto Y(a,x)=\sum_{n\in\Z}a_nx^{-n-1}
\end{equation*}
where for each $v\in V$, $a_nv=0$ for $n$ sufficiently large. Equivalently we can view $Y(\cdot,x)$ as a map $V\otimes_\C V\to V((x))$.
\end{itemize}

These data are subject to the following axioms:
\begin{itemize}
\item (\emph{left vacuum axiom}) $Y(\vac,x)=\id_Vx^0=\id_V$. In other words: $\vac_n=\delta_{n,-1}\id_V$.
\item (\emph{right vacuum axiom}) For any $a\in V$, $Y(a,x)\vac\in V[[x]]$ so that $Y(a,z)\vac$ has a well-defined value at $z=0$ and
\begin{equation*}
Y(a,z)\vac|_{z=0}=a.
\end{equation*}
In other words: $a_n\vac=0$ for $n\geq 0$ and $a_{-1}\vac=a$.
\item (\emph{translation axiom}) For any $a\in V$,
\begin{equation*}
[T,Y(a,x)]=\partial_x Y(a,x)
\end{equation*}
and $T\vac=0$.
\item (\emph{locality axiom}) All fields $Y(a,x)$, $a\in V$, are \emph{local} with respect to each other, i.e.\ for each $a,b\in V$ there is an $N\in\N$ such that
\begin{equation*}
(x-y)^N[Y(a,x),Y(b,y)]=0
\end{equation*}
as formal power series in $\End_\C(V)[[x^{\pm1},y^{\pm1}]]$.
\end{itemize}
\end{defi}
Note that it already follows from the translation axiom together with the right vacuum axiom that the operator $T$ is given by $Ta=a_{-2}\vac$, $a\in V$. So, we would not have needed to introduce $T$ as independent datum.

Under the presence of the other axioms the locality axiom may be equivalently replaced by one of the following three axioms:
\begin{defi}[Borcherds and Jacobi Identity]\label{defi:borid}
\item
\begin{itemize}
\item (\emph{Borcherds identity, original version}) For $a,b\in V$,
\begin{equation*}
(a_n b)_m = \sum_{j=0}^\infty (-1)^j \binom{n}{j} \left(a_{n-j} b_{m+j} - (-1)^n b_{n+m-j} a_j\right)
\end{equation*}
for all $m,n\in\Z$, where the sum becomes finite when applied to any vector in $V$.
\item (\emph{Borcherds identity, modern version}) For $a,b\in V$,
\begin{equation*}
\sum_{j=0}^\infty\binom{m}{j}(a_{n+j}b)_{m+k-j} = \sum_{j=0}^\infty \binom{n}{j} \left((-1)^ja_{m+n-j}b_{k+j} - (-1)^{j+n} b_{k+n-j} a_{m+j}\right)
\end{equation*}
for all $k,m,n\in\Z$, where the sum on the right-hand side becomes finite when applied to any vector in $V$. 
\item (\emph{Jacobi identity}) For $a,b\in V$,
\begin{align*}
&\iota_{x_1,x_0}x_2^{-1}\delta\left(\frac{x_1-x_0}{x_2}\right)Y(Y(a,x_0)b,x_2)\\
&=\iota_{x_1,x_2}x_0^{-1}\delta\left(\frac{x_1-x_2}{x_0}\right)Y(a,x_1)Y(b,x_2)-\iota_{x_2,x_1}x_0^{-1}\delta\left(\frac{x_2-x_1}{-x_0}\right)Y(b,x_2)Y(a,x_1).
\end{align*}
When applied to any element of $V$, the coefficient of each monomial in the formal variables is only a finite sum.
\end{itemize}
\end{defi}
The first identity was used in the original definition by Borcherds in \cite{Bor86}. Borcherds later gave an equivalent definition using the second identity \cite{Bor92}, which is today usually referred to as the Borcherds identity. It is shown in \cite{Kac98} that the definition using locality is equivalent to using the Borcherds identity. Frenkel, Lepowsky and Meurman \cite{FLM88} used the Jacobi identity instead. It is shown for example in \cite{LL04} that this is also equivalent to the Borcherds identity.

We remarked above that the operator $T$ fulfilling the translation axiom is given by $Ta=a_{-2}\vac$ for $a\in V$. If we introduce the operator $T$ like this, the translation axiom will already follow from the other axioms of a vertex algebra if we use the definition including the Jacobi identity. In fact, by \cite{LL04}, Proposition~3.1.21, $[T,Y(a,x)]=Y(Ta,x)=\partial_x Y(a,x)$.

\minisec{Graded Vertex Algebras}
Often, one also introduces a weight $\Z$-grading on the vertex algebra:
\begin{defi}[Graded Vertex Algebra]\label{defi:gva}
A \emph{graded vertex algebra} $V$ is a vertex algebra with:
\begin{itemize}
\item (weight \emph{$\Z$-grading}) Let
\begin{equation*}
V=\bigoplus_{n\in\Z}V_n
\end{equation*}
with \emph{weights} $\wt(v)=n$ for $v\in V_n$. This grading is bounded from below, i.e.\ $V_n=\{0\}$ for sufficiently small $n$ and
\begin{equation*}
\dim_\C(V_n)<\infty
\end{equation*}
for all $n\in\Z$. Furthermore, let $\vac\in V_0$, $T$ be an operator of weight 1 and $a\in V$ be mapped to a field $Y(a,x)=\sum_{n\in\Z}a_nx^{-n-1}$ with $\wt(a_n)=\wt(a)-n-1$ for homogeneous $a$.
\end{itemize}
\end{defi}
For a graded vertex algebra the property that $a_nv=0$ for $a,v\in V$ and $n$ sufficiently large already follows from the boundedness from below of the weight grading on $V$ and the above formula for the weight of $a_n$.

The above system of axioms together with the grading condition is the definition of a vertex algebra as presented in
\cite{FKRW95} with the only exception that their grading begins at $0$ rather than at some possibly negative integer. The systems of axioms for a vertex algebra in \cite{Kac98} and \cite{FKRW95} are inspired by \cite{God89}. These axioms are essentially equivalent to Borcherds' original axioms as is shown in \cite{Kac98}.

The following weakening of the definition of graded vertex algebras will be called weak graded vertex algebras in this text:
\begin{defi}[Weak Graded Vertex Algebra]\label{defi:wgva}
A \emph{weak graded vertex algebras} is defined like a graded vertex algebra but the graded components are not required to be finite-dimensional and the grading does not have to be bounded from below.
\end{defi}
Many elementary properties of graded vertex algebras are still true for this class of vertex algebras.

By definition, graded vertex algebras have a $\Z$-grading that is bounded from below. Quite often, we will require the grading to take values in $\N$ only. This is included in the following definition:
\begin{defi}[CFT-Type]
A graded vertex algebra $V$ is said to be of \emph{CFT-type} if $V=\bigoplus_{n=0}^\infty V_n$ and $\dim_\C(V_0)=1$, i.e.\ the grading is non-negative and $V_0$ is spanned by the vacuum vector $\vac$.
\end{defi}

\minisec{\VOA{}s}

A very important special case of vertex algebras are \emph{\voa{}s}, introduced in \cite{FLM88}. \Voa{}s are graded vertex algebras carrying a representation of the Virasoro algebra that also induces the grading of the vertex algebra.
\begin{defi}[\VOA{}]
Let $V$ be a graded vertex algebra. $V$ is called a \emph{\voa{}} of \emph{central charge} (or \emph{rank}) $c\in\C$ if additionally the following datum is present:
\begin{itemize}
\item (\emph{conformal vector}) a non-zero vector $\omega\in V_2$,
\end{itemize}
subject to the axiom:
\begin{itemize}
\item (\emph{Virasoro relations}) The modes $L_n:=\omega_{n+1}$ of
\begin{equation*}
Y(\omega,x)=\sum_{n\in\Z}\omega_nx^{-n-1}=\sum_{n\in\Z}L_nx^{-n-2}
\end{equation*}
satisfy the \emph{Virasoro relations} at central charge $c$, i.e.\
\begin{equation*}
[L_m,L_n]=(m-n)L_{m+n}+\frac{m^3-m}{12}\delta_{m+n,0}\id_Vc
\end{equation*}
for $m,n\in\Z$. Moreover, $L_{-1}=T$ and $L_0v=\wt(v)v$ for homogeneous $v\in V$.
\end{itemize}
\end{defi}

Since \voa{}s will be the main object of interest in this text, let us at this point write down the complete definition of a \voa{} using the Jacobi identity. This is exactly the definition given in \cite{FLM88} and \cite{FHL93}. Using the Jacobi identity will be useful in order to define modules for \voa{}s in a natural way starting from the definition of \voa{}s.
\begin{defi}[\VOA{}]\label{defi:voa}
A \emph{\voa{}} of \emph{central charge} (or \emph{rank}) $c\in\C$ is given by the following data:
\begin{itemize}
\item (\emph{space of states}) a $\Z$-graded vector space
\begin{equation*}
V=\bigoplus_{n\in\Z}V_n
\end{equation*}
with \emph{weight} $\wt(v)=n$ for $v\in V_n$, $V_n=\{0\}$ for sufficiently small $n$ and $\dim_\C(V_n)<\infty$ for all $n\in\Z$,
\item (\emph{vacuum vector}) a non-zero vector $\vac\in V_0$,
\item (\emph{conformal vector}) a non-zero vector $\omega\in V_2$,
\item (\emph{translation operator}) a linear operator $T\colon V\to V$ of weight 1,
\item (\emph{vertex operators} or \emph{state-field correspondence}) a linear map
\begin{equation*}
Y(\cdot,x)\colon V\to\End_\C(V)[[x^{\pm1}]]
\end{equation*}
taking each $a\in V$ to a \emph{field}
\begin{equation*}
a\mapsto Y(a,x)=\sum_{n\in\Z}a_nx^{-n-1}
\end{equation*}
where for each $v\in V$, $a_nv=0$ for $n$ sufficiently large or equivalently a map $Y(\cdot,x)\colon V\otimes_\C V\to V((x))$. If $a\in V$ is homogeneous, then $\wt(a_n)=\wt(a)-n-1$ for all $n\in\Z$.
\end{itemize}

These data are subject to the following axioms:
\begin{itemize}
\item (\emph{vacuum axiom}) $Y(\vac,x)=\id_V$ and $Y(a,z)\vac|_{z=0}=a$ for all $a\in V$.
\item (\emph{translation axiom}) For any $a\in V$,
\begin{equation*}
[T,Y(a,x)]=\partial_x Y(a,x)
\end{equation*}
and $T\vac=0$.
\item (\emph{Jacobi identity}) For $a,b\in V$,
\begin{align*}
&\iota_{x_1,x_0}x_2^{-1}\delta\left(\frac{x_1-x_0}{x_2}\right)Y(Y(a,x_0)b,x_2)\\
&=\iota_{x_1,x_2}x_0^{-1}\delta\left(\frac{x_1-x_2}{x_0}\right)Y(a,x_1)Y(b,x_2)-\iota_{x_2,x_1}x_0^{-1}\delta\left(\frac{x_2-x_1}{-x_0}\right)Y(b,x_2)Y(a,x_1).
\end{align*}
\item (\emph{Virasoro relations}) The modes $L_n:=\omega_{n+1}$ of
\begin{equation*}
Y(\omega,x)=\sum_{n\in\Z}\omega_nx^{-n-1}=\sum_{n\in\Z}L_nx^{-n-2}
\end{equation*}
satisfy the \emph{Virasoro relations} at central charge $c$, i.e.\
\begin{equation*}
[L_m,L_n]=(m-n)L_{m+n}+\frac{m^3-m}{12}\delta_{m+n,0}\id_Vc
\end{equation*}
for $m,n\in\Z$. Moreover, $L_{-1}=T$ and $L_0v=\wt(v)v$ for homogeneous $v\in V$.
\end{itemize}
\end{defi}

The Jacobi identity implies that $[a_0,Y(b,x)]=Y(a_0b,x)$ for any $a,b\in V$ and hence since $T=L_{-1}=\omega_0$, one obtains that $[T,Y(a,x)]=Y(Ta,x)$. This means that the translation axiom $[T,Y(a,x)]=\partial_x Y(a,x)$ may be equivalently written as $Y(Ta,x)=\partial_x Y(a,x)$. This is the way the definition is presented in \cite{FLM88,FHL93}, for instance.

\begin{defi}[Weak \VOA{}]
A \emph{weak \voa{}} is defined like a \voa{} but we do not require the graded components to be finite-dimensional and the grading does not have to be bounded from below.
\end{defi}

\minisec{Elementary Categorical Notions}
The following are standard definitions for (weak) (graded) vertex (operator) algebras. See for example \cite{LL04}, Section~3.9, or \cite{FHL93}, Section~2.4, for details.
\begin{defi}[Vertex Algebra Homomorphism]\label{defi:hom}
A \emph{homomorphism of vertex algebras} is a linear map $f\colon V^1\to V^2$ between two vertex algebras $(V^1,Y_1,\vac_1,T_1)$ and $(V^2,Y_2,\vac_2,T_2)$ such that
\begin{equation*}
f(Y_1(a,x)b)=Y_2(f(a),x)f(b)
\end{equation*}
for all $a,b\in V^1$ and additionally
\begin{equation*}
f(\vac_1)=\vac_2.
\end{equation*}
From $Ta=a_{-2}\vac$ it follows that $f(T_1a)=T_2f(a)$ for all $a\in V^1$, i.e.\ $f$ intertwines the translation operators.

For (weak) graded vertex algebras we additionally demand that $f$ be grade-preserving, which means that
\begin{equation*}
\wt(f(a))=\wt(a)
\end{equation*}
for all homogeneous $a\in V^1$.

Let $(V^1,Y_1,\vac_1,\omega_1)$ and $(V^2,Y_2,\vac_2,\omega_2)$ be two (weak) \voa{}s. A \emph{homomorphism of (weak) \voa{}s} $f\colon V^1\to V^2$ is a linear map such that
\begin{equation*}
f(Y_1(a,x)b)=Y_2(f(a),x)f(b)
\end{equation*}
for all $a,b\in V^1$ and additionally
\begin{equation*}
f(\vac_1)=\vac_2\quad\text{and}\quad f(\omega_1)=\omega_2.
\end{equation*}
The last condition implies that the homomorphism $f$ is grade-preserving and intertwines the translation operators. Moreover, it follows directly from the definition of homomorphism that $V^1$ and $V^2$ have the same central charge.
\end{defi}

The notions of \emph{isomorphism}, \emph{endomorphism} and \emph{automorphism} are defined in the obvious way.

\begin{defi}[Vertex Subalgebra]\label{defi:subvoa}
Let $(V,Y,\vac,T)$ be a vertex algebra. A \emph{vertex subalgebra} of $V$ is a $T$-invariant subspace $U$ such that $\vac\in U$ and $Y(a,x)b\in U((x))$ for $a,b\in U$.

For a (weak) graded vertex algebra we additionally demand that $U$ be a graded subspace of $V$.

Let $(V,Y,\vac,\omega)$ be a (weak) \voa{}. A subspace $U$ of $V$ is called a \emph{\vosa{}} of $V$ if $\vac\in U$, $\omega\in U$ (hence $U$ is a graded subspace and $T$-invariant) and $Y(a,x)b\in U((x))$ for $a,b\in U$.
\end{defi}
Subalgebras of a (weak) (graded) vertex (operator) algebra $V$ are exactly the images of homomorphisms into $V$.

Sometimes the condition that $U$ and $V$ have the same Virasoro vector is not included in the definition of subalgebra $U$ for a (weak) \voa{} $V$. In that setting subalgebras as defined above are called \emph{full} \vosa{}s. Non-full \vosa{}s will only play a rôle in Chapter~\ref{ch:BRST}, where we will simply call them subalgebras and subalgebras in the above sense will be called full subalgebras.

\begin{defi}[Vertex Algebra Ideal]
Let $V$ be a vertex algebra. An \emph{ideal} $I$ of $V$ is a subspace $I$ such that
\begin{equation*}
Y(a,x)b\in I((x))\quad\text{and}\quad Y(b,x)a\in I((x))
\end{equation*}
for all $a\in V$ and all $b\in I$. It follows from $Ta=a_{-2}\vac$ for all $a\in V$ that $I$ is $T$-invariant.

For a (weak) graded vertex algebra we additionally demand that $I$ be a graded subspace of $V$.

Let $V$ be a (weak) \voa{}. An \emph{ideal} $I$ of $V$ is a subspace $I$ such that
\begin{equation*}
Y(a,x)b\in I((x))\quad\text{and}\quad Y(b,x)a\in I((x))
\end{equation*}
for all $a\in V$ and all $b\in I$. It follows that $T=L_{-1}=\omega_0$ and $L_0=\omega_1$ both map $I$ into itself so that $I$ is automatically graded and $T$-invariant. Moreover, for a (weak) \voa{} the left-ideal and the right-ideal conditions are equivalent (see Remark~3.9.8 in \cite{LL04}).
\end{defi}
Clearly, $\{0\}$ and $V$ are ideals of $V$. If an ideal contains the vacuum vector $\vac$, then it is already all of $V$.

The following standard algebraic facts hold: given an ideal $I$ of a (weak) (graded) vertex (operator) algebra $V$, there is a natural quotient (weak) (graded) vertex (operator) algebra $V/I$. Moreover, for a (weak) \voa{} $V$ the quotient $V/I$ has the same central charge as $V$. Ideals of $V$ are exactly the kernels of homomorphisms from $V$ and the quotient of $V$ by the kernel is isomorphic to the image of the homomorphism.

The following definition will be important:
\begin{defi}[Simplicity]
A (weak) (graded) vertex (operator) algebra $V$ is said to be \emph{simple} if it has no non-trivial ideal, i.e.\ no ideal other than $\{0\}$ or $V$.
\end{defi}

\section{Modules for \VOA{}s}\label{sec:modules}

A module for a (weak) (graded) vertex (operator) algebra $V$ can be defined naturally as a vector space $W$ and a linear map
\begin{equation*}
Y_W(\cdot,x)\colon V\to \End_\C(W)[[x^{\pm1}]]
\end{equation*}
such that \emph{all the defining properties of a (weak) (graded) vertex (operator) algebra that still make sense hold} \cite{FHL93}.

Modules for vertex algebras are defined for instance in \cite{LL04}, Definition~4.1.1. We will only need modules for \voa{}s. These are vertex algebra modules additionally equipped with a $\C$-grading bounded from below and related to the Virasoro vector $\omega\in V$ and whose graded components are finite-dimensional. This automatically implies the Virasoro relations for the modes of $Y_W(\omega,x)$ as is shown in \cite{LL04}, Proposition~4.1.5. We will still make this part of the definition of \voa{} modules as in the original definition in \cite{FLM88}.

The difference between the modern definition of modules for \voa{}s (as given in \cite{DLM97,DLM98,LL04}) and the original one (as in \cite{FLM88}) is that we allow for a $\C$-grading while originally \voa{} modules were defined to be only $\Q$-graded (see also Remark~4.1.2 in \cite{FHL93}).

\begin{defi}[\VOA{} Module]
Let $V$ be a \voa{} of central charge $c\in\C$. A $V$-module $W$ is given by the data:
\begin{itemize}
\item (\emph{space of states}) a $\C$-graded vector space
\begin{equation*}
W=\bigoplus_{\lambda\in\C}W_\lambda
\end{equation*}
with \emph{weight} $\wt(w)=\lambda$ for $w\in W_\lambda$, $\dim_\C(W_\lambda)<\infty$ for all $\lambda\in\C$ and $W_\lambda=\{0\}$ for $\lambda$ ``sufficiently small in the sense of modifications by an integer'', i.e.\ for fixed $\lambda\in\C$, $W_{\lambda+n}=\{0\}$ for sufficiently negative $n\in\Z$,\footnote{In \cite{LL04} the restriction is slightly stricter demanding $W_\lambda=\{0\}$ for $\lambda$ with sufficiently negative real part.}
\item (\emph{vertex operators}) a linear map
\begin{equation*}
Y_W(\cdot,x)\colon V\to\End_\C(W)[[x^{\pm1}]]
\end{equation*}
taking each $a\in V$ to a field
\begin{equation*}
a\mapsto Y_W(a,x)=\sum_{n\in\Z}a_nx^{-n-1}
\end{equation*}
where for each $w\in W$, $a_nw=0$ for $n$ sufficiently large or equivalently a map $Y_W(\cdot,x)\colon V\otimes_\C W\to W((x))$.
\end{itemize}
These data are subject to the following axioms:
\begin{itemize}
\item (\emph{left vacuum axiom}) $Y_W(\vac,x)=\id_W$.
\item (\emph{translation axiom}) For any $a\in V$,
\begin{equation*}
Y_W(Ta,x)=\partial_x Y_W(a,x).
\end{equation*}
\item (\emph{Jacobi identity}) For $a,b\in V$,
\begin{align*}
&\iota_{x_1,x_0}x_2^{-1}\delta\left(\frac{x_1-x_0}{x_2}\right)Y_W(Y(a,x_0)b,x_2)\\
&=\iota_{x_1,x_2}x_0^{-1}\delta\left(\frac{x_1-x_2}{x_0}\right)Y_W(a,x_1)Y_W(b,x_2)\\
&\quad-\iota_{x_2,x_1}x_0^{-1}\delta\left(\frac{x_2-x_1}{-x_0}\right)Y_W(b,x_2)Y_W(a,x_1).
\end{align*}
\item (\emph{Virasoro relations}) The modes $L_n^W:=\omega_{n+1}^W$ of
\begin{equation*}
Y_W(\omega,x)=\sum_{n\in\Z}\omega_n^Wx^{-n-1}=\sum_{n\in\Z}L_n^Wx^{-n-2}
\end{equation*}
satisfy the \emph{Virasoro relations} at central charge $c$, i.e.\
\begin{equation*}
[L^W_m,L^W_n]=(m-n)L^W_{m+n}+\frac{m^3-m}{12}\delta_{m+n,0}\id_Wc
\end{equation*}
for $m,n\in\Z$. Moreover, $L^W_0w=\wt(w)w$ for homogeneous $w\in W$.
\end{itemize}
\end{defi}

We often omit the $W$ in $L_n^W$ if it is clear from the context whether $L_n$ acts on the \voa{} or its module.

Note that the above axioms imply that for homogeneous $v\in V$ and $n\in\Z$
\begin{equation*}
\wt(v_n)=\wt(v)-n-1
\end{equation*}
as operator on $W$ (cf.\ Definition~\ref{defi:voa}).

\begin{defi}[Adjoint Module]
By definition, any \voa{} can be viewed as a module for itself. As such it is called the \emph{adjoint module}.
\end{defi}

\minisec{Elementary Categorical Notions}
Again, there are the following elementary categorical notions. See \cite{FHL93}, Section~4.3, or \cite{LL04}, Section~4.5, for details.
\begin{defi}[\VOA{} Module Homomorphism]
Let $W^1$ and $W^2$ be modules for the same \voa{} $V$. A $V$-module homomorphism $f\in\Hom_V(W^1,W^2)$ is a linear map $f\colon W^1\to W^2$ such that
\begin{equation*}
f(Y_{W^1}(v,x)w)=Y_{W^2}(v,x)f(w)
\end{equation*}
for all $v\in V$, $w\in W^1$.
\end{defi}
A $V$-module homomorphism is automatically grade-preserving, i.e.\ $\wt(f(w))=\wt(w)$ for homogeneous $w\in W^1$, since $f(L^{W^1}_0 w)=L_0^{W^2}f(w)$ by the above equation.

The notions of \emph{isomorphisms} (or \emph{equivalences}), \emph{endomorphisms} and \emph{automorphisms} of $V$-modules are defined in the obvious way.

\begin{defi}[\VOA{} Submodule]
Let $V$ be a \voa{} and $W$ a $V$-module. A subspace $U$ of $W$ is called \emph{submodule} of $W$ if $Y_W(v,x)w\in U((x))$ for all $v\in V$, $w\in U$ or equivalently if $U$ becomes a $V$-module with the restriction of $Y_W(\cdot,x)$ to $U$.
\end{defi}

\emph{Quotient modules}, \emph{direct-sum modules}, \emph{irreducibility} and \emph{complete reducibility} of modules, etc.\ are defined as expected. Irreducible modules are in particular non-zero.

\begin{defi}[Irreducibility]
A \voa{} $V$ is said to be \emph{irreducible} if the adjoint module $V$ is an irreducible $V$-module.
\end{defi}
Clearly, irreducible \voa{}s are simple. The converse is also true since any submodule of a \voa{} viewed as adjoint module is an ideal \cite{FHL93}.
\begin{prop}
The notions of simplicity and irreducibility of \voa{}s are equivalent.
\end{prop}
Note that if we define irreducibility in the same manner for vertex algebras, then only one implication holds, i.e.\ irreducible vertex algebras are always simple but the converse is false (see \cite{Li03} for details).

\minisec{Schur's Lemma}
Schur's lemma can also be formulated for \voa{} modules:
\begin{prop}[Schur's Lemma, \cite{LL04}, Proposition~4.5.5]\label{prop:schur}
Let $V$ be a \voa{} and $W$ an irreducible $V$-module. Then
\begin{equation*}
\End_V(W)\cong\C.
\end{equation*}
\end{prop}
Clearly, if $W^1$ and $W^2$ are two irreducible $V$-modules and if there is a non-zero homomorphism from $W^1$ to $W^2$, then $W^1$ and $W^2$ already have to be isomorphic, i.e.\ $W^1\cong W^2$. Hence:
\begin{cor}\label{cor:schur}
Let $V$ be a \voa{} and let $W^1$ and $W^2$ be irreducible $V$-modules. Then
\begin{equation*}
\Hom_V(W^1,W^2)\cong\begin{cases}\C&\text{if }W^1\cong W^2,\\0&\text{if }W^1\ncong W^2.\end{cases}
\end{equation*}
\end{cor}

\minisec{Conformal Weight}
It is a simple consequence of the definition (see \cite{LL04}, Section~4.1) that any \voa{} module $W$ decomposes into submodules for the congruence classes modulo~1 of the weights, i.e.\
\begin{equation*}
W=\bigoplus_{\mu+\Z\in\C/\Z}W_{\mu+\Z}\quad\text{(as $V$-modules)}
\end{equation*}
with
\begin{equation*}
W_{\mu+\Z}:=\bigoplus_{\lambda\in\mu+\Z}W_{\lambda}\quad\text{(as vector spaces)}.
\end{equation*}

\begin{defi}[Conformal Weight]
If $W$ is an irreducible module for a \voa{}, then the above considerations imply that $W$ is of the form $W=W_{\rho+\Z}$ for some $\rho\in\C$. Moreover, we know that the grading has to be bounded from below and hence we can choose $\rho$ such that
\begin{equation*}
W=\bigoplus_{\lambda\in\rho+\N}W_{\lambda}=\bigoplus_{k=0}^\infty W_{\rho+k}.
\end{equation*}
This $\rho$ is called the \emph{conformal weight} of $W$ and denoted by $\rho(W)$.
\end{defi}

If $V$ is a simple \voa{}, then viewed as its adjoint module it is irreducible and has some conformal weight $\rho(V)\in\Z$. If $V$ is of CFT-type, then $\rho(V)=0$.

\section{Rationality and \texorpdfstring{$C_2$}{C\_2}-Cofiniteness}

In the following we introduce two very important \emph{niceness} properties of \voa{}s, namely rationality and $C_2$-cofiniteness, which will be used throughout this text.

\minisec{$C_2$-Cofiniteness}

The notion of $C_2$-cofiniteness was first introduced by Zhu (part of ``finiteness condition $C$'' in \cite{Zhu96}) as a property needed to prove his modular invariance result (see Section~\ref{sec:zhu}).
\begin{defi}[$C_2$-Cofiniteness]
Let $V$ be a \voa{} and let
\begin{equation*}
C_2(V):=\spn_\C(\{a_{-2}b\;|\;a,b\in V\})
\end{equation*}
be the linear span of the elements of the form $a_{-2}b$. The \voa{} $V$ is said to be \emph{$C_2$-cofinite} if the space $C_2(V)$ has finite codimension in $V$.
\end{defi}

\minisec{Rationality and Regularity}

In order to define rationality and the related notion of regularity we also need the concept of \emph{weak} and \emph{admissible} modules of \voa{}s (as defined for example in \cite{DLM97,DLM98}).
We will not give the definition but note that for a \voa{} $V$ there are the following inclusions:
\begin{equation*}
\{\text{weak $V$-modules}\}\supseteq\{\text{admissible $V$-modules}\}\supseteq\{\text{(ordinary) $V$-modules}\}.
\end{equation*}

\begin{defi}[Rationality~I, \cite{DLM97}]
A \voa{} $V$ is called \emph{rational} if every admissible $V$-module is completely reducible, i.e.\ isomorphic to a direct sum of irreducible admissible $V$-modules.
\end{defi}
\begin{rem}
If $V$ is a rational \voa{}, then one can show (see \cite{DLM97}, Remark~2.4) that there are only finitely many irreducible admissible $V$-modules up to isomorphism and that each irreducible admissible $V$-module is an (ordinary) $V$-module.
\end{rem}
We can hence rewrite the definition of rationality with a formally stronger condition:
\begin{defi}[Rationality~II]
A \voa{} $V$ is called \emph{rational} if every admissible $V$-module is isomorphic to a direct sum of irreducible (ordinary) $V$-modules.
\end{defi}
If $V$ is a rational \voa{}, we denote by $\Irr(V)$ the \emph{finite} set of isomorphism classes of irreducible $V$-modules.
By an element $W\in\Irr(V)$ we sometimes mean an isomorphism class of $V$-modules but more often some arbitrary representative of that class.

Similarly, we define the stronger concept of regularity by replacing admissible modules with weak modules in the above definition.
\begin{defi}[Regularity]
A \voa{} $V$ is called \emph{regular} if every weak $V$-module is isomorphic to a direct sum of irreducible (ordinary) $V$-modules.
\end{defi}
\begin{rem}\label{rem:regrat}
Clearly, every regular \voa{} is rational. In \cite{Li99} it is shown that any regular \voa{} is $C_2$-cofinite. Conversely, a rational, $C_2$-cofinite \voa{} of CFT-type is regular \cite{ABD04}. In total, this means that a \voa{} of CFT-type is regular if and only if it is rational and $C_2$-cofinite.
\end{rem}

The following result is proved in Dong, Li and Mason's modular invariance paper \cite{DLM00} (see also Section~\ref{sec:dlmmodinv} below) and justifies the use of the word ``rational'':
\begin{thm}[\cite{DLM00}, Theorem~11.3]\label{thm:11.3}
Let $V$ be a rational, $C_2$-cofinite \voa{}. Then the central charge of $V$ is rational and each irreducible $V$-module has rational conformal weight.
\end{thm}

\minisec{Holomorphicity}
The following is an important special case of rationality:
\begin{defi}[Holomorphicity]\label{defi:hol}
A rational \voa{} $V$ is called \emph{holomorphic} if the adjoint module $V$ is the only irreducible $V$-module up to isomorphism.
\end{defi}
Some authors do not include rationality in the definition of holomorphicity. Also note that holomorphic \voa{}s are sometimes called \emph{self-dual} (see for example \cite{Mon94,Hoe95}) or \emph{meromorphic}.\footnote{A more suitable name for ``holomorphic'' is perhaps \emph{``unimodular''} describing the fact that such \voa{}s have only one irreducible module but also acknowledging that lattice \voa{}s (see Section~\ref{sec:latvoa}) associated with unimodular lattices are holomorphic. For lattices, of course, ``unimodular'' means that the determinant is of unit modulus.}

\begin{rem}
A holomorphic \voa{} is by definition irreducible and therefore simple.
\end{rem}

The following is a well-known consequence of Zhu's modular invariance result:
\begin{prop}\label{prop:div8}
Let $V$ be a holomorphic, $C_2$-cofinite \voa{} of CFT-type. Then the central charge $c$ of $V$ is a positive integer divisible by 8.
\end{prop}
We include a proof for completeness. It depends on definitions and results occurring later in this text.
\begin{proof}
We use Zhu's modular invariance result (Theorem~\ref{thm:zhumodinv}). The character of $V$, $\ch_V(\tau)=\tr_V q_\tau^{L_0-c/24}=q_\tau^{-c/24}\sum_{n=0}^\infty \dim_\C(V_n)q_\tau^n$, $q_\tau=\e^{2\pi\i\tau}$, transforms under $\SLZ$ as $\ch_V(M.\tau)=\sigma(M)\ch_V(\tau)$ with a representation $\sigma\colon\SLZ\to\C^\times$. Let us define $\lambda=\sigma(S)$ where $S=\left(\begin{smallmatrix}0&-1\\1&0\end{smallmatrix}\right)$. Then, since $S^2.\tau=\tau$ we know that $\lambda^2=1$, i.e.\ $\lambda=\pm 1$ and from Lemma~\ref{lem:lambdapos} we get that indeed $\lambda=1$. For $T=\left(\begin{smallmatrix}1&1\\0&1\end{smallmatrix}\right)\in\SLZ$, $\sigma(T)=\e^{(2\pi\i)(-c/24)}$ as can be easily seen from the $q_\tau$-expansion of the character. Also, $(ST)^3.\tau=\tau$ and hence $1=\lambda^3\e^{(2\pi\i)(-3c/24)}=\e^{(2\pi\i)(-3c/24)}$, which implies $8\mid c$. Finally, we note that $\ch_V(\tau)$ is a modular form for $\SLZ$ of weight $0$ and possibly some character of order 3 which is holomorphic on $\H$ and meromorphic at the cusp $\i\infty$. This implies that $\ch_V(\tau)$ has a pole at $\i\infty$, which means that the $q_\tau$-expansion starts at some negative exponent, i.e. $c$ is positive.

Indeed, assume that $\ch_V(\tau)$ has no pole at $\i\infty$. Then $\ch_V(\tau)-\ch_V(\i\infty)$ has a zero at $\i\infty$ and is holomorphic on $\H$ so that the valence formula (see e.g.\ \cite{HBJ94}, Theorem~I.4.1) implies that $\ch_V(\tau)-\ch_V(\i\infty)$ vanishes, i.e.\ that $\ch_V(\tau)$ is constant, which is a contradiction.
\end{proof}

\minisec{Nice \VOA{}s}

For the purposes of this text, we will mostly deal with \voa{}s satisfying the following five niceness properties:
\begin{customass}{\textbf{\textsf{N}}}[Niceness]\label{ass:n}
Let the \voa{} be simple, rational, $C_2$-cofinite, self-contragredient and of CFT-type.
\end{customass}
These \voa{}s are particularly well-behaved and many strong results hold under this assumption. Note that self-contragredience will be defined in the next section.

\section{Contragredient Modules and Invariant Bilinear Forms}\label{sec:cmibf}
Given a module $W$ of a \voa{} $V$, it is possible to construct another module $W'$, \emph{dual} to $W$, with the same grading as $W$. This module $W'$ is called the contragredient module. The following steps can be found in \cite{FHL93}, Section~5.2.
\minisec{Contragredient Modules}
Let $W$ be a $V$-module with weight grading
\begin{equation*}
W=\bigoplus_{\lambda\in\C}W_\lambda.
\end{equation*}
We define $W'$ as a vector space to be the graded dual space of $W$, i.e.\
\begin{equation*}
W':=\bigoplus_{\lambda\in\C}W_\lambda^*
\end{equation*}
where $W_\lambda^*$ denotes the dual vector space of $W_\lambda$. We then define the \emph{adjoint vertex operators} $Y_{W'}(\cdot,x)\colon V\to\End_\C(W')[[x^{\pm1}]]$ via
\begin{equation*}
\langle Y_{W'}(v,x)w',w\rangle=\langle w',Y_W^*(v,x)w\rangle
\end{equation*}
with
\begin{equation*}
Y_W^*(v,x)=\sum_{n\in\Z}v^*_n x^{-n-1}:=Y_W(\e^{xL_1}(-x^2)^{L_0}v,x^{-1})
\end{equation*}
for $v\in V$, $w\in W$, $w'\in W'$. Here, $\langle\cdot,\cdot\rangle$ denotes the canonical pairing between $W$ and its graded dual space $W'$. The $\C$-grading on $W'$ is simply given by $(W')_\lambda:=(W_\lambda)^*$. Then the following theorem holds:
\begin{thm}[\cite{FHL93}, Theorem~5.2.1]
The pair $(W',Y_{W'})$ carries the structure of a $V$-module, called the \emph{contragredient} (or \emph{dual}) \emph{module} of $W$.
\end{thm}
By definition, $\dim_\C(W_\lambda)=\dim_\C(W'_\lambda)$ for all $\lambda\in\C$ and the following holds:
\begin{prop}[\cite{FHL93}, Proposition~5.3.1]
Let $V$ be a \voa{} and let $W$ be a $V$-module. Then $W''\cong W$.
\end{prop}

\begin{prop}[\cite{FHL93}, Proposition~5.3.2]\label{prop:conirr}
Let $V$ be a \voa{} and let $W$ be a $V$-module. The module $W$ is irreducible if and only if $W'$ is.
\end{prop}

Recall that a \voa{} $V$ itself can be viewed as a module, the adjoint module, for which we can also consider the contragredient module $V'$.
\begin{defi}[Self-Contragredience]
Let $V$ be a \voa{} and let $W$ be a $V$-module. Then $W$ is called \emph{self-contragredient} (usually called \emph{self-dual}\footnote{We do not use this term to avoid confusion with holomorphicity, for which sometimes also the term ``self-dual'' is used (see Definition~\ref{defi:hol}).}) if $W$ is isomorphic to its contragredient module $W'$.

The \voa{} $V$ is called \emph{self-contragredient} if the adjoint module $V$ is isomorphic to its contragredient module $V'$, i.e.\ if $V\cong V'$ (as $V$-modules).
\end{defi}

Holomorphic \voa{}s are always self-contragredient.

\minisec{Invariant Bilinear Forms}
Let $V$ be a \voa{} and let $W$ be a $V$-module. Assume that $W$ is self-contragredient, i.e.\ that there is an isomorphism of $V$-modules $\phi_W\colon W\to W'$. Consider the natural bilinear form $(\cdot,\cdot)_W$ on $W$ defined by
\begin{equation*}
(u,w)_W:=\langle\phi_W(u),w\rangle,
\end{equation*}
for $u,w\in W$ (see \cite{FHL93}, Remark~5.3.3).
\begin{defi}[Invariant Bilinear Form]
Let $V$ be a \voa{}. We say a bilinear form $(\cdot,\cdot)$ on a $V$-module $W$ is \emph{invariant} if
\begin{equation*}
(Y_W(v,x)u,w)=(u,Y_W^*(v,x)w)
\end{equation*}
for all $v\in V$, $u,w\in W$.
\end{defi}
\begin{rem}[\cite{Li94}, Remark~2.5]
Let $V$ be a \voa{} and $(\cdot,\cdot)$ a bilinear form on the $V$-module $W$. If $(\cdot,\cdot)$ is invariant, then also
\begin{equation*}
(Y_W^*(v,x)u,w)=(u,Y_W(v,x)w)
\end{equation*}
for all $v\in V$, $u,w\in W$.
\end{rem}
The bilinear form $(\cdot,\cdot)_W$ we defined above on the self-contragredient module $W$ is non-degenerate and it is invariant since
\begin{align*}
(Y_W(v,x)u,w)_W&=\langle\phi_W Y_W(v,x)u,w\rangle=\langle Y_{W'}(v,x)\phi_W(u),w\rangle=\langle\phi_W(u),Y_W^*(v,x)w\rangle\\
&=(u,Y_W^*(v,x)w)_W,
\end{align*}
where we used that $\phi_W\in\Hom_{V}(W,W')$. In fact, specifying a non-degenerate bilinear form on $W$ is equivalent to specifying a $V$-module isomorphism $W\cong W'$. This yields a characterisation of self-contragredient modules: a module $W$ for a \voa{} $V$ is self-contragredient if and only if there is a non-degenerate, invariant bilinear form on $W$. More generally:
\begin{prop}[\cite{FHL93}]\label{prop:bilhom}
The space of all invariant bilinear forms on a module $W$ for a \voa{} $V$ is naturally isomorphic to $\Hom_V(W,W')$. 
\end{prop}

Let $W$ be an irreducible, self-contragredient $V$-module. By Schur's lemma (Proposition~\ref{prop:schur}) and the above proposition, there is exactly one invariant bilinear form on $W$ up to multiplication by a complex scalar. This bilinear form is non-degenerate and up to a scalar given by the natural invariant bilinear form $(\cdot,\cdot)_W$ on $W$ defined above.

The following property of invariant bilinear forms on irreducible modules holds:
\begin{prop}[\cite{Li94}, Proposition~2.8]\label{prop:2.8}
Let $V$ be a \voa{} and $W$ an irreducible $V$-module. Then any invariant bilinear form on $W$ is either symmetric or antisymmetric.
\end{prop}
\begin{prop}[\cite{Li94}, Proposition~2.6, \cite{FHL93}, Proposition~5.3.6]\label{prop:2.6}
Any invariant bilinear form on a \voa{} $V$ is symmetric.
\end{prop}

The proof of the second proposition uses that on a \voa{} $V$
\begin{equation}\label{eq:skew}
Y_V(a,x)b=\e^{xL_{-1}}Y_V(b,-x)a
\end{equation}
for all $a,b\in V$, a property which is called \emph{skew-symmetry} (see e.g.\ (2.3.19) of \cite{FHL93}).

A simple calculation shows:
\begin{prop}[\cite{FHL93,Li94}]\label{prop:bilorth}
Let $V$ be a \voa{} and $(\cdot,\cdot)$ an invariant bilinear form on the $V$-module $W$. Then the modes of the Virasoro vector obey
\begin{equation*}
(L_nu,v)=(u,L_{-n}v)
\end{equation*}
for $u,v\in W$ and $n\in\Z$ so that
\begin{equation*}
(W_\lambda,W_\mu)=0
\end{equation*}
for $\lambda\neq\mu\in\C$.
\end{prop}

In the following we will characterise the space of all symmetric, invariant bilinear forms on a \voa{} $V$ (see \cite{Li94}, Section~3). Let $(\cdot,\cdot)$ be such a bilinear form. Then $(u,v)$ is equal to the constant term of
\begin{equation}\label{eq:biluni}
(Y(u,x)\vac,v)=(\vac,Y^*(u,x)v)
\end{equation}
for $u,v\in V$. Let $f$ be the linear functional $f\colon V_0\to \C$ defined by
\begin{equation*}
f(u):=(\vac,u)
\end{equation*}
for $u\in V_0$. Then by the above Proposition~\ref{prop:bilorth} and \eqref{eq:biluni} it is evident that $(\cdot,\cdot)$ is uniquely determined by $f$. On the other hand, since $L_{-1}\vac=0$, it follows that
\begin{equation*}
(\vac,L_1V_1)=(L_{-1}\vac,V_1)=0,
\end{equation*}
which means that $f$ vanishes on the subspace $L_1V_1$ of $V_0$. Therefore we can view $f$ as a linear functional on the quotient space $V_0/L_1V_1$. Conversely, a linear functional $f$ on $V_0$ determines a symmetric, invariant bilinear form if and only if $L_1V_1\subseteq\ker(f)$ \cite{Li94}. Therefore:
\begin{thm}[\cite{Li94}, Theorem~3.1]\label{thm:3.1li}
Let $V$ be a \voa{}. The space of symmetric, invariant bilinear forms is naturally isomorphic to the dual space of $V_0/L_1V_1$.
\end{thm}

In particular, combining the above result with that of Proposition~\ref{prop:bilhom} we obtain:
\begin{cor}
Let $V$ be a \voa{}. Then there is a natural isomorphism of $\C$-vector spaces
\begin{equation*}
\left(V_0/L_1V_1\right)^*\cong\Hom_V(V,V').
\end{equation*}
\end{cor}

\minisec{Self-Contragredient \VOA{}s}
The above corollary implies
\begin{equation*}
\dim_\C(V_0/L_1V_1)=\dim_\C(\Hom_V(V,V')).
\end{equation*}
If we assume that $V$ is simple, then we can combine this with Schur's lemma (Corollary~\ref{cor:schur})
to obtain:
\begin{cor}
Let $V$ be a simple \voa{}. Then
\begin{equation*}
\dim_\C(V_0/L_1V_1)=\begin{cases}1&\text{if }V\cong V',\\0&\text{if }V\ncong V'.\end{cases}
\end{equation*}
\end{cor}
Clearly, if $V$ is additionally of CFT-type, then $\dim_\C(V_0)=1$ and hence the above result implies that $V$ is self-contragredient if and only if $L_1V_1=\{0\}$. The $\impliedby$-direction is true for any simple \voa{} $V$.

We also need the following result:
\begin{prop}\label{prop:subsd}
Let $V$ be a simple \voa{} of CFT-type and let $U$ be a \vosa{} of $V$ such that $U$ is also simple. If $V$ is self-contragredient, i.e.\ $V\cong V'$, then so is $U$, i.e.\ $U\cong U'$.
\end{prop}
\begin{proof}
Since $V$ is self-contragredient and of CFT-type, $L_1V_1=\{0\}$ holds. Then also $L_1U_1=\{0\}$ and $U$ is self-contragredient.
\end{proof}
For the above proof to be correct it is important that the Virasoro vector $\omega$ and hence $L_1$ is the same on $U$ and on $V$ (see Definition~\ref{defi:subvoa}).

\section{Intertwining Operators and Fusion Product}\label{sec:intops}
We briefly introduce the closely related concepts of intertwining operators and fusion products. 
\minisec{Intertwining Operators}
Vertex operators are operators parametrised by elements of the vertex algebra acting on the vertex algebra itself or on modules. A more general concept is that of intertwining operators (see \cite{FHL93}, Section~5.4), which are parametrised by the vectors in one module and map elements in a second module to elements in a third one.
\begin{defi}[Intertwining Operators]\label{defi:intops}
Let $V$ be a \voa{} and let $W^1,W^2,W^3$ be $V$-modules (not necessarily distinct and possibly equal to $V$). An \emph{intertwining operator} of type $\binom{W^3}{W^1\,W^2}$ is a linear map $W^1\to \Hom_\C(W^2,W^3)\{x\}$ or equivalently $W^1\otimes_\C W^2\to W^3\{x\}$,
\begin{equation*}
w\mapsto\mathcal{Y}(w,x)=\sum_{n\in\C}w_nx^{-n-1}
\end{equation*}
where for each $w'\in W^2$, $w_n w'=0$ for $n$ with sufficiently large real part such that all the defining properties of a module action that still make sense in this setting hold, i.e.:
\begin{itemize}
\item (\emph{Jacobi identity}) Let $Y_{W^1}(\cdot,x)$,
$Y_{W^2}(\cdot,x)$, $Y_{W^3}(\cdot,x)$ denote the module vertex operators on the modules $W^1,W^2,W^3$. For $v\in V$ and $w\in W^1$,
\begin{align*}
&\iota_{x_1,x_0}x_2^{-1}\delta\left(\frac{x_1-x_0}{x_2}\right)\mathcal{Y}(Y_{W^1}(v,x_0)w,x_2)\\
&=\iota_{x_1,x_2}x_0^{-1}\delta\left(\frac{x_1-x_2}{x_0}\right)Y_{W^3}(v,x_1)\mathcal{Y}(w,x_2)\\
&\quad-\iota_{x_2,x_1}x_0^{-1}\delta\left(\frac{x_2-x_1}{-x_0}\right)\mathcal{Y}(w,x_2)Y_{W^2}(v,x_1).
\end{align*}
Note that all terms become meaningful when applied to an element $w'\in W^2$. Moreover, note that this Jacobi identity involves integral powers of $x_0$ and $x_1$ and complex powers of $x_2$.
\item (\emph{translation axiom}) For $w\in W^1$,
\begin{equation*}
\partial_x\mathcal{Y}(w,x) = \mathcal{Y}(L_{-1}w,x)
\end{equation*}
where $L_{-1}$ is the operator $L_{-1}^{W^1}$ acting on $W^1$.
\end{itemize}
\end{defi}
This definition is more general than the one in \cite{FHL93} since it allows complex rather than only rational exponents. The intertwining operators of type $\binom{W^3}{W^1\,W^2}$ form a vector space, which we denote by $\V_{W^1\,W^2}^{W^3}$. Clearly, for a \voa{} $V$, the vertex operator $Y(\cdot,x)$ is an intertwining operator of type $\binom{V}{V\,V}$ and for a $V$-module $W$ the module vertex operator $Y_W(\cdot,x)$ is an intertwining operator of type $\binom{W}{V\,W}$.

It follows from the definition of intertwining operators that $\wt(w_n)=\wt(w)-n-1$ for homogeneous $w\in W^1$ and $n\in\C$. In particular, assume there are irreducible $V$-modules $W^1,W^2,W^3$ with conformal weights $\rho(W^1),\rho(W^2),\rho(W^3)\in\C$. Then for the intertwining operators of type $\binom{W^3}{W^1\,W^2}$,
\begin{equation*}
\mathcal{Y}(\cdot,x)\colon W^1\to x^{B}\Hom_\C(W^2,W^3)[[x^{\pm1}]]
\end{equation*}
where $B=\rho(W^3)-\rho(W^1)-\rho(W^2)\in\C$. Equivalently we can view $\mathcal{Y}(\cdot,x)$ as a map $W^1\otimes_\C W^2\to x^{B}W^3((x))$.

\minisec{$S_3$-Symmetry}

In the following we discuss certain isomorphisms between spaces of intertwining operators. These have first been stated in \cite{FHL93} under an integrality assumption on the weight grading but hold more generally.
\begin{prop}[\cite{HL95}, Propositions 7.1 and 7.3, \cite{FHL93}, Propositions 5.4.7 and 5.5.2]\label{prop:intsym}
Let $V$ be a \voa{} and let $W^1,W^2,W^3$ be $V$-modules. There are natural isomorphisms of spaces of intertwining operators
\begin{equation*}
\V_{W^1\,W^2}^{W^3}\cong\V_{W^2\,W^1}^{W^3}\quad\text{and}\quad\V_{W^1\,W^2}^{W^3}\cong\V_{W^1\,(W^3)'}^{(W^1)'}.
\end{equation*}
\end{prop}
This immediately leads to the following result:
\begin{cor}[\cite{HL95}, Proposition~7.5, (5.5.8) in \cite{FHL93}]
Let $V$ be a \voa{} and let $W^1,W^2,W^3$ be $V$-modules. If we define
\begin{equation*}
\V_{W^1,W^2,W^3}:=\V_{W^1\,W^2}^{(W^3)'},
\end{equation*}
then for any permutation $\sigma\in S_3$,
\begin{equation*}
\V_{W^1,W^2,W^3}\cong\V_{W^{\sigma(1)},W^{\sigma(2)},W^{\sigma(3)}}.
\end{equation*}
\end{cor}

\minisec{Fusion Product}

We introduce the concept of the fusion-product (or tensor-product) module.
\begin{defi}[Fusion Product, \cite{Li98}, Definition~3.1]
Let $W^1$ and $W^2$ be $V$-modules. A \emph{fusion product} (or \emph{tensor product}) for the ordered pair $(W^1,W^2)$ is a pair $(W^1\boxtimes_VW^2,\mathcal{F}(\cdot,x))$ consisting of a $V$-module $W^1\boxtimes_VW^2$ and a $V$-intertwining operator $\mathcal{F}(\cdot,x)$ of type $\binom{W^1\boxtimes_VW^2}{W^1\,W^2}$ satisfying the following universal property: for any $V$-module $U$ and any intertwining operator $\mathcal{Y}(\cdot,x)$ of type $\binom{U}{W^1\,W^2}$ there is a unique $V$-module homomorphism $\psi\colon W^1\boxtimes_VW^2\to U$ such that $\mathcal{Y}(\cdot,x)=\psi \mathcal{F}(\cdot,x)$.
\end{defi}
Viewing the intertwining operator $\mathcal{F}(\cdot,x)$ as a map $W^1\otimes_\C W^2\to (W^1\boxtimes_VW^2)\{x\}$ and $\mathcal{Y}(\cdot,x)\colon W^1\otimes_\C W^2\to U\{x\}$ we can express the above situation by the following commutative diagram:
\begin{equation*}
\begin{tikzcd}
W^1\otimes_\C W^2 \arrow{ddrr}{\mathcal{Y}(\cdot,x)} \arrow{rr}{\mathcal{F}(\cdot,x)}&&(W^1\boxtimes_VW^2)\{x\}\arrow[dashed]{dd}{\exists !\psi}\\
\\
&&U\{x\}
\end{tikzcd}
\end{equation*}
We will also write $\boxtimes$ instead of $\boxtimes_V$ if the \voa{} $V$ is clear from the context.

\begin{rem}[\cite{Li98}, Remark~3.2]
It follows directly from the definition via the universal property that if a fusion product exists, then it is unique up to isomorphism.
\end{rem}

The existence of a fusion product for two given $V$-modules is far from obvious. The theory of tensor products of modules of a \voa{} was first developed by Huang and Lepowsky in \cite{HL92,HL94,HL95,HL95b} and by Li in \cite{Li98}. In particular, it is shown in \cite{HL95,HL95b} and \cite{Li98} that if the \voa{} $V$ is rational, then a fusion product $W^1\boxtimes_V W^2$ for any two $V$-modules $W^1$ and $W^2$ exists.

In fact, for any $z\in\C^\times$ Huang and Lepowsky construct a tensor product, which they call $P(z)$-tensor product, denoted by $\boxtimes_{P(z)}$. For convenience, we choose $z=1$ and write $\boxtimes$ or $\boxtimes_V$ for $\boxtimes_{P(1)}$.
\begin{prop}[\cite{HL95}, Corollary~6.5, \cite{HL95b}, Corollary~13.12, \cite{Li98}, Theorem~3.20]
Let $V$ be a rational \voa{} and $W^1$ and $W^2$ two $V$-modules. Then a tensor product $W^1\boxtimes_V W^2$ exists.
\end{prop}
Also note that the fusion product will in general not have the vector space tensor product as its underlying vector space.

The fusion product is by definition closely related to the concept of intertwining operators:
\begin{prop}[\cite{Li98}, Corollary~3.4, \cite{HL95}, Corollary~4.10, \cite{HL95b}, Proposition~12.3]\label{prop:cor3.4}
Let $V$ be a \voa{} and $W^1,W^2,W^3$ be $V$-modules. Assume that the tensor product $W^1\boxtimes_V W^2$ exists. Then, naturally
\begin{equation*}
\Hom_V(W^1\boxtimes_V W^2,W^3)\cong\V_{W^1\,W^2}^{W^3}
\end{equation*}
(as vector spaces), where $\V_{W^1\,W^2}^{W^3}$ is the space of intertwining operators of type $\binom{W^3}{W^1\,W^2}$.
\end{prop}
The isomorphism in the above proposition can be explicitly described as follows: let $(W^1\boxtimes_VW^2,\mathcal{F}(\cdot,x))$ be the fusion product for $(W^1,W^2)$ and $\varphi\in\Hom_V(W^1\boxtimes_V W^2,W^3)$. Then $\varphi\circ\mathcal{F}(\cdot,x)$ is an intertwining operator of type $\binom{W^3}{W^1\,W^2}$.

\begin{prop}[\cite{Li98}, Remark~3.5]\label{prop:rem3.5a}
Let $V$ be a \voa{} and $W^1,W^2$ be $V$-modules. Assume that the tensor product $W^1\boxtimes_V W^2$ exists. Then also $W^2\boxtimes_V W^1$ exists and
\begin{equation*}
W^1\boxtimes_V W^2\cong W^2\boxtimes_V W^1
\end{equation*}
(as $V$-modules).
\end{prop}
Note that in the above proposition $W^1\boxtimes_V W^2$ and $W^2\boxtimes_V W^1$ are not isomorphic as fusion products but only as $V$-modules.

The following unitality property of the adjoint module holds:
\begin{prop}[\cite{Li98}, Remark~3.5]\label{prop:rem3.5b}
Let $V$ be a \voa{} and $W$ a $V$-module. Then the tensor product $V\boxtimes_V W$ exists and
\begin{equation*}
V\boxtimes_V W\cong W
\end{equation*}
(as $V$-modules).
\end{prop}
More precisely, if $(W,Y_W(\cdot,x))$ is a $V$-module, then $(W,Y_W(\cdot,x))$ is a tensor product for the pair $(V,W)$. Propositions \ref{prop:cor3.4} and \ref{prop:rem3.5b} imply that $\V_{V\,W^1}^{W^2}\cong\Hom_{V}(W^1,W^2)$ for any $V$-modules $W^1$ and $W^2$.

The associativity of the fusion product has been established under certain assumptions (in particular if $V$ is rational, $C_2$-cofinite and of CFT-type) in \cite{Hua95b,Hua96,Hua96b,Hua05} (see also \cite{DLM97b}):
\begin{prop}\label{prop:fusionassoc}
Let $V$ be a rational, $C_2$-cofinite \voa{} of CFT-type and let $W^1,W^2,W^3$ be three $V$-modules. Then
\begin{equation*}
W^1\boxtimes_V(W^2\boxtimes_VW^3)\cong(W^1\boxtimes_VW^2)\boxtimes_VW^3
\end{equation*}
(as $V$-modules).
\end{prop}

\minisec{Fusion Algebra}
In the following assume that the \voa{} $V$ is rational. Then every $V$-module is completely reducible and there is a finite set $\Irr(V)$ of isomorphism classes of irreducible $V$-modules. We consider the fusion product of two irreducible modules $X,Y\in\Irr(V)$ and decompose it again as a direct sum of irreducible modules, i.e
\begin{equation*}
X\boxtimes_V Y\cong\bigoplus_{W\in \Irr(V)}\underbrace{(W\oplus\ldots\oplus W)}_{N_{X,Y}^W\text{ times}}=\bigoplus_{W\in \Irr(V)}N_{X,Y}^W W
\end{equation*}
for $X,Y\in \Irr(V)$ where the numbers $N_{X,Y}^W\in\N$ depend only on the isomorphism classes and are called \emph{fusion rules}.

The fusion rules are closely related to the dimensions of the spaces of intertwining operators. Indeed, by Proposition~\ref{prop:cor3.4} and Schur's lemma (Corollary~\ref{cor:schur}),
\begin{align*}
\V_{X\,Y}^{W}&\cong\Hom_V(X\boxtimes_V Y,W)\cong \Hom_V(\bigoplus_{U\in \Irr(V)}N_{X,Y}^U U,W)\\
&\cong\prod_{U\in \Irr(V)}(\underbrace{\Hom_V(U,W)}_{\cong\delta_{U,W}\C})^{N_{X,Y}^U}\cong\C^{N_{X,Y}^W}
\end{align*}
(as vector spaces) and hence
\begin{equation*}
\dim_\C(\V_{X\,Y}^{W})=N_{X,Y}^W
\end{equation*}
for any $X,Y,W\in \Irr(V)$.

We then define the \emph{fusion algebra} (or \emph{Verlinde algebra}) $\V(V)$ associated with $V$ on the finite-dimensional vector space $\V(V):=\bigoplus_{W\in \Irr(V)}\C W$, spanned by a formal basis $\Irr(V)$, by equipping it with an algebra product defined by linear continuation of the above fusion rules.

By Proposition~\ref{prop:rem3.5a} we know that the fusion rules are symmetric, i.e.\ $N_{X,Y}^W=N_{Y,X}^W$ for all $X,Y,W\in \Irr(V)$. This also follows directly from Proposition~\ref{prop:intsym} about spaces of intertwining operators together with the identity $\dim_\C(\V_{X\,Y}^{W})=N_{X,Y}^W$. Finally, by Proposition~\ref{prop:fusionassoc} we know that if $V$ is rational, $C_2$-cofinite and of CFT-type, then the fusion product is associative.

We collect the above observations in the following omnibus theorem:
\begin{thm}\label{thm:fusalg}
Let $V$ be a rational, $C_2$-cofinite \voa{} of CFT-type. Then the associated fusion algebra $\V(V)$ is a finite-dimensional, commutative, associative, unital $\C$-algebra with unit $V$.
\end{thm}
The fact that $V$ is a unit is exactly the statement of Proposition~\ref{prop:rem3.5b}.

\section{Tensor Product}\label{sec:tensor}

In the previous section we studied the fusion product (or tensor product) of two modules for a fixed \voa{}, which we denoted by $\boxtimes_V$. A much more elementary concept is that of the tensor product of two, or in general finitely many, possibly distinct \voa{}s (see \cite{FHL93}, Section~2.5). Namely, given two \voa{}s $V^{(1)}$ and $V^{(2)}$, it is possible to define a \voa{} structure on the tensor product of vector spaces $V^{(1)}\otimes V^{(2)}:=V^{(1)}\otimes_\C V^{(2)}$ by setting
\begin{equation*}
Y_{V^{(1)}\otimes V^{(2)}}(v_1\otimes v_2,x):=Y_{V^{(1)}}(v_1,x)\otimes Y_{V^{(2)}}(v_2,x)
\end{equation*}
for $v_1\in V^{(1)}$ and $v_2\in V^{(2)}$,
\begin{equation*}
\vac_{V^{(1)}\otimes V^{(2)}}:=\vac_{V^{(1)}}\otimes\vac_{V^{(2)}}
\end{equation*}
and defining the Virasoro vector
\begin{equation*}
\omega_{V^{(1)}\otimes V^{(2)}}:=\omega_{V^{(1)}}\otimes\vac_{V^{(2)}}+\vac_{V^{(1)}}\otimes\omega_{V^{(2)}}.
\end{equation*}
Clearly, the central charges add, i.e.\
\begin{equation*}
c_{V^{(1)}\otimes V^{(2)}}=c_{V^{(1)}}+c_{V^{(2)}}.
\end{equation*}
It follows directly from the definition of the Virasoro vector that
\begin{equation*}
\wt(v_1\otimes v_2)=\wt(v_1)+\wt(v_2)
\end{equation*}
for two homogeneous vectors $v_1\in V^{(1)},v_2\in V^{(2)}$.

Furthermore, the following properties hold:
\begin{prop}[\cite{FHL93}, Corollary~4.7.3]
Given two \voa{}s $V^{(1)}$ and $V^{(2)}$, the tensor-product \voa{} $V^{(1)}\otimes V^{(2)}$ is simple if and only if both $V^{(1)}$ and $V^{(2)}$ are simple.
\end{prop}
By definition of the Virasoro vector:
\begin{prop}
Let $V^{(1)}$ and $V^{(2)}$ be two \voa{}s of CFT-type. Then also the tensor-product \voa{} $V^{(1)}\otimes V^{(2)}$ is of CFT-type.
\end{prop}
\begin{proof}
Note that for any two \voa{}s $V^{(1)}$ and $V^{(2)}$,
\begin{equation*}
(V^{(1)}\otimes V^{(2)})_0=\ldots\oplus(V^{(1)}_{-1}\otimes_\C V^{(2)}_{1})\oplus(V^{(1)}_{0}\otimes_\C V^{(2)}_{0})\oplus(V^{(1)}_{1}\otimes_\C V^{(2)}_{-1})\oplus\ldots.
\end{equation*}
The statement follows immediately.
\end{proof}

In the same way we defined the tensor product for \voa{}s, we can also define the tensor product for modules (see \cite{FHL93}, Section~4.6), which is well-defined under a certain assumption on the grading of the original modules. This assumption is in particular fulfilled if the original modules are irreducible. Given two modules $W^{(1)}$ and $W^{(2)}$ for the \voa{}s $V^{(1)}$ and $V^{(2)}$, respectively, the tensor-product module $W^{(1)}\otimes W^{(2)}$ is indeed a module for the tensor-product \voa{} $V^{(1)}\otimes V^{(2)}$ (see \cite{FHL93}, Proposition~4.6.1). Again
\begin{equation*}
\wt(w_1\otimes w_2)=\wt(w_1)+\wt(w_2)
\end{equation*}
for two homogeneous vectors $w_1\in W^{(1)}$, $w_2\in W^{(2)}$. This means in particular that
\begin{equation*}
\rho(W^{(1)}\otimes W^{(2)})=\rho(W^{(1)})+\rho(W^{(2)})
\end{equation*}
for irreducible modules $W^{(1)}$ and $W^{(2)}$. Note that $W^{(1)}\otimes W^{(2)}$ is also irreducible, as the following proposition shows:
\begin{prop}[\cite{FHL93}, Section~4.7, \cite{DMZ94}, Proposition~2.7]
Let $V^{(1)}$ and $V^{(2)}$ be two \voa{}s. The irreducible $V^{(1)}\otimes V^{(2)}$-modules are up to isomorphism exactly tensor-product modules of the form $W^{(1)}\otimes W^{(2)}$ where $W^{(1)}$, $W^{(2)}$ are irreducible modules for $V^{(1)}$, $V^{(2)}$, respectively.

Furthermore, the tensor-product \voa{} $V^{(1)}\otimes V^{(2)}$ is rational if and only if both $V^{(1)}$ and $V^{(2)}$ are rational.
\end{prop}
Note that in \cite{FHL93} modules are defined to be only $\Q$-graded, which requires the authors to add a rationality condition in \cite{FHL93}, Theorem~4.7.4, since the tensor product of two modules with non-rational grading could have rational grading.

\begin{prop}[\cite{DLM97}, Proposition~3.3]
Let $V^{(1)}$ and $V^{(2)}$ be two regular \voa{}s. Then also the tensor-product \voa{} $V^{(1)}\otimes V^{(2)}$ is regular.
\end{prop}

\begin{prop}[\cite{Abe13}, Proposition~5.1]
Given two \voa{}s $V^{(1)}$ and $V^{(2)}$ the tensor-product \voa{} $V^{(1)}\otimes V^{(2)}$ is $C_2$-cofinite if and only if both $V^{(1)}$ and $V^{(2)}$ are $C_2$-cofinite.
\end{prop}

There is the following result on spaces of intertwining operators for tensor-product modules:
\begin{prop}[\cite{DMZ94}, Proposition~2.10]
Let $W^{(i),1},W^{(i),2},W^{(i),3}$ be irreducible modules for the \voa{} $V^{(i)}$, $i=1,2$. Assume that the spaces of intertwining operators of type $\binom{W^{(i),3}}{W^{(i),1}\,W^{(i),2}}$, $i=1,2$, are finite-dimensional. Then
\begin{equation*}
\V^{W^{(1),3}\otimes W^{(2),3}}_{W^{(1),1}\otimes W^{(2),1}\,W^{(1),2}\otimes W^{(2),2}}\cong\V^{W^{(1),3}}_{W^{(1),1}\,W^{(1),2}}\otimes_\C\V^{W^{(2),3}}_{W^{(2),1}\,W^{(2),2}}.
\end{equation*}
\end{prop}
Everything we wrote down in this section generalises in the obvious way to the case of the tensor product of finitely many \voa{}s or modules.

\section{Zhu's Modular Invariance of Trace Functions}\label{sec:zhu}

In this section we present a very strong result stating that the trace functions and characters of \voa{}s are under certain circumstances vector-valued modular forms for $\SLZ$. Using the Verlinde formula (see Theorem~\ref{thm:verlinde} below) this result is a powerful tool for determining the fusion rules amongst the irreducible modules of a \voa{}, such as for simple-current \voa{}s (see Section~\ref{sec:scvoa}).

\minisec{Trace Functions}

Let $W$ be a module of a \voa{} $V$. We define $o(v)$ to be the grade-preserving operator in $\End_\C(W)$ associated with $v\in V$ defined by linear continuation of $o(v)=v_{\wt(v)-1}$, the $(\wt(v)-1)$-th mode of $Y_{W}(v,x)$, for homogeneous $v$ with respect to the weight grading on $V$.

Let $V$ be a \voa{} and $W$ an irreducible $V$-module. Recall that $\rho(W)\in\C$ denotes the conformal weight of $W$ and $c\in\C$ the central charge of $V$. We will study the \emph{trace function} (or \emph{graded trace}) of $W$
\begin{equation*}
T_W(\cdot,q)\colon V\to q^{\rho(W)-c/24}\C[[q]]
\end{equation*}
defined by
\begin{equation*}
T_{W}(v,q):=\tr_{W}o(v)q^{L_0-c/24}=q^{\rho(W)-c/24}\sum_{n=0}^\infty\tr_{W_{\rho(W)+n}}o(v)q^n,
\end{equation*}
which is linear in $v\in V$ and takes values in the formal power series in $q$ with some complex shift $\rho(W)-c/24\in\C$ in the exponents.

Replacing $q$ by $q_\tau=\e^{2\pi\i\tau}$ for $\tau\in\H$ we can view
\begin{equation*}
T_{W}(v,\tau):=T_{W}(v,q_\tau)
\end{equation*}
as a function
\begin{equation*}
T_W\colon V\times\H\to\C
\end{equation*}
on $V$ times the upper half-plane $\H$ if it is well-defined, i.e.\ if the sum converges. Using his $C_2$-cofiniteness Zhu was able to show that the trace functions $T_{W}(v,\tau)$ are indeed well-defined and even holomorphic on $\H$. Moreover, he established that these trace functions exhibit a certain modular invariance property.

\minisec{Modular Invariance}

Given a \voa{} $V$, in \cite{Zhu96} the author introduced a second \voa{} structure on $V$ with a new grading
\begin{equation*}
V=\bigoplus_{k\in\Z}V_{[k]},
\end{equation*}
which facilitates the formulation of his modular invariance result. The weight of a homogeneous vector $v\in V$ with respect to this second grading is denoted by $\wt[v]$, as opposed to $\wt(v)$ for the original grading $V=\bigoplus_{k\in\Z}V_{k}$.

We let $\SLZ$ act on the upper half-plane $\H$ via the usual \emph{Möbius transformation}
\begin{equation*}
M.\tau:=\frac{a\tau +b}{c\tau +d}
\end{equation*}
for $M=\left(\begin{smallmatrix}a&b\\c&d\end{smallmatrix}\right)\in\SLZ$, $\tau\in\H$. Slightly reformulated\footnote{Dong, Li and Mason also removed some unnecessary assumptions made by Zhu.} to match the presentation in \cite{DLM00} Zhu's modular invariance result reads:
\begin{thm}[\cite{Zhu96}, Theorem~5.3.2, Theorem~4.4.1]\label{thm:zhumodinv}
Let $V$ be a rational, $C_2$-cofinite \voa{} of central charge $c$.
Then:
\begin{enumerate}
\item The trace functions $T_{W}(v,\tau)=\tr_{W}o(v)q_\tau^{L_0-c/24}$, $W\in\Irr(V)$, $v\in V$, are well-defined, holomorphic functions on $\H$ (in $\tau$).
\item There is a representation
\begin{equation*}
\rho_V\colon\SLZ\to\GL(\V(V))
\end{equation*}
of $\SLZ$ on the finite-dimensional $\C$-vector space $\V(V)=\bigoplus_{W\in \Irr(V)}\C W$ such that for each $M=\left(\begin{smallmatrix}a&b\\c&d\end{smallmatrix}\right)\in\SLZ$, $W\in\Irr(V)$,
\begin{equation*}
(c\tau +d)^{-k}T_{W}(v,M.\tau)=\sum_{X\in\Irr(V)} \rho_V(M)_{W,X}T_{X}(v,\tau)
\end{equation*}
if $v\in V_{[k]}$ for $k\in\Z$.
\end{enumerate}
\end{thm}

The above theorem shows that the trace functions $T_{W}(v,\tau)$, $W\in \Irr(V)$, form a vector-valued modular form of weight $k=\wt[v]$ for $\rho_V$ under $\SLZ$. We remark that the $T_{W}(v,\tau)$, $W\in \Irr(V)$, are linearly independent, allowing $\tau$ \emph{and} $v$ to vary.
This implies that the above representation $\rho_V$ is unique and we call $\rho_V$ \emph{Zhu's representation} for $V$.

\minisec{Characters}

We discuss one special choice of $v$, namely $v=\vac$, the vacuum vector in $V$. This vector fulfils the homogeneity condition with $\wt[\vac]=0$ and $o(\vac)=\id_W$, the identity operator on the module $W$. Then the trace function
\begin{equation*}
\ch_{W}(\tau):=T_{W}(\vac,\tau)=\tr_{W}q_\tau^{L_0-c/24}=q_\tau^{\rho(W)-c/24}\sum_{n=0}^\infty\dim_\C(W_{\rho(W)+n})q_\tau^n
\end{equation*}
is called the \emph{character} (or \emph{graded dimension}) of the module $W$. Since a module $W$ and its contragredient module $W'$, which in general are not isomorphic, have the same grading, their characters are identical and hence this shows that for fixed $v$ the trace functions need not be linearly independent.

\minisec{Congruence Subgroup Property}
It is a central question whether the representation $\rho_V$ of $\SLZ$ on $\V(V)$ becomes trivial on a congruence subgroup of $\SLZ$, i.e.\ whether there is an integer $N\in\Ns$ such that
\begin{equation*}
\Gamma(N)\leq\ker(\rho_V)
\end{equation*}
where $\Gamma(N)$ is the principal congruence subgroup
\begin{equation*}
\Gamma(N)=\left\{\left(\begin{smallmatrix}a&b\\c&d\end{smallmatrix}\right)\in\SLZ\xmiddle|a=d=1\pmod{N}\text{ and }b=c=0\pmod{N}\right\}.
\end{equation*}
The minimal such $N$ is called the level of $\rho_V$. In this case, the trace functions $T_{W}(v,\tau)$, $W\in \Irr(V)$, are scalar-valued modular forms for a congruence subgroup of $\SLZ$ of level $N$.

In general, i.e.\ under the assumptions of Zhu's theorem, this problem is still open. It is not even known that $\ker(\rho_V)$ has finite index in $\SLZ$. However, if Assumption~\ref{ass:n} holds, then Huang has shown that there is the structure of a \mtc{} on the modules of $V$ (see Theorem~\ref{thm:voamtc}). This can be used to show the congruence subgroup property:
\begin{thm}[\cite{DLN15}, Theorem~I]\label{thm:thm1dln12}
Let $V$ satisfy Assumption~\ref{ass:n} and have central charge $c$.
Let $v\in V$ be homogeneous with respect to Zhu's second grading. Then the trace functions $T_{W}(v,\tau)$, $W\in \Irr(V)$, are modular forms of weight $\wt[v]$ for a congruence subgroup of $\SLZ$ of level $N$ where $N$ is the smallest positive integer such that $N(\rho(W)-c/24)\in\Z$ for all $W\in \Irr(V)$.
\end{thm}
Note that by \cite{DLM00}, Theorem~11.3, both $c$ and the $\rho(W)$ are rational, which is why the number $N\in\Ns$ in the above theorem always exists.

\section{\texorpdfstring{$S$}{S}-Matrix, \texorpdfstring{$T$}{T}-Matrix and Verlinde Formula}

It is a well-known fact that the matrices
\begin{equation*}
S=\begin{pmatrix}0&-1\\1&0\end{pmatrix}\quad\text{and}\quad T=\begin{pmatrix}1&1\\0&1\end{pmatrix}
\end{equation*}
generate the group $\SLZ$. In fact, $\SLZ$ is obtained from the free group generated by $S$ and $T$ via the relations
\begin{equation*}
(ST)^3=S^2\quad\text{and}\quad S^4=\id.
\end{equation*}
We also define $Z:=S^2=-\id$.

Let us continue in the setting of the previous section, i.e.\ we let $V$ be a rational, $C_2$-cofinite \voa{} and consider the representation $\rho_V$ from Zhu's modular invariance theorem. We introduce the $S$- and $T$-matrix
\begin{equation*}
\S:=\rho_V(S)\quad\text{and}\quad\T:=\rho_V(T)
\end{equation*}
in $\GL(\V(V))$, which have complex entries and already fully determine Zhu's representation $\rho_V$. By Theorem~\ref{thm:zhumodinv} and the comment on the uniqueness of the representation $\rho_V$, the matrices $\S$ and $\T$ depend only on $V$. By definition, the transformation behaviour of the trace functions under $\tau\mapsto S.\tau=-1/\tau$ and $\tau\mapsto T.\tau=\tau+1$ is
\begin{align*}
(1/\tau)^kT_{W}(v,-1/\tau)&=\sum_{X\in\Irr(V)} \S_{W,X}T_{X}(v,\tau),\\
T_{W}(v,\tau+1)&=\sum_{X\in\Irr(V)} \T_{W,X}T_{X}(v,\tau).
\end{align*}

The transformation behaviour under $T$ can be easily read off from the definition of the trace functions. One obtains:
\begin{prop}\label{prop:tmatrix}
Let $V$ be a rational, $C_2$-cofinite \voa{} of central charge $c$. Then
\begin{equation*}
\T_{X,Y}=\delta_{X,Y}\e^{(2\pi\i)(\rho(X)-c/24)}
\end{equation*}
for $X,Y\in \Irr(V)$. In particular, $\T$ is diagonal.
\end{prop}

To determine the $S$-matrix is in general a much harder problem. In the following we collect some known properties of $\S$. For convenience let us also introduce $\mathcal{C}:=\S^2=\rho_V(Z)$.

\minisec{Verlinde Formula}

The following remarkable property, establishing a relation between the entries of the $S$-matrix and the fusion rules of a \voa{} $V$, was in this context first mentioned by Verlinde \cite{Ver88} and proved by Moore and Seiberg \cite{MS89} and Huang \cite{Hua08}:
\begin{thm}[Verlinde Formula, \cite{Hua08}, Section~5]\label{thm:verlinde}
Let $V$ satisfy Assumption~\ref{ass:n}.
Then $\S$ is symmetric and the square $\mathcal{C}=\S^2$ is a permutation matrix which shifts $W$ to the contragredient module $W'$. Moreover, $\S_{V,W}\neq 0$ for $W\in\Irr(V)$ and the formula
\begin{equation}\tag{Verlinde formula}
N_{X,Y}^W=\sum_{U\in \Irr(V)}\frac{\S_{X,U}\S_{Y,U}\S_{U,W'}}{\S_{V,U}}
\end{equation}
for the fusion rules $N_{X,Y}^W$ of $V$ holds.
\end{thm}

\minisec{Further Properties of the $S$-Matrix}
Using results on quantum dimensions (see Section~\ref{sec:mtc}) from \cite{DJX13} one can show:
\begin{prop}\label{prop:pospos}
Let $V$ satisfy Assumptions~\ref{ass:n}\ref{ass:p}. Then
\begin{equation*}
\S_{W,V}=\S_{V,W}\in\R_{>0}
\end{equation*}
for all $W\in \Irr(V)$.
\end{prop}
The proposition needs the positivity assumption (Assumption~\ref{ass:p}), which will be introduced in the next chapter.
\begin{proof}
We consider Zhu's modular invariance for $W=V\in\Irr(V)$, $M=S\in\SLZ$, $v=\vac$ and $\tau=\i y$ for some $y>0$, i.e.\
\begin{equation*}
\ch_{V}(\i/y)=\sum_{W\in\Irr(V)}\S_{V,W}\ch_{W}(\i y)=\S_{V,V}\sum_{W\in\Irr(V)}\frac{\S_{V,W}}{\S_{V,V}}\ch_{W}(\i y).
\end{equation*}
It follows from the definition that $\ch_{V}(\i y)>0$ for any $y>0$ since $c,\rho(W)\in\Q$ for all $W\in \Irr(V)$ by Theorem~\ref{thm:11.3}. Moreover, under the assumptions of the proposition, Lemma~4.2 in \cite{DJX13} states that $\S_{V,W}/\S_{V,V}>0$ for all $W\in \Irr(V)$. Together with the above equation this implies that $\S_{V,V}>0$ and hence $\S_{W,V}>0$ for all $W\in \Irr(V)$.
\end{proof}

Moreover, one can show:
\begin{prop}[\cite{DLN15}, Corollary~3.6]\label{prop:sunitary}
Let $V$ satisfy Assumption~\ref{ass:n}. Then $\S^\dagger=\overline{\S}=\S\mathcal{C}=\S^3$. In particular $\S^\dagger\S=\S\S^\dagger=\id$, i.e.\ $\S$ is unitary.
\end{prop}

We can immediately deduce that the representation $\rho_V$ is unitary under the assumptions of the above proposition.
\begin{cor}\label{cor:rhounitary}
Let $V$ satisfy Assumption~\ref{ass:n}. Then the representation $\rho_V\colon\SLZ\to\GL(\V(V))$ is unitary.
\end{cor}
\begin{proof}
We just saw that under the given assumptions $\S$ is unitary. As stated above, the $T$-matrix is in general a diagonal matrix with entries $\e^{(2\pi\i)(\rho(W)-c/24)}$, which makes the matrix unitary since $c$ and all $\rho(W)$ are in $\Q$ by Theorem~\ref{thm:11.3}. Altogether, the representation $\rho_V\colon\SLZ\to\GL(\V(V))$ is unitary since $S$ and $T$ generate $\SLZ$.
\end{proof}

Finally, we can use the relation $TSTST=S$ to prove the following formula for the $S$-matrix involving the conformal weights:
\begin{prop}\label{prop:ststst}
Let $V$ be a rational, $C_2$-cofinite \voa{} of central charge $c$. Then
\begin{equation*}
\S_{X,Y}=\sum_{W\in \Irr(V)}\S_{X,W}\S_{W,Y}\e^{(2\pi\i)(\rho(X)+\rho(Y)+\rho(W)-c/8)}
\end{equation*}
for all $X,Y\in\Irr(V)$.
\end{prop}
\begin{proof}
This follows directly from the corresponding formula $\S=\T\S\T\S\T$ and the explicit formula for $\T$ in Proposition~\ref{prop:tmatrix}.
\end{proof}

\section{Twisted Modules}\label{sec:twmod}

We also need the concept of twisted modules for \voa{}s. A good historical overview of the development of twisted modules can be found in the introduction of \cite{DL96}. In the definition of twisted modules, the Jacobi identity is twisted by an automorphism of the \voa{}. Twisted vertex operators were first studied in \cite{Lep85,FLM87} in the context of \voa{}s associated with lattices. These twisted vertex operators satisfy the twisted Jacobi identity as described in \cite{FLM87,FLM88,Lep85,Lep88}. The important case where the \voa{} isomorphism originates from the $(-1)$-isometry on the lattice is essential in the construction of the Moonshine module $V^\natural$ \cite{FLM84,FLM87,FLM88}. The general notion of twisted module was first introduced in \cite{FFR91,Don94}.

\minisec{Automorphisms of \VOA{}s}

For convenience, let us repeat the definition of \voa{} automorphisms (see Definition~\ref{defi:hom}). An automorphism of a \voa{} $V$ is a vector-space automorphism $g\in\Aut_\C(V)$ such that $gY(a,x)g^{-1}=Y(ga,x)$ for all $a\in V$, $g\vac=\vac$ and $g\omega=\omega$. By definition, $gY(\omega,x)=Y(\omega,x)g$, i.e.\ $g$ commutes with the modes of $\omega$, such as $L_0=\omega_1$. In particular, any \voa{} automorphism leaves the grading invariant, i.e.\ $\wt(gv)=\wt(v)$ for homogeneous $v\in V$.

Let $g$ be an automorphism of $V$ of finite order $n\in\Ns$. Then $V$ decomposes as
\begin{equation*}
V=\bigoplus_{r\in\Z_n}V^r
\end{equation*}
into eigenspaces
\begin{equation*}
V^r:=\{v\in V\;|\;gv=\xi_n^rv\}
\end{equation*}
of $g$ where $\xi_n:=\e^{(2\pi\i)1/n}$ is a primitive $n$-th root of unity.

\minisec{Twisted Modules}
We are now ready to define twisted modules. Note that there is a subtlety in the definition of twisted modules for \voa{}s:
\begin{rem}[Sign Convention]\label{rem:convtwmod}
What is defined as a twisted module in \cite{DLM00} for instance (newer definition) is defined as the twisted module for the inverse automorphism in \cite{DL96,DLM98} for instance (older definition). In this text, we will use the definition from \cite{DLM00} since this makes the modular invariance theorem (see Theorem~\ref{thm:1.4}) have the expected form, as mentioned in \cite{DLM00}, Remark~3.1.
\end{rem}

Moreover, we allow twisted modules to be $\C$-graded as in \cite{DLM98,DLM00} and not just $\Q$-graded as in older definitions, such as in \cite{DL96}.

\begin{defi}[Twisted \VOA{} Module]
Let $V$ be a \voa{} of central charge $c$ and let $g$ be an automorphism of $V$ of order $n\in\Ns$. A $g$-twisted $V$-module $W$ is given by the data:
\begin{itemize}
\item (\emph{space of states}) a $\C$-graded vector space
\begin{equation*}
W=\bigoplus_{\lambda\in\C}W_\lambda
\end{equation*}
with \emph{weight} $\wt(w)=\lambda$ for $w\in W_\lambda$, $\dim_\C(W_\lambda)<\infty$ for all $\lambda\in\C$ and $W_\lambda=\{0\}$ for $\lambda$ ``sufficiently small in the sense of modifications by $(1/n)\Z$'', i.e.\ for fixed $\lambda\in\C$, $W_{\lambda+k/n}=\{0\}$ for sufficiently negative $k\in\Z$,
\item (\emph{vertex operators}) a linear map
\begin{equation*}
Y_W(\cdot,x)\colon V\to\End_\C(W)[[x^{\pm{1/n}}]]
\end{equation*}
taking each $a\in V$ to a field
\begin{equation*}
a\mapsto Y_W(a,x)=\sum_{k\in(1/n)\Z}a_kx^{-k-1}
\end{equation*}
where for each $w\in W$, $a_kw=0$ for $k$ sufficiently large or equivalently a map $Y_W(\cdot,x)\colon V\otimes_\C W\to W((x^{1/n}))$.
\end{itemize}
These data are subject to the following axioms:
\begin{itemize}
\item (\emph{twisting compatibility}) For $a\in V^r$,
\begin{equation*}
Y_W(a,x)=\sum_{k\in -r/n+\Z}a_nx^{-k-1}.
\end{equation*}
\item (\emph{left vacuum axiom}) $Y_W(\vac,x)=\id_W$.
\item (\emph{translation axiom}) For any $a\in V$,
\begin{equation*}
Y_W(Ta,x)=\partial_x Y_W(a,x).
\end{equation*}
\item (\emph{twisted Jacobi identity}) For $a\in V^r$, $b\in V$,
\begin{align*}
&\iota_{x_1,x_0}x_2^{-1}\left(\frac{x_1-x_0}{x_2}\right)^{r/n}\delta\left(\frac{x_1-x_0}{x_2}\right)Y_W(Y(a,x_0)b,x_2)\\
&=\iota_{x_1,x_2}x_0^{-1}\delta\left(\frac{x_1-x_2}{x_0}\right)Y_W(a,x_1)Y_W(b,x_2)\\
&\quad-\iota_{x_2,x_1}x_0^{-1}\delta\left(\frac{x_2-x_1}{-x_0}\right)Y_W(b,x_2)Y_W(a,x_1).
\end{align*}
\item (\emph{Virasoro relations}) The modes $L_k^W:=\omega_{k+1}^W$ of
\begin{equation*}
Y_W(\omega,x)=\sum_{k\in\Z}\omega_k^Wx^{-k-1}=\sum_{k\in\Z}L_k^Wx^{-k-2}
\end{equation*}
satisfy the \emph{Virasoro relations} at central charge $c$, i.e.\
\begin{equation*}
[L^W_m,L^W_k]=(m-k)L^W_{m+k}+\frac{m^3-m}{12}\delta_{m+k,0}\id_Wc
\end{equation*}
for $m,k\in\Z$. Moreover, $L^W_0w=\wt(w)w$ for homogeneous $w\in W$.
\end{itemize}
\end{defi}

Analogously to the untwisted case one can also define \emph{weak} and \emph{admissible $g$-twisted modules} \cite{DLM98}.

\minisec{Elementary Categorical Notions}

\emph{Homomorphisms}, \emph{isomorphisms}, \emph{endomorphisms} and \emph{automorphism} between $g$-twisted modules for the same \voa{} are defined as expected. The same applies to \emph{quotient modules}, \emph{direct-sum modules}, \emph{irreducible} modules, \emph{completely reducible} modules, etc.

\minisec{Schur's Lemma}

Schur's lemma (c.f.\ Proposition~\ref{prop:schur}) also holds in the twisted case:
\begin{prop}[Schur's Lemma, Twisted Version]\label{prop:twistedschur}
Let $V$ be a \voa{}, $g$ an automorphism of $V$ of finite order and $W$ an irreducible $g$-twisted $V$-module. Then
\begin{equation*}
\End_V(W)=\C.
\end{equation*}
\end{prop}
\begin{cor}\label{cor:twistedschur}
Let $V$ be a \voa{}, $g$ an automorphism of $V$ of finite order and $W^1$ and $W^2$ irreducible $g$-twisted $V$-modules. Then
\begin{equation*}
\Hom_V(W^1,W^2)\cong\begin{cases}\C&\text{if }W^1\cong W^2,\\0&\text{if }W^1\ncong W^2.\end{cases}
\end{equation*}
\end{cor}

\minisec{Conformal Weight}

As in the untwisted case, we can define a conformal weight for irreducible twisted modules (see \cite{DLM00}, Section~3):

\begin{defi}[Conformal Weight]\label{defi:sec3}
Let $V$ be a \voa{} and $g$ an automorphism of $V$ of order $n\in\Ns$. Let $W$ be an irreducible $g$-twisted $V$-module. Then $W$ has grading
\begin{equation*}
W=\bigoplus_{\lambda\in\rho+(1/n)\N} W_{\lambda}=\bigoplus_{k=0}^\infty W_{\rho+k/n}
\end{equation*}
for some $\rho\in\C$, called \emph{conformal weight} of $W$ and denoted by $\rho(W)$.
\end{defi}

\minisec{Holomorphicity}

Holomorphic \voa{}s have by definition only one irreducible module up to isomorphism. If the \voa{} is $C_2$-cofinite, then the same is true for the twisted modules. As for untwisted modules, there is also a statement about the rationality of the conformal weights, which we only need for the special case of holomorphic \voa{}s.

\begin{thm}[\cite{DLM00}, Theorem~1.2 or Theorems 10.3 and 11.1]\label{thm:10.3}
Let $V$ be a holomorphic \voa{} satisfying condition $C_2$. Let $g$ be an automorphism of $V$ of finite order. Then $V$ possesses a unique irreducible $g$-twisted $V$-module up to isomorphism, call it $V(g)$. Moreover, the conformal weight of $V(g)$ is rational.
\end{thm}

In Theorem~\ref{thm:confwnn} we prove an improvement of the second statement by showing that the conformal weight of $V(g)$ lies in $(1/n^2)\Z$ if $n$ is the order of $g$.

\minisec{Contragredient Module}

\emph{Contragredient modules} are defined exactly in the same way as in the untwisted case (see \cite{DLM98}, Section~3). One has the following result:
\begin{prop}[\cite{DLM98}, Lemma~3.7]\label{prop:lem3.7}
Let $V$ be a \voa{} and $g$ an automorphism of $V$ of finite order. Let $W$ be an admissible $g$-twisted $V$-module. Then the contragredient module $W'$ is an admissible $g^{-1}$-twisted $V$-module.
\end{prop}
The analogous statement is true for (ordinary) twisted modules.

\section{Dong, Li and Mason's Modular Invariance}\label{sec:dlmmodinv}
Dong, Li and Mason have generalised Zhu's modular invariance to twisted modules \cite{DLM00}. We mainly need the special case for holomorphic \voa{}s and cyclic automorphism groups: we consider a holomorphic, $C_2$-cofinite \voa{} $V$ and cyclic group $G=\langle\sigma\rangle$ of automorphisms of $V$ of order $n\in\Ns$. In this case, by Theorem~\ref{thm:10.3}, for each $i\in\Z_n$ there is an up to isomorphism unique irreducible $\sigma^i$-twisted $V$-module, called $V(\sigma^i)$, with some conformal weight $\rho_i:=\rho(V(\sigma^i))\in\Q$.

Consider the automorphisms $\phi_i(\sigma^j)$ of the vector spaces $V(\sigma^i)$ obeying
\begin{equation*}
\phi_i(\sigma^j)Y_{V(\sigma^i)}(v,x)\phi_i^{-1}(\sigma^j)=Y_{V(\sigma^i)}(\sigma^jv,x)
\end{equation*}
for all $i,j\in\Z_n$, $v\in V$, introduced in Section~\ref{sec:fpvosa} (and~\ref{sec:duality}), which are unique up to a scalar factor. Then for $g=\sigma^i,h=\sigma^j\in G$, $i,j\in\Z_n$, we consider the \emph{twisted trace function}
\begin{equation*}
T(\cdot,i,j,q)\colon V\to q^{\rho_i-c/24}\C[[q^{1/n}]]
\end{equation*}
defined by
\begin{equation*}
T(v,i,j,q):=\tr_{V(\sigma^i)} o(v)\phi_i(\sigma^j)q^{L_0-c/24}=q^{\rho_i-c/24}\!\!\!\sum_{k\in((i,n)/n)\N}\!\!\!\tr|_{V(\sigma^i)_{\rho_i+k}}o(v)\phi_i(\sigma^j)q^{k}
\end{equation*}
for $v\in V$. Again, it is not clear that this formal power series in $q^{(i,n)/n}$ with some shift $\rho_i-c/24$ in the exponents is well-defined viewed as a function on the upper half-plane $\H$ by substituting $q\mapsto\e^{2\pi\i\tau}$.

We let $\SLZ$ act from the right on $\Z_n\times\Z_n$ by matrix multiplication. The modular invariance result for the twisted trace functions reads:
\begin{thm}[\cite{DLM00}, Theorem~1.4]\label{thm:1.4}
Let $V$ be a holomorphic, $C_2$-cofinite \voa{} and $G=\langle\sigma\rangle$ a cyclic group of automorphisms of $V$ of order $n\in\Ns$. Then:
\begin{enumerate}
\item For each $v\in V$ and each pair of automorphisms $g=\sigma^i,h=\sigma^j\in G$, $i,j\in\Z_n$, the twisted trace function $T(v,i,j,\tau)$ is a well-defined, holomorphic function on $\H$.
\item The twisted trace functions satisfy
\begin{equation*}
(c\tau+d)^{-k}T(v,i,j,M.\tau)=\sigma(i,j,M)T(v,(i,j)M,\tau)
\end{equation*}
for each $M=\left(\begin{smallmatrix}a&b\\c&d\end{smallmatrix}\right)\in\SLZ$ and for each homogeneous $v\in V_{[k]}$ with respect to Zhu's second grading, $k\in\Z$, where $\sigma(i,j,M)\in\C$ are constants only depending on $i$, $j$ and $M$ and on the choice of $\phi_i(\sigma^j)$ for the different $i$ and $j$.
\end{enumerate}
\end{thm}
This shows that the $T(v,i,j,\tau)$, $i,j\in\Z_n$, form a vector-valued modular form of weight $k=\wt[v]$ for some representation of $\SLZ$ on $\C[\Z_n\times\Z_n]$.

\minisec{\texorpdfstring{$S$}{S}-Transformation}
The constants $\lambda_{i,j}:=\sigma(i,j,S)\in\C$ for $S=\left(\begin{smallmatrix}0&-1\\1&0\end{smallmatrix}\right)$ describing the $S$-transformation will be of particular importance in Chapter~\ref{ch:fpvosa}. Explicitly,
\begin{equation}\label{eq:lambda}
(1/\tau)^kT(v,i,j,-1/\tau)=\lambda_{i,j}T(v,j,-i,\tau)
\end{equation}
for all $i,j\in\Z_n$, homogeneous $v\in V_{[k]}$ and $k\in\Z$. The constants $\lambda_{i,j}$ depend only on $i$, $j$ and the choice of the $\phi_i$.

Consider the non-zero trace function
\begin{equation*}
\ch_{V(\sigma^i)}(\tau):=T(\vac,i,0,\tau)=\tr_{V(\sigma^i)} q_\tau^{L_0-c/24}=q_\tau^{\rho_i-c/24}\!\!\!\sum_{k\in((i,n)/n)\N}\!\!\!\dim_\C(V(\sigma^i)_{\rho_i+k})q_\tau^k.
\end{equation*}
Applying the $S$-transformation four times gives back the original function with a factor of $\lambda_{i,0}\lambda_{0,-i}\lambda_{-i,0}\lambda_{0,i}$, which has to be 1 and hence we can conclude that $\lambda_{i,0}\neq0\neq\lambda_{0,i}$ for all $i\in\Z_n$.

As we will see in Remark~\ref{rem:phi0}, $\phi_0(\sigma^j)=\sigma^j$ is a possible choice for the automorphisms $\phi_0(\sigma^j)$ on the untwisted $V$-module $V(\sigma^0)\cong V$. For the orbifold theory developed in Chapter~\ref{ch:fpvosa} we need to determine the value of the $\lambda_{i,j}$ under certain assumptions. At this point we can deduce the following result:
\begin{lem}\label{lem:lambdapos}
Let $V$ be a holomorphic, $C_2$-cofinite \voa{} and let $G=\langle\sigma\rangle$ be a cyclic group of automorphisms of $V$ of order $n\in\Ns$. Assume in addition that $V$ is of CFT-type. If we choose $\phi_0(\sigma^j)=\sigma^j$, then $\lambda_{0,j}\in\R_{>0}$, $j\in\Z_n$.
\end{lem}
\begin{proof}
From Theorem~\ref{thm:1.4} we obtain
\begin{equation*}
T(\vac,0,j,-1/\tau)=\lambda_{0,j}T(\vac,j,0,\tau),
\end{equation*}
where both sides are holomorphic functions on $\H$. The left-hand side is given by
\begin{equation*}
T(\vac,0,j,-1/\tau)=\tr_{V}\sigma^j\e^{(2\pi\i)(-1/\tau)(L_0-c/24)}=\sum_{k=0}^\infty \e^{(2\pi\i)(-1/\tau)(k-c/24)}\tr_{V_k}\sigma^j|_{V_k},
\end{equation*}
the right-hand side is
\begin{equation*}
\lambda_{0,j}T(\vac,j,0,\tau)=\lambda_{0,j}\tr_{V(\sigma^j)}q_\tau^{L_0-c/24}=\lambda_{0,j}\!\!\!\sum_{k\in\rho_j+((j,n)/n)\N}\!\!\!\e^{(2\pi\i\tau)(k-c/24)}\dim_\C(V(\sigma^j)_k)
\end{equation*}
where $\rho_j\in\Q$ is the conformal weight of $V(\sigma^j)$. We specialise to $\tau=\i/t$ with $t\in\R_{>0}$ and obtain
\begin{equation*}
\frac{1}{\lambda_{0,j}}\sum_{k=0}^\infty\!\!\!\e^{-(2\pi t)(k-c/24)}\tr_{V_k}\sigma^j|_{V_k}=\sum_{k\in\rho_j+((j,n)/n)\N}\e^{-(2\pi/t)(k-c/24)}\dim_\C(V(\sigma^j)_k).
\end{equation*}
For convenience, we bring $c$ on the right-hand side and get
\begin{equation*}
\frac{1}{\lambda_{0,j}}\sum_{k=0}^\infty \e^{-2\pi t k}\tr_{V_k}\sigma^j|_{V_k}=\e^{2\pi(1/t-t)c/24}\!\!\!\sum_{k\in\rho_j+((j,n)/n)\N}\!\!\!\e^{-2\pi k/t}\dim_\C(V(\sigma^j)_k)\in\R_{\geq0}
\end{equation*}
since $k,c\in\R$ (by Proposition~\ref{prop:div8} and Definition~\ref{defi:sec3}) and the dimensions are non-negative. Since $\R_{\geq0}$ is closed in $\C$, this means that the limit for $t\to\infty$ of the left-hand side is in $\R_{\geq0}$ if it exists.

The limit of the left-hand side is given by
\begin{align*}
\lim_{t\to\infty}\frac{1}{\lambda_{0,j}}\sum_{k=0}^\infty \e^{-2\pi t k}\tr_{V_k}\sigma|_{V_k}&\stackrel{!}{=}\frac{1}{\lambda_{0,j}}\sum_{k=0}^\infty\lim_{t\to\infty}\e^{-2\pi t k}\tr_{V_k}\sigma|_{V_k}\\
&=\frac{1}{\lambda_{0,j}}\sum_{k=0}^\infty\delta_{k,0}\tr_{V_k}\sigma|_{V_k}=\frac{1}{\lambda_{0,j}}\tr_{V_0}\sigma|_{V_0}\\
&=\frac{1}{\lambda_{0,j}},
\end{align*}
where we used that $V_0=\C\vac$ and $\sigma\vac=\vac$ and assumed that we are allowed to interchange the infinite sum and the limit. Hence $\lambda_{0,j}\in\R_{>0}$.

It remains to show that the summation and limit-taking processes may be swapped. For this we use the dominated convergence theorem for series. The summands
\begin{equation*}
a_t^{(k)}:=\e^{-2\pi t k}\tr_{V_k}\sigma|_{V_k}
\end{equation*}
are dominated by
\begin{equation*}
b^{(k)}:=\e^{-2\pi k}\dim_\C(V_k),
\end{equation*}
i.e.\
\begin{equation*}
|a_t^{(k)}|\leq b^{(k)}
\end{equation*}
for all $t\geq 1$, where we use that for a finite-order automorphism $A$ on a finite-dimensional $\C$-vector space of dimension $m$, $|\tr(A)|=|\sum_{i=1}^m\mu_i|\leq\sum_{i=1}^m|\mu_i|=\sum_{i=1}^m 1=m$ where the $\mu_i$ are the eigenvalues of $A$ counted with algebraic multiplicities. Each of the limits
\begin{equation*}
\lim_{t\to\infty}a_t^{k}=\lim_{t\to\infty}\e^{-2\pi t k}\tr_{V_k}\sigma|_{V_k}=\delta_{k,0}\tr_{V_k}\sigma|_{V_k},
\end{equation*}
$k\in\N$, exists and hence we are allowed to interchange the infinite sum and the limit if the sum
\begin{equation*}
\sum_{k=0}^\infty b^{(k)}=\sum_{k=0}^\infty \e^{-2\pi k}\dim_\C(V_k)
\end{equation*}
converges. But this is nothing else but the Fourier expansion of $\e^{-2\pi c/24}T(\vac,0,0,\tau)$ evaluated at $\tau=i$.
\end{proof}

\chapter{Simple Currents}\label{ch:sc}

In this chapter we study \voa{}s whose irreducible modules are all simple currents and determine their fusion algebra. We show that under certain assumptions these \voa{}s have group-like fusion and that the conformal weights modulo~1 of the irreducible modules define a quadratic form on the fusion group. Moreover, in this case, Zhu's representation is up to a character the well-known Weil representation. We also describe simple-current extensions and their representation theory for a \voa{} with group-like fusion.

\section{Simple Currents}\label{sec:simplecurrents}
Let $V$ be a rational \voa{}. Then the fusion product $\boxtimes_V$ exists and is unique up to isomorphism for any two $V$-modules. We define:
\begin{defi}[Simple Current]
Let $V$ be a rational \voa{}. A $V$-module $U$ is called \emph{simple-current module} (or \emph{simple current} in short) if the tensor product $U\boxtimes_V W$ is irreducible
for every irreducible $V$-module $W$.
\end{defi}
Clearly, the adjoint module $V$ is always a simple current since it is the unit in the fusion algebra $\V(V)$. The following property holds:
\begin{prop}\label{prop:scirr}
Let $V$ be a rational \voa{}. Then the following are equivalent:
\begin{enumerate}
 \item $V$ is simple.
 \item Every simple-current module is irreducible.
\end{enumerate}
\end{prop}
\begin{proof}
Let $V$ be simple. Let $U$ be a simple-current module. Since $V$ is simple, i.e.\ irreducible, $U\cong U\boxtimes_V V$ has to be irreducible by the definition of simple-current module. Conversely, assume that every simple-current module is irreducible. Since $V$ is a simple-current module, it is irreducible, i.e.\ simple.
\end{proof}

\minisec{Positivity Assumption}

A number of results in this chapter and the subsequent ones depend on the following positivity assumption:
\begin{customass}{\textbf{\textsf{P}}}[Positivity Assumption]\label{ass:p}
Let $V$ be a simple \voa{}. Suppose that for all irreducible $V$-modules $W$, $\rho(W)>0$ if $W\not\cong V$ and $\rho(V)=0$.
\end{customass}

\minisec{$S$-Matrix for Simple-Current Modules}
Consider a simple, rational \voa{} with an irreducible module $U$. Whether $U$ is a simple current can under certain assumptions be read off from the $S$-matrix:
\begin{prop}[\cite{DJX13}, Proposition~4.17]\label{prop:4.17}
Let $V$ satisfy Assumptions~\ref{ass:n}\ref{ass:p}. Then an irreducible module $U\in \Irr(V)$ is a simple-current if and only if $\S_{V,U}=S_{V,V}$.
\end{prop}

The following is formula~(2) in \cite{SY89}, where it is proved in a different context.
\begin{prop}\label{prop:ssss}
Let $V$ satisfy Assumption~\ref{ass:n}. Let $U$ be a simple current. Then
\begin{equation*}
\S_{U,Y}\S_{X,Y}=\S_{V,Y}\S_{U\boxtimes X,Y}
\end{equation*}
for any two irreducible modules $X,Y\in \Irr(V)$.
\end{prop}
\begin{proof}
Since $U$ is a simple current, the Verlinde formula implies
\begin{equation*}
\delta_{U\boxtimes X,Z}=N_{U,X}^Z=\sum_{W\in \Irr(V)}\frac{\S_{U,W}\S_{X,W}\S_{Z',W}}{\S_{V,W}}
\end{equation*}
for $X,Z\in \Irr(V)$ and hence
\begin{align*}
\S_{U\boxtimes X,Y}&=\sum_{Z\in \Irr(V)}\S_{Z,Y}\delta_{U\boxtimes X,Z}=\sum_{Z,W\in \Irr(V)}\S_{Z,Y}\frac{\S_{U,W}\S_{X,W}\S_{Z',W}}{\S_{V,W}}\\
&=\sum_{W\in \Irr(V)}\frac{\S_{U,W}\S_{X,W}}{\S_{V,W}}\sum_{Z\in \Irr(V)}\S_{Z,Y}\S_{Z',W}.
\end{align*}
Using that $(\S^2)_{Z,L}=\delta_{Z',L}$ for $Z,L\in\Irr(V)$ we get
\begin{align*}
\sum_{Z\in \Irr(V)}\S_{Z,Y}\S_{Z',W}&=\sum_{Z,L\in \Irr(V)}\S_{Y,Z}\delta_{Z',L}\S_{L,W}=\sum_{Z,L\in \Irr(V)}\S_{Y,Z}(\S^2)_{Z,L}\S_{L,W}\\
&=(\S^4)_{Y,W}=\delta_{Y,W}
\end{align*}
and hence the above formula yields
\begin{align*}
\S_{U\boxtimes X,Y}&=\sum_{W\in \Irr(V)}\frac{\S_{U,W}\S_{X,W}}{\S_{V,W}}\delta_{Y,W}=\frac{\S_{U,Y}\S_{X,Y}}{\S_{V,Y}}.
\end{align*}
\end{proof}
There is a simple corollary:
\begin{cor}\label{cor:ss}
Let $V$ satisfy Assumptions~\ref{ass:n}\ref{ass:p}. Let $U$ be a simple current. Then
\begin{equation*}
\S_{V,X}=\S_{V,U\boxtimes X}
\end{equation*}
for all $X\in \Irr(V)$.
\end{cor}
\begin{proof}
We apply the above proposition for $Y=V$, i.e.\ $\S_{U,V}\S_{X,V}=\S_{V,V}\S_{U\boxtimes X,V}$. Then we use that $\S_{V,U}=S_{V,V}$ by Proposition~\ref{prop:4.17}. This gives $\S_{X,V}=\S_{U\boxtimes X,V}$ since $\S_{V,U}\neq 0$ by Proposition~\ref{prop:pospos}.
\end{proof}

The following is essentially formula~(5) in \cite{SY89}, which is a special case of a more general formula in \cite{BYZ89}.
\begin{prop}\label{prop:sbilform}
Let $V$ satisfy Assumptions~\ref{ass:n}\ref{ass:p}. Let $U$ be a simple current. Then
\begin{equation*}
\S_{X,U}=\S_{X,V}\e^{(2\pi\i)(\rho(X)+\rho(U)-\rho(U\boxtimes X))}
\end{equation*}
for all $X\in \Irr(V)$.
\end{prop}
\begin{proof}
Consider the formula from Proposition~\ref{prop:ststst} for $Y=U$ and let $U$ be a simple current. Then Proposition~\ref{prop:ssss} implies
\begin{align*}
\S_{X,U}&=\sum_{W\in \Irr(V)}\S_{X,W}\S_{W,U}\e^{(2\pi\i)(\rho(X)+\rho(U)+\rho(W)-c/8)}\\
&=\sum_{W\in \Irr(V)}\S_{V,W}\S_{W,U\boxtimes X}\e^{(2\pi\i)(\rho(X)+\rho(U)+\rho(W)-c/8)}.
\end{align*}
On the other hand, consider the same formula for $Y=V$. This yields
\begin{equation*}
\S_{X,V}=\sum_{W\in \Irr(V)}\S_{X,W}\S_{W,V}\e^{(2\pi\i)(\rho(X)+\rho(W)-c/8)}
\end{equation*}
or equivalently
\begin{equation*}
\S_{X,V}\e^{(2\pi\i)(-\rho(X)+c/8)}=\sum_{W\in \Irr(V)}\S_{X,W}\S_{W,V}\e^{(2\pi\i)\rho(W)}.
\end{equation*}
Consequently, for the simple current $U$ we get
\begin{align*}
\S_{X,U}&=\e^{(2\pi\i)(\rho(X)+\rho(U)-c/8)}\sum_{W\in \Irr(V)}\S_{V,W}\S_{W,U\boxtimes X}\e^{(2\pi\i)\rho(W)}\\
&=\e^{(2\pi\i)(\rho(X)+\rho(U)-c/8)}\S_{U\boxtimes X,V}\e^{(2\pi\i)(-\rho(U\boxtimes X)+c/8)}\\
&=\S_{U\boxtimes X,V}\e^{(2\pi\i)(\rho(X)+\rho(U)-\rho(U\boxtimes X))}\\
&=\S_{X,V}\e^{(2\pi\i)(\rho(X)+\rho(U)-\rho(U\boxtimes X))},
\end{align*}
where we used Corollary~\ref{cor:ss} in the last step. 
\end{proof}

\section{Simple-Current \VOA{}s}\label{sec:scvoa}

For a simple, rational \voa{} $V$ we just saw in Proposition~\ref{prop:scirr} that
\begin{equation*}
\{\text{simple-current modules}\}\subseteq\{\text{irreducible modules}\}.
\end{equation*}
In the following we consider \voa{}s for which the converse is true, namely
\begin{equation*}
\{\text{irreducible modules}\}\subseteq\{\text{simple-current modules}\}.
\end{equation*}
\begin{defi}[Simple-Current \VOA{}]
Let $V$ be a rational \voa{}. $V$ is called \emph{simple-current \voa{}} if all irreducible $V$-modules are simple currents.
\end{defi}
For the purposes of this text we will usually consider simple simple-current \voa{}s. For these the irreducible modules are exactly the simple currents, i.e.\
\begin{equation*}
\{\text{irreducible modules}\}=\{\text{simple-current modules}\}.
\end{equation*}

Let us consider a simple-current \voa{} $V$. We index the finitely many isomorphism classes of irreducible $V$-modules by the set $F_V$, i.e.\
\begin{equation*}
\Irr(V)=\{W^\alpha\;|\;\alpha\in F_V\}
\end{equation*}
and, if necessary, choose representatives $W^\alpha$ for $\alpha\in F_V$. Then, since all irreducible modules are simple currents, for any $\alpha,\beta\in F_V$,
\begin{equation*}
W^\alpha\boxtimes_V W^\beta\cong :W^{\alpha+\beta}
\end{equation*}
for a certain element $\alpha+\beta\in F_V$. This defines a binary operation $+\colon F_V\times F_V\to F_V$. In terms of the fusion rules,
\begin{equation*}
N_{W^\alpha,W^\beta}^{W^\gamma}=:N_{\alpha,\beta}^\gamma=\delta_{\alpha+\beta,\gamma}
\end{equation*}
for $\alpha,\beta,\gamma\in F_V$.
Since the fusion product $\boxtimes_V$ is symmetric, so is the binary operation $+$ on $F_V$, which justifies the usage of the symbol $+$.

If $V$ is in addition $C_2$-cofinite and of CFT-type, then by Proposition~\ref{prop:fusionassoc} we know that the fusion algebra is associative, which proves that also $+$ is associative. This shows that $+$ endows the set of indices $F_V$ with the structure of a commutative semigroup.

If we further assume that the \voa{} $V$ is simple, then the adjoint module $V$ is irreducible and hence $+$ has a neutral element, namely the index of the unit $V$ in $\V(V)$. In this case, $F_V$ even has the structure of a commutative monoid. For convenience, let us assume that $0\in F_V$ and that 0 is the index of $V$, i.e.\ $V=W^0$. Then $0$ is the neutral element for $+$.

In order to get the structure of an abelian group on $F_V$ we still need the existence of an inverse. This can be proved to exist using the Verlinde formula (see Theorem~\ref{thm:verlinde}) under the assumption that $V$ is in addition self-contragredient, i.e.\ $V\cong V'$. We will denote by $\alpha'$ the index of the contragredient module $W^{\alpha'}\cong(W^\alpha)'$ of $W^\alpha$, $\alpha\in F_V$. In total, one obtains (see also \cite{LY08}, Corollary~1):
\begin{prop}\label{prop:sca}
Let $V$ satisfy Assumption~\ref{ass:n}. Assume that all irreducible $V$-modules are simple currents. Then the fusion algebra $\V(V)$ of $V$ is the group algebra $\C[F_V]$ of some finite abelian group $(F_V,+)$, i.e.\
\begin{equation*}
W^\alpha\boxtimes_V W^\beta\cong W^{\alpha+\beta}
\end{equation*}
for all $\alpha,\beta\in F_V$ where the neutral element in $F_V$ is given by 0, the index of $V=W^0$, and the inverse of $\alpha$ in $F_V$ is given by $\alpha'=:-\alpha$ the index of the contragredient module $W^{\alpha'}\cong(W^\alpha)'$ of $W^\alpha$.
\end{prop}
\begin{defi}[Fusion Group]
In the situation of the above proposition we call the finite abelian group $F_V$ such that $\V(V)=\C[F_V]$ the \emph{fusion group} of $V$.
\end{defi}

The statement of the above proposition immediately follows from the following one. At this point we only need the $\implies$-direction of the equivalence.
\begin{prop}\label{prop:scvoa}
Let $V$ be a rational, $C_2$-cofinite \voa{} of CFT-type. Then the following are equivalent:
\begin{enumerate}
\item\label{enum:equiv1} $V$ is simple, self-contragredient and all irreducible $V$-modules are simple currents.
\item\label{enum:equiv2} $W'\boxtimes_V W\cong V$ for all $W\in \Irr(V)$.
\end{enumerate}
\end{prop}
\begin{proof}
First, assume that \ref{enum:equiv1} holds. For this direction of the proof we need the results from Theorem~\ref{thm:verlinde}, which is valid under the above assumptions. By the Verlinde formula we know that
\begin{equation*}
\delta_{X,Y}=N_{Y,V}^X=\sum_{W\in \Irr(V)}\S_{Y,W}\S_{W,X'}
\end{equation*}
since $V$ is the unit in $\V(V)$. Hence, using that $V$ is self-contragredient and the symmetry of $\S$, we get
\begin{equation*}
N_{X,X'}^V=\sum_{W\in \Irr(V)}\frac{\S_{X,W}S_{X',W}S_{W,V'}}{S_{V,W}}=\sum_{W\in \Irr(V)}S_{X,W}S_{X',W}=\delta_{X',X'}=1,
\end{equation*}
which shows that $X\boxtimes X'\cong V$ for all $X\in \Irr(V)$ since any $X\in \Irr(V)$ is a simple current. This is \ref{enum:equiv2}.

Conversely, assume that \ref{enum:equiv2} holds. For $V$ we get $V'\cong V'\boxtimes V\cong V$, so that $V$ is self-contragredient. Then consider the module $U\cong X\boxtimes_V Y$ for some $X,Y\in \Irr(V)$. First we show that $U$ is non-zero. Assume that $U=\{0\}$. Then, using associativity and that $V$ is the unit, we get
\begin{equation*}
\{0\}\cong X'\boxtimes_V \{0\} \cong X'\boxtimes_VX\boxtimes_V Y\cong V\boxtimes_VY\cong Y,
\end{equation*}
a contradiction. Hence the fusion product of any two irreducible $V$-modules is non-zero. The same statement immediately follows for not necessarily irreducible modules.

Consider again the product $U\cong X\boxtimes_V Y$. To show that it is irreducible we consider
\begin{equation*}
X'\boxtimes_V U\cong X'\boxtimes_V X\boxtimes_V Y\cong V\boxtimes_V Y\cong Y.
\end{equation*}
Decomposing $U$ into its irreducible components $U\cong \bigoplus_{W\in \Irr(V)}k_W W$ for some $k_W\in\N$ with $\sum_{W\in \Irr(V)}k_W\geq1$ we obtain
\begin{equation*}
Y\cong X'\boxtimes_V U\cong \bigoplus_{W\in \Irr(V)}k_W(X'\boxtimes_V W).
\end{equation*}
The terms in the brackets on the right-hand side are all non-zero and so for $Y$ to be irreducible all $k_W$ except for one have to vanish. This means that $U$ is irreducible.

The above argument for the product $W'\boxtimes W\cong V$ for some $W\in \Irr(V)$ also yields that $V$ is irreducible.
\end{proof}

\minisec{Assumptions}

\Voa{}s with the properties in Proposition~\ref{prop:sca} play an important rôle in this text. For convenience, let us define the following assumption:
\begin{customass}{\textbf{\textsf{SN}}}[Group-Like Fusion]\label{ass:sn}
Let $V$ satisfy Assumption~\ref{ass:n}, i.e.\ $V$ is a simple, rational, $C_2$-cofinite, self-contragredient \voa{} of CFT-type. Moreover, assume that $V$ is a simple-current \voa{}, i.e.\ all irreducible $V$-modules are simple currents. Then we know that the irreducible $V$-modules are given by $\Irr(V)=\{W^\alpha\;|\;\alpha\in F_V\}$ and the fusion algebra $\V(V)$ is the group algebra $\C[F_V]$ of the fusion group $(F_V,+)$ with the inverse given by the index of the contragredient module.
\end{customass}
The assumption also includes a choice of representatives $W^\alpha$ for $\alpha\in F_V$ with $V=W^0$.

\minisec{\FQS{}}

We just saw that for a \voa{} $V$ satisfying Assumption~\ref{ass:sn} the fusion algebra $\V(V)$ is the group algebra $\C[F_V]$ of the group $F_V$ of indices of irreducible modules, i.e.\ $\Irr(V)=\{W^\alpha\;|\;\alpha\in F_V\}$.

In the following we will see that the fusion group $F_V$ from the above considerations admits an additional structure, namely that of a \fqs{} (see Appendix~\ref{ch:fqs}, especially Definition~\ref{defi:fqs}), i.e.\ there is a natural non-degenerate finite quadratic form on $F_V$. This quadratic form will be given by the conformal weights modulo~1 of the irreducible modules.

We write
\begin{equation*}
Q_\rho(\alpha):=\rho(W^\alpha)+\Z\in\Q/\Z,
\end{equation*}
$\alpha\in F_V$, for the conformal weight of $W^\alpha$ modulo~1, i.e.\ we view $Q_\rho$ as a function $Q_\rho\colon F_V\to\Q/\Z$. Let us also define the symmetric function $B_\rho\colon F_V\times F_V\to\Q/\Z$ by
\begin{equation*}
B_\rho(\alpha,\beta):=Q_\rho(\alpha+\beta)-Q_\rho(\alpha)-Q_\rho(\beta)\in\Q/\Z.
\end{equation*}
The function $Q_\rho$ is the candidate for the quadratic form on $F_V$ and $B_\rho$ would then be its associated bilinear form. For later convenience we also introduce the functions $q_\rho\colon F_V\to\C^\times$ and $b_\rho\colon F_V\times F_V\to\C^\times$ defined by
\begin{align*}
\e^{(2\pi\i)Q_\rho(\alpha)}&=:q_\rho(\alpha),\\
\e^{(2\pi\i)B_\rho(\alpha,\beta)}&=:b_\rho(\alpha,\beta)
\end{align*}
for $\alpha,\beta\in F_V$.

First, we derive some more properties of the $S$-matrix of $V$ under the assumption that $V$ is a simple-current \voa{}. Let us write $\S_{\alpha,\beta}:=\S_{W^\alpha,W^\beta}$ for $\alpha,\beta\in F_V$.
\begin{prop}\label{prop:ssca}
Let $V$ satisfy Assumptions~\ref{ass:sn}\ref{ass:p}. Then
\begin{equation*}
\S_{0,0}=S_{0,\alpha}=\frac{1}{\sqrt{|F_V|}}
\end{equation*}
for all $\alpha\in F_V$.
\end{prop}
\begin{proof}
The first equality is simply Proposition~\ref{prop:4.17} since all $W^\alpha$, $\alpha\in F_V$, are simple currents by assumption. For the second statement we consider
\begin{equation*}
1=\delta_{0,-0}=(\S^2)_{0,0}=\sum_{\gamma\in F_V}\S_{0,\gamma}\S_{\gamma,0}=|F_V|\S_{0,0}^2.
\end{equation*}
The statement follows since $\S_{0,0}>0$ by Proposition~\ref{prop:pospos}.
\end{proof}

We immediately get the following important result, namely a closed formula for the entries of the $S$-matrix depending only on the conformal weights of the irreducible $V$-modules modulo~1.
\begin{prop}\label{prop:sbilform2}
Let $V$ satisfy Assumptions~\ref{ass:sn}\ref{ass:p}. Then
\begin{equation*}
\S_{\alpha,\beta}=\frac{1}{\sqrt{|F_V|}}\e^{(2\pi\i)(Q_\rho(\alpha)+Q_\rho(\beta)-Q_\rho(\alpha+\beta))}=\frac{1}{\sqrt{|F_V|}}\e^{-(2\pi\i)B_\rho(\alpha,\beta)}
\end{equation*}
for all $\alpha,\beta\in F_V$.
\end{prop}
\begin{proof}
This follows directly from Propositions \ref{prop:sbilform} and \ref{prop:ssca}.
\end{proof}

Propositions \ref{prop:ssss} and \ref{prop:ssca} immediately yield the formula
\begin{equation*}
\S_{\alpha,\gamma}\S_{\beta,\gamma}=\frac{1}{\sqrt{|F_V|}}\S_{\alpha+\beta,\gamma}
\end{equation*}
for all $\alpha,\beta,\gamma\in F_V$. Considering the formula from the above proposition this means exactly that $B_\rho$ is linear in the first argument and by symmetry of $\S$ also in the second argument. Hence $B_\rho$ is a finite bilinear form. We also know that $Q_\rho(0)=\rho(V)+\Z=0+\Z$ and $Q_\rho(\alpha)=\rho(W^\alpha)+\Z=\rho((W^\alpha)')+\Z=Q_\rho(-\alpha)$. Then Proposition~\ref{prop:bilquad} implies that $Q_\rho\colon F_V\to\Q/\Z$ is a finite quadratic form with associated bilinear form $B_\rho$.
\begin{thm}\label{thm:fafqs}
Let $V$ satisfy Assumptions~\ref{ass:sn}\ref{ass:p}. Then the function $Q_\rho\colon F_V\to\Q/\Z$ given by the conformal weights modulo~1 of the irreducible $V$-modules is a finite quadratic form on $F_V$ and its associated bilinear form $B_\rho\colon F_V\times F_V\to\Q/\Z$ is non-degenerate, i.e.\ $(F_V,Q_\rho)$ admits the structure of a \fqs{}.
\end{thm}
\begin{proof}
It only remains to show the non-degeneracy of $B_\rho$. Consider
\begin{equation*}
\delta_{\alpha+\beta,0}=\delta_{\alpha,-\beta}=(\S^2)_{\alpha,\beta}=\sum_{\gamma\in F_V}\S_{\alpha,\gamma}\S_{\gamma,\beta}=\frac{1}{|F_V|}\sum_{\gamma\in F_V}\e^{(2\pi\i)B_\rho(\alpha+\beta,\gamma)}.
\end{equation*}
Hence
\begin{equation*}
\delta_{\alpha,0}=\frac{1}{|F_V|}\sum_{\gamma\in F_V}\e^{(2\pi\i)B_\rho(\alpha,\gamma)}.
\end{equation*}
Now assume that $B_\rho(\alpha,\gamma)=0+\Z$ for all $\gamma\in F_V$. Then the above formula immediately gives $\alpha=0$, which proves that $B_\rho$ is non-degenerate.
\end{proof}
In the following we will write $F_V=(F_V,Q_\rho)$ for the fusion group together with the quadratic form $Q_\rho$ on it. We conclude with a few remarks:
\begin{rem}\label{rem:fafqs}
\item
\begin{enumerate}
\item Lemma~\ref{lem:rationallocality} below directly proves the existence of a quadratic form under Assumption~\ref{ass:sn}, i.e.\ Assumption~\ref{ass:p} is not required. However, in the above considerations the positivity assumption ensures that the relation between the $S$-matrix and the bilinear form $B_\rho$ is as described, which otherwise might have some additional minus sign.
In particular, it is not clear that the quadratic form is non-degenerate if we drop Assumption~\ref{ass:p}.

\item Recall that the exponents of the formal variable of the intertwining operators of type $\binom{W^\gamma}{W^\alpha\,W^\beta}$ lie in $\rho(W^\gamma)-\rho(W^\alpha)-\rho(W^\beta)+\Z$. Since we are in the simple-current situation, intertwining operators only exist for $\gamma=\alpha+\beta$ and hence the exponents lie exactly in $Q_\rho(\alpha+\beta)-Q_\rho(\alpha)-Q_\rho(\beta)=B_\rho(\alpha,\beta)$.

\item Any given \fqs{}, i.e.\ a finite abelian group with a non-degenerate quadratic form on it, can be realised as fusion group of some \voa{}. Indeed, for a given \fqs{} $D$ it is always possible to find a positive-definite, even lattice $L$ with discriminant form $L'/L\cong D$, which is also the fusion group of the associated lattice \voa{} $V_L$ (see Proposition~\ref{prop:cor12.10}), i.e.\ $F_{V_L}\cong D$ as \fqs{}s.
\end{enumerate}
\end{rem}

\minisec{Weil Representation}

Let us return to Zhu's representation $\rho_V\colon\SLZ\to\GL(\C[F_V])$ in the simple-current situation, i.e.\ we consider a \voa{} satisfying Assumptions~\ref{ass:sn}\ref{ass:p}. The $S$- and $T$-matrices associated with this representation are given by
\begin{align*}
\rho_V(S)_{\alpha,\beta}&=\S_{\alpha,\beta}=\frac{1}{\sqrt{|F_V|}}\e^{-(2\pi\i)B_\rho(\alpha,\beta)},\\
\rho_V(T)_{\alpha,\beta}&=\T_{\alpha,\beta}=\delta_{\alpha,\beta}\e^{(2\pi\i)(Q_\rho(\alpha)-c/24)}
\end{align*}
in terms of the quadratic form $Q_\rho$ and the non-degenerate associated bilinear form $B_\rho$ on the fusion group $F_V$.

We contrast this to the well-known Weil representation $\rho_D$ with respect to a given \fqs{} $D$ (see Section~\ref{sec:weil}). Since the \fqs{} structure on the fusion group can have odd signature, we have to consider the Weil representation of the metaplectic group $\MpZ$, the double-cover of $\SLZ$. Via the natural covering map $\MpZ\to\SLZ$ we can view $\rho_V$ as a representation of $\MpZ$ on $\C[F_V]$. Then:
\begin{align*}
\rho_V(\widetilde{S})_{\alpha,\beta}&=\frac{1}{\sqrt{|F_V|}}\e^{-(2\pi\i)B_\rho(\alpha,\beta)},\\
\rho_V(\widetilde{T})_{\alpha,\beta}&=\delta_{\alpha,\beta}\e^{(2\pi\i)(Q_\rho(\alpha)-c/24)}
\end{align*}
for the standard generators $\widetilde{S}$ and $\widetilde{T}$ of $\MpZ$. The Weil representation $\rho_{F_V}$ on the other hand is given by
\begin{align*}
\rho_{F_V}(\widetilde{S})_{\alpha,\beta}&=\frac{1}{\sqrt{|F_V|}}\e^{(2\pi\i)(-B_\rho(\alpha,\beta)-\sign(F_V)/8)},\\
\rho_{F_V}(\widetilde{T})_{\alpha,\beta}&=\delta_{\alpha,\beta}\e^{(2\pi\i)Q_\rho(\alpha)}.
\end{align*}
Here $\sign(F_V)\in\Z_8$ is the signature of the \fqs{} $F_V$. An immediate consequence is the following:
\begin{prop}
Let $V$ satisfy Assumptions~\ref{ass:sn}\ref{ass:p}. Then the \fqs{} $F_V$ from Theorem~\ref{thm:fafqs} has signature
\begin{equation*}
\sign(F_V)=c\pmod{8}
\end{equation*}
where $c$ is the central charge of the \voa{} $V$.
\end{prop}
\begin{proof}
Consider the scalar representation $\chi\colon\MpZ\to\C^\times$ defined by
\begin{equation*}
\chi((M,\varphi))\id:=\rho_V((M,\varphi))(\rho_{F_V}((M,\varphi)))^{-1}
\end{equation*}
for any $(M,\varphi)\in\MpZ$, where it is clear from the above formulæ that the right-hand side is a scalar multiple of the identity matrix. In fact,
\begin{equation*}
\chi(\widetilde{S})=\e^{(2\pi\i)\sign(F_V)/8}\quad\text{and}\quad\chi(\widetilde{T})=\e^{(2\pi\i)(-c/24)}.
\end{equation*}
Now the relation $\widetilde{S}^2=(\widetilde{S}\widetilde{T})^3$ gives the desired result.
\end{proof}
An important corollary is the following:
\begin{cor}
Let $V$ satisfy Assumptions~\ref{ass:sn}\ref{ass:p}. Then the central charge $c\in\Z$.
\end{cor}
Note that in general, even for a \voa{} satisfying Assumption~\ref{ass:n}, it is only known that $c\in\Q$ (cf. Theorem~\ref{thm:11.3}). Moreover, from the above proof we can directly deduce the following theorem:
\begin{thm}\label{thm:zhuweil}
Let $V$ be a \voa{} satisfying Assumptions~\ref{ass:sn}\ref{ass:p}. Then the representation $\rho_V\colon\SLZ\to\GL(\C[F_V])$ is given by
\begin{equation*}
\rho_V(M)=\eps(\widetilde{M})^{-c}\rho_{F_V}(\widetilde{M})
\end{equation*}
for all $M\in\SLZ$ where $\widetilde{M}\in\MpZ$ is the image of $M$ under the standard embedding of $\SLZ$ into $\MpZ$, $\rho_{F_V}$ is the Weil representation
of $\MpZ$ on $\C[F_V]$ and $\eps\colon\MpZ\to U_{24}$ is a character of $\MpZ$ defined by
\begin{equation*}
\eps(\widetilde{S}):=\e^{(2\pi\i)(-1/8)}\quad\text{and}\quad\eps(\widetilde{T}):=\e^{(2\pi\i)1/24}.
\end{equation*}
\end{thm}
Note that if $24\mid c$, then Zhu's representation $\rho_V$ and the Weil representation $\rho_{F_V}$ on $\C[F_V]$ coincide.

\begin{rem}
Let us for simplicity assume that $c$ is even so that the Weil representation descends to a representation of $\SLZ$. It is a well-known fact that the Weil representation acts trivially under $\Gamma(N)$ where $N$ is the level of the \fqs{} $(F_V,Q_\rho)$. This essentially proves Theorem~\ref{thm:thm1dln12}, the congruence subgroup property, in the special case where all irreducible $V$-modules are simple currents.
\end{rem}

We compare the character $\eps\colon\MpZ\to\C^\times$ with the modular-transformation properties of the Dedekind eta function
\begin{equation*}
\eta(\tau)=q_\tau^{1/24}\prod_{n=1}^\infty (1-q_\tau^n).
\end{equation*}
The eta function transforms as
\begin{align*}
\eta(T.\tau)&=\e^{(2\pi\i)1/24}\eta(\tau),\\
\eta(S.\tau)&=\tau^{1/2}\e^{(2\pi\i)(-1/8)}\eta(\tau).
\end{align*}
This means that for $M=\left(\begin{smallmatrix}a&b\\c&d\end{smallmatrix}\right)\in\SLZ$,
\begin{equation*}
\eta(M.\tau)=(c\tau+d)^{1/2}\epsilon(M)\eta(\tau)
\end{equation*}
for a certain function $\epsilon\colon\SLZ\to U_{24}$ (see e.g.\ \cite{Apo90}, Theorem~3.4). The function $\epsilon$ is not a character on $\SLZ$ since it fails to be a homomorphism. However, if we define the character $\eps\colon\MpZ\to U_{24}$ on the metaplectic double cover $\MpZ$ via $\eps(\widetilde{T}):=\epsilon(T)$ and $\eps(\widetilde{S}):=\epsilon(S)$, then $\eps(\widetilde{M})=\eps(M)$ for all $M\in\SLZ$ and
\begin{equation*}
\eta(M.\tau)=(c\tau+d)^{1/2}\eps(\widetilde{M})\eta(\tau),
\end{equation*}
meaning that $\eta$ is a modular form of weight $1/2$ for $\MpZ$ and character $\eps\colon\SLZ\to U_{24}$. Moreover, by definition of $\eps$, it is clear that this is exactly the same character as the $\eps$ occurring in the above theorem. This implies the following nice corollary to Theorems \ref{thm:zhumodinv} and \ref{thm:zhuweil}:
\begin{cor}\label{cor:zhuweil}
Let $V$ be as in Assumptions~\ref{ass:sn}\ref{ass:p} of central charge $c\in\Z$. Let $F_V$ be the fusion group of $V$, a \fqs{} by Theorem~\ref{thm:fafqs}. Then the functions
\begin{equation*}
F_W(\tau):=T_W(v,\tau)\eta(\tau)^c
\end{equation*}
for $W\in\Irr(V)$ form a vector-valued modular form of weight $k=\wt[v]+c/2$ for the Weil representation $\rho_{F_V}$ of $\MpZ$, i.e.\
\begin{equation*}
(c\tau +d)^{-k}F_W(M.\tau)=\sum_{X\in\Irr(V)} \rho_{F_V}(\widetilde{M})_{W,X}F_X(v,\tau)
\end{equation*}
for $M=\left(\begin{smallmatrix}a&b\\c&d\end{smallmatrix}\right)\in\SLZ$.
\end{cor}

\section{Simple-Current Extensions}\label{sec:sce}

We need some results on simple-current extensions. Most of them are well-known.
\begin{defi}[\VOA{} Extension]
Let $V^0$ be a \voa{}. Another \voa{} $V$ is called an \emph{extension} of $V^0$ if $V$ contains a \vosa{} isomorphic to $V^0$ (having the same vacuum and Virasoro vector as $V$, see Definition~\ref{defi:subvoa}). The extension is called \emph{simple} if $V$ is simple as a \voa{}.
\end{defi}

Now, let $V^0$ be a simple and rational \voa{} and $I$ an abelian group. Let $\{V^\alpha\;|\;\alpha\in I\}$ be a set of irreducible $V^0$-modules, indexed by the additive group $I$ such that the \voa{} $V^0$ is indexed by the neutral element of $I$.
\begin{lem}[\cite{SY03}, Lemma~3.1]
Assume that the direct sum $\bigoplus_{\alpha\in I} V^\alpha$ carries a \voa{} structure such that $0\neq V^\alpha\cdot V^\beta:=\spn_\C\{v_n w\;|\;v\in V^\alpha,w\in V^\beta,n\in\Z\}\subseteq V^{\alpha+\beta}$, i.e.\ $V$ is $I$-graded. This \voa{} is simple if and only if all $V^\alpha$, $\alpha\in I$, are non-isomorphic, irreducible $V^0$-modules.
\end{lem}
This motivates the following definition:
\begin{defi}[Graded Extension]
Let $V^0$ be a simple and rational \voa{}. An \emph{$I$-graded extension} $V_I$ of $V^0$ is a simple extension of $V^0$ of the form $V_I=\bigoplus_{\alpha\in I}V^\alpha$ where all $V^\alpha$, $\alpha\in I$, are non-isomorphic irreducible $V^0$-modules and the vertex operators in $I$ satisfy $Y(v,x)w\in V^{\alpha+\beta}((x))$ for any $v\in V^\alpha$ and $w\in V^\beta$.
\end{defi}
The rationality of $V^0$ implies that if an $I$-graded extension $V_I$ of $V^0$ exists, then $I$ is a \emph{finite} abelian group. Also, all the $V^0$-modules $V^\alpha$, $\alpha\in I$, have to be $\Z$-graded, i.e.\ their conformal weight $\rho(V^\alpha)$ has to be in $\Z$.

Finally, we define simple-current extensions:
\begin{defi}[Simple-Current Extension]\label{defi:sce}
If in addition in the above definition all $V^\alpha$, $\alpha\in I$, are simple-current $V^0$-modules, then $V_I$ is called \emph{$I$-graded simple-current extension} of $V^0$.
\end{defi}

In an $I$-graded extension $V_I$, $V^\alpha\cdot V^\beta\subseteq V^{\alpha+\beta}$ by definition. This means that the vertex operation on $V_I$ restricts to intertwining operators of type $\binom{W^{\alpha+\beta}}{W^\alpha\,W^\beta}$, implying that the space $\V_{W^\alpha,W^\beta}^{W^{\alpha+\beta}}$ is at least one-dimensional. For an $I$-graded simple-current extension this clearly implies $N_{W^\alpha,W^\beta}^{W^\gamma}=\delta_{\alpha+\beta,\gamma}$ and hence
\begin{equation*}
W^\alpha\boxtimes_{V_0}W^\beta\cong W^{\alpha+\beta}.
\end{equation*}

We summarise some properties of simple-current extensions. The following is essentially Proposition~1 in \cite{LY08}:
\begin{prop}[\cite{ABD04}, \cite{DM04b}, Section~5, \cite{Yam04}, Lemma~2.6, \cite{Lam01}, Theorem~4.5]\label{prop:scerc2}
Let $V^0$ be a simple, rational \voa{}. Let $V_I=\bigoplus_{\alpha\in I}V^\alpha$ be an $I$-graded simple-current extension of $V^0$. Then:
\begin{enumerate}
\item \label{enum:sce1} $V_I$ is rational.
\item \label{enum:sce2} If $V^0$ is $C_2$-cofinite and of CFT-type, then also $V_I$ is $C_2$-cofinite.
\item \label{enum:sce3}(Uniqueness) If $\hat V_I=\bigoplus_{\alpha\in I}\hat V^\alpha$ is another simple-current extension of $V^0$ such that $\hat V^\alpha\cong V^\alpha$ for all $\alpha\in I$, then $V_I$ and $\hat V_I$ are isomorphic \voa{}s.
\end{enumerate}
\end{prop}
Item \ref{enum:sce3} shows that the \voa{} structure of the simple-current extension only depends on the isomorphism classes of the simple-current $V^0$-modules.

\minisec{Modules for Simple-Current Extensions}

In the following we will review the representation theory of simple-current extensions, which is developed in \cite{Lam01,Yam04}.
\begin{prop}[\cite{SY03}, Lemma~3.6]
Let $V_I$ be an $I$-graded simple-current extension of the simple and rational \voa{} $V^0$. Let $X$ be a $V_I$-module. Let $W$ be an irreducible $V^0$-submodule of $X$. Then all
\begin{equation*}
V^\alpha\cdot W=\spn_\C\{v_n w\;|\;v\in V^\alpha,w\in W,n\in\Z\},
\end{equation*}
$\alpha\in I$, are also irreducible $V^0$-submodules of $X$.
\end{prop}

Let $X$ be an irreducible $V_I$-module. Since $V^0$ is rational, there is always an irreducible $V^0$-submodule $W$ of $X$. Since $X$ is irreducible, we get
\begin{equation*}
X=V_I\cdot W=\bigoplus_{\alpha\in I}V^\alpha\cdot W.
\end{equation*}
In view of the above proposition we define
\begin{equation*}
I_W:=\{\alpha\in I\;|\;V^\alpha\cdot W\cong W\},
\end{equation*}
which is a subgroup of $I$ since both $V^\alpha\cdot(V^\beta\cdot W)$ and $V^{\alpha+\beta}\cdot W$ are irreducible $V^0$-modules by the previous proposition and by associativity they are isomorphic.

\begin{defi}[$I$-Stability]
A $V_I$-module $X$ is said to be $I$-stable if $I_W=\{0\}$ for some irreducible $V^0$-submodule $W$ of $X$.
\end{defi}
Clearly, if $I_W=\{0\}$ for some irreducible $V^0$-submodule $W$ of $X$, then the same is true for all $V^0$-submodules of $X$. We will see that $I$-stable $V_I$-modules behave very nicely.
\begin{prop}[\cite{SY03}, Proposition~3.8]
Let $V_I$ be an $I$-graded simple-current extension of the simple, rational \voa{} $V^0$. Then the structure of every $I$-stable, irreducible $V_I$-module is completely determined by its $V^0$-module structure.
\end{prop}
The statement of the proposition also follows from the more general results Theorems 2.8 and 2.14 in \cite{Yam04}.

\begin{prop}[\cite{Yam04}, Lemma~2.16, \cite{SY03}, Lemma~3.12]\label{prop:scefusion}
Let $V_0$ be a simple, rational, $C_2$-cofinite \voa{} of CFT-type and let $V_I$ be an $I$-graded simple-current extension of $V^0$. Let $X^1,X^2,X^3$ be irreducible $I$-stable $V_I$-modules and let $W^1,W^2,W^3$ be $V^0$-submodules of $X^1,X^2,X^3$, respectively. Then there is the following isomorphism of spaces of intertwining operators
\begin{equation*}
\binom{X^3}{X^1\,X^2}_{V_I}\cong\bigoplus_{\alpha\in I}\binom{V^\alpha\boxtimes_{V^0}W^3}{W^1\,W^2}_{V^0}.
\end{equation*}
\end{prop}

Let $V_I$ be an $I$-graded simple-current extension of $V_0$. We consider the group $\hat{I}$ of characters $\chi\colon I\to\C^\times$. Then the elements $\chi$ of $\hat{I}$ define automorphisms of $V_I$ which leave $V^0$ pointwise invariant, i.e.\ $V^0\subseteq (V_I)^\chi$ for all $\chi\in\hat{I}$ where $(V_I)^\chi$ denotes the fixed points of $V_I$ under $\chi$. In fact the converse is also true:
\begin{prop}[\cite{Yam04}, Lemma~2.18]
An automorphism $\chi\in\Aut(V)$ satisfies $V^0\subseteq (V_I)^\chi$ if and only if $\chi\in\hat{I}$.
\end{prop}

Let $W$ be an irreducible $V^0$-module. We consider the function $\chi_W\colon I\to\C^\times$,
\begin{equation*}
\chi_W(\alpha):=\e^{(2\pi\i)(\rho(V^\alpha\boxtimes W)-\rho(W))}.
\end{equation*}
That this is a character, i.e.\ $\chi_W\in\hat{I}$, is shown under suitable assumptions in \cite{Yam04}, Lemma~3.1.
\begin{thm}[\cite{Yam04}, Theorems 3.2 and 3.3]\label{thm:irrmodsce}
Let $V_0$ be a simple, rational, $C_2$-cofinite \voa{} of CFT-type and let $V_I$ be an $I$-graded simple-current extension of $V^0$. Then every irreducible $V^0$-module $W$ is contained in an irreducible $\chi_W$-twisted $V_I$-module where $\chi_W$ is the character defined above.
\end{thm}
\begin{rem}
Note that in \cite{Yam04} the character $\chi_W$ is defined with opposite sign. It has to be chosen such that it is in agreement with the sign convention in the definition of twisted modules (see Remark~\ref{rem:convtwmod}).
\end{rem}

\minisec{Simple-Current Extensions for Simple-Current \VOA{}s}

For a simple-current extension $V_I=\bigoplus_{\alpha\in I}V^\alpha$ the irreducible $V_0$-modules $V^\alpha$, $\alpha\in I$, are all simple currents. However, this does not mean that \emph{all} irreducible $V_I$-modules are simple currents. Indeed, given a simple-current extension $V_I$ there is a subset $I$ of the isomorphism classes of irreducible $V^0$-modules such that all these modules are $\Z$-graded, simple currents and $I$ is closed under fusion.

In the following we want to study the case where all irreducible $V^0$-modules are simple currents. Let $V^0$ be a \voa{} satisfying Assumption~\ref{ass:sn}. Then the conformal weights define a quadratic form $Q_\rho$ on the fusion group $F_V$. Note that by Remark~\ref{rem:fafqs}, Assumption~\ref{ass:p} is not needed for the existence of the quadratic form, only for its non-degeneracy.

Clearly, any $I$-graded extension $V_I$ of $V^0$ will have to exist on a direct sum $V_I=\bigoplus_{\alpha\in I}W^\alpha$ for some isotropic subgroup $I$ of $F_V$ and this extension will automatically be a simple-current extension.

In the following we study the representation theory of such an extension $V_I$ and give a classification of the irreducible $V_I$-modules.

\begin{prop}\label{prop:aboveprop}
Let $V^0$ fulfil Assumption~\ref{ass:sn}. Assume that for some isotropic subgroup $I\leq F_V$ of the fusion group the direct sum $V_I=\bigoplus_{\alpha\in I}W^\alpha$ is an $I$-graded simple-current extension of $V^0$. Then any irreducible $V_I$-module $X$ is of the form
\begin{equation*}
X\cong\bigoplus_{\alpha\in I}W^{\alpha+\gamma}=:X^{\gamma+I}
\end{equation*}
for some $\gamma\in F_V$.
\end{prop}
\begin{proof}
As $V^0$-module, $X$ is a direct sum of irreducible $V^0$-modules. Since $X$ is non-zero, let us assume that $X$ contains $W^\gamma$ for some $\gamma\in F_V$. Then $W^\alpha\cdot W^\gamma\subseteq W^{\alpha+\gamma}$ has to be in $X$ for any $\alpha\in I$ and hence any $W^{\alpha+\gamma}$ is contained in the decomposition of $X$ into irreducible $V_0$-modules. This shows that
\begin{equation*}
X\supseteq\bigoplus_{\alpha\in I}W^{\alpha+\gamma}
\end{equation*}
up to isomorphism. Both sides are irreducible $V_I$-modules and hence we get equality up to isomorphism.
\end{proof}

On the other hand, one can ask the question which objects of the form $X^{\gamma+I}=\bigoplus_{\alpha\in I}W^{\alpha+\gamma}$ for some $\gamma\in F_V$ are irreducible $V_I$-modules. The answer is given by Theorem~\ref{thm:irrmodsce}.
\begin{thm}\label{thm:scemodclass}
Let $V^0$ fulfil Assumption~\ref{ass:sn} and let $V_I$ be as in Proposition~\ref{prop:aboveprop}. Then the irreducible (untwisted) $V_I$-modules are up to isomorphism exactly given by
\begin{equation*}
X^{\gamma+I}=\bigoplus_{\alpha\in I}W^{\alpha+\gamma}
\end{equation*}
for $\gamma\in I^\bot$, i.e.\ the irreducible $V_I$-modules are indexed by $I^\bot/I$.
\end{thm}
\begin{proof}
The above proposition gives all possible candidates for the irreducible $V_I$-modules. Consider $X^{\gamma+I}=\bigoplus_{\alpha\in I}W^{\alpha+\gamma}$ for some $\gamma\in F_V$. Then $X^{\gamma+I}$ contains the irreducible $V^0$-module $W^\gamma$. Theorem~\ref{thm:irrmodsce} on the other hand states that the $V_0$-module $W^\gamma$ is contained in an irreducible, possibly twisted $V_I$-module, which can only be $X^{\gamma+I}$ by the above proposition since no other of the $X^{\beta+I}$ for $\beta+I\in F_V/I$ contains $W^\gamma$. This $V_I$-module is twisted by the character of $I$
\begin{align*}
\chi_{W^\gamma}(\alpha)&=\e^{(2\pi\i)(\rho(W^\alpha\boxtimes W^\gamma)-\rho(W^\gamma))}\\
&=\e^{(2\pi\i)(Q_\rho(\alpha+\gamma)-Q_\rho(\gamma))}\\
&=\e^{(2\pi\i)(Q_\rho(\alpha+\gamma)-Q_\rho(\gamma)-Q_\rho(\alpha))}\\
&=\e^{(2\pi\i)B_\rho(\alpha,\gamma)},
\end{align*}
$\alpha\in I$, and is untwisted if and only if $B_\rho(\alpha,\gamma)=0+\Z$ for all $\alpha\in I$, i.e.\ $\gamma\in I^\bot$.
\end{proof}

We remark that the irreducible $V_I$-modules $X^{\gamma+I}$, $\gamma+I\in I^\bot/I$, have conformal weights in $Q_\rho(\alpha)$ for some $\alpha$ in $\gamma+I$. This makes sense since $\gamma\in I^\bot$ and hence $Q_\rho(\alpha)=Q_\rho(\gamma)$ for all $\alpha\in\gamma+I$.

We can also calculate the fusion rules of the irreducible $V_I$-modules
\begin{equation*}
\Irr(V_I)=\left\{X^{\alpha+I}\xmiddle|\alpha+I\in I^\bot/I\right\}.
\end{equation*}
The fusion group will turn out to be $I^\bot/I$, as expected.

\begin{lem}
Let $V^0$ fulfil Assumption~\ref{ass:sn} and let $V_I$ be as in Proposition~\ref{prop:aboveprop}. Then every irreducible $V_I$-module $X^{\alpha+I}$, $\alpha+I\in I^\bot/I$, is $I$-stable.
\end{lem}
\begin{proof}
We determine the $\beta\in I$ for which $W^\beta\cdot W^\alpha\cong W^\beta\boxtimes_{V^0} W^\alpha\cong W^{\beta+\alpha}$ is isomorphic to $W^\alpha$. This is clearly only the case for $\beta=0$ and hence $X^{\alpha+I}$ is an $I$-stable $V_I$-module.
\end{proof}

\begin{thm}\label{thm:scefusion}
Let $V^0$ fulfil Assumption~\ref{ass:sn} and let $V_I$ be as in Proposition~\ref{prop:aboveprop}. Then
\begin{equation*}
X^{\alpha+I}\boxtimes_{V_I} X^{\beta+I}\cong X^{\alpha+\beta+I}
\end{equation*}
for all $\alpha+I,\beta+I\in I^\bot/I$. In particular, $V_I$ is a simple-current \voa{}, i.e.\ all irreducible $V_I$-modules are simple currents.
\end{thm}
\begin{proof}
By Proposition~\ref{prop:scefusion},
\begin{equation*}
\binom{X^{\gamma+I}}{X^{\alpha+I}\,X^{\beta+I}}_{V_I}\cong\bigoplus_{\delta\in I}\binom{W^\delta\boxtimes_{V^0}W^\gamma}{W^\alpha\,W^\beta}_{V^0}=\bigoplus_{\delta\in I}\binom{W^{\delta+\gamma}}{W^\alpha\,W^\beta}_{V^0}.
\end{equation*}
Then
\begin{equation*}
N_{\alpha+I,\beta+I}^{\gamma+I}:=N_{X^{\alpha+I},X^{\beta+I}}^{X^{\gamma+I}}=\sum_{\delta\in I}\delta_{\alpha+\beta,\gamma+\delta}=\delta_{\alpha+\beta+I,\gamma+I}.
\end{equation*}
\end{proof}

Finally the following proposition holds:
\begin{prop}\label{prop:sceselfcon}
Let $V^0$ fulfil Assumption~\ref{ass:sn} and let $V_I$ be as in Proposition~\ref{prop:aboveprop}. Then
\begin{equation*}
(X^{\gamma+I})'\cong X^{-\gamma+I}
\end{equation*}
for $\gamma+I\in I^\bot/I$. In particular $V_I$ is self-contragredient.
\end{prop}
\begin{proof}
Consider
\begin{equation*}
(X^{\gamma+I})'=\left(\bigoplus_{\alpha\in I}W^{\alpha+\gamma}\right)'\cong\bigoplus_{\alpha\in I}(W^{\alpha+\gamma})'\cong\bigoplus_{\alpha\in I}W^{-\alpha-\gamma}=\bigoplus_{\alpha\in I}W^{\alpha-\gamma}=X^{-\gamma+I}.
\end{equation*}
Only the second step is non-trivial. But this follows directly from the definition of the contragredient module.
\end{proof}

In total, we have shown:
\begin{cor}\label{cor:scefusion}
Let $V^0$ fulfil Assumption~\ref{ass:sn} and let $V_I$ be as in Proposition~\ref{prop:aboveprop}. Then the fusion group of $V_I$ is $F_{V_I}=I^\bot/I$.
\end{cor}

\chapter{\AIA{}s}\label{ch:aia}

In this chapter we show that the irreducible modules of a \voa{} $V$ with group-like fusion form an \aia{}. If in addition the positivity assumption holds, the associated quadratic form is the negative of the quadratic form determined by the conformal weights. We also prove an extension theorem stating that an isotropic subgroup of the fusion group yields a \voa{} extending $V$.

\section{\AIA{}s}

\Aia{}s are important generalisations of \voa{}s and were first introduced in \cite{DL93}, Chapter~12 and \cite{DL94}. We begin by recalling some basic cohomological definitions.

\minisec{Cohomology of Abelian Groups}

We consider the \emph{cohomology of abelian groups} introduced by Eilenberg and Mac Lane \cite{ML52,Eil52,EML53,EML54} (see also \cite{DL93}, p.\ 130 for a summary).
\begin{defi}[Abelian 3-Cocycle]\label{defi:abeliancocycle}
Let $D$ be an abelian group. An \emph{abelian 3-cocycle} for $D$ with values in $\C^\times$ is a pair of maps
\begin{equation*}
F\colon D\times D\times D\to \C^\times\quad\text{and}\quad\Omega\colon D\times D\to \C^\times
\end{equation*}
satisfying
\begin{align*}
&F(\alpha,\beta,\gamma)F(\alpha,\beta,\gamma+\delta)^{-1}F(\alpha,\beta+\gamma,\delta)F(\alpha+\beta,\gamma,\delta)^{-1}F(\beta,\gamma,\delta)=1,\\
&F(\alpha,\beta,\gamma)^{-1}\Omega(\alpha,\beta+\gamma)F(\beta,\gamma,\alpha)^{-1}=\Omega(\alpha,\beta)F(\beta,\alpha,\gamma)^{-1}\Omega(\alpha,\gamma),\\
&F(\alpha,\beta,\gamma)\Omega(\alpha+\beta,\gamma)F(\gamma,\alpha,\beta)=\Omega(\beta,\gamma)F(\alpha,\gamma,\beta)\Omega(\alpha,\gamma)
\end{align*}
for all $\alpha,\beta,\gamma,\delta\in D$. The abelian 3-cocycle $(F,\Omega)$ is \emph{normalised} if
\begin{align*}
&F(\alpha,\beta,0)=F(\alpha,0,\gamma)=F(0,\beta,\gamma)=1,\\
&\Omega(\alpha,0)=\Omega(0,\beta)=1
\end{align*}
for all $\alpha,\beta,\gamma\in D$.
\end{defi}
The abelian 3-cocycles for $D$ with values in $\C^\times$ form a group under multiplication, denoted by $Z^3_\text{ab.}(D,\C^\times)$.

Given an abelian 3-cocycle $(F,\Omega)$ we define the map $B\colon D\times D\times D\to\C^\times$ by
\begin{equation*}
B(\alpha,\beta,\gamma):=F(\beta,\alpha,\gamma)^{-1}\Omega(\alpha,\beta)F(\alpha,\beta,\gamma)
\end{equation*}
for all $\alpha,\beta,\gamma\in D$. We also define $q_\Omega\colon D\to\C^\times$ by
\begin{equation*}
q_\Omega(\alpha):=\Omega(\alpha,\alpha)
\end{equation*}
for $\alpha\in D$. One can show that $q_\Omega$, called the \emph{trace} of $F$, is a quadratic form on $D$ \cite{ML52}. Moreover, the bilinear form $b_\Omega$ associated with $q_\Omega$ is given by
\begin{equation*}
b_\Omega(\alpha,\beta)=q_\Omega(\alpha+\beta)q_\Omega(\alpha)^{-1}q_\Omega(\beta)^{-1}=\Omega(\alpha,\beta)\Omega(\beta,\alpha)
\end{equation*}
for $\alpha,\beta\in D$. We also define the (additive) quadratic and bilinear forms $Q_\Omega\colon D\to\C/\Z$ and $B_\Omega\colon D\times D\to \C/\Z$ by
\begin{align}\begin{split}\label{eq:QOQr}
\e^{(2\pi\i)Q_\Omega(\alpha)}&:=q_\Omega(\alpha),\\
\e^{(2\pi\i)B_\Omega(\alpha,\beta)}&:=b_\Omega(\alpha,\beta)
\end{split}\end{align}
for $\alpha,\beta\in D$.

Clearly, if the abelian group $D$ is finite, then the values of $q_\Omega$ and $b_\Omega$ lie in some finite, cyclic subgroup of $\C^\times$ and hence have unit modulus (this does not have to be true for the values of $F$ and $\Omega$). Then the quadratic and bilinear forms $Q_\Omega$ and $B_\Omega$ are maps $Q_\Omega\colon D\to \Q/\Z$ and $B_\Omega\colon D\times D\to \Q/\Z$.

\begin{rem}\label{rem:omegatilde}
Let $(F,\Omega)$ be a (normalised) abelian 3-cocycle. Then so is $(F,\widetilde\Omega)$ where
\begin{equation*}
\widetilde\Omega(\alpha,\beta)=\Omega(\beta,\alpha)^{-1}.
\end{equation*}
Clearly, the traces fulfil $Q_{\widetilde\Omega}=-Q_\Omega$.
\end{rem}

\begin{defi}[Abelian 3-Coboundary]
Let $f\colon D\times D\to\C^\times$ be any function. Then $(F_f,\Omega_f)$ with
\begin{align*}
F_f(\alpha,\beta,\gamma)&:=f(\beta,\gamma)f(\alpha+\beta,\gamma)^{-1}f(\alpha,\beta+\gamma)f(\alpha,\beta)^{-1},\\
\Omega_f(\alpha,\beta)&:=f(\alpha,\beta)f(\beta,\alpha)^{-1}
\end{align*}
defines an abelian 3-cocycle.
We call $(F_f,\Omega_f)$ an \emph{abelian 3-coboundary}.
\end{defi}
We observe that two functions $f$ and $f'$ which only differ by a multiplicative constant define the same abelian 3-coboundary. We denote by $B^3_\text{ab.}(D,\C^\times)$ the group of abelian 3-coboundaries for $D$ with values in $\C^\times$. As usual, we say that two abelian 3-cocycles $(F,\Omega)$ and $(F',\Omega')$ are \emph{cohomologous} if there is an abelian 3-coboundary $(F_f,\Omega_f)$ such that $F'=F\cdot F_f$ and $\Omega'=\Omega\cdot\Omega_f$ and we define the \emph{cohomology classes}
\begin{equation*}
H^3_\text{ab.}(D,\C^\times):=Z^3_\text{ab.}(D,\C^\times)/B^3_\text{ab.}(D,\C^\times).
\end{equation*}

A simple calculation shows that the coboundary $(F_f,\Omega_f)$ is normalised if and only if
\begin{equation*}
f(\alpha,0)=f(0,\alpha)=c
\end{equation*}
for all $\alpha\in D$ for some constant $c\in\C^\times$. Every cohomology class has a normalised representative, i.e.\ every abelian 3-cocycle is cohomologous to a normalised one. Clearly, if $(F,\Omega)$ and $(F',\Omega')$ are two abelian 3-cocycles and $(F,\Omega)$ is normalised, then the product $(F\cdot F',\Omega\cdot\Omega')$ is normalised if and only if $(F',\Omega')$ is normalised.

Recall that $(F,\Omega)\mapsto q_\Omega$ is a map
\begin{equation*}
Z^3_\text{ab.}(D,\C^\times)\to\{\text{quadratic forms }D\to \C^\times\}
\end{equation*}
into the group of quadratic forms from $D$ into $\C^\times$, called trace map. It is clearly a homomorphism. By definition, abelian 3-coboundaries have trivial trace and hence if we multiply some abelian 3-cocycle $(F,\Omega)$ by a coboundary, then the trace $q_\Omega$ remains unchanged. In fact, the above map induces a group isomorphism \cite{ML52}
\begin{equation*}
H^3_\text{ab.}(D,\C^\times)\cong\{\text{quadratic forms }D\to \C^\times\}.
\end{equation*}

\minisec{\AIA{}s}

We can now define \aia{}s.

\begin{defi}[\AIA{}, \cite{DL93}, Chapter~12]\label{defi:aia}
Let $D$ be an abelian group equipped with a normalised abelian 3-cocycle $(F,\Omega)$. Let $q_\Omega$, $B_\Omega$ and $B$ associated with $(F,\Omega)$ as defined above. Let $N$ be some positive integer such that $B_\Omega$ is restricted to values in $((1/N)\Z)/\Z$.
An \emph{\aia{}} of level $N$ associated with $D$, $F$ and $\Omega$ of central charge $c$ is given by the following data:
\begin{itemize}
\item (\emph{space of states}) a $(1/N)\Z$- and $D$-graded vector space
\begin{equation*}
V=\bigoplus_{n\in\frac{1}{N}\Z}V_n=\bigoplus_{\alpha\in D}V^\alpha
\end{equation*}
with \emph{weight} $\wt(v)=n$ for $v\in V_n$ and the compatibility condition
\begin{equation*}
V^\alpha=\bigoplus_{n\in\frac{1}{N}\Z}V^\alpha_n
\end{equation*}
for $\alpha\in D$ where $V^\alpha_n:=V_n\cap V^\alpha$,
\item (\emph{vacuum vector}) a non-zero vector $\vac\in V_0^0$,
\item (\emph{conformal vector}) a non-zero vector $\omega\in V_2^0$,
\item (\emph{vertex operators}) a linear map
\begin{equation*}
Y(\cdot,x)\colon V\to\End_\C(V)[[x^{\pm 1/N}]]
\end{equation*}
taking each $v\in V$ to a field
\begin{equation*}
v\mapsto Y(v,x)=\sum_{n\in\frac{1}{N}\Z}v_nx^{n-1}
\end{equation*}
where for each $u\in V$, $v_n u=0$ for sufficiently large $n$ or equivalently a map $Y(\cdot,x)\colon V\otimes_\C V\to V((x^{1/N}))$. If $v\in V$ is homogeneous with respect to the weight $(1/N)\Z$-grading, then $\wt(v_n)=\wt(v)-n-1$ for all $n\in(1/N)\Z$.
\end{itemize}
These data are subject to the following axioms:
\begin{itemize}
\item (\emph{vacuum axioms}) $Y(\vac,x)=\id_V$ and $Y(v,z)\vac |_{z=0}=v$ for all $v\in V$.
\item (\emph{translation axiom}) $\partial_x Y(v,x)=Y(L_{-1}v,x)$ for any $v\in V$.
\item (\emph{grading compatibility}) For $v\in V^\alpha$ and $n\in(1/N)\Z$,
\begin{equation*}
v_n V^\beta\subseteq V^{\alpha+\beta}
\end{equation*}
and
\begin{equation*}
Y(v,x)|_{V^\beta}=\sum_{n=B_\Omega(\alpha,\beta)\pmod{1}}v_n x^{-n-1}.
\end{equation*}
\item (\emph{generalised Jacobi identity}) For $\alpha,\beta,\gamma\in D$ and $v_1\in V^\alpha$, $v_2\in V^\beta$, $v_3\in V^\gamma$,
\begin{align*}
&F(\alpha,\beta,\gamma)\iota_{x_1,x_0}x_2^{-1}\delta\left(\frac{x_1-x_0}{x_2}\right)Y(Y(v_1,x_0)v_2,x_2)\left(\frac{x_1-x_0}{x_2}\right)^{-B_\Omega(\alpha,\gamma)}v_3\\
&=\iota_{x_1,x_2}x_0^{-1}\left(\frac{x_1-x_2}{x_0}\right)^{B_\Omega(\alpha,\beta)}\delta\left(\frac{x_1-x_2}{x_0}\right)Y(v_1,x_1)Y(v_2,x_2)v_3\\
&\quad-B(\alpha,\beta,\gamma)\iota_{x_2,x_1}x_0^{-1}\left(\frac{x_2-x_1}{\e^{\pi\i}x_0}\right)^{B_\Omega(\alpha,\beta)}\delta\left(\frac{x_2-x_1}{-x_0}\right)Y(v_2,x_2)Y(v_1,x_1)v_3.
\end{align*}
\item (\emph{Virasoro relations}) The modes $L_n:=\omega_{n+1}$ of
\begin{equation*}
Y(\omega,x)=\sum_{n\in\Z}\omega_nx^{-n-1}=\sum_{n\in\Z}L_nx^{-n-2}
\end{equation*}
satisfy the \emph{Virasoro relations} at central charge $c$, i.e.\
\begin{equation*}
[L_m,L_n]=(m-n)L_{m+n}+\frac{m^3-m}{12}\delta_{m+n,0}\id_Vc
\end{equation*}
for $m,n\in\Z$. Moreover, $L_0v=nv=\wt(v)v$ for $v\in V_n$, $n\in(1/N)\Z$.
\end{itemize}
\end{defi}
When we speak of the level $N$ of an \aia{} we will generally assume that $N$ is \emph{minimal} with the property that $B_\Omega(\alpha,\beta)\in((1/N)\Z)/\Z$ for all $\alpha,\beta\in D$ and that the $L_0$-eigenvalues lie in $((1/N)\Z)/\Z$.

As the name suggests, \aia{}s are closely related to the concept of intertwining operators.
\begin{rem}\label{rem:12.30}
\item
\begin{enumerate}
\item\label{enum:remaia1} (\cite{DL93}, Remark~12.30) Let $V$ be an \aia{} such that $V^0$ is $\Z$-graded, i.e.\ $V^0=\sum_{n\in\Z}V_n^0$. Also assume that the graded components $V_n^\alpha$ are finite-dimensional and that on each $V^\alpha$, $\alpha\in D$, the weight grading is bounded from below. In this case $V^0$ is a \voa{}, the $V^\alpha$, $\alpha\in D$, are $V^0$-modules and $Y(\cdot,x)$ is composed of intertwining operators between these modules.

Indeed, let $\alpha=0$ in the generalised Jacobi identity. Then, since $(F,\Omega)$ is normalised, the $F$- and $B$-factors become 1 and we obtain the Jacobi identity for intertwining operators of type $\binom{V^{\beta+\gamma}}{V^{\beta}\,V^{\gamma}}$ (see Definition~\ref{defi:intops}).

\item\label{enum:remaia2}If we further assume that the $V^\alpha$, $\alpha\in D$, are irreducible as $V_0$-modules, then they have a conformal weight $\rho(V^\alpha)\in\C$. The \aia{} vertex operation $Y(v,x)|_{V^{\beta}}$ restricted to $v\in V^{\alpha}$ is an intertwining operator of type $\binom{V^{\alpha+\beta}}{V^{\alpha}\,V^{\beta}}$, which has exponents in the formal variable in $\rho(V^{\alpha+\beta})-\rho(V^{\alpha})-\rho(V^{\beta})+\Z$ (see Section~\ref{sec:intops}). By the definition of \aia{}s, on the other hand, these exponents lie in $-B_\Omega(\alpha,\beta)$ so that
\begin{equation*}
\rho(V^{\alpha+\beta})-\rho(V^{\alpha})-\rho(V^{\beta})+\Z=-B_\Omega(\alpha,\beta)\in(\frac{1}{N}\Z)/\Z.
\end{equation*}
Finally, assume that the $V^\alpha$, $\alpha\in D$, index all the irreducible $V_0$-modules up to isomorphism and that $V^0$ has group-like fusion, implying in particular that $D$ is finite. Then the left-hand side defines the bilinear form $B_\rho$ and we can conclude that $B_\rho=-B_\Omega$, i.e.\ the finite bilinear forms associated to the conformal weights and to the abelian 3-cocycle are the negatives of each other.

\item\label{enum:remaia3}We will show in Theorem~\ref{thm:2.7} that under certain assumptions even the corresponding quadratic forms fulfil $Q_\rho=-Q_\Omega$. Also, by construction, this is the case for the \aia{} naturally associated with a positive-definite, even lattice (see Theorem~\ref{thm:lataia}, \cite{DL93}, Remark~12.29).
\end{enumerate}
\end{rem}

\begin{rem}[\cite{DL93}, Remark~12.23]
Given an \aia{} with associated normalised abelian 3-cocycle $(F,\Omega)$, we can normalise each restriction of the vertex operator $Y(v,x)|_{V^{\beta}}$ for $v\in V^{\alpha}$ by multiplying it by $f(\alpha,\beta)$ for some function $f\colon D\times D\to\C^\times$. If we demand that $f(\alpha,0)=f(0,\alpha)=1$ for all $\alpha\in D$, then we again obtain an \aia{} associated with the normalised abelian 3-cocycle $(F',\Omega')=(F\cdot F_f,\Omega\cdot\Omega_f)$.\footnote{In principle, any function $f$ with $f(\alpha,0)=f(0,\alpha)=c$, $c\in\C^\times$, would again yield a normalised abelian 3-cocycle but in order to preserve the vacuum axioms of the \aia{} we have to demand that $c=1$.} In particular, the quadratic form $Q_\Omega$ does not change.
\end{rem}

The following proposition describes how \aia{}s generalise \voa{}s:
\begin{prop}\label{prop:aiavoa}
An \aia{} $A$ of level $N$ associated with some abelian 3-cocycle $(F,\Omega)$ carries the structure of a \voa{} upon rescaling of the vertex operators if and only if $Q_\Omega=0$, $N=1$, the weight grading is bounded from below and the graded components are finite-dimensional. 
\end{prop}
\begin{proof}
Assume that the quadratic form $Q_\Omega$ is trivial. Then we can rescale the vertex operators such that we obtain an \aia{} $A'$ with trivial abelian 3-cocycle. If $N=1$, then the weight grading is a $\Z$-grading. If we further assume that the weight grading is bounded from below and that the graded components are finite-dimensional, then the axioms reduce to those of a \voa{}, forgetting about the abelian group structure. The converse statement is clearly also true.
\end{proof}

\section{Simple-Current \AIA{}s I}\label{sec:aiascvoa1}

Let us return to the situation in Section~\ref{sec:scvoa} about simple-current \voa{}s, i.e.\ let $V$ be as in Assumption~\ref{ass:sn} a simple, rational, $C_2$-cofinite, self-contragredient \voa{} of CFT-type such that all irreducible modules are simple currents. Then we saw that there is the structure of a finite abelian group on $F_V$, the fusion group of $V$, and a finite quadratic form $Q_\rho$, given by the conformal weights modulo~1.\footnote
{
We had to additionally assume that $V$ satisfies Assumption~\ref{ass:p} but the result on the quadratic form does not depend on that (except for the non-degeneracy, cf.\ Theorem~\ref{thm:fafqs}). Indeed, that the conformal weights form a quadratic form is part of the statement of Theorem~\ref{thm:pre2.7} and will be shown in Lemma~\ref{lem:rationallocality}, which is needed for the proof of the theorem.
}

It follows from results by Huang:
\begin{thm}\label{thm:pre2.7}
Let $V$ be as in Assumption~\ref{ass:sn}. Then the direct sum of the irreducible $V$-modules
\begin{equation*}
A:=\bigoplus_{\gamma\in F_V}W^\gamma
\end{equation*}
carries the structure of an \aia{} associated with some normalised abelian 3-cocycle $(F,\Omega)$ on $F_V$.

The \aia{} structure on $A$ is the unique one up to a normalised abelian 3-coboundary extending the given \voa{} and module structures.
\end{thm}
\begin{rem}
The uniqueness statement in the above theorem is to be understood in the following sense: the vertex operation $Y(\cdot,x)$ on $A$ has to be composed of intertwining operators of type $\binom{W^{\alpha+\beta}}{W^\alpha\,W^\beta}$, $\alpha,\beta\in F_V$. These are unique up to a scalar. Hence, $Y(\cdot,x)$ may be multiplied by a function $f\colon F_V\times F_V\to\C^\times$ with $f(\alpha,0)=f(0,\alpha)=1$ for all $\alpha\in F_V$, resulting in a change of the abelian 3-cocycle $(F,\Omega)$ by the normalised abelian 3-coboundary $(F_f,\Omega_f)$. In particular the quadratic form $Q_\Omega$ is unique, as is $B_\Omega$.
\end{rem}

It is part of the statement of the theorem that the bilinear forms $B_\rho$ and $B_\Omega$ associated with the conformal weights and the abelian 3-cocycle $(F,\Omega)$, respectively, are the negatives of each other, i.e.\ $B_\rho=-B_\Omega$ (as explained in item~\ref{enum:remaia2} of Remark~\ref{rem:12.30}). From the theory of finite quadratic forms alone this does not imply that the quadratic forms $Q_\rho$ and $-Q_\Omega$ are identical because of the fact that to every bilinear form on $F_V$, there are $|F_V/2F_V|$ many possible quadratic forms with that associated bilinear form (see Remark~\ref{rem:bilquad}).

However, using the theory of \mtcs{} and in particular Huang's construction of \mtcs{} associated with certain \voa{}s, it is possible to show that in the situation of Theorem~\ref{thm:pre2.7} and under Assumption~\ref{ass:p} the quadratic forms $Q_\rho$ and $-Q_\Omega$ are indeed the same.\footnote{Also, since we use Assumption~\ref{ass:p} in the theorem, we know that the quadratic form $Q_\rho=-Q_\Omega$ is non-degenerate, i.e.\ $(F_V,Q_\rho)$ forms a \fqs{} (see Theorem~\ref{thm:fafqs}).} This is the statement of the following main theorem of this chapter, which is
Theorem~2.7 in \cite{HS14} but was stated there with an incomplete proof, as was pointed out by Scott Carnahan (see introduction of \cite{Car14}):
\begin{oframed}
\begin{thm}\label{thm:2.7}
Let $V$ be as in Assumptions~\ref{ass:sn}\ref{ass:p} with fusion group $F_V=(F_V,Q_\rho)$. Then the direct sum
\begin{equation*}
A:=\bigoplus_{\gamma\in F_V}W^\gamma
\end{equation*}
can be given the structure of an \aia{}, the unique one up to a normalised abelian 3-coboundary extending the given \voa{} and module structures, with associated normalised abelian 3-cocycle $(F,\Omega)$ such that
\begin{equation*}
Q_\Omega(\gamma)=-Q_\rho(\gamma)
\end{equation*}
for all $\gamma\in F_V$, i.e.\ the quadratic forms associated with the abelian 3-cocycle and the conformal weights are the negatives of each other. In other words: the \fqs{} $(F_V,Q_\Omega)$ associated with the \aia{} $A$ equals $\overline{F_V}=(F_V,-Q_\rho)$.

The level of the \aia{} $A$ is exactly the level of the \fqs{} $F_V$ (or $\overline{F_V}$).
\end{thm}
\end{oframed}
In the following we will present proofs of Theorems \ref{thm:pre2.7} and \ref{thm:2.7}. The first theorem is essentially well known and a special case of more general results by Huang. Indeed, Theorem~3.7 (and Remark~3.8) in \cite{Hua05} states that for a rational, $C_2$-cofinite \voa{} $V$ of CFT-type the direct sum of all irreducible $V$-modules up to isomorphism admits the structure of an \emph{intertwining operator algebra}. Intertwining operator algebras were first introduced in \cite{Hua97} but Definition~5.1 in \cite{Hua00} gives an equivalent characterisation of them as natural non-abelian generalisations of \aia{}s. From this definition on can read off that the intertwining operator algebra from Theorem~3.7 in \cite{Hua05} becomes an \aia{} if we additionally assume that $V$ is simple, self-contragredient
and all modules are simple currents, i.e.\ under Assumption~\ref{ass:sn}.

For reasons of comprehensibility it seems appropriate to include a direct proof of Theorem~\ref{thm:pre2.7}, which also heavily relies on results by Huang. This will be the rest of this section (Section~\ref{sec:aiascvoa1}).

For the proof of Theorem~\ref{thm:2.7} we need some knowledge of \mtcs{}, in particular those defined by Huang associated with certain \voa{}s. These concepts will be introduced in Section~\ref{sec:mtc}. Finally, Section~\ref{sec:aiascvoa2} gives two independent proofs of Theorem~\ref{thm:2.7}.

\minisec{Proof of Theorem~\ref{thm:pre2.7}}
The following proof is largely based on notes by van Ekeren \cite{Eke15} and uses results by Huang from \cite{Hua96,Hua00,Hua05,Hua08}.
Recall that we interpret fractional powers of complex variables as $z^n=\e^{n\log(z)}$ where $\log(z)=\log(|z|)+\i\arg(z)$ and $0\leq\arg(z)<2\pi$ (see Section~\ref{sec:formal}).

We start our considerations with the following result due to Huang:
\begin{prop}[Huang]\label{prop:restrictions}
Let $V$ be a rational, $C_2$-cofinite \voa{} of CFT-type. Let $A, B, C, D, P$ be $V$-modules and $\mathcal{Y}_1 \in \V^{D}_{A \, P}$ and $\mathcal{Y}_2 \in \V^{P}_{B \, C}$ intertwining operators. Fix vectors $a \in A$, $b \in B$, $c \in C$ and $d'\in D'$. Then there exists a $V$-module $Q$ and intertwining operators $\mathcal{Y}_3 \in \V^{D}_{B \, Q}$ and $\mathcal{Y}_4 \in \V^{Q}_{A \, C}$ such that the series
\begin{equation*}
\left\langle d',\mathcal{Y}_{1}(a, z) \mathcal{Y}_{2}(b, w) c \right\rangle
\quad \text{and} \quad
\left\langle d',\mathcal{Y}_{3}(b, w) \mathcal{Y}_{4}(a, z) c \right\rangle
\end{equation*}
converge in the domains $0<|w|<|z|$ and $0<|z|<|w|$, respectively, and are restrictions to these domains of a multi-valued function $F(z,w)$, analytic on the domain $\{(z, w) \in \C^2 \;|\; z, w, z-w \neq 0\}$. Moreover, there is a $V$-module $R$ and intertwining operators $\mathcal{Y}_5 \in \V^{D}_{R \, C}$ and $\mathcal{Y}_6 \in \V^{R}_{A \, B}$ such that
\begin{equation*}
\left\langle d',\mathcal{Y}_{5}(\mathcal{Y}_{6}(a, z-w)b, w) c \right\rangle
\end{equation*}
converges in the domain $0<|z-w|<|w|$ and is the restriction to that domain of $F(z,w)$.
\end{prop}
\begin{proof}
The proposition is a special case of \cite{Hua00}, Lemma~4.1. The result is also given without the precise statement on the domain of $F$ in \cite{Hua96}, Theorems 1.8 and 3.1. In both cases the result is proved under additional hypotheses on $V$ called ``convergence and extension properties'' (see \cite{Hua96}, p.\ 210). In \cite{Hua05}, Remark~3.8, these hypotheses are shown to hold if $V$ is rational, $C_2$-cofinite and of CFT-type.
\end{proof}

The proof of the following lemma is similar to that of \cite{Yam04}, Lemma~3.1, where a special case of item~\ref{enum:eke1} is proved.
\begin{lem}[\cite{Eke15}, Lemma~4.11]\label{lem:rationallocality}
Let $V$ be as in Assumption~\ref{ass:sn}. In the situation of Proposition~\ref{prop:restrictions} let $A = W^\alpha$, $B = W^\beta$, $C = W^\gamma$ and $D = W^{\alpha+\beta+\gamma}$ for $\alpha,\beta,\gamma\in F_V$. Then:
\begin{enumerate}
\item \label{enum:eke1} $Q_\rho(\alpha)=\rho(W^\alpha)+\Z$, $\alpha\in F_V$, defines a quadratic form on the abelian group $F_V$ with associated bilinear form $B_\rho(\alpha, \beta) = Q_\rho(\alpha+\beta)-Q_\rho(\alpha)-Q_\rho(\beta)$, $\alpha,\beta\in F_V$.
\item \label{enum:eke2} There is an $N \in \Q$, depending only on $a \in A$ and $b \in B$, such that
\begin{equation*}
(x-y)^N \left[ \mathcal{Y}_{1}(a, x) \mathcal{Y}_{2}(b, y) - \mathcal{Y}_{3}(b, y) \mathcal{Y}_{4}(a, x)\right] = 0.
\end{equation*}
\item \label{enum:eke3} For $a \in A$, $b \in B$, $c \in C, d'\in D'$ the series
\begin{align*}
\left\langle d', \mathcal{Y}_{1}(a, x) \mathcal{Y}_{2}(b, y) c \right\rangle \iota_{x, y} (x-y)^{-B_\rho(\alpha, \beta)} x^{-B_\rho(\alpha, \gamma)} y^{-B_\rho(\beta, \gamma)}, \\
\left\langle d', \mathcal{Y}_{3}(b, y) \mathcal{Y}_{4}(a, x) c \right\rangle \iota_{y, x} (x-y)^{-B_\rho(\alpha, \beta)} x^{-B_\rho(\alpha, \gamma)} y^{-B_\rho(\beta, \gamma)}, \\
\left\langle d', \mathcal{Y}_{5}(\mathcal{Y}_{6}(a, x-y)b, y) c \right\rangle \iota_{y, x-y} (x-y)^{-B_\rho(\alpha, \beta)} x^{-B_\rho(\alpha, \gamma)} y^{-B_\rho(\beta, \gamma)}
\end{align*}
are the images of a common element of $\C[x^{\pm 1}, y^{\pm 1}, (x-y)^{-1}]$ under $\iota_{x, y}$, $\iota_{y, x}$ and $\iota_{y, x-y}$, respectively.
\end{enumerate}
\end{lem}
By a slight abuse of notation we interpret $B_\rho(\alpha,\beta)$ in a term like $x^{B_\rho(\alpha,\beta)}$ as an arbitrary representative in $\Q$ of $B_\rho(\alpha,\beta)\in\Q/\Z$.
\begin{proof}
For the proof we consider some expressions in terms of the complex variables $z,w$ rather than the formal variables $x,y$. Fix $a,b,c,d'$ as in Proposition~\ref{prop:restrictions}. By definition, $\mathcal{Y}(W^\alpha, x)W^\beta \subseteq x^{B_\rho(\alpha, \beta)} W^{\alpha+\beta}((x))$ for an intertwining operator $\mathcal{Y}(\cdot,x)$ of type $\binom{W^{\alpha+\beta}}{W^\alpha\,W^\beta}$. The expression $\left\langle d',\mathcal{Y}_{5}(\mathcal{Y}_{6}(a, z-w)b,w)c\right\rangle$ defines the multi-valued function $F(z, w)$ as a series in $z-w$, $w$. By the boundedness-from-below property for $\mathcal{Y}_{6}$ (see Definition~\ref{defi:intops}) there exists an $N \in \Q$, depending only on $a$ and $b$, such that $(z-w)^N F(z, w)$ is regular at $z-w=0$. More precisely, $N \in -B_\rho(\alpha,\beta)$.

We consider the series expansion
\begin{equation*}
(z-w)^N F(z, w) = \iota_{z, w} (z-w)^N \left\langle d', \mathcal{Y}_{1}(a, z) \mathcal{Y}_{2}(b, w) c \right\rangle
\end{equation*}
in the domain $0 < |w| < |z|$. Since $(z-w)^N F(z, w)$ is regular at $z-w=0$, the series on the right-hand side actually converges for all $|z|, |w| > 0$. It is a series in fractional powers of $z$ and $w$ but
\begin{equation*}
z^{B_\rho(\alpha, \beta)-B_\rho(\alpha, \beta+\gamma)} w^{-B_\rho(\beta, \gamma)} \iota_{z, w} (z-w)^N \left\langle d', \mathcal{Y}_{1}(a, z) \mathcal{Y}_{2}(b, w) c \right\rangle
\end{equation*}
contains clearly only integral powers of $z$ and $w$. Similarly,
\begin{equation*}
z^{-B_\rho(\alpha, \gamma)} w^{B_\rho(\alpha, \beta)-B_\rho(\beta, \alpha+\gamma)} \iota_{w, z} (z-w)^N \left\langle d', \mathcal{Y}_{3}(b, w) \mathcal{Y}_{4}(a, z) c \right\rangle
\end{equation*}
is a convergent series in integral powers of $z$ and $w$. Since the functions defined by these two series are single-valued, so is their ratio
\begin{equation*}
z^{B_\rho(\alpha, \beta)-B_\rho(\alpha, \beta+\gamma)+B_\rho(\alpha, \gamma)} w^{B_\rho(\beta, \alpha+\gamma)-B_\rho(\alpha, \beta)-B_\rho(\beta, \gamma)}.
\end{equation*}
This implies
\begin{align*}
B_\rho(\alpha,\beta)-B_\rho(\alpha,\beta+\gamma)+B_\rho(\alpha,\gamma)&=0+\Z,\\
B_\rho(\beta,\alpha+\gamma)-B_\rho(\alpha,\beta)-B_\rho(\beta,\gamma)&=0+\Z,
\end{align*}
which means that $B_\rho$ is bilinear. Since a module and its contragredient have the same weight grading, $Q_\rho(-\alpha)=Q_\rho(\alpha)$ and $Q_\rho(0)=\rho(V)+\Z=0+\Z$ so that by Proposition~\ref{prop:bilquad} $Q_\rho$ is a quadratic form. Above we obtained two convergent series in integral powers of $z$ and $w$ whose ratio is of the form $(z/w)^k$ for some $k \in \Z$. One series has finitely many negative powers of $w$ and finitely many positive powers of $z$, the other vice versa. Hence both lie in $\C[z^{\pm 1}, w^{\pm 1}]$, which proves items \ref{enum:eke2} and \ref{enum:eke3}.
\end{proof}

\begin{rem}\label{rem:zwswaprem}
Recall from \eqref{eq:dl93} that for consistency with the definitions of \cite{DL93} we chose to interpret
\begin{equation*}
\iota_{y,x}(x-y)^n=(\e^{\pi\i})^{-n}\iota_{y,x}(y-x)^n
\end{equation*}
for $n\in\C$. This expression appears in Lemma~\ref{lem:rationallocality} but the convention does not matter for the statement of the lemma because any ambiguity can be absorbed into the choice of $\mathcal{Y}_3$ and $\mathcal{Y}_4$.
\end{rem}
In order to pass to a generalised Jacobi identity for the intertwining operators under Assumption~\ref{ass:sn} we first recall some standard formulæ for Laurent polynomials in several formal variables:
\begin{lem}\label{lem:formaldistributions}
Let $f(x_0,x_1,x_2) \in \C[x_0^{\pm 1}, x_1^{\pm 1}, x_2^{\pm 1}]$. Then:
\begin{enumerate}
\item \label{enum:eker1} \cite{FHL93}, Proposition~3.1.1:
\begin{align*}
&\iota_{x_1,x_0}x_2^{-1}\delta\left(\frac{x_1-x_0}{x_2}\right)f(x_0,x_1,x_1-x_0)\\
&=\iota_{x_1,x_2}x_0^{-1}\delta\left(\frac{x_1-x_2}{x_0}\right)f(x_1-x_2,x_1,x_2)\\
&\quad-\iota_{x_2,x_1}x_0^{-1}\delta\left(\frac{x_2-x_1}{-x_0}\right)f(x_1-x_2,x_1,x_2).
\end{align*}
\item \label{enum:eker2} \cite{FHL93}, Proposition~3.1.1:
\begin{equation*}
\iota_{x_2,x_0} x_1^{-1}\delta\left(\frac{x_2+x_0}{x_1}\right) f(x_0,x_0+x_2,x_2) = \iota_{x_1,x_0} x_2^{-1} \delta\left(\frac{x_1-x_0}{x_2}\right)f(x_0,x_1,x_1-x_0).
\end{equation*}
\item \label{enum:eker3}\cite{FLM88}, Proposition~8.8.22: 
\begin{equation*}
\iota_{x_1, x_0} x_2^{-1} \left( \frac{x_1-x_0}{x_2} \right)^{m} \delta\left( \frac{x_1-x_0}{x_2} \right) = \iota_{x_2, x_0} x_1^{-1} \left( \frac{x_2+x_0}{x_1} \right)^{-m} \delta\left( \frac{x_2+x_0}{x_1} \right)
\end{equation*}
for every $m\in\C$.
\end{enumerate}
\end{lem}
We continue in the situation of Lemma~\ref{lem:rationallocality}. Let $g(x,y)\in\C[x^{\pm 1},y^{\pm 1},(x-y)^{-1}]$ be the common element in item~\ref{enum:eke3} there and let the rational function $f(x_0,x_1,x_2)\in\C[x_0^{\pm 1},x_1^{\pm 1},x_2^{\pm 1}]$ be such that $g(x,y)=f(x-y,x,y)$. Applying items \ref{enum:eker1} and \ref{enum:eker2} of Lemma~\ref{lem:formaldistributions} to $f$ yields
\begin{align}\begin{split}\label{eq:lemmatof}
&\iota_{x_1,x_2} (x_1-x_2)^{-B_\rho(\alpha, \beta)} x_1^{-B_\rho(\alpha, \gamma)} x_2^{-B_\rho(\beta, \gamma)} x_0^{-1} \delta\left( \frac{x_1-x_2}{x_0} \right) \mathcal{Y}_1(a, x_1) \mathcal{Y}_2(b, x_2) c \\
&\quad- \iota_{x_2,x_1}(x_1-x_2)^{-B_\rho(\alpha, \beta)} x_1^{-B_\rho(\alpha, \gamma)} x_2^{-B_\rho(\beta, \gamma)} x_0^{-1} \delta\left( \frac{x_2-x_1}{-x_0} \right) \mathcal{Y}_3(b, x_2) \mathcal{Y}_4(a, x_1) c \\
&= \iota_{x_2,x_0} x_0^{-B_\rho(\alpha, \beta)} (x_0+x_2)^{-B_\rho(\alpha, \gamma)} x_2^{-B_\rho(\beta, \gamma)} x_1^{-1} \delta\left( \frac{x_2+x_0}{x_1} \right) \mathcal{Y}_5(\mathcal{Y}_6(a, x_0)b, x_2) c.
\end{split}\end{align}
Indeed, the identity holds when paired with $d'$ for all $d'\in(W^{\alpha+\beta+\gamma})'$ and hence it holds with $d'$ omitted.
The first term of \eqref{eq:lemmatof} equals
\begin{equation*}
\iota_{x_1,x_2}x_0^{-B_\rho(\alpha,\beta)}x_1^{-B_\rho(\alpha,\gamma)}x_2^{-B_\rho(\beta,\gamma)}x_0^{-1}\left(\frac{x_1-x_2}{x_0}\right)^{-B_\rho(\alpha,\beta)}\!\delta\left(\frac{x_1-x_2}{x_0}\right)\mathcal{Y}_1(a,x_1)\mathcal{Y}_2(b, x_2)c.
\end{equation*}
Using item~\ref{enum:eker3} of Lemma~\ref{lem:formaldistributions}, the third term of \eqref{eq:lemmatof} equates to
\begin{align*}
&\iota_{x_2,x_0}x_0^{-B_\rho(\alpha,\beta)}x_1^{-B_\rho(\alpha,\gamma)}x_2^{-B_\rho(\beta,\gamma)}x_1^{-1}\!\left(\frac{x_2+x_0}{x_1}\right)^{-B_\rho(\alpha,\gamma)}\!\delta\!\left(\frac{x_2+x_0}{x_1}\right)\mathcal{Y}_5(\mathcal{Y}_6(a,x_0)b,x_2)c\\
&=\iota_{x_1,x_0}x_0^{-B_\rho(\alpha,\beta)}x_1^{-B_\rho(\alpha,\gamma)}x_2^{-B_\rho(\beta,\gamma)}x_2^{-1}\!\left(\frac{x_1-x_0}{x_2}\right)^{B_\rho(\alpha,\gamma)}\!\delta\!\left(\frac{x_1-x_0}{x_2}\right)\mathcal{Y}_5(\mathcal{Y}_6(a,x_0)b,x_2)c
\end{align*}
and the second term becomes
\begin{align*}
& -\iota_{x_2,x_1} x_0^{-B_\rho(\alpha, \beta)} \left(\frac{x_1-x_2}{x_0}\right)^{-B_\rho(\alpha, \beta)} x_1^{-B_\rho(\alpha, \gamma)} x_2^{-B_\rho(\beta, \gamma)} x_0^{-1} \delta\left( \frac{x_2-x_1}{-x_0} \right)\\
&\quad\mathcal{Y}_3(b, x_2)\mathcal{Y}_4(a, x_1) c \\
&= -\iota_{x_2,x_1} x_0^{-B_\rho(\alpha, \beta)} \left(\frac{\e^{-\pi \i}(x_2-x_1)}{x_0}\right)^{-B_\rho(\alpha, \beta)} x_1^{-B_\rho(\alpha, \gamma)} x_2^{-B_\rho(\beta, \gamma)} x_0^{-1} \delta\left( \frac{x_2-x_1}{-x_0} \right)\\
&\quad\mathcal{Y}_3(b, x_2)\mathcal{Y}_4(a, x_1) c \\
&= -\iota_{x_2,x_1} x_0^{-B_\rho(\alpha, \beta)} \left(\frac{x_2-x_1}{\e^{\pi \i}x_0}\right)^{-B_\rho(\alpha, \beta)} x_1^{-B_\rho(\alpha, \gamma)} x_2^{-B_\rho(\beta, \gamma)} x_0^{-1} \delta\left( \frac{x_2-x_1}{-x_0} \right)\\
&\quad\mathcal{Y}_3(b, x_2)\mathcal{Y}_4(a, x_1) c,
\end{align*}
where, passing from the first to the second line, we have used the convention in Remark~\ref{rem:zwswaprem}. Cancelling $x_0^{-B_\rho(\alpha, \beta)} x_1^{-B_\rho(\alpha, \gamma)} x_2^{-B_\rho(\beta, \gamma)}$ from all three terms we finally obtain
\begin{align*}
&\iota_{x_1,x_2} x_0^{-1} \left( \frac{x_1-x_2}{x_0} \right)^{-B_\rho(\alpha, \beta)} \delta\left( \frac{x_1-x_2}{x_0} \right) \mathcal{Y}_1(a, x_1) \mathcal{Y}_2(b, x_2) c \\
&\quad- \iota_{x_2,x_1} x_0^{-1} \left( \frac{x_2-x_1}{\e^{\pi \i}x_0} \right)^{-B_\rho(\alpha, \beta)} \delta\left( \frac{x_2-x_1}{-x_0} \right) \mathcal{Y}_3(b, x_2) \mathcal{Y}_4(a, x_1) c \\
&= \iota_{x_1,x_0} x_2^{-1} \left( \frac{x_1-x_0}{x_2} \right)^{+B_\rho(\alpha, \gamma)} \delta\left( \frac{x_1-x_0}{x_2} \right) \mathcal{Y}_5(\mathcal{Y}_6(a, x_0)b, x_2) c.
\end{align*}
This already resembles the generalised Jacobi identity for \aia{}s (see Definition~\ref{defi:aia}).

Recall that in the situation of group-like fusion (Assumption~\ref{ass:sn}) we chose $V$-module representatives $W^\alpha$ for $\alpha\in F_V$ with $W^0=V$. Now we also choose representatives of the one-dimensional spaces $\V^{W^{\alpha+\beta}}_{W^\alpha\,W^\beta}$ of intertwining operators of type $\binom{W^{\alpha+\beta}}{W^\alpha\,W^\beta}$, $\alpha,\beta\in F_V$.
\begin{defi}[System of Scalars]\label{defi:FandB}
Assume that $V$ satisfies Assumption~\ref{ass:sn}. We fix a choice of intertwining operator $\mathcal{Y}_{W^\alpha,W^\beta}^+\in\V^{W^{\alpha+\beta}}_{W^\alpha\,W^\beta}$ for each $\alpha,\beta\in F_V$. Each intertwining operator $\mathcal{Y}_k$, $k = 1, \ldots, 6$ in the discussion above is a scalar multiple of one of the fixed ones and we immediately obtain the following derived \emph{generalised Jacobi identity} (cf.\ Definition~\ref{defi:aia}):
\begin{align*}
&\iota_{x_1,x_2}x_0^{-1}\left(\frac{x_1-x_2}{x_0}\right)^{-B_\rho(\alpha,\beta)}\!\delta\!\left(\frac{x_1-x_2}{x_0}\right)\mathcal{Y}_{W^\alpha,W^{\beta+\gamma}}^+(a,x_1)\mathcal{Y}_{W^\beta,W^\gamma}^+(b,x_2)c\\
&\quad-B(\alpha,\beta,\gamma)\iota_{x_2,x_1}x_0^{-1}\left(\frac{x_2-x_1}{\e^{\pi\i}x_0}\right)^{-B_\rho(\alpha,\beta)}\!\delta\!\left(\frac{x_2-x_1}{-x_0}\right)\mathcal{Y}_{W^\beta,W^{\alpha+\gamma}}^+(b,x_2)\mathcal{Y}_{W^\alpha,W^\gamma}^+(a,x_1)c\\
&=F(\alpha,\beta,\gamma)\iota_{x_1,x_0}x_2^{-1}\left(\frac{x_1-x_0}{x_2}\right)^{B_\rho(\alpha,\gamma)}\!\delta\!\left(\frac{x_1-x_0}{x_2}\right)\mathcal{Y}_{W^{\alpha+\beta},W^\gamma}^+(\mathcal{Y}_{W^\alpha,W^\beta}^+(a,x_0)b,x_2)c
\end{align*}
for some system of non-zero scalar factors $F,B\colon F_V\times F_V\times F_V\to\C^\times$.
\end{defi}
We can immediately also define the function $\Omega\colon F_V\times F_V\to\C^\times$:
\begin{defi}[Braiding Convention]\label{defi:OmegaDef}
For all $\alpha, \beta \in F_V$ we define $\widetilde{\mathcal{Y}}_{W^\alpha,W^\beta}^+$ by
\begin{equation*}
\widetilde{\mathcal{Y}}_{W^\alpha,W^\beta}^+(a,x)b=\e^{xL_{-1}}\mathcal{Y}_{W^\beta,W^\alpha}^+(b,\e^{-\pi \i}x)a.
\end{equation*}
Then by \cite{HL95}, Section~7, $\widetilde{\mathcal{Y}}_{W^\alpha,W^\beta}^+\in\V^{W^{\alpha+\beta}}_{W^\alpha\,W^\beta}$ and we define $\Omega\colon F_V\times F_V\to\C^\times$ by
\begin{equation*}
\mathcal{Y}_{W^\alpha,W^\beta}^+(a,x)b =\widetilde\Omega(\alpha,\beta)\widetilde{\mathcal{Y}}_{W^\alpha,W^\beta}^+(a,x)b
\end{equation*}
where $\widetilde\Omega(\alpha,\beta)=\Omega(\beta,\alpha)^{-1}$ as in Remark~\ref{rem:omegatilde}.
\end{defi}
In the following we show that $F$, $B$ and $\Omega$ satisfy the properties required to endow $A=\bigoplus_{\gamma\in F_V}W^\gamma$ with the structure of an \aia{} whose vertex operation consists of the intertwining operators chosen in Definition~\ref{defi:FandB}. We begin by showing that $F$ and $\Omega$ form an abelian 3-cocycle:
\begin{prop}[Huang, \cite{Eke15}, Proposition~4.20]\label{prop:Huang.3cocycle}
Let $V$ be as in Assumption~\ref{ass:sn}. Let $F$ be as in Definition~\ref{defi:FandB} and $\Omega$ as in Definition~\ref{defi:OmegaDef}. Then $(F,\Omega)$ is an abelian $3$-cocycle. If we choose $\mathcal{Y}_{V,W^\alpha}^+$ to be the $V$-module action of $V$ on $W^\alpha$ and $\mathcal{Y}_{W^\alpha,V}^+:=\widetilde{\mathcal{Y}}_{V,W^\alpha}^+$, then $(F,\Omega)$ is normalised.
\end{prop}
\begin{proof}
The relation $F(0,\beta,\gamma)=1$ follows from the definition of the Jacobi identity for intertwining operators (see Definition~\ref{defi:intops}), the relation $\Omega(0,\beta)=1$ from our choice of $\mathcal{Y}_{W^\alpha,V}^+$. These two relations imply that $(F,\Omega)$ is normalised once we establish that it is a cocycle.

A proof of the cocycle condition is given by Huang in \cite{Hua08}, Section~1, using the Huang-Lepowsky tensor-product theory.
\end{proof}
The above proposition also implies that $Q_\Omega$ associated with $\Omega$ as in equation \eqref{eq:QOQr} is a quadratic form.
\begin{lem}[\cite{Eke15}, Lemma~4.23]\label{lem:OmegaEqualsB}
Let $V$ be as in Assumption~\ref{ass:sn}, $\Omega$ as in Definition~\ref{defi:OmegaDef} and let $B_\Omega$ and $B_\rho$ be the bilinear forms associated with the quadratic forms $Q_\Omega$ and $Q_\rho$, respectively. Then
\begin{equation*}
B_\Omega=-B_\rho.
\end{equation*}
\end{lem}

\begin{proof}
Proposition~\ref{prop:Huang.3cocycle} states that $(F, \Omega)$ is a normalised abelian $3$-cocycle. It is straightforward, though tedious, to derive
\begin{equation*}
\Omega(\alpha+\beta, \alpha+\beta) = \Omega(\alpha, \alpha) \Omega(\alpha, \beta) \Omega(\beta, \alpha) \Omega(\beta, \beta)
\end{equation*}
from Definition~\ref{defi:abeliancocycle}. Hence
\begin{equation*}
b_\Omega(\alpha,\beta)=\e^{(2\pi\i)(Q_\Omega(\alpha+\beta) - Q_\Omega(\alpha)-Q_\Omega(\beta))} = \Omega(\alpha, \beta) \Omega(\beta, \alpha).
\end{equation*}
On the other hand, applying Definition~\ref{defi:OmegaDef} twice gives
\begin{align*}
\mathcal{Y}_{W^\alpha,W^\beta}^+(a, x)b
&= \Omega(\beta, \alpha)^{-1} \e^{xL_{-1}} \mathcal{Y}_{W^\beta,W^\alpha}^+(b, \e^{-\pi\i}x)a \\
&= \Omega(\beta, \alpha)^{-1}\Omega(\alpha, \beta)^{-1} \mathcal{Y}_{W^\alpha,W^\beta}^+(a, \e^{-2\pi \i}x)b \\
&= \Omega(\beta, \alpha)^{-1}\Omega(\alpha, \beta)^{-1} \e^{-(2\pi\i) B_\rho(\alpha, \beta)} \mathcal{Y}_{W^\alpha,W^\beta}^+(a, x)b,
\end{align*}
so that $B_\Omega(\alpha,\beta)=-B_\rho(\alpha,\beta)$ for all $\alpha,\beta\in F_V$.
\end{proof}

As a last step we show that $F$ and $B$ are related via $\Omega$ in the desired way.
\begin{lem}[\cite{Eke15}, Lemma~4.26]\label{lem:FBOM}
Let $V$ be as in Assumption~\ref{ass:sn} and let $F$ and $B$ be as in Definition~\ref{defi:FandB} and $\Omega$ as in Definition~\ref{defi:OmegaDef}. Then
\begin{equation*}
B(\alpha,\beta,\gamma)=F(\beta,\alpha,\gamma)^{-1}\Omega(\alpha,\beta)F(\alpha,\beta,\gamma)
\end{equation*}
for all $\alpha,\beta,\gamma\in F_V$
\end{lem}

\begin{proof}
In the following manipulations, for reasons of readability, we omit $F, B, \Omega$ and most of the subscripts $W^\alpha,W^\beta,W^\gamma$ for $\mathcal{Y}^+(\cdot,x)$ and only reinsert them at the very end. We begin with the generalised Jacobi identity in Definition~\ref{defi:FandB}:
\begin{align*}
&\iota_{x_1,x_2} x_0^{-1} \left( \frac{x_1-x_2}{x_0} \right)^{-B_\rho(\alpha, \beta)} \delta\left( \frac{x_1-x_2}{x_0} \right) \mathcal{Y}^+(a, x_1) \mathcal{Y}^+(b, x_2) c \\
&\quad -\iota_{x_2,x_1} x_0^{-1} \left( \frac{x_2-x_1}{\e^{\pi\i}x_0} \right)^{-B_\rho(\alpha, \beta)} \delta\left( \frac{x_2-x_1}{-x_0} \right) \mathcal{Y}^+(b, x_2) \mathcal{Y}^+(a, x_1) c \\
&=\iota_{x_1,x_0} x_2^{-1} \left( \frac{x_1-x_0}{x_2} \right)^{B_\rho(\alpha, \gamma)} \delta\left( \frac{x_1-x_0}{x_2} \right) \mathcal{Y}^+(\mathcal{Y}^+(a, x_0)b, x_2) c \\
&=\iota_{x_1,x_0} x_2^{-1} \left( \frac{x_1-x_0}{x_2} \right)^{B_\rho(\alpha, \gamma)} \delta\left( \frac{x_1-x_0}{x_2} \right) \mathcal{Y}^+(\e^{x_0L_{-1}}\mathcal{Y}_{W^\beta,W^\alpha}^+(b, \e^{-\pi\i}x_0)a, x_2) c \\
&=\iota_{x_1,x_0} x_2^{-1} \left( \frac{x_1-x_0}{x_2} \right)^{B_\rho(\alpha, \gamma)} \delta\left( \frac{x_1-x_0}{x_2} \right) \mathcal{Y}^+(\mathcal{Y}_{W^\beta,W^\alpha}^+(b, \e^{-\pi\i}x_0)a, x_2+x_0) c,
\end{align*}
using the translation axiom and Taylor expansion in the last step. Now
\begin{align*}
&\iota_{x_1,x_0} x_2^{-1} \left( \frac{x_1-x_0}{x_2} \right)^{B_\rho(\alpha, \gamma)} \delta\left( \frac{x_1-x_0}{x_2} \right) (x_2+x_0)^n\\
&= \iota_{x_2,x_0} x_1^{-1} \left( \frac{x_2+x_0}{x_1} \right)^{-B_\rho(\alpha, \gamma)} \delta\left( \frac{x_2+x_0}{x_1} \right) (x_2+x_0)^n \\
&= \iota_{x_2,x_0} x_1^{-1} \left( \frac{x_2+x_0}{x_1} \right)^{B_\rho(\beta, \gamma)} \delta\left( \frac{x_2+x_0}{x_1} \right) x_1^n
\end{align*}
for any $n \in B_\rho(\alpha+\beta, \gamma)$. Here we used item~\ref{enum:eker3} of Lemma~\ref{lem:formaldistributions}, the bilinearity of $B_\rho$ and the following property of the delta function: $\delta(x/y) p(x) = \delta(x/y) p(y)$ for a polynomial $p$ with integral exponents. Using this we rewrite the expression above as
\begin{equation*}
\iota_{x_2,x_0} x_1^{-1} \left( \frac{x_2+x_0}{x_1} \right)^{B_\rho(\beta, \gamma)} \delta\left( \frac{x_2+x_0}{x_1} \right) \mathcal{Y}^+(\mathcal{Y}_{W^\beta,W^\alpha}^+(b, \e^{-\pi\i}x_0)a, x_1) c.
\end{equation*}
Then we apply the derived generalised Jacobi identity from Definition~\ref{defi:FandB} again:
\begin{align*}
&\left. \iota_{x_2,t} x_1^{-1} \left( \frac{x_2-t}{x_1} \right)^{B_\rho(\beta, \gamma)} \delta\left( \frac{x_2-t}{x_1} \right) \mathcal{Y}^+(\mathcal{Y}_{W^\beta,W^\alpha}^+(b, t)a, x_1) c \right|_{t = \e^{-\pi\i}x_0} \\
&=\left. \iota_{x_2,x_1} t^{-1} \left( \frac{x_2-x_1}{t} \right)^{-B_\rho(\alpha, \beta)} \delta\left( \frac{x_2-x_1}{t} \right) \mathcal{Y}^+(b, x_2) \mathcal{Y}^+(a, x_1) c \right|_{t=\e^{-\pi\i}x_0} \\
&\quad- \left. \iota_{x_1,x_2} t^{-1} \left( \frac{x_1-x_2}{\e^{\pi\i}t} \right)^{-B_\rho(\alpha, \beta)} \delta\left( \frac{x_1-x_2}{-t} \right) \mathcal{Y}^+(a, x_1) \mathcal{Y}^+(b, x_2) c \right|_{t=\e^{-\pi\i}x_0} \\
&=-\iota_{x_2,x_1} x_0^{-1} \e^{-(2\pi\i)B_\rho(\alpha,\beta)}\left( \frac{x_2-x_1}{\e^{\pi\i}x_0} \right)^{-B_\rho(\alpha, \beta)} \delta\left( \frac{x_2-x_1}{-x_0} \right) \mathcal{Y}^+(b, x_2) \mathcal{Y}^+(a, x_1) c \\
&\quad+ \iota_{x_1,x_2} x_0^{-1} \left( \frac{x_1-x_2}{x_0} \right)^{-B_\rho(\alpha, \beta)} \delta\left( \frac{x_1-x_2}{x_0} \right) \mathcal{Y}^+(a, x_1) \mathcal{Y}^+(b, x_2) c.
\end{align*}
Reinserting all the omitted scalar factors and comparing the coefficients we obtain
\begin{align*}
1 &= F(\alpha, \beta, \gamma) \widetilde\Omega(\alpha, \beta) F(\beta, \alpha, \gamma)^{-1} B(\beta, \alpha, \gamma), \\
B(\alpha, \beta, \gamma) &= F(\alpha, \beta, \gamma) \widetilde\Omega(\alpha, \beta) F(\beta, \alpha, \gamma)^{-1}b_\rho(\alpha,\beta)^{-1},
\end{align*}
which can be reorganised to read
\begin{align*}
B(\alpha, \beta, \gamma) &= F(\beta, \alpha, \gamma)^{-1} \Omega(\alpha, \beta) F(\alpha, \beta, \gamma), \\
B(\alpha, \beta, \gamma) &= F(\alpha, \beta, \gamma) \Omega(\beta, \alpha)^{-1} F(\beta, \alpha, \gamma)^{-1}b_\rho(\alpha,\beta)^{-1}.
\end{align*}
Hence, $B$ satisfies the relation we aimed to prove and we also obtain $\Omega(\alpha,\beta)\Omega(\beta,\alpha)=1/b_\rho(\alpha,\beta)$, which already follows from Lemma~\ref{lem:OmegaEqualsB}.
\end{proof}

Finally, collecting all results, we can prove that the direct sum $A$ of the irreducible $V$-modules up to isomorphism admits the structure of an \aia{}. We choose the vertex operation $Y(\cdot,x)$ on $A$ to be composed of the intertwining operators $\mathcal{Y}_{W^\alpha,W^\beta}^+(\cdot,x)$ for $\alpha,\beta\in F_V$ with the normalisation condition described in Proposition~\ref{prop:Huang.3cocycle} fulfilled.

\begin{proof}[Proof of Theorem~\ref{thm:pre2.7}]
With our choice of intertwining operators, they satisfy the generalised Jacobi identity for an \aia{} (see Definition~\ref{defi:aia}) for some system of scalars $F$ and $B$ (see Definition~\ref{defi:FandB}). We saw in Proposition~\ref{prop:Huang.3cocycle} that $(F, \Omega)$ defines a normalised abelian $3$-cocycle. It is necessary that $B$ be compatible with $F$ and $\Omega$, which is Lemma~\ref{lem:FBOM}. For the grading condition on $Y(\cdot,x)$ to hold the bilinear forms associated with $Q_\Omega$ and $Q_\rho$ have to be negatives of each other, which is shown in Lemma~\ref{lem:OmegaEqualsB}. The remaining axioms are easy to verify.
\end{proof}

\section{\MTCs{}}\label{sec:mtc}

For a \voa{} $V$ fulfilling suitable regularity assumptions Huang has shown that the category of $V$-modules carries the structure of a \mtc{}. Translating results of categorical nature into the language of \voa{}s offers additional insight and, in particular, allows to prove Theorem~\ref{thm:2.7}. For an introduction into \mtcs{} the reader is referred to \cite{BK01,Tur10}, and \cite{JS93,EGNO15} for related concepts. We present a short overview.

\minisec{Ribbon Categories}
The most basic concept is that of a monoidal category. A \emph{monoidal category} (or \emph{tensor category}) is a category $\CC$ equipped with a bifunctor
\begin{equation*}
\otimes\colon\CC\times\CC\to\CC,
\end{equation*}
called \emph{tensor product}, a \emph{unit object} $\unit\in\CC$ and three natural isomorphisms: the \emph{associativity isomorphism}
\begin{equation*}
\alpha_{A,B,C}\colon (A\otimes B)\otimes C\to A\otimes (B\otimes C)
\end{equation*}
for $A,B,C\in\CC$ and the \emph{left} and \emph{right unit isomorphisms}
\begin{equation*}
l_A\colon\unit\otimes A\to A\quad\text{and}\quad r_A\colon A\otimes\unit\to A
\end{equation*}
for $A\in\CC$ subject to certain coherence conditions encoded in the pentagon and the triangle diagrams.

A \emph{braiding} on a monoidal category $\CC$ is given by a natural isomorphism
\begin{equation*}
c_{A,B}\colon A\otimes B\to B\otimes A
\end{equation*}
for $A,B\in\CC$ such that the braiding is compatible with the associativity isomorphism as expressed in the two hexagon diagrams. A monoidal category with a braiding is called \emph{braided monoidal category}.

A \emph{twist} on a braided tensor category $\CC$ is a natural isomorphism
\begin{equation*}
\theta_A\colon A\to A
\end{equation*}
for $A\in\CC$ with $\theta_\unit=\id_\unit$ and compatible in a certain way with the braiding. A monoidal category with a braiding and a twist is sometimes called \emph{balanced monoidal category}.

A \emph{right duality} on a monoidal category $\CC$ is an assignment of a \emph{right dual} object $A^*$ to every object $A\in\CC$ together with \emph{evaluation} and \emph{coevaluation} morphisms
\begin{equation*}
d_A\colon A^*\otimes A\to\unit\quad\text{and}\quad b_A\colon\unit\to A\otimes A^*
\end{equation*}
satisfying certain conditions. A monoidal category with a right duality is called \emph{right rigid} (or \emph{right autonomous}). A left duality with dual objects ${}^*\!A$ and evaluation and coevaluation morphisms
\begin{equation*}
\tilde{d}_A\colon A\otimes {}^*\!A\to\unit\quad\text{and}\quad \tilde{b}_A\colon\unit\to {}^*\!A\otimes A
\end{equation*}
is defined analogously. A monoidal category is called \emph{rigid} (or \emph{autonomous}) if it has a right and a left duality.

A \emph{ribbon category} (or \emph{tortile category}) is a right rigid, balanced monoidal category, i.e.\ a monoidal category with a braiding, a twist and a right duality such that the twist and the right duality are compatible in a certain way. Ribbon categories have a number of well-known properties. For instance they come automatically equipped with a left duality by means of the definitions ${}^*\!A:=A^*$ and
\begin{equation*}
\tilde{b}_A:=(\theta_{A^*}\otimes\id_A)\circ c_{A,A^*} \circ b_A\quad\text{and}\quad\tilde{d}_A:=d_A\circ c_{A,A^*}(\theta_A\otimes\id_{A^*})
\end{equation*}
for $A\in\CC$. This means that ribbon categories are rigid. In fact, ribbon categories are examples of \emph{sovereign} categories, where there is a natural isomorphism ${}^*\!A\to A^*$ for $A\in\CC$.
Ribbon categories are also \emph{pivotal} categories, a concept essentially equivalent to the aforementioned sovereign categories. A pivotal category is a right rigid monoidal category equipped with a natural isomorphism
\begin{equation*}
\psi_A\colon A\to A^{**}
\end{equation*}
for $A\in\CC$ that is monoidal, i.e.\ compatible in a certain way with the monoidal structure.

Via the braiding it is possible to define another natural morphism $u_A\colon A\to A^{**}$ as
\begin{align*}
A&\xrightarrow{(\id_A\otimes b_{A^*})\circ r_A^{-1}}A\otimes(A^*\otimes A^{**})
\xrightarrow{\alpha^{-1}_{A,A^*,A^{**}}}(A\otimes A^*)\otimes A^{**}\\
&\xrightarrow{c_{A,A^*}\otimes\id_{A^{**}}}(A^*\otimes A)\otimes A^{**}
\xrightarrow{l_a\circ(d_A\otimes\id_{A^{**}})}A^{**},
\end{align*}
which is well known to be an isomorphism. In fact, $\psi_A$ and $u_A$ are precisely related via the twist isomorphism $\theta_A$, i.e.\
\begin{equation*}
\psi_A=u_A\theta_A.
\end{equation*}

In pivotal categories we define the \emph{left} and \emph{right traces} $\tr_\text{L}(f),\tr_\text{R}(f)\in\Mor_\CC(\unit,\unit)$ of morphisms $f\in\Mor_\CC(A,A)$ via
\begin{align*}
\tr_\text{L}(f)&:\unit\xrightarrow{b_{A^*}}A^*\otimes A^{**}\xrightarrow{\id_{A^*}\otimes\psi_A^{-1}}A^*\otimes A\xrightarrow{\id_{A^*}\otimes f}A^*\otimes A\xrightarrow{d_A}\unit,\\
\tr_\text{R}(f)&:\unit\xrightarrow{b_A}A\otimes A^*\xrightarrow{f\otimes\id_{A^*}}A\otimes A^*\xrightarrow{\psi_A\otimes\id_{A^*}}A^{**}\otimes A^*\xrightarrow{d_{A^*}}\unit.
\end{align*}
Alternatively, in a sovereign category we can express the traces as 
\begin{align*}
\tr_\text{L}(f)&:\unit\xrightarrow{\tilde{b}_A}A^*\otimes A\xrightarrow{\id_{A^*}\otimes f}A^*\otimes A\xrightarrow{d_A}\unit,\\
\tr_\text{R}(f)&:\unit\xrightarrow{b_A}A\otimes A^*\xrightarrow{f\otimes\id_{A^*}}A\otimes A^*\xrightarrow{\tilde{d}_A}\unit.
\end{align*}
Via the traces we define the (categorical) \emph{dimensions}
\begin{equation*}
d_\text{L}(A)=\tr_\text{L}(\id_A)\quad\text{and}\quad d_\text{R}(A)=\tr_\text{R}(\id_A).
\end{equation*}
We will later, in particular in the setting of \mtcs{}, call these \emph{quantum dimensions}. The dimensions only depend on the isomorphism class of the object.

Ribbon categories are also \emph{spherical} categories, pivotal categories in which left and right traces coincide, i.e.\
\begin{equation*}
\tr_\text{L}(f)=\tr_\text{R}(f)=:\tr(f).
\end{equation*}
Then also $d_\text{L}(A)=d_\text{R}(A)=:d(A)$. For reasons of clarity we will also write $\tr_A(f)=\tr(f)$ for a morphism $f\in\Mor_\CC(A,A)$. 

We will later need the following result from \cite{DGNO10} (Proposition~2.32), which even holds in any ribbon category (see \cite{CKL15}, Corollary~2.5):
\begin{prop}\label{prop:2.32}
Let $\CC$ be a ribbon category. Then
\begin{equation*}
\tr_{A\otimes A}(c_{A,A}^{-1})=\tr_A(\theta_A^{-1})
\end{equation*}
for any object $A\in\CC$.
\end{prop}

\minisec{\MTCs{}}
Given a monoidal category $\CC$, we can demand the following additional structure: let $\CC$ be an abelian, $\C$-linear, i.e.\ enriched over $\operatorname{Vect}_\C$, and semisimple monoidal category such that there are only finitely many isomorphism classes of simple objects and the tensor unit $\unit$ is simple.

An abelian category is called semisimple if every object is semisimple, i.e.\ a direct sum of simple objects.
The $\C$-linearity means that $\Hom_\CC(A,B):=\Mor_\CC(A,B)$ is a $\C$-vector space and the composition of morphisms is $\C$-bilinear.

Let us also assume that the spaces $\Hom_\CC(A,B)$ are all finite-dimensional.
Then Schur's lemma holds, i.e.\
\begin{equation*}
\Hom_\CC(A,B)\cong\begin{cases}\C&\text{if }A\cong B,\\0&\text{if }A\ncong B\end{cases}
\end{equation*}
as $\C$-vector spaces for simple objects $A,B\in\CC$.

\begin{defi}[\MTC{}]
A ribbon category $\CC$ with these additional properties and whose ribbon structure is compatible with the $\C$-linear structure satisfying a certain non-degeneracy condition for the braiding, called modularity, is a \emph{\mtc{}}.
\end{defi}
From now on let $\CC$ be a \mtc{}. Since we assumed the unit object $\unit\in\CC$ to be simple, the trace $\tr(f)\in\End_\CC(\unit)$ is simply a complex multiple of $\id_\unit$ and we will view $\tr(f)\in\C$ in the following. In the same way, we view the quantum dimensions $d(A)\in\C$. For a simple object $A\in\CC$ the twist isomorphism $\theta_A\in\End_\CC(A)=\C\id_A$ can also be viewed as a non-zero complex number that only depends on the isomorphism class of $A$.

Since $\CC$ is semisimple, given two simple objects $A,B\in\CC$, we can decompose the tensor product as
\begin{equation*}
A\otimes B\cong\bigoplus_{C}N_{A,B}^C C
\end{equation*}
where the direct sum runs over the finitely many isomorphism classes of simple objects in $\CC$ and the numbers
\begin{equation*}
N_{A,B}^C=\dim_\C(\Hom_\CC(C,A\otimes B))\in\N
\end{equation*}
are called \emph{fusion rules} and only depend on the isomorphism classes of $A$, $B$ and $C$. This is of course completely analogous to the definition of the fusion rules for \voa{}s (see Section~\ref{sec:intops}).

\minisec{Frobenius-Schur Indicator}

The Frobenius-Schur indicator is defined in \cite{FS03} for any sovereign category and specialises to the Frobenius-Schur indicator in a \mtc{} defined in \cite{FFFS02}. First note that for a simple object $A$ in a \mtc{}
\begin{equation*}
1=\dim_\C(\Hom_\CC(A,A))=\dim_\C(\Hom_\CC(\unit,A\otimes A^*))=\dim_\C(\Hom_\CC(A\otimes A^*,\unit))
\end{equation*}
and analogously with $A^*$ and $A$ interchanged.

We call an object $A\in\CC$ \emph{self-dual} if $A\cong A^*$.

\begin{defi}[Frobenius-Schur Indicator]\label{defi:fsmtc}
Let $\CC$ be a \mtc{} and $A$ a simple, self-dual object with isomorphism $\phi_A\colon A\to A^*$. Then the space $\Hom_\CC(\unit,A^*\otimes A)$ is one-dimensional and the factor of proportionality $\nu(A)$ in
\begin{equation*}
\tilde{b}_A=\nu(A)(\phi_A\otimes\phi_A^{-1})\circ b_A\in\Hom_\CC(\unit,A^*\otimes A)
\end{equation*}
is called the \emph{Frobenius-Schur indicator} of $A$. 
\end{defi}
Easy consequences of the definition are $\nu(A)^2=1$, i.e.\ $\nu(A)=\pm1$, $\nu(\unit)=1$ and the fact that $\nu(A)$ only depends on the isomorphism class of $A$ (see \cite{FFFS02}, Lemma~2.1). We can equivalently define the Frobenius-Schur indicator as
\begin{equation*}
\tilde{d}_A=\nu(A)d_A\circ(\phi_A\otimes\phi_A^{-1})\in\Hom_\CC(A\otimes A^*,\unit)
\end{equation*}
or, more symmetrically, as
\begin{equation*}
d(A)=\nu(A)d_A\circ(\phi_A\otimes\phi_A^{-1})\circ b_A,
\end{equation*}
involving the quantum dimension $d(A)$ of $A$ and where both sides are non-zero multiples of $\id_\unit\in\End_\CC(\unit)=\C\id_\unit$ since the unit object $\unit$ is simple.

For a simple object in $A\in\CC$ which is not self-dual we set $\nu(A):=0$.

There is a remarkable formula due to Bántay, expressing the Frobenius-Schur indicator in terms of the quantum dimensions, the fusion rules and the twist isomorphism.
\begin{prop}[Bántay's Formula, \cite{Ban97}, formula~(1)]\label{prop:bantay}
Let $C$ be a simple, not necessarily self-dual object in a \mtc{} $\CC$. Then the Frobenius-Schur indicator is given by
\begin{equation*}
\nu(C)=\frac{1}{D^2}\sum_{A,B}\frac{\theta_A^2}{\theta_B^2}N_{A,C}^Bd(A)d(B)\in\{0,\pm1\}
\end{equation*}
with
\begin{equation*}
D^2=\sum_{A,B}\frac{\theta_A}{\theta_B}d(A)^2d(B)^2
\end{equation*}
where in each case the sum ranges over the finitely many isomorphism classes of simple objects in $\CC$.
\end{prop}
A proof of this statement can be found in \cite{NS07}, Theorem~7.5 and \cite{Wan10}, Theorem~4.25.

\minisec{Huang's \MTC{}}
In the following we describe Huang's construction of a \mtc{} associated with a suitably regular \voa{}.
\begin{thm}[\cite{Hua08b}, Theorem~4.6]\label{thm:voamtc}
Let $V$ be a \voa{} satisfying Assumption~\ref{ass:n}. Then the category of $V$-modules naturally admits the structure of a \mtc{}.
\end{thm}
The proof of the theorem, in particular showing rigidity, makes use of the Verlinde formula for \voa{}s, also proved by Huang \cite{Hua08} (see Theorem~\ref{thm:verlinde}).

In the following we describe some elements of Huang's construction in more detail. Given a \voa{} $V$, let $\CC$ be the abelian, $\C$-linear category of $V$-modules and module maps. Assume that $V$ satisfies Assumption~\ref{ass:n}. The simple objects in $\CC$ are the irreducible $V$-modules and the rationality of $V$ implies the semisimplicity of $\CC$ and that there are only finitely many isomorphism classes of simple objects.

Recall that for each $z\in\C^\times$ there is the Huang-Lepowsky tensor product (or fusion product) $\boxtimes_{P(z)}$. The tensor-product bifunctor $\otimes$ on $\CC$ is taken to be the $P(1)$-tensor product $\boxtimes_{P(1)}$, simply denoted by $\boxtimes$. The unit object $\unit$ for the tensor product is the \voa{} $V$ itself.

When passing from a formal power series in $x$ with non-integer exponents to a function in a complex variable $z$, it is important to have a consistent choice of the logarithm. Recall that we chose the branch of $\log(z)$ such that $0\leq\Im(\log(z))<2\pi$. Also recall from Section~\ref{sec:formal} that given an intertwining operator $\mathcal{Y}(\cdot,x)$ we use the convention
\begin{equation*}
\mathcal{Y}(\cdot,z)=\mathcal{Y}(\cdot,x)|_{x^n=\e^{n\log(z)},\,n\in\C}
\end{equation*}
for $z\in\C^\times$. Note that $\mathcal{Y}(\cdot,z)$ is a $P(z)$-intertwining map (see \cite{HL95b}, Section~12 for the correspondence between intertwining operators and $P(z)$-intertwining maps). We also use the (conflicting) shorthand notation
\begin{equation*}
\mathcal{Y}(\cdot,\e^\zeta)=\mathcal{Y}(\cdot,x)|_{x^n=\e^{n\zeta},\,n\in\C}
\end{equation*}
for $\zeta\in\C$. In particular, if $l(z)$ is some other branch of the logarithm, then, using this notation, $\mathcal{Y}(\cdot,z)\neq\mathcal{Y}(\cdot,\e^{l(z)})$ even though $z=\e^{l(z)}$.

There is a natural isomorphism of vector spaces from the space of module maps $\Hom_V(W^1\boxtimes_{P(z)}W^2,W^3)$ to the space of intertwining operators of type $\binom{W^3}{W^1\,W^2}$ (cf.\ Proposition~\ref{prop:cor3.4}). This isomorphism is based on the choice of the logarithm. Then there is a canonical intertwining operator of type $\binom{W^1\boxtimes_{P(z)}W^2}{W^1\,W^2}$ corresponding to the identity module map on $W^1\boxtimes_{P(z)}W^2$. We denote this intertwining operator by $\mathcal{Y}_{W^1,W^2}^{\boxtimes_{P(z)}}(\cdot,x)$. Given $w_1\in W^1$, $w_2\in W^2$ we define for $z\in\C^\times$ the $P(z)$-tensor-product element
\begin{equation*}
w_1\boxtimes_{P(z)}w_2:=\mathcal{Y}_{W^1,W^2}^{\boxtimes_{P(z)}}(w_1,z)w_2\in \overline{W^1\boxtimes_{P(z)}W^2}
\end{equation*}
where for any $V$-module $W=\coprod_{n\in\C}W_n=\bigoplus_{n\in\C}W_n$ we denote by $\overline{W}=\prod_{n\in\C}W_n$ its algebraic completion. The homogeneous components of $w_1\boxtimes_{P(z)}w_2$ for all $w_1\in W^1$, $w_2\in W^2$ span the tensor-product module $W^1\boxtimes_{P(z)}W^2$. However, the set of all tensor-product elements has in fact almost no intersection with the tensor-product module. A notable exception is $\vac\boxtimes_{P(z)}w$ for $w\in W$, $W$ a $V$-module, which lies in $V\boxtimes_{P(z)}W$. Note that $\vac$ denotes the vacuum vector in $V$ and not the unit object $\unit$ in $\CC$.

The left and right unit isomorphisms $l_W\colon V\boxtimes W\to W$ and $r_W\colon W\boxtimes V\to W$ are characterised by
\begin{equation*}
l_W(\vac\boxtimes w)=w\quad\text{and}\quad\overline{r_W}(w\boxtimes\vac)=\e^{L_{-1}}w
\end{equation*}
for $w\in W$ where $\overline{r_W}$ denotes the natural extension of $r_W$ to the completion $\overline{W\boxtimes V}$ of $W\boxtimes V$.

In order to describe the braiding on $\CC$ we need the notion of \emph{parallel transport} providing an isomorphism between $P(z)$-tensor products for different $z$. Given $V$-modules $W^1$ and $W^2$, $z_1,z_2\in\C^\times$ and a path $\gamma$ in $\C^\times$ from $z_1$ to $z_2$, the parallel-transport isomorphism $\mathcal{T}_\gamma\colon W^1\boxtimes_{P(z_1)}W^2\to W^1\boxtimes_{P(z_2)}W^2$ is defined via
\begin{equation*}
\overline{\mathcal{T}_\gamma}(w_1\boxtimes_{P(z_1)}w_2)=\mathcal{Y}_{W^1,W^2}^{\boxtimes_{P(z_2)}}(w_1,\e^{l(z_1)})w_2
\end{equation*}
where $l(z_1)$ is the value of the logarithm of $z_1$ determined uniquely by $\log(z_2)$ satisfying $0\leq\Im(\log(z))<2\pi$ and the path $\gamma$. Here, $\overline{\mathcal{T}_\gamma}$ denotes the natural extension of $\mathcal{T}_\gamma$ to the completion $\overline{W^1\boxtimes_{P(z_1)}W^2}$ of $W^1\boxtimes_{P(z_1)}W^2$. Clearly, the parallel-transport isomorphism only depends on the homotopy class of $\gamma$ in $\C^\times$.

The braiding isomorphism $c_{W^1,W^2}\colon W^1\boxtimes W^2\to W^2\boxtimes W^1$ can now be described as follows: let $\gamma_1^-$ be a path from $-1$ to $1$ in $\overline{\H}\setminus\{0\}$, the closed upper half-plane with 0 deleted, with the corresponding parallel-transport isomorphism $\mathcal{T}_{\gamma_1^-}\colon W^2\boxtimes_{P(-1)}W^1\to W^2\boxtimes W^1$. Then
\begin{equation*}
\overline{c_{W^1,W^2}}(w_1\boxtimes w_2)=\e^{L_{-1}}\overline{\mathcal{T}_{\gamma_1^-}}(w_2\boxtimes_{P(-1)}w_1)
\end{equation*}
for $w_1\in W^1$, $w_2\in W^2$. In terms of intertwining operators, or $P(z)$-intertwining maps to be precise, this equals
\begin{align*}
\overline{c_{W^1,W^2}}\left(\mathcal{Y}_{W^1,W^2}^{\boxtimes}(w_1,1)w_2\right)&=\e^{L_{-1}}\overline{\mathcal{T}_{\gamma_1^-}}\left(\mathcal{Y}_{W^2,W^1}^{\boxtimes_{P(-1)}}(w_2,-1)w_1\right)\\
&=\e^{L_{-1}}\mathcal{Y}_{W^2,W^1}^{\boxtimes}(w_2,\e^{l(-1)})w_1\\
&=\e^{L_{-1}}\mathcal{Y}_{W^2,W^1}^{\boxtimes}(w_2,\e^{\pi\i})w_1,
\end{align*}
where $l(-1)=\pi\i$ due to the choice of the path $\gamma_1^-$ from $-1$ to $1$.

For a $V$-module $W$ the twist isomorphism $\theta_W\colon W\to W$ is given by 
\begin{equation*}
\theta_W(w)=\e^{(2\pi\i)L_0}w
\end{equation*}
for $w\in W$. In particular, if $W$ is irreducible, then $\theta_W=\e^{(2\pi\i)\rho(W)}\id_W$.

\minisec{Rigidity}

We describe the sovereign structure on $\CC$. In fact, the proof of left and right rigidity for $\CC$ is the main result in \cite{Hua08b}. The left and right dual of a $V$-module $W$ is given by the contragredient module $W'$. Recall that the rigid structure consists of the right and left duality morphisms $b_W\in\Hom_V(V,W\boxtimes W')$, $d_W\in\Hom_V(W'\boxtimes W,V)$, $\tilde{b}_W\in\Hom_V(V,W'\boxtimes W)$ and $\tilde{d}_W\in\Hom_V(W\boxtimes W',V)$.

Since the category $\CC$ is semisimple, it suffices to consider irreducible modules. Let the finitely many isomorphism classes of irreducible $V$-modules be labelled by the set $F$ with the isomorphism class of $V$ corresponding to $0\in F$. We fix a choice of representatives $W^\alpha$, $\alpha\in F$, with $W^0=V$. For $\alpha\in F$, let $\alpha'\in F$ denote the index of the contragredient module of $W^\alpha$, i.e.\ $(W^\alpha)'\cong W^{\alpha'}$. If an irreducible module $W^\alpha$ is self-contragredient, i.e.\ $\alpha=\alpha'$, then there is a module isomorphism $\phi_{W^\alpha}\colon W^\alpha\to(W^\alpha)'$, unique up to a complex scalar, which defines a non-degenerate, invariant bilinear form $(\cdot,\cdot)_{W^\alpha}$ on $W^\alpha$ via $(u,w)_{W^\alpha}:=\langle\phi_{W^\alpha}(u),w\rangle$ for $u,w\in W^\alpha$ where $\langle\cdot,\cdot\rangle$ denotes the natural pairing between $(W^\alpha)'$, the graded dual space of $W^\alpha$, and $W^\alpha$ (see Section~\ref{sec:cmibf}). On $V$, this non-degenerate, invariant bilinear form is symmetric and for convenience we assume that it is normalised as $(\vac,\vac)_V=1$. This fixes the isomorphism $\phi_V$. We also fix choices of the $\phi_{W^\alpha}$, $\alpha\in F$.

Note that the spaces of intertwining operators of types
\begin{equation*}
\binom{W^\alpha}{V\,W^\alpha}, \binom{W^\alpha}{W^\alpha\,V}, \binom{V'}{W^\alpha\,(W^\alpha)'}, \binom{V}{W^\alpha,(W^\alpha)'}, \binom{V}{(W^\alpha)',W^\alpha}
\end{equation*}
are all one-dimensional since they are isomorphic to $\End_V(W^\alpha)$ by Proposition~\ref{prop:cor3.4} and the $S_3$-symmetry and this space is one-dimensional by Schur's lemma.

We recall this $S_3$-symmetry of the intertwining operators (see Section~\ref{sec:intops}). Let $\mathcal{Y}(\cdot,x)$ be an intertwining operator of type $\binom{W^\gamma}{W^\alpha\,W^\beta}$. Then $\sigma_{12}(\mathcal{Y})$ is defined to be the intertwining operator of type $\binom{W^\gamma}{W^\beta\,W^\alpha}$ given by
\begin{equation*}
\sigma_{12}(\mathcal{Y})(w_\alpha,x)w_\beta=\e^{\pi\i(\rho(\gamma)-\rho(\alpha)-\rho(\beta))}\e^{xL_{-1}}\mathcal{Y}(w_\beta,\e^{-\pi\i}x)w_\alpha
\end{equation*}
for $w_\alpha\in W^\alpha$, $w_\beta\in W^\beta$ where $\rho(\alpha):=\rho(W^\alpha)$ denotes the conformal weight of $W^\alpha$, $\alpha\in F$. Again, to clarify the notation, recall from Section~\ref{sec:formal} that given an intertwining operator $\mathcal{Y}(\cdot,x)$ we define $\mathcal{Y}(\cdot,\e^\zeta x)$ for $\zeta\in\C$ as
\begin{equation*}
\mathcal{Y}(\cdot,\e^\zeta x)=\mathcal{Y}(\cdot,y)|_{y^n=\e^{n\zeta}x^n,\,n\in\C}
\end{equation*}
where $x$ and $y$ are formal variables. Furthermore, $\sigma_{23}(\mathcal{Y})$ is defined to be the intertwining operator of type $\binom{(W^\beta)'}{W^\alpha\,(W^\gamma)'}$ determined by
\begin{equation*}
\left\langle\sigma_{23}(\mathcal{Y})(w_\alpha,x)w'_\gamma,w_\beta\right\rangle=\e^{\pi\i\rho(\alpha)}\left\langle w'_\gamma,\mathcal{Y}(\e^{xL_1}\e^{-\pi\i L_0}x^{-2L_0}w_\alpha,x^{-1})w_\beta\right\rangle
\end{equation*}
for $w_\alpha\in W^\alpha$, $w_\beta\in W^\beta$, $w'_\gamma\in (W^\gamma)'$. Upon identifying $W^\alpha$ with $(W^{\alpha})''$ for all $\alpha\in F$ one obtains
\begin{equation*}
\sigma_{12}^2=1,\quad\sigma_{23}^2=1\quad\text{and}\quad\sigma_{12}\sigma_{23}\sigma_{12}=\sigma_{23}\sigma_{12}\sigma_{23},
\end{equation*}
which shows that $\sigma_{12}$ and $\sigma_{23}$ define a representation of $S_3$ on the space of intertwining operators $\bigoplus_{\alpha,\beta,\gamma\in F}\V^{W^\gamma}_{W^\alpha \, W^\beta}$.

In the following we will define certain non-zero intertwining operators in the one-dimensional spaces of intertwining operators introduced above. Let
\begin{equation*}
\mathcal{Y}_{V,W^\alpha}^{W^\alpha}:=Y_{W^\alpha}
\end{equation*}
be the module vertex operation of $V$ on $W^\alpha$. This is an intertwining operator of type $\binom{W^\alpha}{V\,W^\alpha}$. Furthermore, we define
\begin{equation*}
\mathcal{Y}_{W^\alpha,V}^{W^\alpha}:=\sigma_{12}(\mathcal{Y}_{V,W^\alpha}^{W^\alpha})=\sigma_{12}(Y_{W^\alpha}),
\end{equation*}
which is an intertwining operator of type $\binom{W^\alpha}{W^\alpha\,V}$. We then define
\begin{equation*}
\mathcal{Y}_{W^\alpha,(W^\alpha)'}^{V}:=\phi_V^{-1}\sigma_{23}(\mathcal{Y}_{W^\alpha,V}^{W^\alpha})=\phi_V^{-1}\sigma_{23}\sigma_{12}(Y_{W^\alpha})
\end{equation*}
of type $\binom{V}{W^\alpha\,(W^\alpha)'}$ where $\phi_V$ is the isomorphism $V\to V'$. Finally,
\begin{equation*}
\mathcal{Y}_{(W^\alpha)',W^\alpha}^{V}:=\sigma_{12}(\mathcal{Y}_{W^\alpha,(W^\alpha)'}^{V})=\phi_V^{-1}\sigma_{12}\sigma_{23}\sigma_{12}(Y_{W^\alpha})
\end{equation*}
defines an intertwining operator of type $\binom{V}{(W^\alpha)'\,W^\alpha}$.

Then, the right evaluation morphism $d_{W^\alpha}\colon (W^\alpha)'\boxtimes W^\alpha\to V$ is defined via
\begin{equation*}
\overline{d_{W^\alpha}}(w'_\alpha\otimes w_\alpha):=d(\alpha)\mathcal{Y}_{(W^\alpha)',W^\alpha}^{V}(w'_\alpha,1)w_\alpha
\end{equation*}
and the left one $\tilde{d}_{W^\alpha}\colon W^\alpha\boxtimes(W^\alpha)'\to V$ via
\begin{equation*}
\overline{\tilde{d}_{W^\alpha}}(w_\alpha\otimes w'_\alpha):=d(\alpha)\mathcal{Y}_{W^\alpha,(W^\alpha)'}^{V}(w_\alpha,1)w'_\alpha
\end{equation*}
for $w_\alpha\in W^\alpha$, $w'_\alpha\in (W^\alpha)'$ where $d(\alpha)$ is the quantum dimension of $W^\alpha$, which only depends on the isomorphism class of $W^\alpha$. Again, $\overline{d_{W^\alpha}}\colon\overline{(W^\alpha)'\boxtimes W^\alpha}\to\overline{V}$ is the natural extension of $d_{W^\alpha}\colon (W^\alpha)'\boxtimes W^\alpha\to V$ and similarly for $\tilde{d}_{W^\alpha}$.

We omit the definition of the coevaluation morphisms $b_{W^\alpha}$ and $\tilde{b}_{W^\alpha}$ but remark that if we choose two lowest-weight vectors $w\in W^\alpha$, $w'\in W^\alpha$ with $\langle w',w\rangle=1$, then
\begin{equation*}
b_{W^\alpha}(\vac)=P_0(w\boxtimes w')\in W^\alpha\boxtimes(W^\alpha)'
\end{equation*}
where $P_0$ is the projection onto the weight-zero space (see proof of Theorem~3.9 in \cite{Hua08b}).

\minisec{Frobenius-Schur Indicator for \VOA{}s}

Let $V$ be a \voa{} and $W$ some irreducible, self-contragredient $V$-module. We saw in Section~\ref{sec:cmibf} that, given an isomorphism $\phi_W\colon W\to W'$,
\begin{equation*}
(u,w)_W:=\langle\phi_W(u),w\rangle,
\end{equation*}
$u,w\in W$, defines a non-degenerate, invariant bilinear form on $W$ and this is the unique non-degenerate, invariant bilinear form on $W$ up to a scalar. Moreover, this bilinear form is either symmetric or antisymmetric.
\begin{defi}[Frobenius-Schur Indicator]\label{defi:voafsi}
Let $V$ be a \voa{}. For an irreducible, self-contragredient module $W$ we define the \emph{Frobenius-Schur indicator} $\nu(W):=1$ and $\nu(W):=-1$ if the bilinear form $(\cdot,\cdot)_W$ on $W$ is symmetric and antisymmetric, respectively. For an irreducible module $W$ that is not self-contragredient we set $\nu(W):=0$.
\end{defi}
If $V$ is simple, then Proposition~\ref{prop:2.6} shows that $\nu(V)=1$.

There are two notions of Frobenius-Schur indicator for a \voa{} $V$:
\begin{enumerate}
\item\label{enum:fsi1} the modular tensor categorical definition of the Frobenius-Schur indicator (Definition~\ref{defi:fsmtc}) applied to Huang's construction of the modular tensor category associated with $V$ (Theorem~\ref{thm:voamtc}) if $V$ satisfies Assumption~\ref{ass:n}, 
\item\label{enum:fsi2} the definition via the symmetry of the invariant bilinear form (Definition~\ref{defi:voafsi}).
\end{enumerate}

The following proposition shows that these two notions coincide.
\begin{prop}
Let $V$ be a \voa{} satisfying Assumption~\ref{ass:n}. Then the two notions \ref{enum:fsi1} and \ref{enum:fsi2} of Frobenius-Schur indicator agree.
\end{prop}
\begin{proof}
We compute the Frobenius-Schur indicator in the \mtc{} $\CC$ obtained from Huang's construction and show that it is $+1$ or $-1$ for an irreducible module $W^\alpha$ depending on whether the non-degenerate, invariant bilinear form on $W^\alpha$ is symmetric or antisymmetric.

Let $W^\alpha$ be an irreducible, self-contragredient $V$-module. We introduced the module isomorphism $\phi_{W^\alpha}\colon W^\alpha\to(W^\alpha)'$ defining a non-degenerate, invariant bilinear form $(\cdot,\cdot)_{W^\alpha}$ on $W^\alpha$ via $(u,w)_{W^\alpha}:=\langle\phi_{W^\alpha}(u),w\rangle$ for $u,w\in W^\alpha$.

By definition, the Frobenius-Schur indicator $\nu(\alpha)$ is given as the factor of proportionality between
\begin{equation*}
\tilde{d}_{W^\alpha}\circ b_{W^\alpha}\quad\text{and}\quad d_{W^\alpha}\circ(\phi_{W^\alpha}\boxtimes\phi_{W^\alpha}^{-1})\circ b_{W^\alpha},
\end{equation*}
which lie in $\End_V(V)=\C\id_V$. For the left-hand side we get
\begin{equation*}
\tilde{d}_{W^\alpha}(b_{W^\alpha}(\vac))=d(\alpha)\id_V
\end{equation*}
by the definition of the quantum dimension. The computation of the right-hand side is more involved. We evaluate the expression at the vacuum $\vac$. We stated above that
\begin{equation*}
b_{W^\alpha}(\vac)=P_0(w\boxtimes w')\in W^\alpha\boxtimes(W^\alpha)'
\end{equation*}
for two lowest-weight vectors $w\in W^\alpha$ and $w'\in (W^\alpha)'$ with $\langle w',w\rangle=1$. Then
\begin{align*}
d_{W^\alpha}((\phi_{W^\alpha}\boxtimes\phi_{W^\alpha}^{-1})(b_{W^\alpha}(\vac)))&=d_{W^\alpha}((\phi_{W^\alpha}\boxtimes\phi_{W^\alpha}^{-1})(P_0(w\boxtimes w')))\\
&=d_{W^\alpha}(P_0(\phi_{W^\alpha}(w)\boxtimes\phi_{W^\alpha}^{-1}(w')))\\
&=P_0\left(\overline{d_{W^\alpha}}(\phi_{W^\alpha}(w)\boxtimes\phi_{W^\alpha}^{-1}(w'))\right)\\
&=d(\alpha)P_0\left(\mathcal{Y}_{(W^\alpha)',W^\alpha}^{V}(\phi_{W^\alpha}(w),1)\phi_{W^\alpha}^{-1}(w')\right)\\
&=d(\alpha)\Res_x \left(x^{2\rho(\alpha)-1}\mathcal{Y}_{(W^\alpha)',W^\alpha}^{V}(\phi_{W^\alpha}(w),x)\phi_{W^\alpha}^{-1}(w')\right)
\end{align*}
where $\Res_x(\cdot)$ is the coefficient of $x^{-1}$. By definition of $\mathcal{Y}_{(W^\alpha)',W^\alpha}^{V}(\cdot,x)$ the above is
\begin{equation*}
d(\alpha)\e^{-(2\pi\i)\rho(\alpha)}\Res_x \left(x^{2\rho(\alpha)-1}\e^{xL_{-1}}\mathcal{Y}_{W^\alpha,(W^\alpha)'}^{V}(\phi^{-1}_{W^\alpha}(w'),\e^{-\pi\i}x)\phi_{W^\alpha}(w)\right).
\end{equation*}
Then we take the bilinear form of the above expression and $\vac\in V$ and by definition of $\mathcal{Y}_{W^\alpha,(W^\alpha)'}^{V}(\cdot,x)$ we obtain
\begin{align*}
&\left(d(\alpha)\e^{-(2\pi\i)\rho(\alpha)}\Res_x \left(x^{2\rho(\alpha)-1}\e^{xL_{-1}}\mathcal{Y}_{W^\alpha,(W^\alpha)'}^{V}(\phi^{-1}_{W^\alpha}(w'),\e^{-\pi\i}x)\phi_W^{\alpha}(w)\right),\vac\right)_V\\
&=d(\alpha)\e^{-(2\pi\i)\rho(\alpha)}\Res_x\left(x^{2\rho(\alpha)-1}\left(\e^{xL_{-1}}\mathcal{Y}_{W^\alpha,(W^\alpha)'}^{V}(\phi^{-1}_{W^\alpha}(w'),\e^{-\pi\i}x)\phi_{W^\alpha}(w),\vac\right)_V\right)\\
&=d(\alpha)\e^{-(2\pi\i)\rho(\alpha)}\Res_x\left(x^{2\rho(\alpha)-1}\left(\mathcal{Y}_{W^\alpha,(W^\alpha)'}^{V}(\phi^{-1}_{W^\alpha}(w'),\e^{-\pi\i}x)\phi_{W^\alpha}(w),\e^{xL_1}\vac\right)_V\right)_V\\
&=d(\alpha)\e^{-(2\pi\i)\rho(\alpha)}\Res_x\left(x^{2\rho(\alpha)-1}\left\langle\phi_V\left(\mathcal{Y}_{W^\alpha,(W^\alpha)'}^{V}(\phi^{-1}_{W^\alpha}(w'),\e^{-\pi\i}x)\phi_{W^\alpha}(w)\right),\vac\right\rangle\right)\\
&=d(\alpha)\e^{-(2\pi\i)\rho(\alpha)}\Res_x\left(x^{2\rho(\alpha)-1}\left\langle(\sigma_{23}\mathcal{Y}_{W^\alpha,V}^{W^\alpha})(\phi^{-1}_{W^\alpha}(w'),\e^{-\pi\i}x)\phi_{W^\alpha}(w),\vac\right\rangle\right)\\
&=d(\alpha)\e^{-(2\pi\i)\rho(\alpha)}\e^{\pi\i\rho(\alpha)}\Res_x\left(x^{2\rho(\alpha)-1}\right.\\
&\quad\left.\left\langle\phi_{W^\alpha}(w),\mathcal{Y}_{W^\alpha,V}^{W^\alpha}\left(\e^{\e^{-\pi\i}xL_1}\e^{-\pi\i L_0}(\e^{-\pi\i}x)^{-2L_0}\phi^{-1}_{W^\alpha}(w'),(\e^{-\pi\i}x)^{-1}\right)\vac\right\rangle\right),
\end{align*}
where we used $L_1\vac=0$. Now we use that $w'$ and hence $\phi_{W^\alpha}^{-1}(w')$ is a lowest-weight vector so that $L_1\phi_{W^\alpha}^{-1}(w')=0$ and $L_0\phi_{W^\alpha}^{-1}(w')=\rho(\alpha)\phi_{W^\alpha}^{-1}(w')$. Then the above equals
\begin{align*}
&d(\alpha)\e^{-(2\pi\i)\rho(\alpha)}\Res_x\!\left(x^{2\rho(\alpha)-1}\!\left\langle\phi_{W^\alpha}(w),\mathcal{Y}_{W^\alpha,V}^{W^\alpha}\left((\e^{-\pi\i}x)^{-2\rho(\alpha)}\phi^{-1}_{W^\alpha}(w'),(\e^{-\pi\i}x)^{-1}\right)\vac\right\rangle\right)\\
&=d(\alpha)\e^{-(2\pi\i)\rho(\alpha)}\Res_x\!\left(x^{2\rho(\alpha)-1}\!\left\langle\phi_{W^\alpha}(w),\mathcal{Y}_{W^\alpha,V}^{W^\alpha}\left(\e^{(2\pi\i)\rho(\alpha)}x^{-2\rho(\alpha)}\phi^{-1}_{W^\alpha}(w'),\e^{\pi\i}x^{-1}\right)\vac\right\rangle\right)\\
&=d(\alpha)\Res_x\!\left(x^{-1}\!\left\langle\phi_{W^\alpha}(w),\mathcal{Y}_{W^\alpha,V}^{W^\alpha}\left(\phi^{-1}_{W^\alpha}(w'),\e^{\pi\i}x^{-1}\right)\vac\right\rangle\right),
\end{align*}
where we simplified the expressions $(\e^{-\pi\i}x)^{-2\rho(\alpha)}$ and $(\e^{-\pi\i}x)^{-1}$ consistent with our notational conventions. Finally, we use the definition of $\mathcal{Y}_{W^\alpha,V}^{W^\alpha}=\sigma_{12}(Y_{W^\alpha})$ to obtain
\begin{align*}
&d(\alpha)\Res_x\left(x^{-1}\left\langle\phi_{W^\alpha}(w),\e^{\e^{\pi\i}x^{-1}L_{-1}}Y_{W^\alpha}\left(\vac,\e^{-\pi\i}\e^{\pi\i}x^{-1}\right)\phi^{-1}_{W^\alpha}(w')\right\rangle\right)\\
&=d(\alpha)\Res_x\left(x^{-1}\left\langle\e^{\e^{\pi\i}x^{-1}L_1}\phi_{W^\alpha}(w),\phi^{-1}_{W^\alpha}(w')\right\rangle\right)\\
&=d(\alpha)\Res_x\left(x^{-1}\left\langle\phi_{W^\alpha}(w),\phi^{-1}_{W^\alpha}(w')\right\rangle\right)\\
&=d(\alpha)\left\langle\phi_{W^\alpha}(w),\phi^{-1}_{W^\alpha}(w')\right\rangle\\
&=d(\alpha)\left(w,\phi^{-1}_{W^\alpha}(w')\right)_{W^\alpha}.
\end{align*}
Hence, in total we have computed
\begin{equation*}
\left((d_{W^\alpha}\circ(\phi_{W^\alpha}\boxtimes\phi_{W^\alpha}^{-1})\circ b_{W^\alpha})(\vac),\vac\right)_V=d(\alpha)\left(w,\phi^{-1}_{W^\alpha}(w')\right)_{W^\alpha},
\end{equation*}
which we compare with
\begin{equation*}
\left((\tilde{d}_{W^\alpha}\circ b_{W^\alpha})(\vac),\vac\right)_V=d(\alpha)(\vac,\vac)_V=d(\alpha),
\end{equation*}
recalling the normalisation $(\vac,\vac)_V=1$. On the other hand, we chose $\langle w',w\rangle=1$ and, setting $u:=\phi_{W^\alpha}^{-1}(w')$, this reads $(u,w)_{W^\alpha}=1$ in terms of the bilinear form on $W^\alpha$.

Hence, the Frobenius-Schur indicator of $W^\alpha$ is given by
\begin{equation*}
\nu(\alpha)=\left(w,\phi^{-1}_{W^\alpha}(w')\right)_{W^\alpha}=(w,u)_{W^\alpha}=\frac{(w,u)_{W^\alpha}}{(u,w)_{W^\alpha}},
\end{equation*}
i.e.\ it coincides with the definition via the symmetry of the invariant bilinear form on $W^\alpha$ for special non-zero choices of $u,w\in W^\alpha$ and consequently for all $u,w\in W^\alpha$.
\end{proof}

The Frobenius-Schur indicator does not only appear in the symmetry relation for the invariant bilinear form on an irreducible, self-contragredient module but also in a relation similar to the skew-symmetry formula \eqref{eq:skew} for a \voa{} V:
\begin{equation*}
Y_V(a,x)b=\e^{xL_{-1}}Y_V(b,-x)a
\end{equation*}
for $a,b\in V$. Indeed:
\begin{prop}\label{prop:fsiskew}
Let $V$ be a \voa{} satisfying Assumption~\ref{ass:n} and let $W^\alpha$ be an irreducible, self-contragredient module. Let $\mathcal{Y}(x,\cdot)$ be an intertwining operator of type $\binom{V}{W^\alpha\,W^\alpha}$. Then
\begin{equation*}
\mathcal{Y}(w,x)u=\nu(\alpha)\e^{-(2\pi\i)\rho(\alpha)}\e^{xL_{-1}}\mathcal{Y}(u,\e^{-\pi\i}x)w
\end{equation*}
for $u,w\in W^\alpha$.
\end{prop}
\begin{proof}
Note that the space of intertwining operators of type $\binom{V}{W^\alpha\,(W^\alpha)'}$ is, as remarked above, one-dimensional. This means that
\begin{equation*}
\mathcal{Y}_{W^\alpha,(W^\alpha)'}^{V}(\cdot,x)
\end{equation*}
is proportional to
\begin{equation*}
\mathcal{Y}_{(W^\alpha)',W^\alpha}^{V}(\phi_{W^\alpha}(\cdot),x)\phi_{W^\alpha}^{-1}.
\end{equation*}
We call the constant of proportionality $\nu(W^\alpha)$ and obtain
\begin{align*}
\mathcal{Y}_{W^\alpha,(W^\alpha)'}^{V}(w,x)w'&=\nu(W^\alpha)\mathcal{Y}_{(W^\alpha)',W^\alpha}^{V}(\phi_{W^\alpha}(w),x)\phi_{W^\alpha}^{-1}(w')\\
&=\nu(W^\alpha)\e^{-(2\pi\i)\rho(\alpha)}\e^{xL_{-1}}\mathcal{Y}_{W^\alpha,(W^\alpha)'}^{V}(\phi_{W^\alpha}^{-1}(w'),\e^{-\pi\i}x)\phi_{W^\alpha}(w)
\end{align*}
for $w\in W^\alpha$ and $w'\in (W^\alpha)'$. Replacing the formal variable $x$ by $z=1\in\C$ and recalling the definition of $\tilde{d}_{W^\alpha}$ and $d_{W^\alpha}$ in terms of $\mathcal{Y}_{W^\alpha,(W^\alpha)'}^{V}(\cdot,1)$ and $\mathcal{Y}_{(W^\alpha)',W^\alpha}^{V}(\cdot,1)$ we see that $\nu(W^\alpha)=\nu(\alpha)$ is exactly the Frobenius-Schur indicator. 

Let us define the intertwining operator $\mathcal{Y}(x,\cdot)$ of type $\binom{V}{W^\alpha\,W^\alpha}$ by
\begin{equation*}
\mathcal{Y}(\cdot,x)w:=\mathcal{Y}_{W^\alpha,(W^\alpha)'}^{V}(\cdot,x)\phi_{W^\alpha}(w).
\end{equation*}
Then, setting $u=\phi_{W^\alpha}^{-1}(w')$, we obtain
\begin{equation*}
\mathcal{Y}(w,x)u=\nu(\alpha)\e^{-(2\pi\i)\rho(\alpha)}\e^{xL_{-1}}\mathcal{Y}(u,\e^{-\pi\i}x)w
\end{equation*}
for $u,w\in W^\alpha$, which is the relation we want to show. Since the space of intertwining operators of type $\binom{V}{W^\alpha\,W^\alpha}$ is also one-dimensional, this formula holds for any intertwining operator of that type.
\end{proof}
\begin{rem}
\item\begin{enumerate}
\item The factor $\nu(\alpha)\e^{-(2\pi\i)\rho(\alpha)}$ in the above proposition describes the failure of the skew-symmetry relation $\mathcal{Y}(a,x)b=\e^{xL_{-1}}\mathcal{Y}(b,\e^{-\pi\i}x)a$ to hold, which would be true if $\mathcal{Y}(\cdot,x)$ were simply the vertex operation on $V$.

\item The above proposition is a generalisation of Proposition~5.6.1 in \cite{FHL93} and Lemma~2.1 in \cite{Yam13}, which deal with the special cases of $\rho(\alpha)\in\Z$ and $\rho(\alpha)\in(1/2)\Z$, respectively. We do not impose any restrictions on the value of $\rho(\alpha)$.\footnote{Note however that in the case of group-like fusion discussed below it is easy to see that $\rho(\alpha)\in(1/4)\Z$ for any irreducible, self-contragredient module $W^\alpha$.}

\item The proposition actually holds without Assumption~\ref{ass:n} if we use the definition of the Frobenius-Schur indicator not involving \mtcs{}. This is the approach in \cite{FHL93,Yam13}.
\end{enumerate}
\end{rem}

\minisec{Quantum Dimensions for \VOA{}s}

We introduced the notion of quantum dimensions (or categorical dimensions) for \mtcs{}. We can in particular study them for the \mtcs{} associated with certain \voa{}s from Huang's construction. There is also another notion of quantum dimensions defined directly for \voa{}s and it will turn out that under suitable regularity assumptions on $V$, notably Assumption~\ref{ass:p}, both notions agree.

Given a \voa{} $V$ and a $V$-module $W$, assume that the characters $\ch_V(\tau)$ and $\ch_W(\tau)$ are well-defined functions on the upper half-plane. This is the case for example if $V$ is rational and $C_2$-cofinite and $W$ is an irreducible $V$-module (see Theorem~\ref{thm:zhumodinv}). The \emph{quantum dimension} $\qdim_V(W)$ of $W$ is defined as the limit
\begin{equation*}
\qdim_V(W):=\lim_{y\to 0^+}\frac{\ch_W(iy)}{\ch_V(iy)}.
\end{equation*}
It is a priori not clear that this number exists. However:
\begin{prop}[\cite{DJX13}, Lemma~4.2]\label{prop:lem4.2}
Let $V$ be a simple, rational, $C_2$-cofinite \voa{} of CFT-type which satisfies Assumption~\ref{ass:p}. Then for any irreducible module $W\in\Irr(V)$ the quantum dimension exists and
\begin{equation*}
0<\qdim_V(W)=\frac{\S_{W,V}}{\S_{V,V}}<\infty.
\end{equation*}
\end{prop}
\begin{prop}[\cite{DLN15}, Proposition~3.11]\label{prop:3.11}
Let $V$ satisfy Assumptions~\ref{ass:n}\ref{ass:p}. Consider the \mtc{} $\CC$ associated with $V$ from Huang's construction. Then for any irreducible $V$-module $W$
\begin{equation*}
d(W)=\qdim_V(W)
\end{equation*}
where $d(W)$ is the quantum dimension of the simple object $W$ in $\CC$.
\end{prop}

Quantum dimensions play an important rôle since they characterise simple-current modules:
\begin{prop}[\cite{DJX13}, Proposition~4.17]\label{prop:4.17b}
Let $V$ satisfy Assumptions~\ref{ass:n}\ref{ass:p}. Then $W\in\Irr(V)$ is a simple current if and only if $\qdim_V(W)=1$.
\end{prop}

\minisec{Group-Like Fusion}
Finally, let us assume that the \voa{} $V$ has group-like fusion, i.e.\ that $V$ satisfies Assumption~\ref{ass:sn}. Then there is an abelian group structure on the set $F$ indexing the isomorphism classes of irreducible modules $\{W^\alpha\;|\;\alpha\in F\}$ and the fusion group $F_V=F$ carries the quadratic form $Q_\rho(\alpha)=\rho(W^\alpha)+\Z$. In the following, in addition to $Q_\rho$, we also consider the multiplicative quadratic form
$q_\rho\colon F_V\to\C^\times$,
\begin{equation*}
q_\rho(\alpha)=\e^{(2\pi\i)\rho(W^\alpha)}
\end{equation*}
for $\alpha\in F_V$.

In addition to the modular tensor category structure on the $V$-modules, the direct sum of all irreducible $V$-modules $A=\bigoplus_{\alpha\in F_V}W^\alpha$ carries the structure of an \aia{} by Theorem~\ref{thm:pre2.7} with the vertex operation on $A$ composed of the intertwining operators $\mathcal{Y}_{W^\alpha,W^\beta}^+(\cdot,x)$ for $\alpha,\beta\in F_V$ from Section~\ref{sec:aiascvoa1}. Part of this structure is the quadratic form $q_\Omega\colon F_V\to\C^\times$,
\begin{equation*}
q_\Omega(\alpha)=\Omega(\alpha,\alpha)
\end{equation*}
for $\alpha\in F_V$.

In the following we express the twist isomorphism $\theta_{W^\alpha}$, the Frobenius-Schur indicator $\nu(\alpha)$ in two different ways and the braiding $c_{W^\alpha,W^\beta}$ in terms of the quadratic forms $q_\Omega$ and $q_\rho$. Recall that the twist is defined by Huang to be
\begin{equation*}
\theta_{W^\alpha}=q_\rho(\alpha)\id_{W^\alpha}
\end{equation*}
for any irreducible module $W^\alpha$, $\alpha\in F_V$.

The following is an immediate consequence of Proposition~\ref{prop:fsiskew} and the braiding convention in Definition~\ref{defi:OmegaDef}.
\begin{prop}\label{prop:fsi1}
Let $V$ be a \voa{} satisfying Assumption~\ref{ass:sn}. Let $\alpha\in F_V$ with $2\alpha=0$, i.e.\ $W^\alpha$ is self-contragredient. Then
\begin{equation*}
\nu(\alpha)=\frac{q_\rho(\alpha)}{q_\Omega(\alpha)}.
\end{equation*}
\end{prop}
\begin{proof}
The braiding (see Definition~\ref{defi:OmegaDef}) implies
\begin{equation*}
\mathcal{Y}_{W^\alpha,W^\alpha}^+(w,x)u=\frac{1}{q_\Omega(\alpha)}\e^{xL_{-1}}\mathcal{Y}_{W^\alpha,W^\alpha}^+(u,\e^{-\pi\i}x)w
\end{equation*}
for $u,w\in W^\alpha$ where $\mathcal{Y}_{W^\alpha,W^\alpha}^+(\cdot,x)$ is a certain choice of intertwining operator of type $\binom{V}{W^\alpha\,W^\alpha}$. Then the statement follows from Proposition~\ref{prop:fsiskew}.
\end{proof}

We can also compute the Frobenius-Schur indicator using Bántay's formula (see Proposition~\ref{prop:bantay}) and using that the quantum dimensions are all one in the case of group-like fusion and under Assumption~\ref{ass:p}.
\begin{prop}\label{prop:scefsi}
Let $V$ be a \voa{} satisfying Assumptions~\ref{ass:sn}\ref{ass:p}. Let $\alpha\in F_V$ with $2\alpha=0$, i.e.\ $W^\alpha$ is self-contragredient. Then
\begin{equation*}
\nu(\alpha)=\frac{1}{q_\rho(\alpha)^2}.
\end{equation*}
\end{prop}
\begin{proof}
Propositions \ref{prop:3.11} and \ref{prop:4.17b} show that the quantum dimensions $d(\alpha)=1$ for all $\alpha\in F_V$. Then Bántay's formula gives
\begin{align*}
D^2&=\sum_{\alpha,\beta\in F_V}\e^{(2\pi\i)(Q_\rho(\alpha)-Q_\rho(\beta))}=\sum_{\alpha,\beta\in F_V}\e^{(2\pi\i)(Q_\rho(\alpha+\beta)-Q_\rho(\beta))}\\
&=\sum_{\alpha,\beta\in F_V}\e^{(2\pi\i)(B_\rho(\alpha,\beta)+Q_\rho(\alpha))}=\sum_{\alpha\in F_V}\left(\sum_{\beta\in F_V}\e^{(2\pi\i)B_\rho(\alpha,\beta)}\right)\e^{(2\pi\i)Q_\rho(\alpha)}\\
&=\sum_{\alpha\in F_V}|F_V|\delta_{\alpha,0}\e^{(2\pi\i)Q_\rho(\alpha)}=|F_V|,
\end{align*}
using that $2\gamma=0$ and hence $2B_\rho(\alpha,\gamma)=0+\Z$ for all $\alpha\in F_V$. Then
\begin{align*}
\nu(\gamma)&=\frac{1}{|F_V|}\sum_{\alpha,\beta\in F_V}\e^{(2\pi\i)(2Q_\rho(\alpha)-2Q_\rho(\beta))}\delta_{\alpha+\gamma,\beta}=\frac{1}{|F_V|}\sum_{\alpha\in F_V}\e^{(2\pi\i)(2Q_\rho(\alpha)-2Q_\rho(\alpha+\gamma))}\\
&=\frac{1}{|F_V|}\sum_{\alpha\in F_V}\e^{(2\pi\i)(-2B_\rho(\alpha,\gamma)-2Q_\rho(\gamma))}=\e^{(2\pi\i)(-2Q_\rho(\gamma))}=\frac{1}{q_\rho(\alpha)^2}
\end{align*}
since $N_{\alpha,\gamma}^\beta=\delta_{\alpha+\gamma,\beta}$.
\end{proof}
Note that since $2\alpha=0$ for a self-contragredient module, $Q_\rho(\alpha)\in(1/4)\Z$ since $Q_\rho$ is a quadratic form and hence $1/q_\rho(\alpha)^2$ indeed lies in $\{\pm1\}$.

\begin{prop}\label{prop:braidingaia}
Let $V$ be a \voa{} satisfying Assumption~\ref{ass:sn}. Then the braiding isomorphism in the \mtc{} $\CC$ is given by
\begin{equation*}
c_{W^\alpha,W^\alpha}=q_\Omega(\alpha)q_\rho^2(\alpha)\id_{W^\alpha\boxtimes W^\alpha}
\end{equation*}
for all $\alpha\in F_V$.
\end{prop}
\begin{proof}
Recall the definition of the braiding isomorphism
\begin{equation*}
\overline{c_{W^\alpha,W^\beta}}\left(\mathcal{Y}_{W^\alpha,W^\beta}^{\boxtimes}(w,1)u\right)=\e^{L_{-1}}\mathcal{Y}_{W^\beta,W^\alpha}^{\boxtimes}(u,\e^{\pi\i})w
\end{equation*}
for $w\in W^\alpha$, $u\in W^\beta$ where $\mathcal{Y}_{W^\alpha,W^\beta}^{\boxtimes}(\cdot,x)$, $\alpha,\beta\in F_V$, are canonical intertwining operators of type $\binom{W^\alpha\boxtimes W^\beta}{W^\alpha\,W^\beta}$.

On the other hand, to construct the \aia{} in the case of group-like fusion, we chose intertwining operators $\mathcal{Y}_{W^\alpha,W^\beta}^+(\cdot,x)$ of type $\binom{W^{\alpha+\beta}}{W^\alpha\,W^\beta}$. This corresponds to choosing module isomorphisms $\psi_{W^\alpha,W^\beta}\colon W^\alpha\boxtimes W^\beta\to W^{\alpha+\beta}$ such that $\mathcal{Y}_{W^\alpha,W^\beta}^+(\cdot,x)=\psi_{W^\alpha,W^\beta}\mathcal{Y}_{W^\alpha,W^\beta}^{\boxtimes}(\cdot,x)$.

The skew-symmetry formula (see Definition~\ref{defi:OmegaDef}) gives
\begin{equation*}
\mathcal{Y}_{W^\alpha,W^\beta}^+(w,x)u=\frac{1}{\Omega(\beta,\alpha)}\e^{xL_{-1}}\mathcal{Y}_{W^\beta,W^\alpha}^+(u,\e^{-\pi\i}x)w,
\end{equation*}
which we reformulate as
\begin{equation*}
\psi_{W^\alpha,W^\beta}\left(\mathcal{Y}_{W^\alpha,W^\beta}^{\boxtimes}(w,x)u\right)=\frac{1}{\Omega(\beta,\alpha)b_\rho(\alpha,\beta)}\e^{xL_{-1}}\psi_{W^\beta,W^\alpha}\left(\mathcal{Y}_{W^\beta,W^\alpha}^{\boxtimes}(u,\e^{\pi\i}x)w\right)
\end{equation*}
using that the exponents of $\mathcal{Y}_{W^\beta,W^\alpha}^{\boxtimes}(\cdot,x)$ lie in $B_\rho(\alpha,\beta)+\Z$. Finally, setting $\alpha=\beta$ we obtain
\begin{equation*}
\mathcal{Y}_{W^\alpha,W^\alpha}^{\boxtimes}(w,x)u=\frac{1}{q_\Omega(\alpha)q_\rho(\alpha)^2}\e^{xL_{-1}}\mathcal{Y}_{W^\alpha,W^\alpha}^{\boxtimes}(u,\e^{\pi\i}x)w.
\end{equation*}
Inserting $z=1$ for $x$ and comparing with the definition of $c_{W^\alpha,W^\alpha}$ yields
\begin{equation*}
q_\Omega(\alpha)q_\rho(\alpha)^2\mathcal{Y}_{W^\alpha,W^\alpha}^{\boxtimes}(w,1)u=\e^{L_{-1}}\mathcal{Y}_{W^\alpha,W^\alpha}^{\boxtimes}(u,\e^{\pi\i})w=\overline{c_{W^\alpha,W^\beta}}\left(\mathcal{Y}_{W^\alpha,W^\alpha}^{\boxtimes}(w,1)u\right),
\end{equation*}
which proves the statement.
\end{proof}

\section{Simple-Current \AIA{}s II}\label{sec:aiascvoa2}

In this section we give two independent proofs of Theorem~\ref{thm:2.7}, i.e.\ we show that in the situation of Theorem~\ref{thm:pre2.7} the two quadratic forms associated with the \aia{} fulfil $Q_\Omega=-Q_\rho$ or $q_\Omega=1/q_\rho$. The first one is shorter and makes use of Proposition~\ref{prop:2.32} from \cite{DGNO10}. The second proof uses Bántay's remarkable formula for the Frobenius-Schur indicator. 

\minisec{First Proof of Theorem~\ref{thm:2.7}}

\begin{proof}[Proof of Theorem~\ref{thm:2.7}]
Consider the formula
\begin{equation*}
\tr_{W^\alpha\boxtimes W^\alpha}(c_{W^\alpha,W^\alpha}^{-1})=\tr_{W^\alpha}(\theta_{W^\alpha}^{-1})
\end{equation*}
from Proposition~\ref{prop:2.32}. We insert the formulæ $c_{W^\alpha,W^\alpha}=q_\Omega(\alpha)q_\rho^2(\alpha)\id_{W^\alpha\boxtimes W^\alpha}$ and $\theta_{W^\alpha}=q_\rho(\alpha)\id_{W^\alpha}$ from the previous section (see Proposition~\ref{prop:braidingaia}) and obtain
\begin{equation*}
d(2\alpha)\frac{1}{q_\Omega(\alpha)q_\rho^2(\alpha)}=d(\alpha)\frac{1}{q_\rho(\alpha)}.
\end{equation*}
Since we are in the situation of group-like fusion and Assumption~\ref{ass:p}, we conclude with Propositions \ref{prop:3.11} and \ref{prop:4.17b} that all quantum dimensions are one. This implies
\begin{equation*}
q_\Omega(\alpha)=1/q_\rho(\alpha)
\end{equation*}
for all $\alpha\in F_V$.
\end{proof}

\minisec{Second Proof of Theorem~\ref{thm:2.7}}

Combining the formulæ for the Frobenius-Schur indicator from Propositions \ref{prop:fsi1} and \ref{prop:scefsi} immediately gives the following result:
\begin{lem}\label{lem:qOmqrhoscm}
Let $V$ be a \voa{} satisfying Assumptions~\ref{ass:sn}\ref{ass:p}. Let $\alpha\in F_V$ with $2\alpha=0$, i.e.\ $W^\alpha$ is self-contragredient. Then
\begin{equation*}
q_\Omega(\alpha)=1/q_\rho(\alpha).
\end{equation*}
\end{lem}
This is the desired formula, however only for self-contragredient modules. Naturally, the Frobenius-Schur indicator only makes statements about self-contragredient modules. In the following we extend this formula to all modules by considering certain \aia{}s associated with lattices. First, we need another lemma:
\begin{lem}\label{lem:lemthm2.7}
Let $V$ satisfy Assumption~\ref{ass:sn}. Then for a subgroup $I\leq F_V$ the direct sum $V_I=\bigoplus_{\gamma\in I}W^\gamma$ admits the structure of a \voa{} whose vertex operators are composed of rescaled intertwining operators between the irreducible modules $W^\gamma$, $\gamma\in I$, if and only if
\begin{equation*}
q_\rho(\gamma)=1\quad\text{and}\quad q_\Omega(\gamma)=1
\end{equation*}
for all $\gamma\in I$.
\end{lem}
\begin{proof}
Clearly, if $V_I$ is to be a \voa{}, then all $W^\gamma$, $\gamma\in I$, have to be $\Z$-graded, which is the case if and only if $q_\rho=1$ on $I$. Also, in order to get the Jacobi identity for a \voa{}, it has to be possible to rescale the intertwining operators between the $W^\gamma$, $\gamma\in I$, such that all the factors $F$ and $B$ in the Jacobi identity for \aia{}s become 1, which is the case if and only if $(F|_I,\Omega_I)$ is in the trivial cohomology class in $H^3_\text{ab.}(I,\C^\times)$ or equivalently $q_\Omega=1$ on $I$.

Now assume that $q_\rho=1$ and $q_\Omega=1$ on $I$. Then there is a $\Z$-grading on $V_I$ and we can rescale the intertwining operators such that $F|_I=1$ and $\Omega|_I=1$ and hence they fulfil the Jacobi identity for a \voa{} on $V_I$. One can check that all the other axioms of a \voa{} are also fulfilled (see Proposition~\ref{prop:aiavoa}). Note for example that since the vacuum axioms hold in the original \aia{} and the corresponding abelian 3-cocycle is normalised by definition, the vacuum axioms also hold after rescaling.
\end{proof}

The next lemma shows that if $q_\rho=1$ on a subgroup of $F_V$, then already $q_\Omega=q_\rho=1$ on this subgroup.
\begin{lem}\label{lem:thm2.7}
Let $V$ be as in Assumptions~\ref{ass:sn}\ref{ass:p}. Let $I$ be a subgroup of the fusion group $F_V$ isotropic with respect to the quadratic form $q_\rho$. Then also $q_\Omega=1$ on $I$ and the direct sum
\begin{equation*}
V_I:=\bigoplus_{\gamma\in I}W^\gamma
\end{equation*}
carries the structure of a \voa{}, more precisely that of an $I$-graded simple-current extension of $V$.
\end{lem}
\begin{proof}
We define the quadratic form $\tilde{q}$ on $F_V$ by $\tilde{q}(\alpha):=q_\rho(\alpha)q_\Omega(\alpha)$ for all $\alpha\in F_V$. Then, since the bilinear forms $b_\rho$ and $b_\Omega^{-1}$ on $F_V$ are identical, the bilinear form associated with $\tilde{q}$ is constantly 1, which means that $\tilde{q}$ is a homomorphism. Then $\tilde{q}(\alpha)^4=\tilde{q}(2\alpha)=\tilde{q}(\alpha)^2$ and hence $\tilde{q}(2\alpha)=\tilde{q}(\alpha)^2=1$. This implies
that $\tilde{q}(2\alpha)=1$ for all $\alpha\in F_V$.

Consider $H:=2I=\{2\alpha\;|\;\alpha\in I\}$. By assumption, $q_\rho|_H=1$ and the above implies that $q_\Omega|_H=q_\rho|_H=1$. Then $V_{H}$ admits the structure of a \voa{} by Lemma~\ref{lem:lemthm2.7}.
The \voa{} $V_H$ is clearly an $H$-graded simple-current extension of $V$. By Theorem~\ref{thm:scemodclass} and Corollary~\ref{cor:scefusion} we know that the isomorphism classes of irreducible $V_H$-modules
\begin{equation*}
X^{\gamma+H}=\bigoplus_{\alpha\in\gamma+H}W^\gamma
\end{equation*}
are indexed by $\gamma+H\in H^\bot/H$ and are all simple currents with the fusion rules determined by the quotient group structure on $H^\bot/H$. Since we made Assumption~\ref{ass:p} for $V$, which is of CFT-type, also $V_H$ is of CFT-type and Assumption~\ref{ass:p} holds again for $V_H$. Then, in total, it follows from Propositions \ref{prop:scerc2} and \ref{prop:sceselfcon} that $V_H$ fulfils Assumptions~\ref{ass:sn}\ref{ass:p}. By Theorem~\ref{thm:pre2.7} the direct sum of the irreducible $V_H$-modules admits the structure of an \aia{} associated with some normalised abelian 3-cocycle $(F_H,\Omega_H)$ and corresponding quadratic form $q_{\Omega_H}$ on $H^\bot/H$. In addition, there is the quadratic form $q_{\rho_H}$ determined by the conformal weights of the irreducible $V_H$-modules.

We consider the subgroup $I/H=I/(2I)$ of $H^\bot/H$, which is isotropic with respect to $q_{\rho_H}$ by assumption. Moreover, by construction, every element of $I/H$ has order at most 2, i.e.\ every module $X^{\gamma+H}$, $\gamma+H\in I/H$, is self-contragredient (see Proposition~\ref{prop:sceselfcon}). This allows us to apply Lemma~\ref{lem:qOmqrhoscm} and conclude that $q_{\Omega_H}(\gamma+H)=q_{\rho_H}(\gamma+H)=1$ for all $\gamma+H\in I/H$. Then Lemma~\ref{lem:lemthm2.7} implies that
\begin{equation*}
\bigoplus_{\gamma+H\in I/H}X^{\gamma+H}=\bigoplus_{\alpha\in I}W^\alpha=V_I
\end{equation*}
has the structure of a \voa{} extending $V_{2I}$ and hence also extending $V$. Applying once again Lemma~\ref{lem:lemthm2.7} we conclude that $q_\Omega|_I=1$.
\end{proof}
This lemma already gives a solution to the extension problem for \voa{}s but we will formulate this as a separate theorem in the next section (see Theorem~\ref{thm:4.3}). We are now ready to give a second proof of Theorem~\ref{thm:2.7}.
\begin{proof}[Proof of Theorem~\ref{thm:2.7}]
The \voa{} $V$ satisfies Assumptions~\ref{ass:sn}\ref{ass:p}. For convenience we set $V^{(1)}:=V$. By Theorem~\ref{thm:pre2.7} the direct sum $A^{(1)}$ of the irreducible $V^{(1)}$-modules up to isomorphism, indexed by the fusion group $F_{V^{(1)}}$, has the structure of an \aia{}. There are two quadratic forms $Q^{(1)}_\rho$ and $Q^{(1)}_\Omega$ on $F_{V^{(1)}}$, whose associated bilinear forms are the negatives of each other.

It is well known that for every \fqs{} there is a positive-definite, even lattice $L$ whose discriminant form $L'/L$ is isomorphic to this \fqs{}. Hence, let $L$ be a positive-definite, even lattice with discriminant form $(L'/L,Q_L)\cong\overline{F_{V^{(1)}}}=(F_{V^{(1)}},-Q_\rho^{(1)})$. Consider the lattice \voa{} $V^{(2)}:=V_L$ associated with $L$, which also satisfies Assumptions~\ref{ass:sn}\ref{ass:p} and has fusion group $F_{V^{(2)}}\cong L'/L\cong\overline{F_{V^{(1)}}}$ (see Section~\ref{sec:latvoa}). To simplify notation we identify $F_{V^{(1)}}=F_{V^{(2)}}=L'/L$ as sets.

Let $A^{(2)}:=A_L$ be the \aia{} on the direct sum of all $V_L$-modules naturally associated with $L$ (see Theorem~\ref{thm:lataia}). Let us denote the quadratic forms on $A^{(2)}=A_L$ associated with the conformal weights and the abelian 3-cocycle by $Q_\rho^{(2)}$ and $Q_\Omega^{(2)}$, respectively. By construction of $A^{(2)}=A_L$ the quadratic forms $Q_\rho^{(2)}$ and $-Q_\Omega^{(2)}$ are identical and given by $Q_L$, i.e.\
\begin{equation*}
Q_\rho^{(2)}(\alpha)=-Q_\Omega^{(2)}(\alpha)=Q_L(\alpha)=-Q_\rho^{(1)}(\alpha)
\end{equation*}
for all $\alpha\in F_{V^{(2)}}=F_{V^{(1)}}$.

Now consider the tensor-product \voa{} $V^{(1)}\otimes V^{(2)}$. It satisfies Assumption~\ref{ass:sn}\ref{ass:p} (see Section~\ref{sec:tensor}) and its modules are parametrised the group $F_{V^{(1)}}\times F_{V^{(2)}}$, which is also the fusion group. We consider the \aia{} defined on the direct sum of all $V^{(1)}\otimes V^{(2)}$-modules, which is up to an abelian 3-coboundary on $F_{V^{(1)}}\times F_{V^{(2)}}$ the tensor-product \aia{} $A^{(1)}\otimes A^{(2)}$ of $A^{(1)}$ and $A^{(2)}$ defined in \cite{DL93}, Proposition~12.39. The quadratic form associated with the conformal weights for $A^{(1)}\otimes A^{(2)}$ is
\begin{equation*}
Q_\rho(\alpha,\beta)=Q_\rho^{(1)}(\alpha)+Q_\rho^{(2)}(\beta)=Q_\rho^{(1)}(\alpha)-Q_\rho^{(1)}(\beta),
\end{equation*}
while the one associated with the abelian 3-cocycle is
\begin{equation*}
Q_\Omega(\alpha,\beta)=Q_\Omega^{(1)}(\alpha)+Q_\Omega^{(2)}(\beta)=Q_\Omega^{(1)}(\alpha)+Q_\rho^{(1)}(\beta),
\end{equation*}
$(\alpha,\beta)\in F_{V^{(1)}}\times F_{V^{(2)}}$. Now consider the diagonal subgroup
\begin{equation*}
I:=\{(\alpha,\alpha)\;|\;\alpha\in F_{V^{(1)}}\}\leq F_{V^{(1)}}\times F_{V^{(2)}}.
\end{equation*}
This subgroup is by construction isotropic with respect to $Q_\rho(\alpha,\beta)$ and hence it is also isotropic with respect to $Q_\Omega(\alpha,\beta)$ by Lemma~\ref{lem:thm2.7}. This implies
\begin{equation*}
Q_\rho^{(1)}(\alpha)=-Q_\Omega^{(1)}(\alpha)
\end{equation*}
for all $\alpha\in F_V$.
\end{proof}

\section{Extension Theorem}

An important problem in the theory of \voa{}s is the \emph{extension problem}, i.e.\ the problem of combining modules of a given \voa{} into a new \voa{}. In the following we give a complete answer in the case of \voa{}s satisfying Assumptions~\ref{ass:sn}\ref{ass:p}.

The following theorem is an immediate consequence of the main theorem of this chapter (Theorem~\ref{thm:2.7}) and Lemma~\ref{lem:lemthm2.7} although it is also essentially the statement of Lemma~\ref{lem:thm2.7}, which was a key ingredient in the second proof of Theorem~\ref{thm:2.7}.
\begin{oframed}
\begin{thm}[Extension Theorem]\label{thm:4.3}
Let $V$ be as in Assumptions~\ref{ass:sn}\ref{ass:p}. Let $I$ be a subgroup of $F_V$ isotropic with respect to $Q_\rho=-Q_\Omega$. Then:
\begin{enumerate}
\item The direct sum
\begin{equation*}
V_I:=\bigoplus_{\gamma\in I}W^\gamma
\end{equation*}
carries the structure of a \voa{} satisfying Assumptions~\ref{ass:sn}\ref{ass:p} and is an $I$-graded simple-current extension of $V$.
\item The \voa{} structure on $V_I$ is the unique one up to isomorphism extending the \voa{} structures on $V$ and the $V$-module structure on the $W^\gamma$, $\gamma\in I$.
\item The fusion group of $V_I$ is given by $F_{V_I}\cong I^\bot/I$. Specifically, the irreducible $V_I$-modules up to isomorphism are
\begin{equation*}
X^{\alpha+I}:=\bigoplus_{\gamma\in\alpha+I}W^{\gamma},\quad\alpha+I\in I^\bot/I,
\end{equation*}
they are all simple currents, the fusion rules are determined by the quotient group structure and the conformal weight of the module $X^{\alpha+I}$ is in $Q_\rho(\alpha)=Q_\rho(\gamma)$ for all $\gamma\in\alpha+I$.
\item In particular, $V_I$ is holomorphic if and only if $I=I^\bot$.
\end{enumerate}
\end{thm}
\end{oframed}
\begin{proof}
The existence of the \voa{} structure is explained above. In fact, $V_I$ is an $I$-graded simple-current extension of $V$ and hence simple. By Proposition~\ref{prop:scerc2} $V_I$ is rational and $C_2$-cofinite. Assumption~\ref{ass:p} for $V$ immediately yields that $V_I$ is of CFT-type. $V_I$ is self-contragredient by Proposition~\ref{prop:sceselfcon}. The irreducible modules are determined by Theorem~\ref{thm:scemodclass} and the fusion rules by Theorem~\ref{thm:scefusion} and Corollary~\ref{cor:scefusion}. The uniqueness follows from item~\ref{enum:sce3} of Proposition~\ref{prop:scerc2}.
\end{proof}
Note that in \cite{CKL15}, Theorem~1.3, the authors obtain a similar extension result under different assumptions.

\chapter{Orbifold Theory}\label{ch:fpvosa}

In this chapter we study the \fpvosa{} $V^G$ of a holomorphic \voa{} $V$ under a finite, cyclic group $G=\langle\sigma\rangle$ of automorphisms of $V$. We determine the isomorphism classes of irreducible $V^G$-modules using the Schur-Weyl type duality theory \cite{MT04} and determine the fusion algebra of $V^G$ by carefully studying the fusion rules and $S$-matrix entries using the Verlinde formula \cite{Hua08}, the modular invariance results from \cite{Zhu96,DLM00} and the results from Section~\ref{sec:scvoa} on simple-current \voa{}s. As a result we show that under suitable regularity assumptions $V^G$ has group-like fusion and that its fusion group is isomorphic to a central extension of $\Z_n$ by $\Z_n$ whose isomorphism type is determined by the conformal weight of the irreducible $\sigma$-twisted $V$-module $V(\sigma)$. We also describe the level of the corresponding trace functions. The study of $V^G$ and its representation theory is called \emph{orbifold theory}.

Under the positivity assumption the direct sum of irreducible $V^G$-modules is an \aia{} and the restriction of this direct sum to an isotropic subgroup with respect to the conformal weights is a \voa{} extending $V^G$. If this \voa{} is again holomorphic, we call it \emph{orbifold} of $V$.

\section{\FPVOSA{}s}

Given a \voa{} $V$ of central charge $c$ and a finite group $G$ of automorphisms of $V$, one can study the \emph{\fpvosa{}} $V^G$ defined as the vectors in $V$ being pointwise invariant under $G$ endowed with the restriction of the \voa{} structure from $V$ to $V^G$, which again yields a \voa{} structure of central charge $c$ on $V^G$.

In the following we describe results from \cite{DM97,Miy15,CM16} showing that if $V$ satisfies Assumption~\ref{ass:n}, then, for certain $G$, so does $V^G$. The next two results are well known:
\begin{thm}[\cite{DM97}, Theorem~4.4]\label{thm:4.4}
Let $V$ be a simple \voa{} and $G$ a finite group of automorphisms of $V$. Then the \fpvosa{} $V^G$ is again simple.
\end{thm}

The following is a corollary to Proposition~\ref{prop:subsd} and Theorem~\ref{thm:4.4}:
\begin{cor}\label{cor:fixsd}
Let $V$ be a simple, self-contragredient \voa{} of CFT-type and $G$ a finite group of automorphisms of $V$. Then the \fpvosa{} $V^G$ is again simple, self-contragredient and of CFT-type.
\end{cor}
\begin{proof}
It only remains to show that $V^G$ is of CFT-type but this is trivial since $G$ preserves the vacuum vector $\vac$.
\end{proof}

The following two recent contributions are due to Miyamoto and Carnahan. They show that under certain assumptions the \fpvosa{} inherits rationality and $C_2$-cofiniteness. One cannot stress enough the importance of these results. 
\begin{thm}[\cite{Miy15}, Main Theorem]\label{thm:miymain}
Let $V$ be a simple, $C_2$-cofinite \voa{} of CFT-type and $G$ a finite, solvable group of automorphisms of $V$. Then the \fpvosa{} $V^G$ is also $C_2$-cofinite.
\end{thm}

\begin{thm}[\cite{CM16}, Corollary~5.25]\label{thm:a}
Let $V$ be a simple, rational, $C_2$-cofinite, self-contragredient \voa{} of CFT-type and $G$ a finite, solvable group of automorphisms of $V$. Then the \fpvosa{} $V^G$ is also rational.
\end{thm}

We collect the statements of the above results:
\begin{thm}[Main Theorem of Orbifold Theory]\label{thm:orb}
Let $V$ satisfy Assumption~\ref{ass:n}. Let $G$ be a finite, solvable group of automorphisms of $V$. Then the \fpvosa{} $V^G$ also satisfies Assumption~\ref{ass:n}.
\end{thm}

\section{Schur-Weyl Type Duality}\label{sec:fpvosa}

We will see in the next section that under suitable regularity assumptions all the irreducible modules of $V^G$ appear as submodules of the $g$-twisted $V$-modules for $g\in G$. Therefore, in order to determine the irreducible modules of $V^G$, it suffices to study those appearing as submodules of twisted $V$-modules.

For a \voa{} V, let $W$ be an irreducible $g$-twisted $V$-module for some $g\in G$ (an untwisted module for $g=\id_V$). Then $W$ is an untwisted $V^G$-module and one can ask the following questions:
\begin{enumerate}
\item Is $W$ completely reducible as $V^G$-module?
\item What are the irreducible $V^G$-modules occurring as submodules of $W$?
\item What are the relations between irreducible $V^G$-submodules of $W$ and irreducible $V^G$-submodules of another irreducible twisted $V$-module $N$? 
\end{enumerate}
These questions have been completely answered in the untwisted case by \cite{DY02} (and \cite{DLM96,DM97}) and in the twisted case by \cite{MT04} (and \cite{Yam01}) in what is dubbed \emph{Schur-Weyl type duality}. In these works it is shown:
\begin{thm}[\cite{MT04}, Theorem~2]
Let $V$ be a simple \voa{}, $G$ a finite group of automorphisms of $V$ and $g\in G$. Let $W$ be an irreducible $g$-twisted $V$-module. Then $W$ is a completely reducible $V^G$-module.
\end{thm}
Moreover, the isomorphism classes of irreducible $V^G$-submodules of $W$ are explicitly determined (see Theorem~\ref{thm:irrdec} below in the case of cyclic $G$).

\minisec{Special Case: $V$ Holomorphic, $G$ Cyclic}

In the following, we will as a special case describe the situation where $V$ is a holomorphic, and hence simple and rational, and $C_2$-cofinite \voa{}.

In general, in order to determine the irreducible $V^G$-modules from the twisted $V$-modules it is necessary to have a classification of these twisted modules. For a general, say rational, \voa{} such a classification is not known. However, if $V$ is holomorphic, i.e.\ has only one untwisted, irreducible module up to isomorphism, then it is shown in \cite{DLM00} that it only has one $h$-twisted module $V(h)$ up to isomorphism for each finite-order automorphism $h$ of $V$, given that $V$ is also $C_2$-cofinite (see Theorem~\ref{thm:10.3}).

Moreover, we assume that $G=\langle\sigma\rangle$ is a cyclic group of automorphisms of $V$ of order $n\in\Ns$, i.e.\ $G\cong\Z_n$. This is the simplest possible case apart from the trivial case $G=\langle\id_V\rangle$.

By definition, the automorphism $\sigma$ of $V$ has to fulfil the relation
\begin{equation*}
\sigma Y(v,x)\sigma^{-1}=Y(\sigma v,x)
\end{equation*}
for all $v\in V$. In the following we define a representation of $G=\langle\sigma\rangle$ on the twisted $V$-modules $V(h)$, $h\in G$, as vector spaces such that an analogous relation is fulfilled on these modules. Consider the following statement:
\begin{prop}\label{prop:autaction}
Let $V$ be a \voa{} and $G$ a finite group of automorphisms of $V$. For $h\in G$ let $(W,Y_W)$ be an irreducible $h$-twisted $V$-module. For $g\in G$ we define a new $V$-module
\begin{equation*}
(W,Y_W)\cdot g=(W\cdot g,Y_{W\cdot g})
\end{equation*}
where we set $W\cdot g:=W$ as a vector space and $Y_{W\cdot g}(\cdot,x)$ is defined by
\begin{equation*}
Y_{W\cdot g}(v,x):=Y_W(gv,x)
\end{equation*}
for $v\in V$. Then $W\cdot g$ is an irreducible $g^{-1}hg$-twisted $V$-module.
\end{prop}
In the special case we study $G$ is cyclic and in particular abelian. Hence for any automorphism $g\in G$ the module $V(h)\cdot g$ is again an $h$-twisted $V$-module. Moreover, since there is only one $h$-twisted $V$-module up to isomorphism by Theorem~\ref{thm:10.3}, we obtain
\begin{equation*}
V(h)\cdot g\cong V(h)
\end{equation*}
for all $g,h\in G$. In other words, $g$ induces an isomorphism of $V$-modules $\phi_h(g)\colon V(h)\to V(h)\cdot g$. But since $V(h)\cdot g$ and $V(h)$ are identical as vector spaces, this means, by definition of module isomorphism, that $\phi_h(g)$ is an automorphism of the vector space $V(h)$ obeying
\begin{equation*}
\phi_h(g)Y_{V(h)}(v,x)\phi_h(g)^{-1}=Y_{V(h)\cdot g}(v,x)=Y_{V(h)}(gv,x)
\end{equation*}
for all $v\in V$, which is a generalisation of the defining relation for a \voa{} automorphism.

For $g,h\in G$ fixed, the above relation uniquely determines the automorphism $\phi_h(g)$ up to a non-zero scalar factor in $\C$. Indeed, for $\phi_h(g),\phi'_h(g)\in\Aut_{\C}(V(h))$ obeying the relation consider $\phi_h(g)^{-1}\phi'_h(g)$, which fulfils
\begin{equation*}
\phi_h(g)^{-1}\phi'_h(g)Y_{V(h)}(v,x)\phi'_h(g)^{-1}\phi_h(g)=\phi_h(g)^{-1}Y_{V(h)}(gv,x)\phi_h(g)=Y_{V(h)}(v,x)
\end{equation*}
showing that $\phi_h(g)^{-1}\phi'_h(g)$ is a $V$-module automorphism of $V(h)$ and hence has to be multiplication by a non-zero scalar in $\C$ by Schur's lemma (Proposition~\ref{prop:twistedschur}).

We can now make use of this uniqueness: since
\begin{equation*}
\phi_h(k)\phi_h(g)Y_{V(h)}(v,x)\phi_h(g)^{-1}\phi_h(k)^{-1}=\phi_h(k)Y_{V(h)}(gv,x)\phi_h(k)^{-1}=Y_{V(h)}(kgv,x)
\end{equation*}
and
\begin{equation*}
\phi_h(kg)Y_{V(h)}(v,x)\phi_h(kg)^{-1}=Y_{V(h)}(kgv,x)
\end{equation*}
for $g,h,k\in G$, we know that $\phi_h(k)\phi_h(g)$ has to be proportional to $\phi_h(kg)$. In other words there is a proportionality constant $c_h\colon G\times G\to \C^\times$ such that
\begin{equation*}
\phi_h(k)\phi_h(g)=c_h(k,g)\phi_h(kg). 
\end{equation*}
This means that each $\phi_h$ defines a projective representation of $G$ on $V(h)$, $h\in G$. From the associative law applied to $\phi_h(k)\phi_h(g)\phi_h(l)$ it follows that
\begin{equation*}
c_h(k,g)c_h(kg,l)=c_h(k,gl)c_h(g,l)
\end{equation*}
for all $k,g,l\in G$, which means that $c_h$ is a 2-cocycle\footnote{Throughout this text we will always mean 2-cocycles for the \emph{trivial} group action.} in $Z^2(G,\C^\times)$. Moreover, rescaling $\phi_h(g)$ for the different $g\in G$ corresponds exactly to multiplying $c$ by a 2-coboundary in $B^2(G,\C^\times)$. As is well known, this means that the projective representation $\phi_h$ on $V(h)$ determines a well-defined cohomology class in $H^2(G,\C^\times)=Z^2(G,\C^\times)/B^2(G,\C^\times)$.

As is also known, the second cohomology group of $H^2(G,\C^\times)$ is trivial if $G$ is cyclic. Hence, the automorphisms can be chosen such that $c_h(k,g)=1$ for all $k,g\in G$. This means that each $\phi_h$ actually defines a (proper) representation of $G$ on $V(h)$, $h\in G$, via
\begin{equation*}
g.w=\phi_h(g)w
\end{equation*}
for $g\in G$ and $w\in V(h)$.

For each $h\in G$ there is a remaining $n$-fold freedom in the choice of the $\phi_h$. Indeed, let $\sigma$ be the generator of the cyclic group $G$, i.e.\ $G=\langle\sigma\rangle$. Then because of $\phi_h(\sigma)^n=\phi_h(\sigma^n)=\phi_h(\id_V)=\id_{V(h)}$, we know that we can redefine $\phi_h(\sigma)\mapsto\xi_n^j\phi_h(\sigma)$ where $\xi_n=\e^{(2\pi\i)1/n}$ and $j\in\Z_n$. On the other hand, since $\sigma$ generates $G$, all the other values of $\phi_h(g)$ for some $g\in G$ are completely determined by the choice of $\phi_h(\sigma)$.

We summarise our findings in the following proposition:
\begin{prop}\label{prop:phi}
Let $V$ be a holomorphic, $C_2$-cofinite \voa{} and let $G=\langle\sigma\rangle$ be a cyclic group of automorphisms of $V$ of order $n\in\Ns$. For $h\in G$ let $V(h)$ denote the unique irreducible $h$-twisted $V$-module up to isomorphism. Then for each $h\in G$ there is a representation
\begin{equation*}
\phi_h\colon G\to\Aut_\C(V(h))
\end{equation*}
of $G$ on the vector space $V(h)$ such that
\begin{equation*}
\phi_h(g)Y_{V(h)}(v,x)\phi_h(g)^{-1}=Y_{V(h)}(gv,x)
\end{equation*}
for all $g\in G$ and $v\in V$. These representations are unique up to multiplication of $\phi_h(\sigma)$ by an $n$-th root of unity.
\end{prop}
\begin{rem}\label{rem:phi0}
For the unique isomorphism class of untwisted $V$-modules $V(\id_V)$ we choose $V$ itself as a representative. Then $\phi_{\id_V}(g):=g$ is a suitable choice for the automorphisms on $V(\id_V)=V$ and we will make this choice in the following.
\end{rem}

\minisec{Irreducible $V^G$-Submodules}
We describe the irreducible $V^G$-modules appearing as submodules of the $V(h)$, $h\in G$. This is treated in \cite{MT04}. We only discuss the special case of a cyclic group $G$.

The representation of $G$ on $V(h)$ commutes with the module action of $V^G$ on $V(h)$. Indeed, for $v\in V$ with $gv=v$ for all $g\in G$ we obtain $[\phi_h(g),Y_{V(h)}(v,x)]=0$. In particular, since the Virasoro vector $\omega$ is in $V^G$, $\phi_h(g)$ commutes with the modes of $\omega$ and hence the representation $\phi_h$ on $V(h)$ restricts to the finite-dimensional weight spaces $V(h)_\lambda$, $\lambda\in\C$. Then, since $\phi_h(\sigma)$ has order dividing $n$, $V(h)$ decomposes into a direct sum of eigenspaces of $\phi_h(\sigma)$. The possible eigenvalues are the $n$-th roots of unity.

The representation of $G$ on $V(h)$ extends in the obvious way to a representation of the group algebra $\C[G]$ on $V(h)$. Following \cite{MT04} we decompose $V(h)$ into irreducible $\C[G]$-modules. The irreducible complex representations of the cyclic group $G$ are one-dimensional because $G$ is abelian and therefore identical to multiplication by the irreducible characters. These are given by
\begin{equation*}
\lambda_j\colon G\to U_n,\quad \lambda_j(\sigma^k)=\xi_n^{jk},
\end{equation*}
$j,k\in\Z_n$, where $\xi_n=\e^{(2\pi\i)1/n}$ is a primitive $n$-th root of unity and $U_n=\langle\xi_n\rangle$. The irreducible characters of $\C[G]$ are the linear continuations $\lambda_j\colon\C[G]\to\C$.

Each $V(h)$ is completely reducible as a $\C[G]$-module. Write $h=\sigma^i$ for some $i\in\Z_n$. Let $W^{(i,j)}$ be the direct sum of the irreducible $\C[G]$-submodules of $V(\sigma^i)$ isomorphic to the representation $\lambda_j$. In other words, let $W^{(i,j)}$ be the eigenspace in $V(\sigma^i)$ of $\phi_{\sigma^i}(\sigma)$ corresponding to the eigenvalue $\xi_n^j$. Then as $\C[G]$-modules,
\begin{equation*}
V(\sigma^i)=\bigoplus_{j\in\Z_n}W^{(i,j)}.
\end{equation*}
One can then show (special case of Theorem~3.9 in \cite{Yam01}) that the $W^{(i,j)}$ are irreducible $V^G$-modules for all $i,j\in\Z_n$. More precisely:
\begin{thm}[Duality Theorem of Schur-Weyl Type, special case of \cite{MT04}, Theorem~2]\label{thm:irrdec}
Let $V$ be a holomorphic, $C_2$-cofinite \voa{} and $G=\langle\sigma\rangle$ a cyclic group of automorphisms of $V$ of order $n\in\Ns$. Let $A:=\bigoplus_{i\in\Z_n}V(\sigma^i)$ be the direct sum of all twisted $V$-modules. Then:
\begin{enumerate}
\item $A$ decomposes as a $\C[G]\otimes V^G$-module into the direct sum
\begin{equation*}
A=\bigoplus_{i,j\in\Z_n}W^{(i,j)}.
\end{equation*}
\item Each $W^{(i,j)}$, $i,j\in\Z_n$, is an irreducible $V^G$-module.
\item $W^{(i_1,j_1)}$ and $W^{(i_2,j_2)}$ are isomorphic as $V^G$-modules if and only if $i_1=i_2$ and $j_1=j_2$.
\end{enumerate}
\end{thm}

For cyclic $G$ we have therefore determined the irreducible $V^G$-modules appearing as submodules of the twisted $V$-modules $V(\sigma^i)$, $i\in\Z_n$. In the subsequent section we will see that these are already all irreducible $V^G$-modules up to isomorphism.

\section{Classification of Irreducible Modules}
Let $G$ be a finite group of automorphisms of the \voa{} $V$. Then any $g$-twisted $V$-module for some $g\in G$ is also an untwisted $V^G$-module. There is also a converse statement:
\begin{prop}\label{prop:thm3.3}
Let $V$ satisfy Assumption~\ref{ass:n} and let $G$ be a finite group of automorphisms of $V$. Then every irreducible $V^G$-module appears as a $V^G$-submodule of some $g$-twisted $V$-module for some $g\in G$.
\end{prop}
\begin{proof}[Idea of Proof]
The idea of the proof is due to Miyamoto (see proof of Lemma~3 in \cite{Miy10}, where only the case of cyclic $G$ is treated). In the stated generality the statement can be found in \cite{DRX15}, Theorem~3.3. The proof makes use of Zhu's modular invariance for $V^G$ and the fact that certain entries of the $S$-matrix of $V^G$ are non-zero (see Theorem~\ref{thm:verlinde}, Verlinde formula). Both results can be applied since $V^G$ satisfies Assumption~\ref{ass:n} by the above theorem. Finally, Dong, Li and Mason's twisted modular invariance (\cite{DLM00}, Theorem~1.3) is used but, in contrast to the version presented in Theorem~\ref{thm:1.4}, $V$ is not holomorphic. In order to apply the result we have to show that $V$ is $g$-rational for all $g\in G$. This follows from the proof of Lemma~4.2 in \cite{ADJR14}.
\end{proof}

\minisec{Special Case: $V$ Holomorphic, $G$ Cyclic}

In the following we study the special case where $V$ is holomorphic and $G$ cyclic. To this end, let $V$ satisfy the following assumption, which will be made extensively throughout the rest of this chapter:
\begin{customass}{{\textbf{\textsf{O}}}}[Orbifold Assumption]\label{ass:o}
Let $V$ be a holomorphic, $C_2$-cofinite \voa{} of CFT-type and let $G=\langle\sigma\rangle$ be a cyclic group of automorphisms of $V$ of order $n$ where $n$ is any positive integer $n\in\Ns$. For convenience we also assume that the representation $\phi_0$ is chosen as in Remark~\ref{rem:phi0}.
\end{customass}
This implies that $V$ satisfies Assumption~\ref{ass:n} and by Theorem~\ref{thm:orb} the \fpvosa{} $V^G$ does as well. Moreover, every irreducible $V^G$-module appears as a $V^G$-submodule of one of the unique irreducible $\sigma^i$-twisted $V$-modules $V(\sigma^i)$, $i\in\Z_n$.

Together with Theorem~\ref{thm:irrdec} from the previous section we are able to classify all irreducible $V^G$-modules:
\begin{thm}[Classification of Irreducible Modules]\label{thm:irrclass}
Let $V$ and $G=\langle\sigma\rangle$ be as in Assumption~\ref{ass:o}. Then, every $V^G$-module is completely reducible and, up to isomorphism, there are exactly $n^2$ irreducible $V^G$-modules, namely $W^{(i,j)}$, $i,j\in\Z_n$. $W^{(i,j)}$ is the eigenspace in $V(\sigma^i)$ of $\phi_{\sigma^i}(\sigma)$ corresponding to the eigenvalue $\xi_n^j$.
\end{thm}

We will also need later:
\begin{customass}{{\textbf{\textsf{OP}}}}\label{ass:op}
As Assumption~\ref{ass:o} but also assume that $V^G$ satisfies Assumption~\ref{ass:p} or equivalently $\rho(V(\sigma^i))>0$ for all $i\in\Z_n\setminus\{0\}$.
\end{customass}

For convenience, given Assumption~\ref{ass:o}, let us in the following also write
\begin{equation*}
\phi_i:=\phi_{\sigma^i}
\end{equation*}
for $i\in\Z_n$.

\section{Duality}\label{sec:duality}

In the situation of Theorem~\ref{thm:irrclass} we study the duality relations amongst the $n^2$ irreducible $V^G$-modules. Propositions \ref{prop:conirr} and \ref{prop:lem3.7} suggest that the contragredient module of $W^{(i,j)}$ is isomorphic to $W^{(-i,j')}$ for some $j'$. In fact:
\begin{prop}\label{prop:dual}
Let $V$ and $G=\langle\sigma\rangle$ be as in Assumption~\ref{ass:o}. Then the contragredient module $(W^{(i,j)})'$ of the module $W^{(i,j)}$ is isomorphic to $W^{(-i,\alpha(i)-j)}$, i.e.\ the indices are shifted as $(i,j)\mapsto(i,j)':=(-i,\alpha(i)-j)$ under the contragredient map, for some function $\alpha\colon\Z_n\to\Z_n$, depending on the choice of the $\phi_i(\sigma)$, with $\alpha(i)=\alpha(-i)$, $i\in\Z_n$.

With the choice of $\phi_0(\sigma):=\sigma$ as in Remark~\ref{rem:phi0}, $\alpha(0)=0$, i.e.\ $(W^{(0,j)})'\cong W^{(0,-j)}$ for all $j\in\Z_n$.
\end{prop}
\begin{proof}
Given an irreducible twisted or untwisted $V$-module $W$, the contragredient module of $W=\bigoplus_{n\in\C}W_n$ is defined on the graded dual space $W':=\bigoplus_{\lambda\in\C}W^*_\lambda$, recalling that all the $W_\lambda$ are finite-dimensional. Let $\langle\cdot,\cdot\rangle$ denote the canonical bilinear pairing $W'\times W\to\C$. The module vertex operator $Y_{W'}(\cdot,x)$ on the graded dual is defined by the relation
\begin{equation*}
\langle Y_{W'}(v,x)w',w\rangle=\langle w',Y_W(\e^{xL_1}(-x^2)^{L_0}v,x^{-1})w\rangle
\end{equation*}
for all $v\in V$, $w\in W$ and $w'\in W'$ where $Y_W(\cdot,x)$ denotes the module vertex operator on $W$.

Let $\phi_W$ denote the representation of $G=\langle\sigma\rangle$ on $W$ and $\phi_W^*\colon W'\to W'$ the (graded) dual operator of $\phi_W$. Recall that $\phi_W(\sigma)$ commutes with the modes $L_n$, $n\in\Z$, of the Virasoro vector $\omega$. We obtain:
\begin{align*}
\langle Y_{W'}(\sigma v,x)w',w\rangle&=\langle w',Y_W(\e^{xL_1}(-x^2)^{L_0}\sigma v,x^{-1})w\rangle=\langle w',Y_W(\sigma\e^{xL_1}(-x^2)^{L_0}v,x^{-1})w\rangle\\
&=\langle w',\phi_W(\sigma)Y_W(\e^{xL_1}(-x^2)^{L_0}v,x^{-1})\phi_W(\sigma)^{-1}w\rangle\\
&=\langle\phi_W^*(\sigma) w',Y_W(\e^{xL_1}(-x^2)^{L_0}v,x^{-1})\phi_W(\sigma)^{-1}w\rangle\\
&=\langle Y_{W'}(v,x)\phi_W^*(\sigma) w',\phi_W(\sigma)^{-1}w\rangle=\langle{\phi_W^*}(\sigma)^{-1}Y_{W'}(v,x)\phi_W^*(\sigma) w',w\rangle
\end{align*}
for all $v\in V$, $w\in W$ and $w'\in W'$, i.e.\
\begin{equation*}
Y_{W'}(\sigma v,x)={\phi_W^*}(\sigma)^{-1}Y_{W'}(v,x)\phi_W^*(\sigma)
\end{equation*}
for all $v\in V$. Hence ${\phi_W^*}(\sigma)^{-1}$ is proportional to $\phi_{W'}(\sigma)$, the representation of $G$ on $W'$.

Let $W=V(\sigma^i)$ for some $i\in\Z_n$. Then, by Proposition~\ref{prop:lem3.7} the contragredient module is a $\sigma^{-i}$-twisted module, which is irreducible by Proposition~\ref{prop:conirr} and thus isomorphic to $V(\sigma^{-i})$. Hence the contragredient module of $W^{(i,j)}$ is isomorphic to $W^{(-i,j')}$ for some $j'\in\Z_n$, using that the contragredient module of an irreducible module is again irreducible (see Proposition~\ref{prop:conirr}).

The subspace $W^{(i,j)}$ of $V(\sigma^i)$ is the eigenspace of $\phi_i(\sigma)$ corresponding to the eigenvalue $\xi_n^j$. Then the contragredient module $W^{(-i,j')}$ has the same eigenvalue $\xi_n^j$ with respect to the dual operator $\phi_i^*(\sigma)$ and hence eigenvalue $\xi_n^{-j}$ with respect to ${\phi_i^*}(\sigma)^{-1}$. Let $\xi_n^{\alpha(i)}\in U_n$ denote the constant of proportionality between $\phi_{-i}(\sigma)$ and ${\phi_i^*}(\sigma)^{-1}$, i.e.\
\begin{equation*}
\phi_{-i}(\sigma)=\xi_n^{\alpha(i)}\phi_i^*(\sigma)^{-1}
\end{equation*}
for some $\alpha(i)\in\Z_n$. Then $W^{(-i,j')}$ has eigenvalue $\xi_n^{\alpha(i)-j}$ with respect to $\phi_{-i}(\sigma)$, i.e.\ $j'=\alpha(i)-j$ or in other words
\begin{equation*}
{W^{(i,j)}}'\cong W^{(i,\alpha(i)-j)}.
\end{equation*}
Finally, $V(\sigma^i)''\cong V(\sigma^i)$ implies $j=\alpha(-i)-(\alpha(i)-j)$ and hence $\alpha(i)=\alpha(-i)$ for all $i\in\Z_n$. This is the assertion.

With $V=V(\sigma^0)$ and $\phi_0(\sigma):=\sigma$ the \fpvosa{} is simply $V^G=W^{(0,0)}$ and Corollary~\ref{cor:fixsd} implies $(W^{(0,0)})'\cong W^{(0,0)}$ so that $\alpha(0)=0$.
\end{proof}

\section{Fusion Algebra I}\label{sec:fusalg}

In the following we determine the fusion algebra $\V(V^G)$ of $V^G$ for a holomorphic \voa{} $V$ and a finite, cyclic group $G=\langle\sigma\rangle\cong\Z_n$ of automorphisms of $V$. It turns out that all $V^G$-modules are simple currents and that the fusion algebra is the group algebra of a central extension of $\Z_n$ by $\Z_n$.

The following steps---though certainly much more general---are inspired by the proof of Lemma~11 in \cite{Miy13b} (or Lemma~18 in \cite{Miy10}).

\minisec{Step: Twisted Modular Invariance}

By the Verlinde formula the $S$-matrix of $V^G$ determines the fusion algebra. The $S$-matrix of $V^G$, which by definition arises from Zhu's modular invariance of trace functions for the $V^G$-modules, can be related to Dong, Li and Mason's twisted modular invariance for $V$.\begin{lem}\label{lem:step1}
Let $V$ and $G=\langle\sigma\rangle$ be as in Assumption~\ref{ass:o}. Then the $S$-matrix of $V^G$ is of the form
\begin{equation*}
\S_{(i,j),(l,k)}=\frac{1}{n}\xi_n^{-(lj+ik)}\lambda_{i,l},
\end{equation*}
$i,j,k,l\in\Z_n$ for some constants $\lambda_{i,l}\in\C$ (see \eqref{eq:lambda}).
\end{lem}
\begin{proof}
We calculate the twisted trace functions $T(v,i,j,\tau)$ on the twisted modules $V(\sigma^i)$, $i,j\in\Z_n$, in terms of the (ordinary) trace functions on the irreducible $V^G$-modules. Using the definition of the $W^{(i,j)}$ as eigenspaces of $\phi_i(\sigma)$ corresponding to the eigenvalue $\xi_n^j$ we get
\begin{align*}
T(v,i,j,\tau)&=\tr_{V}o(v)\phi_i(\sigma^j)q_\tau^{L_0-c/24}=\sum_{k\in\Z_n}\tr_{W^{(i,k)}}o(v)\xi_n^{jk}q_\tau^{L_0-c/24}\\
&=\sum_{k\in\Z_n}\xi_n^{jk}T_{W^{(i,k)}}(v,\tau)
\end{align*}
for $i,j\in\Z_n$. Equation~\eqref{eq:lambda} gives the $S$-transformations of the above twisted trace functions
\begin{equation*}
(1/\tau)^k T(v,i,j,-1/\tau)=\lambda_{i,j}T(v,j,-i,\tau)
\end{equation*}
for $v\in V_{[k]}$. We then write the (untwisted) trace functions on the modules $W^{(i,j)}$ as appropriate linear combinations of the twisted trace functions
\begin{equation*}
T_{W^{(i,j)}}(v,\tau)=\frac{1}{n}\sum_{l\in\Z_n}\xi_n^{-lj}T(v,i,l,\tau)
\end{equation*}
to get the $S$-transformation of the untwisted trace functions
\begin{align*}
(1/\tau)^k T_{W^{(i,j)}}(v,-1/\tau)&=\frac{1}{n}\sum_{l\in\Z_n}\xi_n^{-lj}(1/\tau)^kT(v,i,l,-1/\tau)\\
&=\frac{1}{n}\sum_{l\in\Z_n}\xi_n^{-lj}\lambda_{i,l}T(v,l,-i,\tau)\\
&=\frac{1}{n}\sum_{l,k\in\Z_n}\xi_n^{-(lj+ik)}\lambda_{i,l}T_{W^{(l,k)}}(v,\tau).
\end{align*}
From this we can read off the form of the $S$-matrix of $V^G$
\begin{equation*}
\S_{(i,j),(l,k)}=\frac{1}{n}\xi_n^{-(lj+ik)}\lambda_{i,l}
\end{equation*}
for $i,j,k,l\in\Z_n$.
\end{proof}

\minisec{Step: Easy Formulæ}
We deduce some easy properties of the $\lambda_{i,j}$.
\begin{lem}\label{lem:step2}
Let $V$ and $G=\langle\sigma\rangle$ be as in Assumption~\ref{ass:o}. Then the $\lambda_{i,j}$, $i,j\in\Z_n$ fulfil
\begin{align*}
\tag{a}\label{eq:a}
\lambda_{i,j}&=\lambda_{j,i},\\
\tag{b}\label{eq:b}
\lambda_{i,j}\lambda_{i,-j}&=\xi_n^{i\alpha(j)},\\
\tag{bb}\label{eq:bb}
\lambda_{i,j}&=\xi_n^{j\alpha(i)+i\alpha(j)}\lambda_{-i,-j},\\
\tag{bbb}\label{eq:bbb}
\lambda_{0,i}&=1
\end{align*}
for all $i,j\in\Z_n$.
\end{lem}
\begin{proof}
By Theorem~\ref{thm:verlinde}, the $S$-matrix is symmetric, i.e.\ $\S_{(i,j),(l,k)}=\S_{(l,k),(i,j)}$, and one immediately obtains \eqref{eq:a}.

Using that $\S^2$ is a permutation matrix, sending the index $(i,j)$ to the index of the contragredient module $(i,j)'=(-i,\alpha(i)-j)$ by Proposition~\ref{prop:dual}, we obtain that $\S^2$ is of the form
\begin{equation*}
(\S^2)_{(i,j),(l,k)}=\delta_{(i,j)',(l,k)}=\delta_{-i,l}\delta_{\alpha(i)-j,k},
\end{equation*}
which we conveniently write as
\begin{equation*}
(\S^2)_{(i,j),(l,k)}=\delta_{i,-l}\frac{1}{n}\sum_{a\in\Z_n}\xi_n^{-a(j+k-\alpha(i))}.
\end{equation*}
On the other hand we calculate
\begin{align*}
(\S^2)_{(i,j),(l,k)}&=\sum_{a,b\in\Z_n}\S_{(i,j),(a,b)}\S_{(a,b),(l,k)}
=\frac{1}{n^2}\sum_{a,b\in\Z_n}\xi_n^{-(aj+ib+lb+ak)}\lambda_{i,a}\lambda_{a,l}\\
&=\frac{1}{n^2}\sum_{a\in\Z_n}\xi_n^{-(aj+ak)}\lambda_{i,a}\lambda_{a,l}\sum_{b\in\Z_n}\xi_n^{-(ib+lb)}
=\delta_{i,-l}\frac{1}{n}\sum_{a\in\Z_n}\xi_n^{-a(j+k)}\lambda_{i,a}\lambda_{a,-i}
\end{align*}
using the above determined form of the $S$-matrix.

Let $l=-i$. Both expressions for $\S^2$ have to be identical, i.e.\
\begin{equation*}
\sum_{a\in\Z_n}\xi_n^{-a(j+k)}(\xi_n^{a\alpha(i)}-\lambda_{i,a}\lambda_{a,-i})=0
\end{equation*}
for all $i,j,k\in\Z_n$. Summing over $b:=j+k$ and inserting a factor $\xi_n^{\tilde{a}b}$ for some $\tilde{a}\in\Z_n$ implies
\begin{align*}
0&=\frac{1}{n}\sum_{b\in\Z_n}\xi_n^{\tilde{a}b}\sum_{a\in\Z_n}\xi_n^{-ab}(\xi_n^{a\alpha(i)}-\lambda_{i,a}\lambda_{a,-i})=\sum_{a\in\Z_n}\frac{1}{n}\sum_{b\in\Z_n}\xi_n^{(\tilde{a}-a)b}(\xi_n^{a\alpha(i)}-\lambda_{i,a}\lambda_{a,-i})\\
&=\sum_{a\in\Z_n}\delta_{a,\tilde{a}}(\xi_n^{a\alpha(i)}-\lambda_{i,a}\lambda_{a,-i})=\xi_n^{\tilde{a}\alpha(i)}-\lambda_{i,\tilde{a}}\lambda_{\tilde{a},-i},
\end{align*}
which means $\lambda_{i,a}\lambda_{a,-i}=\xi_n^{a\alpha(i)}$ for all $a,i\in\Z_n$. Using the symmetry relation \eqref{eq:a} we obtain \eqref{eq:b}, which also implies \eqref{eq:bb} and
\begin{equation*}
\lambda_{0,i}^2=\xi_n^{i\alpha(0)}=1
\end{equation*}
for all $i\in\Z_n$ since $\alpha(0)=0$ by Proposition~\ref{prop:dual3}. Lemma~\ref{lem:lambdapos}, which states that $\lambda_{0,i}>0$ for $i\in\Z_n$, implies \eqref{eq:bbb}.
\end{proof}

\minisec{Step: Simple Currents}
Using the shape of the $S$-matrix obtained in Lemma~\ref{lem:step1} and the above formulæ we can show that the fusion of $V^G$ is in fact group-like, i.e.\ that all the irreducible $V^G$-modules are simple currents.
\begin{lem}\label{lem:step4}
Let $V$ and $G=\langle\sigma\rangle$ be as in Assumption~\ref{ass:o}. Then all the irreducible $V^G$-modules $W^{(i,j)}$, $i,j\in\Z_n$, are simple currents.
\end{lem}
\begin{proof}
First, consider the product $\S_{(i,j),(l,k)}\S_{(i,j)',(l,k)}$ where $(i,j)'$ denotes the index of the contragredient module of $W^{(i,j)}$:
\begin{align*}
\S_{(i,j),(l,k)}\S_{(i,j)',(l,k)}&=\S_{(i,j),(l,k)}\S_{(-i,\alpha(i)-j),(l,k)}\\
&=\frac{1}{n}\xi_n^{-(lj+ik)}\lambda_{i,l}\frac{1}{n}\xi_n^{-(l(\alpha(i)-j)+(-i)k)}\lambda_{-i,l}=\frac{1}{n^2}\xi_n^{-l\alpha(i)}\lambda_{i,l}\lambda_{-i,l}=\frac{1}{n^2}
\end{align*}
for all $i,j,k,l\in\Z_n$ by equation \eqref{eq:b}. Then consider the Verlinde formula (see Theorem~\ref{thm:verlinde}) to calculate
\begin{equation*}
N_{(i,j),(i,j)'}^{(l,k)}=\sum_{a,b\in\Z_n}\frac{\S_{(i,j),(a,b)}\S_{(i,j)',(a,b)}\S_{(a,b),(l,k)'}}{\S_{(0,0),(a,b)}}=\frac{1}{n^2}\sum_{a,b\in\Z_n}\frac{\S_{(a,b),(l,k)'}}{\S_{(0,0),(a,b)}}.
\end{equation*}
This yields
\begin{align*}
N_{(i,j),(i,j)'}^{(l,k)}&=\frac{1}{n^2}\sum_{a,b\in\Z_n}\frac{\S_{(a,b),(-l,\alpha(l)-k)}}{\S_{(0,0),(a,b)}}=\frac{1}{n^2}\sum_{a,b\in\Z_n}\frac{\frac{1}{n}\xi_n^{-((-l)b+a(\alpha(l)-k))}\lambda_{a,-l}}{\frac{1}{n}\lambda_{0,a}}\\
&=\frac{1}{n^2}\sum_{a\in\Z_n}\xi_n^{a(k-\alpha(l))}\lambda_{a,-l}\sum_{b\in\Z_n}\xi_n^{lb}=\delta_{l,0}\frac{1}{n}\sum_{a\in\Z_n}\xi_n^{a(k-\alpha(l))}\lambda_{a,0}\\
&=\delta_{l,0}\frac{1}{n}\sum_{a\in\Z_n}\xi_n^{a(k-\alpha(l))}=\delta_{l,0}\delta_{k,\alpha(l)}=\delta_{l,0}\delta_{k,\alpha(0)}=\delta_{l,0}\delta_{k,0},
\end{align*}
which means that
\begin{equation*}
W^{(i,j)}\boxtimes W^{(i,j)'}\cong W^{(0,0)}
\end{equation*}
for all $i,j,\in\Z_n$. Finally, with Proposition~\ref{prop:scvoa} we conclude that all the irreducible $V^G$-modules $W^{(i,j)}$, $i,j\in\Z_n$, are simple currents.
\end{proof}

\minisec{Step: Fusion Group}

Knowing that all irreducible $V^G$-modules are simple currents, i.e.\ $V^G$ fulfils Assumption~\ref{ass:sn}, we can apply the results from Section~\ref{sec:scvoa} on simple-current \voa{}s.

\begin{rem}
A number of results in Sections~\ref{sec:simplecurrents} and \ref{sec:scvoa} depend on the positivity assumption (Assumption~\ref{ass:p}). This assumption only enters through Propositions \ref{prop:pospos} and \ref{prop:4.17}. Note that in the orbifold situation of this chapter we do not know if Assumption~\ref{ass:p} holds but we do know that the statements of Propositions \ref{prop:pospos} and \ref{prop:4.17} are true by direct computations. This follows from Lemma~\ref{lem:lambdapos} and the above formula for $\S$ from Lemma~\ref{lem:step1}. This entails that all results in Sections~\ref{sec:simplecurrents} and \ref{sec:scvoa} that depend on Assumption~\ref{ass:p} hold for $V^G$ in the orbifold setting of this chapter, i.e.\ under Assumption~\ref{ass:o} (without Assumption~\ref{ass:p} for $V^G$).
\end{rem}

We conclude that the fusion algebra $\V(V^G)$ of $V^G$ is the group algebra $\C[F_{V^G}]$ of some fusion group $F_{V^G}$ of order $n^2$. As a set $F_{V^G}=\{(i,j)\;|\;i,j\in\Z_n\}=\Z_n\times\Z_n$. Moreover, the conformal weights $\rho(W^{(i,j)})$ modulo~1 of the irreducible $V^G$-modules determine a non-degenerate quadratic form $Q_\rho$ on $F_{V^G}$. In the following we will see that $F_{V^G}$ is in fact a central extension of $\Z_n$ by $\Z_n$.
\begin{lem}\label{lem:step5}
Let $V$ and $G=\langle\sigma\rangle$ be as in Assumption~\ref{ass:o}. Then
\begin{equation*}
W^{(i,j)}\boxtimes W^{(l,k)}\cong W^{(i+l,j+k+c(i,l))}
\end{equation*}
for some symmetric, normalised 2-cocycle $c\colon\Z_n\times\Z_n\to\Z_n$. This 2-cocycle fulfils
\begin{equation}\label{eq:bullet}
\xi_n^{-ac(i,l)}=\frac{\lambda_{i,a}\lambda_{l,a}}{\lambda_{i+l,a}}
\end{equation}
for all $i,l,a\in\Z_n$.
\end{lem}
\begin{proof}
Let $W^{(i,j)}\boxtimes W^{(l,k)}\cong: W^{(p,q)}$ for some $p,q\in\Z_n$. Then $\S_{(i,j),(a,b)}\S_{(l,k),(a,b)}=(1/n)\S_{(p,q),(a,b)}$ by the results in Section~\ref{sec:scvoa}. This implies
\begin{equation*}
\lambda_{i,a}\lambda_{l,a}/\lambda_{p,a}=\xi_n^{-(pb+qa)}\xi_n^{ib+ja}\xi_n^{lb+ka}
\end{equation*}
for all $a,b\in\Z_n$. With $a=0$ we see that $p=i+l\in\Z_n$ using \eqref{eq:bbb}. Then with $b=0$ we get
\begin{equation*}
\lambda_{i,a}\lambda_{l,a}/\lambda_{i+l,a}=\xi_n^{(j+k-q)a}.
\end{equation*}
This shows that $c:=q-j-k$ depends only on $i$ and $l$ and we get $q=j+k+c(i,l)$.

The associativity of the fusion algebra $\V(V^G)$ implies that $\xi_n^{-c(i,l)}\colon\Z_n\times \Z_n\to U_n$ is a 2-cocycle. This 2-cocycle is symmetric since $\V(V^G)$ is commutative and normalised since $\lambda_{0,1}=1$. Similarly, $c$ is a symmetric, normalised 2-cocycle $c\colon\Z_n\times\Z_n\to\Z_n$.
\end{proof}

The lemma implies that the fusion group $F_{V^G}=\Z_n\times\Z_n$ (as a set) obeys the group law
\begin{equation*}
(i,j)\oplus(l,k)=(i+l,j+k+c(i,l))
\end{equation*}
for $i,j,k,l\in\Z_n$. The associativity of the group law corresponds to $c$ being a 2-cocycle and the abelianness of $F_{V^G}$ to the symmetry of $c$.

The above equation shows that the group $F_{V^G}$ is a central extension of the group $B=\Z_n$ by the group $A=\Z_n$, i.e.\ there is a short exact sequence
\begin{equation*}
1\to A\to F_{V^G}\to B\to 1
\end{equation*}
such that $A$ is in $Z(F_{V^G})$, the centre of $F_{V^G}$. Indeed, with the above definition of $F_{V^G}$ via the 2-cocycle $c\colon B\times B\to A$, i.e.\ $c\in Z^2(B,A)$, it is well known that there is a bijection between the isomorphism classes of central extensions $1\to A\to F_{V^G}\to B\to 1$ and the cohomology group $H^2(B,A)=Z^2(B,A)/B^2(B,A)$ (see e.g.\ \cite{Wei94}). All such central extensions are abelian since we can always find a symmetric representative in $Z^2(\Z_n,\Z_n)$ of an element in $H^2(\Z_n,\Z_n)$.

We have just proved:
\begin{cor}\label{cor:step5}
Let $V$ and $G=\langle\sigma\rangle$ be as in Assumption~\ref{ass:o}. Then the fusion group $F_{V^G}$ is a central extension $1\to \Z_n\to F_{V^G}\to \Z_n\to 1$.
\end{cor}

The group structure of the fusion group $F_{V^G}$ is determined up to isomorphism by the cohomology class of the 2-cocycle $c\colon\Z_n\times\Z_n\to\Z_n$ from Lemma~\ref{lem:step5}.\footnote{But two non-cohomologous 2-cocycles can give isomorphic groups.} It is well known that the cohomology group $H^2(\Z_n,\Z_n)$ is isomorphic to $\Z_n$, where the element in $H^2(\Z_n,\Z_n)$ corresponding to $d\in\Z_n$ can be represented by the symmetric, normalised 2-cocycle $c_d\colon\Z_n\times\Z_n\to \Z_n$ given by
\begin{equation*}
c_d(i,j):=d\left\lfloor\frac{i_n+j_n}{n}\right\rfloor=d\frac{i_n+j_n-(i+j)_n}{n}=\begin{cases}0&\text{if }i_n+j_n<n,\\d&\text{if }i_n+j_n\geq n.\end{cases}
\end{equation*}
Here, we introduced the notation that for $i\in\Z_n$, $i_n$ denotes the representative of $i$ in $\{0,\ldots,n-1\}$. 

In the following let $d\in\Z_n$ denote the cohomology class of the cocycle $c$ from Lemma~\ref{lem:step5}. Then, since $c$ is normalised, $d$ can be determined via
\begin{equation*}
n\cdot(1,j)=\underbrace{(1,j)\oplus\ldots\oplus(1,j)}_{n\text{ times}}=(0,d),
\end{equation*}
i.e.\
\begin{equation*}
d=c(1,1)+c(1,2)+\ldots+c(1,n-1).
\end{equation*}
In the next section we will show that the value of $d\in\Z_n$ is determined by the conformal weight $\rho(V(\sigma))$ of the irreducible $\sigma$-twisted $V$-module $V(\sigma)$.

From the formula for the contragredient module (see Proposition~\ref{prop:dual}) it follows:
\begin{lem}\label{lem:calpha}
Let $V$ and $G=\langle\sigma\rangle$ be as in Assumption~\ref{ass:o}. The 2-cocycle $c\colon\Z_n\times\Z_n\to\Z_n$ from Lemma~\ref{lem:step5} and the function $\alpha\colon\Z_n\to\Z_n$ from Proposition~\ref{prop:dual} are related via
\begin{equation*}
\alpha(i)=-c(i,-i)
\end{equation*}
for all $i\in\Z_n$.
\end{lem}
\begin{proof}
Consider $(0,0)=(i,j)\oplus{(i,j)}'=(i,j)\oplus(-i,\alpha(i)-j)=(0,\alpha(i)+c(i,-i))$.
\end{proof}

\section{Conformal Weights}\label{sec:confweights}
It is possible to relate the representations $\phi_i=\phi_{\sigma^i}$ from Section~\ref{sec:fpvosa} to $L_0=L_0^{V(\sigma^i)}$, the weight-grading operator on $V(\sigma^i)$, and hence to the conformal weights
\begin{equation*}
\rho_i:=\rho(V(\sigma^i))\in\Q
\end{equation*}
of the irreducible twisted $V$-modules $V(\sigma^i)$, $i\in\Z_n$.

\begin{lem}\label{lem:phi2}
Let $V$ and $G=\langle\sigma\rangle$ be as in Assumption~\ref{ass:o} (minus CFT-type).
Then it is possible to choose the representations $\phi_i$ of $G$ on $V(\sigma^i)$, $i\in\Z_n$, such that
\begin{equation*}
\phi_i(\sigma)^{(i,n)}=\e^{(2\pi\i)[i/(i,n)]^{-1}(L_0-\rho_i)}
\end{equation*}
where $[i/(i,n)]^{-1}$ is the inverse of $i/(i,n)$ modulo~$n/(i,n)$.
\end{lem}
\begin{proof}
First, consider the coprime case with $(i,n)=1$. Then $i$ is invertible modulo~$n$ and let $i^{-1}$ denote the inverse of $i$ modulo~$n$. We will show that the representation $\phi_i$ on $V(\sigma^i)$ can be chosen to be
\begin{equation*}
\phi_i(\sigma^j)=\e^{(2\pi\i)i^{-1}j(L_0-\rho_i)}
\end{equation*}
where $L_0$ is the grading operator on $V(\sigma^i)$.

To this end, consider the module $W=V(\sigma^i)$ for some fixed $i\in\Z_n^\times$ so that $\sigma^i$ has order~$n$. $V(\sigma^i)$ is irreducible and hence has grading $W=\bigoplus_{k=0}^\infty W_{\rho_i+k/n}$ for some conformal weight $\rho_i\in\Q$ by Definition~\ref{defi:sec3} (and Theorem~\ref{thm:10.3}). We define the weight subspaces
\begin{equation*}
W_{(k)}:=\bigoplus_{l\in\Z} W_{\rho_i+k/n+l}
\end{equation*}
for $k\in\Z_n$ so that $W=\bigoplus_{k\in\Z_n}W_{(k)}$. 

We will show that $v_t(W_{(k)})\subseteq W_{(k+ir)}$ for any $t\in\Q$ if $v\in V^r$, i.e.\ $v$ has eigenvalue $\xi_n^r$ with respect to $\sigma$. Here, $v_t\in\End_\C(W)$ denotes the $t$-th mode of $Y_W(v,x)$. Indeed, let $v\in V^r$. Then $v$ has eigenvalue $\xi_n^{ri}$ with respect to $\sigma^i$. By the definition of $\sigma^i$-twisted modules, for $v_t$ to be non-vanishing, $t$ needs to lie in $-k/n+\Z$ if $v$ has eigenvalue $\xi_n^k$ with respect to $\sigma^i$ and hence $t\in-ri/n+\Z$ in our case, or else $v_t=0$.
For homogeneous $v$, the weight of the operator $v_t$ is in general given by $\wt(v_t)=\wt(v)-t-1$ and hence $\wt(v_t)\in ri/n+\Z$. Consequently, $v_t$ changes the weight of an element in $W$ by $ri/n$ modulo~1, which proves the statement.

Using the above assertion we can now prove the statement of the lemma in the coprime case. First, we show that $\phi_i(\sigma^j):=\e^{(2\pi\i)i^{-1}jL_0}$ fulfils the equation
\begin{equation*}
\phi_i(\sigma^j)Y_{V(g)}(v,x)\phi_i(\sigma^j)^{-1}w=Y_{V(g)}(\sigma^jv,x)w
\end{equation*}
for all $v\in V$ and all $w\in W=V(\sigma^i)$, which is equivalent to
\begin{equation*}
\phi_i(\sigma^j)v_t\phi_i(\sigma^j)^{-1}w=(\sigma^jv)_tw
\end{equation*}
for all $v\in V$, $w\in W$ and $t\in\Q$. Without loss of generality, let $v\in V^r$, i.e.\ $v$ has eigenvalue $\e^{(2\pi\i)jr/n}$ with respect to $\sigma^j$, and let $t\in-ri/n+\Z$. Then the right-hand side becomes
\begin{equation*}
(\sigma^jv)_tw=\e^{(2\pi\i)jr/n}v_tw.
\end{equation*}
On the other hand, the left-hand side equates to
\begin{align*}
\phi_i(\sigma^j)v_t\phi_i(\sigma^j)^{-1}w&=\e^{(2\pi\i)i^{-1}jL_0}v_t\e^{-(2\pi\i)i^{-1}jL_0}w=\e^{(2\pi\i)i^{-1}jL_0}v_t\e^{-(2\pi\i)i^{-1}j\wt(w)}w\\
&=\e^{(2\pi\i)i^{-1}jL_0}v_tw\e^{-(2\pi\i)i^{-1}j\wt(w)}\\
&=\e^{(2\pi\i)i^{-1}j(\wt(w)+ri/n)}v_tw\e^{-(2\pi\i)i^{-1}j\wt(w)}\\
&=\e^{(2\pi\i)i^{-1}j(ri/n)}v_tw=\e^{(2\pi\i)jr/n}v_tw.
\end{align*}
Hence, the left-hand and the right-hand side coincide, which means we have determined $\phi(\sigma^j)$ up to a scalar factor because of the uniqueness.

In order for $\phi_i(\sigma)$ to be of order $n$ (it has to be a representation of $G$), its eigenvalues have to be $n$-th roots of unity. This can be achieved by redefining $\phi_i(\sigma)$ by multiplying it by an appropriate scalar factor. Since $(L_0-\rho_i)$ has eigenvalues in $(1/n)\Z$, we know that
\begin{equation*}
\phi_i(\sigma^j)=\e^{(2\pi\i)i^{-1}j(L_0-\rho_i)}
\end{equation*}
is a suitable choice, which completes the proof in the coprime case.

Now, consider any $i\in\Z_n$, not necessarily coprime to $n$. We will reduce this case to the coprime case. Let us define
\begin{equation*}
\tilde{i}:=\frac{i}{(i,n)}\quad\text{and}\quad\tilde{n}:=\frac{n}{(i,n)}
\end{equation*}
so that $(\tilde{i},\tilde{n})=1$ as well as
\begin{equation*}
\tilde{\sigma}:=\sigma^{(i,n)},
\end{equation*}
which is an automorphism of order $\tilde{n}$. We consider the representation $\phi_i$ of $G=\langle\sigma\rangle$ on the $i$-twisted module
\begin{equation*}
V(\sigma^i)=V((\sigma^{(i,n)})^{i/(i,n)})=V(\tilde{\sigma}^{\tilde{i}}).
\end{equation*}
We saw that $\phi_i(\sigma)$ is unique up to an $n$-th root of unity. We then consider
\begin{equation*}
\phi_i(\sigma)^{(i,n)}=\phi_i(\sigma^{(i,n)})=\phi_i(\tilde{\sigma}),
\end{equation*}
which by the coprime case applied to the $\tilde{\sigma}^{\tilde{i}}$-twisted module $V(\sigma^i)$, where $\tilde{\sigma}$ is an automorphism of order $\tilde{n}$, and $(\tilde{i},\tilde{n})=1$ has to be of the form
\begin{equation*}
\e^{(2\pi\i)\tilde{i}^{-1}(L_0-\rho_i)}
\end{equation*}
modulo an $\tilde{n}$-th root of unity where $\tilde{i}^{-1}$ is the inverse of $\tilde{i}$ modulo~$\tilde{n}$. This means that
\begin{equation*}
\phi_i(\sigma)^{(i,n)}=\e^{(2\pi\i)\tilde{i}^{-1}(L_0-\rho_i)}\xi_{\tilde{n}}^k
\end{equation*}
for some $k\in\Z$. Here, $\xi_{\tilde{n}}=\e^{(2\pi\i)1/\tilde{n}}$ is a primitive $\tilde{n}$-th root of unity. We can then redefine $\phi_i(\sigma)$ as
\begin{equation*}
\phi_i(\sigma)\mapsto\phi_i(\sigma)\xi_n^{-k},
\end{equation*}
which yields
\begin{equation*}
\phi_i(\sigma^{(i,n)})=\e^{(2\pi\i)\tilde{i}^{-1}(L_0-\rho_i)}\xi_{\tilde{n}}^k\xi_n^{-(i,n)k}=\e^{(2\pi\i)\tilde{i}^{-1}(L_0-\rho_i)}.
\end{equation*}
Indeed, $\phi_i(\sigma^{(i,n)})$ is an $n/(i,n)$-th root of unity since $(L_0-\rho_i)$ has eigenvalues in $((i,n)/n)\Z$.
\end{proof}
For $i=0$ the statement of the lemma is trivial, i.e.\ choosing the representations $\phi_i$, $i\in\Z_n$, as in the lemma gives no restriction on $\phi_0$. Note that we choose $\phi_0$ as in Remark~\ref{rem:phi0}, anyway.

The choices in the above lemma determine the conformal weights modulo~1 of the irreducible $V^G$-modules $W^{(i,j)}$ in terms of the $\rho_i$:
\begin{prop}\label{prop:weight}
Let $V$ and $G=\langle\sigma\rangle$ be as in Assumption~\ref{ass:o} (minus CFT-type).
Moreover, choose the representations $\phi_i$ of $G$ on $V(\sigma^i)$ as in Lemma~\ref{lem:phi2}. Then the irreducible $V^G$-modules $W^{(i,j)}=\{w\in V(\sigma^i)\;|\;\phi_i(\sigma)v=\xi_n^jv\}$ have weight grading in
\begin{equation*}
\rho_i+ij/n+\Z.
\end{equation*}
\end{prop}
\begin{proof}
As in the previous proof we write $\tilde{i}=i/(i,n)$ and $\tilde{n}=n/(i,n)$. The module $W^{(i,j)}$ contains those vectors in $V(\sigma^i)$ with eigenvalue $\xi_n^j$ with respect to $\phi_i(\sigma)$. These vectors all have eigenvalue $\xi_n^{j(i,n)}=\xi_{\tilde{n}}^j$ with respect to $\phi_i(\sigma)^{(i,n)}$. On the other hand, we choose $\phi_i$ such that
\begin{equation*}
\phi_i(\sigma)^{(i,n)}=\e^{(2\pi\i)\tilde{i}^{-1}(L_0-\rho_i)}.
\end{equation*}
Hence, for $w\in W^{(i,j)}$,
\begin{align*}
&\e^{(2\pi\i)\tilde{i}^{-1}(L_0-\rho_i)}w=\xi_{\tilde{n}}^jw=\e^{(2\pi\i)j/\tilde{n}}w\\
\iff&\tilde{i}^{-1}(L_0-\rho_i)w=j/\tilde{n}w\pmod{w}\\
\iff&\tilde{i}^{-1}(\wt(w)-\rho_i)=j/\tilde{n}\pmod{1}\\
\iff&\tilde{i}^{-1}\underbrace{\tilde{n}(\wt(w)-\rho_i)}_{\in\Z}=j\pmod{\tilde{n}}\\
\iff&\tilde{n}(\wt(w)-\rho_i)=\tilde{i}j\pmod{\tilde{n}}\\
\iff&\wt(w)-\rho_i=\tilde{i}j/\tilde{n}=ij/n\pmod{1},
\end{align*}
which shows that the weights of the elements in $W^{(i,j)}$ lie in $\rho_i+ij/n+\Z$.
\end{proof}

Recall that we showed that the contragredient module of $W^{(i,j)}$ is isomorphic to $W^{(-i,\alpha(i)-j)}$ for some function $\alpha\colon\Z_n\to\Z_n$. In fact:
\begin{prop}\label{prop:dual3}
Let $V$ and $G=\langle\sigma\rangle$ be as in Assumption~\ref{ass:o} and choose the representations $\phi_i$ of $G$ on the $V(\sigma^i)$ as in Lemma~\ref{lem:phi2}.
Then
\begin{equation*}
W^{(i,j)}\cong W^{(-i,\alpha(i)-j)}
\end{equation*}
for $i,j\in\Z_n$ with $\alpha\colon\Z_n\to\Z_n$ such that
\begin{equation*}
i\alpha(i)=0\quad\text{and}\quad\alpha(0)=0
\end{equation*}
for $i\in\Z_n$.
\end{prop}
\begin{proof}
We have seen in Proposition~\ref{prop:dual} that the contragredient module of $W^{(i,j)}$ is given by $W^{(-i,\alpha(i)-j)}$ for some function $\alpha\colon\Z_n\to\Z_n$ with $\alpha(i)=\alpha(-i)$, $i\in\Z_n$, and $\alpha(0)=0$. On the other hand, the conformal weight of a module and its contragredient module are the same. Hence $\rho_i+ij/n=\rho_{-i}+(-i)(\alpha(i)-j)/n\pmod{1}$ or equivalently $i\alpha(i)=0$, where we used that $\rho_i=\rho_{-i}$.
\end{proof}

\section{Fusion Algebra II}\label{sec:fusalg2}

In the following we will prove that the value of the conformal weight $\rho_1$ determines $d\in\Z_n$, which specifies the fusion group $F_{V^G}$ up to equivalence of central extensions (see Corollary~\ref{cor:step5}).

\minisec{Step: Quadratic Form}

If we choose the representations $\phi_i$ as in Lemma~\ref{lem:phi2}, then
\begin{equation*}
Q_\rho((i,j))=\rho(W^{(i,j)})+\Z=\rho_i+\frac{ij}{n}+\Z
\end{equation*}
is a finite quadratic form on $F_{V^G}$. A simple calculation gives the associated bilinear form
\begin{equation*}
B_\rho((i,j),(l,k))=\rho_{i+l}-\rho_i-\rho_l+\frac{ik+lj}{n}+\frac{(i+l)c(i,l)}{n}+\Z.
\end{equation*}
Using the formula for $\S$ in terms of the bilinear form (see Proposition~\ref{prop:sbilform2}) we can deduce:
\begin{lem}\label{lem:step3}
Let $V$ and $G=\langle\sigma\rangle$ be as in Assumption~\ref{ass:o} and choose the representations $\phi_i$ as in Lemma~\ref{lem:phi2}. Then the $\lambda_{i,j}$, $i,j\in\Z_n$, fulfil
\begin{align*}
\tag{c}\label{eq:c}
\lambda_{i,i+j}\lambda_{i+j,j}\e^{(2\pi\i)(\rho_{i}+\rho_{j}+\rho_{i+j})}&=\lambda_{i,j},\\
\tag{cc}\label{eq:cc}
\lambda_{i,i}&=\e^{(2\pi\i)(-2\rho_i)}
\end{align*}
for all $i,j\in\Z_n$.
\end{lem}
\begin{proof}
Consider
\begin{align*}
\lambda_{i,l}&=n\xi_n^{lj+ik}\S_{(i,j),(l,k)}=\xi_n^{lj+ik}\e^{-(2\pi\i)B_\rho((i,j),(l,k))}=\e^{(2\pi\i)(-\rho_{i+l}+\rho_i+\rho_l)}\xi_n^{-(i+l)c(i,l)}\\
&=\e^{(2\pi\i)(-\rho_{i+l}+\rho_i+\rho_l)}\lambda_{i,i+l}\lambda_{l,i+l}/\lambda_{i+l,i+l}
\end{align*}
by Proposition~\ref{prop:sbilform2} and \eqref{eq:bullet}. On the other hand
\begin{align*}
\lambda_{i,i}&=n\xi_n^{2ij}\S_{(i,j),(i,j)}=\xi_n^{2ij}\e^{-(2\pi\i)B_\rho((i,j),(i,j))}=\xi_n^{2ij}\e^{-(2\pi\i)2Q_\rho((i,j))}=\e^{(2\pi\i)(-2\rho_i)},
\end{align*}
which is \eqref{eq:cc} and together with the above equation gives \eqref{eq:c}.
\end{proof}

Finally we obtain:
\begin{lem}\label{lem:drho}
Let $V$ and $G=\langle\sigma\rangle$ be as in Assumption~\ref{ass:o}. Then
\begin{equation*}
d=2n^2\rho_1\pmod{n}.
\end{equation*}
\end{lem}
\begin{proof}
For the proof let us assume the representations $\phi_i$ are chosen as in Lemma~\ref{lem:phi2}. We use \eqref{eq:bullet} and consider the telescoping product
\begin{equation*}
\xi_n^d=\xi_n^{c(1,1)+c(1,2)+\ldots+c(1,n-1)}=\frac{\lambda_{2,1}}{\lambda_{1,1}\lambda_{1,1}}\frac{\lambda_{3,1}}{\lambda_{1,1}\lambda_{2,1}}\cdot\ldots\cdot\frac{\lambda_{n,1}}{\lambda_{1,1}\lambda_{n-1,1}}=\frac{\lambda_{n,1}}{\lambda_{1,1}^n}=\lambda_{1,1}^{-n}
\end{equation*}
since $\lambda_{n,1}=\lambda_{0,1}=1$ by \eqref{eq:bbb}. Using that $\lambda_{1,1}=\e^{(2\pi\i)(-2\rho_1)}$, \eqref{eq:cc} gives the desired result.
\end{proof}

\minisec{Step: Conformal Weights}

Since the order of any element in the group $F_{V^G}$ is trivially at most $n^2$, the bilinear form $B_\rho$ takes values in $(1/n^2)\Z$. In general this means that the quadratic form $Q_\rho$ takes values in $(1/(2n^2))\Z$. We can show however that in the orbifold situation, also the quadratic form only has values in $(1/n^2)\Z$.
\begin{oframed}
\begin{thm}\label{thm:confwnn}
Let $V$ and $G=\langle\sigma\rangle$ be as in Assumption~\ref{ass:o}. Then the unique irreducible $\sigma$-twisted $V$-module $V(\sigma)$ has conformal weight $\rho_1\in(1/n^2)\Z$. More generally, $V(\sigma^i)$ has conformal weight $\rho_i\in((i,n)^2/n^2)\Z$ for $i\in\Z_n$.
\end{thm}
\end{oframed}
The above theorem is a considerable improvement of Theorem~1.6 in \cite{DLM00} where only the special cases of $n=2,3$ are stated.
\begin{proof}
Again, let us for the proof assume that the representations $\phi_i$ are chosen as in Lemma~\ref{lem:phi2}. Consider the element $(i,j)\in F_{V^G}$. Then
\begin{equation*}
n/(i,n)\cdot(i,j)=\underbrace{(i,j)\oplus\ldots\oplus(i,j)}_{n/(i,n)\text{ times}}=(0,x)
\end{equation*}
for some $x\in\Z_n$ and hence
\begin{align*}
n^2/(i,n)^2\rho_i+\Z&=n^2/(i,n)^2(\rho_i+ij/n)+\Z=n^2/(i,n)^2 Q_\rho((i,j))=Q_\rho((0,x))\\
&=\rho_0+\Z=0+\Z,
\end{align*}
which proves that $\rho_i\in((i,n)^2/n^2)\Z$.
\end{proof}

In view of this theorem we make the following definition:
\begin{defi}[Type]
Let $V$ and $G=\langle\sigma\rangle$ be as in Assumption~\ref{ass:o}. We define the number $r\in\Z_n$ via
\begin{equation*}
\frac{r}{n^2}:=\rho_1\pmod{1/n}
\end{equation*}
and say that the automorphism $\sigma$ has \emph{type} $n\{r\}$.
\end{defi}

Lemma~\ref{lem:drho} immediately gives:
\begin{lem}\label{lem:d2r}
Let $V$ and $G=\langle\sigma\rangle$ be as in Assumption~\ref{ass:o} and $\sigma$ of type $n\{r\}$. Then
\begin{equation*}
d=2r\in\Z_n.
\end{equation*}
\end{lem}

We also define the level $N$ and the related number $h$:
\begin{defi}[Level]
Let $V$ and $G=\langle\sigma\rangle$ be as in Assumption~\ref{ass:o}. Let $N\in\Ns$ be the smallest multiple of $n$ such that $N\rho_1\in\Z$. Observe that $n\mid N\mid n^2$. Indeed, if $\sigma$ has type $n\{r\}$, then $N=n^2/(r,n)$. We call $N$ the \emph{level} of $\sigma$.

We also define $h:=N/n$. Then $h$ is some positive divisor of $n$ and $h$ is the smallest positive integer such that $\rho_1\in(1/nh)\Z$. Moreover, $h=n/(r,n)$ for an automorphism $\sigma$ of type $n\{r\}$.
\end{defi}
Proposition~\ref{prop:onetoall} below shows that $N=nh$ is the level of the \fqs{} $F_{V^G}=(F_{V^G},Q_\rho)$.

Lemma~\ref{lem:drho} shows that the value of $\rho_1$ determines the group structure of the fusion group $F_{V^G}$. In the same way, we will see that $\rho_1$ determines the quadratic form $Q_\rho$ on $F_{V^G}$ and hence also the values of all the other $\rho_i$, $i\in\Z_n$.
\begin{prop}\label{prop:onetoall}
Let $V$ and $G=\langle\sigma\rangle$ be as in Assumption~\ref{ass:o} and let $\sigma$ be of type $n\{r\}$, i.e.\ $\rho_1=r/n^2\pmod{1/n}$ for some $r\in\Z_n$. Then
\begin{equation*}
\rho_i=\frac{i^2r}{n^2}\pmod{(i,n)/n}
\end{equation*}
for all $i\in\Z_n$.
\end{prop}
Note that the expression $i^2r/n^2$ is well-defined modulo~$(i,n)/n$ for $i,r\in\Z_n$.
\begin{cor}\label{cor:onetoall}
Let $V$ and $G=\langle\sigma\rangle$ be as in Assumption~\ref{ass:o}. Then
\begin{equation*}
\rho_i\in\frac{(i,n)(i,h)}{nh}\Z
\end{equation*}
for all $i\in\Z_n$.
\end{cor}
\begin{proof}[Proof of Proposition~\ref{prop:onetoall}]
For the proof let us assume again that the representations $\phi_i$ are chosen as in Lemma~\ref{lem:phi2}. Using that $Q_\rho$ is a quadratic form we get
\begin{align*}
i^2\rho_1+\Z&=i^2Q_\rho((1,0))=Q_\rho(i\cdot(1,0))=Q_\rho((i,c(1,1)+\ldots+c(1,i-1)))\\
&=\rho_i+i(c(1,1)+\ldots+c(1,i-1))/n+\Z
\end{align*}
and hence
\begin{equation*}
i^2r/n^2=i^2\rho_1=\rho_i\pmod{(i,n)/n}.
\end{equation*}
\end{proof}
\begin{proof}[Proof of Corollary~\ref{cor:onetoall}]
Proposition~\ref{prop:onetoall} states that 
\begin{equation*}
\rho_i\in\frac{i^2r}{n^2}+\frac{(i,n)}{n}\Z.
\end{equation*}
As $r\in(n/h)\Z$, we know that
\begin{equation*}
\rho_i\in\frac{i^2r}{n^2}+\frac{(i,n)}{n}\Z\subseteq\frac{i^2}{nh}\Z+\frac{(i,n)h}{nh}\Z=\frac{(i,n)(i,h)}{nh}\Z
\end{equation*}
since
\begin{equation*}
(i^2,(i,n)h)=(i,n)(i,h)\underbrace{\left(\frac{i}{(i,n)}\frac{i}{(i,h)},\frac{h}{(i,h)}\right)}_{=1}=(i,n)(i,h).
\end{equation*}
\end{proof}

\minisec{Choice of Representatives}

Since the order of $\sigma^i$ is $n/(i,n)$, the weight grading of the twisted module $V(\sigma^i)$ is in $\rho_i+((i,n)/n)\N$. Therefore, on the level of twisted $V$-modules, $\rho_i$ is usually only relevant modulo~$(i,n)/n$. However, when passing down to untwisted $V^G$-modules, whose weight grading is in a coset of $\Z$, we must choose a representative of $\rho_i+((i,n)/n)\Z$ modulo~1. Let $\tilde\rho_i+\Z$ be that representative. Then by definition
\begin{enumerate}
\item\label{enum:rho1} $\tilde\rho_i+\frac{(i,n)}{n}\Z=\rho_i+\frac{(i,n)}{n}\Z$ for all $i\in\Z_n$.
\end{enumerate}
The naïve choice is clearly $\tilde\rho_i+\Z=\rho_i+\Z$ for all $i\in\Z_n$ with
\begin{enumerate}
\setcounter{enumi}{1}
\item\label{enum:rho2} $\tilde\rho_i+\Z=\tilde\rho_{-i}+\Z$ for all $i\in\Z_n$,
\item\label{enum:rho3} $\tilde\rho_0+\Z=0+\Z$.
\end{enumerate}
Recall that by $r_n$ for $r\in\Z_n$ we denote the representative of $r$ in $\{0,\ldots,n-1\}$. For convenience we modify the naïve choice for $i=1$ such that
\begin{enumerate}
\setcounter{enumi}{3}
\item\label{enum:rho4} $\tilde\rho_1+\Z=\frac{r_n}{n^2}+\Z$,
\end{enumerate}
which, by possibly modifying also $\tilde\rho_{n-1}+\Z$, preserves \ref{enum:rho2}. Then the $\tilde\rho_i+\Z$ satisfy all the relations \ref{enum:rho1} to \ref{enum:rho4} and in general $\tilde\rho_i+\Z\neq\rho_i+\Z$.

There are two kinds of results in Sections~\ref{sec:confweights} and \ref{sec:fusalg2} involving the conformal weights, namely:
\begin{itemize}
\item Statements about $\rho_i$ modulo~$(i,n)/n$, being independent of the choice of the representations $\phi_i$ made in Lemma~\ref{lem:phi2}: Lemma~\ref{lem:drho}, Theorem~\ref{thm:confwnn}, Lemma~\ref{lem:d2r}, Proposition~\ref{prop:onetoall} and Corollary~\ref{cor:onetoall}. These results remain correct with $\rho_i$ replaced by $\tilde\rho_i$.
\item Statements about $\rho_i$ modulo~$1$ depending on the choice of the representations $\phi_i$ made in Lemma~\ref{lem:phi2}: Propositions \ref{prop:weight} and Lemma~\ref{lem:step3}. These results remain true if $\rho_i$ is replaced by $\tilde\rho_i$ \emph{and} if the representations $\phi_i$ are chosen similar to Lemma~\ref{lem:phi2} such that
\begin{equation}\label{eq:choice}
\phi_i(\sigma)^{(i,n)}=\e^{(2\pi\i)[i/(i,n)]^{-1}(L_0-\tilde\rho_i)},
\end{equation}
which is clearly also possible.
\end{itemize}
Moreover, Proposition~\ref{prop:dual3} is also true if we replace the choice in Lemma~\ref{lem:phi2} with \eqref{eq:choice} since the $\tilde\rho_i$ also obey \ref{enum:rho2}.

The choice of representatives $\tilde\rho_i+\Z$ facilitates the proof of Lemma~\ref{lem:coboundary} below.

\minisec{Step: Coboundary}
In the following we complete the determination of the fusion algebra of $V^G$ under Assumption~\ref{ass:o}. We know from Lemma~\ref{lem:d2r} that we can write the cocycle $c\colon\Z_n\times\Z_n\to\Z_n$ as
\begin{equation*}
c=c_{2r}+\hat{c}
\end{equation*}
where $\hat{c}$ is a normalised 2-coboundary. This 2-coboundary can be further specified:
\begin{lem}\label{lem:coboundary}
Let $V$ and $G=\langle\sigma\rangle$ be as in Assumption~\ref{ass:o} with $\sigma$ of type $n\{r\}$ and assume the representations $\phi_i$ are given as in \eqref{eq:choice}. Then there is a function $\varphi\colon\Z_n\to\Z_n$ with $\varphi(0)=0$ such that
\begin{equation*}
\hat{c}(i,l)=\varphi(i)+\varphi(l)-\varphi(i+l),
\end{equation*}
i.e.\ the normalised 2-coboundary $\hat{c}$ arises from $\varphi$, with the properties
\begin{equation*}
\lambda_{i,j}=\xi_n^{-i\varphi(j)-j\varphi(i)}\xi_{n^2}^{-2i_nl_nr_n}
\end{equation*}
and
\begin{equation*}
\frac{i\varphi(i)}{n}=\tilde\rho_i-i_n^2r_n/n^2\pmod{1}.
\end{equation*}
\end{lem}
\begin{proof}
We define
\begin{equation*}
\hat\lambda_{i,j}:=\lambda_{i,j}\xi_{n^2}^{2i_nj_nr_n}
\end{equation*}
and
\begin{equation*}
\hat\rho_i+\Z:=\tilde\rho_i-i_n^2r_n/n^2+\Z.
\end{equation*}
Then
\begin{equation}\label{eq:lemcoboundary}
\frac{\hat\lambda_{i,a}\hat\lambda_{l,a}}{\hat\lambda_{i+l,a}}=\frac{\lambda_{i,a}\lambda_{l,a}}{\lambda_{i+l,a}}\e^{(2\pi\i)(2r_na_n/n^2)(i_n+l_n-(i+l)_n)}=\xi_n^{-ac(i,l)}\xi_n^{ac_{2r_n}(i,l)}=\xi_n^{-a\hat{c}(i,l)},
\end{equation}
which is an analogue of \eqref{eq:bullet} for the $\hat\lambda_{i,j}$. By Proposition~\ref{prop:onetoall}, $\hat\rho_i\in((i,n)/n)\Z$. Moreover, $\hat\lambda_{1,1}=1$ by \eqref{eq:cc} and the choice of $\tilde\rho_1$. Also, $\hat\rho_1+\Z=0+\Z$ by definition. Iterating the above equation yields that $\hat\lambda_{i,j}\in U_n$ for all $i,j\in\Z_n$. Hence, let us define the function $\varphi\colon\Z_n\to\Z_n$ by
\begin{equation*}
\hat\lambda_{1,i}=:\xi_n^{-\varphi(i)}.
\end{equation*}
Then the above equation for $a=1$ yields
\begin{equation*}
\hat{c}(i,l)=\varphi(i)+\varphi(l)-\varphi(i+l),
\end{equation*}
i.e.\ the normalised 2-coboundary $\hat{c}$ arises from $\varphi$.

The fact that $\hat\lambda_{1,1}=1$ implies $\varphi(1)=0$, which simplifies the following calculations. Iterating \eqref{eq:lemcoboundary} again yields
\begin{equation*}
\hat\lambda_{i,j}=\xi_n^{-i\varphi(j)-j\varphi(i)},
\end{equation*}
which gives the first property of $\varphi$ in the lemma. This also implies
\begin{equation*}
\hat\lambda_{i,j}=\hat\lambda_{j,1}^i\hat\lambda_{i,1}^j.
\end{equation*}

The second property is equivalent to
\begin{equation*}
\hat\lambda_{1,i}^i=\e^{(2\pi\i)(-\hat\rho_i)}.
\end{equation*}
To see this we define
\begin{equation*}
\theta_i:=\hat\lambda_{1,i}^i\e^{(2\pi\i)\hat\rho_i}
\end{equation*}
and consider, using \eqref{eq:cc},
\begin{equation*}
1=\lambda_{i,i}\e^{(2\pi\i)(2\tilde\rho_i)}=\hat\lambda_{i,i}\e^{(2\pi\i)(2\hat\rho_i)}=\hat\lambda_{1,i}^{2i}\e^{(2\pi\i)2\hat\rho_i}=\theta_i^2,
\end{equation*}
i.e.\ $\theta_i\in\{\pm1\}$. Also, equation \eqref{eq:c} translates to
\begin{equation*}
\hat\lambda_{i,i+j}\hat\lambda_{i+j,j}\e^{(2\pi\i)(\hat\rho_{i}+\hat\rho_{j}+\hat\rho_{i+j})}=\hat\lambda_{i,j},
\end{equation*}
which becomes
\begin{equation*}
\theta_i\theta_j=\theta_{i+j}
\end{equation*}
and since $\theta_1=1$ (as $\hat\lambda_{1,1}=1$ and $\hat\rho_1+\Z=0+\Z$) we get
\begin{equation*}
1=\theta_i=\hat\lambda_{1,i}^i\e^{(2\pi\i)\hat\rho_i}
\end{equation*}
for all $i\in\Z_n$, which is the claim.
\end{proof}

\minisec{Main Results}

Recall that $c_{2r}$ is the 2-cocycle $c_d$ defined above for $d=2r$. We can collect the results of this chapter in the following theorem:
\begin{thm}[Main Result]\label{thm:main}
Let $V$ and $G=\langle\sigma\rangle$ be as in Assumption~\ref{ass:o} with $\sigma$ of type $n\{r\}$ and assume the representations $\phi_i$ are given as in \eqref{eq:choice}. Then there are functions $\alpha,\varphi\colon\Z_n\to\Z_n$ such that:
\begin{enumerate}
\item The fusion rules of $V^G$ are given by
\begin{equation*}
W^{(i,j)}\boxtimes W^{(l,k)}\cong W^{(i+l,j+k+c_{2r}(i,l)+\varphi(i)+\varphi(l)-\varphi(i+l))}
\end{equation*}
for $i,j,k,l\in\Z_n$. In particular, all $W^{(i,j)}$ are simple currents.
\item The module $W^{(i,j)}$ has weights in 
\begin{equation*}
Q_\rho((i,j))=\tilde\rho_i+\frac{ij}{n}+\Z=\frac{i\varphi(i)}{n}+\frac{i_n^2r_n}{n^2}+\frac{ij}{n}+\Z
\end{equation*}
for $i,j\in\Z_n$.
\item The contragredient modules are ${W^{(i,j)}}'\cong W^{(-i,\alpha(i)-j)}$, $i,j\in\Z_n$.
\item The $S$-matrix of $V^G$ is given by
\begin{equation*}
S_{(i,j),(l,k)}=\frac{1}{n}\xi_n^{-(lj+ik)}\lambda_{i,l}=\frac{1}{n}\xi_n^{-(lj+ik+l\varphi(i)+i\varphi(l))}\xi_{n^2}^{-2i_nl_nr_n},
\end{equation*}
i.e.\ $\lambda_{i,l}=\xi_n^{-l\varphi(i)-i\varphi(l)}\xi_{n^2}^{-2i_nl_nr_n}$, for $i,j,k,l\in\Z_n$.
\end{enumerate}
The functions $\alpha,\varphi\colon\Z_n\to\Z_n$ satisfy
\setenumerate[1]{label=(\alph*)}
\begin{enumerate}
\item $\frac{i\varphi(i)}{n}=\tilde\rho_i-\frac{i_n^2r_n}{n^2}\pmod{1}$,
\item $\alpha(i)=-\varphi(i)-\varphi(-i)-c_{2r}(i,-i)$,
\item $\alpha(i)=\alpha(-i)$,
\item $\varphi(0)=0$,
\item $\alpha(0)=0$,
\item $i\alpha(i)=0$
\end{enumerate}
\setenumerate[1]{label=(\arabic*)}
for $i\in\Z_n$.
\end{thm}

Finally we obtain:
\begin{oframed}
\begin{cor}[Main Result]\label{cor:main}
Let $V$ and $G=\langle\sigma\rangle$ be as in Assumption~\ref{ass:o} with $\sigma$ of type $n\{r\}$. Then it is possible to choose the representations $\phi_i$ such that:
\begin{enumerate}
 \item \label{enum:res1} $W^{(i,j)}\boxtimes W^{(l,k)}\cong W^{(i+l,j+k+c_{2r}(i,l))}$,
 \item \label{enum:res2} $W^{(i,j)}$ has weights in $Q_\rho((i,j))=\frac{ij}{n}+\frac{i_n^2r_n}{n^2}+\Z$,
 \item \label{enum:res3} ${W^{(i,j)}}'\cong W^{(-i,-j-c_{2r}(i,-i))}$,
 \item \label{enum:res4} $S_{(i,j),(l,k)}=\frac{1}{n}\xi_n^{-(lj+ik)}\lambda_{i,l}=\frac{1}{n}\xi_n^{-(lj+ik)}\xi_{n^2}^{-2r_ni_nl_n}$, i.e.\ $\lambda_{i,l}=\xi_{n^2}^{-2r_ni_nl_n}$
\end{enumerate}
for $i,j,k,l\in\Z_n$.
\end{cor}
\end{oframed}
Note that the choice of the representations $\phi_i$ in the corollary is not necessarily the one in Lemma~\ref{lem:phi2} or \eqref{eq:choice}. But, since $\varphi(0)=0$, the representation $\phi_0$ is still chosen naturally as in Remark~\ref{rem:phi0}.
\begin{proof}
Assume the representations $\phi_i$ are chosen as in \eqref{eq:choice}. Then the results of Theorem~\ref{thm:main} hold. Define
\begin{equation*}
W_\text{new}^{(i,j)}:=W^{(i,j-\varphi(i))},
\end{equation*}
corresponding to a redefinition of the representations $\phi_i$. This removes the 2-coboundary $\hat{c}(i,l)=\varphi(i)+\varphi(l)-\varphi(i+l)$ from the relations in Theorem~\ref{thm:main}. The results \ref{enum:res1} and \ref{enum:res4} follow immediately. The conformal weight of $W_\text{new}^{(i,j)}=W^{(i,j-\varphi(i))}$ is given by $\tilde\rho_i+i(j-\varphi(i))/n=ij/n+i_n^2r_n/n^2\pmod{1}$, which proves \ref{enum:res2}. For the contragredient module we find
\begin{align*}
{W_\text{new}^{(i,j)}}'&={W^{(i,j-\varphi(i))}}'\cong W^{(-i,\alpha(i)-j+\varphi(i))}=W^{(-i,-j-\varphi(-i)-c_{2r}(i,-i))}\\
&=W_\text{new}^{(-i,-j-c_{2r}(i,-i))},
\end{align*}
which is \ref{enum:res3}.
\end{proof}

\minisec{Reformulations}
Let $E_{c}$ denote the central extension of $\Z_n$ by $\Z_n$ corresponding to the 2-cocycle $c$, i.e.\ $E_c=\Z_n\times\Z_n$ (as set) with group law
\begin{equation*}
(i,j)\oplus(l,k)=(i+l,j+k+c(i,l)).
\end{equation*}
In particular we can consider the 2-cocycles $c_d$ for $d\in\Z_n$, which are the standard representatives of the elements in $H^2(\Z_n,\Z_n)\cong\Z_n$. We just saw in Corollary~\ref{cor:main}:
\begin{prop}\label{prop:e2r}
Let $V$ and $G=\langle\sigma\rangle$ be as in Assumption~\ref{ass:o} with $\sigma$ of type $n\{r\}$. Then the fusion group $F_{V^G}$ is isomorphic as \fqs{} to the group $E_{c_{2r}}$ with the quadratic form of level $N=nh=n^2/(r,n)$,
\begin{equation*}
Q_\rho((i,j))=\frac{ij}{n}+\frac{i_n^2r_n}{n^2}+\Z
\end{equation*}
and the associated bilinear form
\begin{equation*}
B_\rho((i,j),(l,k))=\frac{ik+lj}{n}+\frac{2i_nl_nr_n}{n^2}+\Z
\end{equation*}
for $i,j,k,l\in\Z_n$.
\end{prop}

There are two further convenient ways of writing the \fqs{} $F_{V^G}$ up to isomorphism. It is an easy exercise to show the following:
\begin{lem}\label{lem:centralext}
Let $1\to \Z_n\to E_{c_d}\to \Z_n\to 1$ be a central extension corresponding to the element $d\in\Z_n\cong H^2(\Z_n,\Z_n)$. Then as groups
\begin{equation*}
E_{c_d}\cong\Z_{n^2/(d,n)}\times\Z_{(d,n)}.
\end{equation*}
If we explicitly define the central extension $E_{c_d}=\Z_n\times\Z_n$ (as set) via
\begin{equation*}
(i,j)\oplus(l,k)=(i+l,j+k+c_d(i,l)),
\end{equation*}
then the generators of $\Z_{n^2/(d,n)}$ and $\Z_{(d,n)}$ can be given by
\begin{align*}
e_1&:=(1,0),\\
e_2&:=(0,1)\oplus(-\gamma_d)\cdot(1,0)=\left(-\gamma_d,1+d\left\lfloor\gamma_d(-1)_n/n\right\rfloor\right),
\end{align*}
respectively, with
\begin{equation*}
\gamma_d:=\frac{n}{(d,n)}\left[\frac{d}{(d,n)}\right]^{-1}\pmod{n}
\end{equation*}
where $\left[d/(d,n)\right]^{-1}$ denotes the inverse of $d/(d,n)$ modulo~$n/(d,n)$.
\end{lem}
Note that $\gamma_d$ is by definition well-defined modulo~$n^2/(d,n)^2$ but not necessarily modulo~$n$, so the first entry of $e_2$ depends on the choice of $\gamma_d$ in $\Z$. The statement of the theorem is correct for all choices of $\gamma_d$ modulo~$n$.

The lemma implies that the fusion group $F_{V^G}\cong E_{c_{2r}}$ from the above proposition is isomorphic to $\Z_{n^2/(2r,n)}\times\Z_{(2r,n)}$. We recall the definition of $h=n/(r,n)$, which is some positive divisor of $n$.
\begin{prop}
Let $V$ and $G=\langle\sigma\rangle$ be as in Assumption~\ref{ass:o} with $\sigma$ of type $n\{r\}$. Then the fusion group $F_{V^G}$ is isomorphic as \fqs{} to
\begin{equation*}
\Z_{n^2/(2r,n)}\times\Z_{(2r,n)}=\Z_{\frac{nh}{(h,2)}}\times\Z_{\frac{n(h,2)}{h}}
\end{equation*}
with quadratic form given by
\begin{equation*}
(x,y)\mapsto\frac{(x-\gamma_{2r}y)y}{n}+\frac{r(x-\gamma_{2r}y)^2}{n^2}+\Z
\end{equation*}
for $x\in\Z_{n^2/(2r,n)}$ and $y\in\Z_{(2r,n)}$.
\end{prop}

Finally, there is also the following isomorphism:
\begin{prop}\label{prop:e2rlattice}
Let $V$ and $G=\langle\sigma\rangle$ be as in Assumption~\ref{ass:o} with $\sigma$ of type $n\{r\}$. Then the fusion group $F_{V^G}$ is isomorphic as \fqs{} to the discriminant form $L'/L$ of the even lattice $L$ of signature $(1,1)$ with Gram matrix $\left(\begin{smallmatrix}-2r_n&n\\n&0\end{smallmatrix}\right)$. $L$ can be realised as $n\Z\times n\Z$ embedded into $\Q\times\Q$ with quadratic form $(a,b)\mapsto ab/n-a^2r_n/n^2$ for $a,b\in\Q$.
\end{prop}

\minisec{Summary and Special Case: Type $n\{0\}$}
Let $V$ and $G=\langle\sigma\rangle$ be as in Assumption~\ref{ass:o} with $\sigma$ of type $n\{r\}$. Then we saw that the fusion group of $V^G$ is as \fqs{} given by
\begin{equation*}
F_{V^G}=(E_c,Q_\rho)
\end{equation*}
for some 2-cocycle $c$ cohomologous to the special 2-cocycle $c_{2r}$. If we choose the representations $\phi_i$ as in Corollary~\ref{cor:main}, then
\begin{equation*}
F_{V^G}=(E_{c_{2r}},Q_\rho).
\end{equation*}
Now assume in addition that $r=0$. Then
\begin{equation*}
F_{V^G}=(E_{c_0},Q_\rho)=(\Z_n\times\Z_n,Q_\rho)
\end{equation*}
with $Q_\rho((i,j))=\frac{ij}{n}+\Z$. We describe this special case in more detail:
\begin{cor}[Special Case]\label{cor:r0}
Let $V$ and $G=\langle\sigma\rangle$ be as in Assumption~\ref{ass:o} with $\sigma$ of type $n\{0\}$. If we choose the representations $\phi_i$ as in Corollary~\ref{cor:main}, then:
\begin{enumerate}
 \item $W^{(i,j)}\boxtimes W^{(l,k)}\cong W^{(i+l,j+k)}$,
 \item $W^{(i,j)}$ has weights in $Q_\rho((i,j))=\frac{ij}{n}+\Z$,
 \item ${W^{(i,j)}}'\cong W^{(-i,-j)}$,
 \item $S_{(i,j),(l,k)}=\frac{1}{n}\xi_n^{-(lj+ik)}\lambda_{i,l}=\frac{1}{n}\xi_n^{-(lj+ik)}$, i.e.\ $\lambda_{i,l}=1$
\end{enumerate}
for $i,j,k,l\in\Z_n$. This means that the fusion group of $V^G$ is the abelian group $F_{V^G}=\Z_n\times\Z_n$ with quadratic form
\begin{equation*}
Q_\rho((i,j))=\frac{ij}{n}+\Z
\end{equation*}
and associated bilinear form
\begin{equation*}
B_\rho((i,j),(l,k))=\frac{ik+lj}{n}+\Z.
\end{equation*}
\end{cor}

\section{Modular Invariance of Trace Functions}\label{sec:modinv}

Using the results obtained in this chapter so far, it is possible to make detailed statements about the modular properties of the twisted trace functions $T(v,i,j,\tau)$ or equivalently about the $T_{W^{(i,j)}}(v,\tau)$, under Assumption~\ref{ass:o}. Recall that $8\mid c$, where $c$ is the central charge of $V$ and $V^G$.

For homogeneous $v\in V^G$ with respect to the grading $\wt[\cdot]$ the $T_{W^{(i,j)}}(v,\tau)$, $i,j\in\Z_n$, transform as a vector-valued modular form of weight $\wt[v]$ for Zhu's representation $\rho_{V^G}$ of $\SLZ$. We showed in Theorem~\ref{thm:zhuweil} that this is almost the well-known Weil representation $\rho_{F_{V^G}}$ on the group algebra of the fusion group $F_{V^G}$. Since $c$ is even, the Weil representation is a representation of $\SLZ$ and
\begin{equation*}
\rho_V(M)=\eps(M)^{-c}\rho_{F_{V^G}}(M)
\end{equation*}
for $M\in\SLZ$ where $\eps(S)=\e^{(2\pi\i)(-1/8)}$ and $\eps(T)=\e^{(2\pi\i)1/24}$. For $8\mid c$ we can find an explicit formula
\begin{equation*}
Z(M):=\eps(M)^{-c}=\begin{cases}\e^{(2\pi\i)(-c/24)(\beta-\gamma)\delta}&\text{if }3\nmid \delta,\\\e^{(2\pi\i)(-c/24)(\beta+(\alpha+1)\gamma)}&\text{if }3\mid \delta\end{cases}
\end{equation*}
for $M=\left(\begin{smallmatrix}\alpha&\beta\\\gamma&\delta\end{smallmatrix}\right)\in\SLZ$, taking values in $U_3$.

Using that $\Gamma(N)$ acts trivially under the Weil representation where $N=nh=n^2/(r,n)$ is the level of the \fqs{} $F_{V^G}=(F_{V^G},Q_\rho)$, we immediately get the following:
\begin{oframed}
\begin{thm}\label{thm:modinv}
Let $V$ and $G=\langle\sigma\rangle$ be as in Assumption~\ref{ass:o} and $\sigma$ of type $n\{r\}$. Then the trace functions $T_{W^{(i,j)}}(v,\tau)$ and $T(v,i,j,\tau)$, $i,j\in\Z_n$, are modular forms of weight $\wt[v]$ for a congruence subgroup of $\SLZ$ of level
\begin{equation*}
\widetilde{N}:=\begin{cases}\lcm(3,N)&\text{if }3\nmid c/8,\\N&\text{if }3\mid c/8\end{cases}
\end{equation*}
where $N=nh=n^2/(r,n)$ is the level of the \fqs{} $F_{V^G}$.
\end{thm}
\end{oframed}
The
statement of this theorem also follows directly from Theorem~\ref{thm:thm1dln12} (from \cite{DLN15}) and is a generalisation of Theorem~1.6 (ii) in \cite{DLM00}, which is only formulated for $n=2,3$.

It seems adequate to write down the modular-transformation properties of the trace functions explicitly. Let us assume we rescaled the trace functions, i.e.\ the representations $\phi_i$, as in Corollary~\ref{cor:main}. Then:
\begin{align*}
T(v,i,j,T.\tau)&=\e^{(2\pi\i)(i_n^2r_n/n^2-c/24)}T(v,i,j+i,\tau),\\
\tau^{-\wt[v]}T(v,i,j,S.\tau)&=\e^{(2\pi\i)(-2r_ni_nj_n/n^2)}T(v,j,-i,\tau)
\end{align*}
or equivalently
\begin{align*}
T_{W^{(i,j)}}(v,T.\tau)&=\e^{(2\pi\i)(ij/n+i_n^2r_n/n^2-c/24)}T_{W^{(i,j)}}(v,\tau),\\
\tau^{-\wt[v]}T_{W^{(i,j)}}(v,S.\tau)&=\sum_{l,k\in\Z_n}\e^{(2\pi\i)(-(lj+ik)/n-2r_ni_nl_n/n^2)}T_{W^{(l,k)}}(v,\tau)
\end{align*}
for all $i,j\in\Z_n$ and homogeneous $v\in V$.

\minisec{Special Case: Type $n\{0\}$}
In the following we consider the special case where $\sigma$ is of type $n\{0\}$. Then the fusion group $F_{V^G}$ is isomorphic to the group $\Z_n\times\Z_n$ with the quadratic form $Q_\rho((i,j))=ij/n$ for $(i,j)\in\Z_n\times\Z_n$. Let us also assume the representations $\phi_i$ to be chosen as in Corollary~\ref{cor:main}. Then this isomorphism is even an equality, i.e.\ $F_{V^G}=(\Z_n\times\Z_n,Q_\rho)$.

If $r=0$, then all the factors containing $r$ in the above modular transformations disappear and we get for $M=\left(\begin{smallmatrix}a&b\\c&d\end{smallmatrix}\right)\in\SLZ$,
\begin{equation}\label{eq:trafon0}
(c\tau+d)^{-\wt[v]}T(v,i,j,M.\tau)=Z(M)T(v,(i,j)M,\tau)
\end{equation}
for the character $Z\colon\SLZ\to U_3$ defined above.

If in addition $24\mid c$, then even
\begin{equation}\label{eq:modinvverynice}
(c\tau+d)^{-\wt[v]}T(v,i,j,M.\tau)=T(v,(i,j)M,\tau),
\end{equation}
which means that the $T(v,i,j,\tau)$, $i,j\in\Z_n$, form a vector-valued modular form for a representation of $\SLZ$ with values in $(n^2\times n^2)$-permutation matrices. The modular properties of the single $T(v,i,j,\tau)$ are then particularly simple to read off.

\begin{prop}\label{prop:modinv}
Let $V$ and $G=\langle\sigma\rangle$ be as in Assumption~\ref{ass:o} and $\sigma$ of type $n\{0\}$ and suppose $24\mid c$. Then the trace functions
\begin{enumerate}
\item \label{enum:trace1} $T(v,i,j,\tau)$ and $T_{W^{(i,j)}}(v,\tau)$, $i,j\in\Z_n$, are modular forms for $\Gamma(n)$,
\item \label{enum:trace2} $T(v,0,j,\tau)$, $j\in\Z_n$, are modular forms for $\Gamma_1(n)$,
\item \label{enum:trace3} $T_{W^{(0,0)}}(v,\tau)$ is a modular form for $\Gamma_0(n)$.
\end{enumerate}
Indeed, assume that the trace functions, i.e.\ the representations $\phi_i$, are chosen as in Corollary~\ref{cor:main}. Then for $i,j\in\Z_n$ the trace function $T(v,i,j,\tau)$ is a modular form for the stabiliser $\Gamma_{(i,j)}$ of $(i,j)$ under the right action of $\SLZ$ on $\Z_n\times\Z_n$ by matrix multiplication and
\begin{equation*}
\Gamma(n)\leq\Gamma_{(i,j)}\leq\SLZ.
\end{equation*}
\end{prop}
\begin{proof}
The modular-transformation property of $T(v,i,j,\tau)$ stated above directly yields the modular invariance under $\Gamma_{(i,j)}$. All $T(v,i,j,\tau)$ are modular invariant under $\Gamma(n)$, which is already stated in the above theorem. This is \ref{enum:trace1}. $\Gamma_1(n)=\Gamma_{(0,1)}$ is contained in $\Gamma_{(0,j)}$ for all $j\in\Z_n$, which proves \ref{enum:trace2}. Finally, consider $T_{W^{(0,0)}}(v,\tau)=(1/n)\sum_{j\in\Z_n}T(v,0,j,\tau)$. $\Gamma_0(n)$ keeps $\{(0,j)\mid j\in\Z_n\}$ invariant as a set. This proves item~\ref{enum:trace3}.
\end{proof}

\section{Orbifold Construction}\label{sec:orbifold}

In the following we combine our knowledge from this chapter about the fusion algebra of $V^G$ under Assumption~\ref{ass:o} with the results in Chapter~\ref{ch:aia} about \aia{}s in the case of group-like fusion (under Assumptions~\ref{ass:sn}\ref{ass:p}) to construct a new holomorphic \voa{} $\widetilde{V}$ from the irreducible $V^G$-modules. This process is called \emph{orbifolding} or \emph{orbifold construction}.

Having determined the fusion algebra of $V^G$, we can use Theorem~\ref{thm:2.7} to endow the direct sum of all irreducible $V^G$-modules up to isomorphism with the structure of an \aia{}. Under Assumption~\ref{ass:op} the assumptions of Theorem~\ref{thm:2.7} are fulfilled because of Theorem~\ref{thm:orb}.
\begin{framed}
\begin{thm}\label{thm:aia}
Let $V$ and $G=\langle\sigma\rangle$ be as in Assumption~\ref{ass:op} and $\sigma$ of type $n\{r\}$. Then the direct sum
\begin{equation*}
A:=\bigoplus_{i,j\in\Z_n}W^{(i,j)}=\bigoplus_{\gamma\in F_{V^G}}W^\gamma
\end{equation*}
of all $n^2$ irreducible $V^G$-modules up to isomorphism can be given the structure of an \aia{}, the unique one up to a normalised abelian 3-coboundary extending the \voa{} structure of $V$ and that of its irreducible modules, with associated \fqs{} $\overline{F_{V^G}}=(F_{V^G},-Q_\rho)$.
\end{thm}
\end{framed}
By Theorem~\ref{thm:4.3} we get a \voa{} structure if we restrict to the modules corresponding to an isotropic subgroup of the \fqs{} $F_{V^G}$.
\begin{oframed}
\begin{thm}\label{thm:orbvoa}
Let $V$ and $G=\langle\sigma\rangle$ be as in Assumption~\ref{ass:op} and $\sigma$ of type $n\{r\}$. Let $I$ be an isotropic subgroup of the fusion group $F_{V^G}$. Then the direct sum
\begin{equation*}
V_I=\bigoplus_{\gamma\in I}W^\gamma
\end{equation*}
admits the unique structure of a \voa{} extending the \voa{} structure of $V^G$ and that of its irreducible modules. $V_I$ satisfies Assumptions~\ref{ass:sn}\ref{ass:p}.

If $I=I^\bot$, then $\widetilde{V}:=V_I$ is holomorphic.
\end{thm}
\end{oframed}

We call the holomorphic \voa{} $\widetilde{V}$ the \emph{orbifold} of $V$.\footnote{Other authors use the term ``orbifold'' to refer to the \fpvosa{} $V^G$.}

By Proposition~\ref{prop:aad} any subgroup $I$ of $F_{V^G}$ with $I=I^\bot$ fulfils $|I|=n$ and is maximal isotropic by Remark~\ref{rem:maxiso}, item~\ref{enum:maxiso2}.

\minisec{Maximal Isotropic Subgroups}

In order to apply the above theorem to construct new holomorphic \voa{}s $\widetilde{V}$ we have to find isotropic subgroups $I$ of the fusion group $F_{V^G}$ with $I=I^\bot$. Clearly, the group $I_0:=\{(0,j)\;|\;j\in\Z_n\}$ is isotropic with $I_0=I_0^\bot$ and gives back the original holomorphic \voa{} $V$, namely
\begin{equation*}
\bigoplus_{\gamma\in I_0}W^\gamma=\bigoplus_{j\in\Z_n}W^{(0,j)}=\bigoplus_{j\in\Z_n}V^j=V.
\end{equation*}
To obtain a new holomorphic \voa{}, the isotropic subgroup $I$ has to be chosen such that $I\neq I_0$. In fact, $I$ should have trivial intersection with $I_0$. Indeed, in the following we will show that whenever $I$ and $I_0$ have non-trivial intersection, then the new \voa{} $\widetilde{V}$ can be obtained as an orbifold of some smaller order automorphism with trivial intersection of the corresponding subgroups.
\begin{lem}
Let $V$ and $G=\langle\sigma\rangle$ be as in Assumption~\ref{ass:op} and let $I$ be an isotropic subgroup of $F_{V^G}$ with $I=I^\bot$ and intersection $H:=I\cap I_0\neq\{0\}$ with $I_0$. Let $\widetilde{V}=\bigoplus_{\gamma\in I}W^\gamma$ be the orbifold of $V$ associated with the subgroup $I$. Then:
\begin{enumerate}
\item $|H|=:m$ divides $n$ and $\rho:=\sigma^m$ is an automorphism of $V$ of order $n'=n/m$.
\item The $n'^2=|H^\bot/H|$ irreducible $V^{\langle\rho\rangle}$-modules admit the structure of an \aia{} with associated quadratic space $H^\bot/H$.
\item The \voa{}s $\widetilde{V}$ and $V$ are isomorphic to the direct sums of the irreducible $V^{\langle\rho\rangle}$-modules corresponding to $I/H$ and $I_0/H$, respectively.
\item $I/H\cap I_0/H=\{0\}$.
\end{enumerate}
\end{lem}
\begin{proof}
$H$ is a subgroup of the cyclic group $I_0\cong\Z_n$ and the $V^G$-modules corresponding to the elements in $I_0$ are the eigenspaces of $\sigma$ in $V$. Hence, $H$ is cyclic of some order $m:=|H|$, i.e.\ $H=(n/m)I_0$, and $V_H=\bigoplus_{\gamma\in H}W^\gamma$ is the \fpvosa{} of $V$ under the automorphism $\rho=\sigma^m$ of $V$ of order $n'=n/m$, i.e.\ $V_H\cong V^{\langle\rho\rangle}$.

By Theorem~\ref{thm:4.3} the $n'^2=|H^\bot/H|$ irreducible $V^{\langle\rho\rangle}$-modules are given by $X^{\alpha+H}=\bigoplus_{\gamma\in\alpha+H}W^{\gamma}$ for $\alpha+H\in H^\bot/H$, i.e.\ they are indexed by the elements of $H^\bot/H$. The direct sum of modules $\bigoplus_{\mu+H\in H^\bot/H}X^{\mu+H}$ admits the structure of an \aia{} with associated quadratic space $H^\bot/H$.

The \voa{} $\widetilde{V}$ can be obtained by an orbifold construction as $\widetilde{V}=\bigoplus_{\mu+H\in I/H}X^{\mu+H}=\bigoplus_{\gamma\in I}W^\gamma$, i.e.\ as a direct sum of the irreducible $V^{\langle\rho\rangle}$-modules corresponding to the isotropic subgroup $I/H$ of $H^\bot/H$.
\end{proof}

The above lemma shows that any holomorphic \voa{} $\widetilde{V}$ obtained as orbifold from Theorem~\ref{thm:orbvoa} can be obtained as an orbifold of possibly smaller order where $I$ and $I_0$ intersect trivially. On the other hand, in the next lemma we show that an isotropic subgroup $I$ with $I^\bot=I$ and trivial intersection with $I_0$ can only exist in the case of $r=0$, i.e.\ when $\sigma$ is of type $n\{0\}$.
\begin{lem}
Let $V$ and $G=\langle\sigma\rangle$ be as in Assumption~\ref{ass:op} and $\sigma$ of type $n\{r\}$. Let $I$ be an isotropic subgroup of $F_{V^G}$ with $I^\bot=I$ and $I\cap I_0=\{0\}$. Then $r=0$ and if the trace functions, i.e.\ the representations $\phi_i$, are rescaled as in Corollary~\ref{cor:main}, then $F_{V^G}=E_{c_0}=\Z_n\times\Z_n$ and $I=\Z_n\times\{0\}$ while $I_0=\{0\}\times\Z_n$.
\end{lem}
\begin{proof}
There are two isotropic subgroups $I$ and $I_0$ with trivial intersection. This implies $F_{V^G}=\{0\}^\bot=(I_0\cap I)^\bot=I_0^\bot+I^\bot=I_0+I$, i.e.\ $I$ and $I_0$ generate the fusion group $F_{V^G}$. Both together means that $F_{V^G}$ is isomorphic to a semidirect product of $I$ and $I_0$. Since $I$, $I_0$ and $F_{V^G}$ are abelian, this semidirect product is in fact direct and $F_{V^G}\cong I\times I_0$. It is easy to see that the only $r\in\Z_n$ for which the \fqs{} $F_{V^G}\cong E_{c_{2r}}$ is isomorphic to the direct product of the two isotropic groups $I$ and $I_0$ of order $n$, of which $I_0$ is isomorphic to $\Z_n$, is $r=0$. In this case, if the representations $\phi_i$ are chosen as in Corollary~\ref{cor:main}, $F_{V^G}=E_{c_0}=\Z_n\times\Z_n$, $I_0=\{0\}\times\Z_n$ and $I$ can only be $I=\Z_n\times\{0\}$.
\end{proof}

Together both lemmata imply:
\begin{prop}\label{prop:zerosuff}
To construct new holomorphic \voa{}s $\widetilde{V}$ using the orbifold construction in Theorem~\ref{thm:orbvoa} it suffices to consider $\sigma$ of type $n\{0\}$ so that the fusion group is $F_{V^G}\cong\Z_n\times\Z_n$ and to choose $I\cong\Z_n\times\{0\}$ under this isomorphism, with equality if the representations $\phi_i$ are chosen as in Corollary~\ref{cor:main}.
\end{prop}

\minisec{Special Case: Type $n\{0\}$}
By the above proposition we only need to consider the type $n\{0\}$. The fusion algebra in this special case is described in Corollary~\ref{cor:r0}. Let us assume that representations $\phi_i$ are chosen as in Corollary~\ref{cor:main}. Then the direct sum of irreducible $V^G$-modules
\begin{equation*}
A=\bigoplus_{(i,j)\in\Z_n\times\Z_n}W^{(i,j)}
\end{equation*}
admits the structure of an \aia{} with associated quadratic space $(\Z_n\times\Z_n,-Q_\rho)$ and the direct sum of irreducible $V^G$-modules
\begin{equation}\label{eq:orbifold}
\widetilde{V}:=\bigoplus_{i\in\Z_n}W^{(i,0)}
\end{equation}
corresponding to the isotropic subgroup $\Z_n\times\{0\}$ of $\Z_n\times\Z_n$ admits the structure of a holomorphic \voa{} satisfying Assumptions~\ref{ass:sn}\ref{ass:p} and extending the \voa{} $V^G$. $\widetilde{V}$ is a $\Z_n$-graded simple-current extension of $V^G$.

\minisec{Inverse Orbifold}

We continue in above setting, i.e.\ let $V$ and $G=\langle\sigma\rangle$ satisfy Assumption~\ref{ass:op} with $\sigma$ of type $n\{0\}$. We will show that the orbifolding process $V\to\widetilde{V}$ can be reverted, i.e.\ we can find an automorphism $\amgis$ on $\widetilde{V}$ such that orbifolding with $K=\langle\amgis\rangle$ gives back the original \voa{} $V$.
\begin{thm}\label{thm:invorb}
Let $V$ and $G=\langle\sigma\rangle$ be as in Assumption~\ref{ass:op} with $\sigma$ of type $n\{0\}$ and the representations $\phi_i$ chosen as in Corollary~\ref{cor:main} and let $\widetilde{V}=\bigoplus_{i\in\Z_n}W^{(i,0)}$ be the orbifold \voa{}. Then:
\begin{enumerate}
\item The operator $\amgis$ defined on $\widetilde{V}$ by $\amgis v=\xi_n^i v$ for $v\in W^{(i,0)}$ is an automorphism of the \voa{} $\widetilde{V}$ of type $n\{0\}$.
\item The unique irreducible $\amgis^i$-twisted $\widetilde{V}$-module is given up to isomorphism by $\widetilde{V}(\amgis^i)\cong\bigoplus_{j\in\Z_n}W^{(j,i)}$, $i\in\Z_n$.
\item The orbifold construction for $\widetilde{V}$ and $K=\langle\amgis\rangle$ yields $\bigoplus_{i\in\Z_n}W^{(0,i)}\cong V$.
\end{enumerate}
\end{thm}

The situation is shown in the following table:
\renewcommand{\arraystretch}{1.5}
\begin{equation*}
\begin{array}{c||c|ccc|}
\cline{2-2}
\multicolumn{1}{c|}{}&\multicolumn{1}{c|}{\widetilde{V}}&\widetilde{V}(\amgis^{1})&\multicolumn{1}{c}{\cdots}&\multicolumn{1}{c}{\widetilde{V}(\amgis^{n-1})}\\
\hhline{-#=|===|}
\multicolumn{1}{|c||}{V}&W^{(0,0)}&W^{(0,1)}&\cdots&W^{(0,n-1)}\\
\hline
V(\sigma^1)&W^{(1,0)}&W^{(1,1)}&\cdots&W^{(1,n-1)}\\
\vdots&\vdots&\vdots&\ddots&\vdots\\
V(\sigma^{n-1})&W^{(n-1,0)}&W^{(n-1,1)}&\cdots&W^{(n-1,n-1)}\\
\hhline{~|b|-|---|}
\end{array}
\end{equation*}
\renewcommand{\arraystretch}{1}

\begin{proof}
The automorphism $\sigma$ of $V$ of order $n$ is by definition an automorphism of the vector space $V$ fixing the vacuum and the Virasoro vector and fulfilling
\begin{equation*}
\sigma Y_V(v,x)\sigma^{-1}=Y_V(\sigma v,x)
\end{equation*}
for all $v\in V$. We can decompose $V$ as
\begin{equation*}
V=\bigoplus_{j\in\Z_n}W^{(0,j)},
\end{equation*}
where $\sigma$ acts on $W^{(0,j)}$ by multiplication with $\xi_n^j$.

In analogy, we define a vector-space automorphism $\amgis$ on the orbifolded \voa{}
\begin{equation*}
\widetilde{V}=\bigoplus_{i\in\Z_n}W^{(i,0)}
\end{equation*}
by setting $\amgis v=\xi_n^i v$ for $v\in W^{(i,0)}$. By construction, this fixes the vacuum and the Virasoro vectors, which lie in $W^{(0,0)}$. So, for $\amgis$ to be an automorphism of the \voa{} $\widetilde{V}$, the vertex operator $Y_{\widetilde{V}}(\cdot,x)$ on $\widetilde{V}$ has to fulfil
\begin{equation}\label{eq:amgis}
\amgis Y_{\widetilde{V}}(v,x)\amgis^{-1}=Y_{\widetilde{V}}(\amgis v,x)
\end{equation}
for all $v\in \widetilde{V}$. Indeed, since $\widetilde{V}$ is a $\Z_n$-graded extension of $V^G$, $Y_{\widetilde{V}}(v,x)w\in W^{(i+i',0)}\{x\}$ for $v\in W^{(i,0)}$ and $w\in W^{(i',0)}$. Then the left-hand side of the above equation acting on $w$ is given by
\begin{equation*}
\amgis Y_{\widetilde{V}}(v,x)\amgis^{-1}w=\xi_n^{-i'}\amgis Y_{\widetilde{V}}(v,x)w=\xi_n^{-i'}\xi_n^{i+i'}Y_{\widetilde{V}}(v,x)w=\xi_n^{i}Y_{\widetilde{V}}(v,x)w.
\end{equation*}
On the other hand, for the right-hand side we obtain
\begin{equation*}
Y_{\widetilde{V}}(\amgis v,x)w=Y_{\widetilde{V}}(\xi_n^i v,x)w=\xi_n^i Y_{\widetilde{V}}(v,x)w
\end{equation*}
and hence \eqref{eq:amgis} is fulfilled and we conclude that $\amgis$ is a \voa{} automorphism of $\widetilde{V}$ of order $n$.

We saw that $\widetilde{V}$ is holomorphic, i.e.\ has exactly one irreducible module up to isomorphism, namely the adjoint module $\widetilde{V}$. Then, by Theorem~\ref{thm:10.3}, $\widetilde{V}$ has exactly one irreducible $\amgis^j$-twisted module $\widetilde{V}(\amgis^j)$ up to isomorphism for each $j\in\Z_n$.

Let $X\cong\widetilde{V}(\amgis^{j_0})$ be such an irreducible $\amgis^{j_0}$-twisted module of $\widetilde{V}$. $X$ is an untwisted $\widetilde{V}^K=W^{(0,0)}=V^G$-module and hence a direct sum of some of the modules $W^{(i,j)}$, $i,j\in\Z_n$. By the definition of twisted modules the exponents of the formal variable in the vertex operation of $\widetilde{V}$ on $X$ should lie in $j_0k/n+\Z$ if we restrict to $W^{(k,0)}\subseteq\widetilde{V}$. On the other hand, the intertwining operators of type $\binom{W^{(i,k+j)}}{W^{(k,0)}\,W^{(i,j)}}$ have exponents of the formal variable in $kj/n+\Z$. Hence $X$ can only consist of the modules $W^{(i,j_0)}$, $i\in\Z_n$. Since $X$ is non-empty, we can assume that $X$ contains $W^{(i_0,j_0)}$ for some fixed $i_0$. But then it contains $\widetilde{V}\cdot W^{(i_0,j_0)}=\bigoplus_{i\in\Z_n}W^{(i,0)}\cdot W^{(i_0,j_0)}=\bigoplus_{i\in\Z_n}W^{(i+i_0,j_0)}=\bigoplus_{i\in\Z_n}W^{(i,j_0)}$. Using the irreducibility of $X$ we conclude that $X\cong\bigoplus_{i\in\Z_n}W^{(i,j_0)}$.

The last item follows immediately, which completes the proof.
\end{proof}

\section{Digression: Dimension Formula}\label{sec:dimform}

Let $\widetilde{V}$ be the holomorphic \voa{} \eqref{eq:orbifold} obtained in the last section as orbifold of $V$ for an automorphism of type $n\{0\}$. Assume in addition that the central charge is $c=24$. We stated in Proposition~\ref{prop:modinv} that $T_{W^{(0,0)}}(v,\tau)$ is a modular form for $\Gamma_0(n)$. This fact can be exploited to prove for certain orders $n$ a formula for the dimension of $\widetilde{V}_1$, the weight-one space of $\widetilde{V}$. For $n=2,3,5,7$ this dimension formula was first stated and proved in \cite{Mon94}, formulæ (10), (11) and (24) to (26). In \cite{LS16a}, Theorem~4.4, a more rigorous proof for the case $n=2$ is given. We generalise this result to orders $n=2,3,5,7,13$:
\begin{prop}[Dimension Formula I]\label{prop:dimform}
Let $V$ and $G=\langle\sigma\rangle$ be as in Assumption~\ref{ass:op} with $\sigma$ of type $n\{0\}$ and the representations $\phi_i$ chosen as in Corollary~\ref{cor:main} and let $\widetilde{V}=\bigoplus_{i\in\Z_n}W^{(i,0)}$ be the orbifold \voa{} \eqref{eq:orbifold}. Furthermore, assume that $c=24$. Then for $n=2,3,5,7,13$, the primes $n$ with $(n-1)\mid 24$, the following dimension formulæ hold: for $n=2$
{
\addtolength{\jot}{4pt}
\begin{equation*}
\dim_\C(V_1)+\dim_\C(\widetilde{V}_1)=3\dim_\C(V^G_1)-\frac{3}{256}\dim_\C(V^G_2) + \frac{75471}{64},
\end{equation*}
for $n=3$
\begin{align*}
\dim_\C(V_1)+\dim_\C(\widetilde{V}_1)&=
4\dim_\C(V^G_1)\\
&\quad-\frac{436}{6561}\dim_\C(V^G_2)
+\frac{28}{19683}\dim_\C(V^G_3)\\
&\quad-\frac{114246104}{19683},
\end{align*}
for $n=5$
\begin{align*}
\dim_\C(V_1)+\dim_\C(\widetilde{V}_1)&=
6\dim_\C(V^G_1)\\
&\quad-\frac{4285488}{9765625}\dim_\C(V^G_2)
+\frac{443466}{9765625}\dim_\C(V^G_3)\\
&\quad-\frac{30138}{9765625}\dim_\C(V^G_4)
+\frac{966}{9765625}\dim_\C(V^G_5)\\
&\quad-\frac{439185525204}{9765625},
\end{align*}
for $n=7$
\begin{align*}
\dim_\C(V_1)+\dim_\C(\widetilde{V}_1)&=
8\dim_\C(V^G_1)\\
&\quad-\frac{36795312}{40353607}\dim_\C(V^G_2)
+\frac{4010900}{40353607}\dim_\C(V^G_3)\\
&\quad-\frac{448800}{282475249}\dim_\C(V^G_4)
-\frac{52284}{40353607}\dim_\C(V^G_5)\\
&\quad+\frac{51172}{282475249}\dim_\C(V^G_6)
-\frac{2392}{282475249}\dim_\C(V^G_7)\\
&\quad+\frac{7451576705112}{40353607}
\end{align*}
and for $n=13$
\begin{align*}
\dim_\C(V_1)+\dim_\C(\widetilde{V}_1)&=
14\dim_\C(V^G_1)\\
&\quad-\frac{4340182604352}{1792160394037}\dim_\C(V^G_2)
-\frac{94190051662}{137858491849}\dim_\C(V^G_3)\\
&\quad+\frac{1170998168940}{1792160394037}\dim_\C(V^G_4)
-\frac{261335651400}{1792160394037}\dim_\C(V^G_5)\\
&\quad-\frac{58352587520}{1792160394037}\dim_\C(V^G_6)
+\frac{3919301316}{137858491849}\dim_\C(V^G_7)\\
&\quad-\frac{12697988616}{1792160394037}\dim_\C(V^G_8)
+\frac{692327890}{1792160394037}\dim_\C(V^G_9)\\
&\quad+\frac{393826480}{1792160394037}\dim_\C(V^G_{10})
-\frac{111115482}{1792160394037}\dim_\C(V^G_{11})\\
&\quad+\frac{12706778}{1792160394037}\dim_\C(V^G_{12})
-\frac{577738}{1792160394037}\dim_\C(V^G_{13})\\
&\quad+\frac{4455938899358543724}{1792160394037}.
\end{align*}
}
\end{prop}
This dimension formula is useful in cases where an explicit description of the twisted $V$-modules is not known as the dimension of $\widetilde{V}_1$ in this formula only depends on that of $V_1$ and $V^G_1,\ldots,V^G_n$.
In the following, we will prove this dimension formula in a number of steps, starting with the following lemma:
\begin{lem}\label{lem:chsum}
Under the assumptions of Proposition~\ref{prop:dimform}, but only requiring that $24\mid c$ and $n$ be prime:
\begin{equation*}
\ch_{V^G}(\tau)+\sum_{l\in\Z_n}\ch_{V^G}(ST^l.\tau)=\ch_{V}(\tau)+\ch_{\widetilde{V}}(\tau).
\end{equation*}
\end{lem}
\begin{proof}
Consider
\begin{equation*}
\ch_{V^G}(\tau)=\ch_{W^{(0,0)}}(\tau)=\frac{1}{n}\sum_{k\in\Z_n}T(\vac,0,k,\tau)
\end{equation*}
and hence, since $24\mid c$,
\begin{equation}\label{eq:dimform1}
\ch_{V^G}(ST^l.\tau)=\frac{1}{n}\sum_{k\in\Z_n}T(\vac,k,lk,\tau)
\end{equation}
by \eqref{eq:modinvverynice}. Then
\begin{align*}
&\ch_{V^G}(\tau)+\sum_{l\in\Z_n}\ch_{V^G}(ST^l.\tau)\\
&=\frac{1}{n}\sum_{l\in\Z_n}T(\vac,0,l,\tau)+\sum_{l\in\Z_n}\frac{1}{n}\sum_{k\in\Z_n}T(\vac,k,lk,\tau)\\
&=\frac{1}{n}\sum_{l\in\Z_n}T(\vac,0,l,\tau)+\frac{1}{n}\sum_{k\in\Z_n\setminus\{0\}}\sum_{l\in\Z_n}T(\vac,k,lk,\tau)+T(\vac,0,0,\tau)\\
&=\frac{1}{n}\sum_{l\in\Z_n}T(\vac,0,l,\tau)+\frac{1}{n}\sum_{k\in\Z_n\setminus\{0\}}\sum_{l\in\Z_n}T(\vac,k,l,\tau)+T(\vac,0,0,\tau)\\
&=\frac{1}{n}\sum_{k\in\Z_n}\sum_{l\in\Z_n}T(\vac,k,l,\tau)+T(\vac,0,0,\tau)\\
&=\ch_{V}(\tau)+\ch_{\widetilde{V}}(\tau),
\end{align*}
where we used in the third step that $n$ is prime.
\end{proof}

Recall that for $n=1,2,3,4,5,6,7,8,9,10,12,13,16,18,25$ the modular curve $X_0(n)=\Gamma_0(n)\backslash\H^*$ has genus 0, i.e.\ the field of modular functions\footnote{A \emph{modular function} is a weight-zero modular form for a certain congruence subgroup that is meromorphic on $\H$ and at the cusps.} for $\Gamma_0(n)$ is generated by a single function $t_n$ called a \emph{Hauptmodul}. For those $n$ where $(n-1)\mid 24$, i.e.\ for $n=2,3,4,5,7,9,13,25$, there is a well-known, simple formula for a Hauptmodul for $\Gamma_0(n)$ (see e.g.\ \cite{Apo90}, Chapter~4):
\begin{equation*}
t_n(\tau):=\left(\frac{\eta(\tau)}{\eta(n\tau)}\right)^{\frac{24}{n-1}}=q_\tau^{-1}-\frac{24}{n-1}+\ldots\in q_\tau^{-1}+\Z[[q_\tau]]
\end{equation*}
where $\eta$ is the Dedekind eta function. The Hauptmodul $t_n$ is holomorphic on $\H$ and has no zeroes on $\H$. The $S$-transformation is given by
\begin{equation*}
t_n(S.\tau)=n^{\frac{12}{n-1}}\left(\frac{\eta(\tau)}{\eta(\tau/n)}\right)^{\frac{24}{n-1}}=n^{\frac{12}{n-1}}\left(q_\tau^{1/n}+\ldots\right)\in q_\tau^{1/n}\Q[[q_\tau^{1/n}]].
\end{equation*}
This shows that as modular function for $\Gamma_0(n)$ the Hauptmodul $t_n$ has a first-order pole at the cusp $\i\infty$ and vanishes at the cusp $0=S.\i\infty$.

We also collect some facts about the trace functions, which will be used in the proof of the following lemma. Under the assumptions of Proposition~\ref{prop:dimform}:
\begin{enumerate}
\item \label{enum:prop1} $T_{V^G}(v,\tau)=T_{W^{(0,0)}}(v,\tau)$ is a modular form of weight $\wt[v]$ for $\Gamma_0(n)$ (by Proposition~\ref{prop:modinv}),
\item \label{enum:prop2} $T_{V^G}(v,\tau)=T_{W^{(0,0)}}(v,\tau)$ is holomorphic on $\H$ but may have poles at the cusps (by Theorem~\ref{thm:zhumodinv}),
\item \label{enum:prop3} $\ch_{V^G}(\tau)=\ch_{W^{(0,0)}}(\tau)\in q_\tau^{-1}+\Z[[q_\tau]]$ since $V^G$ is of CFT-type,
\item \label{enum:prop4} $\ch_{V(\sigma^i)}(\tau)=T(\vac,i,0,\tau)\in q_\tau^{-1+1/n}\Z[[q_\tau^{1/n}]]$ for $i\in\Z_n\setminus\{0\}$ by the positivity assumption,
\item \label{enum:prop5} $\ch_{V}(\tau)=T(\vac,0,0,\tau)\in q_\tau^{-1}+\Z[[q_\tau]]$ since $V$ is of CFT-type.
\end{enumerate}

\begin{lem}\label{lem:chvstn}
Under the assumptions of Proposition~\ref{prop:dimform}, in particular $c=24$ and $n=2,3,5,7,13$, the character $ch_{V^G}$ can be written in terms of the Hauptmodul $t_n$ as
\begin{equation*}
\ch_{V^G}(\tau)=t_n(\tau)+c_0+c_{-1}t_n(\tau)^{-1}+\ldots+c_{-n+1}t_n(\tau)^{-n+1}+n^{\frac{11n+1}{n-1}}t_n(\tau)^{-n}
\end{equation*}
for some $c_k\in\Z$, $k=0,1,\ldots,n-1$.
\end{lem}
\begin{proof}
The character $\ch_{V^G}(\tau)$ is a modular function for $\Gamma_0(n)$ by \ref{enum:prop1}. By \ref{enum:prop2} we know that $\ch_{V^G}(\tau)$ is holomorphic on $\H$ with possible poles at the cusps $\{0,\i\infty\}$ modulo~$\Gamma_0(n)$.

It is clear from the $q_\tau$-expansion \ref{enum:prop3} of $\ch_{V^G}(\tau)$ and that of $t_n(\tau)$ that the difference $\ch_{V^G}(\tau)-t_n(\tau)$ is holomorphic at the cusp $\i\infty$ (and on $\H$) but may still have a pole at the cusp $0=S.\i\infty$.

We then consider the $S$-transformation of $\ch_{V^G}(\tau)$, which is
\begin{equation*}
\ch_{V^G}(S.\tau)=\frac{1}{n}\sum_{k\in\Z_n}T(\vac,k,0,\tau)\in\frac{1}{n}q_\tau^{-1}+q_\tau^{-1+1/n}\Q[[q_\tau^{1/n}]]
\end{equation*}
by \ref{enum:prop4} and \ref{enum:prop5} and has the same singular and constant terms as $\ch_{V^G}(S.\tau)-t_n(S.\tau)$. Hence, by subtracting a suitable polynomial in $1/t_n$ of degree $n$
\begin{equation*}
Q(1/t_n)=c_{-1}t_n^{-1}+\ldots+c_{-n}t_n^{-n}
\end{equation*}
we achieve that the $q_\tau$-expansion of $\ch_{V^G}(S.\tau)-t_n(S.\tau)-Q(1/t_n(S.\tau))$ starts with a constant term, i.e.\ $\ch_{V^G}(\tau)-t_n(\tau)-Q(1/t_n(\tau))$ is holomorphic at the cusp $0$. In order for the $q_\tau^{-1}$-term to vanish we need $1/n=c_{-n}n^{-12n/(n-1)}$, i.e.\ $c_{-n}=n^{(11n+1)/(n-1)}$. One can even show that the other coefficients $c_{-n+1},\ldots,c_{-1}\in\Z$ but this is not essential.

In total, $\ch_{V^G}(\tau)-t_n(\tau)-Q(1/t_n(\tau))$ is a modular function for $\Gamma_0(n)$ and is holomorphic on $\H$ and at the cusps, noting that $t_n(\tau)$ has no zeroes on $\H$ and $Q(1/t_n(\tau))$ vanishes at the cusp $\i\infty$. Then, by the valence formula for modular forms, $\ch_{V^G}(\tau)-t_n(\tau)-Q(1/t_n(\tau))$ has to be constant. We call this constant $c_0$, which is also in $\Z$.
\end{proof}

\begin{proof}[Proof of Proposition~\ref{prop:dimform}]
Continuing from the above lemma, the next step is to determine the coefficients $c_0,c_{-1},\ldots,c_{-n+1}$ by comparing the coefficients in the Fourier expansion
\begin{align*}
\ch_{V^G}(\tau)&=t_n(\tau)+c_0+c_{-1}t_n(\tau)^{-1}+\ldots+c_{-n+1}t_n(\tau)^{-n+1}+n^{\frac{11n+1}{n-1}}t_n(\tau)^{-n}\\
&=q_\tau^{-1}+\dim_\C(V^G_1)+\dim_\C(V^G_2)q_\tau+\ldots.
\end{align*}
Inserting the $q_\tau$-expansion of $t_n$, which can easily obtained from the $q_\tau$-expansion of the eta function, we can read off the coefficients $c_0,c_{-1},\ldots,c_{-n+1}$ in terms of the numbers $\dim_\C(V^G_1),\ldots,\dim_\C(V^G_n)$, solving a linear system of equations. We get e.g.\ $c_0=\dim_\C(V^G_1)+24/(n-1)$. The other $c_k$ are more complicated integral linear combinations of $\dim_\C(V^G_2),\ldots,\dim_\C(V^G_n)$.

Then the Fourier coefficients of
\begin{equation*}
\ch_{V^G}(\tau)=t_n(\tau)+c_0+c_{-1}t_n(\tau)^{-1}+\ldots+c_{-n+1}t_n(\tau)^{-n+1}+n^{\frac{11n+1}{n-1}}t_n(\tau)^{-n}
\end{equation*}
are fully determined in terms of $\dim_\C(V^G_1),\ldots,\dim_\C(V^G_n)$ and so are those of
\begin{equation*}
\ch_{V^G}(S.\tau)=t_n(S.\tau)+c_0+c_{-1}t_n(S.\tau)^{-1}+\ldots+c_{-n+1}t_n(S.\tau)^{-n+1}+n^{\frac{11n+1}{n-1}}t_n(S.\tau)^{-n}
\end{equation*}
as we explicitly know $t_n(S.\tau)$ (see proof of Lemma~\ref{lem:chvstn}).

With Lemma~\ref{lem:chsum} we compute
\begin{equation*}
\ch_{V}(\tau)+\ch_{\widetilde{V}}(\tau)=\ch_{V^G}(\tau)+\sum_{l\in\Z_n}\ch_{V^G}(ST^l.\tau)
\end{equation*}
and hence for the zeroth coefficient in the $q_\tau$-expansion we get
\begin{align}\begin{split}\label{eq:dimform2}
\dim_\C(V_1)+\dim_\C(\widetilde{V}_1)&=\left[\ch_{V^G}(\tau)+\sum_{l\in\Z_n}\ch_{V^G}(ST^l.\tau)\right](0)\\
&=\dim_\C(V^G_1)+n\left[\ch_{V^G}(S.\tau)\right](0),
\end{split}\end{align}
which is some rational linear combination of the $\dim_\C(V^G_1),\ldots,\dim_\C(V^G_n)$. More precisely the coefficient of $\dim_\C(V^G_1)$ is $n+1$ and the other coefficients have denominator at most $n^{12(n-1)/(n-1)}/n=n^{11}$.

Performing the steps described above (e.g.\ in \texttt{PARI/GP} \cite{Pari}) we obtain the assertion of the proposition.
\end{proof}

\minisec{Alternative Formula}
We also derive another dimension formula, trading the dependence on the dimensions of $V^G_2,\ldots,V^G_n$ for the dependence on the dimensions of the weight spaces of the twisted modules with weight less than one. For $n=3$ this formula is given in equation (73) of \cite{Mon94} and it is also proved in \cite{LS16b}, Theorem~4.8. We prove the general formula:
\begin{prop}[Dimension Formula II]\label{prop:dimform2}
Let $V$ and $G=\langle\sigma\rangle$ be as in Assumption~\ref{ass:op} with $\sigma$ of type $n\{0\}$ and the representations $\phi_i$ chosen as in Corollary~\ref{cor:main} and let $\widetilde{V}=\bigoplus_{i\in\Z_n}W^{(i,0)}$ be the orbifold \voa{} \eqref{eq:orbifold}. Furthermore, assume that $c=24$. Then for $n=2,3,5,7,13$
the dimension formula
\begin{equation*}
\dim_\C(V_1)+\dim_\C(\widetilde{V}_1)=24+(n+1)\dim_\C(V^G_1)-\frac{24}{n-1}\!\sum_{k=1}^{n-1}\!\sigma_1(n-k)\!\sum_{\mathclap{i\in\Z_n\setminus\{0\}}}\dim_\C(V(\sigma^i)_{k/n})
\end{equation*}
holds where $\sigma_1$ is the usual sum-of-divisors function.
\end{prop}
\begin{proof}
Like for the first dimension formula we start from
\begin{align*}
\ch_{V^G}(\tau)&=t_n(\tau)+c_0+c_{-1}t_n(\tau)^{-1}+\ldots+c_{-n+1}t_n(\tau)^{-n+1}+n^{\frac{11n+1}{n-1}}t_n(\tau)^{-n}\\
&=q_\tau^{-1}+\dim_\C(V^G_1)+\dim_\C(V^G_2)q_\tau+\ldots.
\end{align*}
However, we only use this to fix the coefficient $c_0$, for which we obtain $c_0=\dim_\C(V_1^G)+24/(n-1)$. To fix the other coefficients we then consider
\begin{align*}
\ch_{V^G}(S.\tau)&=t_n(S.\tau)+\dim_\C(V_1^G)+24/(n-1)\\
&\quad+c_{-1}t_n(S.\tau)^{-1}+\ldots+c_{-n+1}t_n(S.\tau)^{-n+1}+n^{\frac{11n+1}{n-1}}t_n(S.\tau)^{-n},
\end{align*}
which by \eqref{eq:dimform1} equals
\begin{align*}
\ch_{V^G}(S.\tau)&=\frac{1}{n}\sum_{i\in\Z_n}\ch_{V(\sigma^i)}(\tau)\\
&=\frac{1}{n}q_\tau^{-1}\sum_{k=0}^\infty\dim_\C(V_k)q_\tau^k+\frac{1}{n}q_\tau^{-1}\sum_{i\in\Z_n\setminus\{0\}}\sum_{k=0}^\infty\dim_\C(V(\sigma^i)_{k/n})q_\tau^{k/n}.
\end{align*}
By comparing the coefficients of $q_\tau^{-(n-1)/n},\ldots,q_\tau^{-1/n}$ we then obtain a linear system of $n-1$ equations, which we can solve for the $n-1$ unknowns $c_{-1},\ldots,c_{-n+1}$. Inserting the obtained values into $\ch_{V^G}(\tau)$ and applying formula \eqref{eq:dimform2} yields the assertion. The necessary calculations are performed in \texttt{PARI/GP} \cite{Pari}.
\end{proof}
The power of the second dimension formula lies in the fact that the right-hand side only depends on the dimensions of the twisted modules for weight less than one and not on the weight-one space of the twisted modules as one would trivially obtain from \eqref{eq:dimform2}. Suppose namely that $\rho(V(\sigma^i))\geq 1$ for all $i\in\Z_n\setminus\{0\}$. Then the contribution from the twisted modules to the dimension formula vanishes and one obtains the very simple formula
\begin{equation}\label{eq:dimform2simple}
\dim_\C(V_1)+\dim_\C(\widetilde{V}_1)=24+(n+1)\dim_\C(V^G_1).
\end{equation}
This occurs for example in case~2 in Section~\ref{sec:newcases}.

Without the condition on $\rho(V(\sigma^i))$ for $i\in\Z_n\setminus\{0\}$ the right-hand side is an upper bound for the left-hand side.

\part{Applications}

\chapter{Lattice \VOA{}s}\label{ch:latvoa}

In this chapter we review the theory of \voa{}s associated with positive-definite, even lattices. We describe how automorphisms of the lattice lift to automorphisms of the associated lattice \voa{}, including a description of the twisted and untwisted modules of lattice \voa{}s. We present a minor generalisation, allowing for twisted modules associated with non-standard lifts.

Lattice \voa{}s are very well-studied and well-behaved objects. They will serve as starting point for the orbifold construction in Chapter~\ref{ch:schellekens} and also play a central rôle in the BRST construction in Chapter~\ref{ch:BRST}.

\section{Lattices and Automorphisms}\label{sec:lattice}
We begin with a short review of lattices and their automorphisms (see e.g.\ \cite{CS99,Ebe13}).
\minisec{Lattices}
\begin{defi}[Rational Lattice]
Let $V$ be a finite-dimensional $\Q$-vector space endowed with a non-degenerate, symmetric $\Q$-bilinear form $\langle\cdot,\cdot\rangle\colon V\times V\to\Q$. A subset $L\subset V$ is called a \emph{(rational) lattice} of \emph{rank} $\rk(L)=\dim_\C(V)=n$ if $L=\Z v_1\oplus\ldots\oplus\Z v_n$ where $\{v_1,\ldots,v_n\}$ is a $\Q$-basis of $V$. In other words: $L$ is a discrete subgroup of $V$ of full rank, i.e.\ $L\otimes_\Z\Q=V$.
\end{defi}
The \emph{norm} of an element $\alpha\in L$ is defined by $\langle\alpha,\alpha\rangle/2$ and can be non-positive.\footnote{Often the quantity $\langle\alpha,\alpha\rangle$ is called the norm of $\alpha\in L$.} For the purposes of this text we are mainly interested in positive-definite and even lattices:
\begin{defi}[Positive Definiteness]
A lattice $L$ is called \emph{positive-definite} if the $\Q$-bilinear form $\langle\cdot,\cdot\rangle$ on the underlying vector space $V=L\otimes_\Z\Q$ is positive-definite.
\end{defi}
\begin{defi}[Integral and Even Lattices]
Let $L$ be a lattice. Then $L$ is called \emph{integral} if $\langle\alpha,\beta\rangle\in\Z$ for all $\alpha,\beta\in L$. If $\langle\alpha,\alpha\rangle\in 2\Z$ for all $\alpha\in L$, the lattice $L$ is called \emph{even}.
\end{defi}
Every even lattice is integral.

\minisec{Lattice Automorphisms}
\begin{defi}[Lattice Automorphism]
An \emph{automorphism} (or \emph{isometry}) $\nu$ of a lattice $L$ is an automorphism $\nu$ of the underlying vector space $V=L\otimes_\Z\Q$ such that $\nu(L)=L$ and the bilinear form on $V$ is invariant under $\nu$, i.e.\ $\langle \nu x,\nu y\rangle=\langle x,y\rangle$ for all $x,y\in V$.
\end{defi}
Let $\Aut(L)$ denote the group of lattice automorphisms (sometimes denoted by $O(L)$). We will only consider lattice automorphisms of finite order. For a positive-definite lattice the automorphism group $\Aut(L)$ is finite and in particular all lattice automorphisms have finite order.

\minisec{Complexified Lattice}

We can embed a lattice $L$ into its complexification, the $\C$-vector space $\h:=L\otimes_{\Z}\C=V\otimes_{\Q}\C$. The symmetric bilinear form $\langle\cdot,\cdot\rangle$ on $V$ can be extended by $\C$-bilinearity to a symmetric, non-degenerate, $\nu$-invariant bilinear form on $\h$.

Given an automorphism $\nu$ of the lattice $L$, we can extend $\nu$ by $\C$-linearity to an isometry $\nu$ of $\h$. Let
$m:=\ord(\nu)$ denote the order of the lattice automorphism $\nu$, i.e.\ the smallest $m\in\Ns$ such that $\nu^m=\id$.
Since $\C$ is algebraically closed, $\nu$ is diagonalisable with the possible eigenvalues being $m$-th roots of unity. Hence $\h$ decomposes as
\begin{equation*}
\h=\bigoplus_{i\in\Z_m}\h_{(i)}
\end{equation*}
with the eigenspaces
\begin{equation*}
\h_{(i)}:=\{h\in\h\;|\;\nu h=\xi_m^{i}h\}
\end{equation*}
for $i\in\Z_m$. Moreover, since $\langle\cdot,\cdot\rangle$ is $\nu$-invariant, a simple calculation shows that the eigenspace $\h_{(0)}$ is orthogonal to all other $\h_{(i)}$ with respect to $\langle\cdot,\cdot\rangle$, i.e.\
\begin{equation*}
\langle\h_{(0)},\h_{(i)}\rangle=0
\end{equation*}
for all $i\in\Z_m\setminus\{0\}$. Hence, $\h$ decomposes into orthogonal subspaces $\h=\h_{(0)}\oplus\h_{(0)}^\bot$ with
\begin{equation*}
h_{(0)}^\bot=\h_{(1)}\oplus\ldots\oplus\h_{(m-1)}.
\end{equation*}

There is a natural projection $\pi_\nu\colon\h\to\h_{(0)}$ given by
\begin{equation*}
\pi_\nu(h)=\frac{1}{m}\sum_{i=0}^{m-1}\nu^i h
\end{equation*}
for $h\in\h$. The kernel of this projection is $\h_{(0)}^\bot$, which means that this projection is orthogonal with respect to the bilinear form $\langle\cdot,\cdot\rangle$. We denote the complementary projection $\h\to\h_{(0)}^\bot$ by $\pi_\nu^\bot=\id-\pi_\nu$.

Clearly, the linear operator $(\id-\nu)$ restricts to an invertible map on $\h_{(0)}^\bot$ while it restricts to the zero map on $\h_{(0)}$.

\minisec{Fixed-Point Sublattice}

For an automorphism $\nu\in\Aut(L)$ we denote by
\begin{equation*}
L^\nu=\{\alpha\in L\;|\;\nu\alpha=\alpha\}=L\cap\h_{(0)}=L\cap\pi_\nu(L)
\end{equation*}
the \emph{$\nu$-invariant sublattice} or \emph{fixed-point sublattice} of $L$ under $\nu$, which is a sublattice of $L$. In particular, the restriction of the bilinear form to $L^\nu$ is non-degenerate. Then $L^\nu\otimes_\Z\C=\h_{(0)}$, i.e.\ $\rk(L^\nu)=\dim_\C(\h_{(0)})$.

In general, for a sublattice $\Lambda$ of a lattice $L$ we denote by $\Lambda^\bot$ the \emph{orthogonal complement} of $\Lambda$ in $L$,
\begin{equation*}
\Lambda^\bot:=\{\alpha\in L\;|\;\langle\alpha,\beta\rangle=0\text{ for all }\beta\in\Lambda\},
\end{equation*}
another sublattice of $L$.

The orthogonal complement of the $\nu$-invariant sublattice $L^\nu$,
\begin{equation*}
L_\nu:=(L^\nu)^\bot=L\cap\h_{(0)}^\bot=L\cap\pi_\nu^\bot(L).
\end{equation*}
is called \emph{$\nu$-coinvariant sublattice} of $L$. Then $L_\nu\otimes_\Z\C=\h_{(0)}^\bot$, i.e.\ $\rk(L_\nu)=\dim_\C(\h_{(0)}^\bot)$. The direct sum $L^\nu\oplus L_\nu$ is a sublattice of $L$ (of full rank), i.e.\
\begin{equation*}
L^\nu\oplus L_\nu\subseteq L.
\end{equation*}

\minisec{Dual Lattice and Discriminant Form}

\begin{defi}[Dual Lattice, Unimodular Lattice]
Let $L$ be a rational lattice (in the $\Q$-vector space $V$). Then the \emph{dual lattice} of $L$ is defined as
\begin{equation*}
L'=\{\alpha\in V\;|\;\langle\alpha,\beta\rangle\in\Z\text{ for all }\beta\in L\}.
\end{equation*}
A lattice $L$ is called unimodular if $L'=L$.
\end{defi}
The lattice $L$ is integral if and only if $L\subseteq L'$. Thus, every unimodular lattice is integral. Moreover, $(L')'=L$ by definition.
\begin{defi}[Level]
Let $L$ be an even lattice. The \emph{level} of $L$ is defined as the smallest $N\in\Ns$ such that $N\langle\alpha,\alpha\rangle\in 2\Z$ for all $\alpha\in L'$.
\end{defi}

For an integral lattice $L$ we can study the quotient $L'/L$, which is a finite abelian group, called \emph{discriminant group}. If the lattice $L$ is even, then the quotient $L'/L$ comes equipped with a non-degenerate finite quadratic form $Q_L\colon L'/L\to\Q/\Z$, naturally given by
\begin{equation*}
\alpha+L\mapsto Q_L(\alpha)=\langle\alpha,\alpha\rangle/2+\Z.
\end{equation*}
A finite abelian group endowed with a non-degenerate quadratic form is called \fqs{} (see Section~\ref{sec:qf}). We call the \fqs{} $L'/L=(L'/L,Q_L)$ the \emph{discriminant form} of $L$.

An automorphism of a \fqs{} is an automorphism of the abelian group preserving the quadratic form. For a \fqs{} $D$ let $\Aut(D)$ denote the group of automorphisms of $D$. Given an even lattice $L$, there is a natural homomorphism
\begin{equation*}
\Aut(L)\to\Aut(L'/L)
\end{equation*}
mapping $\nu\in\Aut(L)$ to
\begin{equation*}
\alpha+L\mapsto\nu\alpha+L,
\end{equation*}
$\alpha+L\in L'/L$, which is well-defined. In the following we present a weakened version of a criterion due to Nikulin, which ensures the surjectivity of this homomorphism under certain assumptions. For a prime $p$ and a \fqs{} $D$ let $D_p$ denote the $p$-component of $D$.\footnote{As a group $D_p$ is just the Sylow $p$-subgroup of the abelian group $D$.} For a finite abelian group $D$ we denote by $l(D)$ the minimum number of generators of $D$. We call $l(D_p)$ for a prime $p$ the \emph{$p$-rank} of $D$.
\begin{prop}[\cite{Nik80}, Theorem~1.14.2]\label{prop:thm1.14.2}
Let $L$ be an even, indefinite lattice with $\rk(L)\geq 3$. Suppose that
\begin{equation*}
\rk(L)\geq l((L'/L)_p)+2
\end{equation*}
for all primes $p$. Then the natural map $\Aut(L)\to\Aut(L'/L)$ is surjective.
\end{prop}

\minisec{Inclusions}
It follows from Lemma~4.6 in \cite{BK04} that
\begin{equation}\label{eq:lem4.6}
\pi_\nu(L')=(L^\nu)'\quad\text{and}\quad\pi_\nu^\bot(L')=(L_\nu)'
\end{equation}
for any lattice $L$. If $L$ is integral, then it is clear that
\begin{equation*}
\pi_\nu(L)\subseteq (L^\nu)'\quad\text{and}\quad\pi_\nu^\bot(L)\subseteq(L_\nu)'
\end{equation*}
and hence for a unimodular lattice
\begin{equation*}
\pi_\nu(L)=(L^\nu)'\quad\text{and}\quad\pi_\nu^\bot(L)=(L_\nu)'.
\end{equation*}

\minisec{Primitive Sublattices}
Let $L$ be a lattice. A sublattice $\Lambda$ is called \emph{primitive} if one of the following equivalent conditions holds:
\begin{enumerate}
\item $L/\Lambda$ is a free $\Z$-module.
\item Every $\Z$-basis of $\Lambda$ can be extended to a $\Z$-basis of $L$.
\item $L$ can be written as the direct sum $L=\Lambda\oplus\widetilde\Lambda$ for some sublattice $\widetilde\Lambda$ of $L$.
\item $\Lambda=\{\alpha\in L\;|\;\alpha\in\Q\Lambda\}$.
\end{enumerate}
If $\Lambda$ is a primitive sublattice of $L$, then so is its orthogonal complement $\Lambda^\bot$. Moreover, fixed-point sublattices $L^\nu$ are always primitive.

\begin{prop}
Let $L$ be a unimodular (and hence integral) lattice and $\Lambda$ a primitive sublattice of $L$. Then there are natural isomorphisms
\begin{equation*}
\Lambda'/\Lambda\cong L/(\Lambda\oplus\Lambda^\bot)\cong(\Lambda^\bot)'/\Lambda^\bot.
\end{equation*}
\end{prop}
\begin{proof}[Idea of Proof]
This result can be obtained by studying the proof of Proposition~1.2 in \cite{Ebe13}.\footnote{Note that the proof is faulty in earlier editions of the book.}
\end{proof}
We only give explicit formulæ for the isomorphisms in the special case described in the next proposition. If we assume that $L$ is even, then so are all its sublattices $\Lambda$ and we can consider their discriminant forms $(\Lambda'/\Lambda,Q_\Lambda)$ as \fqs{}s.
\begin{prop}\label{prop:isomnegqf}
Let $L$ be a unimodular, even lattice and let $\nu\in\Aut(L)$. Consider the fixed-point sublattice $L^\nu$ and its orthogonal complement $L_\nu$ in $L$. Then there is a natural isomorphism of \fqs{}s
\begin{equation*}
\psi\colon\left((L^\nu)'/L^\nu,Q_{L^\nu}\right)\to\left((L_\nu)'/L_\nu,-Q_{L_\nu}\right).
\end{equation*}
More specifically, $\pi_\nu(\alpha)+L^\nu\mapsto\pi_\nu^\bot(\alpha)+L_\nu$. Recall that $\pi_\nu(L)=(L^\nu)'$ and $\pi_\nu^\bot(L)=(L_\nu)'$ since $L$ is unimodular.
\end{prop}
\begin{proof}
The fixed-point sublattice $L^\nu$ is a primitive sublattice of $L$ and so is $L_\nu$. In the proof of Proposition~1.2 in \cite{Ebe13} it is shown that there are isomorphisms of abelian groups
\begin{align*}
L/(L^\nu\oplus L_\nu)\to (L^\nu)'/L^\nu&:\alpha+L^\nu\oplus L_\nu\mapsto\pi_\nu(\alpha)+L^\nu,\\
L/(L^\nu\oplus L_\nu)\to (L_\nu)'/L_\nu&:\alpha+L^\nu\oplus L_\nu\mapsto\pi_\nu^\bot(\alpha)+L_\nu.
\end{align*}
This gives the group isomorphism $\psi\colon(L^\nu)'/L^\nu\to (L_\nu)'/L_\nu$. In order to get an isomorphism of \fqs{}s we observe that
\begin{align*}
0+\Z&=\langle\alpha,\alpha\rangle/2+\Z=\langle\pi_\nu(\alpha)+\pi_\nu^\bot(\alpha),\pi_\nu(\alpha)+\pi_\nu^\bot(\alpha)\rangle/2+\Z\\
&=\langle\pi_\nu(\alpha),\pi_\nu(\alpha)\rangle/2+\langle\pi_\nu(\alpha),\pi_\nu^\bot(\alpha)\rangle+\langle\pi_\nu^\bot(\alpha),\pi_\nu^\bot(\alpha)\rangle/2+\Z\\
&=\langle\pi_\nu(\alpha),\pi_\nu(\alpha)\rangle/2+\langle\pi_\nu^\bot(\alpha),\pi_\nu^\bot(\alpha)\rangle/2+\Z\\
&=Q_{L^\nu}(\pi_\nu(\alpha)+L^\nu)+Q_{L_\nu}(\pi_\nu^\bot(\alpha)+L_\nu)
\end{align*}
and hence
\begin{equation*}
Q_{L^\nu}(\gamma)=-Q_{L_\nu}(\psi(\gamma))
\end{equation*}
for all $\gamma\in(L^\nu)'/L^\nu$.
\end{proof}

\minisec{Rescaled Lattices}

Let $L=(L,\langle\cdot,\cdot\rangle)$ be some rational lattice with ambient space $V=L\otimes_\Z\Q$.
\begin{defi}[Rescaled Lattice]
For $m\in\Q$ we define the \emph{rescaled lattice} $L(m)$ as the same lattice $L$ with the quadratic form rescaled by a factor of $m$, i.e.\
\begin{equation*}
L(m):=\left(L,m\langle\cdot,\cdot\rangle\right).
\end{equation*}
\end{defi}
Alternatively one could define $\sqrt{m}L:=\left(\sqrt{m}L,\langle\cdot,\cdot\rangle\right)$, which has the ambient $\Q$-vector space $\sqrt{m}V$. Then $L(m)\cong \sqrt{m}L$ as lattices.

Clearly, the dual lattice of a rescaled lattice is itself a rescaled version of the dual lattice. More precisely:
\begin{prop}\label{prop:rescaled}
Let $L$ be a lattice and $m\in\Q$. Then the dual lattice of $L(m)=\left(L,m\langle\cdot,\cdot\rangle\right)$ is
\begin{equation*}
(L(m))'=\left((1/m)L',m\langle\cdot,\cdot\rangle\right)\cong L'(1/m)
\end{equation*}
or in terms of $\sqrt{m}L=\left(\sqrt{m}L,\langle\cdot,\cdot\rangle\right)$,
\begin{equation*}
(\sqrt{m}L)'=(1/\sqrt{m})L'.
\end{equation*}
\end{prop}

\minisec{Roots}
An overview of the following definitions can be found for example in \cite{Bor87,Bor90b,Sch06}.
\begin{defi}[Primitive Vector, Root]
Let $L$ be a lattice. A vector $\alpha\in L$ is called \emph{primitive} if $\alpha/n\notin L$ for any $n\in\Ns$. A primitive vector $\alpha\in L$ is called \emph{root} if the norm $\langle\alpha,\alpha\rangle/2>0$ and the reflection $\sigma_\alpha$ through the hyperplane orthogonal to $\alpha$,
\begin{equation*}
\sigma_\alpha(\beta)=\beta-2\frac{\langle\beta,\alpha\rangle}{\langle\alpha,\alpha\rangle}\alpha
\end{equation*}
for all $\beta\in L$, is an automorphism of $L$, i.e.\ maps $L$ into itself.
\end{defi}
Let $\Phi(L)$ denote the set of roots of $L$ and $Q=Q(L)$ the \emph{root sublattice} $Q(L):=\Z\Phi(L)$ of $L$. The subgroup of $\Aut(L)$ generated by the reflections $\sigma_\alpha$ in the roots $\alpha\in L$ is called \emph{reflection group} of $L$. It is an easy consequence of the definition that for any root $\alpha\in L$ the quotient $2\langle\alpha,v\rangle/\langle\alpha,\alpha\rangle\in\Z$ for all $v\in L$.

Let $L$ be an integral lattice. Then any vector $\alpha$ in $L$ of norm $1/2$ or $1$ is a root. If $L$ is integral and unimodular, then these are already all the roots of $L$ and they are called \emph{short} and \emph{long roots} depending on whether their norm is $1/2$ or $1$, respectively. For an even, unimodular lattice the roots are exactly the vectors of norm $1$.

\minisec{Special Lattices}

For $m,n\in\N$ consider the unimodular, even lattice of signature $(m,n)$. We say such a lattice has \emph{genus} $\II_{m,n}$. A lattice of genus $\II_{m,n}$ exists if and only if $8\mid m-n$.\footnote{Note that sometimes (even in this text) the symbol $\II_{m,n}$ is also used to refer to a certain lattice (and not just its genus) for each $m,n\in\Ns$ with $8\mid m-n$. These are exactly the lattices used to prove the existence part of this statement.} Moreover, if the signature is indefinite, i.e.\ $m,n>0$, then there is only one lattice up to isomorphism in the genus $\II_{m,n}$.

For positive-definite lattices, i.e.\ lattices with signature $(m,0)$ and $8\mid m$ the situation is more complicated. There is exactly one lattice of genus $\II_{8,0}$, i.e.\ a positive-definite, even, unimodular lattice of rank $8$. This lattice is called $E_8$. In genus $\II_{16,0}$ there are two lattices, namely $E_8^2$ and $D_{16}^+$.

The most interesting situation occurs in genus $\II_{24,0}$. There are exactly 24 positive-definite, even, unimodular lattices of rank 24, called the \emph{Niemeier lattices}. They are generally denoted by $N(Q)$ where $Q$ is the root lattice of $N(Q)$, i.e.\ the sublattice generated by the roots of $N(Q)$. There is one Niemeier lattice with no roots, called the \emph{Leech lattice} and usually denoted by $\Lambda$.

In larger dimensions, there are extremely many positive-definite, even, unimodular lattices, e.g.\ more than $1\,160\,000\,000$ of genus $\II_{32,0}$.

An even lattice $L$ is unimodular if and only if its discriminant form $L'/L$, which is a \fqs{}, is trivial. For a non-trivial discriminant form $D=L'/L$ we define the symbol $\II_{m,n}(D)$ for the genus of $L$ where the genus only depends on the isomorphism class of the \fqs{} $D$. We usually write $D$ in terms of its Jordan decomposition, i.e.\ for instance $\II_{6,2}(3^{+4} 5^{-4})$.

\section{Lattice \VOA{}s}\label{sec:latvoa}
The construction of the vertex algebra $V_L$ associated with an even lattice $L$ is described in \cite{Bor86,FLM88,Don93} (see also \cite{Kac98,LL04}). We only describe the case of a positive-definite lattice.

\minisec{Twisted Group Algebra}
Let $L$ be a positive-definite, even lattice. There exists a 2-cocycle $\eps\colon L\times L\to \{\pm1\}$ satisfying
\begin{equation}\label{eq:multcommcond}
\eps(\alpha,\alpha)=(-1)^{\langle\alpha,\alpha\rangle/2}\quad\text{and}\quad\eps(\alpha,\beta)/\eps(\beta,\alpha)=(-1)^{\langle\alpha,\beta\rangle}
\end{equation}
for all $\alpha,\beta\in L$. These relations imply that $\eps$ is normalised, i.e.\ $\eps(0,\alpha)=\eps(\alpha,0)=1$ for all $\alpha\in L$. The condition on $\eps(\alpha,\alpha)$ is not essential. Some authors prefer different conventions (cf.\ \cite{Kac98}).
The second condition is essential and means the alternating $\Z$-bilinear form $(-1)^{\langle\cdot,\cdot\rangle}$ is the \emph{skew} of $\eps$. This determines $\eps$ up to a 2-coboundary.
We consider the \emph{twisted group algebra} $\C_\eps[L]$, which is spanned by the $\C$-basis $\{\ee_\alpha\}_{\alpha\in L}$ and with multiplication defined via
\begin{equation*}
\ee_\alpha\ee_\beta=\eps(\alpha,\beta)\ee_{\alpha+\beta}
\end{equation*}
for $\alpha,\beta\in L$. It is $\Z$-graded by weights via $\wt(\ee_\alpha):=\langle\alpha,\alpha\rangle/2$.

\minisec{Heisenberg Module}

We regard the complexified lattice $\h=L\otimes_\Z\C$ as an abelian Lie algebra and define the affine Lie algebra associated with $\h$ as the Lie algebra
\begin{equation*}
\hat{\h }:=\left(\h\otimes_\C\C[t,t^{-1}]\right)\oplus\C\mathbf{k},
\end{equation*}
called \emph{Heisenberg current algebra} in \cite{BK04}, with the Lie bracket defined by linear continuation of
\begin{equation*}
[x(n),y(n')]=\langle x,y\rangle n\delta_{n+n',0}\mathbf{k}\quad\text{and}\quad[u,\mathbf{k}]=0
\end{equation*}
for $x,y\in\h$, $n,n'\in\Z$ and $u\in\hat{\h}$ where we introduce the shorthand notation $h(n):=h\otimes t^n$ for $h\in\h$ and $n\in\Z$.

We decompose $\hat\h$ as
\begin{equation*}
\hat\h=\hat\h_-\oplus\hat\h_0\oplus\hat\h_+\oplus\C\mathbf{k}
\end{equation*}
where
\begin{equation*}
\hat\h_-:=\h\otimes_\C t^{-1}\C[t^{-1}],\quad\hat\h_0:=\h\otimes_\C \C t^0\quad\text{and}\quad\hat\h_+:=\h\otimes_\C t\C[t].
\end{equation*}
We can also consider the Lie subalgebra $\hat{\h}_{\neq 0}:=\hat\h_-\oplus\hat\h_+\oplus\C\mathbf{k}$ of $\hat{\h}$. It is a Heisenberg Lie algebra in the sense that its commutator subalgebra equals its centre and is one-dimensional.

We define Lie subalgebra $\hat{\mathfrak{b}}:=\hat\h_0\oplus\hat\h_+\oplus\C\mathbf{k}=\h\otimes_\C\C[t]\oplus\C\mathbf{k}$ and let it act on the one-dimensional vector space $\C$ as
\begin{equation*}
h(n)\cdot w=0\quad\text{and}\quad\mathbf{k}\cdot w=w
\end{equation*}
for all $w\in\C$, $h\in\h$ and $n\in\N$. Then we define the induced $\hat\h$-module
\begin{equation*}
M_{\hat\h}(1):=\Ind_{\hat{\mathfrak{b}}}^{\hat\h}\C=U(\hat\h)\otimes_{U(\hat{\mathfrak{b}})}\C
\end{equation*}
where $U(\cdot)$ denotes the universal enveloping algebra. $M_{\hat\h}(1)$ is irreducible as module for $\hat{\h}$ (even as module for $\hat{\h}_{\neq0}$). As a vector space $M_{\hat\h}(1)$ is isomorphic to the symmetric algebra $S(\hat\h_-)$ of $\hat\h_-$. In particular, $M_{\hat\h}(1)$ is spanned by elements of the form
\begin{equation*}
h_k(-n_k)\ldots h_1(-n_1)1
\end{equation*}
for $k\in\N$, $h_1,\ldots,h_k\in\h$ non-zero and $n_1,\ldots,n_k\in\Ns$ (and $1\in\C$). We introduce a $\Z$-grading on $M_{\hat\h}(1)$ such that the above element has weight $n_1+\ldots+n_k\in\N$.

More generally, one often writes $M_{\hat\h}(l)=M_{\hat\h}(l,0)=V_{\hat\h}(l,0)$ for $l\in\C^\times$, which admits a \voa{} structure and is called the \emph{Heisenberg \voa{}} associated with $\hat\h$ of level $l$ (see e.g.\ \cite{LL04}, Section~6.3).

\minisec{Lattice \VOA{}}
We define
\begin{equation*}
V_L:=M_{\hat\h}(1)\otimes\C_\eps[L]
\end{equation*}
as graded vector space. Then $V_L$ is spanned by elements of the form 
\begin{equation*}
h_k(-n_k)\ldots h_1(-n_1)1\otimes\ee_\alpha
\end{equation*}
for $k\in\N$, $h_1,\ldots,h_k\in\h$ non-zero, $n_1,\ldots,n_k\in\Ns$, $\alpha\in L$. The weight of such an element is given by
\begin{equation*}
n_1+\ldots+n_k+\frac{1}{2}\langle\alpha,\alpha\rangle\in\N.
\end{equation*}

The Lie algebra $\hat\h$ acts on the module $M_{\hat\h}(1)$ and also on $\C_\eps[L]$ by letting $\hat\h_{\neq0}$ act trivially and $\hat{\h}_0=\h\otimes_\C\C t^0\cong\h$ as 
\begin{equation*}
h\cdot\ee_\alpha=\langle h,\alpha\rangle\ee_\alpha
\end{equation*}
for $h\in\h$, $\alpha\in L$. Then $\hat\h$ acts on $V_L$ in the tensor-product representation.

We define the vertex operators
\begin{equation*}
Y(h(-1)1\otimes\ee_0,x):=h(x):=\sum_{n\in\Z}h(n)x^{-n-1}
\end{equation*}
for $h\in\h$ and
\begin{equation*}
Y(1\otimes\ee_\alpha,x):=Y_\alpha(x):=\ee_\alpha x^\alpha\exp\left(\sum_{n=1}^\infty\alpha(-n)\frac{x^n}{n}\right)\exp\left(\sum_{n=1}^\infty\alpha(n)\frac{x^{-n}}{-n}\right)
\end{equation*}
for $\alpha\in L$ where $x^\alpha$ acts as $x^{\langle\alpha,\beta\rangle}$ on $\ee_\beta$, $\alpha,\beta\in L$.

\begin{thm}[\cite{FLM88}, Theorem~8.10.2, \cite{Bor86}]\label{thm:latticevoa}
The vertex operators $h(x)$, $h\in\h$, and $Y_\alpha(x)$, $\alpha\in L$, generate a \voa{} structure of central charge $c=\rk(L)$ on $V_L=M_{\hat\h}(1)\otimes\C_\eps[L]$ with vacuum and Virasoro vectors
\begin{equation*}
\vac=1\otimes\ee_0\quad\text{and}\quad \omega=\frac{1}{2}\sum_{i=1}^{\rk(L)}a_i(-1)a_i(-1)1\otimes\ee_0
\end{equation*}
for an orthonormal basis $\{a_i\}_{i=1}^n$ of $\h$.
\end{thm}
We remark that two different choices of the normalised 2-cocycle $\eps$ satisfying the second part of \eqref{eq:multcommcond} give isomorphic \voa{} structures.

\Voa{}s associated with lattices have a number of nice properties:
\begin{prop}\label{prop:latnice}
Let $L$ be a positive-definite, even lattice and let $V_L$ be the lattice \voa{} associated with $L$. Then $V_L$ satisfies Assumption~\ref{ass:n}.
\end{prop}
\begin{proof}
It follows from the classification of irreducible $V_L$-modules in \cite{Don93} that $V_L$ is irreducible as $V_L$-module and hence simple and also self-contragredient. That $V_L$ is of CFT-type is evident from the construction. The rationality of $V_L$ is also proved in \cite{Don93}. It is shown in \cite{DLM00}, Proposition~12.5, that $V_L$ is $C_2$-cofinite. The regularity of $V_L$ is shown in \cite{DLM97}, Theorem~3.16, which also implies both rationality and $C_2$-cofiniteness (see Remark~\ref{rem:regrat}). 
\end{proof}

\minisec{Classification of Irreducible Modules}

In the following we describe the irreducible modules of the \voa{} $V_L$ (see \cite{Don93}, \cite{LL04}, Sections 6.4 and 6.5). Given the even and hence integral lattice $L$, consider the dual lattice $L'$, which is a rational lattice, in general neither even nor integral. Consider a normalised $2$-cocycle $\eps\colon L'\times L'\to\C^\times$ such that
\begin{equation}\label{eq:2cocycleirrmod}
\eps(\alpha,\beta)/\eps(\beta,\alpha)=(-1)^{\langle\alpha,\beta\rangle}\quad\text{for all}\quad\alpha,\beta\in L.
\end{equation}
We may assume the values of $\eps$ to lie in a finite, cyclic subgroup of $\C^\times$.

Exactly as before we construct
\begin{equation*}
A_{L'}:=M_{\hat\h}(1)\otimes\C_\eps[L'],
\end{equation*}
which now is $(1/N)\Z$-graded where $N$ is the level of the lattice $L$. The vertex operators $h(x)$ and $Y_\alpha(x)$ for $h\in\h$ and $\alpha\in L$ endow $A_{L'}$ with the structure of a $V_L$-module.

Let us decompose
\begin{equation*}
A_{L'}=\bigoplus_{\lambda+L\in L'/L}V_{\lambda+L}\quad\text{with}\quad V_{\lambda+L}=M_{\hat\h}(1)\otimes\C_\eps[\lambda+L].
\end{equation*}
We write $\C_\eps[\lambda+L]$ for the subspace of $\C_\eps[L']$ with basis $\{\ee_\alpha\}_{\alpha\in \lambda+L}$, which is not an algebra any longer but acted on by $\C_\eps[L]$. Then the $V_{\lambda+L}$ are irreducible modules for $V_L$ of conformal weight $\langle\lambda,\lambda\rangle/2$.

Dong showed that these are all irreducible $V_L$-modules up to isomorphism:
\begin{thm}[\cite{Don93}, Theorem~3.1]\label{thm:3.1}
Let $L$ be a positive-definite, even lattice and let $V_L$ be the lattice \voa{} associated with $L$. The isomorphism classes of irreducible $V_L$-modules can be parametrised by the elements of $L'/L$: every irreducible $V_L$-module is isomorphic to one of the modules $V_{\lambda+L}$ for $\lambda+L\in L'/L$ and $V_{0+L}\cong V_L$ as modules.
\end{thm}

\begin{cor}\label{cor:hol}
Let $L$ be a positive-definite, even lattice which is also unimodular, i.e.\ $L'=L$. Then the associated \voa{} is holomorphic.
\end{cor}

The fusion rules between the irreducible $V_L$-modules are determined in \cite{DL93}.
\begin{prop}[\cite{DL93}, Corollary~12.10]\label{prop:cor12.10}
The fusion product between the irreducible modules of a lattice \voa{} $V_L$ for a positive-definite, even lattice $L$ is given by
\begin{equation*}
V_{\lambda+L}\boxtimes V_{\mu+L}\cong V_{\lambda+\mu+L}
\end{equation*}
for $\lambda+L,\mu+L\in L'/L$, i.e.\ the fusion algebra $\V(V_L)$ of $V_L$ is the group algebra $\C[L'/L]$ and each irreducible $V_L$-module is a simple current.
\end{prop}
The discriminant form $L'/L$ of a positive-definite, even lattice comes equipped with a non-degenerate finite quadratic form $Q_L$ defined by $Q_L(\lambda+L)=\langle\lambda,\lambda\rangle/2+\Z$ for $\lambda+L\in L'/L$. This is exactly the quadratic form $Q_\rho$ on the fusion group $F_{V_L}=L'/L$.

In total, we see that for a positive-definite, even lattice $L$ the associated lattice \voa{} $V_L$ satisfies Assumption~\ref{ass:sn} (group-like fusion) with fusion group $(F_{V_L},Q_\rho)=(L'/L,Q_L)$ and Assumption~\ref{ass:p}.

In order to determine the fusion algebra of $V_L$, the authors had to give intertwining operators of type $\binom{V_{\lambda+\mu+L}}{V_{\lambda+L}\,V_{\mu+L}}$ for $\lambda+L,\mu+L\in L'/L$. These also define an \aia{} structure on $A_{L'}$.
\begin{thm}[\cite{DL93}, Theorem~12.24]\label{thm:lataia}
Let $L$ be a positive-definite, even lattice and $V_L$ the associated lattice \voa{}. Then the direct sum of all irreducible $V_L$-modules up to isomorphism
\begin{equation*}
A_{L'}:=\bigoplus_{\lambda+L\in L'/L}V_{\lambda+L}
\end{equation*}
carries the structure of an \aia{} associated with some abelian 3-cocycle $(F,\Omega)$ with associated quadratic space $(L'/L,Q_\Omega)=\overline{L'/L}=(L'/L,-Q_L)$.
\end{thm}

\section{Automorphisms of Lattice \VOA{}s}\label{sec:autlatvoa}

Lattice \voa{}s have a well-understood class of automorphisms, namely those arising from automorphisms of the underlying lattice. The main sources for this section are \cite{Lep85,FLM88,DL96,BK04}. Let $\nu\in\Aut(L)$ be an automorphism of the positive-definite, even lattice $L$. Consider the alternating $\Z$-bilinear form $(-1)^{\langle\cdot,\cdot\rangle}$. It is $\nu$-invariant since $\langle\nu\alpha,\nu\beta\rangle=\langle\alpha,\beta\rangle$. But then the skew of both $\eps(\alpha,\beta)$ and $\eps(\nu\alpha,\nu\beta)$ is the alternating $\Z$-bilinear form (see second part of \eqref{eq:multcommcond}), which means that they are in the same cohomology class of 2-cocycles. Hence, there is a function $u\colon L\to\{\pm1\}$ such that
\begin{equation*}
\frac{\eps(\alpha,\beta)}{\eps(\nu\alpha,\nu\beta)}=\frac{u(\alpha)u(\beta)}{u(\alpha+\beta)}
\end{equation*}
for all $\alpha,\beta\in L$. We could also, more generally, assume that the values of $u$ lie in some finite, cyclic subgroup of $\C^\times$.

We lift $\nu$ to an automorphism $\hat\nu$ of the twisted group algebra $\C_\eps[L]$ via
\begin{equation*}
\hat\nu(\ee_\alpha)=u(\alpha)\ee_{\nu\alpha}
\end{equation*}
so that $\hat\nu(\ee_\alpha\ee_\beta)=\hat\nu(\ee_\alpha)\hat\nu(\ee_\beta)$ for all $\alpha,\beta\in L$. Moreover, $\nu$ acts naturally on $M_{\hat\h}(1)$ via
\begin{equation*}
\nu(h(n))=(\nu h)(n)
\end{equation*}
for $h\in\h$ and $n\in\Ns$. We then let the tensor-product operator $\nu\otimes\hat\nu$ act on $V_L=M_{\hat\h}(1)\otimes\C_\eps[L]$ (as vector space) and denote this operator in the following by $\hat\nu$.
\begin{prop}[\cite{DL96}, Section~7]
Let $L$ be a positive-definite, even lattice and $V_L$ the associated lattice \voa{}. Let $\nu\in\Aut(L)$. Then the above defined lift $\hat\nu$ is a \voa{} automorphism of $V_L$. In particular, $\hat{\nu}\vac=\vac$ and $\hat{\nu}\omega=\omega$.
\end{prop}

\begin{defi}[Standard Lift, \cite{Lep85}, Section~5]\label{defi:standard}
Given an automorphism $\nu\in\Aut(L)$ the function $u$ can always be chosen such that
\begin{equation*}
u(\alpha)=1\quad\text{for all}\quad\alpha\in L^\nu
\end{equation*}
where $L^\nu$ denotes the fixed-point sublattice of $L$ under $\nu$, i.e.\ $\hat\nu(\ee_\alpha)=\ee_\alpha$ if and only if $\nu\alpha=\alpha$ for all $\alpha\in L$. We call such a lift a \emph{standard lift}.
\end{defi}

\minisec{Powers of Lifted Automorphisms}
In the following we describe powers $\hat\nu^k$ of the lift $\hat\nu$ for $k\in\N$. Clearly, $\hat\nu^k$ is a lift of $\nu^k$. More precisely, if the lift $\hat\nu$ of $\nu$ is determined by the function $u\colon L\to\{\pm1\}$, i.e.\
\begin{equation*}
\hat\nu(\ee_\alpha)=u(\alpha)\ee_{\nu\alpha}
\end{equation*}
for all $\alpha\in L$, then $\hat\nu^k$ acts as
\begin{equation*}
\hat\nu^k(\ee_\alpha)=u(\alpha)u(\nu\alpha)\ldots u(\nu^{k-1}\alpha)\ee_{\nu^k\alpha}
\end{equation*}
for all $\alpha\in L$, i.e.\ the lift $\hat\nu^k$ of $\nu^k$ is determined by the function $w\colon L\to\{\pm1\}$ given by $w(\alpha)=u(\alpha)u(\nu\alpha)\ldots u(\nu^{k-1}\alpha)$.

It is natural to ask how the property of $\hat\nu$ being a standard lift carries over to $\hat\nu^k$. We will see that if $\hat\nu$ is a standard lift, then $\hat\nu^k$ does not always have this property.

Recall that $m:=\ord(\nu)$ denotes the (exact) order of the lattice automorphism $\nu$, i.e.\ the smallest $m\in\Ns$ such that $\nu^m=\id$.
\begin{prop}\label{prop:standardliftpower}
Let $\nu\in\Aut(L)$ and $\hat\nu$ a standard lift of $\nu$, i.e.\ $\hat\nu(\ee_\alpha)=\ee_\alpha$ for all $\alpha\in L^\nu$. Then for $k\in\N$,
\begin{equation*}
\hat{\nu}^k(\ee_\alpha)=\ee_\alpha\cdot\begin{cases}1&\text{if $m$ or $k$ is odd,}\\(-1)^{\langle\alpha,\nu^{k/2}\alpha\rangle}&\text{if $m$ and $k$ is even}\end{cases}
\end{equation*}
for all $\alpha\in L^{\nu^k}$.
\end{prop}
For the proof we need the following lemma (cf.\ \cite{Sch04}, Section~8.1 and \cite{Bor92}, Section~12):
\begin{lem}
Let $k\in\N$ and $\alpha\in L^{\nu^k}$. Then
\begin{equation*}
u(\alpha+\nu\alpha+\ldots\nu^{k-1}\alpha)=u(\alpha)u(\nu\alpha)\ldots u(\nu^{k-1}\alpha)
\end{equation*}
if $k$ is odd and
\begin{equation*}
u(\alpha+\nu\alpha+\ldots\nu^{k-1}\alpha)=(-1)^{\langle\alpha,\nu^{k/2}\alpha\rangle}u(\alpha)u(\nu\alpha)\ldots u(\nu^{k-1}\alpha)
\end{equation*}
if $k$ is even.
\end{lem}
\begin{proof}
For $k=0,1$ the statement is trivial. From
\begin{equation*}
u(\alpha)u(\beta)=u(\alpha+\beta)\eps(\alpha,\beta)/\eps(\nu\alpha,\nu\beta)
\end{equation*}
for $\alpha,\beta\in L$ we deduce
\begin{equation*}
u(\alpha)u(\nu\alpha)=u(\alpha+\nu\alpha)\eps(\alpha,\nu\alpha)/\eps(\nu\alpha,\nu^2\alpha),
\end{equation*}
which is the statement for $k=2$. More generally, by induction over $k$ we show that
\begin{align*}
&u(\alpha)u(\nu\alpha)\ldots u(\nu^{k-1}\alpha)\\
&=u(\alpha+\nu\alpha+\ldots+\nu^{k-1}\alpha)\eps(\alpha,\nu\alpha+\ldots+\nu^{k-1}\alpha)/\eps(\nu\alpha+\ldots+\nu^{k-1}\alpha,\nu^k\alpha)
\end{align*}
for $\alpha\in L$ and $k\in\Z_{\geq2}$. In the induction step we use that $\eps\colon L\times L\to\{\pm1\}$ is a 2-cocycle. We then assume that $\alpha\in L^{\nu^k}$, i.e.\ $\nu^k\alpha=\alpha$. Then
\begin{equation*}
u(\alpha)u(\nu\alpha)\ldots u(\nu^{k-1}\alpha)=u(\alpha+\nu\alpha+\ldots+\nu^{k-1}\alpha)(-1)^{\langle\alpha,\nu\alpha+\ldots+\nu^{k-1}\alpha\rangle}
\end{equation*}
by \eqref{eq:multcommcond} and
\begin{equation*}
\langle\alpha,\nu\alpha+\ldots+\nu^{k-1}\alpha\rangle=\sum_{i=1}^{k-1}\langle\alpha,\nu^i\alpha\rangle=\begin{cases}0&\text{if $k$ is odd,}\\\langle\alpha,\nu^{k/2}\alpha\rangle&\text{if $k$ is even}\end{cases}\pmod{2},
\end{equation*}
where we used that $\langle\alpha,\nu^i\alpha\rangle=\langle\nu^{m-i}\alpha,\alpha\rangle=\langle\alpha,\nu^{m-i}\alpha\rangle$ and paired corresponding elements, which add to 0 modulo~2.
\end{proof}
\begin{proof}[Proof of Proposition~\ref{prop:standardliftpower}]
First, let $m$ be odd. If $k$ is odd, then, by the above lemma, for $\alpha\in L^{\nu^k}$
\begin{equation*}
\hat{\nu}^k(\ee_\alpha)=u(\alpha)u(\nu\alpha)\ldots u(\nu^{k-1}\alpha)\ee_{\nu^k\alpha}=u(\alpha+\ldots+\nu^{k-1}\alpha)\ee_\alpha=\ee_\alpha
\end{equation*}
since $\alpha+\ldots+\nu^{k-1}\alpha\in L^\nu$ for $\alpha\in L^{\nu^k}$ and $u$ vanishes on $L^\nu$ by assumption. In particular, for $k=m$ we get
\begin{equation*}
\hat{\nu}^m(\ee_\alpha)=\ee_\alpha
\end{equation*}
for all $\alpha\in L^{\nu^m}=L$, which implies that $\hat{\nu}$ has order $m$. If $k$ is even, then $k+m$ is odd and
\begin{equation*}
\hat{\nu}^k(\ee_\alpha)=\hat{\nu}^k(\hat{\nu}^m(\ee_\alpha))=\hat{\nu}^{m+k}(\ee_\alpha)=\ee_\alpha.
\end{equation*}
The proof in the case where $m$ is even is analogous.
\end{proof}

The following corollary to the lemma will be needed later:
\begin{cor}\label{cor:uhom}
For any $k\in\N$ the map
\begin{equation*}
\alpha\mapsto u(\alpha)u(\nu\alpha)\ldots u(\nu^{k-1}\alpha)
\end{equation*}
defines a homomorphism $L^{\nu^k}\to\{\pm1\}$.
\end{cor}
\begin{proof}
By the above lemma it suffices to show that $u(\alpha+\ldots+\nu^{k-1}\alpha)$ and $(-1)^{\langle\alpha,\nu^{k/2}\alpha\rangle}$ define homomorphisms $L^{\nu^k}\to\{\pm1\}$. For $\alpha\in L^{\nu^k}$, $\alpha+\ldots+\nu^{k-1}\alpha\in L^\nu$. It follows directly from the definition of $u$ (with respect to $\eps$) that $u$ is a homomorphism on $L^\nu$ and so the first statement follows. A simple calculation yields
\begin{equation*}
\langle\alpha+\beta,\nu^{k/2}(\alpha+\beta)\rangle=\langle\alpha,\nu^{k/2}\alpha\rangle+\langle\beta,\nu^{k/2}\beta\rangle\pmod{2}
\end{equation*}
if $\alpha,\beta\in L^{\nu^k}$ since $L$ is integral. This shows the second assertion and proves the statement of the corollary.
\end{proof}

\minisec{Order of the Lifted Automorphism}

We denote by
\begin{equation*}
\hat{m}:=\ord(\hat{\nu})
\end{equation*}
the order of the lifted automorphism $\hat{\nu}\in\Aut(V_L)$. Clearly, $\hat{m}$ is always a multiple of $m$.

We make this more precise in the case of a standard lift. As a corollary to the above proposition we obtain:
\begin{cor}
Let $\hat\nu$ be a standard lift of $\nu$ of order $m$. If $m$ is odd, then $\hat{\nu}$ has order $m$. If $m$ is even, then $\hat{\nu}$ has order $m$ if $\langle\alpha,\nu^{m/2}\alpha\rangle\in 2\Z$ for all $\alpha\in L$ and order $2m$ otherwise.
\end{cor}

\section{Twisted Modules for Lattice \VOA{}s}\label{sec:twmodlatvoa}

An explicit construction of the modules for a lattice \voa{} twisted by a lift of a lattice isomorphism is given in \cite{DL96,BK04}. They are classified in \cite{BK04}.
We give a slight generalisation by allowing non-standard lifts.

Consider a lattice isomorphism $\nu$ of the positive-definite, even lattice $L$ and a lift $\hat\nu$ to an automorphism of the \voa{} $V_L$. In order to construct the $\hat\nu$-twisted $V_L$-modules we need $\hat\nu$-twisted versions of the $\hat\h$-module $M_{\hat\h}(1)$ and the twisted group algebra. For simplicity we will only describe their structure as modules for $\hat\h[\nu]$, the twisted version of $\hat\h$.

\minisec{Twisted Heisenberg Module}

Recall that $m=\ord(\nu)$.
We define the $\nu$-twisted affine Lie algebra $\hat{\h}[\nu]$ associated with $\h$ to be
\begin{equation*}
\hat{\h}[\nu]:=\left(\bigoplus_{n\in(1/m)\Z}\h_{(mn)}\otimes_\C\C t^n\right)\oplus\C\mathbf{k},
\end{equation*}
called the \emph{$\nu$-twisted Heisenberg current algebra} in \cite{BK04}, with the Lie bracket defined by linear continuation of
\begin{equation*}
[x(n),y(n')]=\langle x,y\rangle n\delta_{n+n',0}\mathbf{k}\quad\text{and}\quad[u,\mathbf{k}]=0
\end{equation*}
for $n,n'\in(1/m)\Z$, $x\in\h_{(mn)}$, $y\in\h_{(mn')}$ and $u\in\hat{\h}[\nu]$. Again we write $h(n):=h\otimes t^n$ for $n\in(1/m)\Z$ and $h\in\h_{(mn)}$. Note that for the identity automorphism $\nu=\id$ the twisted algebra $\hat{\h}[\nu]$ reduces to the untwisted one $\hat{\h}$.

We decompose $\hat\h[\nu]$ as
\begin{equation*}
\hat\h[\nu]=\hat\h[\nu]_-\oplus\hat\h[\nu]_0\oplus\hat\h[\nu]_+\oplus\C\mathbf{k}
\end{equation*}
where
\begin{align*}
\hat\h[\nu]_\pm:=\bigoplus_{n\in(1/m)\Z_{\gtrless 0}}\h_{(mn)}\otimes_\C\C t^n\quad\text{and}\quad\hat\h[\nu]_0:=\h_{(0)}\otimes_\C\C t^0.
\end{align*}
We can also consider the Lie subalgebra $\hat{\h}[\nu]_{\neq 0}:=\hat\h[\nu]_-\oplus\hat\h[\nu]_+\oplus\C\mathbf{k}$ of $\hat{\h}[\nu]$. It is a Heisenberg Lie algebra in the sense that its commutator subalgebra equals its centre and is one-dimensional.

We define the Lie subalgebra $\hat{\mathfrak{b}}[\nu]:=\hat\h[\nu]_0\oplus\hat\h[\nu]_+\oplus\C\mathbf{k}=(\bigoplus_{n\in(1/m)\N}\h_{(mn)}\otimes_\C\C t^n)\oplus\C\mathbf{k}$ and let it act on the one-dimensional vector space $\C$ as
\begin{equation*}
h(n)\cdot w=0\quad\text{and}\quad\mathbf{k}\cdot w=w
\end{equation*}
for all $n\in(1/m)\N$, $h\in\h_{(mn)}$ and $w\in\C$. Then we define the induced $\hat\h[\nu]$-module
\begin{equation*}
M_{\hat\h}(1)[\nu]:=\Ind_{\hat{\mathfrak{b}}[\nu]}^{\hat\h[\nu]}\C=U(\hat\h[\nu])\otimes_{U(\hat{\mathfrak{b}}[\nu])}\C.
\end{equation*}
$M_{\hat\h}(1)[\nu]$ is irreducible as a module for $\hat{\h}[\nu]$ (even as module for $\hat{\h}[\nu]_{\neq 0}$). As a vector space $M_{\hat\h}(1)[\nu]$ is isomorphic to the symmetric algebra $S(\hat\h[\nu]_-)$ of $\hat\h[\nu]_-$. In particular, $M_{\hat\h}(1)[\nu]$ is spanned by elements of the form
\begin{equation*}
h_k(-n_k)\ldots h_1(-n_1)1
\end{equation*}
for $k\in\N$, $n_1,\ldots,n_k\in(1/m)\Ns$, $h_1\in\h_{(mn_1)},\ldots,h_k\in\h_{(mn_k)}$ all non-zero (and $1\in\C$).

We introduce a $\Q$-grading on $M_{\hat\h}(1)[\nu]$ such that an element of the above form has weight $\wt(1)+n_1+\ldots+n_k=\rho_\nu+n_1+\ldots+n_k\in \rho_\nu+(1/m)\N$ where
\begin{equation}\label{eq:vacanomaly}
\wt(1)=\frac{1}{4m^2}\sum_{k=1}^{m-1}k(m-k)\dim_\C(\h_{(k)})=:\rho_\nu
\end{equation}
is the \emph{vacuum anomaly}. This shifted grading is necessary to make the grading compatible with the Virasoro algebra (as explained in \cite{DL96}, Section~6).

\begin{rem}\label{rem:cycleconf}
Let $\nu$ be an automorphism of $L$ of cycle shape $\prod_{t\mid m}t^{b_t}$. Then a simple calculation shows that
\begin{equation*}
\rho_\nu=\frac{1}{24}\sum_{t\mid m}b_t\left(t-\frac{1}{t}\right).
\end{equation*}
\end{rem}

\minisec{Irreducible Twisted Modules}

The automorphism $\nu$ acts naturally on the finite group $L'/L$. Let $(L'/L)^\nu$ denote the fixed points under this action, i.e.\ $(L'/L)^\nu=\{\lambda +L\in L'/L\;|\;(\id-\nu)\lambda\in L\}$. For $\lambda+L\in(L'/L)^\nu$ we set as vector space and $\hat\h[\nu]$-module
\begin{equation*}
V_{\lambda+L}(\hat\nu):=M_{\hat\h}(1)[\nu]\otimes\ee_{\pi_\nu(\lambda)}\C[\pi_\nu(L)]\otimes\C^{d(\nu)}
\end{equation*}
where $\C^{d(\nu)}$ carries the zero representation of $\hat\h[\nu]$ and
\begin{equation*}
d(\nu)=\sqrt{\left|L_\nu/(L\cap(\id-\nu)L')\right|}
\end{equation*}
is the \emph{defect} of $\nu$. An element
\begin{equation*}
h_k(-n_k)\ldots h_1(-n_1)1\otimes\ee_{\pi_\nu(\lambda+\alpha)}\otimes x
\end{equation*}
for $k\in\N$, $n_1,\ldots,n_k\in(1/m)\Ns$, $h_1\in\h_{(mn_1)},\ldots,h_k\in\h_{(mn_k)}$, $\alpha\in L$, $x\in\C^{d(\nu)}$ has weight
\begin{equation*}
\rho_\nu+n_1+\ldots+n_k+\langle\pi_\nu(\alpha+\lambda),\pi_\nu(\alpha+\lambda)\rangle/2.
\end{equation*}

It is shown in \cite{BK04} that $V_{\lambda+L}(\hat\nu)$ for $\lambda+L\in(L'/L)^\nu$ can be endowed with the structure of an irreducible $\hat\nu$-twisted $V_L$-module if $\hat\nu$ is assumed to be a standard lift of $\nu$. The authors also give a complete classification in the case of standard lifts:
\begin{thm}[\cite{BK04}, Theorem~4.2]\label{thm:4.2}
Let $L$ be a positive-definite, even lattice and $V_L$ the associated lattice \voa{}. Let $\nu\in\Aut(L)$ and let $\hat\nu$ be a standard lift of $\nu$. Then the isomorphism classes of irreducible $\hat\nu$-twisted $V_L$-modules are $V_{\lambda+L}(\hat\nu)$ for $\lambda+L\in(L'/L)^\nu$.
\end{thm}

\begin{rem}
\item
\begin{enumerate}
\item It follows from the construction of the irreducible module $V_{\lambda+L}(\hat\nu)$ that its conformal weight $\rho(V_{\lambda+L}(\hat\nu))$ is given by $\rho_\nu$ plus the norm of a shortest vector in the lattice coset $\pi_\nu(\lambda+L)$.
\item In \cite{DL96} the space $\C^{d(\nu)}$ is described as a module for a central extension of $L$ of dimension $\sqrt{|N/R|}$ where $N=L_\nu$ and $R$ is the radical of a certain alternating $\Z$-bilinear form $C_N\colon N\times N\to\C^\times$ \cite{Lep85}. Following (4.43) and (4.44) in \cite{BK04} this bilinear form can be rewritten as $C_N(\alpha,\beta)=\e^{(2\pi\i)\langle\kappa\alpha,\beta\rangle}$ where $\kappa$ is the inverse of $(\id-\nu)|_{\h_{(0)}^\bot}$. Then
\begin{align*}
R&=\{\alpha\in N\;|\;C_N(\alpha,\beta)=1\text{ for all }\beta\in N\}\\
&=\{\alpha\in N\;|\;\langle\kappa\alpha,\beta\rangle\in\Z\text{ for all }\beta\in N\}=\{\alpha\in N\;|\;\kappa\alpha\in N'\}\\
&=N\cap\kappa^{-1}N'=N\cap(\id-\nu)N'.
\end{align*}
On the other hand, since $(\id-\nu)=(\id-\nu)\pi_\nu^\bot$,
\begin{align*}
L\cap(\id-\nu)L'&=L\cap(\id-\nu)\pi_\nu^\bot(L')=L\cap(\id-\nu)N'=L\cap\h_{(0)}^\bot\cap(\id-\nu)N'\\
&=N\cap(\id-\nu)N',
\end{align*}
using \eqref{eq:lem4.6}. This shows that $R=L\cap(\id-\nu)L'$ and indeed $d(\nu)=\sqrt{|N/R|}$.
\item The above construction of $V_{\lambda+L}(\hat\nu)$ yields an irreducible $\hat\nu$-twisted module according to the definition of twisted modules used in this text and in \cite{BK04} (see Remark~\ref{rem:convtwmod}). What is constructed in \cite{DL96} is an irreducible $\hat\nu^{-1}$-twisted module in our convention.
\end{enumerate}
\end{rem}

\minisec{Non-Standard Lifts}
We describe a generalisation of the twisted module construction for non-standard lifts. For simplicity we restrict to the case of a unimodular lattice $L$, i.e.\ $L'/L=1$. Then the corresponding lattice \voa{} $V_L$ is holomorphic and there is only one $\hat\nu$-twisted $V_L$-module up to isomorphism.

Recall that the restriction of $u$ to $L^\nu$ is a homomorphism $L^\nu\to\{\pm 1\}$ (see Corollary~\ref{cor:uhom}). It is hence possible to find a vector $s_\eta\in\h_{(0)}$ with
\begin{equation*}
\e^{(2\pi\i)\langle s_\eta,\alpha\rangle}=u(\alpha)
\end{equation*}
for all $\alpha\in L^\nu$. This vector $s_\eta$ lies in $(1/2)(L^\nu)'$ and is unique up to $(L^\nu)'$. If $u$ corresponds to a standard lift, then $s_\eta\in 0+(L^\nu)'$.
\begin{prop}
Let $L$ be a unimodular, positive-definite, even lattice. Let $\nu\in\Aut(L)$ and let $\hat\nu$ be a lift of $\nu$, not necessarily a standard lift. Then the unique irreducible $\hat\nu$-twisted $V_L$-module up to isomorphism is as vector space and $\hat\h[\nu]$-module given by
\begin{equation*}
V_L(\hat\nu):=M_{\hat\h}(1)[\nu]\otimes\ee_{s_\eta}\C[\pi_\nu(L)]\otimes\C^{d(\nu)}.
\end{equation*}
\end{prop}
The grading of $V_L(\hat\nu)$ and the defect $d(\nu)$ are as in the standard case. The conformal weight $\rho(V_L(\hat\nu))$ is $\rho_\nu$ plus the norm of a shortest vector in the lattice coset $s_\eta+\pi_\nu(L)$.

\section{Characters}

Given a positive-definite, even lattice, the characters of the associated lattice \voa{} and its irreducible modules are well-known. Recall that $\eta(\tau)$ is the Dedekind eta function. Moreover, if $L$ is a positive-definite lattice, we define the \emph{theta function} of $L$ as
\begin{equation*}
\vartheta_L(\tau)=\sum_{\beta\in L}q_\tau^{\langle\beta,\beta\rangle/2},
\end{equation*}
which has integral exponents if $L$ is even. If $L$ is integral we can also consider the theta function $\vartheta_{\alpha+L}$ associated with the coset $\alpha+L$ of $L$ in $L'/L$. It is analogously defined by
\begin{equation*}
\vartheta_{\alpha+L}(\tau)=\sum_{\beta\in\alpha+L}q_\tau^{\langle\beta,\beta\rangle/2}
\end{equation*}
for $\alpha+L\in L'/L$.
\begin{prop}
Let $L$ be a positive-definite, even lattice and $V_L$ the associated lattice \voa{} of central charge $c=\rk(L)$. Then the character of the irreducible $V_L$-module $V_{\alpha+L}$ is given by
\begin{equation*}
\ch_{V_{\alpha+L}}(\tau)=\tr_{V_{\alpha+L}}q_\tau^{L_0-c/24}=\frac{\vartheta_{\alpha+L}(\tau)}{\eta(\tau)^{\rk(L)}}
\end{equation*}
for all $\alpha+L\in L'/L$.
\end{prop}
\begin{proof}
The case of $\alpha=0$ is treated in \cite{FLM88}, Remark~7.1.2. The general case is analogous.
\end{proof}

\minisec{Twisted Characters}
We now insert an automorphism $\hat\nu$ of the lattice \voa{} $V_L$ into the trace. First we assume that $\hat\nu$ is a standard lift of $\nu$. Given an automorphism $\nu$ of the lattice $L$ of order $m=\ord(\nu)$ with cycle shape $\prod_{t\mid m}t^{b_t}$ let us define the eta product
\begin{equation*}
\eta_\nu(\tau):=\prod_{t\mid m}\eta(t\tau)^{b_t}.
\end{equation*}
For $\nu=\id_L$ this gives $\eta_\nu(\tau)=\eta(\tau)^{\rk(L)}$.
\begin{prop}
Let $L$ be a positive-definite, even lattice and $\hat\nu$ an automorphism of $V_L$ obtained as a standard lift of a lattice automorphism $\nu\in\Aut(L)$. Then the twisted character for $\hat\nu$ on the \voa{} $V_L$ of central charge $c=\rk(L)$ is
\begin{equation*}
\tr_{V_L}\hat\nu q_\tau^{L_0-c/24}=\frac{\vartheta_{L^\nu}(\tau)}{\eta_\nu(\tau)}.
\end{equation*}
\end{prop}

More generally we consider a non-standard lift $\hat\nu$ of $\nu$. In this case, the trace will depend on the value of the function $u$ on $L^\nu$ (recall from Corollary~\ref{cor:uhom} that $u$ restricts to a homomorphism on $L^\nu$). Indeed, the twisted character becomes
\begin{equation*}
\tr_{V_L}\hat\nu q_\tau^{L_0-c/24}=\frac{\vartheta_{L^\nu,u}(\tau)}{\eta_\nu(\tau)}
\end{equation*}
with the generalised theta function
\begin{equation*}
\vartheta_{L^\nu,u}(\tau):=\sum_{\beta\in L^\nu}u(\beta)q_\tau^{\langle\beta,\beta\rangle/2}.
\end{equation*}

Finally, consider a power $\hat\nu^k$ of $\hat\nu$ for $k\in\N$. This does not have to be a standard lift of $\nu^k$ even if $\hat\nu$ is a standard lift of $\nu$. If the lift $\hat\nu$ is determined by the function $u\colon L\to\{\pm1\}$, then the lift $\hat\nu^k$ of $\nu^k$ is determined by the function $w\colon L\to\{\pm1\}$ given by $w(\alpha)=u(\alpha)u(\nu\alpha)\ldots u(\nu^{k-1}\alpha)$.

\begin{prop}\label{prop:twistedtrace}
Let $L$ be a positive-definite, even lattice and $\hat\nu$ some lift of a lattice automorphism $\nu\in\Aut(L)$. For $k\in\N$ the twisted character for $\hat{\nu}^k$ on the \voa{} $V_L$ of central charge $c=\rk(L)$ is
\begin{equation*}
\tr_{V_L}\hat{\nu}^kq_\tau^{L_0-c/24}=\frac{\vartheta_{L^{\nu^k},w}(\tau)}{\eta_{\nu^k}(\tau)}
\end{equation*}
where $w(\alpha)=u(\alpha)u(\nu\alpha)\ldots u(\nu^{k-1}\alpha)$.
\end{prop}
Note that if $\nu$ has cycle shape $\prod_{t\mid m}t^{b_t}$, then $\nu^k$ has cycle shape $\prod_{t\mid m}(t/(t,k))^{(t,k)b_t}$.

\section{Orbifolds of Lattice \VOA{}s}\label{sec:latorbifold}

In this section we combine the results of Chapter~\ref{ch:fpvosa} and this chapter to explicitly describe orbifolds of holomorphic lattice \voa{}s where the \voa{} automorphisms are obtained as lifts of lattice automorphisms.

\begin{customass}{{\textbf{\textsf{L}}}}\label{ass:l}
Let $L$ denote be an even, unimodular, positive-definite lattice. Let $V=V_L$ be the corresponding lattice \voa{}, which is holomorphic, satisfies Assumption~\ref{ass:n} and has central charge $c=\rk(L)$.

Let $\nu\in\Aut(L)$ be an automorphism of the lattice of order $m=\ord(\nu)\in\Ns$, which lifts to some automorphism $\hat{\nu}$ of $V_L$ (described by the function $u\colon L\to\{\pm1\}$) of order $\hat{m}=\ord(\hat{\nu})\in\Ns$, some multiple of $m$. Hence we are in the situation of Assumption~\ref{ass:op}. We do not assume that $\hat\nu$ is a standard lift. Moreover, assume that $\hat\nu$ is of type $\hat{m}\{0\}$ (see Proposition~\ref{prop:zerosuff}), i.e.\ the conformal weight of the unique $\hat{\nu}$-twisted $V_L$-module $V_L(\hat\nu)$ is in $(1/\hat{m})\Z$.
\end{customass}

For the rest of this section assume that Assumption~\ref{ass:l} holds. Then the orbifold construction \eqref{eq:orbifold} for $r=0$ yields a new holomorphic \voa{}
\begin{equation*}
\widetilde{V}=\bigoplus_{i\in\Z_{\hat{m}}}W^{(i,0)}
\end{equation*}
(assuming that representations $\phi_i$ are chosen as in Corollary~\ref{cor:main}) built from the irreducible modules of the \fpvosa{} $V_L^{\hat\nu}$. The main goal of this section is to describe the computation of the character $\ch_{\widetilde{V}}(\tau)$, or more precisely its $q_\tau$-expansion.
To this end we need to compute all the twisted trace functions
\begin{equation*}
T(\vac,i,j,\tau)=\tr_{V_L(\hat{\nu}^i)}\phi_i(\hat{\nu}^j)q_\tau^{L_0-c/24}
\end{equation*}
for $i,j\in\Z_{\hat{m}}$. The twisted trace functions on $V_L=V_L(\hat\nu^0)$ are determined in Proposition~\ref{prop:twistedtrace} to be
\begin{equation*}
T(\vac,0,j,\tau)=\tr_{V_L}\hat{\nu}^jq_\tau^{L_0-c/24}=\frac{\vartheta_{L^{\nu^j},w}(\tau)}{\eta_{\nu^j}(\tau)}.
\end{equation*}
In general, the trace functions for $i\in\Z_{\hat{m}}\setminus\{0\}$ cannot be obtained directly since we do not know all the representations $\phi_i$ explicitly in these cases (except if $i$ and $\hat{m}$ are coprime, see proof of Lemma~\ref{lem:phi2}).\footnote{Note that if $i\in\Z_{\hat{m}}\setminus\{0\}$ and $j=0$, we can in fact compute the trace functions $T(\vac,i,0,\tau)$ directly since they are simply the characters of the irreducible $V_L$-modules $V_L(\hat\nu^i)$ but we omit this computation since we obtain them via the modular transformations below, which we have to use in any case.}

\minisec{Modular Transformations}

In order to obtain all the twisted trace functions $T(\vac,i,j,\tau)$ we make use of the modular-transformation properties of the trace functions described in Section~\ref{sec:modinv}. For this we need to explicitly know the $S$- and $T$-transformations of the above trace functions. The transformation behaviour of the eta function and of products thereof is explicitly known (see e.g.\ \cite{Apo90}, Theorem~3.4). To obtain the transformation behaviour of the generalised theta function in the numerator we have to write it in terms of ordinary theta functions.

Consider some lift $\hat{\nu}\in\Aut(\hat{L})$ of $\nu\in\Aut(L)$, not necessarily a standard lift, and fix $j\in\Z_{\hat{m}}$. We showed in Corollary~\ref{cor:uhom} that the map $\alpha\mapsto w(\alpha):=u(\alpha)\ldots u(\nu^{j-1}\alpha)$ defines a homomorphism $L^{\nu^j}\to\{\pm1\}$. We consider the following decomposition of $L^{\nu^j}$ into inverse images with respect to $w$
\begin{equation*}
L^{\nu^j}=L_0^{\nu^j}\oplus L_1^{\nu^j}:=w^{-1}(\{1\})\oplus w^{-1}(\{-1\}).
\end{equation*}
Clearly, $L_0^{\nu^j}=\ker(w)$ is a sublattice of $L^{\nu^j}$. The other inverse image $L_1^{\nu^j}$ is a coset of $L_0^{\nu^j}$ if it is non-empty. Namely, if $\beta_1$ is some arbitrary element in $L_1^{\nu^j}$, then $L_1^{\nu^j}=L_0^{\nu^j}+\beta_1$. Then
\begin{align*}
\vartheta_{L^{\nu^j},w}(\tau)&=\sum_{\alpha\in L^{\nu^j}}u(\alpha)\ldots u(\nu^{j-1}\alpha)q_\tau^{\langle\alpha,\alpha\rangle/2}=\sum_{\alpha\in L^{\nu^j}}w(\alpha)q_\tau^{\langle\alpha,\alpha\rangle/2}\\
&=\vartheta_{L_0^{\nu^j}}(\tau)-\vartheta_{L_1^{\nu^j}}(\tau)\\
&=\begin{cases}\vartheta_{L^{\nu^j}}(\tau)&\text{if $L_1^{\nu^j}=\emptyset$,}\\\vartheta_{L_0^{\nu^j}}(\tau)-\vartheta_{\beta_1+L_0^{\nu^j}}(\tau)&\text{if $L_1^{\nu^j}\neq\emptyset$.}\end{cases}
\end{align*}

The transformation behaviour of these theta functions under $\SLZ$ is known:
\begin{thm}[Special case of \cite{Bor98}, Theorem~4.1]\label{thm:bor98thm4.1}
Let $L$ be a positive-definite, even lattice with discriminant form $D=L'/L$. Define for $\gamma+L\in D$
\begin{equation*}
\vartheta_{\gamma+L}(\tau):=\sum_{\alpha\in \gamma+L}q_\tau^{\langle\alpha,\alpha\rangle/2}.
\end{equation*}
Then
\begin{equation*}
\Theta_L(\tau):=\sum_{\gamma+L\in D}\vartheta_{\gamma+L}(\tau)\ee_{\gamma+L}
\end{equation*}
is a modular form for the Weil representation $\rho_D$ of $\MpZ$, the metaplectic cover of $\SLZ$, on $\C[D]$ of weight $\rk(L)/2$.
\end{thm}

If $L_1^{\nu^j}$ is empty, then $\vartheta_{L^{\nu^j},w}(\tau)=\vartheta_{L^{\nu^j}}(\tau)$ and we need to consider the modular-transformation behaviour of $\Theta_{L^{\nu^j}}(\tau)$.

Now assume that $L_1^{\nu^j}$ is non-empty. Since $\beta_1\in L_1^{\nu^j}\subseteq L^{\nu^j}\leq L\leq L'\leq (L^{\nu^j})'\leq (L_0^{\nu^j})'$, we can view the components $\vartheta_{L_0^{\nu^j}}(\tau)$ and $\vartheta_{L_0^{\nu^j}+\beta_1}(\tau)$ of $\vartheta_{L^{\nu^j},w}(\tau)$ as entries of the vector-valued modular form $\Theta_{L_0^{\nu^j}}(\tau)$ and determine the transformation behaviour according to the above theorem.

While the transformation behaviour of vector-valued modular forms for the Weil representation is known in principle, it only becomes computationally feasible using the explicit formulæ developed \cite{Sch09} for $\SLZ$ and generalised in \cite{Str13} to $\MpZ$.

Combining the modular properties of the eta function and the theta functions we can determine the $q_\tau$-expansion of $T(\vac,0,j,M.\tau)$ for all $M\in\SLZ$.

\minisec{Characters and Dimensions}

It remains to determine the $q_\tau$-expansion of $T(\vac,i,j,\tau)$ for any $i,j\in\Z_{\hat{m}}$. For this, let
\begin{equation*}
M_{i,j}:=\begin{pmatrix}*&*\\\frac{i}{\gcd(i,j)}&\frac{j}{\gcd(i,j)}\end{pmatrix}\in\SLZ.
\end{equation*}
The matrix $M_{i,j}$ and $\gcd(i,j)$ depend in fact not only on $i,j\in\Z_{\hat{m}}$ but on the choice of representatives modulo $\hat{m}$. Then
\begin{equation*}
(0,\gcd(i,j))M_{i,j}=(0,\gcd(i,j))\begin{pmatrix}*&*\\\frac{i}{\gcd(i,j)}&\frac{j}{\gcd(i,j)}\end{pmatrix}=(i,j)
\end{equation*}
and hence
\begin{equation*}
T(\vac,i,j,\tau)=T(\vac,(0,\gcd(i,j))M_{i,j},\tau)=T(\vac,0,\gcd(i,j),M_{i,j}.\tau)/Z(M_{i,j})
\end{equation*}
according to \eqref{eq:trafon0}, where the function $Z$ takes values in $U_3$. If the rank $\rk(L)$ of $L$ is a multiple of 24, the factor $Z(M)=1$ for all $M\in\SLZ$.

Then we can determine the characters of the irreducible $V_L^{\hat{\nu}}$-modules $W^{(i,j)}$ via
\begin{equation*}
\ch_{W^{(i,j)}}(\tau)=\frac{1}{n}\sum_{l\in\Z_{\hat{m}}}\xi_n^{-lj}T(\vac,i,l,\tau).
\end{equation*}
and finally
\begin{equation*}
\ch_{\widetilde{V}}(\tau)=\sum_{i\in\Z_{\hat{m}}}\ch_{W^{(i,0)}}(\tau)=\frac{1}{n}\sum_{i,l\in\Z_{\hat{m}}}T(\vac,i,l,\tau).
\end{equation*}
This allows us to read off the dimensions of the weight spaces $\widetilde{V}_n$ for $n\in\N$.

For all the examples considered in this text the above calculations are performed using \texttt{Sage} and \texttt{Magma} \cite{Sage,Magma}. Moreover, the Weil representation is computed using the ``modules'' package of the PSAGE library by Nils-Peter Skoruppa, Fredrik Strömberg and Stephan Ehlen, with minor modifications and bugfixes by the author of this text \cite{code:psage}.

\minisec{Example}

To illustrate the above steps we study the following simple example, which we will revisit in the next chapter (see Section~\ref{sec:newcases}). Let $L=N(A_4^6)$ be the Niemeier lattice with root lattice $A_4^6$, i.e.\ $L$ is an even, unimodular, positive-definite lattice of rank $\rk(L)=c=24$. We consider a certain automorphism $\nu$ of $L$ of order 5. It has cycle shape $1^{-1}5^5$ and its fixed-point sublattice is isomorphic to $A_4'(5)$, the dual lattice of $A_4$, with the quadratic form rescaled by 5. Let $\hat\nu$ be a standard lift of $\nu$. Since the order of $\nu$ is odd, $\hat\nu^k$ is a standard lift of $\nu^k$ for all $k\in\N$ and $\hat\nu$ also has order 5.

We can compute the conformal weight of $V_L(\hat\nu)$ and obtain $\rho(V_L(\hat\nu))=1$ by Remark~\ref{rem:cycleconf}, i.e.\ $\hat\nu$ is of type $5\{0\}$.

We compute
\begin{equation*}
T(\vac,0,0,\tau)=\frac{\vartheta_{N(A_4^6)}(\tau)}{\eta(\tau)^{24}}
\end{equation*}
and
\begin{equation*}
T(\vac,0,j,\tau)=\frac{\vartheta_{A_4'(5)}(\tau)}{\eta(\tau)^{-1}\eta(5\tau)^5}
\end{equation*}
for $j\in\Z_n\setminus\{0\}$. This yields
\begin{align*}
\ch_{W^{(0,0)}}(\tau)&=\frac{1}{5}\sum_{l\in\Z_{5}}T(\vac,0,l,\tau)=\frac{1}{5}\frac{\vartheta_{N(A_4^6)}(\tau)}{\eta(\tau)^{24}}+\frac{4}{5}\frac{\vartheta_{A_4'(5)}(\tau)}{\eta(\tau)^{-1}\eta(5\tau)^5}\\
&=q_\tau^{-1}+28+39384q_\tau+4298760q_\tau^2+172859970q_\tau^3+\ldots.
\end{align*}
The $S$-transformation yields
\begin{align*}
T(\vac,i,0,\tau)&=T(\vac,0,i,S.\tau)=5\frac{\vartheta_{A_4(1/5)}(\tau)}{\eta(\tau)^{-1}\eta(\tau/5)^5}\\
&=5+125q_\tau^{1/5}+750q_\tau^{2/5}+3375q_\tau^{3/5}+12250q_\tau^{4/5}+39375q_\tau+114000q_\tau^{6/5}\\
&\quad+307000q_\tau^{7/5}+776250q_\tau^{8/5}+1867125q_\tau^{9/5}+4298750q_\tau^2+\ldots
\end{align*}
for $i\in\Z_n\setminus\{0\}$. Then
\begin{align*}
\ch_{W^{(i,0)}}(\tau)&=\frac{1}{5}\sum_{l\in\Z_{5}}T(\vac,i,l,\tau)=\frac{1}{5}\sum_{l\in\Z_{5}}T(\vac,i,0,T^{i^{-1}l}.\tau)\\
&=\frac{1}{5}\sum_{l\in\Z_{5}}T(\vac,i,0,\tau+l)\\
&=5+39375q_\tau+4298750q_\tau^2+172860000q_\tau^3+\ldots
\end{align*}
for $i\in\Z_n\setminus\{0\}$. This is the $T$-invariant part of $T(\vac,i,0,\tau)$, i.e.\ those monomials with integral $q_\tau$-exponents.

Finally we obtain
\begin{align*}
\ch_{\widetilde{V}}(\tau)&=\sum_{i\in\Z_5}\ch_{W^{(i,0)}}(\tau)\\
&=q_\tau^{-1}+48+196884q_\tau+21493760q_\tau^2+864299970q_\tau^3+\ldots.
\end{align*}
We will see in Proposition~\ref{prop:lem2.1} that for central charge $c=24$ the character of $\widetilde{V}$ has to be exactly the $j$-invariant but with constant term given by $\dim_\C(\widetilde{V}_1)$. The only interesting term is hence the dimension of the weight-one space
\begin{equation*}
\dim_\C(\widetilde{V}_1)=\left[\ch_{\widetilde{V}}(q)\right](1-c/24)=\left[\ch_{\widetilde{V}}(q)\right](0)=48.
\end{equation*}

As a test, we can also compute $\dim_\C(V_1^G)$ using the dimension formula in Proposition~\ref{prop:dimform2}. Since $\rho(V_L(\hat\nu^i))=1$ for $i\in\Z_5\setminus\{0\}$, the formula simplifies to \eqref{eq:dimform2simple}. It is straightforward to compute $\dim_\C(V_1)=144$ and $\dim_\C(V_1^G)=28$. This immediately yields $\dim_\C(\widetilde{V}_1)=48$.

\minisec{Prime Order}

From the above example it is clear that we can derive a simple formula for the dimension of $\widetilde{V}_1$ if the automorphism $\hat{\nu}$ is of prime order, i.e.\ of type $p\{0\}$ for some prime $p$. This means in particular that $\hat\nu$ is a standard lift of $\nu$, $\ord(\nu)=\ord(\hat\nu)=p$ and that all powers of $\hat\nu$ are standard lifts of the corresponding powers of $\nu$. The cycle shape of $\nu$ has to be of the form $1^ap^b$ for some $a,b\in\Z$ with $a+pb=\rk(L)$ and $a+b=\rk(L^\nu)$.

\begin{thm}
Let $V_L$ and $\hat\nu$ be as in Assumption~\ref{ass:l} with $\ord(\hat\nu)=p$ for some prime $p$. Then
\begin{align*}
\dim_\C((\widetilde{V}_L)_1)&=\frac{N_1(L)+(p-1)N_1(L^\nu)}{p}+\rk(L^\nu)\\
&\quad+\frac{(p-1)\sqrt{p}^b}{\sqrt{\left|(L^\nu)'/L^\nu\right|}}\left[\frac{\vartheta_{(L^\nu)'}(\tau)}{\eta(\tau)^a\eta(\tau/p)^b}\right](1-c/24)
\end{align*}
where $N_1(\cdot)$ denotes the number of vectors $\alpha$ of norm $\langle\alpha,\alpha\rangle/2=1$ in a positive-definite lattice.
\end{thm}

We conclude with a remark on the type $p\{r\}$ of an arbitrary lifted lattice automorphism $\hat\nu$ of prime order. Remark~\ref{rem:cycleconf} implies that
\begin{equation*}
\rho(V_L(\hat\nu))=\frac{(p-1)(p+1)}{24p}b.
\end{equation*}
For $\hat\nu$ to be of type $p\{0\}$ we have to demand that $(p-1)(p+1)b/24\in\Z$. We observe:
\begin{prop}
Let $n\in\Z$ with $2\nmid n$ and $3\nmid n$. Then $n^2-1=(n+1)(n-1)$ is divisible by 24.
\end{prop}
An immediate consequence is the following:
\begin{cor}
Let $p\neq2,3$ be a prime. Then $(p+1)(p-1)$ is divisible by 24.
\end{cor}

In summary, if $\hat\nu$ is of prime order $p$, then it has type $p\{0\}$ if and only if
\begin{itemize}
\item $p=2$ and $\rk(L)-\rk(L^\nu)\in8\Z$,
\item $p=3$ and $\rk(L)-\rk(L^\nu)\in6\Z$ or
\item $p\geq 5$.
\end{itemize}
For $p=3$ this is the condition in Theorem~B of \cite{Miy13b}.

\chapter{\VOA{}s of Small Central Charge}\label{ch:schellekens}
This chapter deals with the construction and classification of holomorphic \voa{}s of central charge up to 24. We summarise the classification results obtained in \cite{Sch93,DM04,EMS15} and present constructions of five new holomorphic \voa{}s of central charge 24 as cyclic orbifolds of lattice \voa{}s.

\section{Summary of Classification Results}

In the following, we let $V$ be a holomorphic, $C_2$-cofinite \voa{} of CFT-type. Then in particular $V$ satisfies Assumption~\ref{ass:n}. By Zhu's theory we know that the central charge $c$ of such a \voa{} is a positive multiple of 8 (see Proposition~\ref{prop:div8}).

Using the lattice \voa{} construction (see Section~\ref{sec:latvoa}) we obtain a holomorphic \voa{} of central charge $c=\rk(L)$ for each positive-definite, even, unimodular lattice of rank $\rk(L)$. This gives 1 holomorphic \voa{} of central charge 8, 2 of central charge 16, 24 of central charge 24 (associated with the Niemeier lattices) and over 1\,000\,000\,000 of central charge 32. Beyond dimension 32, the numbers increase even more rapidly. It therefore seems that a classification of holomorphic \voa{}s is only feasible up to central charge 24.

It is shown in \cite{DM04} that for the cases of $c=8,16$ the classification of holomorphic, $C_2$-cofinite \voa{}s of CFT-type exactly mirrors that of the positive-definite, even, unimodular lattices:
\begin{thm}[\cite{DM04}, Theorems 1 and 2]
Let $V$ be a holomorphic, $C_2$-cofinite \voa{} of CFT-type.
\begin{enumerate}
\item If $c=8$, then $V$ is isomorphic to $V_{E_8}$, the lattice \voa{} associated with the lattice $E_8$ of rank 8.
\item If $c=16$, then $V$ is isomorphic to $V_{E_8^2}$ or $V_{D_{16}^+}$, the lattice \voa{}s associated with the two positive-definite, even, unimodular lattices $E_8^2$ and $D_{16}^+$ of rank 16, respectively.
\end{enumerate}
\end{thm}

The next case to be studied is that of central charge 24. With the Monster \voa{} $V^\natural$ (also called Moonshine module) constructed in \cite{FLM88} there is a well-known example of a holomorphic, $C_2$-cofinite \voa{} of CFT-type with $c=24$, which is not isomorphic to a lattice \voa{}. This makes the case $c=24$ more interesting than the cases $c=8,16$ and shows that the classification of holomorphic \voa{}s is a proper generalisation of the classification of positive-definite, even, unimodular lattices. Schellekens was the first to study the case $c=24$ \cite{Sch93}, in the language of ``meromorphic conformal field theories''. He produced a list of 71 possible Lie algebra structures of the space $V_1$. It is well known that in any \voa{} of CFT-type the space $V_1$ naturally carries a Lie algebra structure (see also Section~\ref{sec:voalie}). More precisely, the classification concerns not only the Lie algebra structure of $V_1$ but the \emph{affine structure} on the \vosa{} of $V$ generated by $V_1$. This notion will be explained below.

Schellekens' result was proved as a theorem for \voa{}s in \cite{EMS15}:
\begin{thm}[\cite{EMS15}, Theorem~6.4, \cite{Sch93}]\label{thm:schellekens}
Let $V$ be a holomorphic, $C_2$-cofinite \voa{} of CFT-type and central charge 24. Then
\begin{enumerate}
\item\label{enum:V10} $V_1=\{0\}$,
\item $V_1$ is the 24-dimensional abelian Lie algebra $\C^{24}$ and $V$ is isomorphic to the \voa{} associated with the Leech lattice or
\item $V_1$ is one of 69 semisimple Lie algebras (with a certain affine structure) given in Table~1 of \cite{Sch93}. This case can be further split up:
\begin{enumerate}
\item $\rk(V_1)=24$ and $V$ is isomorphic to the \voa{} associated with one of the 23 Niemeier lattices (except for the Leech lattice) or
\item $\rk(V_1)<24$.
\end{enumerate}
\end{enumerate}
\end{thm}
We remark that entry 62 in Table~1 of \cite{Sch93} should read $E_{8,2}B_{8,1}$ (this typo is corrected in the arXiv-version of the paper). The above theorem is an improvement of \cite{DM04}, Theorem~3. An example of case \ref{enum:V10} is the Moonshine module with $V_1^\natural=\{0\}$.

\minisec{Characters}

There is the following result concerning the characters of holomorphic, $C_2$-cofinite \voa{}s of CFT-type, which are meromorphic modular forms of weight 0 for $\SLZ$, possibly with a character modulo 3. For $24\mid c$ this character vanishes and the character of $V$ is a modular function. Let
\begin{equation*}
J(q):=j(q)-744=q^{-1}+0+196884q+\ldots
\end{equation*}
denote the $j$-function (or Klein's $j$-invariant) with constant term removed, i.e.\ the character of the Moonshine module $V^\natural$, which has $\dim_\C(V_1^\natural)=0$.
\begin{prop}[\cite{DM04}, Lemma~2.1, \cite{Sch93}]\label{prop:lem2.1}
Let $V$ be a holomorphic, $C_2$-cofinite \voa{} of CFT-type.
\begin{enumerate}
\item If $c=8$, then $\ch_V(q)=\frac{\vartheta_{E_8}(q)}{\eta(q)^8}=q^{-1/3}(1+248q+\ldots)$.
\item If $c=16$, then $\ch_V(q)=\frac{(\vartheta_{E_8}(q))^2}{\eta(q)^{16}}=q^{-2/3}(1+496q+\ldots)$.
\item If $c=24$, then $\ch_V(q)=J(q)+\dim_\C(V_1)=q^{-1}+\dim_\C(V_1)+196884q+\ldots$.
\end{enumerate}
\end{prop}

\minisec{Affine Structure}

Let $\g$ be a simple, finite-dimensional Lie algebra with invariant bilinear form $(\cdot,\cdot)_\g$, normalised such that $(\theta,\theta)_\g=2$ for a long root $\theta$ of $\g$. Then the affine Kac-Moody algebra $\hat\g$ associated with $\g$ is the Lie algebra $\hat\g:=\g\otimes\C[t,t^{-1}]\oplus\C K$ with central element $K$ and Lie bracket
\begin{equation}\label{eq:affine}
[a\otimes t^m,b\otimes t^n]:=[a,b]\otimes t^{m+n}+m(a,b)_\g\delta_{m+n,0}K
\end{equation}
for $a,b\in\g$, $m,n\in\Z$.

A representation of $\hat\g$ is said to have level $k\in\C$ if $K$ acts as $k\id$. For a dominant integral weight $\lambda$ and $k\in\C$ let $L_{\hat\g}(k,\lambda)$ be the irreducible $\hat\g$-module of level $k$ obtained by inducing the irreducible highest-weight $\g$-module $L_\g(\lambda)$ up in a certain way to a $\hat\g$-module and taking its irreducible quotient. For more details on affine Kac-Moody algebras and their modules we refer the reader to \cite{Kac90}.

For $k\in\Ns$, $L_{\hat\g}(k,0)$ admits the structure of a rational \voa{} whose irreducible modules are given by the $L_{\hat\g}(k,\lambda)$ where $\lambda$ runs through a certain finite subset of the dominant integral weights of $\g$ (see \cite{FZ92}, Theorem~3.1.3 or e.g.\ \cite{LL04}, Sections 6.2 and 6.6).

Now, let $V$ be a \voa{} of CFT-type. Recall that the weight-one subspace $V_1$ of $V$ naturally carries a Lie algebra structure with Lie bracket $[a,b]=a_0b$ for $a,b\in V_1$. If $V$ is also self-contragredient, then there exists a non-degenerate, invariant bilinear form $\langle\cdot,\cdot\rangle$ on $V$, which is unique up to a non-zero scalar and symmetric \cite{Li94} (see Section~\ref{sec:cmibf}). We normalise the form so that $\langle\vac,\vac\rangle=-1$, where $\vac$ is the vacuum vector of $V$. Then the restriction of $\langle\cdot,\cdot\rangle$ to $V_1$ is non-degenerate and given by
\begin{equation*}
a_1b=b_1a=\langle a,b\rangle\vac
\end{equation*}
for $a,b\in V_1$. Indeed, it is easy to see that the homogeneous spaces $V_n$ and $V_m$ are orthogonal with respect to $\langle\cdot,\cdot\rangle$ for $m\neq n$ and hence the restriction of the bilinear form to any $V_n$, $n\in\N$, is non-degenerate.

For a self-contragredient \voa{} $V$ of CFT-type the commutator formula implies
\begin{equation*}
[a_m,b_n]=(a_0b)_{m+n}+m(a_1b)_{m+n-1}=[a,b]_{m+n}+m\langle a,b\rangle\delta_{m+n,0}\id_V
\end{equation*}
for all $a,b\in V_1$, $m,n\in\Z$. Comparing this with \eqref{eq:affine} we see that for a simple Lie subalgebra $\g$ of $V_1$ the map $a\otimes t^n\mapsto a_n$ for $a\in\g$ and $n\in\Z$ defines a representation of $\hat\g$ on $V$ of some level $k_\g\in\C$ with $\langle\cdot,\cdot\rangle=k_\g(\cdot,\cdot)_\g$ restricted to $\g$.

In the following let us assume that $V$ satisfies Assumption~\ref{ass:n}. Then it is shown in \cite{DM04b} that the Lie algebra $V_1$ is reductive, i.e.\ a direct sum of a semisimple and an abelian Lie algebra. Moreover, Theorem~3.1 in \cite{DM06b} states that for a simple Lie subalgebra $\g$ of $V_1$ the restriction of $\langle\cdot,\cdot\rangle$ to $\g$ is non-degenerate, the level $k_\g$ is a positive integer, the
\vosa{} of $V$ generated by $\g$ is isomorphic to $L_{\hat\g}(k_\g,0)$ and $V$ is an integrable $\hat\g$-module.

Now suppose in addition that $c=24$ and that $V$ is holomorphic. It is shown in \cite{DM04} that the Lie algebra $V_1$ is zero, abelian of dimension 24 or semisimple of rank at most $24$, where the case of Lie algebra rank 24 corresponds exactly to the 24 Niemeier lattice cases. Let us assume that $V_1$ is semisimple. Then $V_1$ decomposes into a direct sum
\begin{equation*}
V_1=\g_1\oplus\ldots\oplus\g_r
\end{equation*}
of simple Lie algebras $\g_i$ and the \vosa{} $\langle V_1\rangle$ of $V$ generated by $V_1$ is isomorphic to
\begin{equation*}
\langle V_1\rangle\cong L_{\hat\g_1}(k_1,0)\otimes\ldots\otimes L_{\hat\g_r}(k_r,0)
\end{equation*}
with levels $k_i:=k_{\g_i}\in\Ns$ and has the same Virasoro vector as $V$. This decomposition of the \voa{} $\langle V_1\rangle$ is called the \emph{affine structure} of $V$. We denote it by the symbol
\begin{equation*}
\g_{1,k_1}\ldots\g_{r,k_r}
\end{equation*}
and sometimes omit $k_i$ if it equals 1. By abuse of language we will also refer to the cases of abelian and zero Lie algebras $V_1$ as certain affine structures.

It is known from \cite{DM04} that
\begin{equation*}
\frac{h_i^\vee}{k_i}=\frac{\dim_\C(V_1)-24}{24}
\end{equation*}
for every simple component $\g_i$ of $V_1$ where $h_i^\vee$ is the dual Coxeter number of $\g_i$. It was pointed out in \cite{DM04} that this equation implies that there are only finitely many choices for the family of pairs $(\g_i,k_i)$ in the decomposition of $V_1$. In fact, there are 221 solutions (see \cite{EMS15}, Proposition~6.3, \cite{Sch93}). This equation is however not sufficient to restrict the number of possible cases to 69, which together with the trivial and abelian case for $V_1$ give the 71 affine structures in Schellekens' list \cite{Sch93}. In order to obtain his result, Schellekens had to employ more sophisticated techniques and these were translated into the rigorous language of \voa{}s in \cite{EMS15} in order to obtain Theorem~\ref{thm:schellekens}.

The above equation shows that the Lie algebra structure of $V_1$ already fully determines the levels $k_i$ (via $\dim_\C(V_1)$ and the dual Coxeter numbers $h^\vee_i$) and hence the affine structure of $V$. In order to identify a certain case on Schellekens' list it is hence sufficient to determine the Lie algebra structure on $V_1$.

\section{Orbifold Construction}

It is a priori not clear that each of the 71 affine structures on Schellekens' list can be realised, i.e.\ that there is a holomorphic \voa{} of central charge 24 with that Lie algebra structure on $V_1$.

All the cases on Schellekens' list known to exist that do not arise from lattice \voa{}s were obtained by cyclic orbifolding, i.e.\ starting from an already known holomorphic \voa{} $V$ and constructing the orbifolded \voa{} $\widetilde{V}$ of the same central charge associated with some automorphism $\sigma$ of $V$ of finite order. So far, orbifolding was only possible with automorphisms of order 2 and 3 but the general orbifold theory developed in this text (see Section~\ref{sec:orbifold}) allows us to consider cyclic orbifolds of arbitrary order.

The first non-lattice holomorphic \voa{} of central charge 24, the Monster \voa{} $V^\natural$, was in fact obtained in \cite{FLM88} as the orbifold of $V_\Lambda$ associated with a standard lift of the $(-1)$-involution on the Leech lattice $\Lambda$. An overview of the cases on Schellekens' list obtained from orbifold constructions is given in Section~\ref{sec:orbconsum}.

In this text we will restrict ourselves to orbifolds of lattice \voa{}s, i.e.\ we start from one of the 24 holomorphic \voa{}s of central charge 24 associated with the Niemeier lattices and consider orbifolds associated with lifts of lattice automorphisms (see Section~\ref{sec:latorbifold}). In this manner we give orbifold constructions for five affine structures on Schellekens' list that have not been constructed so far (see Theorem~\ref{thm:newcases}).

\minisec{Montague's Argument}

We describe an argument first considered by Montague \cite{Mon98} to identify the Lie algebra structure $\widetilde{V}_1$ of the orbifolded \voa{} $\widetilde{V}$ obtained in \eqref{eq:orbifold} under Assumption~\ref{ass:op}. By Proposition~\ref{prop:autvoalie} we know that $\sigma\in\Aut(V)$ restricts to a Lie algebra automorphism on $V_1$ of some order $k$ dividing $n=\ord(\sigma)$. This means that $V^G_1$ is the fixed-point Lie subalgebra of $V_1$ by an automorphism of order $k$ or in short a $\Z_k$-subalgebra of $V_1$.

There is a beautiful theory due to Kac on the possible fixed-point Lie subalgebras of a simple, finite-dimensional, complex Lie algebra \cite{Kac90}. We describe Kac's theory and the straightforward generalisation to the case of semisimple Lie algebras in Section~\ref{sec:kac} of the appendix.

By this theory we know exactly what the possible Lie algebra structures on $V^G_1$ are. This is of course not really helpful since knowing the Lie algebra structure of $V_1$ and the automorphism $\sigma$ we can explicitly determine the Lie algebra structure on $V_1^G$ and determine its isomorphism class (e.g.\ using \texttt{Magma} \cite{Magma}).

On the other hand, we know that $V^G_1$ is also a Lie subalgebra of $\widetilde{V}_1$. Together with knowing the dimension of $\widetilde{V}_1$ (from the character $\ch_{\widetilde{V}}(\tau)$, see Section~\ref{sec:latorbifold}) and the possible cases for $\widetilde{V}_1$ with that dimension from Schellekens' list we can sometimes already uniquely determine the Lie algebra structure and hence the affine structure of $\widetilde{V}_1$.\footnote{This is for example trivially the case in dimensions 132, 288 and 384, where there is only one possible affine structure.}

We can considerably strengthen this argument by noting that for the orbifolded \voa{} $\widetilde{V}$ there is an ``inverse'' automorphism $\amgis\in\Aut(\widetilde{V})$, also of order $n$, which gives back $V$ when orbifolding $\widetilde{V}$ with it (see Theorem~\ref{thm:invorb}). This means that $V^G_1=\widetilde{V}^{\langle\amgis\rangle}_1$ is not only a fixed-point Lie subalgebra of $V_1$ but also one of $\widetilde{V}_1$, namely by the restriction of $\amgis$ to $\widetilde{V}_1$, a Lie algebra automorphism of some order $k'$ dividing $n$.

Knowing that $V^G_1$, which we explicitly compute, is a fixed-point Lie subalgebra of $\widetilde{V}_1$, for which there are only a few possible cases on Schellekens' list in the computed dimension for $\widetilde{V}_1$, we can in many cases, using Kac's theory, uniquely determine the Lie algebra structure and hence the affine structure of $\widetilde{V}_1$. This will be demonstrated in the following examples.

\section{Five New Cases on Schellekens' List}\label{sec:newcases}

In this section we describe the construction of five new cases on Schellekens' list as cyclic orbifolds of lattice \voa{}s. In each of the following cases let $\nu$ be an automorphism of order $m$ on some Niemeier lattice $L$ such that a standard lift $\hat{\nu}\in\Aut(V_L)$ is of type $n\{0\}$ where $n=m$ or $2m$. We then consider the orbifold $\widetilde{V}$ of $V=V_L$ by $\hat{\nu}$. In each case we compute the dimension of $\widetilde{V}_1$ using the modular invariance of the trace functions as described in Section~\ref{sec:latorbifold}. Then, we determine the Lie algebra structure of $\widetilde{V}_1$ by Montague's argument described above, and sometimes using some additional results.

The five cases below give new cases on Schellekens' list, i.e.\ cases that were not yet constructed as of the completion of this text. However, Lam and Shimakura were independently searching for constructions of some cases on Schellekens' list. In \cite{LS16a} they announced to have found a construction for the case $A_1C_{5,3}G_{2,2}$ (probably identical to our construction) and were pursuing approaches similar to ours for the cases $C_{4,10}$ and $A_{2,6}D_{4,12}$ \cite{LS15c}.

\minisec{Automorphisms of Niemeier Lattices}

First we describe the automorphisms of the Niemeier lattices we will use. Let $L$ be a Niemeier lattice, i.e.\ one of the 24 even, unimodular, positive-definite lattice of rank 24. Assume that $L$ is not the Leech lattice. Recall that we write $L=N(Q)$ for the Niemeier lattice associated with the root lattice $Q$ (see Section~\ref{sec:lattice}). The Niemeier lattice $N(Q)$ can be obtained as the lattice generated by $Q$ and the glue vectors or \emph{glue code}, certain vectors in $Q'$ (see \cite{CS99}, Chapter~16.1). In the following we will use the glue code given in Table~16.1 of \cite{CS99}.

The root lattice $Q$ will be given as the orthogonal direct sum of the lattices $A_n$, $n\geq 1$, $D_n$, $n\geq 4$, and $E_6$, $E_7$ and $E_8$. These can be realised as embedded into $\Q^d$ with the standard Euclidean inner product for some $d\in\Ns$ (see \cite{CS99}, Sections 4.6, 4.7 and 4.8). In particular,
\begin{equation*}
A_n=\left\{(x_1,\ldots,x_{n+1})\in\Z^{n+1}\xmiddle|x_1+\ldots+x_{n+1}=0\right\}\subset\Z^{n+1}
\end{equation*}
for $n\geq 1$ and
\begin{equation*}
D_n=\left\{(x_1,\ldots,x_n)\in\Z^n\xmiddle|x_1+\ldots+x_n\in2\Z\right\}\subset\Z^n
\end{equation*}
for $n\geq 4$.

In the following we define automorphisms of $L=N(Q)$ by defining them on the root lattice $Q$ in such a way that they are compatible with the choice of glue code in the sense that their linear continuations are also automorphisms of $N(Q)$. This is in general a rather delicate matter but in most examples below there is a natural choice.

We denote by $\sigma_n$ the automorphism of $\Z^n$ permuting the coordinates as
\begin{equation*}
(x_1,\ldots,x_n)\mapsto(x_n,x_1,\ldots,x_{n-1}).
\end{equation*}

\begin{description}
\item[Case 1] Let $Q:=A_9^2D_6$, which is short for the orthogonal direct sum $Q=A_9\oplus A_9\oplus D_6$, and $\nu$ an automorphism of $Q$ acting as follows:
\begin{itemize}
\item On the two copies $A_9^{(1)}$ and $A_9^{(2)}$ of $A_9$, $\nu$ is a signed permutation of these, namely
\begin{equation*}
(\alpha^{(1)},\alpha^{(2)})\mapsto(-\alpha^{(2)},\alpha^{(1)})
\end{equation*}
for $\alpha^{(i)}\in A_9^{(i)}$, $i=1,2$ (so that $\nu^2$ is the $(-1)$-involution on each component). This gives a cycle shape of $2^{-9}4^9$.
\item On $D_6$, $\nu$ is a signed permutation of the coordinates of $D_6\subset\Z^6$
\begin{equation*}
(x_1,x_2,x_3,x_4,x_5,x_6)\mapsto(x_1,x_2,x_4,x_3,-x_6,x_5)
\end{equation*}
of cycle shape $1^24^1$.
\end{itemize}
Overall, this amounts to an automorphism of $Q$ of cycle shape $1^22^{-9}4^{10}$, which also defines an automorphism on $L=N(Q)$.

The automorphism $\nu$ is in one of exactly two conjugacy classes in $\Aut(L)$ with cycle shape $1^22^{-9}4^{10}$, both having length 1\,306\,368\,000.

\item[Case 2] Let $Q:=A_4^6$ be given by 6 copies of $A_4$ and $\nu$ the automorphism of $Q$ acting as follows:
\begin{itemize}
\item On the first copy $A_4^{(1)}$ of $A_4$, $\nu$ permutes the coordinates of $\Z^5$ as
\begin{equation*}
\alpha^{(1)}\mapsto\sigma_5\alpha^{(1)}
\end{equation*}
for $\alpha^{(1)}\in A_4^{(1)}$, giving a cycle shape of $1^{-1}5^1$.
\item On the last five copies $A_4^{(2)},\ldots,A_4^{(6)}$ of $A_4$, $\nu$ permutes these five copies, i.e.\
\begin{equation*}
(\alpha^{(2)},\ldots,\alpha^{(6)})\mapsto(\alpha^{(6)},\alpha^{(2)},\ldots,\alpha^{(5)})
\end{equation*}
for $\alpha^{(i)}\in A_4^{(i)}$, $i=2,\ldots,6$, contributing with $5^4$ to the cycle shape.
\end{itemize}
In total, this gives an automorphism of $Q$ of cycle shape $1^{-1}5^5$, which also defines an automorphism of $L=N(Q)$.

The automorphism $\nu$ belongs to the unique conjugacy class in $\Aut(L)$ with cycle shape $1^{-1}5^5$ and length 119\,439\,360\,000 (there is another conjugacy class with the same cycle shape of length 47\,775\,744).

\item[Case 3] Let $Q:=A_2^{12}$. The automorphism $\nu$ acts on the 12 copies $A_2^{(1)},\ldots,A_2^{(12)}$ as follows:
\begin{itemize}
\item On $A_2^{(4)}\cong A_2$, $\nu$ acts as permutation of the three coordinates of $\Z^3$ times $-1$, i.e
\begin{equation*}
\alpha^{(4)}\mapsto-\sigma_3\alpha^{(4)}
\end{equation*}
for $\alpha^{(4)}\in A_4^{(4)}$ with cycle shape $1^12^{-1}3^{-1}6^1$.
\item On $A_2^{(1)}\oplus A_2^{(5)}\cong A_2^2$ let $\nu$ permute the two copies of $A_2$ and also permute the coordinates and multiply by $-1$, i.e.\
\begin{equation*}
(\alpha^{(1)},\alpha^{(5)})\mapsto(-\sigma_3\alpha^{(5)},-\sigma_3\alpha^{(1)})
\end{equation*}
for $\alpha^{(i)}\in A_2^{(i)}$, $i=1,5$ (so that $\nu^2$ acts as permutation of order $3$ on each copy of $A_2$), which gives cycle shape $2^{-1}6^1$.
\item On $A_2^{(6)}\oplus A_2^{(8)}\oplus A_2^{(11)}\cong A_2^3$ let $\nu$ permute the three copies as $(6\;11\;8)$ and also multiply by $-1$, i.e.
\begin{equation*}
(\alpha^{(6)},\alpha^{(8)},\alpha^{(11)})\mapsto(-\alpha^{(8)},-\alpha^{(11)},-\alpha^{(6)})
\end{equation*}
for $\alpha^{(i)}\in A_2^{(i)}$, $i=6,8,11$ (so that $\nu^3$ acts as $(-1)$-involution on each copy of $A_2$), which contributes with $3^{-2}6^2$ to the cycle shape.
\item On $A_2^{(2)}\oplus A_2^{(3)}\oplus A_2^{(7)}\oplus A_2^{(9)}\oplus A_2^{(10)}\oplus A_2^{(12)}\cong A_2^6$ let $\nu$ permute the six copies of $A_2$ as $(2\;10\;9\;7\;12\;3)$, i.e.\
\begin{equation*}
(\alpha^{(2)},\alpha^{(3)},\alpha^{(7)},\alpha^{(9)},\alpha^{(10)},\alpha^{(12)})\mapsto(\alpha^{(3)},\alpha^{(12)},\alpha^{(9)},\alpha^{(10)},\alpha^{(2)},\alpha^{(7)})
\end{equation*}
for $\alpha^{(i)}\in A_2^{(i)}$, $i=2,3,7,9,10,12$ (so that $\nu^6$ acts as identity on each copy of $A_2$), giving cycle shape $6^2$.
\end{itemize}
Altogether, this gives an automorphism of $Q$ of cycle shape $1^12^{-2}3^{-3}6^6$, which also defines an automorphism of $L=N(Q)$.

The automorphism $\nu$ belongs to the unique conjugacy class in $\Aut(L)$ with cycle shape $1^12^{-2}3^{-3}6^6$ and length 106\,420\,469\,760 (there is another conjugacy class with the same cycle shape of length 656\,916\,480).

\item[Case 4] Let $Q:=E_6^4$. We define an automorphism on $Q$ as follows:
\begin{itemize}
\item On the last three copies $E_6^{(2)},\ldots,E_6^{(4)}$ of $E_6$ let $\nu$ act as permutation of these times $-1$, i.e.\
\begin{equation*}
(\alpha^{(2)},\alpha^{(3)},\alpha^{(4)})\mapsto(-\alpha^{(4)},-\alpha^{(2)},-\alpha^{(3)})
\end{equation*}
for $\alpha^{(i)}\in E_6^{(i)}$, $i=2,3,4$ (so that $\nu^3$ is simply the $(-1)$-involution on each component). This gives a cycle shape of $3^{-6}6^6$.
\item We realise the first copy of $E_6$ in a special way by glueing together three copies of $A_2$. Consider $A_2^3$, which is a sublattice of $E_6$ of full rank. We define the automorphism $\nu$ on each copy of $A_2$ as permutation of the coordinates in $\Z^3$ times $-1$, i.e.\
\begin{equation*}
\alpha\mapsto-\sigma_3\alpha
\end{equation*}
for $\alpha\in A_2$ so that $\nu$ has cycle shape $1^12^{-1}3^{-1}6^1$ on each $A_2$ (and $\nu^3$ acts as $(-1)$-involution). Let the glue code be generated by $[111]$, using the notation in \cite{CS99}. Then indeed, the $\Z$-span of $A_2^3$ and $[111]$ is isomorphic to $E_6$ and $\nu$ defines an automorphism of $E_6$ of cycle shape $1^32^{-3}3^{-3}6^3$.
\end{itemize}
In total, this gives an automorphism of $Q$ of cycle shape $1^32^{-3}3^{-9}6^9$, which also defines an automorphism of $L=N(Q)$.

The automorphism $\nu$ belongs to the unique conjugacy class in $\Aut(L)$ with cycle shape $1^32^{-3}3^{-9}6^9$ and length 1\,719\,926\,784\,000 (there are three more classes with the same cycle shape of lengths 2\,048\,000, 8\,847\,360\,000 and 35\,389\,440\,000).

\item[Case 5] Again, let $Q:=A_4^6$ be given by 6 copies of $A_4$. We define an automorphism on $Q$ as follows:
\begin{itemize}
\item On the first copy $A_4^{(1)}$ of $A_4$, $\nu$ permutes the coordinates of $\Z^5$ times $-1$, i.e.\
\begin{equation*}
\alpha^{(1)}\mapsto-\sigma_5\alpha^{(1)}
\end{equation*}
for $\alpha^{(1)}\in A_4^{(1)}$, giving a cycle shape of $1^12^{-1}5^{-1}10^1$.
\item On the last five copies $A_4^{(2)},\ldots,A_4^{(6)}$ of $A_4$, $\nu$ permutes these five copies times $-1$, i.e.\
\begin{equation*}
(\alpha^{(2)},\ldots,\alpha^{(6)})\mapsto(-\alpha^{(6)},-\alpha^{(2)},\ldots,-\alpha^{(5)})
\end{equation*}
for $\alpha^{(i)}\in A_4^{(i)}$, $i=2,\ldots,6$, contributing with $5^{-4}10^4$ to the cycle shape.
\end{itemize}
Altogether, this gives an automorphism of $Q$ of cycle shape $1^12^{-1}5^{-5}10^5$, which also defines an automorphism of $L=N(Q)$.

The automorphism $\nu$ belongs to the unique conjugacy class in $\Aut(L)$ with cycle shape $1^12^{-1}5^{-5}10^5$ and length 119\,439\,360\,000 (there is another conjugacy class with the same cycle shape of length 47\,775\,744).
\end{description}

\minisec{Orbifold Construction}

Using lifts of the lattice automorphisms described above, we obtain the following orbifold constructions.

\begin{description}
\item[Case 1: Affine Structure $A_2B_2E_{6,4}$]~\\
Let $L=N(A_9^2D_6)$ and $\nu$ the automorphism of cycle shape $1^22^{-9}4^{10}$ described above. Then $\nu$ lifts to an automorphism $\hat{\nu}\in\Aut(V_L)$ of type $4\{0\}$, i.e.\ no order doubling occurs.

We compute $\dim_\C(\widetilde{V}_1)=96$. By Theorem~\ref{thm:schellekens}, $\widetilde{V}_1$ is isomorphic as a Lie algebra to $A_2^{12}$, $B_2^4D_4^2$, $A_2^2A_5^2B_2$, $A_2^2A_8$ or $A_2B_2E_6$. The Lie algebra structure of $V_1^{\hat{\nu}}$ is $A_2B_2D_5\C^1$. Clearly, $A_2B_2D_5\C^1$ cannot embed into $A_2^{12}$, $B_2^4D_4^2$ or $A_2^2A_5^2B_2$. In addition, $D_5$ cannot appear in an invariant Lie subalgebra of $A_8$ which rules out $A_2^2A_8$. This leaves only $\widetilde{V}_1\cong A_2B_2E_6$.

\item[Case 2: Affine Structure $A_{4,5}^2$]~\\
Let $L=N(A_4^6)$ and $\nu$ the automorphism of cycle shape $1^{-1}5^5$ described above. Then $\nu$ lifts to an automorphism $\hat{\nu}\in\Aut(V_L)$ of type $5\{0\}$, i.e.\ no order doubling occurs.

We compute $\dim_\C(\widetilde{V}_1)=48$ (cf.\ example calculation in Section~\ref{sec:latorbifold}). Theorem~\ref{thm:schellekens} implies that $\widetilde{V}_1$ is isomorphic as a Lie algebra to $A_1^{16}$, $A_2^6$, $A_1A_3^3$, $A_4^2$, $A_1A_5B_2$, $A_1D_5$ or $A_6$. The Lie algebra structure of $V_1^{\hat{\nu}}$ is $A_4\C^4$. This leaves only $A_1A_5B_2$ or $A_4^2$ as possible Lie algebra structures of $\widetilde{V}_1$, both of rank 8.

One can show that the $A_4$ in $V_1^{\hat{\nu}}\cong A_4\C^4$ is an ideal in $\widetilde{V}_1$, which means that it appears as a summand in the decomposition of $\widetilde{V}_1$ into simple components. This leaves only $\widetilde{V}_1\cong A_4^2$.

\begin{proof}[Proof of Claim]
Let us describe the Lie algebras $V_1=(V_L)_1$ and $V_1^{\hat{\nu}}$ in more detail. $V_L$ is the \voa{} associated with the Niemeier lattice $L=N(A_4^6)$. The subspace of $V_L$ of elements of weight 1 has dimension $144=24+120$ and is spanned by
\begin{equation*}
h(-1)\otimes\ee_0\quad\text{and}\quad1\otimes\ee_\alpha
\end{equation*}
for $h\in\h$ and $\alpha\in\Phi(L)$, the set of roots of $L$. It is well known that $L$ has 120 roots, spanning the root lattice $Q=A_4^6$.

Recall that we wrote $A_4^6=A_4\oplus A_4^5$, where $\nu\in\Aut(L)$ acts separately on both direct summands. Clearly, $\nu$ permutes the set of roots $\Phi(L)$. It has 24 orbits of length 5 and 4 of these orbits, represented by $\alpha_1,\ldots,\alpha_4$, live in the first copy of $A_4$, 20 in $A_4^5$, represented by $\beta_1,\ldots,\beta_{20}$. For $\alpha\in L$ we define
\begin{equation*}
E_\alpha:=\sum_{i=0}^4\hat{\nu}^i(\ee_\alpha).
\end{equation*}
Then $V^{\hat{\nu}}_1$ is spanned by
\begin{equation*}
h(-1)\otimes\ee_0\quad\text{and}\quad1\otimes E_\alpha
\end{equation*}
for $h\in\h_{(0)}$ (of dimension 4) and $\alpha\in\{\alpha_1,\ldots,\alpha_4,\beta_1,\ldots,\beta_{20}\}$. Hence $V^{\hat{\nu}}_1$ has dimension 28.

It is easy to see that the $1\otimes E_{\beta_i}$, $i=1,\ldots,20$, together with $h(-1)\otimes\ee_0$, $h\in\h_{(0)}$, span a Lie subalgebra of $V_1^{\hat{\nu}}$, which is isomorphic as a Lie algebra to $A_4$ with Cartan subalgebra given by the $h(-1)\otimes\ee_0$, $h\in\h_{(0)}$.

We now want to show that this subalgebra $A_4$ is an ideal in $V_1^{\hat{\nu}}$ and even in $\widetilde{V}_1$. It is easy to see that the Lie bracket between $A_4$ and the $1\otimes E_{\alpha_i}$, $i=1,\ldots,4$, vanishes, which follows essentially from the fact that the two direct summands of $Q=A_4\oplus A_4^5$ are orthogonal and $\h_{(0)}$ is a subspace of the $\C$-span of $A_4^5$. This shows that the subalgebra $A_4$ is an ideal in $V_1^{\hat{\nu}}$.

To see that $A_4$ is an ideal in all of $\widetilde{V}_1$, we still need to consider the Lie bracket between $A_4$ and the twisted contributions to $\widetilde{V}_1$ from $V_L(\hat{\nu}^i)$, $i\in\Z_5\setminus\{0\}$. Then Lemma~2.2.2 in \cite{SS16} states that $(h(-1)\otimes\ee_0)_0$, $h\in\h$, and $(1\otimes\ee_\alpha)_0$ for $\alpha\in\Phi(L)\setminus N$ act as 0 on the lowest-weight space of the twisted module $V_L(\hat{\nu})$ where $N=L_\nu$.
For the given example we calculate the conformal weight of $V_L(\hat{\nu})$ to be 1 using formula \eqref{eq:vacanomaly}. Hence $V_L(\hat{\nu})_1$ is the lowest-weight space. Moreover, $\Phi(L)\setminus N$ are exactly the 100 roots living in the summand $A_4^5$ of $Q=A_4\oplus A_4^5$ (and $\Phi(L)\cap N$ are the 20 roots inside $A_4$).

Indeed, the first copy of $A_4$ in $A_4\oplus A_4^5$ does not contribute to $\h_{(0)}$ and hence all 20 roots in this $A_4$ are in $N$. The vectors $(\alpha,\ldots,\alpha)\in A_4^5$ for $\alpha\in A_4$ span $\h_{(0)}$ and hence the contribution to $\h^\bot_{(0)}$ from these five copies is spanned by $(\alpha^{(2)},\ldots,\alpha^{(6)})$ with the condition that $\alpha^{(2)}+\ldots+\alpha^{(6)}=0$. None of the 100 roots in $A_4^5$ fulfil this requirement and so these roots are not in $N=L\cap\h^\bot_{(0)}$.

Taking all this together, we see that the Lie bracket between the Lie subalgebra $A_4$ from above and the contribution to $\widetilde{V}_1$ from the twisted module $V_L(\hat{\nu})$ vanishes. The same is true for $V_L(\hat{\nu}^2),\ldots,V_L(\hat{\nu}^4)$. Hence, in total, the subalgebra $A_4$ is an ideal in $\widetilde{V}_1$.
\end{proof}

\item[Case 3: Affine Structure $A_{2,6}D_{4,12}$]~\\
Let $L=N(A_2^{12})$ and $\nu$ the automorphism of cycle shape $1^12^{-2}3^{-3}6^6$ described above. Then $\nu$ lifts to an automorphism $\hat{\nu}\in\Aut(V_L)$ of type $6\{0\}$, i.e.\ no order doubling occurs.

One computes $\dim_\C(\widetilde{V}_1)=36$. Then by Theorem~\ref{thm:schellekens} we know that $\widetilde{V}_1$ is isomorphic as a Lie algebra to $A_1^{12}$, $A_2D_4$ or $C_4$. The Lie algebra structure of $V_1^{\hat{\nu}}$ is $A_1A_2\C^3$. Clearly $A_1A_2\C^3$ can only embed into $A_2D_4$ of these three possible cases. Hence $\widetilde{V}_1\cong A_2D_4$.

\item[Case 4: Affine Structure $A_1C_{5,3}G_{2,2}$]~\\
Note that in \cite{LS16a} it is announced that the authors have also found a $\Z_6$-orbifold construction of $A_1C_{5,3}G_{2,2}$ starting from the lattice $E_6^4$, which is probably identical to our construction below.

Let $L=N(E_6^4)$ and $\nu$ the automorphism of cycle shape $1^32^{-3}3^{-9}6^9$ described above. Then $\nu$ lifts to an automorphism $\hat{\nu}\in\Aut(V_L)$ of type $6\{0\}$, i.e.\ no order doubling occurs.

We calculate $\dim_\C(\widetilde{V}_1)=72$. By Theorem~\ref{thm:schellekens}, $\widetilde{V}_1$ is isomorphic as a Lie algebra to $A_1^{24}$, $A_1^4A_3^4$, $A_1^3A_5D_4$, $A_1^2C_3D_5$, $A_1^3A_7$, $A_1C_5G_2$ or $A_1^2D_6$. The Lie algebra structure of $V_1^{\hat{\nu}}$ is $A_1A_2C_4\C^1$. The $C_4$ can only be a $\Z_k$-subalgebra, $k\mid 6$, of $C_5$ and $A_7$ of those simple components appearing in dimension 72. This leaves only $A_1^3 A_7$ and $A_1 C_5 G_2$ as possible cases. The latter is clearly possible. Assume it were the first one. Then $A_7$ would have to have $C_4$ as fixed-point Lie subalgebra via the twisted diagram $A_7^{(2)}$. But then $A_2$ could not appear in the fixed-point Lie subalgebra. Hence $\widetilde{V}_1\cong A_1C_5G_2$.

\item[Case 5: Affine Structure $C_{4,10}$]~\\
Let $L=N(A_4^6)$ and $\nu$ the automorphism of cycle shape $1^12^{-1}5^{-5}10^5$ described above. Then $\nu$ lifts to an automorphism $\hat{\nu}\in\Aut(V_L)$ of type $10\{0\}$, i.e.\ no order doubling occurs.

We calculate $\dim_\C(\widetilde{V}_1)=36$. Again, by Theorem~\ref{thm:schellekens} we know that $\widetilde{V}_1$ is isomorphic as a Lie algebra to $A_1^{12}$, $A_2D_4$ or $C_4$. The Lie algebra structure of $V_1^{\hat{\nu}}$ is $B_2\C^2$, which cannot be a fixed-point Lie subalgebra of $A_1^{12}$. This leaves only $C_4$ and $A_2D_4$ as possible Lie algebra structures of $\widetilde{V}_1$.
Both have $B_2\C^2$ as possible $\Z_{10}$-subalgebra, so we cannot rule out one of the two cases directly. We can however compute the dimensions of the spaces $W_1^{(i,0)}$ in the decomposition $\widetilde{V}_1:=\bigoplus_{i\in\Z_n}W_1^{(i,0)}$ via the characters and it turns out that $W_1^{(5,0)}=\{0\}$, which we can also see from the fact that $\rho(V(\hat{\nu}^5))>1$. This means that $V_1^{\hat{\nu}}=W_1^{(0,0)}=W_1^{(0,0)}\oplus W_1^{(5,0)}$ is also a $\Z_5$-subalgebra of $\widetilde{V}_1$, namely the fixed points under the square of the automorphism on $\widetilde{V}$ describing the inverse orbifold. But $B_2\C^2$ cannot be a $\Z_5$-subalgebra of $A_2D_4$. This leaves only $\widetilde{V}_1\cong C_4$.
\end{description}

We summarise our findings in the following theorem:
\begin{oframed}
\begin{thm}\label{thm:newcases}
There exist holomorphic, $C_2$-cofinite \voa{}s of CFT-type of central charge 24 with the following affine structures:
\renewcommand{\arraystretch}{1.2}
\begin{equation*}
\begin{tabular}{r|l|r|r||l|c}
\cite{Sch93} No.\ & Aff.\ struct.\ & Dim. & Rk. & Lattice $L$ & Aut.\ $\nu$\\\hline
28 & $A_2B_2E_{6,4}$     & 96  & 10 & $N(A_9^2D_6)$ & $1^22^{-9}4^{10}$\\
9  & $A_{4,5}^2$         & 48  & 8  & $N(A_4^6)$    & $1^{-1}5^5$\\
3  & $A_{2,6}D_{4,12}$   & 36  & 6  & $N(A_2^{12})$ & $1^12^{-2}3^{-3}6^6$\\
21 & $A_1C_{5,3}G_{2,2}$ & 72  & 8  & $N(E_6^4)$    & $1^32^{-3}3^{-9}6^9$\\
4  & $C_{4,10}$          & 36  & 4  & $N(A_4^6)$    & $1^12^{-1}5^{-5}10^5$
\end{tabular}
\end{equation*}
\renewcommand{\arraystretch}{1}
\end{thm}
\end{oframed}

\section{Summary}\label{sec:orbconsum}

Cyclic orbifolding with lifts of lattice automorphisms starting from the 24 \voa{}s associated with the Niemeier lattices produces at least 61 of the 71 cases on Schellekens' list (see Figure~\ref{fig:schellekens2}):
\begin{itemize}
 \item 24 \voa{}s associated with the Niemeier lattices \cite{Nie73,Ven80,Bor86,FLM88,Don93},
 \item 15 $\Z_2$-orbifolds using lifts of the $(-1)$-involution \cite{FLM88,DGM90,DGM96},
 \item 3 $\Z_3$-orbifolds \cite{Miy13b,SS16,ISS15},
 \item 19 new $\Z_n$-orbifolds with $n=2,4,5,6,10$.
\end{itemize}
This leaves 10 cases we have not been able to construct using lifts of lattice automorphisms, namely $A_{6,7}$, $A_{1,2}D_{5,8}$, $A_{1,2}A_{5,6}B_{2,3}$, $A_{2,2}F_{4,6}$, $A_1^2D_{6,5}$, $A_1^3 A_{7,4}$, $A_2^2A_{8,3}$, $A_3C_{7,2}$, $A_3D_{7,3}G_2$ and $A_5E_{7,3}$.

In addition to orbifolding by lifts of lattice automorphisms there is also a second approach, using framed \voa{}s \cite{Lam11,LS12a,LS15a}. With the framed construction it is possible to construct 56 (and not more) of the cases on Schellekens' list (see Figure~\ref{fig:schellekens3}).
There are 3 cases which can be realised as framed \voa{}s but have not yet been constructed as lattice orbifolds, namely $A_{1,2}D_{5,8}$, $A_1^3A_{7,4}$ and $A_3C_{7,2}$.

Thirdly, in \cite{LS16a} the authors construct new cases on Schellekens' list by considering orbifolds by inner automorphisms of order 2 of non-lattice cases on Schellekens' list. They obtain the cases $A_3D_{7,3}G_2$, $A_5E_{7,3}$ and $A_{8,3}A_{2,1}^2$. Moreover, they construct the cases $A_{1,2}A_{5,6}B_{2,3}$ and $A_1^2D_{6,5}$ conditionally on the construction of the cases $A_1C_{5,3}G_{2,2}$ and $A_{4,5}^2$, respectively, in this text.
\begin{cor}[\cite{LS16a}]
There exist holomorphic, $C_2$-cofinite \voa{}s of CFT-type of central charge 24 with the affine structures $A_{1,2}A_{5,6}B_{2,3}$ and $A_1^2D_{6,5}$.
\end{cor}
Finally, in \cite{LS16} the authors construct the case $A_{6,7}$ by orbifolding the Leech lattice \voa{} $V_\Lambda$ by a product of an inner automorphism of order 7 and a standard lift of an order 7 lattice automorphism. This result depends on the general orbifold theory developed in this text.

Taking all three methods together, we find that 70 of the 71 cases on Schellekens' list are constructed so far, namely (see Figure~\ref{fig:schellekens1}):
\begin{itemize}
\item 24 \voa{}s associated with the Niemeier lattices \cite{Nie73,Ven80,Bor86,FLM88,Don93},
\item 15 $\Z_2$-orbifolds using standard lifts of the $(-1)$-involution on the Niemeier lattice \voa{}s \cite{FLM88,DGM90,DGM96},
\item 17 framed \voa{}s \cite{Lam11,LS12a,LS15a},
\item 3 $\Z_3$-orbifolds \cite{Miy13b,SS16,ISS15} starting from Niemeier lattice \voa{}s,
\item 5 $\Z_2$-orbifolds using inner automorphisms \cite{LS16a}, partly depending on constructions in this text,
\item 5 new $\Z_n$-orbifolds with $n=4,5,6,8,10$ starting from Niemeier lattice \voa{}s (see Theorem~\ref{thm:newcases}),
\item 1 $\Z_7$-orbifold starting from the Leech lattice using the product of a lattice and an inner automorphism \cite{LS16}, depending on the general orbifold theory developed in this text.
\end{itemize}

Summarising all constructions and using Theorem~\ref{thm:schellekens} we arrive at:
\begin{oframed}
\begin{thm}\label{thm:summary}
Let $V$ be a holomorphic, $C_2$-cofinite \voa{} of CFT-type of central charge $c=24$. Then there are at least 70 and at most 71 possible affine structures for $V$. These are the 71 cases on Schellekens' list except for maybe $A_{2,2}F_{4,6}$.
\end{thm}
\end{oframed}
Recently, Ching Hung Lam and Xingjun Lin have announced the construction of the last remaining case $A_2F_4$ using mirror extensions \cite{LL16}.

\newgeometry{bottom=2cm,top=2cm}
\thispagestyle{empty}
\begin{sidewaysfigure}
\resizebox{1.0\textwidth}{!}{
\begin{tikzpicture}[scale=1.1]
\tikzstyle{lat}=[ellipse,thick,draw=blue!75,fill=blue!20,inner sep=0.05,align=center,minimum size=6mm]
\tikzstyle{invol}=[ellipse,thick,draw=cyan!75,fill=cyan!20,inner sep=0.05,align=center,minimum size=6mm]
\tikzstyle{z3}=[ellipse,thick,draw=green!75,fill=green!20,inner sep=0.05,align=center,minimum size=6mm]
\tikzstyle{new}=[ellipse,thick,draw=yellow!75,fill=yellow!20,inner sep=0.05,align=center,minimum size=6mm]
\tikzstyle{none}=[ellipse,thick,draw=black!75,fill=black!20,inner sep=0.05,align=center,minimum size=6mm]
\tikzstyle{framed}=[ellipse,thick,draw=red!75,fill=red!20,inner sep=0.1,align=center,minimum size=6mm]
\tikzstyle{inner}=[ellipse,thick,draw=orange!75,fill=orange!20,inner sep=0.1,align=center,minimum size=6mm]

\draw (-1.5,-0.5) -- (29.5,-0.5);
\draw (-1.5,1.5) -- (29.5,1.5);
\draw (-1.5,4.5) -- (29.5,4.5);
\draw (-1.5,6.5) -- (29.5,6.5);
\draw (-1.5,8.5) -- (29.5,8.5);
\draw (-1.5,10.5) -- (29.5,10.5);
\draw (-1.5,12.5) -- (29.5,12.5);
\draw (-1.5,13.5) -- (29.5,13.5);
\draw[line width=2pt] (-1.5,14.5) -- (29.5,14.5);

\draw[line width=2pt] (-0.5,15.5) -- (-0.5,-0.5);
\draw (0.5,15.5) -- (0.5,-0.5);
\draw (1.5,15.5) -- (1.5,-0.5);
\draw (2.5,15.5) -- (2.5,-0.5);
\draw (3.5,15.5) -- (3.5,-0.5);
\draw (4.5,15.5) -- (4.5,-0.5);
\draw (5.5,15.5) -- (5.5,-0.5);
\draw (6.5,15.5) -- (6.5,-0.5);
\draw (7.5,15.5) -- (7.5,-0.5);
\draw (8.5,15.5) -- (8.5,-0.5);
\draw (9.5,15.5) -- (9.5,-0.5);
\draw (10.5,15.5) -- (10.5,-0.5);
\draw (11.5,15.5) -- (11.5,-0.5);
\draw (12.5,15.5) -- (12.5,-0.5);
\draw (13.5,15.5) -- (13.5,-0.5);
\draw (14.5,15.5) -- (14.5,-0.5);
\draw (15.5,15.5) -- (15.5,-0.5);
\draw (16.5,15.5) -- (16.5,-0.5);
\draw (17.5,15.5) -- (17.5,-0.5);
\draw (18.5,15.5) -- (18.5,-0.5);
\draw (19.5,15.5) -- (19.5,-0.5);
\draw (20.5,15.5) -- (20.5,-0.5);
\draw (21.5,15.5) -- (21.5,-0.5);
\draw (22.5,15.5) -- (22.5,-0.5);
\draw (23.5,15.5) -- (23.5,-0.5);
\draw (24.5,15.5) -- (24.5,-0.5);
\draw (25.5,15.5) -- (25.5,-0.5);
\draw (26.5,15.5) -- (26.5,-0.5);
\draw (27.5,15.5) -- (27.5,-0.5);
\draw (28.5,15.5) -- (28.5,-0.5);
\draw (29.5,15.5) -- (29.5,-0.5);

\node at (-1,0.5) {24};
\node at (-1,3) {16};
\node at (-1,5.5) {12};
\node at (-1,7.5) {10};
\node at (-1,9.5) {8};
\node at (-1,11.5) {6};
\node at (-1,13) {4};
\node at (-1,14) {0};

\node at (0,15) {0};
\node at (1,15) {24};
\node at (2,15) {36};
\node at (3,15) {48};
\node at (4,15) {60};
\node at (5,15) {72};
\node at (6,15) {84};
\node at (7,15) {96};
\node at (8,15) {108};
\node at (9,15) {120};
\node at (10,15) {132};
\node at (11,15) {144};
\node at (12,15) {156};
\node at (13,15) {168};
\node at (14,15) {192};
\node at (15,15) {216};
\node at (16,15) {240};
\node at (17,15) {264};
\node at (18,15) {288};
\node at (19,15) {300};
\node at (20,15) {312};
\node at (21,15) {336};
\node at (22,15) {360};
\node at (23,15) {384};
\node at (24,15) {408};
\node at (25,15) {456};
\node at (26,15) {552};
\node at (27,15) {624};
\node at (28,15) {744};
\node at (29,15) {1128};

\draw (-1.5,15.5) -- (-0.5,14.5);
\node at (-0.75,15.25) {D.};
\node at (-1.25,14.75) {Rk.};

\node[invol] at (0,14) (0) {$0$};

\node[new] at (2,13) {$C_{4, 10}$};

\node[new] at (2,11.5) {$A_{2, 6}$\\$D_{4, 12}$};
\node[inner] at (3,12) {$A_{6, 7}$};
\node[framed] at (3,11) {$A_{1, 2}$\\$D_{5, 8}$};
\node[none] at (4,11.5) {$A_{2, 2}$\\$F_{4, 6}$};

\node[inner] at (3,10) {$A_{1, 2}A_{5, 6}$\\$B_{2, 3}$};
\node[new] at (3,9) {$A_{4, 5}^{2}$};
\node[inner] at (5,10) {$A_{1}^{2}D_{6, 5}$};
\node[new] at (5,9) {$A_{1}C_{5, 3}$\\$G_{2, 2}$};

\node[framed] at (3,7.5) {$A_{1, 2}$\\$A_{3, 4}^{3}$};
\node[none,framed] at (5,8) {$A_{1}^3A_{7, 4}$};
\node[framed] at (5,7) {$A_{1}^2C_{3, 2}$\\$D_{5, 4}$};
\node[new] at (7,7.5) {$A_2B_2$\\$E_{6, 4}$};
\node[framed] at (9,7.5) {$A_3C_{7, 2}$};

\node[invol] at (2,5.5) {$A_{1, 4}^{12}$};
\node[z3] at (3,5.5) {$A_{2, 3}^{6}$};
\node[framed] at (4,6) {$A_{2, 2}^{4}$\\$D_{4, 4}$};
\node[invol] at (4,5) {$B_{2, 2}^{6}$};
\node[z3] at (5,5.5) {$A_{1}^{3}$\\$A_{5, 3}$\\$D_{4, 3}$};
\node[framed] at (6,6) {$A_{4, 2}^{2}$\\$C_{4, 2}$};
\node[invol] at (6,5) {$B_{3, 2}^{4}$};
\node[inner] at (7,5.5) {$A_{2}^{2}$\\$A_{8, 3}$};
\node[invol] at (8,5.5) {$B_{4, 2}^{3}$};
\node[new,inner] at (9,6) {$A_{3}D_{7, 3}$\\$G_{2}$};
\node[z3] at (9,5) {$E_{6, 3}G_{2}^{3}$};
\node[framed] at (10,5.5) {$A_{8, 2}F_{4, 2}$};
\node[invol] at (12,5.5) {$B_{6, 2}^{2}$};
\node[inner] at (13,5.5) {$A_{5}$\\$E_{7, 3}$};
\node[invol] at (19,5.5) {$B_{12, 2}$};

\node[invol] at (3,3) {$A_{1, 2}^{16}$};
\node[invol] at (5,3) {$A_{1}^{4}A_{3, 2}^{4}$};
\node[framed] at (7,3.5) {$A_{2}^{2}A_{5, 2}^{2}$\\$B_{2}$};
\node[invol] at (7,2.5) {$B_{2}^{4}D_{4, 2}^{2}$};
\node[invol] at (9,3.5) {$A_{3}^{2}D_{5, 2}^{2}$};
\node[framed] at (9,2.5) {$A_{3}A_{7, 2}$\\$C_{3}^{2}$};
\node[framed] at (11,4) {$A_{4}A_{9, 2}$\\$B_{3}$};
\node[invol] at (11,3) {$B_{3}^{2}C_{4}$\\$D_{6, 2}$};
\node[invol] at (11,2) {$C_{4}^{4}$};
\node[framed] at (13,3) {$A_{5}C_{5}$\\$E_{6, 2}$};
\node[invol] at (14,3.5) {$B_{4}^{2}D_{8, 2}$};
\node[framed] at (14,2.5) {$B_{4}C_{6}^{2}$};
\node[invol] at (15,3) {$A_{7}D_{9, 2}$};
\node[framed] at (16,3.5) {$B_{5}E_{7, 2}$\\$F_{4}$};
\node[framed] at (16,2.5) {$C_{8}F_{4}^{2}$};
\node[framed] at (18,3) {$B_{6}$\\$C_{10}$};
\node[framed] at (23,3) {$B_{8}E_{8, 2}$};

\node[lat] at (1,0.5) (Leech) {$\mathbb{C}^{24}$};
\node[lat] at (5,0.5) (24A1) {$A_{1}^{24}$};
\node[lat] at (7,0.5) (12A2) {$A_{2}^{12}$};
\node[lat] at (9,0.5) (8A3) {$A_{3}^{8}$};
\node[lat] at (11,0.5) (6A4) {$A_{4}^{6}$};
\node[lat] at (13,1) (4A5D4) {$A_{5}^{4}D_{4}$};
\node[lat] at (13,0) (6D4) {$D_{4}^{6}$};
\node[lat] at (14,0.5) (4A6) {$A_{6}^{4}$};
\node[lat] at (15,0.5) (2A72D5) {$A_{7}^{2}D_{5}^{2}$};
\node[lat] at (16,0.5) (3A8) {$A_{8}^{3}$};
\node[lat] at (17,1) (2A9D6) {$A_{9}^{2}D_{6}$};
\node[lat] at (17,0) (4D6) {$D_{6}^{4}$};
\node[lat] at (20,1) (A11D7E6) {$A_{11}D_{7}$\\$E_{6}$};
\node[lat] at (20,0) (4E6) {$E_{6}^{4}$};
\node[lat] at (21,0.5) (2A12) {$A_{12}^{2}$};
\node[lat] at (22,0.5) (3D8) {$D_{8}^{3}$};
\node[lat] at (24,0.5) (A15D9) {$A_{15}D_{9}$};
\node[lat] at (25,1) (A17E7) {$A_{17}E_{7}$};
\node[lat] at (25,0) (D102E7) {$D_{10}E_{7}^{2}$};
\node[lat] at (26,0.5) (2D12) {$D_{12}^{2}$};
\node[lat] at (27,0.5) (A24) {$A_{24}$};
\node[lat] at (28,1) (3E8) {$E_{8}^{3}$};
\node[lat] at (28,0) (D16E8) {$D_{16}E_{8}$};
\node[lat] at (29,0.5) (D24) {$D_{24}$};

\draw [fill=white] (23,6) rectangle (29,14);

\node[lat] at (24,13) {};
\node[invol] at (24,12) {};
\node[z3] at (24,11) {};
\node[framed] at (24,10) {};
\node[inner] at (24,9) {};
\node[new] at (24,8) {};
\node[none] at (24,7) {};

\node[anchor=west] at (25,13) {Niemeier lattice};
\node[anchor=west] at (25,12) {$(-1)$-involution orbifold};
\node[anchor=west] at (25,11) {$\mathbb{Z}_3$-orbifold};
\node[anchor=west] at (25,10) {framed construction};
\node[anchor=west] at (25,9) {inner automorphisms};
\node[anchor=west] at (25,8) {new lattice orbifold};
\node[anchor=west] at (25,7) {not yet constructed};
\end{tikzpicture}
}
\caption{Schellekens' list of the 71 possible affine structures of holomorphic, $C_2$-cofinite \voa{}s of CFT-type of central charge 24.}
\label{fig:schellekens1}
\end{sidewaysfigure}
\restoregeometry

\newgeometry{bottom=2cm,top=2cm}
\thispagestyle{empty}
\begin{sidewaysfigure}
\resizebox{1.0\textwidth}{!}{
\begin{tikzpicture}[scale=1.1]
\tikzstyle{lat}=[ellipse,thick,draw=blue!75,fill=blue!20,inner sep=0.05,align=center,minimum size=6mm]
\tikzstyle{invol}=[ellipse,thick,draw=cyan!75,fill=cyan!20,inner sep=0.05,align=center,minimum size=6mm]
\tikzstyle{z3}=[ellipse,thick,draw=green!75,fill=green!20,inner sep=0.05,align=center,minimum size=6mm]
\tikzstyle{new}=[ellipse,thick,draw=yellow!75,fill=yellow!20,inner sep=0.05,align=center,minimum size=6mm]
\tikzstyle{none}=[ellipse,thick,draw=black!75,fill=black!20,inner sep=0.05,align=center,minimum size=6mm]

\draw (-1.5,-0.5) -- (29.5,-0.5);
\draw (-1.5,1.5) -- (29.5,1.5);
\draw (-1.5,4.5) -- (29.5,4.5);
\draw (-1.5,6.5) -- (29.5,6.5);
\draw (-1.5,8.5) -- (29.5,8.5);
\draw (-1.5,10.5) -- (29.5,10.5);
\draw (-1.5,12.5) -- (29.5,12.5);
\draw (-1.5,13.5) -- (29.5,13.5);
\draw[line width=2pt] (-1.5,14.5) -- (29.5,14.5);

\draw[line width=2pt] (-0.5,15.5) -- (-0.5,-0.5);
\draw (0.5,15.5) -- (0.5,-0.5);
\draw (1.5,15.5) -- (1.5,-0.5);
\draw (2.5,15.5) -- (2.5,-0.5);
\draw (3.5,15.5) -- (3.5,-0.5);
\draw (4.5,15.5) -- (4.5,-0.5);
\draw (5.5,15.5) -- (5.5,-0.5);
\draw (6.5,15.5) -- (6.5,-0.5);
\draw (7.5,15.5) -- (7.5,-0.5);
\draw (8.5,15.5) -- (8.5,-0.5);
\draw (9.5,15.5) -- (9.5,-0.5);
\draw (10.5,15.5) -- (10.5,-0.5);
\draw (11.5,15.5) -- (11.5,-0.5);
\draw (12.5,15.5) -- (12.5,-0.5);
\draw (13.5,15.5) -- (13.5,-0.5);
\draw (14.5,15.5) -- (14.5,-0.5);
\draw (15.5,15.5) -- (15.5,-0.5);
\draw (16.5,15.5) -- (16.5,-0.5);
\draw (17.5,15.5) -- (17.5,-0.5);
\draw (18.5,15.5) -- (18.5,-0.5);
\draw (19.5,15.5) -- (19.5,-0.5);
\draw (20.5,15.5) -- (20.5,-0.5);
\draw (21.5,15.5) -- (21.5,-0.5);
\draw (22.5,15.5) -- (22.5,-0.5);
\draw (23.5,15.5) -- (23.5,-0.5);
\draw (24.5,15.5) -- (24.5,-0.5);
\draw (25.5,15.5) -- (25.5,-0.5);
\draw (26.5,15.5) -- (26.5,-0.5);
\draw (27.5,15.5) -- (27.5,-0.5);
\draw (28.5,15.5) -- (28.5,-0.5);
\draw (29.5,15.5) -- (29.5,-0.5);

\node at (-1,0.5) {24};
\node at (-1,3) {16};
\node at (-1,5.5) {12};
\node at (-1,7.5) {10};
\node at (-1,9.5) {8};
\node at (-1,11.5) {6};
\node at (-1,13) {4};
\node at (-1,14) {0};

\node at (0,15) {0};
\node at (1,15) {24};
\node at (2,15) {36};
\node at (3,15) {48};
\node at (4,15) {60};
\node at (5,15) {72};
\node at (6,15) {84};
\node at (7,15) {96};
\node at (8,15) {108};
\node at (9,15) {120};
\node at (10,15) {132};
\node at (11,15) {144};
\node at (12,15) {156};
\node at (13,15) {168};
\node at (14,15) {192};
\node at (15,15) {216};
\node at (16,15) {240};
\node at (17,15) {264};
\node at (18,15) {288};
\node at (19,15) {300};
\node at (20,15) {312};
\node at (21,15) {336};
\node at (22,15) {360};
\node at (23,15) {384};
\node at (24,15) {408};
\node at (25,15) {456};
\node at (26,15) {552};
\node at (27,15) {624};
\node at (28,15) {744};
\node at (29,15) {1128};

\draw (-1.5,15.5) -- (-0.5,14.5);
\node at (-0.75,15.25) {D.};
\node at (-1.25,14.75) {Rk.};

\node[invol] at (0,14) (0) {$0$};

\node[new] at (2,13) {$C_{4, 10}$};

\node[new] at (2,11.5) {$A_{2, 6}$\\$D_{4, 12}$};
\node[none] at (3,12) {$A_{6, 7}$};
\node[none] at (3,11) {$A_{1, 2}$\\$D_{5, 8}$};
\node[none] at (4,11.5) {$A_{2, 2}$\\$F_{4, 6}$};

\node[none] at (3,10) {$A_{1, 2}A_{5, 6}$\\$B_{2, 3}$};
\node[new] at (3,9) {$A_{4, 5}^{2}$};
\node[none] at (5,10) {$A_{1}^{2}D_{6, 5}$};
\node[new] at (5,9) {$A_{1}C_{5, 3}$\\$G_{2, 2}$};

\node[new] at (3,7.5) {$A_{1, 2}$\\$A_{3, 4}^{3}$};
\node[none] at (5,8) {$A_{1}^3A_{7, 4}$};
\node[new] at (5,7) {$A_{1}^2C_{3, 2}$\\$D_{5, 4}$};
\node[new] at (7,7.5) {$A_2B_2$\\$E_{6, 4}$};
\node[none] at (9,7.5) {$A_3C_{7, 2}$};

\node[invol] at (2,5.5) {$A_{1, 4}^{12}$};
\node[z3] at (3,5.5) {$A_{2, 3}^{6}$};
\node[new] at (4,6) {$A_{2, 2}^{4}$\\$D_{4, 4}$};
\node[invol] at (4,5) {$B_{2, 2}^{6}$};
\node[z3] at (5,5.5) {$A_{1}^{3}$\\$A_{5, 3}$\\$D_{4, 3}$};
\node[new] at (6,6) {$A_{4, 2}^{2}$\\$C_{4, 2}$};
\node[invol] at (6,5) {$B_{3, 2}^{4}$};
\node[none] at (7,5.5) {$A_{2}^{2}$\\$A_{8, 3}$};
\node[invol] at (8,5.5) {$B_{4, 2}^{3}$};
\node[none] at (9,6) {$A_{3}D_{7, 3}$\\$G_{2}$};
\node[z3] at (9,5) {$E_{6, 3}G_{2}^{3}$};
\node[new] at (10,5.5) {$A_{8, 2}F_{4, 2}$};
\node[invol] at (12,5.5) {$B_{6, 2}^{2}$};
\node[none] at (13,5.5) {$A_{5}$\\$E_{7, 3}$};
\node[invol] at (19,5.5) {$B_{12, 2}$};

\node[invol] at (3,3) {$A_{1, 2}^{16}$};
\node[invol] at (5,3) {$A_{1}^{4}A_{3, 2}^{4}$};
\node[new] at (7,3.5) {$A_{2}^{2}A_{5, 2}^{2}$\\$B_{2}$};
\node[invol] at (7,2.5) {$B_{2}^{4}D_{4, 2}^{2}$};
\node[invol] at (9,3.5) {$A_{3}^{2}D_{5, 2}^{2}$};
\node[new] at (9,2.5) {$A_{3}A_{7, 2}$\\$C_{3}^{2}$};
\node[new] at (11,4) {$A_{4}A_{9, 2}$\\$B_{3}$};
\node[invol] at (11,3) {$B_{3}^{2}C_{4}$\\$D_{6, 2}$};
\node[invol] at (11,2) {$C_{4}^{4}$};
\node[new] at (13,3) {$A_{5}C_{5}$\\$E_{6, 2}$};
\node[invol] at (14,3.5) {$B_{4}^{2}D_{8, 2}$};
\node[new] at (14,2.5) {$B_{4}C_{6}^{2}$};
\node[invol] at (15,3) {$A_{7}D_{9, 2}$};
\node[new] at (16,3.5) {$B_{5}E_{7, 2}$\\$F_{4}$};
\node[new] at (16,2.5) {$C_{8}F_{4}^{2}$};
\node[new] at (18,3) {$B_{6}$\\$C_{10}$};
\node[new] at (23,3) {$B_{8}E_{8, 2}$};

\node[lat] at (1,0.5) (Leech) {$\mathbb{C}^{24}$};
\node[lat] at (5,0.5) (24A1) {$A_{1}^{24}$};
\node[lat] at (7,0.5) (12A2) {$A_{2}^{12}$};
\node[lat] at (9,0.5) (8A3) {$A_{3}^{8}$};
\node[lat] at (11,0.5) (6A4) {$A_{4}^{6}$};
\node[lat] at (13,1) (4A5D4) {$A_{5}^{4}D_{4}$};
\node[lat] at (13,0) (6D4) {$D_{4}^{6}$};
\node[lat] at (14,0.5) (4A6) {$A_{6}^{4}$};
\node[lat] at (15,0.5) (2A72D5) {$A_{7}^{2}D_{5}^{2}$};
\node[lat] at (16,0.5) (3A8) {$A_{8}^{3}$};
\node[lat] at (17,1) (2A9D6) {$A_{9}^{2}D_{6}$};
\node[lat] at (17,0) (4D6) {$D_{6}^{4}$};
\node[lat] at (20,1) (A11D7E6) {$A_{11}D_{7}$\\$E_{6}$};
\node[lat] at (20,0) (4E6) {$E_{6}^{4}$};
\node[lat] at (21,0.5) (2A12) {$A_{12}^{2}$};
\node[lat] at (22,0.5) (3D8) {$D_{8}^{3}$};
\node[lat] at (24,0.5) (A15D9) {$A_{15}D_{9}$};
\node[lat] at (25,1) (A17E7) {$A_{17}E_{7}$};
\node[lat] at (25,0) (D102E7) {$D_{10}E_{7}^{2}$};
\node[lat] at (26,0.5) (2D12) {$D_{12}^{2}$};
\node[lat] at (27,0.5) (A24) {$A_{24}$};
\node[lat] at (28,1) (3E8) {$E_{8}^{3}$};
\node[lat] at (28,0) (D16E8) {$D_{16}E_{8}$};
\node[lat] at (29,0.5) (D24) {$D_{24}$};

\draw [fill=white] (23,8) rectangle (29,14);

\node[lat] at (24,13) {};
\node[invol] at (24,12) {};
\node[z3] at (24,11) {};
\node[new] at (24,10) {};
\node[none] at (24,9) {};

\node[anchor=west] at (25,13) {Niemeier lattice};
\node[anchor=west] at (25,12) {$(-1)$-involution orbifold};
\node[anchor=west] at (25,11) {$\mathbb{Z}_3$-orbifold};
\node[anchor=west] at (25,10) {new lattice orbifold};
\node[anchor=west] at (25,9) {not yet constructed};
\end{tikzpicture}
}
\caption{Schellekens' list of the 71 possible affine structures of holomorphic, $C_2$-cofinite \voa{}s of CFT-type of central charge 24. Only lattice automorphism orbifolds are considered.}
\label{fig:schellekens2}
\end{sidewaysfigure}
\restoregeometry

\newgeometry{bottom=2cm,top=2cm}
\thispagestyle{empty}
\begin{sidewaysfigure}
\resizebox{1.0\textwidth}{!}{
\begin{tikzpicture}[scale=1.1]
\tikzstyle{none}=[ellipse,thick,draw=black!75,fill=black!20,inner sep=0.05,align=center,minimum size=6mm]
\tikzstyle{framed}=[ellipse,thick,draw=red!75,fill=red!20,inner sep=0.1,align=center,minimum size=6mm]

\draw (-1.5,-0.5) -- (29.5,-0.5);
\draw (-1.5,1.5) -- (29.5,1.5);
\draw (-1.5,4.5) -- (29.5,4.5);
\draw (-1.5,6.5) -- (29.5,6.5);
\draw (-1.5,8.5) -- (29.5,8.5);
\draw (-1.5,10.5) -- (29.5,10.5);
\draw (-1.5,12.5) -- (29.5,12.5);
\draw (-1.5,13.5) -- (29.5,13.5);
\draw[line width=2pt] (-1.5,14.5) -- (29.5,14.5);

\draw[line width=2pt] (-0.5,15.5) -- (-0.5,-0.5);
\draw (0.5,15.5) -- (0.5,-0.5);
\draw (1.5,15.5) -- (1.5,-0.5);
\draw (2.5,15.5) -- (2.5,-0.5);
\draw (3.5,15.5) -- (3.5,-0.5);
\draw (4.5,15.5) -- (4.5,-0.5);
\draw (5.5,15.5) -- (5.5,-0.5);
\draw (6.5,15.5) -- (6.5,-0.5);
\draw (7.5,15.5) -- (7.5,-0.5);
\draw (8.5,15.5) -- (8.5,-0.5);
\draw (9.5,15.5) -- (9.5,-0.5);
\draw (10.5,15.5) -- (10.5,-0.5);
\draw (11.5,15.5) -- (11.5,-0.5);
\draw (12.5,15.5) -- (12.5,-0.5);
\draw (13.5,15.5) -- (13.5,-0.5);
\draw (14.5,15.5) -- (14.5,-0.5);
\draw (15.5,15.5) -- (15.5,-0.5);
\draw (16.5,15.5) -- (16.5,-0.5);
\draw (17.5,15.5) -- (17.5,-0.5);
\draw (18.5,15.5) -- (18.5,-0.5);
\draw (19.5,15.5) -- (19.5,-0.5);
\draw (20.5,15.5) -- (20.5,-0.5);
\draw (21.5,15.5) -- (21.5,-0.5);
\draw (22.5,15.5) -- (22.5,-0.5);
\draw (23.5,15.5) -- (23.5,-0.5);
\draw (24.5,15.5) -- (24.5,-0.5);
\draw (25.5,15.5) -- (25.5,-0.5);
\draw (26.5,15.5) -- (26.5,-0.5);
\draw (27.5,15.5) -- (27.5,-0.5);
\draw (28.5,15.5) -- (28.5,-0.5);
\draw (29.5,15.5) -- (29.5,-0.5);

\node at (-1,0.5) {24};
\node at (-1,3) {16};
\node at (-1,5.5) {12};
\node at (-1,7.5) {10};
\node at (-1,9.5) {8};
\node at (-1,11.5) {6};
\node at (-1,13) {4};
\node at (-1,14) {0};

\node at (0,15) {0};
\node at (1,15) {24};
\node at (2,15) {36};
\node at (3,15) {48};
\node at (4,15) {60};
\node at (5,15) {72};
\node at (6,15) {84};
\node at (7,15) {96};
\node at (8,15) {108};
\node at (9,15) {120};
\node at (10,15) {132};
\node at (11,15) {144};
\node at (12,15) {156};
\node at (13,15) {168};
\node at (14,15) {192};
\node at (15,15) {216};
\node at (16,15) {240};
\node at (17,15) {264};
\node at (18,15) {288};
\node at (19,15) {300};
\node at (20,15) {312};
\node at (21,15) {336};
\node at (22,15) {360};
\node at (23,15) {384};
\node at (24,15) {408};
\node at (25,15) {456};
\node at (26,15) {552};
\node at (27,15) {624};
\node at (28,15) {744};
\node at (29,15) {1128};

\draw (-1.5,15.5) -- (-0.5,14.5);
\node at (-0.75,15.25) {D.};
\node at (-1.25,14.75) {Rk.};

\node[framed] at (0,14) (0) {$0$};

\node[none] at (2,13) {$C_{4, 10}$};

\node[none] at (2,11.5) {$A_{2, 6}$\\$D_{4, 12}$};
\node[none] at (3,12) {$A_{6, 7}$};
\node[framed] at (3,11) {$A_{1, 2}$\\$D_{5, 8}$};
\node[none] at (4,11.5) {$A_{2, 2}$\\$F_{4, 6}$};

\node[none] at (3,10) {$A_{1, 2}A_{5, 6}$\\$B_{2, 3}$};
\node[none] at (3,9) {$A_{4, 5}^{2}$};
\node[none] at (5,10) {$A_{1}^{2}D_{6, 5}$};
\node[none] at (5,9) {$A_{1}C_{5, 3}$\\$G_{2, 2}$};

\node[framed] at (3,7.5) {$A_{1, 2}$\\$A_{3, 4}^{3}$};
\node[framed] at (5,8) {$A_{1}^3A_{7, 4}$};
\node[framed] at (5,7) {$A_{1}^2C_{3, 2}$\\$D_{5, 4}$};
\node[none] at (7,7.5) {$A_2B_2$\\$E_{6, 4}$};
\node[framed] at (9,7.5) {$A_3C_{7, 2}$};

\node[framed] at (2,5.5) {$A_{1, 4}^{12}$};
\node[none] at (3,5.5) {$A_{2, 3}^{6}$};
\node[framed] at (4,6) {$A_{2, 2}^{4}$\\$D_{4, 4}$};
\node[framed] at (4,5) {$B_{2, 2}^{6}$};
\node[none] at (5,5.5) {$A_{1}^{3}$\\$A_{5, 3}$\\$D_{4, 3}$};
\node[framed] at (6,6) {$A_{4, 2}^{2}$\\$C_{4, 2}$};
\node[framed] at (6,5) {$B_{3, 2}^{4}$};
\node[none] at (7,5.5) {$A_{2}^{2}$\\$A_{8, 3}$};
\node[framed] at (8,5.5) {$B_{4, 2}^{3}$};
\node[none] at (9,6) {$A_{3}D_{7, 3}$\\$G_{2}$};
\node[none] at (9,5) {$E_{6, 3}G_{2}^{3}$};
\node[framed] at (10,5.5) {$A_{8, 2}F_{4, 2}$};
\node[framed] at (12,5.5) {$B_{6, 2}^{2}$};
\node[none] at (13,5.5) {$A_{5}$\\$E_{7, 3}$};
\node[framed] at (19,5.5) {$B_{12, 2}$};

\node[framed] at (3,3) {$A_{1, 2}^{16}$};
\node[framed] at (5,3) {$A_{1}^{4}A_{3, 2}^{4}$};
\node[framed] at (7,3.5) {$A_{2}^{2}A_{5, 2}^{2}$\\$B_{2}$};
\node[framed] at (7,2.5) {$B_{2}^{4}D_{4, 2}^{2}$};
\node[framed] at (9,3.5) {$A_{3}^{2}D_{5, 2}^{2}$};
\node[framed] at (9,2.5) {$A_{3}A_{7, 2}$\\$C_{3}^{2}$};
\node[framed] at (11,4) {$A_{4}A_{9, 2}$\\$B_{3}$};
\node[framed] at (11,3) {$B_{3}^{2}C_{4}$\\$D_{6, 2}$};
\node[framed] at (11,2) {$C_{4}^{4}$};
\node[framed] at (13,3) {$A_{5}C_{5}$\\$E_{6, 2}$};
\node[framed] at (14,3.5) {$B_{4}^{2}D_{8, 2}$};
\node[framed] at (14,2.5) {$B_{4}C_{6}^{2}$};
\node[framed] at (15,3) {$A_{7}D_{9, 2}$};
\node[framed] at (16,3.5) {$B_{5}E_{7, 2}$\\$F_{4}$};
\node[framed] at (16,2.5) {$C_{8}F_{4}^{2}$};
\node[framed] at (18,3) {$B_{6}$\\$C_{10}$};
\node[framed] at (23,3) {$B_{8}E_{8, 2}$};

\node[framed] at (1,0.5) (Leech) {$\mathbb{C}^{24}$};
\node[framed] at (5,0.5) (24A1) {$A_{1}^{24}$};
\node[framed] at (7,0.5) (12A2) {$A_{2}^{12}$};
\node[framed] at (9,0.5) (8A3) {$A_{3}^{8}$};
\node[framed] at (11,0.5) (6A4) {$A_{4}^{6}$};
\node[framed] at (13,1) (4A5D4) {$A_{5}^{4}D_{4}$};
\node[framed] at (13,0) (6D4) {$D_{4}^{6}$};
\node[framed] at (14,0.5) (4A6) {$A_{6}^{4}$};
\node[framed] at (15,0.5) (2A72D5) {$A_{7}^{2}D_{5}^{2}$};
\node[framed] at (16,0.5) (3A8) {$A_{8}^{3}$};
\node[framed] at (17,1) (2A9D6) {$A_{9}^{2}D_{6}$};
\node[framed] at (17,0) (4D6) {$D_{6}^{4}$};
\node[framed] at (20,1) (A11D7E6) {$A_{11}D_{7}$\\$E_{6}$};
\node[framed] at (20,0) (4E6) {$E_{6}^{4}$};
\node[framed] at (21,0.5) (2A12) {$A_{12}^{2}$};
\node[framed] at (22,0.5) (3D8) {$D_{8}^{3}$};
\node[framed] at (24,0.5) (A15D9) {$A_{15}D_{9}$};
\node[framed] at (25,1) (A17E7) {$A_{17}E_{7}$};
\node[framed] at (25,0) (D102E7) {$D_{10}E_{7}^{2}$};
\node[framed] at (26,0.5) (2D12) {$D_{12}^{2}$};
\node[framed] at (27,0.5) (A24) {$A_{24}$};
\node[framed] at (28,1) (3E8) {$E_{8}^{3}$};
\node[framed] at (28,0) (D16E8) {$D_{16}E_{8}$};
\node[framed] at (29,0.5) (D24) {$D_{24}$};

\draw [fill=white] (23,11) rectangle (29,14);

\node[framed] at (24,13) {};
\node[none] at (24,12) {};

\node[anchor=west] at (25,13) {framed construction};
\node[anchor=west] at (25,12) {no framed construction};
\end{tikzpicture}
}
\caption{Schellekens' list of the 71 possible affine structures of holomorphic, $C_2$-cofinite \voa{}s of CFT-type of central charge 24. Only framed constructions are considered.}
\label{fig:schellekens3}
\end{sidewaysfigure}
\restoregeometry

\chapter{BRST Cohomology}\label{ch:BRST}

In this chapter we describe the BRST\footnote{Named after Becchi, Rouet and Stora \cite{BRS74,BRS75,BRS76} and Tyutin \cite{Tyu75}.} construction of \BKMa{}s starting from certain vertex algebras of central charge 26 following the work of Frenkel, Garland, Lian and Zuckerman \cite{FGZ86,Zuc89,LZ89,LZ91,LZ93}. We summarise examples of such constructions found in \cite{Bor90,Bor92,Hoe03a,HS03,HS14,Car12b}. Then we present---for most of these cases for the first time---BRST constructions of ten \BKMa{}s whose denominator identities are completely reflective automorphic products of singular weight \cite{Sch04b,Sch06}. Those results depend on two technical conjectures, which we aim to prove in the future.

\minisec{Preliminaries}

Throughout this chapter we will need the notion of vertex (operator) superalgebras. They are $\Z_2$-graded generalisations of vertex algebras that admit an additional sign in the Jacobi identity depending on the $\Z_2$-grading. An introduction with details relevant for this text can be found in Section~\ref{sec:super} of the appendix.

Subalgebras of (weak) \voa{}s are not required to be full in this chapter (see remark after Definition~\ref{defi:subvoa}).

We also introduce additional gradings on vertex (super)algebras:
\begin{defi}[Additional Grading on Vertex (Super)algebra]\label{defi:addgrad}
Let $V$ be a vertex (super)algebra and $\Gamma$ an abelian group. We say $V$ is \emph{$\Gamma$-graded} if $V=\bigoplus_{\alpha\in \Gamma}V^\alpha$ as vector space such that for $a\in V^\alpha$, $\alpha\in \Gamma$, all the modes $a_n$, $n\in\Z$, are operators in $\End_\C(V)$ of degree $\alpha\in\Gamma$.

If the vertex (super)algebra $V$ has a weight grading $V=\bigoplus_{n\in\Z}V_n$, i.e.\ if it is a (weak) graded vertex (super)algebra in the sense of Definitions \ref{defi:gva} or \ref{defi:wgva}, then we demand that the weight $\Z$-grading and the additional $\Gamma$-grading be compatible in the sense that 
\begin{equation*}
V^\alpha=\bigoplus_{n\in\Z}V^\alpha_n
\end{equation*}
for $\alpha\in \Gamma$ where $V^\alpha_n:=V_n\cap V^\alpha$. (We allow weights $n\in(1/2)\Z$ in the supercase.)
\end{defi}
A weight grading on a vertex algebra is not a $\Z$-grading in the above sense since the weight of the operator $a_n$ depends on $n$. Indeed, recall that $\wt(a_n)=\wt(a)-n-1$ for $a\in V$ and $n\in\Z$.

Examples of vertex algebras with an additional grading in the above sense are $I$-graded (simple-current) extensions from Section~\ref{sec:sce}. Also, the $D$-grading on \aia{}s behaves in this way (see Definition~\ref{defi:aia}) and so does the $\Z_2$-grading on a vertex superalgebra.

\section{Bosonic Ghost Vertex Operator Superalgebra}\label{sec:bc}
An important ingredient for the BRST construction is the \emph{bosonic ghost vertex operator superalgebra} (or \emph{$b$-$c$-ghost system}) $V_\text{gh.}$ of central charge $c=-26$, which we define in the following.

Similar to the construction of a \voa{} $V_L$ associated with a positive-definite, even lattice, we more generally obtain a vertex operator superalgebra if the lattice is only integral. Indeed, with exactly the same construction, the space $V_L=M_{\hat\h}(1)\otimes\C_\eps[L]$ has a $(1/2)\N$-grading by weights. The weight of an element of the form $h_1(-n_1)\ldots h_k(-n_k)1\otimes\ee_\alpha$ for $k\in\N$, $n_i\in\Ns$, $h_i\in\h=L\otimes_\Z\C$ and $\alpha\in L$ is given by $n_1+\ldots+n_k+\langle\alpha,\alpha\rangle/2$. If we define a $\Z_2$-grading via
\begin{equation*}
V_L^{\bar 0}:=\bigoplus_{n\in\Z}(V_L)_n\quad\text{and}\quad V_L^{\bar 1}:=\bigoplus_{n\in\Z+\frac{1}{2}}(V_L)_n,
\end{equation*}
then $V_L$ has the structure of a vertex operator superalgebra of central charge $c=\rk(L)$.

We now consider the one-dimensional, positive-definite, integral lattice $L=\Z\sigma$ with $\langle\sigma,\sigma\rangle=1$. The associated lattice vertex operator superalgebra $V_L$ has central charge $c=1$. We choose a 2-cocycle with $\eps(\sigma,\sigma)=1$ for the twisted group algebra $\C_\eps[L]$. In the standard construction, the conformal vector $\omega$ is given by $(1/2)\sigma(-1)\sigma(-1)1\otimes\ee_0$.

We now consider a slight modification of $V_L$ by introducing a shift in the conformal vector $\omega$ (see \cite{DM06} for a general description of shifted lattice theories). If we define the new conformal vector $\omega'$ as
\begin{equation*}
\omega'=\omega+\lambda\sigma(-2)1\otimes\ee_0=\frac{1}{2}\sigma(-1)\sigma(-1)1\otimes\ee_0+\lambda\sigma(-2)1\otimes\ee_0,
\end{equation*}
for some $\lambda\in\C$, then $(V_L,\omega')$ is still a vertex operator superalgebra with exactly the same space of states, vertex operators, vacuum element, parity $\Z_2$-grading but with the following modifications:
\begin{enumerate}
\item The grading of an element of $V_L$ of the form $\ldots\otimes\ee_\alpha$, $\alpha\in L$, is shifted by $-\lambda\langle\sigma,\alpha\rangle$. In order to preserve the $(1/2)\Z$-grading, which is necessary for $V_L$ to remain a vertex operator superalgebra, we have to demand that $\lambda\langle\sigma,\alpha\rangle\in(1/2)\Z$ for all $\alpha\in L$, which is the case if and only if $\lambda\in(1/2)\Z$.
\item The central charge of $V_L$ is given by
\begin{equation*}
c'=\rk(L)-12\lambda^2\langle\sigma,\sigma\rangle=1-12\lambda^2.
\end{equation*}
\end{enumerate}
For every choice of $\lambda\in\C$, the weight grading on $V_L$ is still bounded from below and the graded components are finite-dimensional.

In order to obtain a vertex operator superalgebra of central charge $c'=-26$ we have to choose $\lambda=\pm3/2$. Let us take $\lambda=3/2$. Then the weight of a general element $\ldots\otimes\ee_\alpha\in V_L$ (write $\alpha=n\sigma$ for some $n\in\Z$) is in
\begin{equation*}
\N+\frac{1}{2}n^2-\lambda n=\N+\frac{1}{2}n^2-\frac{3}{2}n=\N+\frac{n(n-3)}{2},
\end{equation*}
i.e.\ the weight grading of $V_L$ is in $\Z_{\geq -1}$ with negative-weight states $1\otimes\ee_{\sigma}$ and $1\otimes\ee_{2\sigma}$ and on the other hand the weight grading becomes purely integral, i.e.\ there are no half-integer weight states. Let us denote $(V_L,\omega')$ with $\lambda=3/2$ in the following by $V_\text{gh.}$.

We define the \emph{ghost} states
\begin{equation*}
b:=1\otimes\ee_{-\sigma}\quad\text{and}\quad c:=1\otimes\ee_{\sigma}.
\end{equation*}
These states are both of odd parity and have weights $2$ and $-1$, respectively. We also introduce the \emph{ghost current}
\begin{equation*}
j^N:=\sigma(-1)1\otimes\ee_0
\end{equation*}
and the \emph{ghost number} $p$ as eigenvalue of the \emph{ghost number operator}
\begin{equation*}
U:=j^N_0=(\sigma(-1)1\otimes\ee_0)_0=\sigma(0),
\end{equation*}
which has even parity and weight 0, i.e.\ it commutes with the parity operator and the weight-grading operator $L_0$. Then
\begin{equation*}
Ub=-b\quad\text{and}\quad Uc=c,
\end{equation*}
so that $b$ has ghost number $p=-1$ and $c$ ghost number $p=1$. More generally an element of the form $\ldots\otimes\ee_\alpha$, $\alpha=n\sigma\in L$, has ghost number $p=\langle\alpha,\sigma\rangle=n\in\Z$. We observe that the elements in $V_\text{gh.}$ of even ($n^2/2\in\Z$) or odd ($n^2/2\in\Z+1/2$) parity are exactly those with even or odd ghost number $p$, respectively. By $(V_\text{gh.})^p$ we denote the subspace of $V_\text{gh.}$ of ghost number $p\in\Z$.

We consider the modes of $b$ and $c$, which are of odd parity and ghost number $p=-1$ and $p=1$, respectively. Indeed, one can show that $[U,b_m]=-b_m$ and $[U,c_m]=c_m$ for all $m\in\Z$. It follows from the definition of the vertex operation on $V_\text{gh.}$ that
\begin{equation*}
b_nb=c_nc=0
\end{equation*}
for $n\geq-1$,
\begin{equation*}
b_nc=c_nb=0
\end{equation*}
for $n\geq 1$ and
\begin{equation*}
b_0c=c_0b=\vac.
\end{equation*}
Borcherds' commutator formula (see e.g.\ formula (4.6.3) in \cite{Kac98}), which also holds for vertex superalgebras, gives
\begin{equation*}
\{b_n,b_m\}=\{c_n,c_m\}=0
\end{equation*}
(in particular $c_n^2=b_n^2=0$) for all $m,n\in\Z$ and
\begin{equation*}
\{b_n,c_m\}=\delta_{m+n,-1}\id_{V_\text{gh.}}.
\end{equation*}
By the vacuum axiom, for $n\geq 0$,
\begin{equation*}
b_n\vac=c_n\vac=0.
\end{equation*}

In total, this means that there are two infinite sequences of fermionic creation operators with associated annihilation operators, namely:
\begin{equation*}
\begin{tabular}{l|ccccc|ccccc}
creation op.&$b_{-1}$&$b_{-2}$&$b_{-3}$&$b_{-4}$&$\ldots$&$c_{-1}$&$c_{-2}$&$c_{-3}$&$c_{-4}$&$\ldots$\\\hline
weight      &$2$     &$3$     &$4$     &$5$     &$\ldots$&$-1$    &$0$     &$1$     &$2$     &$\ldots$\\\hline
ann.\ op.\  &$c_0$   &$c_1$   &$c_2$   &$c_3$   &$\ldots$&$b_0$   &$b_1$   &$b_2$   &$b_3$   &$\ldots$\\\hline
weight      &$-2$    &$-3$    &$-4$    &$-5$    &$\ldots$&$1$     &$0$     &$-1$    &$-2$    &$\ldots$
\end{tabular}
\end{equation*}
The bosonic ghost vertex operator superalgebra $V_\text{gh.}$ is spanned as a vector space by words in the creation operators (applied to the vacuum $\vac$) containing each creation operator at most once. In fact, if we ignore the order of the creation operators in these words, this gives a basis of $V_\text{gh.}$.\footnote{It suffices to show that these words are linearly independent. Then we compare the character of $V_\text{gh.}$ obtained from the standard basis $\sum_{n\in\Z}q^{n(n-3)/2}q^{27/24}/\eta(q)$ with the one computed for the $b$-$c$-basis, $2\eta(q^2)^2/\eta(q)^2$. This leads to an identity similar to Euler's pentagonal number theorem, which can indeed be verified.}

The character of $V_\text{gh.}$ is easily computed:
\begin{align*}
\ch_{V_\text{gh.}}(q)&=\sum_{n=-1}^\infty\dim_\C((V_\text{gh.})_n)q^{n-c/24}=q^{26/24}\prod_{n=2}^\infty(1+q^n)\prod_{n=-1}^\infty(1+q^n)\\
&=q^{26/24}\frac{2}{q}\left(\prod_{n=1}^\infty(1+q^n)\right)^2=q^{26/24}\frac{2}{q}\left(q^{-1/24}\frac{\eta(q^2)}{\eta(q)}\right)^2=2\frac{\eta(q^2)^2}{\eta(q)^2}.
\end{align*}
For the supercharacter (see Section~\ref{sec:super} for the definition) we obtain
\begin{align*}
\sch_{V_\text{gh.}}(q)=\sum_{n=-1}^\infty\sdim((V_\text{gh.})_n)q^{n-c/24}=q^{26/24}\prod_{n=2}^\infty(1-q^n)\prod_{n=-1}^\infty(1-q^n)=0,
\end{align*}
which vanishes because of the term for $n=0$ in the right product, or in other words since adding $c_{-2}$ to a word of creation operators not containing $c_{-2}$ changes the parity from even to odd and vice versa but leaves the weight unchanged.

For later use we also need the supercharacter of the kernel $\ker(b_1)$ of $b_1$ in $V_\text{gh.}$, which is spanned by the words in the creation operators not containing $c_{-2}$. We get
\begin{align*}
\ch_{\ker(b_1)}(q)&=\sum_{n=-1}^\infty\dim_\C((\ker(b_1))_n)q^{n-c/24}=q^{26/24}\prod_{n=2}^\infty(1+q^n)\prod_{n=-1,1,2,\ldots}^\infty(1+q^n)\\
&=q^{26/24}\frac{1}{q}\left(\prod_{n=1}^\infty(1+q^n)\right)^2=\frac{\eta(q^2)^2}{\eta(q)^2}
\end{align*}
and
\begin{align*}
\sch_{\ker(b_1)}(q)&=\sum_{n=-1}^\infty\sdim((\ker(b_1))_n)q^{n-c/24}=q^{26/24}\prod_{n=2}^\infty(1-q^n)\prod_{n=-1,1,2,\ldots}^\infty(1-q^n)\\
&=q^{26/24}\frac{-1}{q}\left(\prod_{n=1}^\infty(1-q^n)\right)^2=-\eta(q)^2.
\end{align*}

\section{BRST Construction}\label{sec:brst}

In this section we describe the BRST construction of \BKMa{}s $\g$ (also called generalised Kac-Moody algebras) from certain vertex algebras $M$ of central charge 26:
\begin{equation*}
\begin{tikzcd}
M\arrow[squiggly]{rr}{\text{BRST}}&&\g.
\end{tikzcd}
\end{equation*}

To this end we let $M$ be a weak \voa{} of central charge 26 and consider the tensor product $W=M\otimes V_\text{gh.}$ of $M$ with the bosonic ghost vertex operator superalgebra $V_\text{gh.}$ of central charge $-26$ from Section~\ref{sec:bc} above. Then $W$ is a weak vertex operator superalgebra of central charge $c=26-26=0$.

On $W$, it is possible to define a BRST current with a corresponding BRST operator $Q$ such that $Q$ increases the ghost number by one and
\begin{equation*}
Q^2=0
\end{equation*}
(Proposition~\ref{prop:QQ}). This yields the cochain complex $(W^\bullet,Q^\bullet)=(W^p,Q^p)_{p\in\Z}$ where $p$ is the ghost number. We call the corresponding cohomological spaces $H^p_\text{BRST}(M)$, $p\in\Z$. The space
\begin{equation*}
\g:=H^1_\text{BRST}(M)
\end{equation*}
is of particular interest since it naturally carries the structure of a Lie algebra \cite{LZ93} (see Corollary~\ref{cor:brstlie}) and is in general infinite-dimensional.

We make additional assumptions such that in particular the weak \voa{} $M$ is graded by the dual $L'$ of an even Lorentzian lattice $L$, in addition to the weight $\Z$-grading. One obtains cochain complexes $(W^\bullet(\alpha),Q^\bullet)=(W^p(\alpha),Q^p)_{p\in\Z}$ for each $\alpha\in L'$ and associated cohomological spaces $H^p(\alpha)$ so that
\begin{equation*}
H_\text{BRST}^p(M)=\bigoplus_{\alpha\in L'}H^p(\alpha).
\end{equation*}

The Lie algebra $\g=H_\text{BRST}^1(M)$ is then also graded by $L'$ and a vanishing theorem \cite{FGZ86,Zuc89} (see Theorem~\ref{thm:vanish}) together with the Euler-Poincaré principle allows us to compute the (finite) dimensions of the graded components of $\g$ (see Theorem~\ref{thm:rootdim}). In many interesting examples the Lie algebra $\g=H_\text{BRST}^1(M)$ turns out to be a \BKMa{} (see Sections \ref{sec:BKMprop} and \ref{sec:brstex}).

Since we are working over the base field $\C$, all Lie algebras in this text are complex, unless otherwise stated.

\minisec{Vertex Superalgebra Cohomology}

In the following we describe the general setting, in which the cohomology of vertex (super)algebras occurs.
\begin{ass*}
Let $W$ be a weak graded vertex superalgebra, i.e.\ a graded vertex superalgebra without the requirement that the graded components be finite-dimensional or that the grading be bounded from below. Assume that in addition to this weight grading there is a $\Z$-grading in the sense of Definition~\ref{defi:addgrad}
\begin{equation*}
W=\bigoplus_{p\in\Z}W^p
\end{equation*}
on $W$, denoted by upper indices.

Furthermore, let $j\in W$ be some vector, homogeneous of degree 1 with respect to the upper grading and homogeneous of some weight $\wt(j)$ with respect to the weight grading. Then the zeroth mode $Q:=j_0\in\End_\C(W)$ is an operator of degree 1 with respect to the upper grading, i.e.\ it raises the degree of a homogeneous element with respect to the upper grading by 1, and homogeneous with respect to the weight grading. Finally, assume that $Q$ satisfies
\begin{equation*}
Q^2=0\quad\iff\quad\im(Q)\subseteq\ker(Q).
\end{equation*}
\end{ass*}

For the moment let us view $W$ as a $\Z$-graded $\C$-vector space (with the upper grading). By definition, $Q$ restricts to
\begin{equation*}
Q^p:=Q|_{W^p}\colon W^p\to W^{p+1}
\end{equation*}
so that
\begin{equation*}
\im(Q^{p-1})\subseteq\ker(Q^p)\subseteq W^p
\end{equation*}
for all $p\in\Z$. This defines a \emph{cochain complex} in the abelian category of $\C$-vector spaces
\begin{equation*}
\ldots\stackrel{Q^{p-2}}{\longrightarrow}W^{p-1}\stackrel{Q^{p-1}}{\longrightarrow}W^p\stackrel{Q^p}{\longrightarrow}W^{p+1}\stackrel{Q^{p+1}}{\longrightarrow}\ldots
\end{equation*}
denoted by $(W^\bullet,Q^\bullet)=(W^p,Q^p)_{p\in\Z}$. We define the \emph{$p$-th cohomological space} of this complex as the quotient space
\begin{equation*}
H^p:=\ker(Q^p)/\im(Q^{p-1})=\left(W^p\cap\ker(Q)\right)/\left(W^p\cap\im(Q)\right)
\end{equation*}
for $p\in\Z$, measuring the non-exactness of the above sequence at position $p$.

One can even show that the direct sum of these spaces
\begin{equation*}
H:=\ker(Q)/\im(Q)=\bigoplus_{p\in\Z}H^p
\end{equation*}
naturally carries the structure of a weak graded vertex superalgebra\footnote{In particular, the homogeneity of $Q$ with respect to the weight grading implies that $H$ is again graded by weights.}, which also inherits the upper grading from $W$ (see e.g.\ Section~5.7.3 of \cite{FBZ04}).

\minisec{BRST Cohomology}

We now present the BRST cohomology at central charge 26. The weak graded vertex superalgebra $W$ from above will be the tensor product $W=M\otimes V_\text{gh.}$ of a weak \voa{} $M$ of central charge $26$ and the bosonic ghost vertex operator superalgebra $V_\text{gh.}$, the upper grading will be the ghost number and for the vector $j\in W$ we will take the BRST current $j^\text{BRST}$.

For the following results we need a series of increasingly strong assumptions on the vertex algebra $M$ in the \emph{matter sector}, starting with the following one:

\begin{ass*}
Let $M$ be a weak \voa{} of central charge $26$, i.e.\ a \voa{} but we do not require the weight grading on $M$ to be bounded from below nor that the graded components be finite-dimensional.
\end{ass*}
In the \emph{ghost sector} let $V_\text{gh.}$ denote the bosonic ghost vertex operator superalgebra of central charge $c=-26$ from Section~\ref{sec:bc}. We consider the tensor product
\begin{equation*}
W:=M\otimes V_\text{gh.},
\end{equation*}
which naturally admits the structure of a weak vertex operator superalgebra of central charge $c=26-26=0$ (see~Section~\ref{sec:super}). We obtain the usual tensor-product weight grading on $W=M\otimes V_\text{gh.}$ via $\omega=\omega_M\otimes\vac_\text{gh.}+\vac_M\otimes\omega_\text{gh.}$ and the tensor-product parity with $M\otimes_\C V_\text{gh.}^{\bar{0}}$ being the even subspace and $M\otimes_\C V_\text{gh.}^{\bar{1}}$ the odd one. Moreover, there is the ghost number operator $\id_M\otimes U$. All three gradings are compatible, i.e.\ each graded subspace with respect to one grading decomposes into a direct sum with respect to the other gradings. This is the case if and only if all the grading operators commute.

We define the \emph{BRST current}
\begin{equation*}
j^\text{BRST}:=(\id_M\otimes c_{-1})(\omega_M\otimes \vac_\text{Gh.}+\frac{1}{2}\vac_{M}\otimes\omega_\text{gh.})=\omega_M\otimes c+\frac{1}{2}\vac_{M}\otimes c_{-1}\omega_\text{gh.}
\end{equation*}
and the \emph{BRST operator}
\begin{equation*}
Q:=j^\text{BRST}_0.
\end{equation*}

We use the following shorthand notation for tensor products: we write operators $A\otimes \id$ or $\id\otimes A$ as $A$ and vectors $a\otimes\vac$ or $\vac\otimes a$ as $a$. Note that the modes of the vertex operators obey $(a\otimes\vac)_n=a_n\otimes\id$ and $(\vac\otimes a)_n=\id\otimes a_n$ by the left vacuum axiom and because $\vac$ is of even parity. We also write $AB$ for the operator $A\otimes B$ when it is clear on which space the operators $A$ and $B$ act.

Using this shorthand notation the BRST current reads $j^\text{BRST}=c_{-1}(\omega_M+\omega_\text{gh.}/2)$.

\begin{prop}[\cite{FGZ86}, \cite{Zuc89}, Section~4]\label{prop:QQ}
The operator $Q$ on $W=M\otimes V_\text{gh.}$ fulfils
\begin{equation*}
Q^2=0,
\end{equation*}
and has weight 0 (i.e.\ $[Q,L_0]=0$), ghost number 1 (i.e.\ $[U,Q]=Q$) and odd parity. More generally
\begin{equation*}
[Q,L_n]=0
\end{equation*}
for all $n\in\Z$. Moreover
\begin{equation*}
\{Q,b_{n+1}\}=L_n.
\end{equation*}
\end{prop}

We can now consider the \emph{BRST complex}
\begin{equation*}
\ldots\stackrel{Q^{p-2}}{\longrightarrow}W^{p-1}\stackrel{Q^{p-1}}{\longrightarrow}W^p\stackrel{Q^p}{\longrightarrow}W^{p+1}\stackrel{Q^{p+1}}{\longrightarrow}\ldots,
\end{equation*}
a cochain complex in the category of $\C$-vector spaces, where the upper index is the ghost number. The corresponding cohomological spaces are
\begin{equation*}
H^p_\text{BRST}(M):=H^p:=\left(W^p\cap\ker(Q)\right)/\left(W^p\cap\im(Q)\right).
\end{equation*}
Each $W^p$ and each $H^p$ is graded by the $L_0$-weights since $U$ commutes with $L_0$ and $Q$ is homogeneous with respect to the $L_0$-grading ($Q$ even commutes with $L_0$, too).

From $\{Q,b_{n+1}\}=L_n$ we see that if $x\in\ker(Q)$, then
\begin{equation*}
\wt(x)x=L_0x=\{Q,b_{n+1}\}x=Qb_{n+1}x
\end{equation*}
for $L_0$-homogeneous $x$ and hence $x=Q(b_{n+1}x/\wt(x))$, i.e.\ $x\in\im(Q)$ if $\wt(x)\neq 0$. This shows that the cohomology is only supported in weight 0 and means that if we study the subcomplex
\begin{equation*}
\ldots\stackrel{Q^{p-2}}{\longrightarrow}W_0^{p-1}\stackrel{Q^{p-1}}{\longrightarrow}W_0^p\stackrel{Q^p}{\longrightarrow}W_0^{p+1}\stackrel{Q^{p+1}}{\longrightarrow}\ldots,
\end{equation*}
then its cohomological spaces are identical to those of the BRST complex $(W^\bullet,Q^\bullet)$, i.e.\
\begin{equation*}
H^p\cong\left(W_0^p\cap\ker(Q)\right)/\left(W_0^p\cap\im(Q)\right)=H^p_0
\end{equation*}
and $H^p_n=\{0\}$ for $n\neq 0$ where $H^p=\bigoplus_{n\in\Z}H^p_n$ ($L_0$-decomposition).

We also define the \emph{relative BRST subcomplex} on the elements of $W$ annihilated by $b_1$ and of weight 0, i.e.\ on
\begin{equation*}
C=W_0\cap\ker(b_1).
\end{equation*}
The weight and ghost number gradings both restrict to $C$ since $b_1$ is homogeneous with respect to both gradings (i.e.\ $[L_0,b_1]=0$ and $[U,b_1]=-b_1$). The operator $Q$ restricts to $C$ because of $\{Q,b_1\}=L_0$ and hence we get a subcomplex
\begin{equation*}
\ldots\stackrel{Q^{p-2}}{\longrightarrow}C^{p-1}\stackrel{Q^{p-1}}{\longrightarrow}C^p\stackrel{Q^p}{\longrightarrow}C^{p+1}\stackrel{Q^{p+1}}{\longrightarrow}\ldots
\end{equation*}
and the cohomological spaces
\begin{equation*}
H^p_\text{rel.}(M)=\left(W_0^p\cap\ker(b_1)\cap\ker(Q)\right)/\left(W_0^p\cap\ker(b_1)\cap\im(Q)\right).
\end{equation*}

We note that the inclusion map of $C^p$ into $W^p$ induces an injective map $H^p_\text{rel.}(M)\to H^p_\text{BRST}(M)$. A priori, the BRST cohomological spaces and the relative cohomological spaces need not be the same. We will see however that under some stronger assumptions $H^1_\text{BRST}(M)\cong H^1_\text{rel.}(M)$. More precisely, we will demand that $W$ carry certain representations of the Heisenberg \voa{}.

\minisec{Lie Algebra Structure}

In the following we will concentrate on $H^1_\text{BRST}(M)$. It is shown in \cite{LZ93} that this space naturally carries a Lie algebra structure. Indeed, consider the bilinear bracket $[\cdot,\cdot]\colon W\times W\to W$ on $W=M\otimes V_\text{gh.}$ defined by
\begin{equation*}
[u,v]:=-(-1)^{|u|}(b_0u)_0v
\end{equation*}
for $u,v\in W$.\footnote{We define the bracket with an additional minus sign relative to the definition in \cite{LZ93}.}
\begin{prop}[\cite{LZ93}, Section~2]
The map $Q$ acts as a derivation on $[\cdot,\cdot]$, i.e.\
\begin{equation*}
Q[u,v]=[Qu,v]-(-1)^{|u|}[u,Qv]
\end{equation*}
for $u,v\in W$.
\end{prop}
It follows by a simple calculation that $[\cdot,\cdot]$ restricts to a map $[\cdot,\cdot]\colon\ker(Q)\times\ker(Q)\to\ker(Q)$ and even induces a well-defined bracket $[\cdot,\cdot]\colon H\times H\to H$ where
\begin{equation*}
H=\ker(Q)/\im(Q)=\bigoplus_{p\in\Z}H^p.
\end{equation*}
\begin{thm}[\cite{LZ93}, Theorem~2.2]
The bracket $[\cdot,\cdot]\colon H\times H\to H$ defines a Lie superbracket on the BRST cohomological space $H$ with the rôles of even and odd elements interchanged, i.e.\ super-antisymmetry
\begin{equation*}
[u,v]=-(-1)^{(|u|-1)(|v|-1)}[v,u]
\end{equation*}
and the super Jacobi identity
\begin{equation*}
(-1)^{(|u|-1)(|t|-1)}[u,[v,t]]+(-1)^{(|t|-1)(|v|-1)}[t,[u,v]]+(-1)^{(|v|-1)(|u|-1)}[v,[t,u]]=0
\end{equation*}
hold for all $u,v,t\in H$.
\end{thm}
For the Lie superalgebra relations to hold, it is necessary to consider the quotient $H=\ker(Q)/\im(Q)$. Note that in \cite{LZ93} the authors additionally show the existence of an associative, supercommutative ``dot product'' on $H$ such that in total $H$ carries the structure of a \emph{Gerstenhaber algebra}.
\begin{prop}[\cite{LZ93}, Theorem~2.2]
The Lie superbracket on $H$ is of degree -1 with respect to the ghost number, i.e.\ it restricts to
\begin{equation*}
[\cdot,\cdot]\colon H^p\times H^q\to H^{p+q-1}
\end{equation*}
for all $p,q\in\Z$.
\end{prop}
Then, clearly, $H^1$ is closed under the bracket. Since $H^1$ has odd ghost number, it is even with respect to the Lie superalgebra structure on $H$ and we obtain:
\begin{cor}\label{cor:brstlie}
$\g:=H^1=H^1_\text{BRST}(M)$ becomes a Lie algebra with the bracket
\begin{equation*}
[u,v]=(b_0u)_0v
\end{equation*}
for $u,v\in H^1$.
\end{cor}

\minisec{Relative BRST Subcomplex}

We return to the relative BRST subcomplex 
\begin{equation*}
\ldots\stackrel{Q^{p-2}}{\longrightarrow}C^{p-1}\stackrel{Q^{p-1}}{\longrightarrow}C^p\stackrel{Q^p}{\longrightarrow}C^{p+1}\stackrel{Q^{p+1}}{\longrightarrow}\ldots
\end{equation*}
on
\begin{equation*}
C=W_0\cap\ker(b_1).
\end{equation*}
We shall see that the BRST complex and the relative BRST complex are connected by a short exact sequence of cochain complexes (see proof of Theorem~5.1 in \cite{LZ89}). Indeed, consider the map
\begin{equation*}
\psi\colon W^p_0\to C^{p-1}, w\mapsto(-1)^{|w|}b_1w.
\end{equation*}
Then for each $p\in\Z$ we obtain a short exact sequence
\begin{equation*}
0\to C^p\hookrightarrow W^p_0\stackrel{\psi}{\longrightarrow}C^{p-1}\to 0.
\end{equation*}
The surjectivity of $\psi$ follows from $\{b_1,c_{-2}\}=\id$. The exactness in the middle is clear from the definition of $\psi$.

The map $\psi$ is even a cochain map $\psi\colon W^\bullet_0\to C^{\bullet-1}$ since $\psi Q=Q\psi$ on $W_0$, which follows from $\{Q,b_1\}=L_0$, and so is the inclusion map $C^\bullet\hookrightarrow W_0^\bullet$. This means that we obtain a short exact sequence of cochain complexes
\begin{equation*}
0\to C^\bullet\hookrightarrow W^\bullet_0\stackrel{\psi}{\longrightarrow}C^{\bullet-1}\to 0.
\end{equation*}
The zig-zag lemma, which holds in any abelian category, then asserts the existence of the long exact sequence
\begin{equation*}
\ldots\to H_\text{rel.}^p\to H_\text{BRST}^p\to H_\text{rel.}^{p-1}\to H_\text{rel.}^{p+1}\to H_\text{BRST}^{p+1}\to H_\text{rel.}^{p}\to H_\text{rel.}^{p+2}\to\ldots.
\end{equation*}

We can also define a Lie algebra structure on the relative cohomology. Indeed, the identity
\begin{equation*}
(b_0u)_0v=b_1(u_{-1}v)-(b_1u)_{-1}v-(-1)^{|u|}u_{-1}(b_1v)
\end{equation*}
holds for all $u,v\in W$ (see Lemma~2.1 in \cite{LZ93}) so that
\begin{equation*}
(b_0u)_0v=b_1u_{-1}v\in\ker(b_1)
\end{equation*}
for $u,v\in\ker(b_1)$. The bracket $[\cdot,\cdot]$ can hence be restricted to $\ker(b_1)$ and we can define a Lie algebra structure on $H^1_\text{rel.}(M)$ in the same manner as for $\g=H^1_\text{BRST}(M)$ above.

\minisec{Grading}
To proceed further we need additional assumptions on the vertex algebra $M$.
\begin{ass*}\label{page:gradass}
Let the weak \voa{} $M$ have an additional grading
\begin{equation*}
M=\bigoplus_{\alpha\in\Gamma}M(\alpha)
\end{equation*}
by some (additive) abelian group $\Gamma$ in the sense of Definition~\ref{defi:addgrad}. In particular, this grading is compatible with the weight grading on $M$. Then also $W=M\otimes V_\text{gh.}$ is naturally graded by $\Gamma$ and this grading is compatible with the weight and ghost gradings on $W$, i.e.\ the corresponding grading operators commute. Let us also assume that $Q$ does not change the $\Gamma$-degree of a $\Gamma$-homogeneous element, i.e.\ that the $\Gamma$-grading operator and $Q$ commute.\footnote{It follows from the definition of $j^\text{BRST}$ and $Q$ that this is the case for example when the $\Gamma$-grading operator commutes with all modes $L_n^M$ of $\omega_M$.}
\end{ass*}

The BRST complex is now also graded by $\Gamma$, in addition to $L_0$, and we get cochain complexes
\begin{equation*}
\ldots\stackrel{Q^{p-2}}{\longrightarrow}W_n^{p-1}(\alpha)\stackrel{Q^{p-1}}{\longrightarrow}W_n^p(\alpha)\stackrel{Q^p}{\longrightarrow}W_n^{p+1}(\alpha)\stackrel{Q^{p+1}}{\longrightarrow}\ldots
\end{equation*}
for all $n\in\Z$ and $\alpha\in\Gamma$ with corresponding cohomological spaces $H_n^p(\alpha)$ so that
\begin{equation*}
H=\bigoplus_{n,p\in\Z,\alpha\in\Gamma}H_n^p(\alpha)=\bigoplus_{p\in\Z,\alpha\in\Gamma}H_0^p(\alpha),
\end{equation*}
noting that as before $H_n^p(\alpha)=\{0\}$ for $n\neq 0$. Also the relative BRST complex is graded by $\Gamma$ since $b_1$ does not change the $\Gamma$-grading and we get the subcomplexes
\begin{equation*}
\ldots\stackrel{Q^{p-2}}{\longrightarrow}C^{p-1}(\alpha)\stackrel{Q^{p-1}}{\longrightarrow}C^p(\alpha)\stackrel{Q^p}{\longrightarrow}C^{p+1}(\alpha)\stackrel{Q^{p+1}}{\longrightarrow}\ldots
\end{equation*}
for $\alpha\in\Gamma$ with corresponding cohomological spaces $H_\text{rel.}^p(\alpha)$. We also obtain the long exact sequence
\begin{equation*}
\ldots\to H_\text{rel.}^p(\alpha)\to H_\text{BRST}^p(\alpha)\to H_\text{rel.}^{p-1}(\alpha)\to H_\text{rel.}^{p+1}(\alpha)\to H_\text{BRST}^{p+1}(\alpha)\to\ldots
\end{equation*}
for $\alpha\in\Gamma$.

It is part of the above assumption (see Definition~\ref{defi:addgrad}) that the $\Gamma$-grading on the vertex algebra $M$ is such that for $v\in M(\alpha)$, the modes $v_n$, $n\in\Z$, are operators of degree $\alpha$, i.e.\ they change the degree of a homogeneous element by $\alpha$. This implies:
\begin{prop}
The Lie algebra $\g=H^1_\text{BRST}(M)=\bigoplus_{\alpha\in\Gamma}H^1_0(\alpha)$ is a $\Gamma$-graded Lie algebra, i.e.\ $[H^1_\text{BRST}(\alpha),H^1_\text{BRST}(\beta)]\subseteq H^1_\text{BRST}(\alpha+\beta)$ for $\alpha,\beta\in\Gamma$.
\end{prop}

\minisec{Vanishing Theorem}
We now make an even stronger assumption and demand that the weak \voa{} $M$ carry certain representations of the Heisenberg \voa{} related to the additional $\Gamma$-grading introduced above. 
\begin{ass*}
\item
\begin{itemize}
\item Let $V_L$ be a lattice vertex algebra associated with some even Lorentzian lattice $L$ of rank $k\geq 2$ and signature $(k-1,1)$. $V_L$ is a weak \voa{} of central charge $k$ whose isomorphism classes of irreducible modules are naturally indexed by $L'/L$.\footnote{\label{footnote:1}Dong's classification result (see Theorem~\ref{thm:3.1}) even holds for even lattices that are not positive-definite as long as the quadratic form is non-degenerate (e.g.\ Lorentzian).}
\item Let $U$ be some \voa{} of central charge $26-k$ satisfying Assumption~\ref{ass:sn} (group-like fusion) and $U_1=\{0\}$.\footnote{The assumption that $U_1=\{0\}$ is not essential but simplifies some calculations.} Assume furthermore that the fusion group $F_U$ of $U$ is isomorphic as \fqs{} to $\overline{L'/L}=(L'/L,-Q_L)$, say via the map $\chi\colon\overline{L'/L}\to F_U$.
\item Assume that $U\otimes V_L$ is isomorphic to a full weak \vosa{} of $M$ such that $M$ decomposes as a $U\otimes V_L$-module according to
\begin{equation}\label{eq:lpldec}
M\cong\bigoplus_{\gamma+L\in L'/L}U(\chi(\gamma+L))\otimes V_{\gamma+L}
\end{equation}
where $V_{\gamma+L}$, $\gamma+L\in L'/L$, are the irreducible $V_L$-modules and $\chi(\gamma+L)$ indexes the irreducible $U$-modules. Note that since $\overline{L'/L}=L'/L$ as groups, $\chi$ also accepts elements of $L'/L$ as input.
\item Let the isomorphism in \eqref{eq:lpldec} not only be an isomorphism of $U\otimes V_L$-modules but even of weak \voa{}s, where on the right-hand side we consider the vertex algebra structure obtained as abelian intertwining subalgebra of the tensor-product \aia{}\footnote{\label{footnote:2}Theorem~\ref{thm:lataia}, which asserts the existence of an \aia{} on $\bigoplus_{\gamma+L\in L'/L}V_{\gamma+L}$, also holds for even, non-degenerate (rather than only positive-definite) lattices. The \aia{} structure on $\bigoplus_{\gamma+L\in L'/L}U(\chi(\gamma+L))$ is due to Theorem~\ref{thm:2.7}, the main result of Chapter~\ref{ch:aia}.}
\begin{equation*}
\left(\bigoplus_{\gamma+L\in L'/L}U(\chi(\gamma+L))\right)\otimes \left(\bigoplus_{\gamma+L\in L'/L}V_{\gamma+L}\right),
\end{equation*}
analogously to Proposition~\ref{prop:aiavoa}.
\end{itemize}
We will refer to the above decomposition \eqref{eq:lpldec} as the \emph{$L'/L$-decomposition} of $M$.
\end{ass*}

Recall that by $M_{\hat\h}(1,0)=M_{\hat\h}(1)=V_{\hat{\h}}(1,0)$ we denote the Heisenberg \voa{} (of level $1$) associated with the $\C$-vector space $\h$ equipped with a non-degenerate symmetric bilinear form $\langle\cdot,\cdot\rangle$ (viewed as abelian Lie algebra with a non-degenerate, symmetric, invariant bilinear form). It has central charge $c=\dim_\C(\h)$ and its irreducible modules are given up to isomorphism by $M_{\hat\h}(1,\alpha)$ for each $\alpha\in\h$ with conformal weight $\langle\alpha,\alpha\rangle/2$ (see e.g.\ \cite{LL04}, Section~6.3). By construction, given an even lattice $L$, the associated lattice vertex algebra $V_L=M_{\hat\h}(1,0)\otimes\C_\eps[L]$ can be decomposed into a direct sum of modules for the full \vosa{} $M_{\hat\h}(1,0)\otimes\C\ee_0\cong M_{\hat\h}(1,0)$ as
\begin{equation*}
V_L=\bigoplus_{\alpha\in L}M_{\hat\h}(1,0)\otimes\C\ee_\alpha\cong\bigoplus_{\alpha\in L}M_{\hat\h}(1,\alpha)
\end{equation*}
where $\h=L\otimes_\Z\C$ (and $\langle\cdot,\cdot\rangle$ is extended $\C$-linearly from $L$ to $\h$) and similarly for the irreducible modules $V_{\gamma+L}$, $\gamma+L\in L'/L$.

We return to the vertex algebra $M$ and note that by the above the direct sum of modules $\bigoplus_{\gamma+L\in L'/L}V_{\gamma+L}$ is naturally graded by $\Gamma:=L'$, namely
\begin{equation*}
\bigoplus_{\gamma+L\in L'/L}V_{\gamma+L}=\bigoplus_{\gamma+L\in L'/L}\bigoplus_{\alpha\in \gamma+L}M_{\hat\h}(1,\alpha)=\bigoplus_{\alpha\in L'}M_{\hat\h}(1,\alpha)
\end{equation*}
(as $M_{\hat\h}(1,0)$-module) with $\h=L\otimes_\Z\C$. This induces a $\Gamma=L'$-grading on $M$:
\begin{equation}\label{eq:lpdec}
M=\bigoplus_{\alpha\in L'}M(\alpha)\quad\text{with}\quad M(\alpha)=U(\chi(\alpha+L))\otimes M_{\hat\h}(1,\alpha).
\end{equation}
The $L'$-grading on $M$ is a grading in the sense of Definition~\ref{defi:addgrad}. This follows from the last two items of the above assumption. In particular, $L_0$ leaves the $L'$-grading of $M$ invariant.

Recall that the $L'$-grading on $M$ naturally induces an $L'$-grading on $W=M\otimes V_\text{gh.}$. The ghost number operator, $L_0$ and $Q$ on $W$ all leave the $L'$-grading of $W$ invariant, i.e.\ they commute with the $L'$-grading operator. In total, the grading assumption on p.\ \pageref{page:gradass} is satisfied.

Consequently, we get cochain complexes $(W^\bullet(\alpha),Q^\bullet)=(W^p(\alpha),Q^p)_{p\in\Z}$ for each $\alpha\in L'$ and associated cohomological spaces $H^p(\alpha)$ so that
\begin{equation*}
H_\text{BRST}^p(M)=\bigoplus_{\alpha\in L'}H^p(\alpha)
\end{equation*}
for $p\in\Z$ and the cohomological spaces of the relative subcomplex $(C^\bullet(\alpha),Q^\bullet)=(C^p(\alpha),Q^p)_{p\in\Z}$, $\alpha\in L'$,
\begin{equation*}
H_\text{rel.}^p(M)=\bigoplus_{\alpha\in L'}H_\text{rel.}^p(\alpha)
\end{equation*}
for $p\in\Z$.

The following vanishing theorem is the central result of this section:
\begin{thm}[Vanishing Theorem]\label{thm:vanish}
Assume that $\alpha\neq 0$. Then
\begin{equation*}
H_\text{rel.}^p(\alpha)=\{0\}
\end{equation*}
for $p\neq 1$.
\end{thm}
\begin{proof}[Sketch of Proof]
This is essentially Theorem~4.9 in \cite{Zuc89}. It generalises Theorem~1.12 in \cite{FGZ86}. The proof requires the identification of the relative BRST cohomology groups with the relative semi-infinite cohomology groups for the Virasoro algebra as introduced by Feigin \cite{Fei84}. Proposition~3.3.5 in \cite{Car12b} gives a mathematical precise formulation in the case $k=2$. In essence, the vanishing theorem is a statement about the representation theory of the Virasoro algebra and the Heisenberg \voa{}. Relevant for computing $H_\text{rel.}^p(\alpha)=H_\text{rel.}^p(M(\alpha))$ is that $M(\alpha)=U(\chi(\alpha+L))\otimes M_{\hat\h}(1,\alpha)$ is the tensor product of a Virasoro representation of central charge $26-k$ and a Heisenberg representation attached to some $\alpha\in\R^{k-1,1}$.
\end{proof}

Knowing that $H_\text{rel.}^p(\alpha)=\{0\}$ for $p\neq 1$ and $\alpha\neq 0$ lets collapse the long exact sequence
\begin{equation*}
\ldots\to H_\text{rel.}^p(\alpha)\to H_\text{BRST}^p(\alpha)\to H_\text{rel.}^{p-1}(\alpha)\to H_\text{rel.}^{p+1}(\alpha)\to H_\text{BRST}^{p+1}(\alpha)\to\ldots.
\end{equation*}
for $\alpha\neq 0$ and we get
\begin{equation*}
H_\text{BRST}^1(\alpha)\cong H_\text{rel.}^1(\alpha)\cong H_\text{BRST}^2(\alpha)
\end{equation*}
for $\alpha\neq 0$ and
\begin{equation*}
H_\text{BRST}^p(\alpha)=\{0\}
\end{equation*}
for $p\neq 1,2$ and $\alpha\neq 0$.

In particular,
\begin{align*}
H_\text{BRST}^1(\alpha)&=\left(W_0^1(\alpha)\cap\ker(Q)\right)/\left(W_0^1(\alpha)\cap\im(Q)\right)\cong\\ H_\text{rel.}^1(\alpha)&=\left(W_0^1(\alpha)\cap\ker(b_1)\cap\ker(Q)\right)/\left(W_0^1(\alpha)\cap\im(b_1)\cap\ker(Q)\right)
\end{align*}
for $\alpha\neq 0$. The vector-space isomorphism is exactly the map $H_\text{rel.}^1(\alpha)\to H_\text{BRST}^1(\alpha)$ induced from the inclusion map $C^p(\alpha)\hookrightarrow W^p(\alpha)$ described above.

\minisec{Euler-Poincaré Principle}

In this step we determine the dimension of $H_\text{rel.}^1(\alpha)$ for $\alpha\neq 0$ using the Euler-Poincaré principle.

Note that the weight grading of each component of the tensor product $W(\alpha)=U(\chi(\alpha+L))\otimes M_{\hat\h}(1,\alpha)\otimes V_\text{gh.}$ is bounded from below so that
\begin{equation*}
\dim_\C(W_n(\alpha))<\infty
\end{equation*}
for all $\alpha\in L'$ and $n\in\Z$. Then in particular the spaces $B_n(\alpha)$, $C(\alpha)$, $H_\text{rel.}(\alpha)$ and $H_\text{BRST}(\alpha)$ are finite-dimensional.

The \emph{Euler-Poincaré characteristic} of the relative BRST complex $(C^\bullet(\alpha),Q^\bullet)$ is given by
\begin{equation*}
\chi(C^\bullet(\alpha)):=\sum_{p\in\Z}(-1)^p\dim_\C(C^p(\alpha)).
\end{equation*}
By the \emph{Euler-Poincaré principle}, this is the same as the Euler-Poincaré characteristic of the cohomology $H_\text{rel.}^\bullet(\alpha)=(H_\text{rel.}^p(\alpha))_{p\in\Z}$ of that complex, i.e.\
\begin{equation*}
\chi(C^\bullet(\alpha))=\chi(H_\text{rel.}^\bullet(\alpha)):=\sum_{p\in\Z}(-1)^p\dim_\C(H_\text{rel.}^p(\alpha))
\end{equation*}
for all $\alpha\in L'$. Finally, using the above vanishing theorem for $H_\text{rel.}^p(\alpha)$ we get
\begin{equation*}
-\dim_\C(H_\text{rel.}^1(\alpha))=\sum_{p\in\Z}(-1)^p\dim_\C(C^p(\alpha))
\end{equation*}
for $\alpha\neq 0$.

Now consider the supercharacter of $W(\alpha)\cap\ker(b_1)=:B(\alpha)$, a subspace of the vertex superalgebra $W$ of central charge $0$,
\begin{equation*}
\sch_{B(\alpha)}(q)=\sum_{n\in\Z}\sdim(B_n(\alpha))q^n=\sum_{n,p\in\Z}(-1)^p\dim_\C(B^p_n(\alpha))q^n
\end{equation*}
where $\sdim$ is the superdimension, i.e.\ the dimension of the even part minus the dimension of the odd part. Recall that for $V_\text{gh.}$ and hence for $W=M\otimes V_\text{gh.}$ the even and odd part have exactly even and odd ghost number, respectively. Hence $\sdim(B_n(\alpha))=\sum_{p\in\Z}(-1)^p\dim_\C(B^p_n(\alpha))$. For the constant coefficient of $\sch_{B(\alpha)}(q)$ we obtain
\begin{equation*}
\left[\sch_{B(\alpha)}(q)\right](0)=\sum_{p\in\Z}(-1)^p\dim_\C(B^p_0(\alpha))=\sum_{p\in\Z}(-1)^p\dim_\C(C^p(\alpha))=-\dim_\C(H_\text{rel.}^1(\alpha))
\end{equation*}
for $\alpha\neq 0$ since $C(\alpha)=B_0(\alpha)$.

To compute the supercharacter of $B(\alpha)=M(\alpha)\otimes(V_\text{gh.}\cap\ker(b_1))$ consider
\begin{equation*}
\sch_{B(\alpha)}(q)=\sch_{M(\alpha)}(q)\cdot\sch_{V_\text{gh.}\cap\ker(b_1)}(q)=\ch_{M(\alpha)}(q)\cdot\sch_{V_\text{gh.}\cap\ker(b_1)}(q).
\end{equation*}
The character of the kernel of $b_1$ in $V_\text{gh.}$ is computed in Section~\ref{sec:bc} to be
\begin{equation*}
\sch_{V_\text{gh.}\cap\ker(b_1)}(q)=-\eta(q)^2.
\end{equation*}
The characters of the Heisenberg \voa{} modules are known to be
\begin{equation*}
\ch_{M_{\hat\h}(1,\alpha)}(q)=\frac{q^{\langle\alpha,\alpha\rangle/2}}{\eta(q)^{\dim_\C(\hat{\h})}}=\frac{q^{\langle\alpha,\alpha\rangle/2}}{\eta(q)^{k}}
\end{equation*}
and we obtain for the character of $M(\alpha)=U(\chi(\alpha+L))\otimes M_{\hat\h}(1,\alpha)$:
\begin{equation*}
\ch_{M(\alpha)}(q)=\ch_{U(\chi(\alpha+L))}(q)\cdot\ch_{M_{\hat\h}(1,\alpha)}(q)=\ch_{U(\chi(\alpha+L))}(q)\frac{q^{\langle\alpha,\alpha\rangle/2}}{\eta(q)^{k}}.
\end{equation*}
Finally, we compute the dimension of $\g(\alpha)=H^1_\text{BRST}(\alpha)\cong H^1_\text{rel.}(\alpha)$ and obtain
\begin{align*}
\dim_\C(\g(\alpha))&=-\left[\sch_{B(\alpha)}(q)\right](0)=\left[\ch_{M(\alpha)}(q)\eta(q)^2\right](0)\\
&=\left[\ch_{U(\chi(\alpha+L))}(q)\frac{q^{\langle\alpha,\alpha\rangle/2}}{\eta(q)^{k-2}}\right](0)\\
&=\left[\ch_{U(\chi(\alpha+L))}(q)\frac{1}{\eta(q)^{k-2}}\right](-\langle\alpha,\alpha\rangle/2)
\end{align*}
for $\alpha\neq 0$. We have just proved:
\begin{thm}[Dimension Formula]\label{thm:rootdim}
The dimension of $\g(\alpha)$ is
\begin{align*}
\dim_\C(\g(\alpha))&=\left[\ch_{M(\alpha)}(q)\eta(q)^2\right](0)\\
&=\left[\ch_{U(\chi(\alpha+L))}(q)/\eta(q)^{k-2}\right](-\langle\alpha,\alpha\rangle/2)
\end{align*}
for $\alpha\in L'\setminus\{0\}$.
\end{thm}

\begin{rem}\label{rem:rank2}
In the special case of $k=2$ the above formula simplifies and $\g(\alpha)$ is isomorphic to $U(\chi(\alpha+L))_{1-\langle\alpha,\alpha\rangle/2}$, the subspace of $U(\chi(\alpha+L))$ of weight $1-\langle\alpha,\alpha\rangle/2$.
\end{rem}

\minisec{Zero-Component}

The vanishing theorem does not make a statement about the cohomology for $\alpha=0\in L'$. We compute the component $\g(0)$ directly using $U_1=\{0\}$ and that $U$ is of CFT-type:
\begin{prop}\label{prop:brstcartan1}
The component $\g(0)$ is spanned by vectors of the form
\begin{equation*}
\vac_U\otimes h(-1)1\otimes c
\end{equation*}
for $h\in\h=L\otimes_\Z\C$ and hence
\begin{equation*}
\g(0)\cong\h\cong\C^{k}\quad\text{or}\quad\dim_\C(\g(0))=\rk(L)=k.
\end{equation*}
\end{prop}
\begin{proof}
The goal is to compute
\begin{equation*}
\g(0)=H^1_\text{BRST}(0)=\left(W_0^1(0)\cap\ker(Q)\right)/\left(W_0^1(0)\cap\im(Q)\right),
\end{equation*}
where the elements of $W_0^1(0)$ have weight 0, ghost number 1 and degree $0\in L'$. Recall that $W(0)=M(0)\otimes V_\text{gh.}=U\otimes M_{\hat\h}(1,0)\otimes V_\text{gh.}$. The subspace $W_0^1(0)$ is spanned by the vectors
\begin{equation*}
\vac_M\otimes c_{-2}\vac_\text{gh.}\quad\text{and}\quad\vac_U\otimes h(-1)1\otimes c
\end{equation*}
for $h\in\h=L\otimes_\Z\C$.\footnote{Note that if we did not have $U_1=\{0\}$, there would be further contributions of the form $u\otimes\vac\otimes c$ for $u\in U_1$.} A simple calculation shows that the latter are in $\ker(Q)$ while the former is not. Moreover, $W_0^1(0)\cap\im(Q)=\{0\}$. 
\end{proof}

It is easy to see that $W_0^1(0)\cap\ker(Q)$ is in the kernel of $b_1$ so that
\begin{equation*}
H^1_\text{BRST}(0)=H^1_\text{rel.}(0)=\left(W_0^1(0)\cap\ker(b_1)\cap\ker(Q)\right)/\left(W_0^1(0)\cap\ker(b_1)\cap\im(Q)\right).
\end{equation*}
Together with $H^1_\text{BRST}(\alpha)\cong H^1_\text{rel.}(\alpha)$ for $\alpha\neq 0$, which followed from the vanishing theorem, this shows that
\begin{equation*}
\g=H^1_\text{BRST}(M)\cong H^1_\text{rel.}(M).
\end{equation*}
Recall that we defined a Lie algebra structure on $\g=H^1_\text{BRST}(M)$ and on $H^1_\text{rel.}(M)$. The vector-space isomorphism between $H^1_\text{BRST}(M)$ and $H^1_\text{rel.}(M)$, which is induced from the inclusion map $C^1\hookrightarrow W^1$, is clearly also an isomorphism of Lie algebras.

\section{Borcherds-Kac-Moody Property}\label{sec:BKMprop}

We continue in the setting of the previous section and study further properties of the infinite-dimensional Lie algebra $\g=H^1_\text{BRST}(M)$. The goal is to establish that under certain assumptions $\g$ is a \BKMa{}. \emph{\BKMa{}s} were introduced by Borcherds in \cite{Bor88}, where he called them \emph{generalised Kac-Moody algebras}. They generalise Kac-Moody algebras by, amongst other things, allowing simple roots to be imaginary.

\minisec{Invariant Bilinear Form}

A necessary condition for a Lie algebra $\g$ to be a \BKMa{} is the existence of a non-degenerate, symmetric, invariant bilinear form on $\g$. In the following steps we show this for the Lie algebra $\g=H^1_\text{BRST}(M)$ from the previous section.

In Section~\ref{sec:cmibf} we described invariant bilinear forms on \voa{}s. These concepts can be generalised to weak vertex operator (super)algebras. Since the grading of a weak vertex operator (super)algebra $V$ is not bounded from below, we have to impose the condition that $L_1$ acts locally nilpotent on $V$, i.e.\ that for each $v\in V$ there is an $n\in\N$ such that $L_1^nv=0$. Then there is still a nice theory of invariant bilinear forms on $V$ developed in \cite{Sch97,Sch98}, where it is assumed that the even and odd part of the vertex operator superalgebra are distinguished by their $L_0$-eigenvalues, i.e.\ that the even part is $\Z$-graded and the odd part $(\Z+1/2)$-graded. However, the ghost vertex operator superalgebra $V_\text{gh.}$ and consequently the weak vertex operator superalgebra $W$ are purely $\Z$-graded. This situation is described in \cite{Sch00,Sch04}.
\begin{defi}[Invariant Bilinear Form]
A bilinear form $(\cdot,\cdot)$ on a $\Z$-graded weak vertex operator superalgebra $V$ is called invariant if
\begin{equation*}
(u_nv,w)=(-1)^{|u||v|}(v,u_n^*w)
\end{equation*}
for all $u,v,w\in V$ and $n\in\Z$.
\end{defi}
The adjoint vertex operators with modes $u_n^*$, $n\in\Z$, for $u\in V$ are defined exactly as for \voa{}s (see Section~\ref{sec:cmibf}).
\begin{prop}[\cite{Sch00}, (2.36), Theorem~2.3]
Let $V$ be a $\Z$-graded weak vertex operator superalgebra and assume that $L_1$ acts locally nilpotent on $V$. Let $(\cdot,\cdot)$ be an invariant bilinear form on $V$. Then:
\begin{enumerate}
\item $(V_n,V_m)=0$ for $n\neq m\in\Z$,
\item $V$ is super-symmetric, i.e.\ $(u,v)=(-1)^{|u||v|}(v,u)$ for $\Z_2$-homogeneous $u,v\in V$.
\end{enumerate}
Moreover, the space of super-symmetric, invariant bilinear forms on $V$ is naturally isomorphic to the dual space of $V_0/L_1V_1$.
\end{prop}
This means that analogues of Propositions \ref{prop:2.6}, \ref{prop:bilorth} and Theorem~\ref{thm:3.1li} hold in this setting. The isomorphism to $(V_0/L_1V_1)^*$ can be described exactly as in the \voa{} case.

Note that $L_1$ acts locally nilpotent on $M$ (cf.\ Proposition~2.16 in \cite{Sch97}), on $V_\text{gh.}$ (trivially since the weight grading on $V_\text{gh.}$ is bounded from below) and hence on the tensor product $W=M\otimes V_\text{gh.}$.

\minisec{Step 1}

We proceed as in Proposition~3.1.8 in \cite{Car12b} to obtain an invariant bilinear form on $M$. Recall the $L'$-decomposition of $M$ in \eqref{eq:lpdec} given by
\begin{equation*}
M=\bigoplus_{\alpha\in L'}U(\chi(\alpha+L))\otimes M_{\hat\h}(1,\alpha).
\end{equation*}
There are natural pairings $\langle\cdot,\cdot\rangle$ between the $U$-modules $U(\chi(\gamma))$ and their contragredient modules $U(\chi(\gamma))'$ for $\gamma\in L'/L$. Let $\phi_\gamma$ be the module isomorphism $\phi_\gamma\colon U(\chi(-\gamma))\to U(\chi(\gamma))'$, which is unique up to a scalar. Then via
\begin{equation*}
(a,b)_{U(\chi(\gamma))}:=\langle\phi_\gamma(a),b\rangle
\end{equation*}
for $a\in U(\chi(-\gamma))$, $b\in U(\chi(\gamma))$ we can define symmetric bilinear pairings between $U(\chi(\gamma))$ and $U(\chi(-\gamma))$. These pairings are invariant in the sense that
\begin{equation*}
(Y_{U(\chi(-\gamma))}(v,x)a,b)_{U(\chi(\gamma))}=(a,Y^*_{U(\chi(\gamma))}(v,x)b)_{U(\chi(\gamma))}
\end{equation*}
for $v\in U$, $a\in U(\chi(-\gamma))$, $b\in U(\chi(\gamma))$. We proceed analogously for the Heisenberg \voa{} $M_{\hat\h}(1,0)$ and obtain symmetric, invariant bilinear pairings $(\cdot,\cdot)_{M_{\hat\h}(1,\alpha)}$ between the $M_{\hat\h}(1,0)$-modules $M_{\hat\h}(1,\alpha)$ and $M_{\hat\h}(1,-\alpha)$ for $\alpha\in L'$.

Then, on the graded tensor product $M$ we define a bilinear form via
\begin{equation*}
(a\otimes u,b\otimes v)_M:=(a,b)_{U(\chi(\alpha+L))}(u,v)_{M_{\hat\h}(1,\alpha)}
\end{equation*}
for $a\in U(\chi(-\alpha+L))$, $b\in U(\chi(\alpha+L))$, $u\in M_{\hat\h}(1,-\alpha)$ and $v\in M_{\hat\h}(1,\alpha)$ for $\alpha\in L$ and linear continuation to all of $M$. Then $(\cdot,\cdot)_M$ is clearly symmetric and invariant under $M(0)=U\otimes M_{\hat\h}(1,0)$, i.e.\ it fulfils the invariance equation with $Y(v,x)$ where $v\in M(0)$. In the following we will show that $(\cdot,\cdot)_M$ is an invariant bilinear form on $M$, i.e.\ that it is invariant under all of $M$.

To this end we define a second bilinear form on $M$. Note that since $U$ and $M_{\hat\h}(1,0)$ are of CFT-type, so is $M(0)$ and in particular the space $M(0)_0=(U\otimes M_{\hat\h}(1,0))_0$ is one-dimensional and spanned by $\vac_M=\vac_U\otimes\vac_{M_{\hat\h}(1,0)}$. Let $f$ be the unique linear functional on $M_0$ that sends $\vac_M$ to 1 and annihilates $M(\alpha)_0$ for all $\alpha\in L'\setminus\{0\}$. It is apparent that $f$ annihilates all vectors in $L_1^MM_1$ since $\vac_M\notin L^M_1M_1$. We can therefore view $f$ as a linear functional on $M_0/L_1^MM_1$ and by the above proposition this defines a symmetric, invariant bilinear form on $M$, which we call $(\cdot,\cdot)_f$.

The \voa{} $M(0)=U\otimes M_{\hat\h}(1,0)$ is of CFT-type and since $\vac_M\notin L_1^MM_1$, it follows that $\dim_\C(M(0)_0/L_1^MM(0)_1)=1$. By the above proposition this means that there is a unique invariant bilinear form on $M(0)$ up to a scalar and $(\cdot,\cdot)_M$ and $(\cdot,\cdot)_f$ coincide on $M(0)$ up to a scalar since both forms are invariant under $M(0)$. After normalising $(\cdot,\cdot)_M$ on $M(0)$ we obtain $(\cdot,\cdot)_M=(\cdot,\cdot)_f$ on $M(0)$. One can then show that $(\cdot,\cdot)_M=(\cdot,\cdot)_f$ on all of $M$ after a suitable normalisation of $(\cdot,\cdot)_M$ (see \cite{Car12b}, proof of Proposition~3.1.8). This shows in particular that $(\cdot,\cdot)_M$ is an invariant bilinear form on $M$.

We summarise these findings:
\begin{prop}
Let $(\cdot,\cdot)_M$ be the unique up to a scalar symmetric, invariant bilinear form on $M$, normalised such that $(\vac_M,\vac_M)_M=1$. Then:
\begin{enumerate}
\item\label{enum:M1} $(\cdot,\cdot)_M$ is non-degenerate on $M$.
\item\label{enum:M2} $(u,v)_M=0$ for all $u\in M(\alpha)$, $v\in M(\beta)$ with $\alpha\neq -\beta\in L'$.
\item\label{enum:M3} $(u,v)_M=0$ for all $u\in M_n$, $v\in M_m$ with $n\neq m\in\Z$.
\item\label{enum:M4} $(\cdot,\cdot)_M$ is induced by the contragredient pairings for the $U$- and $M_{\hat\h}(1,0)$-modules as described above.
\end{enumerate}
\end{prop}
\begin{proof}
Item~\ref{enum:M3} follows from the above proposition. Item~\ref{enum:M4} we just showed. Items \ref{enum:M1} and \ref{enum:M2} follow from \ref{enum:M4}.
\end{proof}

\minisec{Step 2}

Now consider the bosonic ghost vertex operator superalgebra $V_\text{gh.}$. To define an invariant bilinear form on $V_\text{gh.}$ we proceed similarly to \cite{Sch04}, Section~4. The weight-zero space $(V_\text{gh.})_0$ is 4-dimensional and $L_1^\text{gh.}(V_\text{gh.})_1$ is a subspace of dimension 3 so that
\begin{equation*}
\dim_\C\left((V_\text{gh.})_0/L_1^\text{gh.}(V_\text{gh.})_1\right)=1
\end{equation*}
and hence there is a unique invariant bilinear form on $V_\text{gh.}$ up to a scalar. Note that $\vac_\text{gh.}=1\otimes\ee_0\in L_1^\text{gh.}(V_\text{gh.})_1$ but for instance $1\otimes\ee_{3\sigma}\notin L_1^\text{gh.}(V_\text{gh.})_1$. We define $(\cdot,\cdot)_\text{gh.}$ to be the unique invariant bilinear form on $V_\text{gh.}$ with normalisation
\begin{equation*}
(\vac_\text{gh.},1\otimes\ee_{3\sigma})_\text{gh.}=1.
\end{equation*}
This corresponds to the linear functional $f$ on $(V_\text{gh.})_0$ defined via $f(L_1^\text{gh.}(V_\text{gh.})_1)=\{0\}$ and $f(1\otimes\ee_{3\sigma})=1$.
\begin{prop}
Let $(\cdot,\cdot)_\text{gh.}$ be the unique super-symmetric, invariant bilinear form on $V_\text{gh.}$, normalised such that $(\vac_\text{gh.},1\otimes\ee_{3\sigma})_\text{gh.}=1$. Then:
\begin{enumerate}
\item\label{enum:Vgh1} $(\cdot,\cdot)_\text{gh.}$ is non-degenerate on $V_\text{gh.}$.
\item\label{enum:Vgh2} $(u,v)_\text{gh.}=0$ for $u\in(V_\text{gh.})_n^p$ and $v\in(V_\text{gh.})_m^q$ with $p+q\neq 3$ or $m\neq n$.
\item\label{enum:Vgh3} $(\cdot,\cdot)_\text{gh.}$ vanishes on $\ker(b_1)$, the kernel of $b_1$ in $V_\text{gh.}$.
\end{enumerate}
\end{prop}
Also note that
\begin{equation*}
b_n^*=b_{2-n}\quad\text{and}\quad c_n^*=-c_{-4-n}
\end{equation*}
for $n\in\Z$.
\begin{proof}
Cf. \cite{Sch04}, Section~4 and \cite{Sch00}, Proposition~5.2. To prove item~\ref{enum:Vgh3} we use $b_1^*=b_1$, $\{b_1,c_{-2}\}=\id$ and the invariance of $(\cdot,\cdot)_\text{gh.}$ to get
\begin{equation*}
(u,v)_\text{gh.}=(b_1c_{-2}u,v)_\text{gh.}=(-1)^{|b||c_{-2}u|}(c_{-2}u,b_1v)_\text{gh.}=0
\end{equation*}
for $u,v\in\ker(b_1)$.
\end{proof}

\minisec{Step 3}

Finally, we define the tensor-product bilinear form $(\cdot,\cdot)_W$ on $W=M\otimes V_\text{gh.}$ via
\begin{equation*}
(a\otimes u,b\otimes v)_W:=(a,b)_M(u,v)_\text{gh.}
\end{equation*}
for $a,b\in M$, $u,v\in V_\text{gh.}$. Then $(\cdot,\cdot)_W$ is a non-degenerate, super-symmetric, invariant bilinear form on $W$ normalised such that $(\vac_M\otimes\vac_\text{gh.},\vac_M\otimes\ee_{3\sigma})_W=1$. Moreover, $(a\otimes u,b\otimes v)_W$ vanishes on $\ker(b_1)$, the kernel of $b_1$ in $W$. Note that
\begin{equation*}
Q^*=-Q
\end{equation*}
(cf.\ \cite{Sch04}, Section~4).

\minisec{Step 4}
In order to define a non-trivial bilinear form on $B=W\cap\ker(b_1)$ we proceed as in \cite{Sch00}, Section~5. We set
\begin{equation*}
(u,v)_B:=(c_{-2}u,v)_W
\end{equation*}
for $u,v\in B$.
\begin{prop}
Let $(\cdot,\cdot)_B$ be the bilinear form on $B=W\cap\ker(b_1)$ defined above. Then:
\begin{enumerate}
\item\label{enum:B1} $(\cdot,\cdot)_B$ is non-degenerate on $B$.
\item\label{enum:B2} $(u,v)_B=0$ for $u\in B_n^p(\alpha)$, $v\in B_m^q(\beta)$ with $\alpha+\beta\neq0$, $m\neq n$ or $p+q\neq2$.
\item\label{enum:B3} $(u,v)_B=-(-1)^{|u||v|}(v,u)_B$ for $\Z_2$-homogeneous $u,v\in B$.
\end{enumerate}
\end{prop}
\begin{proof}
Cf. \cite{Sch00}, Proposition~5.4. For the proof of item~\ref{enum:B1} let $u\in B$ non-zero. Because of $\{b_1,c_{-2}\}=\id$ the mode $c_{-2}$ is injective on $B$ so that $c_{-2}u$ is non-zero. The non-degeneracy of $(\cdot,\cdot)_W$ implies that there is a $v\in W$ such that $(c_{-2}u,v)_W\neq 0$. Now consider $w:=b_1c_{-2}v$. Then $w\in B$ and
\begin{align*}
(u,w)_B&=(c_{-2}u,w)_W=(c_{-2}u,b_1c_{-2}v)_W=-(-1)^{|u|}(u,c_{-2}b_1c_{-2}v)_W\\
&=-(-1)^{|u|}(u,c_{-2}v)_W=(c_{-2}u,v)_W\neq 0.
\end{align*}
Moreover
\begin{align*}
(u,v)_B&=(c_{-2}u,v)_W=-(-1)^{|u|}(u,c_{-2}v)_W=-(-1)^{|u||v|}(c_{-2}v,u)_W\\
&=-(-1)^{|u||v|}(v,u)_B
\end{align*}
for $u,v\in B$, which proves \ref{enum:B3}.
\end{proof}

\minisec{Step 5}

We now restrict $(\cdot,\cdot)_B$ to $C=B_0=W_0\cap\ker(b_1)$. We already saw that $Q$ restricts to an operator on $C$, namely the relative BRST operator.
\begin{prop}
Let $(\cdot,\cdot)_C$ be the bilinear form obtained as the restriction of $(\cdot,\cdot)_B$ to $C$. Then:
\begin{enumerate}
\item\label{enum:C1} $(\cdot,\cdot)_C$ is non-degenerate on $C$.
\item\label{enum:C2} $(u,v)_C=0$ for $u\in C^p(\alpha)$, $v\in C^q(\beta)$ with $\alpha+\beta\neq0$ or $p+q\neq2$.
\item\label{enum:C3} $(u,v)_C=-(-1)^{|u||v|}(v,u)_C$ for $\Z_2$-homogeneous $u,v\in C$.
\item\label{enum:C4} $(Qu,v)_C=-(-1)^{|u|}(u,Qv)_C$ for $u,v\in C$ and $u$ $\Z_2$-homogeneous.
\end{enumerate}
\end{prop}
\begin{proof}
The first three items follow directly from the corresponding statements for $(\cdot,\cdot)_B$. In order to prove item~\ref{enum:C4} note that $L_0c_{-2}u=0$ for $u\in C$. Using $\{Q,b_1\}=L_0$ (see Proposition~\ref{prop:QQ}) this implies for $u,v\in C$
\begin{align*}
(Qu,v)_C&=(c_{-2}Qu,v)_W=(c_{-2}L_0c_{-2}u-c_{-1}b_1Qc_{-2}u,v)_W\\
&=-(c_{-1}b_1Qc_{-2}u,v)_W=-(Qc_{-2}u,b_1c_{-2}v)_W=-(Qc_{-2}u,v)_W\\
&=-(-1)^{|u|}(c_{-2}u,Qv)_W=-(-1)^{|u|}(u,Qv)_C,
\end{align*}
where we also used $Q^*=-Q$.
\end{proof}

\minisec{Step 6}

The last item of the above proposition implies that $(u,v)_C=0$ if $u\in\im(Q)$ and $v\in\ker(Q)$ or vice versa or, since $Q^2=0$, if both $u,v\in\im(Q)$. This entails that $(\cdot,\cdot)_C$ induces a well-defined bilinear form $(\cdot,\cdot)_{H_\text{rel.}}$ on $H_\text{rel.}=(\ker(Q)\cap C)/(\im(Q)\cap C)$.
\begin{prop}
Let $(\cdot,\cdot)_{H_\text{rel.}}$ be the bilinear form induced from $(\cdot,\cdot)_C$ on $H_\text{rel.}$. Then:
\begin{enumerate}
\item\label{enum:Hrel1} $(\cdot,\cdot)_{H_\text{rel.}}$ is non-degenerate on $H_\text{rel.}$.
\item\label{enum:Hrel2} $(u,v)_C=0$ for $u\in H_\text{rel.}^p(\alpha)$, $v\in H_\text{rel.}^q(\beta)$ with $\alpha+\beta\neq0$ or $p+q\neq2$.
\item\label{enum:Hrel3} $(u,v)_{H_\text{rel.}}=-(-1)^{|u||v|}(v,u)_{H_\text{rel.}}$ for $\Z_2$-homogeneous $u,v\in H_\text{rel.}$.
\end{enumerate}
\end{prop}
\begin{proof}
Cf. \cite{Sch00}, Proposition~5.5. Items \ref{enum:Hrel2} and \ref{enum:Hrel3} follow directly from the corresponding items for $(\cdot,\cdot)_C$.
\end{proof}

\minisec{Step 7}

By restriction of $(\cdot,\cdot)_{H_\text{rel.}}$ we finally arrive at a bilinear form $(\cdot,\cdot)_\g$ on the Lie algebra $H^1_\text{rel.}(M)\cong H^1_\text{BRST}(M)=\g$.
\begin{prop}\label{prop:lainvbil}
Let $(\cdot,\cdot)_\g$ be the bilinear form on $\g$ obtained from the restriction of $(\cdot,\cdot)_{H_\text{rel.}}$ to $\g$. Then:
\begin{enumerate}
\item\label{enum:gg1} $(\cdot,\cdot)_\g$ is non-degenerate on $\g$.
\item\label{enum:gg2} $(u,v)_\g=0$ for $u\in\g(\alpha)$, $v\in \g(\beta)$ with $\alpha+\beta\neq0$.
\item\label{enum:gg3} $(\cdot,\cdot)_\g$ is symmetric, i.e.\ $(u,v)_\g=(v,u)_\g$ for all $u,v\in\g$.
\item\label{enum:gg4} $(\cdot,\cdot)_\g$ is invariant, i.e.\ $([u,v],w)_\g=(u,[v,w])_\g$ for all $u,v,w\in\g$
\end{enumerate}
\end{prop}
\begin{proof}
Items \ref{enum:gg1} and \ref{enum:gg2} follow directly from the corresponding items for $(\cdot,\cdot)_{H_\text{rel.}}$. Item~\ref{enum:gg3} follows from item~\ref{enum:gg3} of the above proposition with $|u|=|v|=1$. The proof of item~\ref{enum:gg4} is a special case of the proof given in \cite{Sch00}, Proposition~5.17. 
\end{proof}

Recall that the bilinear form $\langle\cdot,\cdot\rangle$ on the lattice $L$ extends
to a non-degenerate bilinear form on the $\C$-vector space $\h=L\otimes_\Z\C$. On the other hand, we know from the above proposition that $(\cdot,\cdot)_\g$ restricts to a non-degenerate bilinear form on $\g(0)$. We showed in Proposition~\ref{prop:brstcartan1} that $\g(0)\cong\h$ as vector spaces. In fact, with the chosen normalisation of $(\cdot,\cdot)_\g$ this isomorphism is an isometry:
\begin{prop}
There is an isometry
\begin{equation*}
\left(\h,\langle\cdot,\cdot\rangle\right)\cong\left(\g(0),(\cdot,\cdot)_\g\right)
\end{equation*}
defined by $h\mapsto\vac_U\otimes h(-1)1\otimes c$ for all $h\in\h$.
\end{prop}
\begin{proof}
Cf. \cite{Sch04}, Section~4.2.
\end{proof}

\minisec{\BKMA{}}

In the following we investigate under which conditions the Lie algebra $\g$ is a \BKMa{}. In \cite{Bor95a}, Theorem~1, Borcherds describes a criterion for certain Lie algebras to be \BKMa{}s. However, the theorem only treats the case of Lie algebras defined over the real numbers $\R$. Carnahan describes a generalisation to complex Lie algebras with a real structure on the Cartan subalgebra, noting that Borcherds' proof still works in this case.
\begin{thm}[\cite{Car12b}, Lemma~3.4.2]\label{thm:bkma}
Let $\g$ be a complex Lie algebra. If $\g$ satisfies the following conditions, then $\g$ is a complex \BKMa{}:
\begin{enumerate}
\item\label{enum:bkma1} $\g$ admits a non-degenerate, symmetric, invariant bilinear form $(\cdot,\cdot)$.
\item\label{enum:bkma2} $\g$ has a self-centralising subalgebra $H$, called a \emph{Cartan subalgebra}, such that $\g$ is the direct sum of eigenspaces under the adjoint action of $H$, and the non-zero eigenvalues, called \emph{roots}, have finite multiplicity.
\item\label{enum:bkma3} There is a subspace $H_\R$ of $H$ as real vector space such that $H=H_\R\oplus\i H_\R$, the bilinear form is real-valued on $H_\R$ and the roots lie in the dual $(H_\R)^*$.
\item\label{enum:bkma4} The norms of the roots under the inner product $(\cdot,\cdot)$ are bounded from above.
\item\label{enum:bkma5} There exists an element $h_\text{reg.}\in H_\R$ such that:
\begin{enumerate}
\item $H=C_\g(h_\text{reg.})$, the centraliser of $h_\text{reg.}$ in $\g$,
\item for any $M\in \R$, there exist only finitely many roots $\alpha$ such that $|\alpha(h_\text{reg.})|<M$.
\end{enumerate}
This vector $h_\text{reg.}$ is called a \emph{regular element}. If $\alpha(h_\text{reg.})<0$, then we say that the root $\alpha$ is negative and if $\alpha(h_\text{reg.})>0$, then we say that $\alpha$ is positive.
\item\label{enum:bkma6} Any two roots of non-positive norm that are both positive or both negative have inner product at most zero and if the inner product is zero, then their root spaces commute.
\end{enumerate}
\end{thm}

In order to prove that $\g=H^1_\text{BRST}(M)$ is a \BKMa{} we have to show that each of the items \ref{enum:bkma1} to \ref{enum:bkma6} in the above theorem is fulfilled.
\begin{prop}
The Lie algebra $\g=H^1_\text{BRST}(M)$ from the BRST construction fulfils conditions \ref{enum:bkma1} to \ref{enum:bkma5} in the above theorem.
\end{prop}
\begin{proof}
Item~\ref{enum:bkma1} is simply the statement of Proposition~\ref{prop:lainvbil}. Recall that the Lie algebra $\g$ is graded by $L'$, i.e.\ $[\g(\alpha),\g(\beta)]\subseteq\g(\alpha+\beta)$, $\alpha,\beta\in L'$. Then
\begin{equation*}
H:=\g(0)
\end{equation*}
is a Lie subalgebra of $\g$. $H$ acts on $\g$ in the adjoint representation as $[x,y]=\langle h,\alpha\rangle y$ for $x=\vac_U\otimes h(-1)1\otimes c\in H$ and $y\in\g(\alpha)$, $\alpha\in L'$. This implies that
\begin{equation*}
C_\g(H)=H,
\end{equation*}
i.e.\ $H$ is self-centralising. $H$ is a Cartan subalgebra of $\g$.

We abuse notation and write $h\in\h$ for the element $\vac_U\otimes h(-1)1\otimes c\in H=\g(0)$, identifying $H$ with $\h$. Since the bilinear form on $H$ is non-degenerate, we can further identify $H\cong\h$ with $\h^*$ via $\alpha(\cdot):=\langle\alpha,\cdot\rangle$ for $\alpha\in\h$. Then
\begin{equation*}
[h,x]=\alpha(h)x
\end{equation*}
for $h\in H$ and $x\in\g(\alpha)$, i.e.\ $\g(\alpha)$ is the root space associated with $\alpha\in L'\setminus\{0\}$. The set of roots $\Phi\subseteq L'\setminus\{0\}$ are those $\alpha$ for which $\g(\alpha)\neq 0$. Then $\g$ decomposes into the direct sum
\begin{equation*}
\g=H\oplus\bigoplus_{\alpha\in\Phi}\g(\alpha)
\end{equation*}
with Cartan subalgebra $H$ and root spaces $\g(\alpha)$, $\alpha\in\Phi$. Theorem~\ref{thm:rootdim} states in particular that $\dim_\C(\g(\alpha))<\infty$ for all $\alpha\in L'\setminus\{0\}$, i.e.\ the root spaces are finite-dimensional. This completes the proof of item~\ref{enum:bkma2}.

We identified $H\cong\h=L\otimes_\Z\C$, which has a natural real subspace $H_\R:=L\otimes_\Z\R$, on which the bilinear form takes real values, and the roots, identified with elements of the lattice $L'$, lie in $H_\R^*$. This shows item~\ref{enum:bkma3}.

Under the identifications presented above the norm of a root $\alpha\in\Phi$ is exactly $\langle\alpha,\alpha\rangle/2$. Then from
\begin{equation*}
\dim_\C(\g(\alpha))=\left[\ch_{U(\chi(\alpha+L))}(q)/\eta(q)^{k-2}\right](-\langle\alpha,\alpha\rangle/2)
\end{equation*}
(see Theorem~\ref{thm:rootdim}) we conclude that $\g(\alpha)=\{0\}$ if $\langle\alpha,\alpha\rangle/2$ is larger than a certain bound since for each $\alpha+L\in L'/L$ the character in the above formula has some smallest exponent and $L'/L$ is finite. This proves \ref{enum:bkma4}.

It is shown in \cite{Bor95a}, Theorem~2 that item~\ref{enum:bkma5} is automatically fulfilled if the bilinear form $(\cdot,\cdot)_\g$ is Lorentzian when restricted to $H$. In the case of complex $\g$ we have to replace $H$ by $H_\R\cong L\otimes_\Z\R$ and can apply the same argument. Indeed, we assumed the lattice $L$ to be Lorentzian, i.e.\ the quadratic form on $L$ (and $L\otimes_\Z\R$) has signature $(k-1,1)$. As regular element we can take any negative-norm vector $h_\text{reg.}\in H_\R$ whose inner product with any root is non-zero. This gives \ref{enum:bkma5}.
\end{proof}

It remains to investigate condition~\ref{enum:bkma6} in the above theorem for $\g=H^1_\text{BRST}(M)$ from the BRST construction. One can either show \ref{enum:bkma6} directly or make use of the following observation due to Borcherds for the Lorentzian case:
\begin{prop}[\cite{Bor95a}, Theorem~2]\label{prop:bor95thm2}
Let $\g$ be a complex Lie algebra satisfying conditions \ref{enum:bkma1} to \ref{enum:bkma4} in the above theorem. Assume that the bilinear form restricted to $H_\R$ is Lorentzian, i.e.\ has signature $(\dim_\C(H)-1,1)$. Then \ref{enum:bkma6} is fulfilled if the following holds: if two roots are positive multiples of the same norm-zero vector, then their root spaces commute.
\end{prop}
Proving condition~\ref{enum:bkma6} requires further details of the weak \voa{} $M$ in the matter sector.

\section{Examples}\label{sec:brstex}

Before proceeding with the BRST construction of the ten \BKMa{}s from \cite{Sch04b,Sch06} (see Section~\ref{sec:10bkmas}) we give an overview of all the BRST constructions of \BKMA{}s obtained so far using the above described procedure. These examples are summarised in Table~\ref{tab:BRST}. (There is one construction \cite{CKS07} which we omit in this summary.) All of the following constructions, except for the one from \cite{Car12b}, are carried out over $\R$ by considering both the matter vertex algebra $M$ and the ghost vertex operator superalgebra over $\R$.

\minisec{The Fake Monster Lie Algebra (Rank 26)}

Borcherds constructed the following Lie algebra in \cite{Bor90}, Theorem~3, and originally called it Monster Lie algebra but it is now referred to as the Fake Monster Lie algebra. Recall that $\II_{25,1}$ denotes the unique even, unimodular lattice of Lorentzian signature $(25,1)$ and consider the corresponding weak \voa{}
\begin{equation*}
M:=V_{\II_{25,1}}
\end{equation*}
of central charge $26$. We set $L:=\II_{25,1}$. Then $M$ has the trivial $L'/L$-decomposition
\begin{equation*}
M\cong V_{\{0\}}\otimes V_{\II_{25,1}}.
\end{equation*}
We set $U=V_{\{0\}}$ and clearly $U_1=\{0\}$.

It is shown that $\g=H^1_\text{BRST}(M)$ is a \BKMa{} graded by the lattice $L'=(\II_{25,1})'=\II_{25,1}$ of rank 26, called the \emph{Fake Monster Lie algebra}.

\minisec{The Monster Lie Algebra (Rank 2)}

Constructing the Monster Lie algebra is an important step in Borcherds' proof of the Moonshine conjecture \cite{Bor92}. Let $U:=V^\natural$ be the Moonshine module, i.e.\ the \voa{} of central charge 24 obtained by orbifolding the \voa{} $V_\Lambda$ associated with the Leech lattice $\Lambda$ by a standard lift (of order 2) of the $(-1)$-involution in $\Aut(\Lambda)$. We obtain a weak \voa{} of central charge 26 by taking the tensor product with $V_{\II_{1,1}}$, the vertex algebra associated with the unique even, unimodular lattice of Lorentzian signature $(1,1)$, i.e.\ we consider
\begin{equation*}
M:=V^\natural\otimes V_{\II_{1,1}}.
\end{equation*}
We set $L:=\II_{1,1}$. Then this is also the (trivial) $L'/L$-decomposition of $M$ and we set $U:=V^\natural$ so that $U_1=(V^\natural)_1=\{0\}$.

Borcherds showed that $\g=H^1_\text{BRST}(M)$ is a \BKMa{} of rank 2, graded by $L'=(\II_{1,1})'=\II_{1,1}$, called the \emph{Monster Lie algebra}.

\minisec{The Fake Baby Monster Lie Algebra (Rank 18)}

The following Lie algebra was obtained in a BRST construction in \cite{HS03}. Let $N(A_3^8)$ be the Niemeier lattice with root lattice $A_3^8$ and $V_{N(A_3^8)}$ the corresponding lattice \voa{} of central charge 24. As in the above case we consider the $(-1)$-involution orbifold $\widetilde{V}$ of $V_{N(A_3^8)}$. ($\widetilde{V}$ then corresponds to case 5 in Schellekens' list with affine structure $A_{1,2}^{16}$ \cite{DGM90}.) Again, we set
\begin{equation*}
M:=\widetilde{V}\otimes V_{\II_{1,1}},
\end{equation*}
which is a weak \voa{} of central charge 26.

Let $\Lambda_{16}$ denote the Barnes-Wall lattice (the rank 16 positive-definite, even lattice of discriminant $2^8$ with no norm-one vectors, see e.g.\ \cite{CS99}, Section~4.10) and let $\II_{1,1}(2)$ be the lattice $\II_{1,1}$ with the quadratic form rescaled by 2. (While $\II_{1,1}$ is unimodular, $\II_{1,1}(2)$ is not. Indeed, $(\II_{1,1}(2))'/\II_{1,1}(2)\cong\Z_2^2$, cf.\ Proposition~\ref{prop:rescaled}.) We set $L:=\Lambda_{16}\oplus \II_{1,1}(2)$ with discriminant form $L'/L\cong \Z_2^{10}$ (and some quadratic form).

Moreover, let $V_{E_8}$ be the \voa{} of central charge 8 associated with the lattice $E_8$ and $U:=V_{E_8}^F$ the \fpvosa{} under some group $F\leq\Aut(V_{E_8})$. More specifically $F\cong\Z_2^5$ is a certain subgroup described in \cite{Gri98}. It is also shown that $U_1=(V^F_{E_8})_1=\{0\}$. The rational \voa{} $U$ has group-like fusion with fusion group $F_U\cong\Z_2^{10}$ (as group). In fact, as \fqs{}s $F_U\cong\overline{L'/L}$ (via some isomorphism $\chi$).
$M$ has $L'/L$-decomposition
\begin{equation*}
M\cong\bigoplus_{\gamma+L\in L'/L} V_{E_8}^F(\chi(\gamma+L))\otimes V_{\gamma+L}
\end{equation*}
where $V_{E_8}^F(\chi(\gamma+L))$ is the irreducible $V_{E_8}^F$-module associated with $\chi(\gamma+L)$ in the fusion group of $V_{E_8}^F$.

The Lie algebra $\g=H^1_\text{BRST}(M)$ is a \BKMa{} of rank 18, graded by $L'=(\Lambda_{16}\oplus \II_{1,1}(2))'$ and called the \emph{Fake Baby Monster Lie algebra}.

\minisec{The Baby Monster Lie Algebra (Rank 2)}

The following BRST construction is described in \cite{Hoe03a}. Let $t$ be a $2A$-involution in the Monster group $M\cong \Aut(V^\natural)$ (following the notation in \cite{CCNPW85}). We consider the \fpvosa{} $U:=(V^\natural)^t$, which has fusion group $F_U\cong\Z_2^2$. We set $L:=\II_{1,1}(2)$. Then $F_U\cong \overline{L'/L}$ as \fqs{}s (via some isomorphism $\chi$) and the $L'/L$-graded tensor product
\begin{equation*}
M:=\bigoplus_{\gamma+L\in L'/L} (V^\natural)^t(\chi(\gamma+L))\otimes V_{\gamma+L}
\end{equation*}
is a weak \voa{} of central charge 26. This is the $L'/L$-decomposition of $M$ and clearly $U_1=\{0\}$.

We call $\g=H^1_\text{BRST}(M)$ the \emph{Baby Monster Lie algebra} and it is an $L'=(\II_{1,1}(2))'$-graded \BKMa{} of rank 2.

\minisec{Some \BKMA{} of Rank 14}

In \cite{HS14} the authors describe the construction of a \BKMa{} of rank 14. Consider the positive-definite, even lattice $K=D_{12}^+(2)$ of rank 12 (lattice $D_{12}^+$ with quadratic form rescaled by 2). Let $U:=V_K^+$ be the \fpvosa{} of the corresponding lattice \voa{} $V_K$ of central charge 12 under a standard lift of the $(-1)$-involution. The fusion of $U=V_K^+$ is group-like with fusion group $F_U\cong\Z_2^{10}\times\Z_4^2$.

Set $L:=D_{12}(2)\oplus \II_{1,1}$, an even lattice of rank 14 and Lorentzian signature. Then $F_U\cong\overline{L'/L}$ as \fqs{}s (via some isomorphism $\chi$) and the direct sum
\begin{equation*}
M:=\bigoplus_{\gamma+L\in L'/L} V_K^+(\chi(\gamma+L))\otimes V_{\gamma+L}
\end{equation*}
is a weak \voa{} of central charge 26. This is the $L'/L$-decomposition of $M$ and $U_1=\{0\}$.

The result of the BRST construction is an $L'=(D_{12}(2)\oplus \II_{1,1})'$-graded \BKMa{} $\g=H^1_\text{BRST}(M)$ of rank 14.

\minisec{\BKMA{}s for Fricke Elements in the Monster (Rank 2)}

In \cite{Car12b} the author generalises the BRST constructions in \cite{Bor92} and \cite{Hoe03a}, making use of the cyclic orbifold theory developed in this text and \cite{EMS15}. Let $g$ be any automorphism in $M\cong \Aut(V^\natural)$ of Fricke type, i.e.\ the McKay-Thompson series of $g$ has a pole at zero. (143 of the 194 conjugacy classes in the Monster are Fricke, including $g=\id$ and the conjugacy class $2A$.) We consider the \fpvosa{} $U:=(V^\natural)^g$, which has some fusion group $F_U$ of order $n^2$ where $n=\ord(g)$.

We can find an even, hyperbolic lattice $L$ of genus $\II_{1,1}(\overline{F_U})$, i.e.\ with $F_U\cong\overline{L'/L}$ as \fqs{}s (via some isomorphism $\chi$). Then the $L'/L$-graded tensor product
\begin{equation*}
M:=\bigoplus_{\gamma+L\in L'/L} (V^\natural)^g(\chi(\gamma+L))\otimes V_{\gamma+L}
\end{equation*}
is a weak \voa{} of central charge 26. Clearly, $U_1=\{0\}$.

The BRST construction takes $M$ to an $L'$-graded \BKMa{} $\g=H^1_\text{BRST}(M)$ of rank $2$. The proof of the Borcherds-Kac-Moody property makes use of the Fricke property of $g$ but can probably be carried out in the non-Fricke case as well.

Carnahan uses these results to show that for all Fricke elements in the Monster group, the characters of centralisers acting on the corresponding irreducible twisted modules are Hauptmoduln. This proves the remaining open claims in Norton's generalised Moonshine conjecture \cite{Nor87,Nor01}.

\renewcommand{\arraystretch}{1.35}
\newgeometry{bottom=2cm,top=2cm}
\thispagestyle{empty}
\begin{sidewaystable}
\centering
\begin{tabular}{c|c|c|c|c}
$M$&
\begin{tabular}{c}
$L'/L$-Decomposition\\
$c=(26-k)+k$
\end{tabular}
&
\begin{tabular}{c}
$L$\\
$k=\rk(L)$\\
$\sign(L)$
\end{tabular}
&BKMA $\g=H^1_\text{BRST}(M)$&Ref.\\\hline\hline

\begin{tabular}{c}
$M=V^\natural\otimes V_{\II_{1,1}}$\\
$V^\natural$: $(-1)$-orbifold of $V_{\Lambda}$\\
$\Lambda$: Leech lat.
\end{tabular}
&
\begin{tabular}{c}
$M=V^\natural\otimes V_L$\\
$24+2$
\end{tabular}
&
\begin{tabular}{c}
$\II_{1,1}$\\
2\\
$(1,1)$
\end{tabular}
&
\begin{tabular}{c}
Monster Lie alg.\\
(MA)
\end{tabular}
&\cite{Bor92}\\\hline

\begin{tabular}{c}
$M=\bigoplus_{\gamma\in L'/L} (V^\natural)^t(\gamma)\otimes V_{\gamma+L}$\\
$(V^\natural)^t$: f.p.\ of $V^\natural$ under aut.\ $t$ in class $2A$\\
$(V^\natural)^t(\gamma)$: irr.\ $(V^\natural)^t$-modules, ind.\ by $\gamma\in \overline{L'/L}$
\end{tabular}
&
\begin{tabular}{c}
$M=\bigoplus\limits_{\gamma\in L'/L} (V^\natural)^t(\gamma)\otimes V_{\gamma+L}$\\
$24+2$
\end{tabular}
&
\begin{tabular}{c}
$\II_{1,1}(2)$\\
2\\
$(1,1)$
\end{tabular}
&
\begin{tabular}{c}
Baby Monster Lie alg.\\
(BMA)
\end{tabular}
&\cite{Hoe03a}\\\hline

\begin{tabular}{c}
$M=\bigoplus_{\gamma\in L'/L} (V^\natural)^g(\gamma)\otimes V_{\gamma+L}$\\
$(V^\natural)^g$: f.p.\ of $V^\natural$ under a Fricke aut.\ $g$ of ord.\ $n$\\
$(V^\natural)^g(\gamma)$: irr.\ $(V^\natural)^g$-modules, ind.\ by $\gamma\in \overline{L'/L}$
\end{tabular}
&
\begin{tabular}{c}
$M=\bigoplus\limits_{\gamma\in L'/L} (V^\natural)^g(\gamma)\otimes V_{\gamma+L}$\\
$24+2$
\end{tabular}
&
\begin{tabular}{c}
$L$\\
2\\
$(1,1)$
\end{tabular}
&
\begin{tabular}{c}
BKMAs of rank 2\\
(incl.\ MA and BMA)
\end{tabular}
&\cite{Car12b}\\\hline\hline

\begin{tabular}{c}
$M=V_\Lambda\otimes V_{\II_{1,1}}\cong V_{\II_{25,1}}$
\end{tabular}
&
\begin{tabular}{c}
$M\cong V_{\{0\}}\otimes V_L$\\
$0+26$
\end{tabular}
&
\begin{tabular}{c}
$\II_{25,1}$\\
26\\
$(25,1)$
\end{tabular}
&
\begin{tabular}{c}
Fake Monster Lie alg.\\
(FMA)
\end{tabular}
&\cite{Bor90}\\\hline

\begin{tabular}{c}
$M=\widetilde{V}\otimes V_{\II_{1,1}}$\\
$\widetilde{V}$: $(-1)$-orbifold of $V_{N(A_3^8)}$\\
$N(A_3^8)$: Niemeier lat.\ with root lattice $A_3^8$
\end{tabular}
&
\begin{tabular}{c}
$M\cong\bigoplus\limits_{\gamma\in L'/L} V_{E_8}^F(\gamma)\otimes V_{\gamma+L}$\\
(some $F<\Aut(V_{E_8})$)\\ 
$8+18$
\end{tabular}
&
\begin{tabular}{c}
$\Lambda_{16}\oplus \II_{1,1}(2)$\\
18\\
$(17,1)$
\end{tabular}
&
\begin{tabular}{c}
Fake Baby Monster Lie alg.\\
(FBMA)
\end{tabular}
&\cite{HS03}\\\hline

\begin{tabular}{c}
$M=\bigoplus_{\gamma\in K'/K} V_\Lambda^{\hat\nu}(\gamma)\otimes V_{\gamma+K}$\\
$\nu$ aut.\ in $M_{23}$ of Leech lat.\ $\Lambda$ of sq.f.\ ord.\ $m$\\
$m=1,2,3,5,6,7,11,14,15,23$\\
$K=\II_{1,1}(m)$\\
\end{tabular}
&
\begin{tabular}{c}
$M\cong\bigoplus\limits_{\gamma\in L'/L} V_N^{\hat\rho}(\gamma)\otimes V_{\gamma+L}$\\
($N=\Lambda_\nu$, $\rho=\nu|_N$)\\
$(24-\rk(\Lambda^\nu))+(2+\rk(\Lambda^\nu))$
\end{tabular}
&
\begin{tabular}{c}
$\Lambda^\nu\oplus \II_{1,1}(m)$\\
$\rk(\Lambda^\nu)+2$\\
$(\rk(\Lambda^\nu)+1,1)$
\end{tabular}
&
\begin{tabular}{c}
10 BKMAs \cite{Sch04b,Sch06}\\
of rank 26, 18, 14, 10,\\
10, 8, 6, 6, 6, 4\\
(incl. FMA and FBMA)
\end{tabular}
&
\begin{tabular}{c}
this\\
text
\end{tabular}
\\\hline\hline

\begin{tabular}{c}
$M=\bigoplus_{\gamma\in L'/L} V_K^+(\gamma)\otimes V_{\gamma+L}$\\
$V_K^+$: f.p.\ of $V_K$ under $(-1)$-aut.\\
$K=D_{12}^+(2)$: even pos.-def.\ lat.\ of rank 12\\
$V_K^+(\gamma)$: irr.\ $V_K^+$-modules, ind.\ by $\gamma\in \overline{L'/L}$
\end{tabular}
&
\begin{tabular}{c}
$M=\bigoplus\limits_{\gamma\in L'/L} V_K^+(\gamma)\otimes V_{\gamma+L}$\\
$12+14$
\end{tabular}
&
\begin{tabular}{c}
$D_{12}(2)\oplus \II_{1,1}$\\
14\\
$(13,1)$
\end{tabular}
&BKMA of rank 14&\cite{HS14}\\\hline

\end{tabular}
\caption{Overview of BRST constructions of \BKMa{}s.}
\label{tab:BRST}
\end{sidewaystable}
\restoregeometry
\renewcommand{\arraystretch}{1}

\section{Non-Holomorphic Orbifold Theory}\label{sec:nonholorb}

A necessary ingredient for the BRST construction of the ten \BKMa{}s in Section~\ref{sec:10bkmas} is an orbifold theory for \fpvosa{}s of \emph{non-holomorphic} lattice \voa{}s.

In Chapter~\ref{ch:fpvosa} we developed an orbifold theory for a holomorphic \voa{} $V$ and some finite, cyclic group $G=\langle\sigma\rangle$ of automorphisms of $V$, i.e.\ we studied the representation theory of the \fpvosa{} $V^G$. It is natural to ask whether this can be generalised to non-holomorphic \voa{}s. We will certainly need to assume that $V$ fulfils Assumption~\ref{ass:n}. In the holomorphic case it was essential that we had a complete classification of the $g$-twisted $V$-modules for any $g\in G$: there is one such module up to isomorphism for each $g$ in this case. Indeed, recall that all $V^G$-modules could be obtained as $V^G$-submodules of the $\sigma^i$-twisted modules.

Now let $N$ be a positive-definite, even lattice and $V_N$ the associated lattice \voa{}. Unlike before, we do not require $N$ to be unimodular. $V_N$ fulfils Assumptions~\ref{ass:n}\ref{ass:p}. Moreover, for a standard lift $\hat{\rho}\in\Aut(V_N)$ of a lattice automorphism $\rho\in\Aut(N)$ the $\hat{\rho}$-twisted modules are classified in \cite{BK04} (see Theorem~\ref{thm:4.2}). Only in this section let $\h:=N\otimes_\Z\C$ denote the complexification of $N$ (instead of that of the lattice $L$).

In the following we study the representation theory of $V_N^{\hat{\rho}}$ under certain simplifying assumptions:
\begin{customass}{{\textbf{\textsf{B}}}}\label{ass:b}
Let $N$ be a positive-definite, even lattice and $V_N$ the associated lattice \voa{}. Let $\rho\in\Aut(N)$ have order $m\in\Ns$ and let $\hat{\rho}\in\Aut(V_N)$ be a standard lift of $\rho$. We further assume:
\begin{enumerate}
\item\label{enum:assb1} The power $\hat{\rho}^k$ is a standard lift of $\rho^k$ for all $k\in\N$. This is the case if and only if for even $m$, $\langle\rho^{k/2}\alpha,\alpha\rangle\in 2\Z$ for all $\alpha\in N^{\rho^k}$, $k\in\N$ even. In particular, $\ord(\hat{\rho})=\ord(\rho)=m$ and more generally $\ord(\hat{\rho}^k)=\ord(\rho^k)=m/(m,k)$ for all $k\in\N$.
\item\label{enum:assb2} The lattice automorphism $\rho$ and hence $\rho^i$ for all $i\in\Z_m$ act trivially on $N'/N$, i.e.\ $(N'/N)^{\rho^i}=N'/N$ for all $i\in\Z_m$. This is the case if and only if $(\id-\rho^i)\alpha\in N$ for all $\alpha\in N'$, $i\in\Z_m$.
\item\label{enum:assb3} If $m$ is even, then $\langle(\rho^{m/2}-\id)\alpha,\alpha\rangle\in 2\Z$ for all $\alpha\in N'$.
\item\label{enum:assb4} The vacuum anomaly $\rho_{\rho^i}=0\pmod{(m,i)/m}$ for all $i\in\Z_m$.
\item\label{enum:assb5} If $m$ is even, there is a basis $\{\alpha_1,\ldots,\alpha_r\}$ of $N'$ such that $\{m_1\alpha_1,\ldots,m_r\alpha_r\}$ is a basis of $N$ for suitable $m_1,\ldots,m_r\in\Ns$ and $\langle m_i\alpha_i,\alpha_i\rangle\in 2\Z$ for all $i=1,\ldots,r$.\footnote{For any integral lattice $N$ there is a basis $\{\alpha_1,\ldots,\alpha_r\}$ of $N'$ such that $\{m_1\alpha_1,\ldots,m_r\alpha_r\}$ is a basis of $N$ for some $m_1,\ldots,m_r\in\Ns$.}
\item\label{enum:assb6} If $m$ is even, then for all even $k\in\N$, $\langle\alpha,(\rho^{k/2}-\id)\alpha\rangle\in 2\Z$ for all $\alpha\in(N')^{\rho^k}$. This includes items \ref{enum:assb1} and \ref{enum:assb3} as special cases.
\end{enumerate}
\end{customass}

\minisec{Irreducible Modules}

We begin by determining the irreducible $\rho^i$-twisted $V_N$-modules:
\begin{lem}
Let $N$ and $\hat\rho$ be as in Assumption~\ref{ass:b} with \ref{enum:assb1} and \ref{enum:assb2}. Then for any $i\in\Z_m$ the isomorphism classes $V_{\alpha+N}(\hat\rho^i)$ of irreducible $\hat\rho^i$-twisted $V_N$-modules are parametrised by $\alpha+N\in N'/N$ (see Theorem~\ref{thm:4.2}).
\end{lem}
\begin{proof}
The isomorphism classes of irreducible untwisted $V_N$-modules are parametrised by $N'/N$ according to \cite{Don93} (see Theorem~\ref{thm:3.1}). More generally, in \cite{BK04} the authors classify the irreducible twisted modules for lattice \voa{}s (see Theorem~\ref{thm:4.2}). The result is that the isomorphism classes of irreducible $\hat{\rho}^i$-twisted $V_N$-modules are parametrised by $(N'/N)^{\rho^i}$, assuming that $\hat{\rho}^k$ is a standard lift of $\rho^k$ for all $k\in\N$, which is item~\ref{enum:assb1} of Assumption~\ref{ass:b}. We also assumed that $\rho$ and hence $\rho^i$ for all $i\in\Z_m$ act trivially on the quotient $N'/N$ (see item~\ref{enum:assb2} of Assumption~\ref{ass:b}). Hence for each $i\in\Z_m$ the irreducible $\hat{\rho}^i$-twisted $V_N$-modules are parametrised by $(N'/N)^{\rho^i}=N'/N$ and explicit models for them are constructed in \cite{DL96,BK04}.
\end{proof}

By Theorem~\ref{thm:orb} and Proposition~\ref{prop:thm3.3}, $V_N^{\hat{\rho}}$ satisfies Assumption~\ref{ass:n} and each irreducible $V_N^{\hat{\rho}}$-module appears as submodule of the irreducible $\hat{\rho}^i$-twisted $V_N$-modules for some $i\in\Z_m$. Then $V_N^{\hat{\rho}}$ also satisfies Assumption~\ref{ass:p} as can be easily seen from the construction of the irreducible $\hat{\rho}^i$-twisted $V_N$-modules.

We determine the irreducible $V_N^{\hat{\rho}}$-modules as in Section~\ref{sec:fpvosa} using the Schur-Weyl-type duality \cite{MT04} but without the premise of holomorphicity:
\begin{prop}\label{prop:nonholirrmodparam}
Let $N$ and $\hat\rho$ be as in Assumption~\ref{ass:b} with \ref{enum:assb1} and \ref{enum:assb2}. Then the isomorphism classes of irreducible $V_N^{\hat\rho}$-modules are parametrised by
\begin{equation*}
V_N^{\hat\rho}(\alpha+N,i,j)\quad\text{for}\quad(\alpha+N,i,j)\in N'/N\times\Z_m\times\Z_m.
\end{equation*}
\end{prop}
The definition of the $V_N^{\hat\rho}(\alpha+N,i,j)$ is given in the proof of the proposition, which follows after a lemma.

Recall that since $G=\langle\hat{\rho}\rangle$ is cyclic, $G$ acts on the set $\mathcal{M}(\hat{\rho}^i)$ of irreducible $\hat{\rho}^i$-twisted $V_N$-modules from the right via
\begin{equation*}
Y_{W\cdot h}(v,x):=Y_W(hv,x),
\end{equation*}
where $W\cdot h=W$ as vector space, for $h\in G$ and $W\in\mathcal{M}(\hat\rho^i)$ (see Proposition~\ref{prop:autaction}). We claim that in the lattice case, this action is given naturally by multiplication with $h^{-1}$ from the left on $(N'/N)^{\rho^i}\cong \mathcal{M}(\hat{\rho}^i)$.
\begin{lem}\label{lem:autactionlat}
Let $N$ and $\hat\rho$ be as in Assumption~\ref{ass:b} with \ref{enum:assb1}. For $i\in\Z_m$ let $V_{\alpha+N}(\hat\rho^i)$ denote the up to isomorphism unique irreducible $\hat{\rho}^i$-twisted $V_N$-module indexed by $\alpha+N\in(N'/N)^{\rho^i}$. Then
\begin{equation*}
V_{\alpha+N}(\hat\rho^i)\cdot \hat{\rho}^j\cong V_{\rho^{-j}\alpha+N}(\hat\rho^i)
\end{equation*}
for all $j\in\Z_m$, noting that $\rho^{-j}\alpha+N\in(N'/N)^{\rho^i}$.
\end{lem}
\begin{proof}
We first consider the untwisted case. Let $W$ be some irreducible $V_N$-module. Then $W$ is isomorphic to $V_{\alpha+N}$ for some $\alpha+N\in N'/N$. The isomorphism class of $W$ can be identified by the action of $(h(-1)1\otimes\ee_0)_0=h(0)$, the zeroth mode of the module vertex operator $Y_W(h(-1)1\otimes\ee_0,x)$, for $h\in\h$ on the \emph{vacuum space} of $W$ (see \cite{Don93}, Section~3). A vector $h\in\h$ acts via $Y_W(\cdot,x)$ on the vacuum space of $W$ as multiplication by $\langle h,\beta\rangle$ for some $\beta\in N'$ and in this case $W$ is isomorphic to $V_{\beta+N}$. Then $h\in\h$ acts via $Y_{W\cdot\hat{\rho}}(\cdot,x)=Y_W(\hat{\rho}\,\cdot,x)$ on the vacuum space of $W$ as multiplication by $\langle\rho h,\beta\rangle=\langle h,\rho^{-1}\beta\rangle$, i.e.\ $W\cdot\hat{\rho}=(W,Y_W\cdot\hat{\rho})$ is isomorphic to $V_{\rho^{-1}\beta+N}$. This proves the assertion in the untwisted case. The same argument can be applied to the twisted case, using the explicit description of the twisted modules in \cite{BK04}.
\end{proof}

\begin{proof}[Proof of Proposition~\ref{prop:nonholirrmodparam}]
We assumed that $\rho$ acts trivially on $N'/N$ (see item~\ref{enum:assb2} of Assumption~\ref{ass:b}) and hence the action of $G$ on $\mathcal{M}(\hat{\rho}^i)$ is trivial for all $i\in\Z_m$. Then there is a projective representation $\phi_W$ of $G$ on $W=V_{\alpha+N}(\hat\rho^i)$ as a vector space such that
\begin{equation*}
\phi_W(\hat{\rho}^j)Y_W(v,x)\phi_W^{-1}(\hat{\rho}^j)=Y_W(\hat{\rho}^jv,x)
\end{equation*}
and since $G$ is cyclic, we may assume that $\phi_W$ is a proper representation, unique up to multiplication of $\phi_W(\hat{\rho})$ by an $m$-th root of unity. The decomposition of $W$ into eigenspaces of $\phi_W(\hat{\rho})$ yields irreducible $V_N^{\hat\rho}$-modules. Let $V_N^{\hat\rho}(\alpha+N,i,j)$ denote the eigenspace of $\phi_W(\hat{\rho})$ in $W=V_{\alpha+N}(\hat\rho^i)$ corresponding to the eigenvalue $\xi_m^j$.
\end{proof}

\minisec{Conformal Weights}

The explicit description of the twisted modules in \cite{DL96,BK04} reveals:
\begin{lem}
Let $N$ and $\hat\rho$ be as in Assumption~\ref{ass:b} with \ref{enum:assb1} to \ref{enum:assb3}. Then for $i\in\Z_m$ the conformal weight of the irreducible $\hat{\rho}^i$-twisted $V_N$-module $V_{\alpha+N}(\hat\rho^i)$ indexed by $\alpha+N\in N'/N$ is given by
\begin{equation*}
\rho_{\rho^i}+\frac{1}{2}\langle\alpha,\alpha\rangle\pmod{(m,i)/m}
\end{equation*}
where $\rho_{\rho^i}\in\Q$ is the vacuum anomaly of the $\hat{\rho}^i$-twisted $V_N$-modules (see \eqref{eq:vacanomaly}).
\end{lem}
\begin{proof}
The weights of the homogeneous elements in $V_{\alpha+N}(\hat\rho^i)$ are in
\begin{equation*}
\rho_{\rho^i}+\frac{1}{2}\langle\pi_{\rho^i}(\alpha+\lambda),\pi_{\rho^i}(\alpha+\lambda)\rangle+\frac{(m,i)}{m}\Ns
\end{equation*}
where $\lambda$ runs through $N$ and $\pi_{\rho^i}=((m,i)/m)\sum_{j=0}^{m/(m,i)-1}\rho^{ij}$ is the projection of $\h=N\otimes_\Z\C$ onto $N^{\rho^i}\otimes_\Z\C$. The second term is
\begin{align*}
&\frac{1}{2}\langle\pi_{\rho^i}(\alpha+\lambda),\pi_{\rho^i}(\alpha+\lambda)\rangle\\
&=\frac{1}{2}\langle\pi_{\rho^i}(\alpha),\pi_{\rho^i}(\alpha)\rangle+\langle\pi_{\rho^i}(\alpha),\pi_{\rho^i}(\lambda)\rangle+\frac{1}{2}\langle\pi_{\rho^i}(\lambda),\pi_{\rho^i}(\lambda)\rangle.
\end{align*}
In this decomposition the second term is in $((m,i)/m)\Z$ since $N'$ and $N$ pair integrally and so is the third term since $N$ is even and because $\langle\rho^{m/2}\lambda,\lambda\rangle\in 2\Z$ (part of item~\ref{enum:assb1} in Assumption~\ref{ass:b}). Then modulo~$(m,i)/m$ the conformal weight of $V_{\alpha+N}(\hat\rho^i)$ is given by $\rho_{\rho^i}$ plus
\begin{equation*}
\frac{1}{2}\langle\pi_{\rho^i}(\alpha),\pi_{\rho^i}(\alpha)\rangle\stackrel{!}{=}\frac{1}{2}\langle\alpha,\alpha\rangle.
\end{equation*}
The last equality is true modulo~$(m,i)/m$ and once we have it established, the proof of the statement is complete. Consider
\begin{align*}
&\frac{1}{2}\langle\pi_{\rho^i}(\alpha),\pi_{\rho^i}(\alpha)\rangle-\frac{1}{2}\langle\alpha,\alpha\rangle\\
&=\frac{1}{2}\frac{(m,i)}{m}\sum_{j=0}^{m/(m,i)-1}\langle(\rho^{ij}-\id)\alpha,\alpha\rangle=\frac{1}{2}\frac{(m,i)}{m}\sum_{j=1}^{m/(m,i)-1}\langle(\rho^{(m,i)j}-\id)\alpha,\alpha\rangle\\
&=\frac{1}{2}\frac{(m,i)}{m}\left(\left(\langle(\rho^{(m,i)}-\id)\alpha,\alpha\rangle+\langle(\rho^{m-(m,i)}-\id)\alpha,\alpha\rangle\right)\right.\\
&\quad+\left(\langle(\rho^{2(m,i)}-\id)\alpha,\alpha\rangle+\langle(\rho^{m-2(m,i)}-\id)\alpha,\alpha\rangle\right)+\cdots\\
&\quad\left.+\langle(\rho^{m/2}-\id)\alpha,\alpha\rangle\right)
\end{align*}
where the last term only occurs if $m/(m,i)$ is even. In each of the pairs in the sum both terms are the same and in $\Z$ by item~\ref{enum:assb2} of Assumption~\ref{ass:b} and hence each pair is in $2\Z$. Modulo $(m,i)/m$ only the last single term remains, if at all, and
\begin{equation*}
\frac{1}{2}\frac{(m,i)}{m}\langle(\rho^{m/2}-\id)\alpha,\alpha\rangle\in\frac{(m,i)}{m}\Z
\end{equation*}
by item~\ref{enum:assb3} of Assumption~\ref{ass:b}.
\end{proof}
If we further assume that
\begin{equation*}
\rho_{\rho^i}=0\pmod{(m,i)/m}
\end{equation*}
(item~\ref{enum:assb4} of Assumption~\ref{ass:b}), then
\begin{equation*}
\rho(V_{\alpha+N}(\hat\rho^i))\in\frac{1}{2}\langle\alpha,\alpha\rangle+\frac{(m,i)}{m}\Z.
\end{equation*}
This implies:
\begin{prop}\label{prop:nonholconfweights}
Let $N$ and $\hat\rho$ be as in Assumption~\ref{ass:b} with \ref{enum:assb1} to \ref{enum:assb4}. Then we can choose the representations $\phi_{V_{\alpha+N}(\hat{\rho}^i)}$ such that the conformal weight of the $V_N^{\hat{\rho}}$-module $V_N^{\hat{\rho}}(\alpha+N,i,j)$ obeys
\begin{equation*}
\rho(V_N^{\hat{\rho}}(\alpha+N,i,j))\in\frac{\langle\alpha,\alpha\rangle}{2}+\frac{ij}{m}+\Z
\end{equation*}
for all $\alpha+N\in N'/N$, $i,j\in\Z_m$.
\end{prop}
\begin{proof}
Similar to Lemma~\ref{lem:phi2} in the holomorphic case we can show that the representation $\phi_W$ on $W=V_{\alpha+N}(\hat\rho^i)$ can be chosen such that
\begin{equation*}
\phi_W(\hat\rho)^{(i,m)}=\e^{(2\pi\i)[i/(i,m)]^{-1}(L_0^W-\langle\alpha,\alpha\rangle/2)}
\end{equation*}
where $[i/(i,m)]^{-1}$ is the inverse of $i/(i,m)$ modulo~$m/(i,m)$. Note that the eigenvalues of $L^W_0-\langle\alpha,\alpha\rangle/2$ lie in $((m,i)/m)\Z$. Following the proof of Proposition~\ref{prop:weight} we see that this implies the statement.
\end{proof}
Moreover, we can and will assume in the following that the representation for $\alpha+N=0+N$ and $i=0$ is naturally given by $\phi_{V_N}(\hat\rho^j)=\hat\rho^j$ for $j\in\Z_m$ (cf.\ Remark~\ref{rem:phi0}).

\minisec{Twisted Modular Invariance}
For the proof of the following results we make use of Dong, Li and Mason's twisted modular invariance (see Theorem~1.3 in \cite{DLM00}). In contrast to Theorem~\ref{thm:1.4} we do not assume holomorphicity. Let $V_N$ and $\rho$ be as in Assumption~\ref{ass:b} with \ref{enum:assb1} and \ref{enum:assb2}. For $\alpha+N\in N'/N$ and $i,j\in\Z_m$ we consider the trace functions
\begin{equation*}
T_{\alpha+N}(v,i,j,\tau):=\tr_{V_{\alpha+N}(\hat{\rho}^i)}o(v)\phi_{V_{\alpha+N}(\hat{\rho}^i)}(\hat{\rho}^j)q_\tau^{L_0-c/24}
\end{equation*}
where $c=\rk(N)$ is the central charge of $V_N$. The twisted modular invariance result implies that the trace functions $T_{\alpha+N}(v,i,j,\tau)$ converge to holomorphic functions on $\H$ and
\begin{equation}\label{eq:dlmnonhol}
(c\tau+d)^{-k}T_{\alpha+N}(v,i,j,M.\tau)=\sum_{\beta+N\in N'/N}\sigma_{\alpha+N,\beta+N}(i,j,M)T_{\beta+N}(v,(i,j)M,\tau)
\end{equation}
for each $M=\left(\begin{smallmatrix}a&b\\c&d\end{smallmatrix}\right)\in\SLZ$ and for each homogeneous $v\in (V_N)_{[k]}$ with respect to Zhu's second grading, $k\in\N$, with constants $\sigma_{\alpha+N,\beta+N}(i,j,M)\in\C$. The sum runs over all $\hat\rho^{ai+cj}$-twisted, $\hat\rho^{bi+dj}$-stable $V_N$-modules. The latter property is always fulfilled since $\rho$ acts trivially on $N'/N$ and hence we sum over all $\hat\rho^{ai+cj}$-twisted $V_N$-modules and these are indexed by $N'/N$. In order to apply the twisted modular invariance result we use that $V_N$ is $g$-rational for all $g\in G=\langle\hat{\rho}\rangle$ (see Lemma~4.2 in \cite{ADJR14}).

It follows directly from the definition of the irreducible $V_N^{\hat{\rho}}$-modules that
\begin{equation*}
T_{\alpha+N}(v,i,j,\tau)=\sum_{k\in\Z_m}\xi^{jk}T_{V_N^{\hat{\rho}}(\alpha+N,i,k)}(v,\tau)
\end{equation*}
and reversely
\begin{equation*}
T_{V_N^{\hat{\rho}}(\alpha+N,i,j)}(v,\tau)=\frac{1}{n}\sum_{k\in\Z_m}\xi^{-jk}T_{\alpha+N}(v,i,k,\tau)
\end{equation*}
for all $\alpha+N\in N'/N$ and $i,j\in\Z_m$. Zhu's modular invariance for $V_N^{\hat{\rho}}$ reads
\begin{equation*}
(c\tau+d)^{-k}T_{V_N^{\hat{\rho}}(\alpha+N,i,j)}(v,M.\tau)=\sum_{\substack{\beta+N\in N'/N,\\l,k\in\Z_m}}\rho_{V_N^{\hat{\rho}}}(M)_{(\alpha+N,i,j),(\beta+N,l,k)}T_{V_N^{\hat{\rho}}(\beta+N,l,k)}(v,\tau).
\end{equation*}
This allows us to relate the constants $\sigma_{\alpha+N,\beta+N}(i,j,M)$ to Zhu's representation $\rho_{V_N^{\hat{\rho}}}(M)$. In particular, we study the $S$-transformation. We set
\begin{equation*}
\lambda_{\alpha+N,\beta+N,i,j}:=\sigma_{\alpha+N,\beta+N}(i,j,S)
\end{equation*}
and analogously to Lemma~\ref{lem:step1} relate these constants to the $S$-matrix of $V_N^{\hat{\rho}}$. We obtain
\begin{equation}\label{eq:lemstep1}
\S_{(\alpha+N,i,j),(\beta+N,l,k)}=\frac{1}{m}\xi_m^{-(lj+ik)}\lambda_{\alpha+N,\beta+N,i,l}
\end{equation}
for all $\alpha+N,\beta+N\in N'/N$ and $i,j,k,l\in\Z_m$.

\minisec{Simple Currents}
In the following we show that $V_N^{\hat{\rho}}$ has group-like fusion.
\begin{prop}
Let $V_N$ and $\rho$ be as in Assumption~\ref{ass:b} with \ref{enum:assb1} and \ref{enum:assb2}. Then all irreducible $V_N^{\hat{\rho}}$-modules are simple currents, i.e.\ $V_N^{\hat{\rho}}$ satisfies Assumption~\ref{ass:sn}.
\end{prop}
\begin{proof}
The symmetry of the $S$-matrix implies that
\begin{equation*}
\lambda_{\alpha+N,\beta+N,i,l}=\lambda_{\beta+N,\alpha+N,l,i}
\end{equation*}
for all $\alpha+N,\beta+N\in N'/N$ and $i,l\in\Z_m$. Since $V_N^{\hat{\rho}}=V_N^{\hat{\rho}}(0+N,0,0)$ is self-contragredient,
\begin{align*}
1=(\S^2)_{(0+N,0,0),(0+N,0,0)}&=\sum_{\alpha+N\in N'/N}\sum_{i,j\in\Z_m}\S_{(0+N,0,0),(\alpha+N,i,j)}\S_{(\alpha+N,i,j),(0+N,0,0)}\\
&=\sum_{\alpha+N\in N'/N}\sum_{i,j\in\Z_m}\frac{1}{m}\lambda_{\alpha+N,0+N,i,0}\frac{1}{m}\lambda_{0+N,\alpha+N,0,i}\\
&=\frac{1}{m}\sum_{\alpha+N\in N'/N}\sum_{i\in\Z_m}\lambda_{\alpha+N,0+N,i,0}^2.
\end{align*}
Theorem~4.4 in \cite{DRX15} states that $\lambda_{\alpha+N,0+N,i,0}\in\R_{>0}$ for all $\alpha+N\in N'/N$, $i\in\Z_m$.\footnote{In the notation of \cite{DRX15} $\lambda_{\alpha+N,0+N,i,0}$ equals $S_{W,V}$ for $V=V_N$ and $W=V_N(\hat{\rho}^i)$ and similarly $\lambda_{0+N,\alpha+N,0,i}$ is $S_{V,W}$.} Lemma~4.14 in \cite{DJX13} and Proposition~\ref{prop:lem4.2}, both results on quantum dimensions, show that $\S_{(\alpha+N,i,j),(0+N,0,0)}/\S_{(0+N,0,0),(0+N,0,0)}\geq 1$ and hence
\begin{equation*}
\frac{\lambda_{\alpha+N,0+N,i,0}}{\lambda_{0+N,0+N,0,0}}\geq 1
\end{equation*}
for all $\alpha+N\in N'/N$ and $i,j\in\Z_m$. Then
\begin{align*}
1&=\frac{1}{m}\sum_{\alpha+N\in N'/N}\sum_{i\in\Z_m}\lambda_{\alpha+N,0+N,i,0}^2=\frac{1}{m}\lambda_{0+N,0+N,0,0}^2\sum_{\alpha+N\in N'/N}\sum_{i\in\Z_m}\frac{\lambda_{\alpha+N,0+N,i,0}^2}{\lambda_{0+N,0+N,0,0}^2}\\
&\geq|N'/N|\lambda_{0+N,0+N,0,0}^2.
\end{align*}
If we show that $\lambda_{0+N,0+N,0,0}^2=1/|N'/N|$, then the above inequality is an equality, which can only hold if $\lambda_{\alpha+N,0+N,i,0}^2=1/|N'/N|$ for all $\alpha+N\in N'/N$ and $i\in\Z_m$ since $\lambda_{\alpha+N,0+N,i,0}>0$. This implies that 
\begin{equation*}
\S_{(\alpha+N,i,j),(0+N,0,0)}=\S_{(0+N,0,0),(0+N,0,0)}=\frac{1}{m\sqrt{|N'/N|}},
\end{equation*}
which by Proposition~\ref{prop:4.17} means in particular that all irreducible $V_N^{\hat{\rho}}$-modules are simple currents.

Consider the twisted modular invariance \eqref{eq:dlmnonhol} for $i=j=0$ and $M=S$, which reads
\begin{equation*}
(1/\tau)^kT_{V_{\alpha+N}}(v,S.\tau)=\sum_{\beta+N\in N'/N}\lambda_{\alpha+N,\beta+N,0,0}T_{V_{\beta+N}}(v,\tau).
\end{equation*}
This is simply an instance of Zhu's modular invariance for the \voa{} $V_N$ and we conclude that
\begin{equation*}
\lambda_{\alpha+N,\beta+N,0,0}=\S_{\alpha+N,\beta+N}=\frac{1}{\sqrt{|N'/N|}}\e^{-(2\pi\i)\langle\alpha,\beta\rangle}
\end{equation*}
since $V_N$ has group-like fusion. Here, $\S_{\alpha+N,\beta+N}$ denotes the entries of the $S$-matrix for $V_N$. In particular, $\lambda_{0+N,0+N,0,0}=1/\sqrt{|N'/N|}$, which completes the proof.
\end{proof}
So, under the stated assumptions, especially item~\ref{enum:assb2} of Assumption~\ref{ass:b}, $V_N^{\hat{\rho}}$ inherits the simple-current property from $V_N$.

\minisec{Fusion Rules}

Finally, knowing that all irreducible $V_N^{\hat{\rho}}$-modules are simple currents, we determine the fusion group of $V_N^{\hat{\rho}}$. This will depend on the conformal weights and hence on items~\ref{enum:assb3} and \ref{enum:assb4} of Assumption~\ref{ass:b}. The following is probably true:
\begin{customconj}{1}\label{conj:1}
Let $V_N$ and $\rho$ be as in Assumption~\ref{ass:b} with \ref{enum:assb1} to \ref{enum:assb4}. Then the fusion group $F_{V_N^{\hat{\rho}}}$ of $V_N^{\hat{\rho}}$ is $N'/N\times\Z_m\times\Z_m$, i.e.\ the representations $\phi_{V_{\alpha+N}(\hat{\rho}^i)}$ can be chosen such that
\begin{equation*}
V_N^{\hat{\rho}}(\alpha+N,i,j)\boxtimes V_N^{\hat{\rho}}(\beta+N,l,k)\cong V_N^{\hat{\rho}}(\alpha+\beta+N,i+l,j+k)
\end{equation*}
for all $\alpha+N,\beta+N\in N'/N$ and $i,j,k,l\in\Z_m$, not changing the conformal weights from Proposition~\ref{prop:nonholconfweights} and compatible with the natural choice of the representations $\phi_{V_{\alpha+N}}$ from the proof of Proposition~\ref{prop:nonholchar} below.
\end{customconj}

It is easy to prove the following weaker version of the conjecture. More precisely, we determine the fusion group as \fqs{} up to isomorphism. We have to additionally demand that neither $m$ nor the level of $N$ be divisible by 4 but this is probably not necessary (see remark after Proposition~\ref{prop:valuesfqs}).
\begin{thm}\label{thm:nonholweak}
Let $V_N$ and $\rho$ be as in Assumption~\ref{ass:b} with \ref{enum:assb1} to \ref{enum:assb4}. Suppose in addition that 4 does not divide $m$ nor the level of $N$. Then the fusion group of $V_N^{\hat{\rho}}$ is isomorphic as \fqs{} to the group $N'/N\times\Z_m\times\Z_m$ with the quadratic form $(\alpha+N,i,j)\mapsto\langle\alpha,\alpha\rangle/2+ij/m+\Z$. In other words: there is a group isomorphism $\kappa\colon N'/N\times\Z_m\times\Z_m\to F_{V_N^{\hat{\rho}}}$ such that
\begin{equation*}
V_N^{\hat{\rho}}(\kappa(\alpha+N,i,j))\boxtimes V_N^{\hat{\rho}}(\kappa(\beta+N,l,k))\cong V_N^{\hat{\rho}}(\kappa(\alpha+\beta+N,i+l,j+k))
\end{equation*}
and
\begin{equation*}
\rho(V_N^{\hat{\rho}}(\kappa(\alpha+N,i,j)))+\Z=\frac{\langle\alpha,\alpha\rangle}{2}+\frac{ij}{m}+\Z
\end{equation*}
for all $\alpha+N,\beta+N\in N'/N$ and $i,j,k,l\in\Z_m$.
\end{thm}
Of course, this result does not depend on the choice of the representations $\phi_{V_{\alpha+N}(\hat{\rho}^i)}$; only the isomorphism $\kappa$ does.
\begin{proof}
As a set, the fusion group $F_{V_N^{\hat{\rho}}}$ is given by $N'/N\times\Z_m\times\Z_m$. If the representations $\phi_{V_{\alpha+N}(\hat{\rho}^i)}$ are chosen as in Proposition~\ref{prop:nonholconfweights}, then the conformal weights modulo~1 are given by $Q_\rho(\alpha+N,i,j)=\langle\alpha,\alpha\rangle/2+ij/m+\Z$. Theorem~\ref{thm:fafqs} implies that $Q_\rho$ defines a non-degenerate quadratic form with respect to the group structure on $F_{V_N^{\hat{\rho}}}$. A possible group structure making $Q_\rho$ a quadratic form is simply addition in $N'/N\times\Z_m\times\Z_m$. The divisibility assumption implies that the largest power of 2 in the denominators of the values of $Q_\rho$ is 2. Hence we may apply Proposition~\ref{prop:valuesfqs}, where we proved that the possible structure of $F_{V_N^{\hat{\rho}}}$ as \fqs{} is uniquely determined up to isomorphism by the values of $Q_\rho$. This completes the proof.
\end{proof}

\minisec{Summary}
Due to the level of detail required in the following, the statement of the above theorem does not suffice. We therefore have to assume Conjecture~\ref{conj:1}. Specifically, the computation of the characters of the irreducible $V_N^{\hat{\rho}}$-modules depends on the exact fusion structure. Modulo the above conjecture we obtain:
\begin{thm}\label{thm:nonhol}
Let $V_N$ and $\rho$ be as in Assumption~\ref{ass:b} with \ref{enum:assb1} to \ref{enum:assb4} and assume that Conjecture~\ref{conj:1} holds. Then $V_N^{\hat{\rho}}$ satisfies Assumptions~\ref{ass:sn}\ref{ass:p} and the isomorphism classes of irreducible $V_N^{\hat{\rho}}$-modules can be parametrised by $(\alpha+N,i,j)\in N'/N\times\Z_m\times\Z_m$ such that with a certain choice of the representations $\phi_{V_{\alpha+N}(\hat{\rho}^i)}$ the conformal weights modulo~1 are given by
\begin{equation*}
\rho(V_N^{\hat{\rho}}(\alpha+N,i,j))+\Z=\frac{\langle\alpha,\alpha\rangle}{2}+\frac{ij}{m}+\Z
\end{equation*}
and the fusion algebra of $V_N^{\hat{\rho}}$ is the group algebra of $N'/N\times\Z_m\times\Z_m$, i.e.\
\begin{equation*}
V_N^{\hat{\rho}}(\alpha+N,i,j)\boxtimes V_N^{\hat{\rho}}(\beta+N,l,k)\cong V_N^{\hat{\rho}}(\beta+N,i+l,j+k)
\end{equation*}
for all $\alpha+N,\beta+N\in N'/N$ and $i,j,k,l\in\Z_m$.
\end{thm}
The choice of the representations $\phi_{V_{\alpha+N}(\hat{\rho}^i)}$ is compatible with the natural choice of the representations $\phi_{V_{\alpha+N}}$ from the proof of Proposition~\ref{prop:nonholchar} below if additionally items \ref{enum:assb5} and \ref{enum:assb6} of Assumption~\ref{ass:b} hold.

\minisec{Computation of Characters}
Finally, we describe the computation of the characters of the irreducible $V_N^{\hat{\rho}}$-modules, proceeding analogously to Section~\ref{sec:latorbifold}. We continue to assume that $V_N$ and $\rho$ are as in Assumption~\ref{ass:b} with \ref{enum:assb1} to \ref{enum:assb4}. In equation \eqref{eq:lemstep1} we related the $S$-matrix of $V_N^{\hat{\rho}}$ to the $S$-transformation of the twisted trace functions $T_{\alpha+N}(v,i,j,\tau)$. With Proposition~\ref{prop:sbilform2}, which holds in the case of group-like fusion, we obtain
\begin{equation*}
\sigma_{\alpha+N,\beta+N}(i,l,S)=\xi_n^{ik+jl}\frac{1}{\sqrt{|N'/N|}}\e^{-(2\pi\i)B_\rho((\alpha+N,i,j),(\beta+N,l,k))}
\end{equation*}
for all $\alpha+N,\beta+N\in N'/N$ and $i,j,k,l\in\Z_m$. A similar computation for the $T$-transformation yields
\begin{equation*}
\sigma_{\alpha+N,\beta+N}(i,l,T)=\xi_n^{-ij}\delta_{\alpha+N,\beta+N}\e^{(2\pi\i)(Q_\rho(\alpha+N,i,j)-\rk(N)/24)}
\end{equation*}
for all $\alpha+N,\beta+N\in N'/N$ and $i,j,k,l\in\Z_m$. To proceed we must insert the values of $Q_\rho$ and $B_\rho$. While the result for the $T$-transformation only depends on the values of the quadratic form (computed in Proposition~\ref {prop:nonholconfweights}), for the $S$-transformation we need the bilinear form and hence the full group structure (postulated in Conjecture~\ref{conj:1}). Assuming that Conjecture~\ref{conj:1} holds and that the representations $\phi_{V_{\alpha+N}(\hat{\rho}^i)}$ are chosen as in Theorem~\ref{thm:nonhol} we obtain
\begin{align*}
\sigma_{\alpha+N,\beta+N}(S):=\sigma_{\alpha+N,\beta+N}(i,l,S)&=\frac{1}{\sqrt{|N'/N|}}\e^{-(2\pi\i)\langle\alpha,\beta\rangle},\\
\sigma_{\alpha+N,\beta+N}(T):=\sigma_{\alpha+N,\beta+N}(i,l,T)&=\delta_{\alpha+N,\beta+N}\e^{(2\pi\i)(\langle\alpha,\alpha\rangle/2-\rk(N)/24)}
\end{align*}
for all $\alpha+N,\beta+N\in N'/N$ and $i,l\in\Z_m$. In both cases the result is independent of $i,j\in\Z_m$. In fact, this is just Zhu's representation for $V_N$, i.e.\
\begin{equation*}
\sigma_{\alpha+N,\beta+N}(M)=\rho_{V_N}(M)_{\alpha+N,\beta+N}
\end{equation*}
for $M\in\SLZ$ and $\alpha+N,\beta+N\in N'/N$.

In Proposition~\ref{prop:nonholchar} below we compute $T_{\alpha+N}(\vac,0,j,\tau)$, $\alpha+N\in N'/N$, $j\in\Z_m$, explicitly in terms of theta and eta functions, whose transformation behaviour under $\SLZ$ is known. This allows us to compute all the trace functions $T_{\alpha+N}(\vac,i,j,\tau)$ and $T_{V_N^{\hat{\rho}}(\alpha+N,i,j)}(\vac,\tau)=\ch_{V_N^{\hat{\rho}}(\alpha+N,i,j)}(\tau)$ for $\alpha+N\in N'/N$ and $i,j\in\Z_m$ using the above modular invariance results. Indeed,
\begin{equation*}
T_{\alpha+N}(\vac,i,j,\tau)=\sum_{\beta+N\in N'/N}\sigma_{\alpha+N,\beta+N}(M_{i,j}^{-1})T_{\beta+N}(\vac,0,\gcd(i,j),M_{i,j}.\tau)
\end{equation*}
for all $\alpha+N\in N'/N$, $i,j\in\Z_m$ where as before
\begin{equation*}
M_{i,j}=\begin{pmatrix}*&*\\\frac{i}{\gcd(i,j)}&\frac{j}{\gcd(i,j)}\end{pmatrix}\in\SLZ
\end{equation*}
so that $(0,\gcd(i,j))M_{i,j}=(i,j)$. Recall that $M_{i,j}$ and $\gcd(i,j)$ depend not only on $i,j\in\Z_m$ but on the choice of representatives modulo $m$. Then
\begin{equation*}
\ch_{V_N^{\hat{\rho}}(\alpha+N,i,j)}(\tau)=\frac{1}{n}\sum_{k\in\Z_m}\xi^{-jk}\sum_{\beta+N\in N'/N}\sigma_{\alpha+N,\beta+N}(M_{i,k}^{-1})T_{\beta+N}(\vac,0,\gcd(i,k),M_{i,k}.\tau).
\end{equation*}

As final ingredient we need to compute the trace functions $T_{\alpha+N}(\vac,0,j,\tau)$, $\alpha+N\in N'/N$, $j\in\Z_m$. The result is as nice as it could be but depends on special properties in Assumption~\ref{ass:b}.
\begin{prop}\label{prop:nonholchar}
Let $V_N$ and $\rho$ be as in Assumption~\ref{ass:b} with items \ref{enum:assb1}, \ref{enum:assb2}, \ref{enum:assb3}, \ref{enum:assb5} and \ref{enum:assb6}. Then there is a natural choice of the representations $\phi_{V_{\alpha+N}}$, $\alpha+N\in N'/N$, on the irreducible untwisted $V_N$-modules such that
\begin{equation*}
T_{\alpha+N}(\vac,0,j,\tau)=\tr_{V_{\alpha+N}}\phi_{V_{\alpha+N}}(\hat\rho)^j q_\tau^{L_0-c/24}=\frac{\vartheta_{(\alpha+N)^{\rho^j}}(\tau)}{\eta_{\rho^j}(\tau)}
\end{equation*}
for all $\alpha+N\in N'/N$ and $j\in\Z_m$ where $(\alpha+N)^{\rho^j}$ is the set of vectors in the lattice coset $\alpha+N$ pointwise invariant under $\rho^j$.
\end{prop}
Any other choice of the representations $\phi_{V_{\alpha+N}}$ will of course only modify these characters by a scalar factor.
\begin{proof}
In order to compute these characters, we have to find an explicit description of the representations $\phi_{V_{\alpha+N}}$, $\alpha+N\in N'/N$, which are unique up to a scalar. For $\alpha=0$, we already chose $\phi_{V_N}(\hat\rho)=\hat\rho$, which is simply the lift of the lattice automorphism to an automorphism $\hat\rho$ of the \voa{}
\begin{equation*}
V_N=M_{\hat\h}(1)\otimes\C_\eps[N]
\end{equation*}
as described in Section~\ref{sec:autlatvoa}. Clearly, $\rho$ extends to an automorphism of the dual lattice $N'$ and in the following we will lift $\rho$ to an automorphism $\check\rho$ of
\begin{equation*}
A_{N'}=\bigoplus_{\alpha+N\in N'/N}V_{\alpha+N}=M_{\hat\h}(1)\otimes\C_\eps[N']
\end{equation*}
where $\eps\colon N'\times N'\to\C^\times$ is a 2-cocycle with $\eps(\alpha,\beta)/\eps(\beta,\alpha)=(-1)^{\langle\alpha,\beta\rangle}$ for $\alpha,\beta\in N$ whose values lie in a finite, cyclic subgroup of $\C^\times$ (see Section~\ref{sec:latvoa}). Because of item~\ref{enum:assb2} in Assumption~\ref{ass:b} this automorphism will restrict to an automorphism of each $V_{\alpha+N}$, $\alpha+N\in N'/N$.
We consider the skew of $\eps$, i.e.\ the quotient $c(\alpha,\beta):=\eps(\alpha,\beta)/\eps(\beta,\alpha)$, which is an alternating $\Z$-bilinear form on $N'$. Its restriction to $N$ is the alternating $\Z$-bilinear form $(-1)^{\langle\alpha,\beta\rangle}$. The latter is also $\rho$-invariant, which is essential for lifting the automorphism $\rho$ to an automorphism $\hat\rho$ of the twisted group algebra $\C_\eps[N]$. Now assume that it is possible to choose the 2-cocycle $\eps$ such that $c$ is $\rho$-invariant on all of $N'$. Then, by the same argument, we can lift $\rho$ to an automorphism of the twisted group algebra $\C_\eps[N']$. Indeed, $\eps(\alpha,\beta)$ and $\eps(\rho\alpha,\rho\beta)$ both have the skew $c$ and hence are cohomologous, i.e.\ there is a function $u\colon N'\to\C^\times$ such that
\begin{equation}\label{eq:uduallat}
\frac{\eps(\alpha,\beta)}{\eps(\rho\alpha,\rho\beta)}=\frac{u(\alpha)u(\beta)}{u(\alpha+\beta)}
\end{equation}
for all $\alpha,\beta\in N'$. Then $\rho$ lifts to an automorphism $\check\rho$ of $\C_\eps[N']$ via
\begin{equation*}
\check\rho(\ee_\alpha)=u(\alpha)\ee_{\rho\alpha},
\end{equation*}
which obeys $\check\rho(\ee_\alpha\ee_\beta)=\check\rho(\ee_\alpha)\check\rho(\ee_\beta)$ for all $\alpha,\beta\in N'$.

Let us study the $\rho$-invariance of $c$. Since any alternating $\Z$-bilinear form on a \emph{finitely generated} abelian group is the skew of a suitable 2-cocycle, it suffices to find an alternating $\Z$-bilinear form on $N'$ that is $\rho$-invariant and restricts to $(-1)^{\langle\alpha,\beta\rangle}$ on $N$. The appropriate 2-cocycle $\eps$ can be chosen subsequently. An explicit construction of an alternating $\Z$-bilinear form $c$ on $N'$ descending as desired to $N$ is given in \cite{LL04}, Remark~6.4.12. If $m=\ord(\rho)$ is odd, we may always assume that $c$ is $\rho$-invariant since we can consider
\begin{equation*}
\prod_{r=0}^{m-1}c(\rho^r\alpha,\rho^r\beta),
\end{equation*}
which is clearly alternating, $\Z$-bilinear, $\rho$-invariant and restricts to
\begin{equation*}
\prod_{r=0}^{m-1}(-1)^{\langle\rho^r\alpha,\rho^r\beta\rangle}=(-1)^{m\langle\alpha,\beta\rangle}=(-1)^{\langle\alpha,\beta\rangle}
\end{equation*}
for $\alpha,\beta\in N$. Now suppose that $m$ is even. A short calculation shows that the explicit construction of $c$ from \cite{LL04} is $\rho$-invariant if items \ref{enum:assb2} and \ref{enum:assb5} of Assumption~\ref{ass:b} hold.

We let $\check\rho$ act on $A_{N'}=M_{\hat\h}(1)\otimes\C_\eps[N']$ as $\check\rho:=\rho\otimes\check\rho$ (like $\hat\rho$ in Section~\ref{sec:autlatvoa}). Then, because of item~\ref{enum:assb2} of Assumption~\ref{ass:b}, this restricts to an action on each irreducible module $V_{\alpha+N}=M_{\hat\h}(1)\otimes\C_\eps[\alpha+N]$, $\alpha+N\in N'/N$. We define the vector-space automorphisms
\begin{equation*}
\phi_{V_{\alpha+N}}(\hat\rho):={\check\rho}|_{V_{\alpha+N}}
\end{equation*}
for $\alpha+N\in N'/N$. One easily checks that 
\begin{equation*}
\rho h(x) \rho^{-1}=(\rho h)(x)
\end{equation*}
for $h\in\h$ and \eqref{eq:uduallat} ensures that
\begin{equation*}
\check\rho Y_\alpha(x)\check\rho^{-1}=u(\alpha)Y_{\rho\alpha}(x)
\end{equation*}
for $\alpha\in N'$ (recall the definitions of $h(x)$ and $Y_\alpha(x)$ from Section~\ref{sec:latvoa}) so that these automorphisms have the desired property that
\begin{equation}\label{eq:desprop}
\phi_{V_{\alpha+N}}(\hat\rho)Y_{V_{\alpha+N}}(v,x)\phi_{V_{\alpha+N}}(\hat\rho)^{-1}=Y_{V_{\alpha+N}}(\hat\rho v,x)
\end{equation}
for all $v\in V_N$.

It is clear that the restriction $\phi_{V_N}(\hat\rho)=\check{\rho}|_{V_N}$ of $\check\rho$ to $V_N$ is exactly the same as the lift of the automorphism $\rho\in\Aut(N)$ to an automorphism $\hat\rho$ of the \voa{} $V_N$, described in Section~\ref{sec:autlatvoa}, with the exception that $u$, like $\eps$, may take values in some finite, cyclic subgroup of $\C^\times$ and not just $\{\pm 1\}$, but this is not essential. We may and will assume in the following that the automorphism $\hat\rho$ on $V_N$ is simply the restriction of the automorphism $\check\rho$ on $A_{N'}$ to $V_N$.

In analogy to the standard-lift property (see Definition~\ref{defi:standard}), it is possible to choose the function $u\colon N'\to\C^\times$ such that $u(\alpha)=1$, i.e.\ ${\check\rho}\ee_\alpha=\ee_\alpha$, for all $\alpha\in (N')^\rho$. For the trace functions to have the simple form stated in the theorem we demand that this property hold for all powers of $\rho$, i.e.\ that for $k\in\N$,
\begin{equation}\label{eq:standardliftpowerdual1}
{\check\rho}^k\ee_\alpha=\ee_\alpha\quad\text{for all}\quad\alpha\in(N')^{\rho^k}.
\end{equation}
Similarly to Proposition~\ref{prop:standardliftpower}, one obtains that this is the case if and only if $m=\ord(\rho)$ is odd or $m$ is even and
\begin{equation}\label{eq:standardliftpowerdual2}
c(\alpha,\rho^{k/2}\alpha)=1
\end{equation}
for all $\alpha\in(N')^{\rho^k}$ if $k$ is even. By restriction, this condition includes as a special case item~\ref{enum:assb1} of Assumption~\ref{ass:b}. Using the explicit construction of $c$ from \cite{LL04}, one can show that $c(\alpha,\rho^{k/2}\alpha)=(-1)^{\langle\alpha,(\rho^{k/2}-\id)\alpha\rangle}$ for all $\alpha\in N'$ if items \ref{enum:assb2} and \ref{enum:assb5} of Assumption~\ref{ass:b} are satisfied. Then, \eqref{eq:standardliftpowerdual2} holds if and only if item~\ref{enum:assb6} of Assumption~\ref{ass:b} does.

The standard lift-property \eqref{eq:standardliftpowerdual1} for $k=m$ implies that the automorphism $\phi_{V_{\alpha+N}}(\hat\rho)=\check\rho|_{V_{\alpha+N}}$ of $V_{\alpha+N}$ has order $m$ so that $\phi_{V_{\alpha+N}}$ is a representation of $\langle\check{\rho}\rangle\leq\Aut(V_N)$ on $V_{\alpha+N}$ as vector space for all $\alpha+N\in N'/N$. Together with \eqref{eq:desprop} this implies that we have indeed found explicit descriptions of the representations $\phi_{V_{\alpha+N}}$ for $\alpha+N\in N'/N$, which are unique up to a scalar.

Finally, we compute the twisted trace functions of the irreducible $V_N$-modules. We obtain
\begin{align*}
T_{\alpha+N}(\vac,0,j,\tau)&=\tr_{V_{\alpha+N}}\phi_{V_{\alpha+N}}(\hat{\rho}^j)q_\tau^{L_0-c/24}\\
&=\left(\tr_{M_{\hat\h}(1)}\rho^jq_\tau^{L_0-c/24}\right)\left(\tr_{\C_\eps[\alpha+N]}\check{\rho}^jq_\tau^{L_0}\right)\\
&=\frac{1}{\eta_{\rho^j}(\tau)}\sum_{\beta\in(\alpha+N)^{\rho^j}}u(\beta)u(\rho\beta)\ldots u(\rho^{j-1}\beta)q_\tau^{\langle\beta,\beta\rangle/2}\\
&=\frac{1}{\eta_{\rho^j}(\tau)}\sum_{\beta\in(\alpha+N)^{\rho^j}}q_\tau^{\langle\beta,\beta\rangle/2}\\
&=\frac{\theta_{(\alpha+N)^{\rho^j}}(\tau)}{\eta_{\rho^j}(\tau)}
\end{align*}
for all $\alpha+N\in N'/N$ and $j\in\Z_m$, where we used \eqref{eq:standardliftpowerdual1} in terms of the function $u$ in the second-to-last step, noting that $(\alpha+N)^{\rho^j}\leq (N')^{\rho^j}$. This completes the proof.
\end{proof}

\section{Natural Construction of Ten \BKMA{}s}\label{sec:10bkmas}

In the following we describe the BRST construction of ten \BKMa{}s from \cite{Sch04b,Sch06} whose denominator identities are completely reflective automorphic products of singular weight. To find such ``natural constructions'' was proposed by Borcherds \cite{Bor92}. These constructions are called natural since in physics they describe bosonic strings moving on suitable spacetimes.

\minisec{Introduction}

\BKMa{}s are a class of Lie algebras naturally generalising Kac-Moody algebras, which in turn generalise the finite-dimensional, semisimple Lie algebras. In particular, a \BKMa{} $\g$ admits a character formula for highest-weight modules and a denominator identity
\begin{equation*}
\ee^\rho\prod_{\alpha\in\Phi^+}(1-\ee^\alpha)^{\mult(\alpha)}=\sum_{w\in W}\det(w)w\left(\ee^\rho\sum_{\alpha\in\Phi}\eps(\alpha)\ee^\alpha\right),
\end{equation*}
an identity of formal exponentials, where the second sum is over all roots $\alpha$ in the root system $\Phi$ of $\g$ and the product ranges over the set $\Phi^+$ of positive roots, $W$ denotes the Weyl group, $\rho$ the Weyl vector, $\mult(\alpha)$ the multiplicity of the root $\alpha$ and $\eps(\alpha)$ is $(-1)^n$ if $\alpha$ is the sum of $n$ pairwise orthogonal imaginary simple roots and 0 otherwise.

Sometimes, the denominator identity of a \BKMa{} is an automorphic product. Automorphic products are automorphic forms on orthogonal groups, which are in the image of the \emph{Borcherds lift}. In \cite{Bor98} Borcherds constructed a lift from vector-valued modular forms for the Weil representation of $\MpZ$ to automorphic forms on orthogonal groups, which have an infinite-product expansion and are hence called \emph{automorphic products}.

It is believed that all \BKMa{}s whose denominator identities are automorphic products of singular weight can be realised as bosonic strings. The following construction adds further evidence to this conjecture, which so far has been proved for some special cases in \cite{Bor90,Bor92,HS03,Hoe03a,CKS07,HS14,Car12b} (see examples in Section~\ref{sec:brstex}). The general situation is depicted in the following diagram:
\begin{equation*}
\begin{tikzcd}
\begin{tabular}{c}Vertex\\Algebras\end{tabular}\arrow{rr}{\text{BRST}}\arrow[bend right]{rrrr}{\text{lift of char.}}&&\begin{tabular}{c}Borcherds-Kac-\\Moody algebras\end{tabular}\arrow[<->]{rr}{\text{den. id.}}&&\begin{tabular}{c}Automorphic\\Products\end{tabular}
\end{tikzcd}
\end{equation*}
We explain the bended arrow in the diagram: in the known cases, the automorphic product can be obtained by applying the Borcherds lift to the characters of a certain \vosa{} of the vertex algebra in the matter sector of the BRST construction.

\minisec{Twisting the Fake Monster Algebra}
In the following we describe how to obtain new Borcherds-Kac-Moody (super)algebras by twisting the denominator identity of the Fake Monster Lie algebra (see \cite{Bor92}, Sections 12 and 13), both as Lie (super)algebras over $\R$. Consider the $\II_{25,1}$-graded Fake Monster Lie algebra, which we will call $\g$ in this section, obtained by a BRST construction with $M=V_{\II_{25,1}}$ (as vertex algebra over the real numbers $\R$) in the matter sector (see Section~\ref{sec:brstex}). The denominator identity of $\g$ is
\begin{equation*}
\ee^\rho\prod_{\alpha\in\Phi^+}(1-\ee^\alpha)^{[1/\eta^{24}](-\langle\alpha,\alpha\rangle/2)}=\sum_{w\in W}\det(w)w(\eta^{24}(\ee^\rho)).
\end{equation*}
Here, the Weyl group $W$ of $\g$ is the full reflection group of $\II_{25,1}$ and $\eta^{24}(\ee^\rho)=\ee^\rho\prod_{n=1}^\infty(1-\ee^{n\rho})^{24}$. The root lattice of $\g$ is $\II_{25,1}$ and $\Phi^+$ denotes the set of positive roots. Upon replacing the formal exponentials by complex ones,\footnote{That is to say we replace $\ee^\alpha$ by $\e^{(2\pi\i)\langle\alpha,Z\rangle}$ where the variable $Z$ may take values in a certain subset of $\II_{25,1}\otimes_\Z\C$.} the above is the expansion of a certain automorphic product $\Psi$ on the lattice $\II_{25,1}\oplus \II_{1,1}\cong \II_{26,2}$ of weight 12.

Consider the automorphism group $\Aut(V_\Lambda)$ of the vertex algebra $V_\Lambda$ associated with the Leech lattice $\Lambda$, the unique positive-definite, even, unimodular lattice in dimension 24 with no roots. It acts naturally on the Fake Monster Lie algebra $\g$. Indeed, by considering $M\cong V_\Lambda\otimes V_{\II_{1,1}}$ and forgetting about the grading by $\Lambda$, we can also view $\g$ as a $\II_{1,1}$-graded Lie algebra whose graded components are given by $\g(\alpha)\cong(V_\Lambda)_{1-\langle\alpha,\alpha\rangle/2}$ for non-zero $\alpha\in \II_{1,1}$ (see Remark~\ref{rem:rank2}) and $\g(0)\cong(V_\Lambda)_1\otimes_\R\R^2$ for $0\in \II_{1,1}$ (cf.\ Proposition~\ref{prop:brstcartan1}). Hence, since $\Aut(V_\Lambda)$ acts on the graded subspaces of $V_\Lambda$, it also acts on $\g$.

Given an automorphism $g\in\Aut(V_\Lambda)$, we obtain a \emph{$g$-twisted denominator identity} of $\g$. Now consider an automorphism $\nu$ of the Leech lattice of order $m$, which we lift to an automorphism $\hat\nu$ of $V_\Lambda$. Assume that all powers of $\hat\nu$ are standard lifts of the corresponding power of $\nu$. In particular, $\hat\nu$ has order $m$. Then the twisted denominator identity associated with $\hat\nu$ is computed in \cite{Bor92}, Section~13, and it is shown that this $\hat{\nu}$-twisted denominator identity is exactly the \emph{untwisted} denominator identity of some real Borcherds-Kac-Moody (super)algebra $\g_{\hat\nu}$.\footnote{Whether $\g_{\hat\nu}$ is a Borcherds-Kac-Moody algebra or superalgebra depends on the cycle shape of $\nu$ (see \cite{Bor92} for details).}

Scheithauer has shown that for certain elements $\hat\nu$, namely those with square-free level\footnote{The level of an automorphism $\nu$ with some cycle shape $\prod_{t\mid m}t^{b_t}$ is defined as the level of the subgroup of $\SLZ$ fixing the eta product $\eta_\nu$. This is exactly the smallest positive multiple $N$ of $m=\ord(\nu)$ such that $24\mid N\sum_{t\mid m}b_t/t$.}, the $\hat\nu$-twisted denominator identity of $\g$, which equals the denominator identity of $\g_{\hat\nu}$, is an automorphic form of singular weight $-w:=\rk(\Lambda^\nu)/2=:k/2-1$ in the image of the Borcherds lift, i.e.\ an automorphic product (see \cite{Sch06}, Theorem~10.1). More precisely, starting from the automorphism $\hat\nu$ he constructs, using the Borcherds lift, an automorphic product $\Psi_{\hat\nu}$ whose expansion at a certain cusp gives the denominator identity of $\g_{\hat\nu}$ (see also Remark~\ref{rem:vvmfF}).

A nice special case is obtained for ten particular conjugacy classes of automorphisms of the Leech lattice $\Lambda$. For the Leech lattice $\rk(\Lambda^\nu)=k-2$ is always even for any automorphism $\nu\in\Aut(\Lambda)$. Let $m\in\Ns$ such that $m$ is square-free and
\begin{equation*}
\sigma_1(m)\mid24
\end{equation*}
with the sum-of-divisors function $\sigma_1$. Explicitly, let $m\in\{1,2,3,5,6,7,11,14,15,23\}$. For each such $m$ let $\nu$ be the up to conjugacy unique\footnote{Note that for $m=23$ there are two conjugacy classes in $\Aut(\Lambda)$ with cycle shape $1.23$, which are algebraically conjugate in the sense of \cite{CCNPW85} and behave identically for our purposes. More precisely, if $\nu$ is in one conjugacy class, then $\nu^{-1}$ is in the other.} automorphism with cycle shape
\begin{equation}\label{eq:cycleshapenu}
\prod_{t\mid m}t^{b_t}=\prod_{t\mid m}t^{24/\sigma_1(m)}.
\end{equation}
These automorphisms have order and level $m$. We remark that the fixed-point lattices $\Lambda^\nu$ are the unique lattices in their genus without roots. The rank of the fixed-point lattice $\Lambda^\nu$ is given by $\rk(\Lambda^\nu)=k-2=24\sigma_0(m)/\sigma_1(m)$. The above automorphisms correspond exactly to the elements of square-free order in the sporadic group $M_{23}$, which acts naturally on the Leech lattice $\Lambda$.

It is shown in \cite{Sch04}, Theorem~10.1, that the twisted denominator identity of $\g$ associated with one of these ten automorphisms $\hat\nu$ of square-free order $m$ is
\begin{equation*}
\ee^\rho\prod_{d\mid m}\prod_{\alpha\in\Phi^+\cap d\Delta'}(1-\ee^\alpha)^{[1/\eta_\nu](-\langle\alpha,\alpha\rangle/2d)}=\sum_{w\in W}\det(w)w(\eta_\nu(\ee^\rho))
\end{equation*}
with root lattice $\Delta=\Lambda^\nu\oplus \II_{1,1}$, positive roots $\Phi^+$, Weyl group $W$, which is again the full reflection group of $\Delta$, and Weyl vector $\rho$. Here, $\eta_\nu$ is again the eta product
\begin{equation*}
\eta_\nu(q)=\prod_{t\mid m}\eta(q^t)^{24/\sigma_1(m)}=q\prod_{t\mid m}\prod_{n=1}^\infty(1-q^{tn})^{24/\sigma_1(m)}
\end{equation*}
associated with the cycle shape of $\nu$. This is the denominator identity of the real \BKMa{} $\g_{\hat\nu}$, which is graded by $\Delta=\Lambda^\nu\oplus \II_{1,1}$ with dimensions
\begin{equation*}
\dim_\R(\g_{\hat\nu}(\alpha))=\sum_{d\mid m}\delta_{\alpha\in \Delta\cap d\Delta'}\left[\frac{1}{\eta_\nu}\right]\left(-\frac{1}{d}\frac{\langle\alpha,\alpha\rangle}{2}\right)
\end{equation*}
for all non-zero $\alpha\in\Delta$ and $\dim_\R(\g_{\hat\nu}(0))=k$. This denominator identity is the expansion at any cusp of the automorphic product $\Psi_{\hat\nu}$ on the lattice $P=L\oplus\II_{1,1}=\Lambda^\nu\oplus\II_{1,1}(m)\oplus\II_{1,1}$ of singular weight $-w=12\sigma_0(m)/\sigma_1(m)\in\Z$ (where $L=\Lambda^\nu\oplus\II_{1,1}(m)$, see below). The lattice $P$ is even, of signature $(k,2)$ and has level $m$.

One observes that for these ten cases the automorphic form $\Psi_{\hat\nu}$ is completely reflective (see \cite{Sch06}, Section~9). In fact, one can show that in a sense these are the only completely reflective automorphic products of singular weight arising as denominator identities of \BKMa{}s. More precisely:
\begin{thm}[\cite{Sch06}, Theorem~12.7, \cite{Moe12}, Satz~6.4.2]
Let $P$ be an even lattice of signature $(k,2)$ with $k\geq 4$, square-free level $m$ and $p$-ranks of the discriminant form at most $k+1$. Then a real \BKMa{} whose denominator identity is a completely reflective automorphic product in the image of the Borcherds lift of singular weight $-w=k/2-1$ on $P$ is isomorphic to $\g_{\hat\nu}$ for some $\nu$ of order $m$ in $M_{23}$.
\end{thm}
This is a slight improvement of the theorem in \cite{Sch06} due to the author of this text \cite{Moe12}, removing the assumption that $P$ splits two hyperbolic planes.

What is actually shown is the following classification of automorphic products, to which the above theorem is a corollary:
\begin{thm}[\cite{Sch06}, Theorem~12.6, \cite{Moe12}, Satz~6.4.1]
Let $P$ be an even lattice of signature $(k,2)$ with $k\geq 4$, square-free level $m$ and $p$-ranks of the discriminant form at most $k+1$. Then a completely reflective automorphic product in the image of the Borcherds lift of singular weight $-w=k/2-1$ exists on $P$ if and only if $P$ is isomorphic to one of the following lattices\footnote{We denote these lattices by their genus since the genera in the table consist only of one isomorphism class of lattices.}:
\begin{equation*}
\renewcommand{\arraystretch}{1.3}
\begin{array}{r|l} 
\multicolumn{1}{c|}{-w} & \multicolumn{1}{c}{P}\\\hline
1 & \II_{4,2}(23^{-3})\\
2 & \II_{6,2}(11^{-4}),\,\II_{6,2}(2_{\II}^{+4}7^{-4}),\,\II_{6,2}(3^{+4}5^{-4})\\
3 & \II_{8,2}(7^{-5})\\
4 & \II_{10,2}(5^{+6}),\,\II_{10,2}(2_{\II}^{+6}3^{-6})\\
6 & \II_{14,2}(3^{-8})\\
8 & \II_{18,2}(2_{\II}^{+10})\\
12& \II_{26,2} 
\end{array} 
\end{equation*}
Moreover, all these lattices are of the form $P\cong\Lambda^\nu\oplus\II_{1,1}(m)\oplus\II_{1,1}$ for an element $\nu$ of square-free order $m$ in $M_{23}$.
\end{thm}
The restriction on the $p$-ranks (see definition before Proposition~\ref{prop:thm1.14.2}) is essential since it guarantees the finiteness of the above list (see \cite{Sch06}, remark after Theorem~12.3).

It was asked by Borcherds \cite{Bor92} whether there are natural constructions for the Borcherds-Kac-Moody (super)algebras obtained by twisting denominator identities of the Fake Monster Lie algebra, for instance. In the following we present such a construction, using BRST cohomology, for the ten real \BKMa{}s $\g_{\hat\nu}$ associated with elements of square-free order in $M_{23}$. More precisely, since we are working over $\C$, we construct their complexification.

Some of these ten \BKMa{}s have already been constructed in a BRST approach. Clearly, for $\nu=\id$ we obtain the Fake Monster Lie algebra $\g$ itself. For the automorphism of order $2$ in $M_{23}$ one obtains the Fake Baby Monster Lie algebra. With a slightly less effective method in \cite{CKS07} the authors construct the four \BKMa{}s associated with the automorphisms in $M_{23}$ of order $2$, $3$, $5$ and $7$ depending on some conjectures.

\minisec{Matter Sector}

We consider the ten conjugacy classes of automorphisms $\nu$ of the Leech lattice $\Lambda$ introduced above. To avoid confusion we write $\k:=\Lambda\otimes_\Z\C$ for the complexification of the Leech lattice $\Lambda$ while $\h$ will be the complexification of the Lorentzian lattice $L$ below.
\begin{customass}{{\textbf{\textsf{M}}}}\label{ass:m}
Let $\nu$ be an element of square-free order in $M_{23}$ viewed as automorphism of the Leech lattice $\Lambda$, i.e.\ let $\nu$ be one of the conjugacy classes in $\Aut(\Lambda)$ of order $m=1,2,3,5,6,7,11,14,15,23$ with cycle shape \eqref{eq:cycleshapenu}. Let $\hat\nu$ be a standard lift of $\nu$ to an automorphism of the holomorphic \voa{} $V_\Lambda$ associated with $\Lambda$.
\end{customass}
Recall that for these ten cases this already implies that $\hat{\nu}^k$ is a standard lift of $\nu^k$ for all $k\in\N$. In particular, $\hat{\nu}$ has order $m$, and not $2m$.

\begin{prop}\label{prop:confweightnu}
Let $\hat\nu$ be as in Assumption~\ref{ass:m}. Then the conformal weight $\rho(V_\Lambda(\hat{\nu}))$ of the unique $\hat{\nu}$-twisted $V_\Lambda$-module $V_\Lambda(\hat{\nu})$, i.e.\ the vacuum anomaly $\rho_\nu$, is $(m-1)/m$. In particular, $\hat{\nu}$ has type $m\{0\}$.
\end{prop}
\begin{proof}
By Remark~\ref{rem:cycleconf} the vacuum anomaly $\rho_\nu$ is given by
\begin{align*}
\rho_\nu&=\frac{1}{24}\sum_{t\mid m}b_t\left(t-\frac{1}{t}\right)=\frac{1}{24}\frac{24}{\sigma_1(m)}\sum_{t\mid m}\left(t-\frac{1}{t}\right)=\frac{1}{\sigma_1(m)}\left(\sigma_1(m)-\frac{\sigma_1(m)}{m}\right)\\
&=\frac{m-1}{m},
\end{align*}
which vanishes modulo~$1/m$. Since $\Lambda$ is unimodular, this is also the conformal weight $\rho(V_\Lambda(\hat{\nu}))$ of the unique irreducible $\hat\nu$-twisted $V_\Lambda$-module.
\end{proof}

By Corollary~\ref{cor:main} the \fpvosa{} $V_\Lambda^{\hat{\nu}}=(V_\Lambda)^{\hat{\nu}}$ has exactly $n^2$ irreducible modules $V_\Lambda^{\hat{\nu}}(i,j)$, indexed by $(i,j)\in\Z_m\times\Z_m$, which is also the fusion group, with conformal weights $\rho(V_\Lambda^{\hat{\nu}}(i,j))\in ij/m+\Z=:Q_m((i,j))$. Moreover, we know by Theorem~\ref{thm:aia} that the direct sum of irreducible $V_\Lambda^{\hat{\nu}}$-modules
\begin{equation*}
\bigoplus_{i,j\in\Z_m}V_\Lambda^{\hat{\nu}}(i,j)
\end{equation*}
admits the structure of an \aia{} with associated \fqs{} $(\Z_m\times\Z_m,-Q_m)$ and central charge $c=\rk(\Lambda)=24$.

Recall that $\II_{1,1}$ is the up to isomorphism unique even, unimodular lattice of Lorentzian signature $(1,1)$ and let $K:=\II_{1,1}(m)$ be the same lattice with quadratic form rescaled by $m$. Proposition~\ref{prop:e2rlattice} implies that the discriminant form $K'/K=(K'/K,Q_K)$ is isomorphic as \fqs{} to $(\Z_m\times\Z_m,Q_m)$ and in fact it is also isomorphic to $(\Z_m\times\Z_m,-Q_m)$. We choose some isomorphism
\begin{equation}\label{eq:isophi}
\varphi\colon (K'/K,Q_K)\to(\Z_m\times\Z_m,-Q_m).
\end{equation}

We consider the vertex algebra $V_K$ associated with $K$. It is a weak \voa{} of central charge 2
with irreducible modules $V_{\alpha+K}$ for $\alpha\in K'/K$ (see footnote~\ref{footnote:1}). Moreover, the direct sum
\begin{equation*}
\bigoplus_{\alpha+K\in K'/K}V_{\alpha+K}
\end{equation*}
admits the structure of an \aia{} with associated quadratic space $\overline{K'/K}=(K'/K,-Q_K)$ and central charge $2$ (see footnote \ref{footnote:2}).

We define the weak \voa{} $M$ in the matter sector of the BRST construction:
\begin{prop}\label{prop:mdef}
Let $\hat\nu$ be as in Assumption~\ref{ass:m}. Then the direct sum
\begin{equation*}
M:=\bigoplus_{\alpha+K\in K'/K}V_\Lambda^{\hat{\nu}}(\varphi(\alpha+K))\otimes V_{\alpha+K}
\end{equation*}
admits the structure of a weak \voa{} of central charge 26.
\end{prop}
\begin{proof}
The tensor-product \aia{}
\begin{equation*}
\left(\bigoplus_{i,j\in\Z_m}V_\Lambda^{\hat{\nu}}(i,j)\right)\otimes\left(\bigoplus_{\alpha+K\in K'/K}V_{\alpha+K}\right)=\bigoplus_{\substack{i,j\in\Z_m,\\\alpha+K\in K'/K}}V_\Lambda^{\hat{\nu}}(i,j)\otimes V_{\alpha+K}
\end{equation*}
is an \aia{} with associated \fqs{}
\begin{equation*}
(\Z_m\times\Z_m,-Q_m)\times(K'/K,-Q_K).
\end{equation*}
Clearly, by definition of $\varphi$, the subgroup of all elements of the form
\begin{equation*}
(\varphi(\gamma),\gamma)\quad\text{for}\quad\gamma\in K'/K
\end{equation*}
is isotropic. By the generalisation of Proposition~\ref{prop:aiavoa} this implies that $M$ is a weak \voa{}. The central charges add up to 26.
\end{proof}

\begin{rem}\label{rem:choicedep}
It is not obvious from the definition that the isomorphism class of $M$ is independent of the choice \eqref{eq:isophi} of the isomorphism $\varphi$. This will be proved later in Proposition~\ref{prop:choiceindep}.
\end{rem}

We collect some properties of the ten cases. Recall that $\Aut(\Lambda)\cong\text{Co}_0$ and that the sporadic group $\text{Co}_1$ is the quotient $\text{Co}_1/\{\pm1\}$ of $\text{Co}_0$ by its centre.
\renewcommand{\arraystretch}{1.2}
\begin{equation*}
\begin{tabular}{r|c|c|l|l}
$\nu$: class in $\text{Co}_1$ & $\nu$: cycle shape & $\rho(V_\Lambda(\hat{\nu}))$ & Genus $\Lambda^\nu$ & Genus $K$\\\hline
$1A$   & $1^{24}$       &$0$    &$\II_{24,0}$                     &$\II_{1,1}$                     \\
$2A$   & $1^82^8$       &$1/2$  &$\II_{16,0}(2_{\II}^{+8})$       &$\II_{1,1}(2_{\II}^{+2})$       \\
$3B$   & $1^63^6$       &$2/3$  &$\II_{12,0}(3^{+6})$             &$\II_{1,1}(3^{-2})$             \\
$5B$   & $1^45^4$       &$4/5$  &$\II_{ 8,0}(5^{+4})$             &$\II_{1,1}(5^{+2})$             \\
$6E$   & $1^22^23^26^2$ &$5/6$  &$\II_{ 8,0}(2_{\II}^{+4}3^{+4})$ &$\II_{1,1}(2_{\II}^{+2}3^{-2})$ \\
$7B$   & $1^37^3$       &$6/7$  &$\II_{ 6,0}(7^{+3})$             &$\II_{1,1}(7^{-2})$             \\
$11A$  & $1^211^2$      &$10/11$&$\II_{ 4,0}(11^{+2})$            &$\II_{1,1}(11^{-2})$            \\
$14B$  & $1.2.7.14$     &$13/14$&$\II_{ 4,0}(2_{\II}^{+2}7^{+2})$ &$\II_{1,1}(2_{\II}^{+2}7^{-2})$ \\
$15D$  & $1.3.5.15$     &$14/15$&$\II_{ 4,0}(3^{-2}5^{-2})$       &$\II_{1,1}(3^{-2}5^{+2})$       \\
$23A,B$& $1.23$         &$22/23$&$\II_{ 2,0}(23^{+1})$            &$\II_{1,1}(23^{-2})$
\end{tabular}
\end{equation*}
\renewcommand{\arraystretch}{1}

\minisec{Grading}
The weak \voa{} $M$ will be the starting point of the BRST construction. By construction, $M$ is graded by $K'$ in the sense of Definition~\ref{defi:addgrad}. In the following we will see that $M$ is even naturally graded by $L'$ where
\begin{equation*}
L:=\Lambda^\nu\oplus K=\Lambda^\nu\oplus \II_{1,1}(m).
\end{equation*}
In order to see this we have to ``split off'' a lattice \voa{} from $V_\Lambda^{\hat{\nu}}$. Clearly, the lattice \voa{} $V_{\Lambda^\nu}$ is a \vosa{} of $V_\Lambda^{\hat{\nu}}$. The former has central charge $\rk(\Lambda)=24$, the latter is of central charge $\rk(\Lambda^\nu)=k-2$.

On the other hand, let $N:=\Lambda_\nu$ be the orthogonal complement of $\Lambda^\nu$ in $\Lambda$. Then $\nu$ restricts to a fixed-point free automorphism $\rho:=\nu|_N$ of $N$ of order $m$. The cycle shape of $\rho$ can be derived from that of $\nu$ \eqref{eq:cycleshapenu} and is given by
\begin{equation}\label{eq:cycleshaperho}
1^{-\rk(\Lambda^\nu)}\prod_{t\mid m}t^{24/\sigma_1(m)}.
\end{equation}
The function $u\colon\Lambda\to\{\pm1\}$ used to lift $\nu$ to an automorphism $\hat{\nu}$ of $V_\Lambda$ restricted to $N$ defines a lift $\hat{\rho}$ of $\rho$. Again, $\hat{\rho}^k$ is a standard lift of $\rho^k$ for all $k\in\N$. Clearly, also the \voa{} $U:=V_N^{\hat{\rho}}=(V_N)^{\hat{\rho}}$ is a \vosa{} of $V_\Lambda^{\hat{\nu}}$ and has central charge $24-\rk(\Lambda^\nu)$.

\begin{prop}\label{prop:fullsubvoa}
Let $\hat\nu$ be as in Assumption~\ref{ass:m}. Then the \vosa{}s $V_N^{\hat{\rho}}$ and $V_{\Lambda^\nu}$ of $V_\Lambda^{\hat{\nu}}$ satisfy
\begin{equation*}
V_{\Lambda^\nu}\cap V_N^{\hat{\rho}}=\C\vac
\end{equation*}
and the full \vosa{} of $V_\Lambda^{\hat{\nu}}$ generated by $V_{\Lambda^\nu}$ and $V_N^{\hat{\rho}}$ is isomorphic to $V_N^{\hat{\rho}}\otimes V_{\Lambda^\nu}$.
\end{prop}

We determine $U_1$:
\begin{prop}
Let $\hat\nu$ be as in Assumption~\ref{ass:m}. Then
\begin{equation*}
U_1=(V_N^{\hat{\rho}})_1=\{0\}.
\end{equation*}
\end{prop}
\begin{proof}
The lattice $\Lambda$ has no norm-one vectors and neither has the sublattice $N$. Hence
\begin{equation*}
(V_N)_1=\left\{h(-1)1\otimes\ee_0\xmiddle|h\in\k_{(0)}^\bot\right\}.
\end{equation*}
Then
\begin{equation*}
(V_N^{\hat{\rho}})_1=(V_N)_1\cap V_N^{\hat{\rho}}=\{0\}
\end{equation*}
since none of the elements of $\k_{(0)}^\bot=N\otimes_\Z\C$ are fixed by $\rho$.
\end{proof}

Proposition~\ref{prop:fullsubvoa} implies that we can decompose $V_\Lambda^{\hat{\nu}}$ and its irreducible modules $V_\Lambda^{\hat{\nu}}(i,j)$, $i,j\in\Z_m$, into a direct sum of modules for $V_N^{\hat{\rho}}\otimes V_{\Lambda^\nu}$. For this it is necessary to know the representation theory of $V_N^{\hat{\rho}}\otimes V_{\Lambda^\nu}$. Both $V_N^{\hat{\rho}}$ and $V_{\Lambda^\nu}$ fulfil Assumption~\ref{ass:n} by Theorem~\ref{thm:orb} and Proposition~\ref{prop:latnice} and they also fulfil Assumption~\ref{ass:p}. $V_{\Lambda^\nu}$ is simply a lattice \voa{} and therefore its representation theory is well known. In order to determine the fusion structure of $V_N^{\hat{\rho}}$ we need to apply Theorem~\ref{thm:nonhol}, which depends on Conjecture~\ref{conj:1}. The following lemma shows that the requirements of the theorem are met.
\begin{lem}
Let $\hat\nu$ be as in Assumption~\ref{ass:m}. Then $N$ and $\hat\rho$ satisfy Assumption~\ref{ass:b}.
\end{lem}
\begin{proof}
We have to show that all the items in Assumption~\ref{ass:b} are fulfilled. Items~\ref{enum:assb1}, \ref{enum:assb3} and \ref{enum:assb6} can be verified case by case.

It is easy to see that item~\ref{enum:assb2} holds. Indeed, let $\alpha\in N'$. It suffices to show that $(\id-\rho)\alpha\in\Lambda$. Let $\beta\in\Lambda$ be arbitrary. Then $\langle(\id-\rho)\alpha,\beta\rangle=\langle\alpha,(\id-\rho^{-1})\beta\rangle\in\Z$ since $(\id-\rho^{-1})\beta\in N$ and $\alpha\in N'$. Hence $(\id-\rho)\alpha\in\Lambda'=\Lambda$, using the unimodularity of $\Lambda$.

The vacuum anomaly in item~\ref{enum:assb4} depends on the cycle shape of $\rho$, which is given in \eqref{eq:cycleshaperho}. This cycle shape is closely related to that of $\nu$ \eqref{eq:cycleshapenu} and the vacuum anomaly can be computed similarly to Proposition~\ref{prop:confweightnu}.

Finally, we show that item~\ref{enum:assb5} is fulfilled for every even lattice $N$ of square-free level and hence in particular for the lattice $N$ in the ten cases. Indeed, since 4 does not divide the level $s$ of $N$, all 2-adic components appearing in the Jordan decomposition of $N'/N$ have to be even, which implies that the level $s$ of $N$ equals the exponent. Recall that since $N$ is integral, there is a basis $\{\alpha_1,\ldots,\alpha_r\}$ of $N'$ such that $\{m_1\alpha_1,\ldots,m_r\alpha_r\}$ is a basis of $N$ for some $m_1,\ldots,m_r\in\Ns$. Then, the exponent of the group $N'/N$ is $\lcm(m_1,\ldots,m_r)$. Now consider $\langle m_i\alpha_i,\alpha_i\rangle$ for some $i=1,\ldots,r$. If $m_i$ is odd, then
\begin{equation*}
m_i\langle m_i\alpha_i,\alpha_i\rangle=\langle m_i\alpha_i,m_i\alpha_i\rangle\in2\Z
\end{equation*}
since $N$ is an even lattice and hence $\langle m_i\alpha_i,\alpha_i\rangle$ has to be even since $m_i$ is odd. If $m_i$ is even on the other hand, then
\begin{equation*}
\frac{s}{m_i}\langle m_i\alpha_i,\alpha_i\rangle=\langle s\alpha_i,\alpha_i\rangle\in 2\Z
\end{equation*}
since $N$ has level $s$, implying that $\langle m_i\alpha_i,\alpha_i\rangle$ is even since $s/m_i$ is odd.
\end{proof}

Knowing the fusion group of $V_{\Lambda^\nu}$ and $V_N^{\hat{\rho}}$ and with the properties of tensor-product \voa{}s (see Section~\ref{sec:tensor}) we determine the fusion group of $V_N^{\hat{\rho}}\otimes V_{\Lambda^\nu}$.
\begin{prop}
Let $\hat\nu$ be as in Assumption~\ref{ass:m}, assume that Conjecture~\ref{conj:1} holds and that the representations $\phi_{V_{\alpha+N}(\hat\rho^i)}$ are chosen as in that conjecture or Theorem~\ref{thm:nonhol}. Then the \voa{} $V_N^{\hat{\rho}}\otimes V_{\Lambda^\nu}$ fulfils Assumptions~\ref{ass:sn}\ref{ass:p} and its irreducible modules are up to isomorphism
\begin{equation*}
V_N^{\hat{\rho}}(\alpha+N,i,j)\otimes V_{\beta+\Lambda^\nu}
\end{equation*}
for $\alpha+N\in N'/N$, $i,j\in\Z_m$, $\beta+\Lambda^\nu\in(\Lambda^\nu)'/\Lambda^\nu$ with conformal weights in $\langle\alpha,\alpha\rangle/2+ij/m+\langle\beta,\beta\rangle/2+\Z$. The fusion group is $F_{V_N^{\hat{\rho}}\otimes V_{\Lambda^\nu}}=N'/N\times\Z_m\times\Z_m\times(\Lambda^\nu)'/\Lambda^\nu$.
\end{prop}

We apply Proposition~\ref{prop:isomnegqf} to the unimodular lattice $\Lambda$ and its primitive sublattice $\Lambda^\nu$ to obtain a natural isomorphism of \fqs{}s
\begin{equation}\label{eq:isopsi}
\psi\colon\left((\Lambda^\nu)'/\Lambda^\nu,Q_{\Lambda^\nu}\right)\to\left(N'/N,-Q_N\right).
\end{equation}
More specifically, $\pi_\nu(\alpha)+\Lambda^\nu\mapsto\pi_\nu^\bot(\alpha)+N$ where $\pi_\nu$ and $\pi_\nu^\bot$ are the orthogonal projections of $\k=\Lambda\otimes_\Z\C$ onto $\k_{(0)}$ and $\k_{(0)}^\bot$, respectively. Recall that $\pi_\nu(\Lambda)=(\Lambda^\nu)'$ and $\pi_\nu^\bot(\Lambda)=N'$ since $\Lambda$ is unimodular.

For convenience, given Assumption~\ref{ass:m}, let us fix certain choices. This is always possible and we combine these choices into the following assumption:
\begin{customass}{{\textbf{\textsf{M'}}}}\label{ass:mp}
Let Assumption~\ref{ass:m} hold. In addition, assume that the representation $\phi_0=\phi_{V_\Lambda}$ of $\langle\hat\nu\rangle$ on $V_\Lambda$ is chosen naturally as in Remark~\ref{rem:phi0} as $\phi_{V_\Lambda}(\hat\nu)=\hat\nu$ and similarly that the representations $\phi_{V_{\alpha+N}}$ of $\langle\hat\rho\rangle$ on $V_{\alpha+N}$ are chosen naturally as in Proposition~\ref{prop:nonholchar} as $\phi_{V_{\alpha+N}}(\hat\rho)=\check\rho|_{V_{\alpha+N}}$ for $\alpha+N\in N'/N$.

Apart from that we assume that the representations $\phi_i=\phi_{V_\Lambda(\hat\nu^i)}$ of $\langle\hat\nu\rangle$ on $V_\Lambda(\hat\nu^i)$ for $i\in\Z_m\setminus\{0\}$ are chosen as in Corollary~\ref{cor:main} and that the representations $\phi_{V_{\alpha+N}(\hat\rho^i)}$ of $\langle\hat\rho\rangle$ on $V_{\alpha+N}(\hat\rho^i)$ for $\alpha+N\in N'/N$ and $i\in\Z_m\setminus\{0\}$ are chosen as in Conjecture~\ref{conj:1} or Theorem~\ref{thm:nonhol} so that in both cases the fusion rules have the simplest possible form.

Finally, given the lift $\check\rho$ of $\rho=\nu|_N$ to an automorphism of $\C_{\eps_1}[N']$ via the function $u_1\colon N'\to\C^\times$, and the identity as lift of $\id_{\Lambda^\nu}=\nu|_{\Lambda^\nu}$ to an automorphism of $\C_{\eps_2}[(\Lambda^\nu)']$ via the function $u_2\colon (\Lambda^\nu)'\to\C^\times$ with $u_2=1$, we define the lift $\hat\nu$ of $\nu$ to $\C_\eps[\Lambda]$ as follows: by a straightforward tensor-product construction the automorphism $\check\rho\otimes\id$ is a lift of $(\rho,\id)$ to $\C_{\eps_1}[N']\otimes\C_{\eps_2}[(\Lambda^\nu)']\cong\C_{\eps_1\times\eps_2}[N'\oplus(\Lambda^\nu)]$ via the function $u_1\times u_2$.
For this it is important to choose the 2-cocycle on $N'\oplus(\Lambda^\nu)$ as the product $\eps_1\times\eps_2$.
This 2-cocycle will have the desired property \eqref{eq:2cocycleirrmod}. Then, by restriction to $\Lambda\leq N'\oplus(\Lambda^\nu)$, we obtain a lift $\hat\nu$ of $\nu$ to $\C_\eps[\Lambda]$ via the function $u$ where $u=(u_1\times u_2)|_\Lambda$ and $\eps=(\eps_1\times \eps_2)|_\Lambda$ take values in a finite, cyclic subgroup of $\C^\times$, not necessarily equal to $\{\pm1\}$.
\end{customass}

With the help of the isomorphism $\psi$ \eqref{eq:isopsi} and the choices made in Assumption~\ref{ass:mp} we are able to decompose the irreducible $V_\Lambda^{\hat{\nu}}$-modules $V_\Lambda^{\hat{\nu}}(i,j)$, $i,j\in\Z_m$, into irreducible $V_N^{\hat{\rho}}\otimes V_{\Lambda^\nu}$-modules:
\begin{prop}
Let Assumption~\ref{ass:mp} and Conjecture~\ref{conj:1} hold. Then as $V_N^{\hat{\rho}}\otimes V_{\Lambda^\nu}$-modules,
\begin{equation*}
V_\Lambda^{\hat{\nu}}(i,j)\cong\bigoplus_{\alpha+\Lambda^\nu\in(\Lambda^\nu)'/\Lambda^\nu}V_N^{\hat{\rho}}(\psi(\alpha+\Lambda^\nu),i,j)\otimes V_{\alpha+\Lambda^\nu}
\end{equation*}
for $i,j\in\Z_m$.
\end{prop}
\begin{proof}
We saw that $V_N^{\hat{\rho}}\otimes V_{\Lambda^\nu}$ is a full \vosa{} of $V_\Lambda^{\hat{\nu}}$. Consequently, every $V_\Lambda^{\hat{\nu}}$-module is also a $V_N^{\hat{\rho}}\otimes V_{\Lambda^\nu}$-module and can be decomposed into a direct sum of irreducible $V_N^{\hat{\rho}}\otimes V_{\Lambda^\nu}$-modules.

It follows from the proof of Proposition~\ref{prop:isomnegqf} that
\begin{align*}
\Lambda&\cong\bigcup_{\beta+N\oplus\Lambda^\nu\in\Lambda/(N\oplus\Lambda^\nu)}(\pi_\nu^\bot(\beta),\pi_\nu(\beta))+(N\oplus \Lambda^\nu)\\
&=\bigcup_{\alpha+\Lambda^\nu\in(\Lambda^\nu)'/\Lambda^\nu}\psi(\alpha+\Lambda^\nu)\oplus(\alpha+\Lambda^\nu)
\end{align*}
and hence
\begin{equation*}
V_\Lambda\cong\bigoplus_{\alpha+\Lambda^\nu\in(\Lambda^\nu)'/\Lambda^\nu}V_{\psi(\alpha+\Lambda^\nu)}\otimes V_{\alpha+\Lambda^\nu}
\end{equation*}
holds.\footnote{\label{footnote:3}This uses the fact that for two positive-definite, even lattices $L_1$ and $L_2$, the \voa{}s $V_{L_1}\otimes V_{L_2}$ and $V_{L_1\oplus L_2}$ are isomorphic. Moreover, under this identification, the irreducible modules $V_{\alpha+L_1}\otimes V_{\beta+L_2}$ and $V_{(\alpha,\beta)+L_1\oplus L_2}$ are isomorphic for $\alpha+L_1\in L'_1/L_1$ and $\beta+L_2\in L'_2/L_2$.} Then the detailed requirements on the automorphisms $\hat\nu$ on $V_\Lambda$ and $\check\rho$ on $V_N$ and its modules in Assumption~\ref{ass:mp} imply the statement of the proposition in the untwisted sector:
\begin{equation*}
V_\Lambda^{\hat{\nu}}(0,j)\cong\bigoplus_{\alpha+\Lambda^\nu\in(\Lambda^\nu)'/\Lambda^\nu}V_N^{\hat{\rho}}(\psi(\alpha+\Lambda^\nu),0,j)\otimes V_{\alpha+\Lambda^\nu}
\end{equation*}
for $j\in\Z_m$.

To obtain the assertion for the twisted modules, consider the \voa{} $V_N^{\hat{\rho}}\otimes V_{\Lambda^\nu}$ with fusion group $N'/N\times\Z_m\times\Z_m\times(\Lambda^\nu)'/\Lambda^\nu$. Then by Corollary~\ref{cor:scefusion} the \voa{}
\begin{equation*}
V_\Lambda^{\hat{\nu}}\cong\bigoplus_{\alpha+\Lambda^\nu\in(\Lambda^\nu)'/\Lambda^\nu}V_N^{\hat{\rho}}(\psi(\alpha+\Lambda^\nu),0,0)\otimes V_{\alpha+\Lambda^\nu}
\end{equation*}
has fusion group $\Z_m\times\Z_m$, which is the orthogonal complement of the ``diagonal'' isotropic subgroup $\left\{(\psi(\alpha+\Lambda^\nu),0,0,\alpha+\Lambda^\nu)\xmiddle|\alpha+\Lambda^\nu\in(\Lambda^\nu)'/\Lambda^\nu\right\}$ in $N'/N\times\Z_m\times\Z_m\times(\Lambda^\nu)'/\Lambda^\nu$. In other words, the irreducible modules for this \voa{} are given by
\begin{equation*}
W(i,j):=\bigoplus_{\alpha+\Lambda^\nu\in(\Lambda^\nu)'/\Lambda^\nu}V_N^{\hat{\rho}}(\psi(\alpha+\Lambda^\nu),i,j)\otimes V_{\alpha+\Lambda^\nu}
\end{equation*}
for $i,j\in\Z_m$ with fusion given by addition in $\Z_m\times\Z_m$. On the other hand, we already determined the irreducible $V_\Lambda^{\hat{\nu}}$-modules, which are $V_\Lambda^{\hat{\nu}}(i,j)$ for $i,j\in\Z_m$, again with fusion rules given by addition in $\Z_m\times\Z_m$.

To match up the $W(i,j)$ and the $V_\Lambda^{\hat{\nu}}(i,j)$ we have to find a group automorphism of $\Z_m\times\Z_m$, which also preserves the quadratic form given by the conformal weights modulo~1, i.e.\ an automorphism of the \fqs{} $(\Z_m\times\Z_m,Q_m)$ where we recall that $Q_m((i,j))=ij/m+\Z$. In general, for $m\in\Ns$, an endomorphism of the group $\Z_m\times\Z_m$ is of the form
\begin{equation*}
(i,j)\mapsto(ai+bj,ci+dj)
\end{equation*}
for some $a,b,c,d\in\Z_m$. If this endomorphism preserves the given quadratic form, then $ac=bd=0\pmod{m}$ and $ad+bc=1\pmod{m}$. In our case, we also know that $(0,j)\mapsto(0,j)$ for all $j\in\Z_m$. This implies that $b=0\pmod{m}$ and $d=1\pmod{m}$, which together with the above conditions entails $a=1\pmod{m}$ and $c=0\pmod{m}$. This leaves only the identity endomorphism
\begin{equation*}
(i,j)\mapsto(i,j),
\end{equation*}
which is clearly an automorphism. Hence we obtain
\begin{equation*}
W(i,j)\cong V_\Lambda^{\hat{\nu}}(i,j)
\end{equation*}
for all $i,j\in\Z_m$. This proves the assertion.
\end{proof}

The proposition together with the definition of $M$ in Proposition~\ref{prop:mdef} implies:
\begin{thm}
Let Assumption~\ref{ass:mp} and Conjecture~\ref{conj:1} hold. Then the weak \voa{} $M$ has $L'/L$-decomposition \eqref{eq:lpldec},
\begin{align*}
M\cong\bigoplus_{\gamma+L\in L'/L}V_N^{\hat{\rho}}(\chi(\gamma+L))\otimes V_{\gamma+L}
\end{align*}
where $L=\Lambda^\nu\oplus K=\Lambda^\nu\oplus \II_{1,1}(m)$ and $\chi$ is the isomorphism $\chi=(\psi,\varphi)$:
\begin{equation}\label{eq:isochi}
(L'/L,Q_L)
\longrightarrow\left(N'/N,-Q_N\right)\times(\Z_m\times\Z_m,-Q_m).
\end{equation}
\end{thm}
\begin{proof}
With the above proposition we can decompose $M$ as $V_N^{\hat{\rho}}\otimes V_L$-module
\begin{align*}
M&=\bigoplus_{\beta+K\in K'/K}V_\Lambda^{\hat{\nu}}(\varphi(\beta+K))\otimes V_{\beta+K}\\
&\cong\bigoplus_{\beta+K\in K'/K}\left(\bigoplus_{\alpha+\Lambda^\nu\in(\Lambda^\nu)'/\Lambda^\nu}V_N^{\hat{\rho}}(\psi(\alpha+\Lambda^\nu),\varphi(\beta+K))\otimes V_{\alpha+\Lambda^\nu}\right)\otimes V_{\beta+K}\\
&\cong\bigoplus_{\gamma+L\in L'/L}V_N^{\hat{\rho}}(\chi(\gamma+L))\otimes V_{\gamma+L},
\end{align*}
where we use that $V_{\alpha+\Lambda^\nu}\otimes V_{\beta+K}\cong V_{(\alpha,\beta)+L}$ (see footnote~\ref{footnote:3}).

We can view the right-hand side as the vertex algebra obtained as abelian intertwining subalgebra of the tensor product of the \aia{}s on the direct sum of the irreducible modules of $V_N^{\hat{\rho}}$ and $V_L$, respectively. Then the above is even an isomorphism of weak \voa{}s.
\end{proof}

The theorem states in particular that $M$ is naturally graded by $L'$ in the sense of Definition~\ref{defi:addgrad}, i.e.\
\begin{equation*}
M=\bigoplus_{\alpha\in L'}M(\alpha)\quad\text{with}\quad M(\alpha)=U(\chi(\alpha+L))\otimes M_{\hat\h}(1,\alpha)
\end{equation*}
where $U=V_N^{\hat{\rho}}$ and $U_1=\{0\}$.

The definition of $M$ in Proposition~\ref{prop:mdef} depends on the choice \eqref{eq:isophi} of the isomorphism $\varphi\colon (K'/K,Q_K)\to(\Z_m\times\Z_m,-Q_m)$ (see Remark~\ref{rem:choicedep}). With the help of the $L'/L$-decomposition it is possible to show that the isomorphism class of $M$ does not depend on this choice.
\begin{lem}\label{lem:surjective}
Let $\hat\nu$ be as in Assumption~\ref{ass:m}. Then the natural group homomorphism $\Aut(L)\to\Aut(L'/L)$ is surjective.
\end{lem}
\begin{proof}
Of the ten lattices $L=\Lambda^\nu\oplus\II_{1,1}(m)$ all but one fulfil the assumptions of Proposition~\ref{prop:thm1.14.2}, a weakened version of a result in \cite{Nik80}, which implies the statement of the lemma. Only the lattice in the case $m=23$ of genus $\II_{3,1}(23^{-3})$ has to be treated separately. A direct computation using \texttt{Magma} \cite{Magma} proves the assertion for this case.
\end{proof}
\begin{prop}\label{prop:choiceindep}
Let Assumption~\ref{ass:mp} and Conjecture~\ref{conj:1} hold. Then the isomorphism class of $M$ does not depend on the isomorphism $\chi\colon (L'/L,Q_L)\to\left(N'/N,-Q_N\right)\times(\Z_m\times\Z_m,-Q_m)$ \eqref{eq:isochi} and is in particular independent of the choice \eqref{eq:isophi} of the isomorphism $\varphi\colon (K'/K,Q_K)\to(\Z_m\times\Z_m,-Q_m)$.
\end{prop}
\begin{proof}
We proceed analogously to the proof of Lemma~3.1 in \cite{HS14}. Given a \voa{} $V$, the automorphism group $\Aut(V)$ acts naturally from the right on the set of isomorphism classes of irreducible $V$-modules (see Proposition~\ref{prop:autaction}, untwisted case). This remains true in the case of weak \voa{}s. Now suppose that $L$ is some even lattice, not necessarily positive-definite. Then $V_L$ is a weak \voa{} whose irreducible modules are indexed by $L'/L$ and the action of an automorphism of $V_L$ obtained as lift of an automorphism of $L$ can be explicitly determined. Indeed,
\begin{equation*}
V_{\alpha+L}\cdot\hat\tau\cong V_{\tau^{-1}\alpha+L}
\end{equation*}
for $\tau\in\Aut(L)$ and $\hat\tau$ some lift, as can be seen from the proof of Lemma~\ref{lem:autactionlat}, which remains valid in this setting.

We proved in Lemma~\ref{lem:surjective} that $\Aut(L)$ maps surjectively onto $\Aut(L'/L)$. The above considerations show that the same is true for the natural action of $\Aut(V_L)$ on the set of isomorphism classes of irreducible $V_L$-modules, indexed by $L'/L$. Hence, any change in the isomorphism $\chi$ can be absorbed in an automorphism of $L'/L$, which is induced from an automorphism of $V_L$.
As in \cite{HS14}, Lemma~3.1, this proves the assertion.
\end{proof}

\minisec{BRST Construction}
We just showed that all the assumptions made in Section~\ref{sec:brst} on the weak \voa{} in the matter sector like the $L'/L$-decomposition are fulfilled by $M$. Hence we can consider the complex Lie algebra
\begin{equation*}
\g^{\hat\nu}:=H_\text{BRST}^1(M)
\end{equation*}
from the BRST construction. The goal is to show that $\g^{\hat\nu}$ is a \BKMa{} and isomorphic to $\g_{\hat\nu}$, or more precisely its complexification, where $\g_{\hat\nu}$ is the real \BKMa{} obtained by twisting the denominator identity of the Fake Monster Lie algebra $\g$.

With the results in Section~\ref{sec:brst} we immediately obtain:
\begin{prop}
Let Assumption~\ref{ass:mp} and Conjecture~\ref{conj:1} hold. Then the complex Lie algebra $\g^{\hat\nu}$ is graded by the lattice $L'=(\Lambda^\nu)'\oplus(\II_{1,1}(m))'$ with
\begin{equation*}
\dim_\C(\g^{\hat\nu}(\alpha))=\left[\ch_{V_N^{\hat{\rho}}(\chi(\alpha+L))}(q)/\eta(q)^{\rk(\Lambda^\nu)}\right](-\langle\alpha,\alpha\rangle/2)
\end{equation*}
for $\alpha\in L'\setminus\{0\}$ and
\begin{equation*}
\dim_\C(\g^{\hat\nu}(0))=\rk(L)=\rk(\Lambda^\nu)+2=k=2+\frac{24\sigma_0(m)}{\sigma_1(m)}.
\end{equation*}
The component $\g^{\hat\nu}(0)$ is a Cartan subalgebra for $\g^{\hat\nu}$ so that $\g^{\hat\nu}$ has rank $k=\rk(\Lambda^\nu)+2$.
\end{prop}

Recall that the $L'$-grading on the Lie algebra $\g^{\hat\nu}$ means that
\begin{equation*}
\g^{\hat\nu}=\bigoplus_{\alpha\in L'}\g^{\hat\nu}(\alpha)\quad\text{with}\quad\g^{\hat\nu}(\alpha)=H^1_\text{BRST}(\alpha)
\end{equation*}
and $[\g^{\hat\nu}(\alpha),\g^{\hat\nu}(\beta)]\subseteq\g^{\hat\nu}(\alpha+\beta)$. Also recall that the Cartan subalgebra $\g^{\hat\nu}(0)$ acts on the graded components via $[x,y]=\langle h,\alpha\rangle y$ for $x=\vac_{V_N^{\hat\rho}}\otimes h(-1)1\otimes c\in\g^{\hat\nu}(0)$ and $y\in\g^{\hat\nu}(\alpha)$, $\alpha\in L'$ (see Theorem~\ref{thm:bkma}).

\minisec{Dimensions}
We describe the dimensions of the graded components $\g^{\hat\nu}(\alpha)$ in terms of certain vector-valued modular forms. Consider the eta product
\begin{equation*}
f(\tau):=\frac{1}{\eta_\nu(\tau)}=\prod_{t\mid m}\eta(t\tau)^{-24/\sigma_1(m)},
\end{equation*}
which depends on the cycle shape of $\nu\in\Aut(\Lambda)$. Borcherds showed that products of rescaled eta functions of this form are sometimes modular forms. More precisely, Theorem~6.2 in \cite{Bor00} implies that $f(\tau)$ is an almost holomorphic\footnote{By \emph{almost holomorphic} we mean that the modular form is holomorphic as function on $\H$ but may have poles at the cusps.} (scalar-valued) modular form of weight $w:=-12\sigma_0(m)/\sigma_1(m)=1-k/2\in\Z$ for $\Gamma_0(m)$ and a certain character $\chi_s$ where $s=s(m)\in\Ns$ is chosen such that $s\prod_{t\mid m}t^{-24/\sigma_1(m)}$ is a perfect square, i.e.:
\renewcommand{\arraystretch}{1.2}
\begin{equation*}
\begin{tabular}{c|cccccccccc}
$m$&1&2&3&5&6&7&11&14&15&23\\\hline
$s$&1&1&1&1&1&7& 1& 1& 1&23
\end{tabular}
\end{equation*}
\renewcommand{\arraystretch}{1}

The Dirichlet character $\chi_s$ (of a certain modulus) for fixed $s\in\Ns$ is defined as the Kronecker symbol $\chi_s(j):=(j/s)$, $j\in\Z$. For instance, if $s$ is an odd prime, then $\chi_s$ is a character modulo~$s$. For $s=1$ we get the trivial character. Given a quadratic Dirichlet character $\chi\colon\Z\to\{\pm1\}$ of some modulus $k$ we can view it as a character $\chi\colon\Gamma_0(k)\to\{\pm1\}$ by setting $\chi(M):=\chi(a)=\chi(d)$ for $M=\left(\begin{smallmatrix}a&b\\c&d\end{smallmatrix}\right)\in\Gamma_0(k)$. Then clearly, $\chi$ is also a character on $\Gamma_0(l)$ for any multiple $l$ of $k$. In that manner, for the $s=s(m)$ obtained from Borcherds' theorem, $\chi_s$ is a character on $\Gamma_0(m)$ and it is the trivial character except for $m=7,23$.

Consider now the lattice $L=\Lambda^\nu\oplus K$ and its discriminant form $L'/L$. It has level $m$ and even signature. For any \fqs{} $D$ of even signature and level $N$ we can define
\begin{equation*}
\chi_D(j):=\left(\frac{j}{\left|D\right|}\right)e\left(\left(j-1\right)\oddity(D)/8\right),
\end{equation*}
$j\in\Z$, which is a quadratic Dirichlet character modulo~$N$ (see e.g.\ Section~6 in \cite{Sch06}). If 4 does not divide the level $N$, for instance if $N$ is square-free, then the character simplifies and becomes 
\begin{equation*}
\chi_D(j)=\left(\frac{j}{\left|D\right|}\right).
\end{equation*}
Using the elementary properties of the Kronecker symbol we see that
\begin{equation*}
\chi_{L'/L}=\chi_s
\end{equation*}
for $s$ as above.

We now describe a procedure developed in \cite{Sch06} to lift scalar-valued modular forms to vector-valued ones.
\begin{thm}[Special case of \cite{Sch06}, Theorem~6.2]\label{thm:6.2}
Let $D$ be a \fqs{} of even signature and level dividing $N$. Let $f$ be a (meromorphic) modular form of weight $w\in\Z$ for $\Gamma_0(N)$ and character $\chi_D$. Then
\begin{equation*}
F(\tau)=\sum_{\gamma\in D}F_\gamma(\tau)\ee_\gamma
\end{equation*}
with
\begin{equation*}
F_\gamma(\tau):=\sum_{M\in\Gamma_0(N)\backslash\Gamma}(c\tau+d)^{-w}f(M.\tau)\rho_D(M^{-1})_{\gamma,0}
\end{equation*}
is a (meromorphic) vector-valued modular form of weight $w$ for the Weil representation $\rho_D$ of $\SLZ$ on $\C[D]$, which is invariant under the automorphisms of the \fqs{} $D$. $F$ is called the \emph{lift of $f$} (with trivial support).
\end{thm}
The statement of the theorem remains true if we replace the Weil representation by the dual Weil representation $\overline{\rho}_D$, the complex conjugate representation.\footnote{In the notation of \cite{Sch06} the Weil representation and the dual Weil representation are interchanged compared to our notation.}

We apply the above theorem for $D=L'/L$ and proceed by lifting $f(\tau)=1/\eta_\nu(\tau)$, which is a modular form of weight $w=1-k/2=-12\sigma_0(m)/\sigma_1(m)\in\Z$ for $\Gamma_0(m)$ with character $\chi_{L'/L}$ to a vector-valued modular form $F(\tau)=\sum_{\alpha+L\in L'/L}F_{\alpha+L}(\tau)\ee_{\alpha+L}$ of weight $w$ for the \emph{dual} Weil representation $\overline\rho_{L'/L}$ with
\begin{equation}\label{eq:vvmf}
F_{\alpha+L}(\tau):=\sum_{M\in\Gamma_0(m)\backslash\Gamma}(c\tau+d)^{-w}\frac{1}{\eta_\nu(M.\tau)}\overline\rho_{L'/L}(M^{-1})_{\alpha+L,0+L}.
\end{equation}
It turns out that the components of $F(\tau)$ are closely related to the characters of the irreducible $V_N^{\hat{\rho}}$-modules. Recall that there is an isomorphism $\chi\colon (L,Q_L)\to\left(N'/N,-Q_N\right)\times(\Z_m\times\Z_m,-Q_m)$ \eqref{eq:isochi}.
\begin{thm}\label{thm:charlift}
Let Assumption~\ref{ass:mp} and Conjecture~\ref{conj:1} hold. Then, with the vector-valued modular form $F(\tau)$ defined in \eqref{eq:vvmf},
\begin{equation*}
\ch_{V_N^{\hat{\rho}}(\chi(\alpha+L))}(\tau)/\eta(\tau)^{\rk(\Lambda^\nu)}=F_{\alpha+L}(\tau)
\end{equation*}
for all $\alpha+L\in L'/L$.
\end{thm}
The following corollary is immediate:
\begin{cor}\label{cor:thmcharlift}
Let Assumption~\ref{ass:mp} and Conjecture~\ref{conj:1} hold. Then the dimensions of the $L'$-graded Lie algebra $\g^{\hat\nu}$ fulfil
\begin{equation*}
\dim_\C(\g^{\hat\nu}(\alpha))=\left[F_{\alpha+L}\right](-\langle\alpha,\alpha\rangle/2)
\end{equation*}
for all $\alpha\in L'\setminus\{0\}$.
\end{cor}
\begin{proof}[Proof of Theorem~\ref{thm:charlift}]
We define $\widetilde{F}(\tau):=\sum_{\alpha+L\in L'/L}F_{\alpha+L}(\tau)\ee_{\alpha+L}$ with component functions $\widetilde{F}_{\alpha+L}:=\ch_{V_N^{\hat{\rho}}(\chi(\alpha+L))}(\tau)/\eta(\tau)^{\rk(\Lambda^\nu)}$. We have to prove that $\widetilde{F}=F$.

By definition, $F$ is a vector-valued modular form of weight $w=-12\sigma_0(m)/\sigma_1(m)\in\Z$ for $\overline{\rho}_{L'/L}$. The \voa{} $V_N^{\hat{\rho}}$ has central charge $c=24-\rk(\Lambda^\nu)$. Corollary~\ref{cor:zhuweil} states that the functions $\ch_{V_N^{\hat{\rho}}(\chi(\alpha+L))}\eta(\tau)^{24-\rk(\Lambda^\nu)}$ form a vector-valued modular form of weight $c/2=12-\rk(\Lambda^\nu)/2$ for the Weil representation with respect to the fusion group $\left(N'/N,Q_N\right)\times(\Z_m\times\Z_m,Q_m)\cong(L'/L,-Q_L)=\overline{L'/L}$ (via $\chi$), which is the same as the dual Weil representation $\overline{\rho}_{L'/L}$ with respect to $L'/L=(L'/L,Q_L)$. Dividing by $\eta(\tau)^{24}$, which is a modular form of weight 12 for $\SLZ$ (and trivial character), we conclude that $\widetilde{F}$ is also a vector-valued modular form of weight $-\rk(\Lambda^\nu)/2=w$ for $\overline{\rho}_{L'/L}$.

Now $F$ and $\widetilde{F}$ are both vector-valued modular forms of the same \emph{negative} weight $w$ for $\overline{\rho}_{L'/L}$. Moreover, $F$ and $\widetilde{F}$ are almost holomorphic, i.e.\ holomorphic on $\H$. Indeed, all eta products are holomorphic on $\H$ including the function $f$ and hence the lift $F$ is holomorphic on $\H$. Similarly, $\widetilde{F}$ is holomorphic on $\H$ by Theorem~\ref{thm:zhumodinv}.

We can compute the $q_\tau$-expansion of $F(\tau)$ and $\widetilde{F}(\tau)$ explicitly and verify that the coefficients with negative $q_\tau$-exponents are identical. The lift $F(\tau)$ takes a very simple form (see Proposition~\ref{prop:simplelift} below) and hence its $q_\tau$-expansion can be easily determined knowing the $q_\tau$-expansion of the eta function. The computation of the characters of the irreducible $V_N^{\hat{\rho}}$-modules, which enter $\widetilde{F}(\tau)$, is described at the end of Section~\ref{sec:nonholorb}. The calculations were performed in \texttt{Sage} and \texttt{Magma} \cite{Sage,Magma}.

Then $G:=F-\widetilde{F}$ is an almost holomorphic modular form of negative weight, which is finite at $q_\tau=0$ ($\tau=\i\infty$) and therefore has to vanish by the valence formula. Hence $F=\widetilde{F}$. Indeed, all the components of $G(\tau)$ are finite at the cusp $\i\infty$. Moreover, they are scalar-valued modular forms with trivial character for $\Gamma(m)$. The expansion of such a component $G_{\alpha+L}(\tau)$ at any other cusp $M.\i\infty$ of $\Gamma(m)$ is given by $(c\tau+d)^{-w}G_{\alpha+L}(M.\tau)$, which is just a linear combination of the $G_{\beta+L}(\tau)$, $\beta+L\in L'/L$. Hence all the components of $G$ are finite at all cusps of $\Gamma(m)$. The valence formula (see e.g.\ \cite{HBJ94}, Theorem~I.4.1) then implies that each component $G_{\alpha+L}$ of $G$ has to vanish.
\end{proof}

\minisec{Explicit Formula}

In the ten cases at hand there is a nice explicit formula for the vector-valued modular form $F(\tau)$. Consider again the function $f(\tau)=1/\eta_\nu(\tau)$. For a divisor $d$ of $m$ we consider $f(\tau/d)$, which has a Fourier expansion in $q_\tau^{1/d}$. For $j\in\Z_d$ let $g_{d,j}(\tau)$ be the function obtained by only keeping the terms in $f(\tau/d)$ with $q_\tau$-exponents in $j/d+\Z$. Then $g_{d,j}(\tau)$ transforms under $T$ like $g_{d,j}(\tau+1)=\e^{(2\pi\i)j/d}g_{d,j}(\tau)$ and $f(\tau/d)=g_{d,0}(\tau)+\ldots+g_{d,d-1}(\tau)$. We can also define the $g_{d,j}$, $j\in\Z_d$, via
\begin{equation*}
g_{d,j}(\tau)=\frac{1}{d}\sum_{k\in\Z_d}\e^{(2\pi\i)(-kj/d)}f\left(\frac{\tau+k}{d}\right).
\end{equation*}
\begin{prop}\label{prop:simplelift}
Let $\hat\nu$ be as in Assumption~\ref{ass:m} and let $F$ be the vector-valued modular form \eqref{eq:vvmf} obtained as lift of $f$. Then
\begin{equation*}
F_{\alpha+L}(\tau)=\sum_{d\mid m}\delta_{\alpha\in L'\cap \frac{1}{d}L}\,g_{d,j_{\alpha+L,d}}(\tau)
\end{equation*}
for all $\alpha+L\in L'/L$ where $j_{\alpha+L,d}\in\Z_d$ is such that $-j_{\alpha+L,d}/d=\langle\alpha,\alpha\rangle/2\pmod{1}$.
\end{prop}
Together with Corollary~\ref{cor:thmcharlift} this implies:
\begin{cor}\label{cor:dimform}
Let Assumption~\ref{ass:mp} and Conjecture~\ref{conj:1} hold. Then the dimensions of the graded components $\g^{\hat\nu}(\alpha)$ are
\begin{equation*}
\dim_\C(\g^{\hat\nu}(\alpha))=\left[F_{\alpha+L}\right](-\langle\alpha,\alpha\rangle/2)=\sum_{d\mid m}\delta_{\alpha\in L'\cap \frac{1}{d}L}\left[\frac{1}{\eta_\nu}\right]\left(-d\frac{\langle\alpha,\alpha\rangle}{2}\right)
\end{equation*}
for all $\alpha\in L'\setminus\{0\}$.
\end{cor}
Before we prove the above proposition we make some remarks on the vector-valued modular form $F(\tau)$:
\begin{rem}\label{rem:vvmfF}
\item
\begin{enumerate}
\item Since $L=\Lambda^\nu\oplus\II_{1,1}(m)$ and $P=L\oplus\II_{1,1}$ have the same discriminant form $L'/L\cong P'/P$,
we can view $F(\tau)$ also as a vector-valued modular form for the dual Weil representation $\overline{\rho}_{P'/P}$ on $\C[P'/P]$. As such $F(\tau)$ is \emph{completely reflective} (as defined in \cite{Sch06}, Section~9). Note that the lattice $P$ has signature $(k,2)$ and $F(\tau)$ weight $w=1-k/2$ with $k\geq 4$ even.

In the ten cases at hand this means that singular terms in the $q_\tau$-expansion of $F(\tau)$ appear exactly in the components $F_{\alpha+P}(\tau)$, $\alpha+P\in P'/P$, with $\langle\alpha,\alpha\rangle/2=1/d\pmod{1}$ for $d\mid m$ and in such a component the only singular term is $1\cdot q^{-1/d}$.
\item As completely reflective modular form $F(\tau)$ is in particular \emph{symmetric}, i.e.\ invariant under the automorphisms of the \fqs{} $P'/P\cong L'/L$ (see Section~9 in \cite{Sch06} and note that $m$, the level of $P$ or $L$, is square-free). This also follows immediately from Theorem~\ref{thm:6.2}.

Then, by Theorem~\ref{thm:charlift}, the characters of the irreducible $V_N^{\hat{\rho}}$-modules are invariant under the automorphisms of the fusion group $F_{V_N^{\hat{\rho}}}$ as \fqs{}. In particular, the characters $\ch_{V_N^{\hat{\rho}}(\chi(\alpha+L))}(\tau)$ do not depend on the choice of the isomorphism $\chi\colon (L'/L,Q_L)\to\left(N'/N,-Q_N\right)\times(\Z_m\times\Z_m,-Q_m)$ \eqref{eq:isochi}.

\item The automorphic product $\Psi_{\hat\nu}$ on $P$, which is the denominator identity of $\g_{\hat\nu}$ (and that of $\g^{\hat\nu}$ as we will see shortly), is constructed in \cite{Sch06} precisely as the Borcherds lift of the modular form $F(\tau)$.
\end{enumerate}
\end{rem}
\begin{proof}[Proof of Proposition~\ref{prop:simplelift}]
We make use of Theorem~6.5 in \cite{Sch06} (in the special case of trivial support). Let $F$ be the lift of a scalar-valued modular form $f$ for the dual Weil representation $\overline{\rho}_D$ for some discriminant form $D$ of even signature as in Theorem~\ref{thm:6.2} (with $N$ now called $m$) and assume that $m$ is square-free. Then for $\gamma\in D$,
\begin{equation*}
F_\gamma(\tau)=\sum_{c\mid m}\delta_{\gamma\in D_c}\xi_{\frac{m}{c}}\frac{1}{\sqrt{|D_c|}}c\,h_{c,j_{\gamma,c}}(\tau)
\end{equation*}
where for $c\mid m$ the $\xi_c$ are certain factors of unit modulus and the $h_{c,j}$, $j\in\Z$, are obtained from $f_{m/c}(\tau)$ in the same manner as the $g_{c,j}$ are obtained from $f(\tau/c)$ (as described above). The $f_c(\tau)$ for $c\mid m$ are defined as $f_c(\tau):=(c\tau+d)^{-w}f(M_c.\tau)$ where the matrices $M_c=\left(\begin{smallmatrix}a&b\\c&d\end{smallmatrix}\right)$ are chosen in $\SLZ$ with $d=1\pmod{c}$ and $d=0\pmod{m/c}$. Finally, $D_c=\left\{\gamma\in D\xmiddle| c\gamma=0\right\}$.

Now let us return to the specific cases at hand. The modular-transformation properties of the eta function and rescaled eta functions are explicitly known (see e.g.\ \cite{Sch09}, Proposition~6.2). This allows us to compute the $f_c(\tau)$. Due to the highly symmetric nature of the eta product $f(\tau)=1/\eta_\nu(\tau)=\prod_{t\mid m}\eta(t\tau)^{-24/\sigma_1(m)}$ one obtains that $f_{m/c}(\tau)$ is up to a phase factor (call it $\psi_{m/c}$) of unit modulus given by
\begin{equation*}
f(\tau/c)\prod_{t\mid m}(t,c)^{12/\sigma_1(m)}
\end{equation*}
and hence
\begin{equation*}
h_{c,j}(\tau)=g_{c,j}(\tau)\psi_{\frac{m}{c}}\prod_{t\mid m}(t,c)^{12/\sigma_1(m)}.
\end{equation*}
As a group
\begin{equation*}
D=L'/L\cong\Z_m^2\times\prod_{t\mid m}\Z_t^{24/\sigma_1(m)}
\end{equation*}
and hence the cardinality of $D_c=(L'\cap(1/c)L)/L$ is
\begin{equation*}
|D_c|=c^2\prod_{t\mid m}(t,c)^{24/\sigma_1(m)}.
\end{equation*}
Consequently all factors of non-unit modulus cancel and
\begin{equation*}
F_\gamma(\tau)=\sum_{c\mid m}\delta_{\gamma\in D_c}\xi_{\frac{m}{c}}\psi_{\frac{m}{c}}g_{c,j_{\gamma,c}}(\tau).
\end{equation*}
Finally, a case-by-case study reveals that $\xi_c\psi_c=1$ for all $m=1,2,3,5,6,7,11,14,15,23$ and all $c\mid m$, completing the proof.
\end{proof}

\minisec{Rescaling}

The complex Lie algebra $\g^{\hat\nu}$ is graded by the rational lattice
\begin{equation*}
L'=(\Lambda^\nu)'\oplus (\II_{1,1}(m))'.
\end{equation*}
Due to the special form of the ten automorphisms the corresponding fixed-point sublattices $\Lambda^\nu$ have the property that $(\Lambda^\nu)'\cong \Lambda^\nu(1/m)$.
Proposition~\ref{prop:rescaled} yields that $(\II_{1,1}(m))'\cong \II_{1,1}(1/m)$ and hence
\begin{equation*}
L'\cong \Lambda^\nu(1/m)\oplus \II_{1,1}(1/m).
\end{equation*}
For convenience, we rescale the quadratic form by $m$ and obtain
\begin{equation*}
\Delta:=L'(m)=\Lambda^\nu\oplus \II_{1,1},
\end{equation*}
which is an even lattice. Then Corollary~\ref{cor:dimform} implies:
\begin{prop}\label{prop:rescaleddim}
Let Assumption~\ref{ass:mp} and Conjecture~\ref{conj:1} hold. After the above rescaling, the complex Lie algebra $\g^{\hat\nu}$ is graded by the even lattice $\Delta:=L'(m)=\Lambda^\nu\oplus \II_{1,1}$ with the dimension of the graded components given by
\begin{align*}
\dim_\C(\g^{\hat\nu}(\alpha))&=\sum_{d\mid m}\delta_{\alpha\in \Delta\cap \frac{m}{d}\Delta'}\left[\frac{1}{\eta_\nu}\right]\left(-\frac{d}{m}\frac{\langle\alpha,\alpha\rangle}{2}\right)\\
&=\sum_{d\mid m}\delta_{\alpha\in \Delta\cap d\Delta'}\left[\frac{1}{\eta_\nu}\right]\left(-\frac{1}{d}\frac{\langle\alpha,\alpha\rangle}{2}\right)
\end{align*}
for all $\alpha\in\Delta\setminus\{0\}$.
\end{prop}

\minisec{Identification of the Lie Algebra}
We observe from the above proposition that $\g^{\hat\nu}$ has the same grading and the same dimensions of the graded components as the real \BKMa{} $\g_{\hat\nu}$ from the beginning of this section.

It remains to prove that $\g^{\hat\nu}$ is a complex \BKMa{}. This amounts to showing that item~\ref{enum:bkma6} of Theorem~\ref{thm:bkma} holds, either directly or by applying Proposition~\ref{prop:bor95thm2}.
\begin{oframed}
\begin{customconj}{2}\label{conj:2}
Let $\hat\nu$ be as in Assumption~\ref{ass:m}. Then the Lie algebra $\g^{\hat\nu}$ is a complex \BKMa{}.
\end{customconj}
\end{oframed}

Assuming that this conjecture is true we show in the following that the \BKMa{} $\g^{\hat\nu}=H^1_\text{BRST}(M)$ is isomorphic to the complexification of $\g_{\hat\nu}$, the real \BKMa{} obtained by twisting the denominator identity of the Fake Monster Lie algebra. To this end we need to determine the roots of $\g^{\hat\nu}$, which is graded by the root lattice $\Delta=\Lambda^\nu\oplus \II_{1,1}$ of level $m$.

The real roots of $\g^{\hat\nu}$, i.e.\ the roots $\alpha\in\Delta$ with $\langle\alpha,\alpha\rangle>0$ (and $\dim_\C(\g^{\hat\nu}(\alpha))>0$) can be easily read off from the dimension formula in Corollary~\ref{cor:dimform} or the rescaled version in Proposition~\ref{prop:rescaleddim}.
\begin{prop}\label{prop:realroots}
Let Assumption~\ref{ass:mp} and Conjectures \ref{conj:1} and \ref{conj:2} hold. Then the real roots of $\g^{\hat\nu}$ are the $\alpha\in\Delta\cap d\Delta'$ with $\langle\alpha,\alpha\rangle/2=d$ for $d\mid m$, all with root multiplicity $\dim_\C(\g^{\hat\nu}(\alpha))=1$. Moreover, the real roots of $\g^{\hat\nu}$ are exactly the roots of the lattice $\Delta$.
\end{prop}
\begin{proof}
Let $\alpha\in\Delta$ such that $\langle\alpha,\alpha\rangle>0$. Then using that
\begin{equation*}
\frac{1}{\eta_\nu(\tau)}=\prod_{t\mid m}\eta(t\tau)^{-24/\sigma_1(m)}=\frac{1}{q_\tau}+\frac{24}{\sigma_1(m)}+\ldots
\end{equation*}
we obtain
\begin{align*}
\dim_\C(\g^{\hat\nu}(\alpha))&=\sum_{d\mid m}\delta_{\alpha\in \Delta\cap d\Delta'}\left[\frac{1}{\eta_\nu}\right]\left(-\frac{1}{d}\frac{\langle\alpha,\alpha\rangle}{2}\right)\\
&=\sum_{d\mid m}\delta_{\alpha\in \Delta\cap d\Delta'}\left[q^{-1}+\ldots\right]\left(-\frac{1}{d}\frac{\langle\alpha,\alpha\rangle}{2}\right)=\sum_{d\mid m}\delta_{\alpha\in \Delta\cap d\Delta'}\delta_{d,\langle\alpha,\alpha\rangle/2},
\end{align*}
which proves the first claim. The second claim follows directly from Propositions 2.1 and 2.2 in \cite{Sch06}.
\end{proof}
The Weyl group $W\leq\Aut(\Delta)$ of $\g^{\hat\nu}$ is by definition the group generated by the reflections through the hyperplanes orthogonal to the real roots of $\g^{\hat\nu}$ and hence in this case it is the full reflection group of the lattice $\Delta$, i.e.\ the group generated by the reflections through the hyperplanes orthogonal to the roots of $\Delta$.

Hence, the real simple roots
of $\g^{\hat\nu}$ are exactly the simple roots of the reflection group $W$ of $\Delta$.

The following proposition shows that the reflection group $W$ has a Weyl vector, i.e.\ a vector $\rho\in\Delta\otimes_\Z\R$ such that a set of simple roots of $W$ is given by the roots $\alpha\in\Delta$ satisfying $\langle\alpha,\rho\rangle=-\langle\alpha,\alpha\rangle/2$.
\begin{prop}
Let $\hat\nu$ be as in Assumption~\ref{ass:m}. Then there exists a primitive norm-zero vector $\rho\in\Delta$ which is a Weyl vector for the reflection group $W$ of $\Delta$.
\end{prop}
\begin{proof}
The Lorentzian lattice $\Delta$ is given by the direct sum $\Delta=\Lambda^\nu\oplus \II_{1,1}$. As remarked earlier, the even lattice $\Lambda^\nu$ has no roots. This allows us to apply Theorem~3.3 in \cite{Bor90b}. It states that there is a norm-zero vector $\rho\in\Delta$ such that the simple roots of the reflection group $W$ of $\Delta$ are exactly
the roots $\alpha$ of $\Delta$ such that $\langle\alpha,\rho\rangle$ is negative and divides $\langle\alpha,v\rangle$ for all vectors $v\in\Delta$. We will show that $\rho$ is a Weyl vector, i.e.\ that the simple roots determined in \cite{Bor90b} are exactly the roots $\alpha$ satisfying $\langle\alpha,\rho\rangle=-\langle\alpha,\alpha\rangle/2$. Recall that for any root $\alpha$ the quotient $2\langle\alpha,v\rangle/\langle\alpha,\alpha\rangle\in\Z$ for all $v\in L$.

Now let $\alpha$ be a root with $\langle\alpha,\rho\rangle=-\langle\alpha,\alpha\rangle/2$. Then $\langle\alpha,\rho\rangle$ is negative since $\langle\alpha,\alpha\rangle/2>0$ by the definition of root and
\begin{equation*}
\frac{\langle\alpha,v\rangle}{\langle\alpha,\rho\rangle}=-\frac{2\langle\alpha,v\rangle}{\langle\alpha,\alpha\rangle}\in\Z
\end{equation*}
by the above property, which shows that $\alpha$ is a simple root.

Conversely, let a simple root $\alpha$ be given. For any root $\alpha$ the quotient $2\langle\alpha,\rho\rangle/\langle\alpha,\alpha\rangle$ is in $\Z$. Since $\alpha$ is simple, also $\langle\alpha,\alpha\rangle/\langle\alpha,\rho\rangle\in\Z$ by definition. Since $\alpha$ has positive norm, this leaves only
\begin{equation*}
\langle\alpha,\rho\rangle=-\frac{\langle\alpha,\alpha\rangle}{2}\quad\text{and}\quad\langle\alpha,\rho\rangle=-\langle\alpha,\alpha\rangle.
\end{equation*}
In general, in the situation of Theorem~3.3 in \cite{Bor90b}, both cases can occur. However, the second case is only possible if the discriminant form $\Delta'/\Delta$ contains an odd 2-adic component in its Jordan decomposition, which is not the case for the ten lattices at hand.

Indeed, let $d:=\langle\alpha,\alpha\rangle/2$ for some $d\in\Z$. Then, by Proposition~2.1 in \cite{Sch06}, $\alpha\in\Delta\cap d\Delta'$ and $d$ divides the level of $\Delta$. Then $-\langle\alpha,\rho\rangle=\langle\alpha,\alpha\rangle=2d$ divides $\langle\alpha,v\rangle$ for all $v\in\Delta$, i.e.\ $\alpha\in 2d\Delta'$. Consider $\gamma:=\alpha/(2d)+\Delta\in\Delta'/\Delta$ in the discriminant form. Then $2d\cdot\gamma=0$ and $Q_{\Delta}(\gamma)=1/(4d)+\Z$. This can only occur if $\Delta'/\Delta$ contains an odd 2-adic Jordan component.

Finally, any vector $\rho\in\Delta$ such that for a simple root $\alpha$, $\langle\alpha,\rho\rangle$ divides $\langle\alpha,v\rangle$ for all vectors $v\in\Delta$ is clearly primitive.
\end{proof}
\begin{rem}
A possible choice of Weyl vector is given by $\rho=(0,\eta)\in\Lambda^\nu\oplus\II_{1,1}$ for any primitive norm-zero vector in $\eta\in\II_{1,1}$ (cf.\ \cite{CKS07}, directly before Theorem~6.2).
\end{rem}

We fix a Weyl vector $\rho$ as in the remark. This fixes a set of simple roots of $W$ and also the fundamental Weyl chamber, which is the set of vectors in $\Delta\otimes_\Z\R$ with non-positive\footnote{This definition of the fundamental Weyl chamber is opposite to the usual one but this sign convention is advantageous as explained e.g.\ in \cite{Bor92}, p.\ 420.} inner product with the simple roots. The Weyl vector $\rho$ lies in the fundamental Weyl chamber. We obtain:
\begin{prop}
Let Assumption~\ref{ass:mp} and Conjectures \ref{conj:1} and \ref{conj:2} hold. Then the real simple roots of
$\g^{\hat\nu}$ are the $\alpha\in\Delta\cap d\Delta'$ with $\langle\alpha,\alpha\rangle/2=d$ for $d\mid m$ and $\langle\rho,\alpha\rangle=-\langle\alpha,\alpha\rangle/2$. These are precisely the simple roots of the reflection group $W$ of $\Delta$.
\end{prop}

We then determine the imaginary simple roots of $\g^{\hat\nu}$.
\begin{prop}
Let Assumption~\ref{ass:mp} and Conjectures \ref{conj:1} and \ref{conj:2} hold. Then the positive multiples $n\rho$ of the Weyl vector $\rho$ are imaginary simple roots of $\g^{\hat\nu}$ with multiplicity $24\sigma_0((m,n))/\sigma_1(m)$.
\end{prop}
\begin{proof}
Proceeding as in \cite{HS03}, Proposition~3.6, or \cite{HS14}, Proposition~3.7, we obtain that all the positive multiples $n\rho$ of the Weyl vector are imaginary simple roots. The multiplicity of these roots is given by
\begin{align*}
\dim_\C(\g^{\hat\nu}(n\rho))&=\sum_{d\mid m}\delta_{n\rho\in\Delta\cap d\Delta'}\left[\frac{1}{\eta_\nu}\right](0)\\
&=\frac{24}{\sigma_1(m)}\sum_{d\mid m}\delta_{n\rho\in\Delta\cap d\Delta'},
\end{align*}
using Proposition~\ref{prop:rescaleddim}. Since the Weyl vector $\rho=(0,\eta)$ is primitive in $\Delta=\Lambda^\nu\oplus \II_{1,1}$, we obtain that $n\rho\in\Delta\cap d\Delta'$ if and only if $d\mid n$ and hence
\begin{equation*}
\dim_\C(\g^{\hat\nu}(n\rho))=\frac{24}{\sigma_1(m)}\sum_{d\mid m}\delta_{d\mid n}=\frac{24\sigma_0((m,n))}{\sigma_1(m)},
\end{equation*}
which completes the proof.
\end{proof}
We will see that these are already all the imaginary simple roots of $\g^{\hat\nu}$.

\begin{oframed}
\begin{thm}\label{thm:10brst}
Let Assumption~\ref{ass:mp} and Conjectures \ref{conj:1} and \ref{conj:2} hold. Then $\g^{\hat\nu}$ is isomorphic to the complexification of $\g_{\hat\nu}$.
\end{thm}
\end{oframed}
\begin{proof}
Consider the \BKMa{} $\g_{\hat\nu}$ obtained by twisting the denominator identity of the Fake Monster Lie algebra $\g$. Its denominator identity is given by
\begin{equation*}
\ee^\rho\prod_{d\mid m}\prod_{\alpha\in\Phi^+\cap d\Delta'}(1-\ee^\alpha)^{[1/\eta_\nu](-\langle\alpha,\alpha\rangle/2d)}=\sum_{w\in W}\det(w)w(\eta_\nu(\ee^\rho)).
\end{equation*}
Hence $\g^{\hat\nu}$ and $\g_{\hat\nu}$ have the same root multiplicities. Once a Cartan subalgebra and a fundamental Weyl chamber are fixed, the root multiplicities of a \BKMa{} determine the simple roots via the denominator identity (cf.\ \cite{Bor92}, proof of Theorem~7.2). It follows that $\g^{\hat\nu}$ and $\g_{\hat\nu}$ have the same simple roots and hence are isomorphic upon complexifying $\g_{\hat\nu}$.
\end{proof}
This means that modulo the two conjectures we have found a systematic, natural construction of the ten \BKMa{}s in \cite{Sch04b,Sch06} whose denominator identities are completely reflective automorphic products of singular weight.

As a corollary we obtain:
\begin{cor}
Let Assumption~\ref{ass:mp} and Conjectures \ref{conj:1} and \ref{conj:2} hold. Then the denominator identity of $\g^{\hat\nu}$ is
\begin{equation*}
\ee^\rho\prod_{d\mid m}\prod_{\alpha\in\Phi^+\cap d\Delta'}(1-\ee^\alpha)^{[1/\eta_\nu](-\langle\alpha,\alpha\rangle/2d)}=\sum_{w\in W}\det(w)w(\eta_\nu(\ee^\rho)).
\end{equation*}
Moreover, a set of simple roots of $\g^{\hat\nu}$ is as follows: the real simple roots of $\g^{\hat\nu}$ are the $\alpha\in\Delta\cap d\Delta'$ with $\langle\alpha,\alpha\rangle/2=d$ for $d\mid m$ and $\langle\rho,\alpha\rangle=-\langle\alpha,\alpha\rangle/2$ with multiplicity 1 and the imaginary simple roots are the positive multiples $n\rho$ of the Weyl vector $\rho$ with multiplicity $24\sigma_0((m,n))/\sigma_1(m)$.
\end{cor}

A summary of this section, which provides further examples for the wondrous connection between Lie algebras, vertex algebras and automorphic forms, is depicted in the following diagram:
\begin{equation*}
\begin{tikzcd}
\text{VA $M$}&&\text{BKMA}&&\text{Aut. Prod.}\\
V_\Lambda\otimes V_{\II_{1,1}}\cong V_{\II_{25,1}}\arrow{dd}{\hat{\nu}}\arrow{rr}{\text{BRST}}\arrow[bend right]{rrrr}{\text{lift of char.}}&&\text{FMA }\g\arrow[<->]{rr}{\text{den. id.}}&&\Psi\arrow{dd}{\hat{\nu}}\\
\\
\bigoplus_{\gamma\in K'/K}V_\Lambda^{\hat{\nu}}(\gamma+K)\otimes V_{\gamma+K}\arrow{rr}{\text{BRST}}\arrow[bend right]{rrrr}{\text{lift of char.}}&&\g^{\hat\nu}\cong\g_{\hat\nu}\arrow[<->]{rr}{\text{den. id.}}&&\Psi_{\hat\nu}
\end{tikzcd}
\end{equation*}

\bookmarksetupnext{level=part}
\begin{appendices}

\chapter{\FQS{}s}\label{ch:fqs}

\Fqs{}s are finite abelian groups equipped with a non-degenerate quadratic form and occur naturally as discriminant forms of lattices or fusion groups of \voa{}s whose irreducible modules are all simple currents.

\section{Finite Quadratic and Bilinear Forms}\label{sec:qf}

We begin with the definition of finite bilinear and quadratic forms:
\begin{defi}[Finite Bilinear Form]
Let $D$ be a finite abelian group. A map $B\colon D\times D\to \Q/\Z$ is called \emph{finite bilinear form} if it is a symmetric $\Z$-bilinear form on $D$, i.e.\
\begin{enumerate}
\item\label{enum:bil1} $B(a\gamma+b\delta,\beta)=aB(\gamma,\beta)+bB(\delta,\beta)$ for all $a,b\in\Z$ and $\beta,\gamma,\delta\in D$,
\item\label{enum:bil2} $B(\gamma,\delta)=B(\delta,\gamma)$ for all $\gamma,\delta\in D$.
\end{enumerate}
\end{defi}
The symmetry is explicitly included in the definition of a finite bilinear form.
\begin{defi}[Finite Quadratic Form]
Let $D$ be a finite abelian group. A map $Q\colon D\to\Q/\Z$ is called \emph{finite quadratic form} if
\begin{enumerate}
\item\label{enum:qf1} $Q(a\gamma)=a^2Q(\gamma)$ for all $a\in\Z$ and $\gamma\in D$,
\item\label{enum:qf2} the map $B_Q\colon D\times D\to \Q/\Z$ defined by $B_Q(\gamma,\delta):=Q(\gamma+\delta)-Q(\gamma)-Q(\delta)$ is a finite bilinear form (and is called the \emph{associated bilinear form}).
\end{enumerate}
\end{defi}

\begin{rem}\label{rem:bilquad}
Given a finite quadratic form $Q$ on $D$ we can form the associated bilinear form $B_Q$. On the other hand, given a finite bilinear form $B$ we obtain a finite quadratic form via $\gamma\mapsto B(\gamma,\gamma)$ for all $\gamma\in D$. Unfortunately, these processes are not inverse. In fact, the composition of these two operations is multiplication by 2 either on the finite quadratic forms or on the finite bilinear forms depending on the order of the composition.

This has the following implication. By definition, every finite quadratic form $Q$ has a unique associated bilinear form $B_Q$. On the other hand, given a finite bilinear form $B$, there are $|D/2D|$-many quadratic forms $Q$ with $B_Q=B$.
\end{rem}

\begin{rem}
Via the map $\e^{(2\pi\i)(\cdot)}:\Q/\Z\stackrel{\cong}{\longrightarrow}\mathbb{T}_T\leq \C^\times$ we can view finite bilinear and quadratic forms as functions with values in $\mathbb{T}_T$, the torsion subgroup of the circle group $\mathbb{T}=\{z\in\C\;|\;|z|=1\}$, i.e.\ the multiplicative group of all $n$-th roots of unity for all $n\in\Ns$. For that purpose let us introduce the following notation: for a quadratic form $Q\colon D\to\Q/\Z$ or a bilinear form $B\colon D\times D\to\Q/\Z$ we write
\begin{align*}
q(\alpha)&:=\e^{(2\pi\i)Q(\alpha)},\\
b(\alpha,\beta)&:=\e^{(2\pi\i)B(\alpha,\beta)},
\end{align*}
$\alpha,\beta\in D$, for the corresponding functions $Q\colon D\to\C^\times$ and $B\colon D\times D\to\C^\times$.
\end{rem}

We will make use of the following result:
\begin{prop}\label{prop:bilquad}
Let $D$ be a finite abelian group equipped with a finite bilinear form $B\colon D\times D\to\Q/\Z$ and assume that there is a function $Q\colon D\to\Q/\Z$ such that $B(\gamma,\delta)=Q(\gamma+\delta)-Q(\gamma)-Q(\delta)$ for all $\gamma,\delta\in D$. Then, if $Q$ fulfils $Q(\gamma)=Q(-\gamma)$, $Q$ is already a finite quadratic form.
\end{prop}
\begin{proof}
We only need to show property \ref{enum:qf1} in the definition of quadratic form since property \ref{enum:qf2} is fulfilled by assumption knowing that $B=B_Q$ is bilinear. First note that $0+\Z=B(0,0)=Q(0+0)-Q(0)-Q(0)=Q(0)$. We show \ref{enum:qf1} by induction. Knowing that $Q(-\gamma)=Q(\gamma)$ we only need to show it for $a\in\N$. For $a=0$ it is true because of $Q(0)=0+\Z$ and for $a=1$ the statement is trivial. The induction step is from $a=n-2$ and $a=n-1$ to $a=n$ and goes as follows:
\begin{align*}
0+\Z&=B((n-1)\gamma,\gamma)+B((n-1)\gamma,-\gamma)\\
&=Q(n\gamma)-Q((n-1)\gamma)-Q(\gamma)+Q((n-2)\gamma)-Q((n-1)\gamma)-Q(-\gamma)\\
&=\ldots=Q(n\gamma)-n^2Q(\gamma).
\end{align*}
This completes the proof.
\end{proof}

\minisec{Orthogonal Complement and Isotropic Subgroups}

Since quadratic forms are assumed to be symmetric, there is a natural notion of orthogonality:
\begin{defi}[Orthogonal Complement]
Let $D$ be a finite abelian group with a finite quadratic form $Q\colon D\to\Q/\Z$ (and associated bilinear form $B_Q\colon D\times D\to\Q/\Z$). Two elements $\gamma,\delta\in D$ are called \emph{orthogonal} if $B_Q(\gamma,\delta)=0+\Z$.

Given a subset $S\subseteq D$, the \emph{orthogonal complement} $S^\bot$ of $S$ is the subgroup of $D$ defined by
\begin{equation*}
S^\bot:=\{\gamma\in D\;|\;B_Q(\gamma,\mu)=0+\Z\text{ for all }\mu\in S\}.
\end{equation*}
\end{defi}

\begin{defi}[Non-Degeneracy, \FQS{}]\label{defi:fqs}
Let $D$ be a finite abelian group and let $Q$ be a finite quadratic form on $D$. The quadratic form is called \emph{non-degenerate} if the associated bilinear form $B_Q$ is non-degenerate, i.e.\ if $D^\bot=\{0\}$.

A finite abelian group together with a non-degenerate quadratic form is called \emph{\fqs{}}.
\end{defi}

\begin{defi}
Let $D=(D,Q)$ be a finite abelian group $D$ with a finite quadratic form $Q$. By $\overline{D}:=(D,-Q)$ we denote the same abelian group with the negative of the quadratic form $Q$. If $D$ is a \fqs{}, then so is $\overline{D}$.
\end{defi}

\begin{prop}\label{prop:aad}
If the finite quadratic form on the finite abelian group $D$ is non-degenerate, then $|A||A^\bot|=|D|$ for any subgroup $A\leq D$.
\end{prop}

We also need the definition of an isotropic subgroup.
\begin{defi}[Isotropic Subgroup]
Let $D$ be a finite abelian group with a finite quadratic form $Q\colon D\to\Q/\Z$. A subgroup $I$ of $D$ is called \emph{isotropic} if $Q(\gamma)=0+\Z$ for all $\gamma\in I$. $I$ is called \emph{maximal isotropic} if $I$ is isotropic and not a proper subgroup of an isotropic subgroup.
\end{defi}

\begin{rem}\label{rem:maxiso}
Let $D$ be a finite abelian group with a finite quadratic form $Q$. Then:
\begin{enumerate}
\item\label{enum:maxiso1} For an isotropic subgroup $I$ of $D$, $I\subseteq I^\bot$. The converse is false: $I\subseteq I^\bot$ implies that the bilinear form $B_Q$ vanishes on $I$ but due to the $|D/2D|$-fold ambiguity, the quadratic form $Q$ is not necessarily trivial on $I$.
\item\label{enum:maxiso2} Given an isotropic subgroup $I$ of $D$ with $I=I^\bot$, it is easy to see that $I$ is a maximal isotropic subgroup. Again, the converse is not quite true.
\end{enumerate}
\end{rem}

\section{A Classification Result}

In the following we show that a \fqs{} is determined up to isomorphism by the multiset of values of the quadratic form. This result is probably known but we include a proof for completeness. The idea of the proof is from Gerald Höhn. For the result to hold it is essential that the quadratic form is non-degenerate. We make use of the classification of \fqs{}s treated in \cite{CS99}, Section~15.7.
\begin{prop}\label{prop:valuesfqs}
Let $D$ be a finite set and $Q\colon D\to\Q/\Z$ some function. Assume that the largest power of 2 in the denominators of the values of $Q$ is 2. Then there is at most one \fqs{} up to isomorphism with $D$ as underlying set of the abelian group and $Q$ as non-degenerate quadratic form on $D$.
\end{prop}
\begin{proof}
Suppose that there is an abelian group law on $D$ such that $(D,Q)$ is a \fqs{}. We want to prove that this \fqs{} is unique up to isomorphism. We may assume that $|D|>1$ and decompose the order of the set $D$ as
\begin{equation*}
|D|=p_1^{r_1}\ldots p_k^{r_k}
\end{equation*}
where the $p_i$ are pairwise distinct primes and $r_i\in\Ns$, $i=1,\ldots,k$, for some $k\in\Ns$. Any abelian group structure on $D$ can be uniquely written as the direct product
\begin{equation*}
D=D_1\times\ldots\times D_k
\end{equation*}
where $D_i$ is a group of order $p_i^{r_i}$. The same is true for \fqs{}s.\footnote{Some caution should be exercised. In general, the decomposition of a \fqs{} into a direct product of indecomposable components is not unique. For example, the product of the Jordan components $5^1$ and $5^1$ is isomorphic to that of the components $5^{-1}$ and $5^{-1}$. However, if we decompose the \fqs{} into components with prime power orders for distinct primes as above, then this decomposition is unique.}
Hence it suffices to study the case of a set $D$ of prime power order.

First, let $p$ be an odd prime and suppose that $|D|=p^r$ for some $r\in\Ns$. In the odd-prime case, the decomposition of a finite abelian group and that of a \fqs{} of order $p^r$ into cyclic components is unique. We collect all the cyclic components of a certain prime power so that
\begin{equation*}
D\cong \Z_p^{l_1}\times\Z_{p^2}^{l_2}\times\ldots\times\Z_{p^s}^{l_s}
\end{equation*}
for some $s\in\Ns$ and $l_j\in\N$ for $j=1,\ldots,s$ with $l_1+2l_2+\ldots+sl_s=r$. On each $\Z_{p^j}^{l_j}$ there are two possible finite quadratic forms up to isomorphism denoted by the Jordan components $(p^j)^{+l_j}$ and $(p^j)^{-l_j}$. They may be distinguished, for instance, by the number of elements whose value under the quadratic form has the maximal denominator $p^j$. For example, the Jordan component $25^{+2}$ has 400 such elements while $25^{-2}$ has 600.

Now suppose that the function $Q$ defines a quadratic form on the set $D$ for some group law we do not yet know. The denominators of the values of $Q$ have to be powers of $p$. Let $q$ be the largest such power. Then the decomposition of $D$ has to contain a maximal Jordan component of the form $q^{\pm l}$ for some $l\in\Ns$. Since all the other components in the decomposition of $D$ have smaller denominators, the \emph{relative} amount of elements whose value under $Q$ has the maximal denominator $q$ is the same in $q^{\pm l}$ and all of $D$. Hence, the sign in the exponent of $q^{\pm l}$ can be read off from the values of $Q$ on $D$.

Splitting off the component $q^{\pm l}$ and successively treating smaller denominators we obtain that the structure of $D$ as \fqs{} is unique up to isomorphism.

The case of the even prime $p=2$ is in general more complicated. Let $|D|=2^r$ for some $r\in\Ns$. The decomposition of a \fqs{} of order $2^r$ into cyclic components is usually not unique. For example, $2_1^{+1}\times4_1^{+1}\cong2_3^{-1}\times4_3^{-1}$, which involves odd 2-adic Jordan components. However, since we assumed that the largest power of 2 in the denominators of the values of $Q$ is 2, only the even Jordan components $2_{\II}^{\pm l}$, $l\in2\Ns$, may appear. In this case, we may proceed as in the odd-prime case and obtain the desired result.
\end{proof}
The divisibility assumption in the proposition on the denominators of the values of $Q$ can most probably be dropped. Then the case of $p=2$ in the proof has to be investigated more thoroughly.

\section{Weil Representation}\label{sec:weil}

In this section we introduce a representation of $\SLZ$ or its double cover $\MpZ$, called Weil representation since it is a special case of a construction due to Weil \cite{Wei64}.

\minisec{The Metaplectic Group $\MpZ$}
In the following we define the metaplectic group $\MpZ$, a double cover of $\SLZ$. For $z\in\C$ let us by $\sqrt{z}$ denote the principal branch of the complex square root, so that $\arg(\sqrt{z})\in(-\pi/2,\pi/2]$. Let $\MpZ$ be the double cover of $\SLZ$ consisting of pairs
\begin{equation*}
(M,\varphi(\tau))
\end{equation*}
where $M=\left(\begin{smallmatrix}a&b\\c&d\end{smallmatrix}\right)\in\SLZ$ and $\varphi$ is a holomorphic function on $\H$ with $\varphi(\tau)^2=c\tau+d$, i.e.\ $\varphi(\tau)=\pm\sqrt{c\tau+d}$. The group law on $\MpZ$ is given by
\begin{equation*}
(M_1,\varphi_1(\tau))(M_2,\varphi_2(\tau))=(M_1M_2,\varphi_1(M_2.\tau)\varphi_2(\tau))
\end{equation*}
where $M.\tau=(a\tau+b)/(c\tau+d)$ denotes the usual action of $\SLZ$ on the upper half-plane $\H$, i.e.\ the Möbius transformation. There is the canonical embedding of $\SLZ$ into $\MpZ$ given by
\begin{equation*}
M\mapsto\widetilde{M}=(M,\sqrt{c\tau+d}).
\end{equation*}
The group $\MpZ$ is generated by the elements
\begin{equation*}
\widetilde{S}=(S,\sqrt{\tau})\quad\text{and}\quad\widetilde{T}=(T,1)
\end{equation*}
where $S=\left(\begin{smallmatrix}0&-1\\1&0\end{smallmatrix}\right)$ and $T=\left(\begin{smallmatrix}1&1\\0&1\end{smallmatrix}\right)$ are the standard generators of $\SLZ$. These generators obey the relations
\begin{equation*}
\widetilde{S}^2=(\widetilde{S}\widetilde{T})^3=\widetilde{Z}\quad\text{and}\quad\widetilde{Z}^4=\id.
\end{equation*}
Note that
\begin{equation*}
\widetilde{Z}=(Z,\i)\quad\text{and}\quad\widetilde{Z}^2=(\id,-1)
\end{equation*}
where $Z=S^2=-\id$. (For comparison: in $\SLZ$ the relations $S^2=(ST)^3=Z$ and $Z^2=\id$ hold.)

\minisec{The Weil Representation}

Let $D$ be a \fqs{} with quadratic form $Q$ and associated bilinear form $B_Q$. Let $\C[D]$ be the group algebra of $D$. The Weil representation $\rho_D$ of $\MpZ$ on $\C[D]$ is defined via
\begin{align*}
\rho_D(\widetilde{S})_{\alpha,\beta}&=\frac{1}{\sqrt{|D|}}\e^{(2\pi\i)(-B_Q(\alpha,\beta)-\sign(D)/8)},\\
\rho_D(\widetilde{T})_{\alpha,\beta}&=\delta_{\alpha,\beta}\e^{(2\pi\i)Q(\alpha)},
\end{align*}
$\alpha,\beta\in D$ (see e.g.\ \cite{Bor98}). These equations define a unitary representation on $\C[D]$.

Note that $\rho_D(\widetilde{Z})=\e^{(2\pi\i)(-\sign(D)/4)}$ and hence $\rho_D(\widetilde{Z}^2)=\e^{(2\pi\i)(-\sign(D)/2)}$. On the other hand, the two different pullbacks of $M\in\SLZ$ to $\MpZ$ are $\widetilde{M}$ and $\widetilde{M}\widetilde{Z}^2$. Then $\rho_D(\widetilde{M}\widetilde{Z}^2)=\e^{(2\pi\i)(-\sign(D)/2)}\rho_D(\widetilde{M})$. This means that if $\sign(D)$ is even, $\rho_D$ descends to a representation of $\SLZ$ while if $\sign(D)$ is odd, we only obtain a projective representation of $\SLZ$.

\chapter{Lie Algebras}

Lie algebras occur at various places in the theory of vertex algebras. In this chapter we collect a few results on Lie algebras needed for this text.

\section{Vertex Algebras and Lie Algebras}\label{sec:voalie}
The results in this section are known but we give proofs for completeness. For example, it is a well-known fact that for a given \voa{} $V$ of CFT-type, the space $V_1$ can be endowed with the structure of a Lie algebra.
\begin{prop}
Let $V$ be a \voa{} of CFT-type. Then the weight-one subspace $V_1$ of $V$ carries the natural structure of a finite-dimensional, complex Lie algebra given by
\begin{equation*}
[u,v]:=u_0v
\end{equation*}
for $u,v\in V_1$.
\end{prop}
\begin{proof}
For $u,v\in V_1$, clearly also $u_0v\in V_1$. The bilinearity of the Lie bracket is clear. We next show the antisymmetry of the Lie bracket. We need the skew-symmetry property, which holds in any \voa{} (see (2.3.19) in \cite{FHL93}):
\begin{equation*}
Y(u,x)v=\e^{xL_{-1}}Y(v,-x)u.
\end{equation*}
for all $u,v\in V$, i.e.\
\begin{equation*}
\sum_{n\in\Z}u_nx^{-n-1}v=\e^{xL_{-1}}\sum_{n\in\Z}v_n(-x)^{-n-1}u.
\end{equation*}
Taking the coefficient of $x^{-1}$ on both sides gives
\begin{equation*}
u_0v=-\sum_{n=0}^\infty (-1)^n\frac{(L_{-1})^n}{n!}v_nu.
\end{equation*}
Now assume that $u,v\in V_1$. Then $v_nu\in V_{1-n}$ and hence $v_nu=0$ for $n\geq2$ since $V$ is of CFT-type. Also $v_1u\in\C\vac$ and by the translation axiom $L_{-1}v_1u=0$. Hence all terms in the sum except for $n=0$ vanish and we get
\begin{equation*}
u_0v=-v_0u.
\end{equation*}

Finally, we show that the Jacobi identity for Lie algebras holds. Not surprisingly, this follows from the Jacobi identity for vertex algebras. In fact, it is more convenient to use the equivalent Borcherds identity, which yields
\begin{equation*}
[u_0,v_0]=(u_0v)_0
\end{equation*}
as a special case (see (2.3.16) in \cite{FHL93}) where $[\cdot,\cdot]$ denotes the commutator. Hence
\begin{align*}
[u,[v,w]]&=u_0v_0w=[u_0,v_0]w+v_0u_0w=(u_0v)_0w+[v,[u,w]]\\
&=[[u,v],w]-[v,[w,u]]=-[w,[u,v]]-[v,[w,u]]
\end{align*}
for all $u,v,w\in V_1$, where we used the antisymmetry of the Lie bracket.
\end{proof}

\begin{prop}
Let $V$ be a \voa{} of CFT-type. Then the weight spaces $V_n$, $n\in\N$, carry the natural structure of Lie algebra modules for $V_1$ via
\begin{equation*}
v\cdot w:=v_0w
\end{equation*}
for $v\in V_1$ and $w\in V_n$.
\end{prop}
\begin{proof}
It is clear that $\wt(v_0w)=\wt(v)-0-1+\wt(w)=\wt(w)$ for $v\in V_1$. Hence $v\cdot w\in V_n$ for $v\in V_1$ and $w\in V_n$. That $V_n$ is a $V_1$-module follows from $[u_0,v_0]=(u_0v)_0$ for all $u,v\in V$. Indeed:
\begin{equation*}
[u,v]\cdot w=(u_0v)\cdot w=(u_0v)_0w=[u_0,v_0]w=u\cdot (v\cdot w)-v\cdot (u\cdot w)
\end{equation*}
for all $u,v\in V_1$ and $w\in V_n$.
\end{proof}

We can further generalise the above result:
\begin{prop}
Let $V$ be a \voa{} of CFT-type and $W$ a $V$-module. Then all the weight spaces $W_\lambda$, $\lambda\in\C$, naturally become Lie algebra modules for the Lie algebra $V_1$ via
\begin{equation*}
v\cdot w:=v_0 w
\end{equation*}
for $v\in V_1$, $w\in W_\lambda$ where $v_0\in\End_\C(W)$ denotes the $0$-th mode of $Y_W(v,x)$.
\end{prop}
\begin{proof}
Also for the $0$-th modes of $Y_W(\cdot,x)$, $\wt(v_0w)=\wt(v)-0-1+\wt(w)=\wt(w)$ for $v\in V_1$ and
\begin{equation*}
[u_0,v_0]=(u_0v)_0
\end{equation*}
for all $u,v\in V$ (see (4.2.10) in \cite{FHL93}). Then, as above,
\begin{equation*}
[u,v]\cdot w=(u_0v)\cdot w=(u_0v)_0w=[u_0,v_0]w=u\cdot (v\cdot w)-v\cdot (u\cdot w)
\end{equation*}
for all $u,v\in V_1$ and $w\in W_\lambda$, which proves the claim.
\end{proof}

Recall that since an automorphism of a \voa{} fixes the Virasoro vector, it can be restricted to an automorphism of each $L_0$-component. We show that on $V_1$ it restricts to a Lie algebra automorphism (see Definition~\ref{defi:laaut} below):
\begin{prop}\label{prop:autvoalie}
Let $V$ be a \voa{} of CFT-type and let $g$ be an automorphism of $V$. Then $g$ restricts to a Lie algebra automorphism of $V_1$.
\end{prop}
\begin{proof}
To prove that $g$ is a Lie algebra automorphism we need to show that
\begin{equation*}
(gu)_0(gv)=[gu,gv]=g([u,v])=g(u_0v)
\end{equation*}
for all $u,v\in V_1$. By definition of \voa{} automorphism we know that
\begin{equation*}
gv_ng^{-1}=(gv)_n
\end{equation*}
for all $v\in V$. Hence
\begin{equation*}
(gu)_0(gv)=gu_0g^{-1}gv=gu_0v,
\end{equation*}
which completes the proof.
\end{proof}

\section{Lie Algebra Automorphisms and Fixed-Point Lie Subalgebras}\label{sec:kac}
The purpose of this section is to describe a beautiful theory by Kac on the possible fixed-point Lie subalgebras of a simple, finite-dimensional Lie algebra and its generalisation to the semisimple case. In the following, all Lie algebras will be over the field $\C$. The main sources for this section are \cite{Hum73,FH91,Kac90,Ser01}.

\minisec{Lie Algebra Automorphisms}
The main object in this section are Lie algebra automorphisms:
\begin{defi}[Lie Algebra Automorphism]\label{defi:laaut}
Let $\g$ be a Lie algebra.
A vector-space automorphism $\varphi\in\Aut_\C(\g)$ is called a \emph{Lie algebra automorphism} if
\begin{equation*}
[\varphi(x),\varphi(y)]=\varphi([x,y])
\end{equation*}
for all $x,y\in\g$. We denote by $\Aut(\g)$ the group of automorphisms of $\g$.
\end{defi}

\begin{defi}[Adjoint Endomorphism]
Let $\g$ be a Lie algebra and $x\in\g$. Then we can define the \emph{adjoint endomorphism} $\ad_x\colon\g\to\g$,
\begin{equation*}
\quad\ad_x(y):=[x,y]
\end{equation*}
for all $y\in\g$.
\end{defi}
Then the map $\ad\colon\g\to\End(\g)$, $x\mapsto\ad_x$ is a representation of the Lie algebra $\g$ over the vector space $\g$, called the \emph{adjoint representation}.

Let $\g$ be a Lie algebra and $x\in\g$ such that $\ad_x$ is nilpotent, i.e.\ $(\ad_x)^k=0$ for some $k\in\Ns$. Then the usual exponential power series
\begin{equation*}
\exp(\ad_x):=1+\ad_x+\ldots+\frac{(\ad_x)^{k-1}}{(k-1)!}
\end{equation*}
has only finitely many terms. One can show that $\exp(\ad_x)\in\Aut(\g)$.

\begin{defi}[Inner Automorphism Group, \cite{Hum73}, Section~2.3]
Let $\g$ be a Lie algebra. The subgroup of $\Aut(\g)$ generated by the automorphisms of the form $\exp(\ad_x)$ for some $x\in\g$ with $\ad_x$ nilpotent is called the \emph{inner automorphism group} of $\g$ and denoted by $\Inn(\g)$. Its elements are called \emph{inner automorphisms}.
\end{defi}
One can show that $\Inn(\g)$ is a normal subgroup of $\Aut(\g)$.
\begin{defi}[Outer Automorphism Group]
Let $\g$ be a Lie algebra. The quotient
\begin{equation*}
\Out(\g):=\Aut(\g)/\Inn(\g)
\end{equation*}
is called the \emph{outer automorphism group} of $\g$.
\end{defi}

\minisec{Simple Lie Algebras}

In the following we consider finite-dimensional, simple and later semisimple Lie algebras. First, we state the well-known classification result by Killing, Cartan and Dynkin.
\begin{thm}[Classification, \cite{Ser01}, Section~II.7]
Let $\g$ be a finite-dimensional, semisimple Lie algebra. Then $\g$ is a direct sum of (finite-dimensional) simple Lie algebras. Every finite-dimensional, simple Lie algebra is isomorphic to an element of the following four infinite families (labelled $A$, $B$, $C$ and $D$) or five exceptional cases:\footnote{One could start the enumeration of the infinite series and the $E$-series at lower indices but this would give a redundancy because of the following \emph{exceptional isomorphisms}: $A_1\cong B_1\cong C_1$, $C_2\cong B_2$, $D_3\cong A_3$, $E_4\cong A_4$, $E_5\cong D_5$. Moreover, sometimes the following Lie algebras are also considered but they are not simple (but rather semisimple, trivial or abelian): $A_0\cong B_0\cong C_0\cong D_0\cong 0$ (trivially), $D_1\cong\C$, $D_2\cong A_1^2$, $E_3\cong A_1A_2$.
}
\begin{itemize}
 \item $A_n=\sl(n+1)$ for $n\in\Z_{\geq1}$,
 \item $B_n=\so(2n+1)$ for $n\in\Z_{\geq2}$,
 \item $C_n=\sp(2n)$ for $n\in\Z_{\geq3}$,
 \item $D_n=\so(2n)$ for $n\in\Z_{\geq4}$,
 \item $E_6$, $E_7$, $E_8$, $F_4$, $G_2$.
\end{itemize}
\end{thm}
We write for example $A_5D_7$ for $A_5\oplus D_7$ and $A_4^3$ for $A_4\oplus A_4\oplus A_4$. We write $X_l$, $X=A,B,C,D,E,F,G$, for a generic simple Lie algebra where $l$ denotes the \emph{rank} of $X_l$, i.e.\ the dimension of a Cartan subalgebra $H$, which does not depend on the choice of $H$. The dimensions of the simple Lie algebras are given by:
\begin{equation*}
\begin{tabular}{c|l|c|r}
Lie alg.&Dim.&Lie alg.&Dim.\\\hline
$A_n$&$n(n+2)$ &$E_6$&$78$\\
$B_n$&$n(2n+1)$&$E_7$&$133$\\
$C_n$&$n(2n+1)$&$E_8$&$248$\\
$D_n$&$n(2n-1)$&$F_4$&$52$\\
     &         &$G_2$&$14$
\end{tabular}
\end{equation*}

The finite-dimensional, simple Lie algebras are classified by their root systems, which in turn are classified by their (directed) Dynkin diagrams. The connected Dynkin diagrams of the finite-dimensional, simple Lie algebras are shown in Figure~\ref{fig:dynkin1}.
\begin{figure}[p]
\centering
\begin{align*}
&A_l,\,l\geq 1 && 
\begin{tikzpicture}[start chain,node distance=2em,every path/.style={shorten >=4pt,shorten <=4pt},line width=0.5pt,baseline=-1ex]
\dnode{1}
\dnode{2}
\dydots
\dnode{l-1}
\dnode{l}
\end{tikzpicture}
\\
&B_l,\,l\geq 2 &&
\begin{tikzpicture}[start chain,node distance=2em,every path/.style={shorten >=4pt,shorten <=4pt},line width=0.5pt,baseline=-1ex]
\dnode{1}
\dnode{2}
\dydots
\dnode{n-1}
\dnodenj{n}
\path (chain-4) -- node[anchor=mid] {\(\Rightarrow\)} (chain-5);
\end{tikzpicture}
\\
&C_l,\,l\geq 3 &&
\begin{tikzpicture}[start chain,node distance=2em,every path/.style={shorten >=4pt,shorten <=4pt},line width=0.5pt,baseline=-1ex]
\dnode{1}
\dnode{2}
\dydots
\dnode{l-1}
\dnodenj{l}
\path (chain-4) -- node[anchor=mid] {\(\Leftarrow\)} (chain-5);
\end{tikzpicture}
\\
&D_l,\,l\geq4 &&
\begin{tikzpicture}[node distance=2em,every path/.style={shorten >=4pt,shorten <=4pt},line width=0.5pt,baseline=-1ex]
\begin{scope}[start chain]
\dnode{1}
\dnode{2}
\node[chj,draw=none] {\dots};
\dnode{l-2}
\dnode{l-1}
\end{scope}
\begin{scope}[start chain=br going above]
\chainin(chain-4);
\dnodebr{l}
\end{scope}
\end{tikzpicture}
\\
&E_6 &&
\begin{tikzpicture}[node distance=2em,every path/.style={shorten >=4pt,shorten <=4pt},line width=0.5pt,baseline=-1ex]
\begin{scope}[start chain]
\foreach \dyni in {1,...,5} {
\dnode{\dyni}
}
\end{scope}
\begin{scope}[start chain=br going above]
\chainin (chain-3);
\dnodebr{6}
\end{scope}
\end{tikzpicture}
\\
&E_7 &&
\begin{tikzpicture}[node distance=2em,every path/.style={shorten >=4pt,shorten <=4pt},line width=0.5pt,baseline=-1ex]
\begin{scope}[start chain]
\foreach \dyni in {1,...,6} {
\dnode{\dyni}
}
\end{scope}
\begin{scope}[start chain=br going above]
\chainin (chain-3);
\dnodebr{7}
\end{scope}
\end{tikzpicture}
\\
&E_8 &&
\begin{tikzpicture}[node distance=2em,every path/.style={shorten >=4pt,shorten <=4pt},line width=0.5pt,baseline=-1ex]
\begin{scope}[start chain]
\foreach \dyni in {1,...,7} {
\dnode{\dyni}
}
\end{scope}
\begin{scope}[start chain=br going above]
\chainin (chain-5);
\dnodebr{8}
\end{scope}
\end{tikzpicture}
\\
&F_4 &&
\begin{tikzpicture}[start chain,node distance=2em,every path/.style={shorten >=4pt,shorten <=4pt},line width=0.5pt,baseline=-1ex]
\dnode{1}
\dnode{2}
\dnodenj{3}
\dnode{4}
\path (chain-2) -- node[anchor=mid] {\(\Rightarrow\)} (chain-3);
\end{tikzpicture}
\\
&G_2 &&
\begin{tikzpicture}[start chain,node distance=2em,every path/.style={shorten >=4pt,shorten <=4pt},line width=0.5pt,baseline=-1ex]
\dnodenj{1}
\dnodenj{2}
\path (chain-1) -- node {\(\Rrightarrow\)} (chain-2);
\end{tikzpicture}
\end{align*}
\caption{The connected Dynkin diagrams. The index $l$ of $X_l$ equals the number of nodes. Figure based on \cite{pict:dynkin}.}
\label{fig:dynkin1}
\end{figure}

Finite-dimensional, semisimple Lie algebras correspond to not necessarily connected Dynkin diagrams where each connected component corresponds to a simple component of the Lie algebra and is one of those in Figure~\ref{fig:dynkin1}.

\minisec{Automorphisms of Simple Lie Algebras}
We now recall some facts about automorphism of simple and semisimple Lie algebras.
\begin{thm}[\cite{FH91}, Proposition~D.40]
Let $\g$ be a finite-dimensional, simple Lie algebra. Then the outer automorphism group $\Out(\g)=\Aut(\g)/\Inn(\g)$ is isomorphic to the automorphism group of the corresponding Dynkin diagram. 
\end{thm}
In fact, $\Aut(\g)$ is isomorphic to the semidirect product of $\Inn(\g)$ and $\Out(\g)$, i.e.\ the short exact sequence
\begin{equation*}
1\to\Inn(\g)\to\Aut(\g)\to\Out(\g)\to 1
\end{equation*}
is split.

With the above theorem and the classification result one can immediately determine the outer automorphism groups of all the finite-dimensional, simple Lie algebras. One obtains:
\begin{prop}
\item
\begin{itemize}
\item $\Out(A_1)\cong 1$, $\Out(A_n)\cong\Z_2$ for $n\geq 2$,
\item $\Out(B_n)\cong\Out(C_n)\cong 1$,
\item $\Out(D_4)\cong S_3$, $\Out(D_n)\cong\Z_2$ for $n\geq 5$,
\item $\Out(E_6)\cong\Z_2$, $\Out(E_7)\cong\Out(E_8)\cong 1$,
\item $\Out(F_4)\cong\Out(G_2)\cong 1$.
\end{itemize}
\end{prop}

We are interested in the Lie algebra structure of the fixed points under a Lie algebra automorphism. For simple Lie algebras this was studied by Kac.
\begin{prop}[\cite{Kac90}, Lemma~8.1~c]\label{prop:fixreductive}
Let $\g$ be a finite-dimensional, simple Lie algebra and $\varphi$ an automorphism of $\g$ of finite order. Then the fixed-point Lie subalgebra $\g^\varphi$ is reductive, i.e.\ isomorphic to the direct sum of a semisimple Lie algebra and an abelian one.
\end{prop}

Kac classified all finite-order automorphisms of finite-dimensional, simple Lie algebras and their fixed-point Lie subalgebras \cite{Kac69,Kac90}. As main ingredient for this classification one needs the \emph{affine Dynkin diagrams}, which classify Cartan matrices of affine Lie algebras.\footnote{The Dynkin diagrams in Figure~\ref{fig:dynkin1} corresponding to simple Lie algebras are also called \emph{finite} Dynkin diagrams to distinguish them from the affine ones.} The affine Dynkin diagrams are classified in \cite{Kac90}, Theorem~4.8. Each affine Dynkin diagram corresponds to a finite Dynkin diagram and the affine diagrams are denoted by $X_l^{(1)}$, $X_l^{(2)}$ or $X_l^{(3)}$, where $X_l$ is the corresponding finite diagram. The diagrams $X_l^{(1)}$ are called \emph{extended affine Dynkin diagrams}, $X_l^{(2)}$ and $X_l^{(3)}$ are called \emph{twisted affine Dynkin diagrams}. The $X_l^{(k)}$ for $k=2,3$ are called twisted diagrams since they each correspond to the diagram automorphism of the finite diagram $X_l$ (outer automorphism of the simple Lie algebra $X_l$) of order $k$. All the affine Dynkin diagrams are shown in Figures~\ref{fig:dynkin3} and \ref{fig:dynkin4} (where the diagram $A_3^{(2)}$ is for convenience called $D_3^{(2)}$). Note that each node in the affine diagrams has a \emph{Kac label} $a_i$ which plays a rôle in the following theorem. (See Remark~8.3 and introduction of Section~8.6 in \cite{Kac90} for the definition of the Kac labels.)

\let\dlabel=\mlabel
\begin{figure}[p]
\centering
\begin{align*}
&A_1^{(1)} &&
\begin{tikzpicture}[start chain,node distance=2em,every path/.style={shorten >=4pt,shorten <=4pt},line width=0.5pt,baseline=-1ex]
\dnodenj{1}
\dnodenj{1}
\path (chain-1) -- node[anchor=mid] {\(\Longleftrightarrow\)} (chain-2);
\end{tikzpicture}
\\
&A_l^{(1)},\, l \ge 2 &&
\begin{tikzpicture}[start chain,node distance=2em,every path/.style={shorten >=4pt,shorten <=4pt},line width=0.5pt,baseline=-1ex,node distance=1ex and 2em]
\dnode{1}
\dnode{1}
\dydots
\dnode{1}
\dnode{1}
\begin{scope}[start chain=br going above]
\chainin(chain-3);
\node[ch,join=with chain-1,join=with chain-5,label={[inner sep=1pt]10:\(1\)},fill=yellow] {};
\end{scope}
\end{tikzpicture}
\\
&B_l^{(1)},\, l \ge 3 &&
\begin{tikzpicture}[node distance=2em,every path/.style={shorten >=4pt,shorten <=4pt},line width=0.5pt,baseline=-1ex]
\begin{scope}[start chain]
\dnode{1}
\dnode{2}
\dnode{2}
\dydots
\dnode{2}
\dnodenj{2}
\end{scope}
\begin{scope}[start chain=br going above]
\chainin(chain-2);
\dnodebr{1}
\end{scope}
\path (chain-5) -- node{\(\Rightarrow\)} (chain-6);
\end{tikzpicture}
\\
&C_l^{(1)},\, l \ge 2 &&
\begin{tikzpicture}[start chain,node distance=2em,every path/.style={shorten >=4pt,shorten <=4pt},line width=0.5pt,baseline=-1ex]
\dnodenj{1}
\dnodenj{2}
\dydots
\dnode{2}
\dnodenj{1}
\path (chain-1) -- node{\(\Rightarrow\)} (chain-2);
\path (chain-4) -- node{\(\Leftarrow\)} (chain-5);
\end{tikzpicture}
\\
&D_l^{(1)},\, l \ge 4 &&
\begin{tikzpicture}[node distance=2em,every path/.style={shorten >=4pt,shorten <=4pt},line width=0.5pt,baseline=-1ex]
\begin{scope}[start chain]
\dnode{1}
\dnode{2}
\dnode{2}
\dydots
\dnode{2}
\dnode{1}
\end{scope}
\begin{scope}[start chain=br going above]
\chainin(chain-2);
\dnodebr{1};
\end{scope}
\begin{scope}[start chain=br2 going above]
\chainin(chain-5);
\dnodebr{1};
\end{scope}
\end{tikzpicture}
\\
&G_2^{(1)} &&
\begin{tikzpicture}[start chain,node distance=2em,every path/.style={shorten >=4pt,shorten <=4pt},line width=0.5pt,baseline=-1ex]
\dnode{1}
\dnode{2}
\dnodenj{3}
\path (chain-2) -- node{\(\Rrightarrow\)} (chain-3);
\end{tikzpicture}
\\
&F_4^{(1)} &&
\begin{tikzpicture}[start chain,node distance=2em,every path/.style={shorten >=4pt,shorten <=4pt},line width=0.5pt,baseline=-1ex]
\dnode{1}
\dnode{2}
\dnode{3}
\dnodenj{4}
\dnode{2}
\path (chain-3) -- node[anchor=mid]{\(\Rightarrow\)} (chain-4);
\end{tikzpicture}
\\
&E_6^{(1)} &&
\begin{tikzpicture}[node distance=2em,every path/.style={shorten >=4pt,shorten <=4pt},line width=0.5pt,baseline=-1ex]
\begin{scope}[start chain]
\foreach \dyi in {1,2,3,2,1} {
\dnode{\dyi}
}
\end{scope}
\begin{scope}[start chain=br going above]
\chainin(chain-3);
\dnodebr{2}
\dnodebr{1}
\end{scope}
\end{tikzpicture}
\\
&E_7^{(1)} &&
\begin{tikzpicture}[node distance=2em,every path/.style={shorten >=4pt,shorten <=4pt},line width=0.5pt,baseline=-1ex]
\begin{scope}[start chain]
\foreach \dyi in {1,2,3,4,3,2,1} {
\dnode{\dyi}
}
\end{scope}
\begin{scope}[start chain=br going above]
\chainin(chain-4);
\dnodebr{2}
\end{scope}
\end{tikzpicture}
\\
&E_8^{(1)} &&
\begin{tikzpicture}[node distance=2em,every path/.style={shorten >=4pt,shorten <=4pt},line width=0.5pt,baseline=-1ex]
\begin{scope}[start chain]
\foreach \dyi in {1,2,3,4,5,6,4,2} {
\dnode{\dyi}
}
\end{scope}
\begin{scope}[start chain=br going above]
\chainin(chain-6);
\dnodebr{3}
\end{scope}
\end{tikzpicture}
\end{align*}
\caption{The extended affine Dynkin diagrams $X_n^{(1)}$. The number of nodes of each diagram is given by $l+1$ with indices $a_i$, $i=0,\ldots,l$. Figure based on \cite{pict:dynkin}.}
\label{fig:dynkin3}
\end{figure}
\begin{figure}[p]
\centering
\begin{align*}
&A_2^{(2)} &&
\begin{tikzpicture}[start chain,node distance=2em,every path/.style={shorten >=4pt,shorten <=4pt},line width=0.5pt,baseline=-1ex]
\dnodenj{2}
\dnodenj{1}
\path (chain-1) -- node {\QLeftarrow} (chain-2);
\end{tikzpicture}
\\
&A_{2l}^{(2)},\,l \ge 2 &&
\begin{tikzpicture}[start chain,node distance=2em,every path/.style={shorten >=4pt,shorten <=4pt},line width=0.5pt,baseline=-1ex]
\dnodenj{2}
\dnodenj{2}
\dydots
\dnode{2}
\dnode{1}
\path (chain-1) -- node[anchor=mid] {\(\Leftarrow\)} (chain-2);
\end{tikzpicture}
\\
&A_{2l-1}^{(2)},\, l \ge 3 &&
\begin{tikzpicture}[node distance=2em,every path/.style={shorten >=4pt,shorten <=4pt},line width=0.5pt,baseline=-1ex]
\begin{scope}[start chain]
\dnode{1}
\node[chj,label={below:\mlabel{2}},fill=yellow] {};
\dnode{2}
\dydots
\dnode{2}
\dnodenj{1}
\path (chain-5) -- node[anchor=mid] {\(\Leftarrow\)} (chain-6);
\end{scope}
\begin{scope}[start chain=br going above]
\chainin(chain-2);
\node[chj,label={below right:\mlabel{1}},fill=yellow] {};
\end{scope}
\end{tikzpicture}
\\
&D_{l+1}^{(2)},\, l \ge 2 &&
\begin{tikzpicture}[start chain,node distance=2em,every path/.style={shorten >=4pt,shorten <=4pt},line width=0.5pt,baseline=-1ex]
\dnode{1}
\dnodenj{1}
\dydots
\dnode{1}
\dnodenj{1}
\path (chain-1) -- node[anchor=mid] {\(\Leftarrow\)} (chain-2);
\path (chain-4) -- node[anchor=mid] {\(\Rightarrow\)} (chain-5);
\end{tikzpicture}
\\
&E_6^{(2)} &&
\begin{tikzpicture}[start chain,node distance=2em,every path/.style={shorten >=4pt,shorten <=4pt},line width=0.5pt,baseline=-1ex]
\dnode{1}
\dnode{2}
\dnode{3}
\dnodenj{2}
\dnode{1}
\path (chain-3) -- node[anchor=mid] {\(\Leftarrow\)} (chain-4);
\end{tikzpicture}
\\
\\
\\
&D_4^{(3)} &&
\begin{tikzpicture}[start chain,node distance=2em,every path/.style={shorten >=4pt,shorten <=4pt},line width=0.5pt,baseline=-1ex]
\dnode{1}
\dnode{2}
\dnodenj{1}
\path (chain-2) -- node {\(\Lleftarrow\)} (chain-3);
\end{tikzpicture}
\end{align*}
\caption{The twisted affine Dynkin diagrams $X_n^{(k)}$ for $k=2,3$. The number of nodes of each diagram is given by $l+1$ with indices $a_i$, $i=0,\ldots,l$. Figure based on \cite{pict:dynkin}.}
\label{fig:dynkin4}
\end{figure}
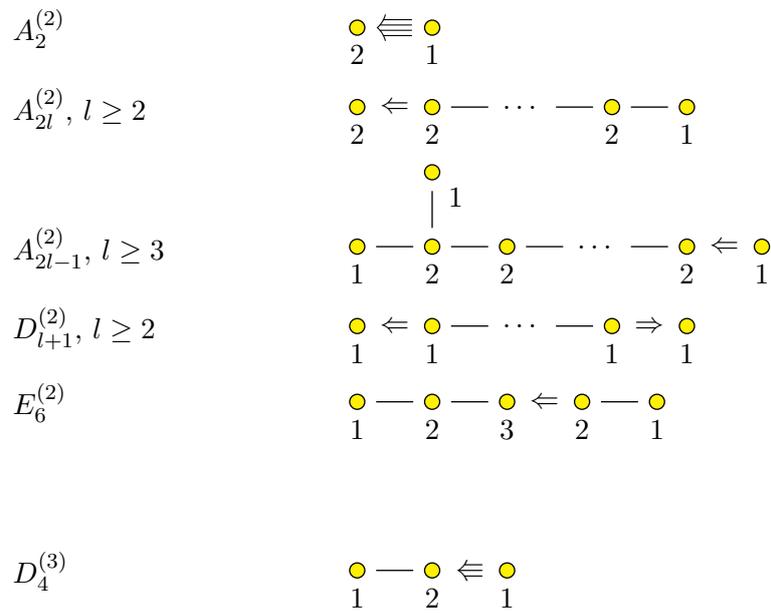

\begin{thm}[\cite{Kac90}, Theorem~8.6, Proposition~8.6]
Let $\g=X_n$ be a finite-dimensional, simple Lie algebra.
\begin{enumerate}
\item Let $X_n^{(k)}$ be a corresponding affine Dynkin diagram (with $l+1$ nodes) and let $s=(s_0,\ldots,s_l)$ be a sequence of non-negative, relatively prime integers $s_i\in\N$ (associated with the nodes in $X_n^{(k)}$). These data define (we omit the construction here) an automorphism $\sigma=\sigma(k,s)$ of $\g$ of order
\begin{equation*}
m=k\sum_{i=0}^{l}a_is_i.
\end{equation*}
\item Up to conjugation by an automorphism of $\g$, the automorphisms $\sigma$ obtained in this way exhaust all finite-order automorphisms of $\g$.
\item Two automorphisms $\sigma(k,s)$ and $\sigma(k',s')$ obtained in this way are conjugate by an automorphism of $\g$ if and only if $k=k'$ and the sequence $s$ can be transformed into $s'$ by an automorphism of the diagram $X_n^{(k)}$.
\item The number $k$ is the least positive integer for which $\sigma^k$ is an inner automorphism.
\item Let $i_1,\ldots,i_r$ be the indices for which the $s_i$ vanish, i.e.\ $s_{i_1}=\ldots=s_{i_r}=0$. Then the fixed-point Lie subalgebra $\g^\sigma$ is isomorphic to the direct sum of the abelian Lie algebra $\C^{l-r}$ and the semisimple Lie algebra whose Dynkin diagram is the subdiagram of the affine diagram $X_n^{(k)}$ consisting of the nodes $i_1,\ldots,i_r$.
\end{enumerate}
\end{thm}
We omit the explicit description of the automorphisms of $\g$ since we are only interested in the possible fixed-point Lie subalgebras. Just note that the diagram $X_n^{(k)}$, $k=1,2,3$, corresponds exactly to the outer automorphism $\sigma\cdot\Inn(\g)\in\Out(\g)$, i.e.\ a certain diagram automorphism of $X_n$ of order $k$.

Cartan subalgebras and rank can also be defined for semisimple and reductive Lie algebras.
\begin{cor}\label{cor:subrank}
Let $\g$ be simple Lie algebra and $\sigma$ some automorphism of $\g$. Then the fixed-point Lie subalgebra $\g^\sigma$ is reductive and $\rk(\g^\sigma)\leq\rk(\g)$.
\end{cor}

\minisec{Automorphisms of Semisimple Lie Algebras}

In the following we consider the case of an automorphism $\varphi$ of a finite-dimensional, semisimple Lie algebra
\begin{equation*}
\g=\g_1\oplus\ldots\oplus\g_s
\end{equation*}
with $\g_i$, $i=1,\ldots,s$, simple ideals and $s\in\N$. Let $\varphi\in\Aut(\g)$. Then the image $\varphi(\g_i)$ of one of the simple components $\g_i$ under $\varphi$ is isomorphic to $\g_i$ and an ideal in $\g$, i.e.\ equal to one of the simple components of $\g$ isomorphic to $\g_i$.

The above considerations imply that we can sort the simple components of $\g$ into ideals according to their isomorphism class and any automorphism $\varphi$ on $\g$ acts separately on each of these ideals. Hence, let us write $\g$ as
\begin{equation*}
\g\cong(\g_1)^{t_1}\oplus\ldots\oplus(\g_r)^{t_r}
\end{equation*}
with the $\g_i$ simple ideals and $t_i\in\Ns$, $i=1,\ldots,r$, for $r\in\N$ and $\g_i\not\cong\g_j$ for $i\ne j$. To describe the action of an automorphism $\varphi\in\Aut(\g)$ it suffices to study the case of a Lie algebra
\begin{equation*}
\g=(X_n)^t
\end{equation*}
for some simple Lie algebra $X_n$. We know that $\varphi$ acts as permutation $\tau\in S_t$ on the $t$ copies of $X_n$, i.e.\ if we label these copies $X_n^{(1)},\ldots,X_n^{(t)}$, then $\varphi(X_n^{(i)})=X_n^{(\tau(i))}$ for all $i=1,\ldots,t$ and we can view each $\varphi_i:=\varphi|_{X_n^{(i)}}$ as an automorphism of $X_n$. The permutation $\tau$ decomposes into cycles, on which $\varphi$ acts separately. Let $(X_n^{(p_1)}~X_n^{(p_2)}~\ldots~X_n^{(p_d)})$ be one such cycle of length $d\in\Ns$. The situation is depicted in the following diagram:

\begin{equation*}
\begin{tikzpicture}[every path/.style={shorten >=6pt,shorten <=6pt}]
\def \n {5}
\def \radius {1.8cm}
\def \margin {8} \foreach \s in {1,...,3}
{
  \node[draw=none] at ({360/\n * (\s - 1)}:\radius) {$X_n^{(p_\s)}$};
  \draw[->, >=latex] ({360/\n * (\s - 1)+\margin}:\radius) 
    arc ({360/\n * (\s - 1)+\margin}:{360/\n * (\s)-\margin}:\radius);
  \node[draw=none] at ({360/\n * (\s - 0.5)}:\radius * 0.85) {$\varphi_\s$};
}
\foreach \s in {4}
{
  \node[draw=none] at ({360/\n * (\s - 1)}:\radius) {$X_n^{(p_\s)}$};
  \draw[dotted,->, >=latex] ({360/\n * (\s - 1)+\margin}:\radius) 
    arc ({360/\n * (\s - 1)+\margin}:{360/\n * (\s)-\margin}:\radius);
  \node[draw=none] at ({360/\n * (\s - 0.5)}:\radius * 0.85) {};
}
\foreach \s in {5}
{
  \node[draw=none] at ({360/\n * (\s - 1)}:\radius) {$X_n^{(p_d)}$};
  \draw[->, >=latex] ({360/\n * (\s - 1)+\margin}:\radius) 
    arc ({360/\n * (\s - 1)+\margin}:{360/\n * (\s)-\margin}:\radius);
  \node[draw=none] at ({360/\n * (\s - 0.5)}:\radius * 0.85) {$\varphi_d$};
}
\end{tikzpicture}
\end{equation*}
Clearly, each map on $(X_n)^t$ built from $t$ automorphisms $\varphi_i$ of $X_n$ and a permutation $\tau\in S_t$ in the above manner will be an automorphism of $\g=(X_n)^t$.

The order of $\varphi$ restricted to the above cycle is
\begin{equation*}
d\cdot\ord(\varphi_d\circ\ldots\circ\varphi_1),
\end{equation*}
where $\varphi_d\circ\ldots\circ\varphi_1\in\Aut(X_n)$ and the order of $\varphi\in\Aut((X_n)^t)$ is the least common multiple of all the orders of $\varphi$ restricted to the cycles of $\tau$.

In the end we are interested in the fixed-point Lie subalgebra with respect to a Lie algebra automorphism $\varphi$ on a semisimple Lie algebra $\g$. Since this automorphism acts separately on each collection of isomorphic simple components $(X_n)^t$ and indeed on each cycle $X_n^{(p_1)}\oplus X_n^{(p_2)}\oplus\ldots\oplus X_n^{(p_d)}$, the fixed-point Lie subalgebra $\g^\varphi$ is the direct sum of the fixed-point Lie subalgebra of each of the cycles. For the cycle above, the fixed-point Lie subalgebra is given by
\begin{align*}
&\left\{(x_1,\ldots,x_d)\in (X_n)^d \xmiddle| x_2=\varphi_1(x_1),\ldots,x_d=\varphi_{d-1}(x_{d-1}),x_1=\varphi_d(x_d)\right\}\\
&=\left\{(x_1,\ldots,x_d)\in (X_n)^d \xmiddle| x_1=(\varphi_d\circ\ldots\circ\varphi_1)(x_1),x_2=\varphi_1(x_1),\ldots,x_d=\varphi_{d-1}(x_{d-1})\right\}\\
&=\left\{(x_1,\varphi_1(x_1),\ldots,(\varphi_{d-1}\circ\ldots\circ\varphi_1)(x_1))\in (X_n)^d \xmiddle| x_1=(\varphi_d\circ\ldots\circ\varphi_1)(x_1)\right\}\\
&\cong\left\{x\in X_n \xmiddle| x=(\varphi_d\circ\ldots\circ\varphi_1)(x)\right\}\\
&=(X_n)^{\varphi_d\circ\ldots\circ\varphi_1}.
\end{align*}

In particular, the following generalisation of Proposition~\ref{prop:fixreductive} and Corollary~\ref{cor:subrank} to semisimple Lie algebras holds:
\begin{prop}\label{prop:fixreductive2}
Let $\g$ be a finite-dimensional, semisimple Lie algebra and $\varphi$ an automorphism of $\g$ of finite order. Then the fixed-point Lie subalgebra $\g^\varphi$ is a reductive Lie algebra and $\rk(\g^\sigma)\leq\rk(\g)$.
\end{prop}

\chapter{Vertex Superalgebras}
\section{Vertex Superalgebras}\label{sec:super}

It is easy to adopt the notion of a vertex superalgebra starting from the definition of ordinary vertex algebras (see e.g.\ \cite{Kac98}, Section~1.3). The space of states $V$ has to be a superspace, i.e.\ a $\Z_2$-graded vector space $V=V^{\bar 0}\oplus V^{\bar 1}$. The $\Z_2$-grading is called \emph{parity} and written as $|a|=\bar\i$ for $a\in V^{\bar i}$. The coefficients $a_n$ of the vertex operator $Y(a,x)$ have to be endomorphisms of $V$ of parity $\bar i$ if $a$ is homogeneous of parity $\bar i$ for all $n\in\Z$. The vacuum vector has to obey $\vac\in V^{\bar 0}$ and $T$ has to be an operator of even parity (both facts already follow from the other axioms). Moreover, the locality axiom has to be modified: for each $a,b\in V$
\begin{equation*}
(x-y)^N\left(Y(a,x)Y(b,y)-(-1)^{|a||b|}Y(b,y)Y(a,x)\right)=0
\end{equation*}
for sufficiently large $N\in\N$. Alternatively, the super Jacobi identity reads
\begin{align*}
&\iota_{x_1,x_0}x_2^{-1}\delta\left(\frac{x_1-x_0}{x_2}\right)Y(Y(a,x_0)b,x_2)\\
&=\iota_{x_1,x_2}x_0^{-1}\delta\left(\frac{x_1-x_2}{x_0}\right)Y(a,x_1)Y(b,x_2)\\
&\quad-(-1)^{|a||b|}\iota_{x_2,x_1}x_0^{-1}\delta\left(\frac{x_2-x_1}{-x_0}\right)Y(b,x_2)Y(a,x_1)
\end{align*}
for $\Z_2$-homogeneous $a,b\in V$ and is equivalent to super-locality under the presence of the other axioms. There is a small subtlety concerning the translation axiom (see \cite{Li96b}, Remark~2.2.2): it should be formulated as
\begin{equation*}
[T,Y(a,x)]=Y(Ta,x)=\partial_x Y(a,x).
\end{equation*}
The first equality has to be additionally assumed. For vertex algebras this was a consequence of the other definitions. Also, for vertex operator superalgebras the first equality already follows from the Jacobi identity.

For a (weak) graded vertex superalgebra, generalised from a (weak) graded vertex algebra, or a (weak) vertex operator superalgebra, generalised from a (weak) \voa{}, we allow the weight grading on $V$ to be $(1/2)\Z$-valued instead of just $\Z$-valued, i.e.\
\begin{equation*}
V=\bigoplus_{n\in\frac{1}{2}\Z}V_n
\end{equation*}
with $\wt(a)=n\in(1/2)\Z$ for $a\in V_n$. We also assume that the $\Z_2$- and the weight $(1/2)\Z$-grading are compatible in the sense that both $V^{\bar 0}$ and $V^{\bar 1}$ are $(1/2)\Z$-graded subspaces of $V$, i.e.\
\begin{equation*}
V^{\bar i}=\bigoplus_{n\in\frac{1}{2}\Z}V_n^{\bar i}
\end{equation*}
where $V_n^{\bar i}=V_n\cap V^{\bar i}$. For a vertex operator superalgebra this is true if $\omega$ has even parity since then also $L_0=\omega_1$ is of even parity and hence $L_0$ restricts to an operator on both $V^{\bar 0}$ and $V^{\bar 1}$.

In addition, it is often assumed that the weight grading on $V^{\bar 0}$ is in $\Z$ and possibly also that the grading on $V^{\bar 1}$ is in $\Z+1/2$. But it is also possible to consider purely $\Z$-graded vertex superalgebras or even ``$(1/2)\Z$-graded vertex algebras'' (by setting $V^{\bar 1}=\{0\}$ and allowing $V^{\bar 0}$ to be $(1/2)\Z$-graded).

Finally, let us remark that an \aia{} with $N=2$, $D=\Z_2$, $F(\alpha,\beta,\gamma)=1$ and $\Omega(\alpha,\beta)=(-1)^{\alpha\cdot\beta}$ for all $\alpha,\beta,\gamma\in\Z_2$ is exactly a vertex operator superalgebra (with compatible gradings as described above) if we additionally assume that the weight $(1/N)\Z=(1/2)\Z$-grading is bounded from below and that the graded components of $V$ are finite-dimensional.

\minisec{Tensor Product}
Similar to the case of vertex algebras the tensor product $V\otimes W$ of two vertex superalgebras $V$ and $W$ admits the structure of a vertex superalgebra if we define $\vac=\vac_V\otimes\vac_W$ and
\begin{equation*}
Y_{V\otimes W}(a\otimes b,x)(v\otimes w)=(-1)^{|b||v|}Y_V(a,x)v\otimes Y_W(b,x)w
\end{equation*}
for $\Z_2$-homogeneous $a,v\in V$ and $b,w\in W$. Then the locality axiom (applied to some vector $v\otimes w$) becomes
\begin{align*}
&(x-y)^NY(a\otimes b,x)Y(c\otimes d,y)\\
&=(x-y)^NY(a\otimes b,x)\left((-1)^{|d||v|}Y(c,y)\otimes Y(d,y)\right)\\
&=(x-y)^N\left((-1)^{|b|(|v|+|c|)}Y(a,x)\otimes Y(b,x)\right)\left((-1)^{|d||v|}Y(c,y)\otimes Y(d,y)\right)\\
&=(x-y)^N(-1)^{|a||c|+|b||d|}\left((-1)^{|d||v|}Y(c,y)\otimes Y(d,y)\right)\left((-1)^{|b|(|v|+|c|)}Y(a,x)\otimes Y(b,x)\right)\\
&=(x-y)^N(-1)^{|a||c|+|b||d|}(-1)^{-|d||a|}Y(c\otimes d,y)(-1)^{|b||c|}Y(a\otimes b,x)\\
&=(x-y)^N(-1)^{(|a|+|b|)(|c|+|d|)}Y(c\otimes d,y)Y(a\otimes b,x),
\end{align*}
where we used that $Y(\cdot,x)$ is parity-preserving, i.e.\ that for instance the parity of the coefficients of $Y(c,y)v$ is $|c|+|v|$ for $\Z_2$-homogeneous $c,v\in V$. This shows that the parity grading on $V\otimes W$ is necessarily given by
\begin{equation*}
(V\otimes W)^{\bar{0}}=(V^{\bar{0}}\otimes_\C W^{\bar{0}})\oplus(V^{\bar{1}}\otimes_\C W^{\bar{1}})
\end{equation*}
and
\begin{equation*}
(V\otimes W)^{\bar{1}}=(V^{\bar{0}}\otimes_\C W^{\bar{1}})\oplus(V^{\bar{1}}\otimes_\C W^{\bar{0}}).
\end{equation*}

For graded vertex superalgebras $V$ and $W$, defining
\begin{equation*}
\wt(v\otimes w):=\wt(v)+\wt(w)
\end{equation*}
for homogeneous $v\in V$ and $w\in W$ makes $V\otimes W$ again a graded vertex superalgebra and for vertex operator superalgebras we define
\begin{equation*}
\omega_{V\otimes W}=\omega_V\otimes\vac+\vac\otimes\omega_W
\end{equation*}
to make $V\otimes W$ a vertex operator superalgebra.

\minisec{Supercharacter}
Recall that the character of a \voa{} $V$ (and that of a vertex operator superalgebra) is defined as 
\begin{equation*}
\ch_{V}(\tau)=\sum_{n\in\Z}\dim_\C(V_n)q_\tau^{n-c/24}.
\end{equation*}
For a vertex operator superalgebra $V$ we additionally introduce the notion of \emph{supercharacter}
\begin{equation*}
\sch_{V}(\tau):=\sum_{n\in\Z}\sdim_\C(V_n)q_\tau^{n-c/24}
\end{equation*}
where $\sdim_\C(V_n)$ is the \emph{superdimension} of $V_n$, i.e.\ $\sdim_\C(V_n)=\dim_\C(V_n^{\bar{0}})-\dim_\C(V_n^{\bar{1}})$ for $n\in\Z$.

\end{appendices}

\bookmarksetupnext{level=part}
\bibliographystyle{alpha_noseriescomma}
\bibliography{quellen}{}

\end{document}